\documentclass[10pt]{amsbook}

\usepackage{amsmath,amssymb,amscd,amsfonts}
\usepackage{graphicx}
\usepackage{mathrsfs}
\usepackage[all]{xy}
\usepackage{multirow}
\usepackage{xspace,color}
\usepackage[dvipdfmx]{hyperref}

\newcommand{\mathbbm}[1]{\mathbf{#1}}
\newcommand{\A}{\mathbb{A}}
\newcommand{\C}{\mathbb{C}}
\newcommand{\bD}{\mathbb{D}}
\newcommand{\bK}{\mathbb{K}}

\newcommand{\Et}{\mathrm{Et}}
\newcommand{\F}{\mathbb{F}}
\newcommand{\Ga}{\mathbb{G}_a}
\newcommand{\Gm}{\mathbb{G}_m}
\newcommand{\I}{\mathbb{I}}
\newcommand{\bJ}{\mathbb{J}}

\newcommand{\Mbar}{\overline{M}}
\newcommand{\bbM}{\mathbb{M}}
\newcommand{\bM}{\mathbf{M}}
\newcommand{\N}{\mathbb{N}}

\newcommand{\bN}{\mathbf{N}}

\newcommand{\Q}{\mathbb{Q}}

\newcommand{\R}{\mathbb{R}}

\newcommand{\Z}{\mathbb{Z}}

\newcommand{\cA}{\mathcal{A}}
\newcommand{\cB}{\mathcal{B}}
\newcommand{\cC}{\mathcal{C}}
\newcommand{\cCo}[1]{\mathcal{C}^{#1}}
\newcommand{\cCXo}[1]{\mathcal{C}_{X}^{#1}}
\newcommand{\Cip}{{(\mathcal{C}_0,\iota_0)}}
\newcommand{\cD}{\mathcal{D}}
\newcommand{\cE}{\mathcal{E}}
\newcommand{\cF}{\mathcal{F}}
\newcommand{\cFC}{\mathcal{FC}}
\newcommand{\cFCo}[1]{\mathcal{FC}^{#1}}
\newcommand{\cFCot}[1]{\wt{\mathcal{FC}}^{#1}}
\newcommand{\cFCotu}[1]{\wt{\mathcal{FC}}_m^{#1}}

\newcommand{\cH}{\mathcal{H}}
\newcommand{\cHom}{{\mathscr{H}\mathit{om}}}
\newcommand{\sI}{\mathscr{I}}
\newcommand{\sJ}{\mathscr{J}}

\newcommand{\cK}{\mathcal{K}}

\newcommand{\cM}{\mathcal{M}}
\newcommand{\cO}{\mathcal{O}}
\newcommand{\cS}{\mathcal{S}}
\newcommand{\cT}{\mathcal{T}}

\newcommand{\eps}{\epsilon}
\newcommand{\vep}{\varepsilon}

\newcommand{\Aut}{\mathrm{Aut}}
\newcommand{\Bil}{\mathcal{BIL}}
\newcommand{\Bl}{\mathrm{Bl}}
\newcommand{\BS}{\mathrm{SB}}

\newcommand{\diag}{\mathrm{diag}}

\newcommand{\End}{\mathrm{End}}

\newcommand{\Ext}{\mathrm{Ext}}
\newcommand{\ffor}{\mathrm{for}}
\newcommand{\Fil}{\mathrm{Fil}}

\newcommand{\Frac}{\mathrm{Frac}}
\newcommand{\Gal}{\mathrm{Gal}}

\newcommand{\GL}{\mathrm{GL}}
\newcommand{\Gr}{\mathrm{Gr}}
\newcommand{\gr}{\mathrm{gr}}

\newcommand{\GSp}{\mathrm{GSp}}
\newcommand{\Hom}{\mathrm{Hom}}
\newcommand{\id}{\mathrm{id}}

\newcommand{\Inn}{\mathrm{Inn}}

\newcommand{\Isom}{\mathrm{Isom}}
\newcommand{\length}{\mathrm{length}}

\newcommand{\Ker}{\mathrm{Ker}}
\newcommand{\Lie}{\mathrm{Lie}}
\newcommand{\LS}{\mathcal{LS}}

\newcommand{\Mat}{\mathrm{Mat}}
\newcommand{\Mor}{\mathrm{Mor}}

\newcommand{\Obj}{\mathrm{Obj}}
\newcommand{\op}{\mathrm{op}}
\newcommand{\ord}{\mathrm{ord}}

\newcommand{\Presh}{\mathrm{Presh}}
\newcommand{\Pro}{\mathrm{Pro}}

\newcommand{\res}{\mathrm{res}}

\newcommand{\SCH}{\mathcal{SCH}}

\newcommand{\sm}{\mathrm{sm}}
\newcommand{\str}{\mu}
\newcommand{\uppr}{\dagger}

\newcommand{\Spec}{\mathrm{Spec}\,}

\newcommand{\Sub}{\mathrm{Sub}}
\newcommand{\Supp}{\mathrm{Supp}\,}

\newcommand{\tr}{\mathrm{tr}}

\newcommand{\frf}{\mathfrak{f}}
\newcommand{\frg}{\mathfrak{g}}
\newcommand{\frh}{\mathfrak{h}}
\newcommand{\frX}{\mathfrak{X}}
\newcommand{\frU}{\mathfrak{U}}

\newcommand{\inj}{\hookrightarrow}
\newcommand{\surj}{\twoheadrightarrow}
\newcommand{\resp}{resp.\ }
\newcommand{\xto}[1]{\xrightarrow{#1}}
\newcommand{\wt}[1]{\widetilde{#1}}
\newcommand{\wh}[1]{\widehat{#1}}

\newcommand{\advertisement}[1]{}

\newcommand{\bd}{{\mathbbm{d}}}

\newcommand{\pr}{\mathrm{pr}}

\newcommand{\modx}{\,\mathrm{mod}\,}

\newcommand{\bL}{\mathbb{L}}

\newcommand{\Shv}{\mathrm{Shv}}

\newcommand{\add}[1]{{\color{blue}#1}}

\newcommand{\fram}{\mathfrak{m}}
\newcommand{\bc}{\mathbf{c}}
\newcommand{\bE}{\mathbf{E}}
\newcommand{\bF}{\mathbf{F}}
\newcommand{\Fbar}{\overline{\F}}

\newcommand{\bi}{\mathbf{i}}
\newcommand{\bj}{\mathbf{j}}
\newcommand{\bS}{\mathbf{S}}
\newcommand{\bs}{\mathbf{s}}
\newcommand{\bT}{\mathbf{T}}
\newcommand{\bt}{\mathbf{t}}

\newcommand{\bw}{\mathbf{w}}

\newcommand{\bMbar}{\overline{\mathbf{M}}}

\newcommand{\Lat}{\mathbf{Lat}}
\newcommand{\Pair}{\mathbf{Pair}}

\newcommand{\quotobj}{quotient object\xspace}
\newcommand{\quotobjs}{quotient objects\xspace}

\newcommand{\fibr}{fibration\xspace}
\newcommand{\fibrs}{fibrations\xspace}

\newcommand{\Fibrs}{Fibrations\xspace}

\newcommand{\cofibr}{cofibration\xspace}
\newcommand{\cofibrs}{cofibrations\xspace}

\newcommand{\quot}[2]{{#2}\backslash{#1}}

\newcommand{\quotu}[2]{({#2}\backslash{#1})^\dagger}
\newcommand{\quotid}[1]{{\{\mathrm{id}_{#1} \}}\backslash{#1}}

\newcommand{\cCot}[1]{\wt{\mathcal{C}}^{#1}}
\newcommand{\cCotu}[1]{\wt{\mathcal{C}}_m^{#1}}

\newcommand{\frV}{\mathfrak{V}}

\theoremstyle{plain}
\newtheorem{thm}{Theorem}[chapter]
\newtheorem{lem}[thm]{Lemma}
\newtheorem{prop}[thm]{Proposition}

\newtheorem{cor}[thm]{Corollary}
\newtheorem{conj}[thm]{Conjecture}
\newtheorem{ex}[thm]{Example}

\theoremstyle{definition}
\newtheorem{defn}[thm]{Definition}

\theoremstyle{remark}
\newtheorem{rmk}[thm]{Remark}

\numberwithin{section}{chapter}
\numberwithin{equation}{chapter}

\makeindex

\begin{document}

\title[Distributions and Euler systems]
{Distributions and Euler Systems for the General Linear Group}

\author{Satoshi Kondo}
\address{
Middle East Technical University \\
Northern Cyprus Campus, Kalkanli \\
Guzelyurt, Mersin 10, Turkey;
Kavli Institute for the Physics and Mathematics of the Universe\\
University of Tokyo\\
5-1-5 Kashiwanoha\\
Kashiwa 277-8583\\ Japan
}
\author{Seidai Yasuda}
\address{
Department of Mathematics, Hokkaido University,
Kita 10, Nishi 8, Kita-Ku, Sapporo, Hokkaido, 060-0810, Japan 
}
\date{\add{\today}}


\begin{abstract}
The main aim is to make rigorous the slogan 
``the $d$-fold tensor product of distributions is an Euler system for 
$\GL_d$'' where $d$ is any positive integer.  
Of the few known examples of Euler systems, let us look at those of cyclotomic units and of Beilinson-Kato elements.   The cyclotomic units satisfy distribution property, and this is the key to the proof of the norm relation property for $\GL_1$.   
For the Beilinson-Kato elements, the Siegel units satisfy distribution property, and the 2-fold tensor product, giving rise to elements in the $K_2$ of 
modular curves, satisfies the norm relation for $\GL_2$.  
We make this common property clear, generalizing everything to $\GL_d$ for arbitrary $d$.

As an application (our main arithmetic result), we construct elements in the motivic cohomology of Drinfeld modular schemes with integral coefficient and 
show that the norm relation common to Euler systems (i.e., the norm of one element is described using the local L-factor times another element) hold.

We use the language of Y-sites which was introduced in \cite{Grids}
in order to simplify the computation common to the theory 
of automorphic forms.
Instead of double cosets, we work more systematically with torsion modules and Q-morphisms (of Quillen) between them.   
The idea is that torsion modules are (abstract) 
level structures, and Q-morphisms 
induce morphisms between some moduli spaces.  Chapters 1 and 3 serve as a sequel to our previous paper \cite{Grids}: Further generalities on Y-sites, more examples of Y-sites, and proofs of some statements in loc.cit are given.   An application is also given: We provide a group theoretic formulation of a conjecture of Tamagawa on affine curves over an algebraic closure of a finite field.
\end{abstract}

\maketitle
\tableofcontents
\chapter{Introduction}
\section{Introduction}
An Euler system is a useful tool in Iwasawa theory, but not many examples are known.
The main aim is to give a method in constructing (something similar to) 
Euler systems (norm compatible systems).
Using an abstract setup (using $Y$-sites), 
we show that any product of distributions gives rise to such a norm compatible system.
As an application, we provide a new Euler system in the (integral) 
motivic cohomology grouops of Drinfeld modular schemes.

Let us describe the contents of this Chapter.

The main result in this book is the construction of 
Euler systems (a system of elements satisfying some 
compatibility under norm maps, more on this below) in motivic cohomology of Drinfeld modular schemes over $A$, where $A$ 
is the ring of integers (in Drinfeld modular context) of a global field in positive characteristic (Theorem~\ref{thm:Drinfeld Euler}).
A less sharp version is given in Section~\ref{sec:intro Euler}.

We discuss Euler systems and how we understand it from the automorphic side 
in Section~\ref{sec:intro Euler automorphic}.

The tool ($Y$-sites) we use for the proof is provided by our 
paper \cite{Grids}.   This was a category theoretic result 
rather than arithmetic.  In this book, we develop more generalities 
concerning $Y$-sites, provide the proof of norm relation in this abstract setting (Theorem~\ref{main theorem}), 
and then apply it to arithmetic (the Drinfeld modular setup).
The overview of results on $Y$-sites is given in Section~\ref{sec:intro Y-sites}.   
Having introduced the notion of $Y$-sites,
we give slightly more technical overview of our result on Euler systems in Section~\ref{sec:intro Euler tech}.

\section{Euler systems on the automorphic side}
\label{sec:intro Euler automorphic}
We study Euler systems.   For the definition, history, applications, etc. of Euler systems, the reader is referred to Rubin's book \cite{Rubin}.
Our use of the term ``Euler system" is different from that of Rubin's in two respects.   One is that we are on the automorphic side rather than on the Galois side (in the sense of the Langlands correspondence).  
The other is that our definition concerns solely with the norm relation, while Rubin adds some other conditions so 
that there is an immediate consequence in Iwasawa theory.
Let us explain this in more detail in this Section.

Recall that an Euler system in loc.\ cit.\ is a collection of elements in the Galois cohomology of some $p$-adic Galois representation, indexed by the ideals of the ring of integers of some number field.   The main defining property is that the norm relation (the Euler system relation) holds.  That is, ``an element in the collection is sent to another element times the local $L$-factor under the norm map".   The norm relation does not readily give an application to Iwasawa theory type of result and it requires some more work.   The extra axioms in Rubin's book are one sufficient condition; there are other methods.    We do not consider this aspect in this book, and we take this norm relation as our definition of Euler system and nothing else.   

Note also that the $L$-factor mentioned above is the $L$-factor on the Galois side, i.e., it is 
written in terms of the characteristic polynomial of Frobenius.   Our definition is on the automorphic side.
That is, we use the $L$-factor in terms of Hecke operators.   Let us illustrate this in the following two examples.

The simplest example is the Euler system of cyclotomic units.   
(See Section~\ref{sec:cyclotomic Euler} 
where we worked out the 
technical details.) 
Cyclotomic units can be understood as elements in the multiplicative group of cyclotomic fields, 
and the norm map is the norm map of fields for extension of cyclotomic fields (a divisibility of two ideals, i.e., an integer $N_1$ dividing $N_2$, gives an extension of cyclotomic fields).   The cyclotomic field $\Q(\zeta_N)$ is the moduli space of 
the multiplicative group scheme $\Gm$ with level $\Z/N\Z$-structure.
   Then, cyclotomic units or elements of the form $1-\zeta_N$ form an Euler system.   (Here, the multiplicative group of a cyclotomic field is understood as a Galois cohomology (or something related) or as the $K_1$ group of the spectrum of the cyclotomic field).  
When $N$ is an integer and $p$ is prime which is prime to $N$, there is a norm map between the multiplicative groups of cyclotomic fields.     The norm relation states that the norm of $1-\zeta_{Np}$ can be expressed as the $L$-factor
at $p$ 
times $1-\zeta_{N}$.   The $L$-factor is in terms of Frobenius (see \cite[Ch.\ 3.2]{Rubin}), but it can also be understood as 
a polynomial in Hecke operators obtained by the morphisms induced by the change of level structure.

The other example is Kato's Euler system (see \cite[Prop.\ 2.4, p.126]{Kato}, \cite[Prop.\ 2.3.6, p.399]{Scholl}).  
The elements are products of Siegel units (Beilinson type elements) in motivic cohomology of (affine) moduli of elliptic curves indexed by the level structures.    When there are two levels $N_1$ and $N_2$, one dividing the other, there is a
norm map between the motivic cohomology groups (or K-groups) of the moduli spaces.    Kato's theorem is that 
one element is sent to another in the system multiplied by the L-factor in some Hecke operators.

We regard the two examples above in a uniform manner from the automorphic point of view.   
The cyclotomic case is for 
$\GL_1$ of $\Q$ and the modular curve case is for $\GL_2$ of $\Q$.     The $L$-factor is actually the $L$-factor on the automorphic side and may be interpreted in terms of Hecke operators.
This point of view enables us to generalize to the 
case for $\GL_d$ of $K$ where $K$ is a global field and $d \ge 1$. 
We therefore take as our definition of Euler system to be a collection of elements in an 
abelian group (with some functorial structures such as pushforward (norm), pullback) indexed by ideals in a Dedekind domain 
with finite residue fields, satisfying a norm relation that is described in terms of the $L$-factor of Hecke operators.

In the following section, we provide an example in the case 
for $\GL_d$ of $K$ where $K$ is a global field of positive characteristic.
It is our main result.    

\section{Our theorem on Drinfeld modular schemes}
\label{sec:intro Euler}
Let us state our main theorem concerning Euler systems in
the motivic cohomology of Drinfeld modular schemes.
Below is a simpler version of Theorem \ref{thm:Drinfeld Euler}.

\subsection{}
We recall the usual setup for arithmetic over function fields
and for Drinfeld modules.
Let $C$ be a smooth projective geometrically
irreducible curve over a finite
field $\F_q$ of $q$ elements.
Let $F$ denote the function field of $C$.
Fix a closed point $\infty$ of $C$.
Let $A=\Gamma(C\setminus \{\infty\},\cO_C)$ be the coordinate
ring of the affine $\F_q$-scheme
$C \setminus \{\infty\}$.

\subsection{}
Let $d \ge 1$ be an integer.
Let $U \subset \Spec A$
be an open subscheme.
Let $N$ be a finitely generated torsion 
$A$-module.
Let $\cM_{N, U}^d$ denote 
the moduli of Drinfeld modules of rank $d$ over $S$
with structure of level $N$.
For the representability and regularity, see Proposition~\ref{prop:level N moduli}.
(The base scheme $U$ can often be taken to be $\Spec A$.)

\subsection{}
We construct a (universal) theta function 
$\theta \in \Gamma(E \setminus {0}, \cO_E^\times)$
as an invertible function (unit)
on the universal Drinfeld module $E$ minus the zero section.
See Section \ref{sec:theta functions}.

When  $b \in N$, and $b: \cM^d_{N,U} \to E$
is the section given by the level $N$ structure,
we denote by $g_{N,b}^{(d)} \in \cO(\cM^d_{N,U})$
the pullback of the theta function to the moduli
by the section $b$.    These elements are units
when the section $b$ does not intersect the zero section.

\subsection{}
For $i=1, \dots, d$,
let $N_i \in \cC^d_U$
be a nonzero $A$-module of finite length
that is generated by one element.
Let $b_i \in N_i \setminus \{0\}$.
Set $\bN=\bigoplus_{i=1}^d N_i$.
For $i=1, \dots, d$,
let 
$\iota_i\colon N_i \to \bN$
be the inclusion into the $i$-th factor.
We take $b_i$'s so that 
$g_{N_i, b_i}^{(d)}$ is a unit.
Each $\iota_i$ induces
$f_i\colon  \cM_{\bN, U}^d \to \cM_{N_i, U}^d$.

\subsection{}
We use the motivic cohomology of 
Mazza-Voevodsky-Weibel
from \cite{MVW}.
We identify the group of units
$\cO(X)^\times$ and the 
motivic cohomology
$H^1_\cM(X, \Z(1))$ of a scheme $X$
smooth over a field.
Using the product structure 
$H^1_\cM(X, \Z(1))^{\otimes d} 
\to H^d_\cM(X, \Z(d))$
for positive integers $d$
of motivic cohomology, we set
\[
\wt{\kappa}_{\bN, U, (b_i)}
=g_{N_1, U, (b_1)} \otimes \cdots \otimes g_{N_d, U, (b_d)}
\in
H^d_\cM(\cM_{\bN, U}^d,
\Z(d)).
\]

\subsection{}
Let $N_i'$ be a nonzero quotient $A$-module of $N_i$
for $i=1, \dots, d$.
Let $b_i'$ denote the 
image of $b_i$ in $N_i'$.
We write
$\bN'=\bigoplus_{j=1}^d N_j'$
and $N_i''=\Ker(N_i \to N_i')$.
Let 
$m\colon  \cM^d_{\bN, U} \to \cM^d_{\bN', U}$
be the morphism induced by the surjection
$\bN \to \bN'$.
Since $m$ is finite surjective, we have the pushforward map
\[
m_*\colon 
H^d_\cM(\cM_{\bN, U}^d, \Z(d))
\to
H^d_\cM(\cM_{\bN', U}^d, \Z(d))
\]
between the motivic cohomology groups.

\begin{thm}[Theorem {\ref{thm:Drinfeld Euler}}]
\label{thm:intro Drinfeld Euler}

The following statements hold.
(1) If $\Supp N_i'' \subset \Supp N_j'$
for any 
$1 \le i,j \le d$, then
\[
m_* \wt{\kappa}_{\bN, (b_j), U}
=
\wt{\kappa}_{\bN', (b_j'), U}
\]

(2) Let $\wp$ be a closed point of $C$.
Suppose that
$\Supp N_i'' \subset \{\wp\}
\subset \Supp N_i$
for every $i$.
Let $e$ denote the number of $i$'s
with $\wp \notin \Supp N_i'$.
Then
\[
m_*
\wt{\kappa}_{\bN, (b_i), U}=
\sum_{r=0}^e
(-1)^r
q_\wp^{r(r-1)/2}
T_{[\wp^{\oplus r}]}
\wt{\kappa}_{\bN', (b_j'), U}
\]
where $T_{[\wp^{\oplus r}]}$
are the Hecke operators as 
defined in Section~\ref{sec:exam Hecke}.
\end{thm}
The $\wt{\kappa}$ above is a multiple of the $\kappa$ that 
appears in Theorem~\ref{thm:Drinfeld Euler}.
The statement in Theorem~\ref{thm:Drinfeld Euler} is sharper than
the one above.

\subsection{}
Statement (2) is what we call the Euler system relation.   The norm (pushforward)
of the element $\wt{\kappa}$ is sent to another $\wt{\kappa}$ multiplied
by a polynomial in Hecke operators.
The sum of Hecke operators is actually a local $L$-factor at the prime $\wp$.

\section{$Y$-sites}
\label{sec:intro Y-sites}
We introduced $Y$-sites in our paper \cite{Grids}.
This will serve as a tool in the proof of the theorem on Euler systems.
Let us describe the contents of this book 
concerning $Y$-sites in this section.
In Chapter~\ref{cha:Y-sites}, we develop some 
techniques concerning $Y$-sites.
This part contains no arithmetic, and may be read as a sequel to
\cite{Grids}, independently of other chapters.

Let us briefly recall the contents of \cite{Grids}.  
We ask the reader to read the introduction of loc.cit., 
especially Section 1.11.
A $Y$-site and a grid give rise to the associated absolute Galois monoid 
$M$, which is a topological monoid, and a fiber functor, i.e., a functor from the topos to the category of continuous $M$-sets.    Under some conditions (the examples treated in this book satisfy them), the fiber functor is an equivalence.     Notice that we focus on sites rather than toposes.    
The techniques introduced in this book will enable us to increase the number of objects of the underlying category of a Y-site.
We mention that something similar to $Y$-sites were considered by Caramello \cite{Caramello} and she also has considered 
these notions of enlarging the underlying categories of sites.

Let us illustrate the procedure using a (slightly artificial) example.
Let $F$ be a field.
Let $\cC$ be the category of finite Galois extensions of $F$.
We equip this category with the atomic topology $J$, and obtain a $Y$-site $(\cC, J)$.
There exists a grid for this $Y$-site and the associated absolute Galois monoid is the usual absolute Galois group $G_F$ of $F$.
The category of sheaves on this $Y$-site is equivalent to the category of continuous $G_F$-sets.
Now let us see the effect of the procedures 
of adding quotient objects (Section \ref{sec:add quotient}) 
and of adding finite coproducts (Section \ref{sec:add coproduct}).
The procedure of adding quotients applied to the category $\cC$
will give us a new category, say $\tilde{\cC}$, consisting of all separable extensions of $F$.
This is because any separable extension is a quotient of some Galois extension of $F$.
The procedure of adding finite coproducts to 
the category $\tilde{\cC}$
will give us a new category, 
denoted $\tilde{\cFC}$, whose objects are of the form
$\coprod_{i \in I} L_i$ where each $L_i$ is an object of $\tilde{\cC}$ and $I$ is a finite set.
We equip each of $\tilde{\cC}$ and $\tilde{\cFC}$ with a Grothendieck topology.   
Comparison Lemmas (Section \ref{sec:comparison lemmas}) 
imply that the associated toposes are equivalent to the starting one.    
Note that the starting category 
is not the usual \'etale site of the spectrum of $F$, but $\tilde{\cFC}$ is.
One merit of enlarging the category is that fiber products exist in
$\tilde{\cFC}$.    The following isomorphism is a key to Galois theory: 
$L \otimes_K L \cong \prod_{g \in G} L$ where 
$L/K$ is a Galois extension of Galois group $G$.   
We can consider this type of isomorphisms 
in the enlarged category.

We need the notion of norm (pushforward) for 
our application.
Assuming a finiteness 
condition on the underlying
category (e.g., the hom sets are finite), 
we define the notion of degree.    (In the example 
above, the degree is the degree of the field extension.)
Then we define presheaf with transfers on a $Y$-site.
(Transfer is another name for norm, pushforward.)
We see that a sheaf of abelian groups is equipped with 
the canonical structure of a presheaf with transfers.
It is also true that a presheaf with transfers with 
values in $\Q$-vector spaces is a sheaf.

We also construct what we call a compact induction functor.
We do not give the general definition but we illustrate this 
in the example above.   
Let $F \subset F_1 \subset F_2$ be 
field extensions (i.e., objects of $\tilde{\cC}$) and assume that all extensions are Galois.
Let $G_1, G_2$ be subgroups of $G_F$ corresponding to $F_1, F_2$ respectively.
Let $H=G_1/G_2$.
We obtain a functor from the category of $H$-sets to the category of $G_F$-sets by
inflation to $G_1$ and inducing up to $G_F$.    This is a typical example of the compact induction functor.
In our proof of the norm relation theorem, we wish to compute the norm with respect to groups such as 
$K=\Ker[\GL_2(\Z/N\Z)\to \GL_2(\Z/Np\Z)]$.   
We do not need to consider the whole group (e.g., $\GL_d(\A_F^\infty)$) 
for the computation.    We can study the problem not on the whole 
$Y$-site, but within 
the category of $K$-sets.   Then to translate the results in the category of $K$-sets, we use the compact induction functor.

Another aim is to give many examples of $Y$-sites.
We have given examples of $Y$-sites such that the associated
absolute Galois monoid is $\Z$ and $\N$ (the monoid of natural numbers)
in \cite{Grids}.  We study the category $\cC^d$ with the atomic topology
in detail in this paper.   The absolute Galois monoid is $\GL_d(\A_X)$ 
($\A_X$ is the ring of finite adeles; see Section~\ref{sec:adeles} for 
the precise definition).    
The first kind is for parabolic subgroups of the general linear groups.
The second kind is for classical similitude groups.
The third kind is for an arbitrary connected locally noetherian scheme over a base scheme.    
The absolute Galois monoid 
contains the \'etale fundamental group, but may not be profinite in 
general.   We give some example 
computations in the case of curves.
We formulate a conjecture concerning this group.
The fourth kind is from Riemannian symmetric spaces.   
We do not recover the (analytic) topology, but obtain something close.

\section{Euler systems on $Y$-sites}
\label{sec:intro Euler tech}
Having recalled 
$Y$-sites, we can give slightly more technical details on 
the method of proof of the Euler system relation theorem
(Theorem~\ref{main theorem}).

The setup is as follows.   
Let $d\ge 1$ be an integer.
We introduce a category $\cC^d$, which will be the 
underlying category of a $Y$-site.   
Let $R$ be a Dedekind domain with finite residue fields.
For our Drinfeld modular setup, we may take, for example,
$R=\F_q[T]$.
The objects of $\cC^d$ are torsion $R$-modules 
that are generated by at most $d$-elements.
The morphisms are $Q$-morphisms of Quillen.
We do not recall the definition here, but recall that
if $B$ is a subquotient of $A$, then it represents a 
morphisms from $A$ to $B$.   (We take arrows in 
the opposite direction of Quillen's.)
We consider the $Y$-site with atomic topology.

The following is actually simpler than the setup in the actual proof of the theorem.    (The theorem is more complicated since we need to consider the $(q_\infty^d-1)$-torsion.)    Here are the given data:
A presheaf $\BS'$ (called presheaf of distributions), a sheaf $F$, a morphism of presheaves $x\colon \BS' \to F$, a presheaf of rings with transfers $G$,  a morphism of presheaves $y\colon F \to G$.    

Let us take a nonzero ideal $I \subset R$ and set $N_I=(R/I)^d$
regarded as an object of $\cC^d$.  
We take $e_i=(0,\dots, 0, 1, 0, \dots, 0) \in N$ with $1$ in the $i$-th place for each $1 \le i \le d$.   As $\BS(N_I)=N_I$ as sets, we regard $e_i$ as elements of $\BS'(N_I)$.   Then we obtain elements $yx(e_1), \dots, yx(e_d) \in G(N_I)$, and the product $\kappa_I=\prod_{i=1}^d yx(e_i) \in G(N_I)$.   These are the elements which form an Euler system, as $I$ runs over the nonzero ideals.

The statement of the norm relation takes the following form.
Let $I$ be an ideal as above and let $\wp$ be a prime ideal prime to $I$.
Then we have canonical surjections $R/I \to R/I\wp$
and $N_I \to N_{I\wp}$.
We have a ($Q$-)morphism $m\colon N_{I\wp} \leftarrow N_I =N_I$ 
in $\cC^d$, where the arrow is the surjection. 
The transfer structure of $G$ gives a map (norm map) 
$m_*\colon G(N_{I\wp}) \to G(N_I)$.

The norm relation says that the norm image of something upstairs 
can be written as the local $L$-factor times something downstairs.
That is, $m_*\kappa_{I\wp}$ is equal to 
the local $L$-factor in Hecke operators times $\kappa_I$.

To describe the $L$-factor, we define Hecke operators as follows.
Let $N_i=(R/I)^d \oplus (R/\wp)^i$.
For each $i$, 
we have the canonical surjection 
$N_i \to N_I$ to the first factor,
 and the canonical inclusion 
$N_I \to N_i$  to the second factor. 
We obtain morphisms
$m_i\colon N_I \leftarrow N_i = N_i$
and 
$r_i\colon N_I \leftarrow N_I \subset N_i$
in $\cC^d$.
Then the $i$-th Hecke operator $T_i$ 
is defined to be $r_{i*} m_i^*$.

Then the $L$-factor is 
$\sum_{i=0}^d
(-1)^i (N\wp)^{i(i-1)/2} T_i
$.
This kind of $L$-factor is found in, for example, 
\cite[Thm. 3.21 Ch.\ 3]{Shimura}.
The norm relation reads
\[
m_* \kappa_{I\wp}
=
\sum_{i=0}^d
(-1)^i (N\wp)^{i(i-1)/2} T_i
\kappa_I.
\]

Let us illustrate this using the example of cyclotomic units.     
In this case $R=\Z$, $I=m\Z$ for some positive integer $m$, and 
$\wp=p\Z$ for a prime number $p$ prime to $m$.  In this case $d=1$.
We can take $F=G$ to be the sheaf such that $G(m)=F(m)=\Q(\zeta_m)^\times$
where $\zeta_m$ is an $m$-th root of unity.
The fact that this $F$ is a sheaf follows from that for any $m | m_1$,
we have 
$\Q(\zeta_m)=\Q(\zeta_{m_1})^G$ 
where $G$ is the Galois group of the extension $\Q(\zeta_{m_1})/\Q(\zeta_m)$.
(One can formulate a variant with the units in the 
ring of integers of $\Q(\zeta_m)$.)
We have elements $a/m \in \Z/m\Z=\BS'(I)$.
The map $x\colon \BS' \to F$ is given by sending 
$a/m$ to $1-\zeta_m^a$.    The fact that this is a 
map of presheaves follows from that these cyclotomic units 
satisfy distribution property.

Of the two Hecke operators, the Hecke operator $T_0$ is the identity map.
The computation of the Hecke operator $T_1$ is done 
in Section~\ref{sec:cyclotomic Euler}.
This comes to the action of the Frobenius, giving the usual $L$-factor.

In general, we have in our mind the moduli of ``something" (examples are $\Gm$, elliptic curves, and Drinfeld modules) 
with various level $N$ structures.
We set $F$ to be the presheaf of units (the (1,1)-motivic cohomology, $K_1$, etc.) 
of the moduli spaces, which turns out to be a sheaf.
For $G$, we are thinking of some $d$-th cohomology groups (the ($d,d$)-motivic cohomology, the $d$-th algebraic $K$-group $K_d$, etc.)



We remark that, when $G$ is a sheaf, the proof of the norm relation is considerably easier.  
This proof appears essentially in Grigorov \cite{Grigorov}, see also our paper \cite{Lepsilon}.
We also mention that this sheaf case has an application to
the computation of certain zeta integral \cite{KY:zeta}.


%

\section{Organization}
In Chapter~\ref{cha:Y-sites}, we give statements on $Y$-sites.
The prerequisite for this chapter is our paper~\cite{Grids}
and 
the chapter serves as a sequel to it.  The results of 
this section are used throughout the book.    

In Chapter~\ref{cha:Cd}, we prove the main norm relation 
theorem (Theorem~\ref{main theorem}) 
in an abstract setting of presheaves on $Y$-sites.   
The reader who is not interested in arithmetic applications
may skip this chapter.
We first introduce the category $\cC^d$ and 
give it a certain topology (for our main application,
we use the atomic topology).    Then we 
show that it is indeed a $Y$-site so that the results
of Chapter~\ref{cha:Y-sites} apply.   
Section~\ref{sec:Euler distributions} 
is devoted to the proof of the main theorem.
This realizes the slogan ``product of distributions
is an Euler system''.    The applications are given
in Chapter~\ref{cha:applications}.

In Chapter~\ref{cha:examples}, we provide more
examples of $Y$-sites.   This chapter depends on 
\cite{Grids} but logically on no other parts of this book.
We note that the example for classical groups 
(Section~\ref{sec:classical groups}) is a generalization 
of the category $\cC^d$ of Chapter~\ref{cha:Cd}.

Chapter~\ref{cha:applications} contains arithmetic applications of the main
norm relation theorem.   This chapter does not depend on 
Chapter~\ref{cha:examples}. In Section~\ref{sec:cyclotomic Euler},
we give an application to the norm computation of cyclotomic 
units.   The details are given to provide how our setup and 
this language of $Y$-sites act in practice, before going 
on to our main application for Drinfeld modular schemes.
In Section~\ref{sec:Hecke double}, we show how our 
definition of Hecke operators is compatible with the 
definition using double cosets.
Section~\ref{sec:universal Euler} contains the special case 
of our theorem when the presheaf with transfer is 
replaced by a sheaf.    This section explains how our
setup is related to the proof of norm relation by 
Colmez \cite{Colmez}.
Section~\ref{sec:Drinfeld Euler} contains our 
main theorem on the norm relation of 
elements in motivic cohomology of Drinfeld modular 
schemes.

\chapter*{Acknowledgment}
During this research, the first author was supported as a 
Twenty-First Century COE Kyoto Mathematics Fellow, was 
partially supported by JSPS Grant-in-Aid for Scientific
Research 17740016
and by World Premier International Research Center Initiative (WPI Initiative), MEXT, Japan.
The second author was partially supported by JSPS KAKENHI 
Grant Number 21H00969,15H03610, 24540018, 21540013, 16244120.
The second author would like to thank Akio Tamagawa
for his kind explanation of his unpublished result related to
Theorem 0.7 of \cite{Tama}.

\chapter{More generalities on $Y$-sites}
\label{cha:Y-sites}
In this chapter, we give some generalities on $Y$-sites.
We refer to Section~\ref{sec:intro Y-sites}
for the general introduction to 
the contents of this chapter.

\section{Comparison lemmas}
\label{sec:comparison lemmas}
Let $(\cC,J)$ and $(\cD,J')$ be sites and let $F\colon \cC \to \cD$ be a functor.
In \cite[EXPOSE III, Th\'eor\`eme 4.1]{SGA4} it is shown, under 
certain conditions, 
that the functor $\Presh(\cD) \to \Presh(\cC)$ given by the composition with $F$ induces an equivalence $\Shv(\cD,J') \to \Shv(\cC,J)$ of 
categories.
This result is called ``comparison lemma" in loc.\ cit., and some 
generalizations are known (see, for example, \cite[p.152]{KM}
or \cite[p.547, Thm 2.2.3]{Johnstone}).
In this paragraph we give two variants (Proposition \ref{prop:comp1} 
and Proposition \ref{prop:comp2} below) of comparison lemmas
which we will use in later sections.

\subsection{ }
We refer to \cite{SGA4} for set theory and 
the theory of sheaves.

Let $\frU$ be a (Grothendieck) universe which contains
an infinity.

\subsubsection{}
We define essentially $\frU$-small category below.
A set is called $\frU$-small if it 
is isomorphic to an element of $\frU$.
A category $\cC$ is said to be $\frU$-small
if the set of morphisms of $\cC$ is $\frU$-small
(\cite[TeX Exp I 1.0 footnote]{SGA4}.
A category $\cC$ is a $\frU$-category
if for any pair $x,y$ of objects of $\cC$,
the set $\Hom_\cC(x,y)$ is $\frU$-small.
\begin{defn}
A category $\cC$ is essentially $\frU$-small
if $\cC$ is a $\frU$-category 
and $\cC$ is equivalent to a $\frU$-small category.
\end{defn}
\subsubsection{}
\newcommand{\Usets}{($\frU$-$\mathrm{Ens}$)\,}
Let \Usets denote the category of sets that
belong to $\frU$.
Let $\cC$ be a category.
By a presheaf on $\cC$,
we mean a contravariant functor from $\cC$
to \Usets.   
We note that if $\cC$ is an essentially $\frU$-small category,
then $\Presh(\cC)$ is a $\frU$-category.

We let $\Presh(\cC)$ 
denote the category of 
presheaves on $\cC$.
For a $\frU$-category $\cC$, there is 
the fully faithful Yoneda embedding
\[
h_\cC\colon \cC \to \Presh(\cC)
\]
which sends an object $X\in \Obj(\cC)$
to the presheaf that $X$ represents.
For a presheaf $F$ on $\cC$ and 
an object $X$ of $\cC$, we let 
\[
y_{F,X}\colon \Hom_{\Presh(\cC)}(h_\cC(X), F) \to F(X)
\]
denote the isomorphism inverse to that 
given in \cite{SGA4}.

From here on, we assume there exist Grothendieck universes $\frU$ and $\frV$
such that $\frU$ contains an infinity and $\frU \in \frV$.   
Then the category of $\frU$-presheaves
on a $\frU$-category is a $\frV$-category.
We will write simply category, small category, essentially small category, etc. to
mean $\frU$-category, etc.

\subsection{Grothendieck topology}
\subsubsection{Sieves}
Let us recall the notion of a sieve (cf.\ 
\cite[EXPOSE I, D\'efinition 4.1]{SGA4}, 
\cite[Arcata, (6.1)]{SGA4h}).
Let $\cC$ be a category and 
let $X$ be an object of $\cC$.
A sieve on $X$ is a full subcategory $R$ of the overcategory
$\cC_{/X}$ satisfying the following condition: Let $f\colon Y \to X$
be an object of $\cC_{/X}$ and suppose that there exist
an object $g\colon Z \to X$ of $R$ and a morphism $h\colon Y \to Z$
in $\cC$ satisfying $f=g \circ h$. Then $f$ is an object of $R$.

For a sieve $R$ on $X$, we denote by $\frh_\cC(R)$ 
the following subpresheaf of $\frh_\cC(X)$: 
For each object $Y$ of $\cC$,
the subset $\frh_\cC(R)(Y) \subset \frh_\cC(X)(Y) 
\cong \Hom_\cC(Y,X)$ 
consists of the morphisms $f\colon Y \to X$ in $\cC$ 
such that $f$ is an object of $R$.
It follows from \cite[EXPOSE I, 4.2]{SGA4} that,
by associating $\frh_\cC(R)$ to $R$, we have a one-to-one
correspondence between the sieves on $X$ and the 
subpresheaves of $\frh_\cC(X)$.

\subsubsection{ }
Let $\cC$ and $\cD$ be categories,
let $X$ be an object of $\cC$, 
and let $Y$ be an object of $\cD$.
Suppose that a covariant functor $F\colon \cC_{/X}
\to \cD_{/Y}$ is given.
For a sieve $R$ on $Y$, we denote by $F^* R$
the full subcategory of $\cC_{/X}$ whose objects
are those objects $f\colon Z \to X$ of $\cC_{/X}$ 
such that $F(f)$ is an object of $R$.
It is then easy to check that $F^* R$ 
is a sieve on $X$.

Let $G\colon \cC \to \cD$ be a covariant functor.
Suppose that $G(X) = Y$ and that $F$ is equal to
the covariant functor $\cC_{/X} \to \cD_{/Y}$ 
induced by $G$. In this case we
denote the sieve $F^* R$ on $X$ by $G^* R$.

Let $f\colon X \to Z$ be a morphism in $\cC$. 
Suppose that $\cC=\cD$, $Y=Z$, and
$F$ is equal to the covariant functor 
$\cC_{/X} \to \cC_{/Z}$ 
which sends an object $g\colon W \to X$ of $\cC_{/X}$
to the object $f \circ g$ of $\cC_{/Z}$.
In this case we denote the sieve $F^* R$ 
on $Y$ by $R \times_Z X$ and call it the pullback
of $R$ with respect to the morphism $f$.

\subsubsection{Grothendieck topology}
\label{sec:topology}
Let $\cC$ be a category.
Let us recall the notion of Grothendieck topology 
(cf.\ 
\cite[EXPOSE II, D\'efinition 1.1]{SGA4}, 
\cite[Arcata, (6.2)]{SGA4h}).
A Grothendieck topology 
$J$ on $\cC$ is an assignment
of a set $J(X)$ of sieves on $X$ to each object $X$ of $\cC$
satisfying the following conditions:
\begin{enumerate}
\item For any object $X$ of $\cC$, the overcategory
$\cC_{/X}$ is an element of $J(X)$.
\item For any morphism $f\colon  Y \to X$ in $\cC$
and for any element $R$ of $J(X)$, the sieve
$R \times_X Y$ on $Y$ is an element of $J(Y)$.
\item Let $X$ be an object of $\cC$, and let $R$, $R'$ be 
two sieves on $X$. Suppose that $R$ is an element of $J(X)$
and that for any object $f\colon Y \to X$ of $R$, the sieve
$R' \times_X Y$ on $Y$ is an element of $J(Y)$.
Then $R'$ is an element of $J(X)$.
\end{enumerate}

Let $J$ be a Grothendieck topology on $\cC$ and
let $X$ be an object of $\cC$.
We say that a morphism $f\colon F \to \frh_\cC(X)$
of presheaves on $\cC$ is a covering of $X$ with respect to $J$ 
if the image of $f$ is equal to the subpresheaf 
$\frh_\cC(R)$ of $\frh_\cC(X)$ for some sieve 
$R$ on $X$ which belongs to $J(X)$.
We say that a morphism $f\colon F \to G$ of presheaves on $\cC$ 
is a covering with respect to $J$
if for any object
$X$ of $\cC$ and for any element $\xi \in G(X)$, the
first projection from the fiber product 
$\frh_\cC(X) \times_G F$ of the diagram
$$
\frh_\cC(X) \xto{y_{G,X}^{-1}(\xi)} G \xleftarrow{f} F
$$
to $\frh_\cC(X)$ is a covering of $X$ with respect to $J$.
When $G = \frh_\cC(X)$ for some object $X$ of $\cC$,
it follows from Condition (2) above
that $f$ is a covering with respect to $J$ if and only if
$f$ is a covering of $X$ with respect to $J$.
Let $(f_i\colon Y_i \to X)_{i \in I}$ be a family of 
objects of $\cC_{/X}$ indexed by a set $I$.
We say that $(f_i)_{i \in I}$ is a family covering $X$
with respect to $J$ if the sieve $R_{(f_i)_{i \in I}}$
on $X$ belongs to $J(X)$.

\subsection{Notation}

\subsubsection{ }\label{sec:IFX}
Let $\cC$ and $\cD$ be categories and
let $F\colon \cC \to \cD$ be a covariant functor.
For an object $X$ of $\cD$, we denote by
$I^F_X$ the following category.
%
The objects of $I^F_X$ are the pairs 
$(Y,f)$ of an object $Y$ of $\cC$ and 
a morphism $f\colon F(Y) \to X$ in $\cD$.
For two objects $(Y_1,f_1)$ and $(Y_2,f_2)$
of $I^F_X$, the morphisms from $(Y_1,f_1)$
to $(Y_2,f_2)$ in $I^F_X$ are the morphisms
$g\colon Y_1 \to Y_2$ in $\cC$ satisfying
$f_1 = f_2 \circ F(g)$.

\subsubsection{ }
For a morphism $f\colon Y \to X$ in a category $\cC$, we let $R_f$ denote the
full subcategory of $\cC_{/X}$ whose objects are the morphisms
$g\colon Z \to X$ in $\cC$ such that $g=f \circ h$ for some
morphism $h\colon Z \to Y$ in $\cC$. It is then easy to
check that $R_f$ is a sieve on $X$.

More generally, suppose that $X$ is an object of a category $\cC$ and
that a family $(f_i\colon Y_i \to X)_{i \in I}$ of objects of $\cC_{/X}$ 
indexed by a set $I$ is given.
We then let $R_{(f_i)_{i \in I}}$ denote the full subcategory of 
$\cC_{/X}$ whose objects are the morphisms 
$g \colon  Z \to X$ in $\cC$ such that $g = f_i \circ h$
for some $i\in I$ and for some morphism $h\colon  Z \to Y_i$
in $\cC$. It is then easy to
check that $R_{(f_i)_{i \in I}}$ is a sieve on $X$.

Let $X$ be an object of a category $\cC$.
We say that a sieve $R$ on $X$ has a small generator
if there exists a set $I$ and a family $(f_i\colon Y_i \to X)_{i \in I}$
of objects of $\cC_{/X}$ satisfying $R = R_{(f_i)_{i \in I}}$.

\subsubsection{ } \label{sec:Fshrink_sieve}
Let $\cC$ and $\cD$ be categories,
let $F\colon \cC \to \cD$ be a covariant functor, and
let $X$ be an object of $\cC$.
For a sieve $R$ on $X$, we denote by $F_! R$ the
full subcategory of $\cD_{/F(X)}$ whose objects are
the morphisms $f\colon Y \to F(X)$ in $\cD$ such that
$f=F(g)\circ h$ for some object $g\colon Z \to X$
in $R$ and for some morphism $h\colon Y \to F(Z)$ in $\cD$.
It is then easy to check that $F_! R$ is a sieve on $F(X)$.

\subsection{Grothendieck topologies and functors}
\label{sec:top_functor}

\subsubsection{}
Let $\cC$ and $\cD$ be categories, and let
$F\colon \cC \to \cD$ be a covariant functor.
For an object $X$ of $\cD$ and for an object
$(Y,f)$ of $I^F_X$, let $F_{Y,f}\colon  \cC_{/Y} \to \cD_{/X}$
denote the covariant functor which associates, to each
object $g\colon Z \to Y$ of $\cC_{/Y}$, the object 
$f \circ F(g)$ of $\cD_{/X}$.

\subsubsection{Continuous and cocontinuous functors}

Let $J$, $J'$ be Grothendieck topologies on $\cC$, $\cD$, respectively.
Following \cite[EXPOSE III, D\'efinition 1.1]{SGA4}, 
we say that $F$ is continuous
with respect to $J$ and $J'$ 
if for any sheaf $\cF$ on $D$, the presheaf $\cF \circ F^\op$ on
$\cC$ is a sheaf.
Following \cite[EXPOSE III, D\'efinition 2.1]{SGA4},
we say that $F$ is cocontinuous
 with respect to $J$ and $J'$
if for any object $X$ of $\cC$ and for any sieve $R$ on $F(X)$
which belongs to $J'(F(X))$, the sieve $F_{X,\id_{F(X)}}^* R$
on $X$ belongs to $J(X)$.

\subsubsection{Induced Grothendieck topology}
Let $J$ be a Grothendieck topology on the category $\cD$.
We let $F^* J$ denote the Grothendieck topology
induced on the category $\cC$ by $J$ with respect to the functor 
$F$ in the sense of \cite[EXPOSE III, 3.1]{SGA4}.
By definition, $F^* J$ is the finest Grothendieck topology on
$\cC$ such that $F$ is continuous with respect to $F^* J$ and $J$.

It follows from Proposition 5.1 in \cite[EXPOSE I]{SGA4}
that the pullback functor 
$F^* \colon  \Presh(\cD) \to \Presh(\cC)$
given by the composite with $F$ has a left adjoint,
which we denote by
$F_! \colon  \Presh(\cC) \to \Presh(\cD)$.
It follows from \cite[EXPOSE I, 5.2]{SGA4} that
the functor $F_!$ is left exact.
Let $X$ be an object of $\cC$.
It follows from \cite[EXPOSE I, 5.4(3)]{SGA4}
that we have $F_! \frh_\cC(X) = \frh_\cD(F(X))$.
If $R$ is a sieve on $X$, then an explicit construction
of the functor $F_!$, given in the proof of Proposition 5.1 
in \cite[EXPOSE I]{SGA4}, induces an isomorphism 
$F_! \frh_\cC(R) \cong \frh_\cD(F_!R)$,
where $F_! R$ is the sieve on $F(X)$ introduced in 
Section \ref{sec:Fshrink_sieve}, such that the diagram
$$
\begin{CD}
F_! \frh_\cC(R) @>>>
F_! \frh_\cC(X) \\
@V{\cong}VV @| \\
\frh_\cD(F_!R) @>>>
\frh_\cD(F(X))
\end{CD}
$$
is commutative.

\begin{lem}\label{lem:bicovering} 
Let $X$ be an object of $\cC$ and let $R$ be a sieve on $X$.
Then $R$ belongs to $F^*J (X)$ if and only if
the following two conditions hold for any morphism $f\colon Y\to X$ in $\cC$:
\begin{enumerate}
\item The sieve $F_!(R \times_X Y)$ on $F(Y)$ belongs to $J(F(Y))$.
\item Let $g_1\colon  Z_1 \to Y$ and $g_2\colon  Z_2 \to Y$ be two objects of 
$R \times_X Y$ and let
$$
\begin{CD}
W @>{h_1}>> F(Z_1) \\
@V{h_2}VV @VV{F(g_1)}V \\
F(Z_2) @>{F(g_2)}>> F(Y)
\end{CD}
$$
be a commutative diagram in $\cD$. Let $R'$ denote the full subcategory of 
$\cD_{/W}$ whose objects are the
morphisms $f'\colon V \to W$ in $\cD$ such that there exists an object $U$
of $\cC$, a morphism $f'' \colon  V \to F(U)$ in $\cD$, and a commutative
diagram
$$
\begin{CD}
U @>{h'_1}>> Z_1 \\
@V{h'_2}VV @VV{g_1}V \\
Z_2 @>{g_2}>> Y
\end{CD}
$$
in $\cC$ satisfying $h_1 \circ f' = F(h'_1) \circ f''$ and
$h_2 \circ f' = F(h'_2) \circ f''$.
Then $R'$ is a sieve on $W$ which belongs to $J(W)$.
\end{enumerate}
\end{lem}

\begin{proof}
We apply Proposition 3.2 of \cite[EXPOSE III]{SGA4}.
It follows that $R$ belongs to $J(X)$ if and only if
for any morphism $f\colon Y\to X$ in $\cC$, the morphism 
\begin{equation}\label{eq:F_shrink}
F_! \frh_\cC(R \times_X Y) \to 
F_! \frh_\cC(Y) = \frh_\cD(F(Y))
\end{equation}
induced by the inclusion $\frh_\cC(R \times_X Y)
\inj \frh_\cC(Y)$ is a bicovering in the sense of
\cite[EXPOSE II, 5.2]{SGA4}.

It follows from the explicit description of the functor 
$F_!$ given in (5.1.1) of \cite[EXPOSE I]{SGA4} that
the image of the morphism \eqref{eq:F_shrink} is equal
to the subpresheaf $\frh_{\cD}(F_!(R\times_X Y))$
of $\frh_{\cD}(F(Y))$.
Hence the morphism \eqref{eq:F_shrink} is a covering
of $F(Y)$ with respect to $J$ if and only if
Condition (1) holds.
%
It follows from the definition 
of bicovering that the morphism \eqref{eq:F_shrink}
is a bicovering if and only if it is a covering
of $F(Y)$ with respect to $J$ and the diagonal 
morphism
$$
F_!\frh_\cC(R\times_X Y)
\to F_!\frh_\cC(R\times_X Y)
\times_{\frh_\cD(Y)}
F_!\frh_\cC(R\times_X Y)
$$
induced by the morphism \eqref{eq:F_shrink}
is a covering with respect to $J$.
Since Condition (2) is a straightforward paraphrase
of the last condition, the claim follows.
\end{proof}

\begin{cor}\label{cor:bicovering}
Let $X$ be an object of $\cC$ and let $\frf_\cC(X)\colon \cC_{/X} \to \cC$
denote the functor which associates, 
to each object $Y \to X$ of $\cC_{/X}$, the object $Y$ of $\cC$.
Let $f\colon  Y \to X$ be a morphism and let $G\colon (\cC_{/X})_{/f}
\to \cC_{/Y}$ denote the functor induced by $\frf_\cC(X)$.
Note that $G$ is an isomorphism of categories.
Let $R$ be a sieve on $f$ and let $G(R)$ be a sieve on
$Y$ corresponding to $R$ via the isomorphism $G$ of categories.
Then $R$ belongs to $\frf_\cC(X)^* J(f)$ if and only if
$G(R)$ belongs to $J(Y)$.
\end{cor}

\begin{proof}
We apply Lemma \ref{lem:bicovering} when $\cC$, $\cD$, $F$, $X$ 
in loc.\ cit.\  are $\cC_{/X}$, $\cC$, $\frf_\cC(X)$, $f$, and $R$
respectively.
One can check that the full subcategory $R'$ of $\cD_{/W}$ 
in Condition (2) is equal to the entire $\cD_{/W}$.
Hence $R$ belongs to $\frf_\cC(X)^* J(f)$ if and only if
the sieve $\frf_\cC(X)_!(R \times_f f \circ g)$ on
$Z$ belongs to $J(Z)$ for any morphism $g\colon Z \to Y$ in $\cC$.
Since $\frf_\cC(X)_!(R \times_f f \circ g) = G(R) \times_Y Z$,
the claim follows.
\end{proof}

\subsection{Pushforwards of Grothendieck topologies}
\label{sec:pushforward}

\subsubsection{}
Let $(\cC, J)$ be a site and let $F\colon \cC \to \cD$ be 
a functor.   We know from 
\cite[Exp III Prop 3.7, p.287]{SGA4}
that there exists a finest Grothendieck topology
on $\cD$ such that the functor $F$ is 
cocontinuous.   We call this topology the 
pushforward of $J$ and denote it $F_*J$.
Let us give an explicit description of this 
pushforward topology.

\subsubsection{}\label{sec:J'}

Let $J$ be a Grothendieck topology on the category $\cC$.
For each object $X$ of $\cD$, we let $J'(X)$ denote the
set of sieves $R$ on $X$ satisfying the
following property: for any object $(Y,f)$ of $I^F_X$, 
the sieve $F_{Y,f}^* R$ belongs to $J(Y)$.

\begin{lem}
The assignment $J'$ of $J'(X)$ to each object $X$ of $\cD$
is a Grothendieck topology on the category $\cD$.
\end{lem}

\begin{proof}
We prove that $J'$ satisfies the three conditions in
Section \ref{sec:topology}.

Let $X$ be an object of $\cD$. For any object $Y$ of $\cC$ and for any 
morphism $f\colon F(Y) \to X$ in $\cD$, we have
$F_{Y,f}^* \cD_{/X} = \cC_{/Y}$.
By Condition (1) in
Section \ref{sec:topology} for $J$, we have $\cC_{/Y} \in J(Y)$.
Hence we have $\cD_{/X} \in F_*J(X)$.
Hence $J'$ satisfies Condition (1) in Section \ref{sec:topology}.

Let $f\colon Y \to X$ be a morphism in $\cD$ and let $R$ be a sieve on $X$
which belongs to $J'(X)$.
For any object $Z$ of $\cC$ and for any 
morphism $g\colon F(Z) \to Y$ in $\cD$, we have
$F_{Z,g}^* (R \times_X Y) = F_{Z,f \circ g}^* R$.
Since $R$ belongs to $J'(X)$, we have
$F_{Z,f \circ g}^* R \in J(Z)$.
Hence we have $R \times_X Y \in J'(Y)$.
This shows that 
$J'$ satisfies Condition (2) in Section \ref{sec:topology}.

Let $X$ be an object of $\cD$ and let $R$ be a sieve on $X$
which belongs to $J'(X)$.
Let $R'$ be another sieve on $X$ and suppose that $R' \times_X Y$
belongs to $J'(Y)$ for any object $Y \to X$ of $R$.
Let $Z$ be an object of $\cC$ and let $g\colon F(Z) \to X$ be a morphism in $\cD$.
For any morphism $h\colon W \to Z$ in $\cC$,
we have $(F_{Z,g}^* R')\times_Z W = F_{W,\id_{F(W)}}^* (R' \times_X F(W))$.
Here $R' \times_X F(W)$ denotes the pullback of $R'$
with respect to the composite $g \circ F(h)$.
Suppose that $h$ is an object of $F_{Z,g}^* R$.
Since $g \circ F(h)$ is an object of $R$,
the sieve $R' \times_X F(W)$ on $F(W)$ belongs to $J'(F(W))$.
Hence $(F_{Z,g}^* R')\times_Z W = F_{W,\id_{F(W)}}^* (R' \times_X F(W))$
belongs to $J(W)$.
By Condition (3) in
Section \ref{sec:topology} for $J$, we have $F_{Z,g}^* R' \in J(Z)$.
This shows that $R'$ belongs to $J'(X)$.
Hence $J'$ satisfies Condition (3) in Section \ref{sec:topology}.
This completes the proof.
\end{proof}

It is clear that $J'$ above is the finest topology on $\cD$
such that $F$ is cocontinuous.
Hence by definition, we have $F_*J=J'$.

We record the following transitivity property.
\begin{lem}\label{lem:pushforward_transitivity}
Let $\cE$ be another category and let $G\colon \cD \to \cE$
be a covariant functor. Then we have $G_*(F_*J)
=(G \circ F)_* J$.
\end{lem}

\begin{proof}
Let $X$ be an object of $\cE$ and let $R$ be a sieve on $X$.
Then $R$ belongs to $G_*(F_*J)(X)$ if and only if for any
object $Y$ of $\cD$, for any morphism $f\colon G(Y) \to X$ in $\cE$,
for any object $Z$ of $\cC$, and for any morphism $g\colon  F(Z) \to Y$ in $\cD$,
the full subcategory $R'$ of $\cC_{/Z}$, whose objects are the morphisms
$h\colon W \to Z$ such that $f \circ G(g \circ F(h))$ is an object of $R$,
is a sieve on $Z$ which belongs to $J(Z)$.
We set $f' = f \circ G(g)$.
When $(Y,Z,f,g)$ runs over the quadruples of
an object $Y$ of $\cD$, a morphism $f\colon G(Y) \to X$ in $\cE$,
an object $Z$ of $\cC$, and a morphism $g\colon  F(Z) \to Y$ in $\cD$,
the pair $(Z,f')$ runs through the pairs of an object $Z$ of
$\cC$ and a morphism $f'\colon G(F(Z)) \to X$ in $\cE$.
Hence $R$ belongs to $G_*(F_*J)(X)$ if and only if for any
object $Z$ of $\cC$ and for any morphism $f'\colon  G(F(Z)) \to X$ in $\cE$,
the full subcategory $R'$ of $\cC_{/Z}$, whose objects are the morphisms
$h\colon W \to Z$ such that $f' \circ G(F(h))$ is an object of $R$,
is a sieve on $Z$ which belongs to $J(Z)$.
This shows that $G_*(F_*J)(X)$ is equal to $(G \circ F)_* J(X)$,
which proves the claim.
\end{proof}

\begin{lem}\label{lem:induced}
Suppose that the functor $F$ is fully faithful.
Then we have $J = F^* F_* J$.
\end{lem}

\begin{proof}
We apply Lemma \ref{lem:bicovering}.
Let $X$ be an object of $\cC$ and let $R$ be a sieve on $X$.
Let us consider the two conditions for $R$ 
in Lemma \ref{lem:bicovering} with $J$ 
in loc.\ cit.\ replaced with $F_* J$.
We refer to these two conditions as Conditions (1)
and (2) for $F_* J$.
It suffices to show that $R$ belongs to $J(X)$ if and only if
Conditions (1) and (2) for $F_* J$
hold for any morphism $f\colon Y\to X$ in $\cC$.
Let $g_1\colon Z_1\to Y$, $g_2\colon Z_2 \to Y$,
$h_1\colon W \to F(Z_1)$ and $h_2\colon W \to F(Z_2)$
be as in Condition (2) in Lemma \ref{lem:bicovering}.
Let us consider the sieve $R'$ on $Y$
in Condition (2).
Let $V$ be an object of $\cC$ and let
$m\colon  F(V) \to W$ be a morphism in $\cD$.
Then it follows from the definition of the sieve
$R'$ that the sieve $F_{V,m}^* R'$ on $V$
is equal to the entire $\cC_{/V}$.
This shows that Condition (2) for $F_* J$ holds
for any $R$ and for any $f$.
Hence it suffices to show that $R$ belongs to $J(X)$ 
if and only if $F_!(R \times_X Y)$ belongs to
$F_* J(F(Y))$ for any morphism $f\colon Y\to X$ in $\cC$.

First we prove the ``if" part.
Suppose that for any morphism $f\colon Y\to X$ in $\cC$, 
the sieve $F_!(R \times_X Y)$ on $F(Y)$
belongs to $F_* J (F(Y))$.
This in particular implies that the sieve
$F_! R$ on $F(X)$
belongs to $F_* J (F(X))$.
Hence the sieve $F_{X,\id}^* F_! R$ on $X$
belongs to $J(X)$.
By definition the sieve $F_{X,\id}^* F_! R$
is the full subcategory of $\cC_{/X}$ whose objects are
the morphisms $f\colon Y \to X$ in $\cC$ such that
the morphism $F(Y) \to F(X)$ in $\cD$
given by the composition with $f$ is an object of 
$F_! R$.
Hence it follows from the definition of $F_! R$ that
the morphism $f\colon Y \to X$ in $\cC$ is an object of 
$F_{X,\id}^*F_! R$ if and only if 
there exist an object $g\colon Z \to X$ in $R$ and a
morphism $h\colon  F(Y) \to F(Z)$
in $\cD$ such that $F(f) = F(g) \circ h$.
Since $F$ is fully faithful, this shows that
the sieve $F_{X,\id}^* F_! R$ on $X$ is equal to $R$.
Hence we have $R \in J(X)$.

Next we prove the ``only if" part.
Suppose that $R$ belongs to $J(X)$.
We show that the sieve $F_!(R \times_X Y)$ on 
$F(Y)$ belongs to $F_* J (F(Y))$
for any morphism $f\colon Y\to X$ in $\cC$.
Since $R \times_X Y$ belongs to $J(Y)$, it suffices
to show that $F_! R$ belongs to $F_* J (F(X))$.
Let $Y$ be an object of $\cC$ and let $f\colon F(Y) \to F(X)$
be a morphism in $\cD$. We prove that $F_{Y,f}^* F_! R$
belongs to $J(Y)$.
Since $F$ is fully faithful, there exists
a morphism $f'\colon Y \to X$ in $\cC$ such that $f=F(f')$.
We claim that $F_{Y,f}^* F_! R$ is equal to
the pullback $R \times_X Y$ of $R$ with respect to $f'$.
Let $g\colon Z \to Y$ be a morphism in $\cC$.
Then $g$ is an object of $F_{Y,f}^* F_! R$
if and only if $f \circ F(g) = F(h_1) \circ h_2$
for some object $W$ of $\cC$, for some object
$h_1\colon  W \to X$ of $R$, and for some morphism $h_2\colon  F(Z) \to 
F(W)$ in $\cD$.
Since $F$ is fully faithful, this shows that
$g$ is an object of $F_{Y,f}^* F_! R$
if and only if $f' \circ g$ is an object of $R$.
Hence we have $F_{Y,f}^* F_! R = R \times_X Y$,
as we claimed.
This in particular implies that 
$F_{Y,f}^* F_! R$ belongs to $J(Y)$,
which completes the proof.
\end{proof}

\subsection{A full subcategory of the category of presheaves}
Let $\cC$ be an essentially small category.
\subsubsection{ }
Let $J$ be a Grothendieck topology on the category $\cC$.
Since $\cC$ is essentially small, there exists a set $G$
of objects of $\cC$ such that any object of $\cC$ is isomorphic
to an object of $\cC$ which belongs to $G$.
It is then clear that the set $G$ is a topologically
generating family of the site $(\cC,J)$ in the sense of
\cite[EXPOSE II, D\'efinition 3.0.1]{SGA4}.
It follows that $(\cC,J)$ is a $\frU$-site.


\subsubsection{}
Let $\cD$ be a full subcategory of $\Presh(\cC)$ 
satisfying the following condition: For any object $X$ of $\cC$, 
the representable presheaf $\frh_\cC(X)$ on $\cC$ is an object of $\cD$.    
We let $\iota_0\colon  \cC \to \cD$ denote the covariant functor
which sends an object $X$ of $\cC$ to the object 
$\frh_\cC(X)$ of $\cD$.

\begin{prop}\label{prop:comp1}
Let $J$ be a Grothendieck topology on the category $\cC$.
Then the functor $\Presh(\cD) \to \Presh(\cC)$ given by the
composition with $\iota_0$ induces an equivalence
$\Shv(\cD, (\iota_0)_* J) \xto{\cong} \Shv(\cC,J)$
of categories.
\end{prop}

\begin{proof}
Let $\cD'$ denote the essential image 
of the composite $\cD \to \Presh(C) \to \Presh(C')$.
Then the Yoneda embedding factors as
$\cC' \to \cD' \to \Presh(C')$.
Let $J_\cD'$ denote the topology 
induced by $(\iota_0)_* J$ on $\cD'$.
It follows from Lemma \ref{lem:induced} that
$J'$ is equal to the Grothendieck topology induced on $\cC'$
by $J_\cD'$.
We apply Th\'eor\`eme 4.1 in \cite[EXPOSE III]{SGA4}.

Hence it suffices to show that, for any object $F$ in $\cD$,
there exist a set $I$, a family $(X_i)_{i \in I}$ of
objects of $\cC$ indexed by $I$, and a morphism
$f_i\colon \iota_0(X_i) \to F$ in $\cD$ for each $i \in I$
such that $(f_i)_{i \in I}$ is a family covering
$F$ with respect to $(\iota_0)_* J$.
Let us take a set $S$ of objects of $\cC$ such that
any object of $\cC$ is isomorphic to some element in $S$.
Let $F$ be an object of $\cD$.
We let $I = \coprod_{X \in S} F(X)$ denote the set of
pairs $(X,\xi)$ of an element $X \in S$ and an element $\xi \in F(X)$.
For $i=(X,\xi) \in I$, we set $X_i = X$.
We let $f_i \colon  \iota_0(X) \to F$ denote the morphism which is sent
to $\xi$ under the bijection  $y_{F,X}$. 

We claim that $(f_i)_{i \in I}$ is a family covering
$F$ with respect to $(\iota_0)_* J$.
Proving this claim is equivalent to proving that
for any object $Y$ of $\cC$ and for any morphism
$g\colon \iota_0(Y) \to F$, the sieve $(\iota_0)_{Y,g}^* R_{(f_i)_{i \in I}}$
on $Y$ belongs to $J(Y)$.
It follows from the definition of the set $S$ that
there exists an element $X \in S$ and an isomorphism
$h\colon  X \xto{\cong} Y$ in $\cC$. Let $\xi \in F(X)$ denote
the image of $g \circ \iota_0(h)$ under the bijection
$y_{F,X}$. We set $i=(X,\xi)$.
We then have $g = f_i \circ \iota_0(h^{-1})$.
This shows that the identity morphism $\id_Y$ of $Y$
belongs to the sieve $(\iota_0)_{Y,g}^* R_{(f_i)_{i \in I}}$ on $Y$.
Hence the sieve $(\iota_0)_{Y,g}^* R_{(f_i)_{i \in I}}$ is
equal to $\cC_{/Y}$. In particular it belongs to $J(Y)$, as we claimed.
This completes the proof.
\end{proof}

\subsection{On the canonical topology}

\subsubsection{ }
Let $\cC$ be an essentially small category.
Let $J$ be a Grothendieck topology on the category $\cC$.
Since $\cC$ is essentially small, there exists a set $G$
of objects of $\cC$ such that any object of $\cC$ is isomorphic
to an object of $\cC$ which belongs to $G$.
It is then clear that the set $G$ is a topologically
generating family of the site $(\cC,J)$ in the sense of
\cite[EXPOSE II, D\'efinition 3.0.1]{SGA4}.
It follows that $(\cC,J)$ is a $\frU$-site.

It is proved in \cite[EXPOSE II, Th\'eor\`eme 3.4]{SGA4} that the 
inclusion functor $\Shv(\cC,J) \inj \Presh(\cC)$ has a left adjoint,
which we denote by $a_J \colon  \Presh(\cC) \to \Shv(\cC,J)$.
For a presheaf $F$ on $\cC$ and for a $J$-sheaf $G$ on $\cC$,
we denote by $b_{F,G}$ the natural bijection
$$
b_{F,G}\colon \Hom_{\Presh(\cC)}(F,G) \xto{\cong} \Hom_{\Shv(\cC,J)}(a_J(F),G).
$$

\subsubsection{ }
We let $\eps_{\cC,J}\colon \cC \to \Shv(\cC,J)$ denote the
functor which associates, to each object $X$ of $\cC$,
the $J$-sheaf $a_J \frh_\cC(X)$ on $\cC$.

Let $J'$ denote the canonical Grothendieck
topology on $\Shv(\cC,J)$ (see \cite[EXPOSE II, D\'efinition 2.5]{SGA4} 
for the definition).
It follows from Proposition 3.5 of \cite[EXPOSE III]{SGA4}
that the functor $\eps_{\cC,J}$ is continuous in the sense 
of \cite[EXPOSE III, D\'efinition 1.1]{SGA4}, with respect to the 
Grothendieck topologies $J$ and $J'$.

\begin{lem}\label{lem:univ_epi}
Let $F$ be an object of $\Shv(\cC,J)$.
Then for a sieve $R$ on $F$, the following
two conditions are equivalent:
\begin{enumerate}
\item $R$ belongs to $J'(F)$.
\item For any object $X$ of $\cC$ and
for any element $\xi \in F(X)$, there exists
a sieve $R' \in J(X)$ on $X$ satisfying the following
condition: for any object $f\colon Y \to X$ of $R'$,
the composite $\eps_{\cC,J}(Y) 
\xto{\eps_{\cC,J}(f)} \eps_{\cC,J}(X)
\xto{\wt{\xi}} F$ is an object of $R$.
Here $\wt{\xi}\colon  \eps_{\cC,J}(X) \to F$
is the morphism in $\Shv(\cC,J)$ corresponding
to $\xi$ via the bijections
$F(X) \xleftarrow{y_{F,X}} \Hom_{\Presh(\cC)}
(\frh_C(X),F) \cong \Hom_{\Shv(\cC,J)}
(\eps_{\cC,J}(X),F)$.
\end{enumerate}
\end{lem}

\begin{proof}
Let $\Sub(F)$ denote the set of
subpresheaves of $F$. 
Since $\cC$ is essentially $\frU$-small,
the set $\Sub(F)$ is a $\frU$-small set.
Let $\Sub(F)_R \subset \Sub(F)$ 
denote the subset whose elements
are the subsheaves $F' \subset F$ 
satisfying the following property:
$F'$ is equal to the image (in the category $\Presh(\cC)$) 
of a morphism $G \to F$ in $\Shv(\cC,J)$ 
which is an object of $R$.
Let $F_R \subset F$ denote the subpresheaf 
on $\cC$ which associates,
for each object $X$ of $\cC$, 
the subset $\bigcup_{F' \in \Sub(F)_R} F'(X)$
of $F(X)$.
It follows from 
\cite[EXPOSE IV, Corollaire 4.3.12]{SGA3}
that $R$ belongs to $J'(F)$ if and only if
the inclusion morphism $F_R \inj F$ in $\Presh(\cC)$
induces an isomorphism $a_J(F_R) \xto{\cong} F$ 
in $\Shv(\cC,J)$.
It is straightforward to check that the
latter condition is equivalent to Condition (2).
This proves the claim.
\end{proof}

\begin{prop}\label{prop:cocontinuous}
The functor $\eps_{\cC,J}$ is cocontinuous
with respect to the Grothendieck topologies $J$ and $J'$.
\end{prop}

\begin{proof}
Let $F$ be an object of $\Shv(\cC,J)$ and let $R$ be a sieve on
$F$ which belongs to $J'(F)$. Let $X$ be an object of $\cC$ and
let $f\colon \eps_{\cC,J}(X) \to F$ be a morphism in $\Shv(\cC,J)$.
Let us consider the functor
$(\eps_{\cC,J})_{X,f}\colon  \cC_{/X} \to \Shv(\cC,J)_{/F}$ 
introduced in Section \ref{sec:top_functor}.
It then suffices to show that the sieve $((\eps_{\cC,J})_{X,f})^* R$
on $X$ belongs to $J(X)$.

It follows from Lemma \ref{lem:univ_epi} that
there exists a sieve $R'$ on $X$ which belongs
to $J(X)$ such that for any object $g\colon Y \to X$ of $R'$, 
the composite $\eps_{\cC,J}(Y) \xto{\eps_{\cC,J}(g)}
\eps_{\cC,J}(X) \to F$ is an object of $R$.
This shows that the sieve $R'$ is a full subcategory of 
the sieve $((\eps_{\cC,J})_{X,f})^* R$. Hence it follows from
Condition (3) in Section \ref{sec:topology} that
$((\eps_{\cC,J})_{X,f})^* R$ belongs to $J(X)$.
This completes the proof.
\end{proof}

\begin{prop}\label{prop:canonical}
Let $J'$ denote the canonical Grothendieck
topology on $\Shv(\cC,J)$.
Then we have $J' = (\eps_{\cC,J})_* J$.
\end{prop}

\begin{proof}
Let $F$ be an object of $\Shv(\cC,J)$ and let $R$ be a sieve on $F$.
It follows from Proposition \ref{prop:cocontinuous} that
$R$ belongs to $(\eps_{\cC,J})_* J(F)$ if $R$ belongs to $J'(F)$.
We prove that $R$ belongs to $J'(F)$ 
if $R$ belongs to $(\eps_{\cC,J})_* J(F)$.
Suppose that $R$ belongs to $(\eps_{\cC,J})_* J(F)$.
We check that $F$ satisfies Condition (2) in
Lemma \ref{lem:univ_epi}.
Let $X$ be an object of $\cC$ and let $\xi \in F(X)$.
Let $\wt{\xi}\colon  \eps_{\cC,J}(X) \to F$ be
as in Condition (2) in Lemma \ref{lem:univ_epi}.
We set $R' = ((\eps_{\cC,J})_{X,\wt{\xi}})^* R$.
Since $R$ belongs to $(\eps_{\cC,J})_* J(F)$,
the sieve $R'$ on $X$ belongs to $J(X)$.
It follows from the definition of $R'$ that
for any object $f\colon Y \to X$ in $R'$,
the composite $\wt{\xi} \circ \eps_{\cC,J}(f)$
is an object in $R$.
This shows that Condition (2) 
in Lemma \ref{lem:univ_epi} is satisfied,
which proves the claim.
\end{proof}

\subsubsection{ }
Let $\cC'$ be a full subcategory of $\Shv(\cC,J)$ 
satisfying the following condition: for any object $X$ of $\cC$, 
the sheaf $a_J(\frh_\cC(X))$ on $\cC$ is an object of $\cC'$.
We let $\iota_0\colon  \cC \to \cC'$ denote the covariant functor
which sends an object $X$ of $\cC$ to the object 
$a_J(\frh_\cC(X))$ of $\cC'$.
We then have $\eps_{\cC,J} = \iota_{\cC'}
\circ \iota_0$, where 
$\iota_{\cC'} \colon  \cC' \inj \Shv(\cC,J)$
denotes the inclusion functor.

\begin{cor} \label{cor:4408}
Let the notation be as above.
Then we have $\iota_{\cC'}^* J' = (\iota_0)_* J$.
\end{cor}

\begin{proof}
By Proposition \ref{prop:canonical}
and Lemma \ref{lem:pushforward_transitivity},
we have $J' = (\eps_{\cC,J})_* J
= (\iota_{\cC'})_* (\iota_0)_* J$.
Hence by Lemma \ref{lem:induced}, we have 
$\iota_{\cC'}^* J' = (\iota_0)_* J$.
This proves the claim.
\end{proof}

\subsection{A full subcategory of the category of sheaves}
\subsubsection{}
Let $\cC$ be an essentially small category.
By definition, there exists a small category
$\cC'$ and a functor $F\colon \cC' \to \cC$
which is an equivalence of categories.
Let $\cC''$ denote the full subcategory of 
$\cC$ whose set of objects is $F(\Obj(\cC'))$.
Then $\cC''$ is small
and the inclusion $\cC'' \hookrightarrow \cC$
is an equivalence of categories.

\subsubsection{}
Let $J$ be a Grothendieck topology on $\cC$.
Let $J''$ denote the Grothendieck topology on
$\cC''$ induced by the inclusion.   Then the 
functor 
$\Shv(\cC, J) \to \Shv(\cC'', J'')$
induced by the restriction 
is an equivalence of categories.
It follows that $\Shv(\cC,J)$ is a $\frU$-topos.

\subsubsection{}
Let $\cD$ be a full subcategory of 
$\Shv(\cC,J)$, which is essentially small.
Assume that the following condition holds:
For any object $X$ of $\cC$, the sheaf
$a_J(h_\cC(X))$
is an object of $\cD$.

As $\cD$ is essentially small, 
we can find a small full subcategory $\cD''$
of $\cD$ such that the inclusion is an equivalence
of categories, in the manner similar to the case 
of $\cC$ above.

\subsubsection{}
Let $\cF$ be the full subcategory of $\Shv(\cC,J)$ whose 
set of objects is 
\[
\{a_J(h_{\cC}(X))\,|\, X \in \Obj(\cC'')\} 
\cup \Obj(\cD'').
\]
Then $\cF$ is small.
Let $\iota_\cF\colon  \cF \hookrightarrow \Shv(\cC,J)$
denote the inclusion. 

\subsubsection{}
We note here that by \cite[Prop.4.10, Exp II]{SGA4},
the set $\Obj(\cC'')$ is a set of 
topological generators for the canonical topology on 
$\Shv(\cC,J)$.

\subsubsection{}
Let $J_{can}$ denote the canonical topology
on $\Shv(\cC,J)$.
\begin{lem}
Let the notation be as above.  The functor
\[
\iota_\cF\colon  \Shv(\cF, \iota_\cF^* J_{can}) \to \Shv(\cC, J)
\]
induced by the inclusion 
is an equivalence of categories.
\end{lem}
\begin{proof}
This follows from \cite[Cor 1.2.1 Exp IV]{SGA4}
\end{proof}

\begin{prop}\label{prop:comp2}
Let $\iota_\cD\colon  \cD \to \Shv(\cC, J)$ denote the inclusion.
The functor $\Presh(\cD) \to \Presh(\cC)$ given by the
composition with $a_J(h_\cC(-))\colon \cC \to \cD$ induces an equivalence
$\Shv(\cD, \iota_{\cD}^* J_{can}) 
\xto{\cong} \Shv(\cC,J)$
of categories.
\end{prop}

\begin{proof}
By construction, the inclusion 
$\cF \to \cD$ is an equivalence of categories.
Hence we have 
$\Shv(\cD, \iota_{\cD}^* J_{can})
\cong
\Shv(\cF, \iota_{\cF}^*J_{can}) .
$
The claim follows from the previous lemma.
\end{proof}
Let $\iota$ denote the functor
$a_J(h_\cC(-))\colon \cC \to \cD$.
\begin{cor} \label{cor:5563}
The functor $\Presh(\cD) \to \Presh(\cC)$ given by the
composition with $\iota$ induces an equivalence
\[
\Shv(\cD, \iota_* J) \to \Shv(\cC, J)
\]
of categories.
\end{cor}
\begin{proof}
This follows from the previous proposition and 
Corollary~\ref{cor:4408}.
\end{proof}

\section{Adding quotient objects}
\label{sec:add quotient}
We briefly explained the motivation for adding quotient objects 
in Section~\ref{sec:intro Y-sites} with 
the example in the case where $(\cC, J)$ is the $Y$-site coming from the classical Galois theory.
The idea is, if one is given a Galois covering $Y \to X$ of Galois group $G$, to add all intermediary coverings $Z \to X$
where $Z=Y/H$ for subgroups $H$ of $G$.

Let a $Y$-site $(\cC, J)$ be given.   We have a canonical map $\cC \to \Shv(\cC, J)$ given by the Yoneda embedding and the sheafification functor.    Note that in $\Shv(\cC, J)$, a quotient object of any element exists.   We wish to take a suitable full subcategory $\cD \subset \Shv(\cC, J)$ so that it contains the (suitable, may not be all) quotient objects of objects coming from $\cC$.
We have a functor $\iota\colon  \cC \to \cD$ and endow $\cD$ 
with the pushforward topology.    
We show that the associated toposes are equivalent. 
Moreover, when $(\cC,J)$ is a $B$-site (\resp $Y$-site),
we show that $(\wt{\cC}, \iota_* J)$ is a $B$-site
(\resp $Y$-site).

We do not add all quotient objects but restrict ourselves to objects of 
the following form (called admissible).   Suppose given
a morphism $Y \to X$ which is a Galois covering of Galois group $G$.    
Then we add quotients $Y/H$ (more precisely, $Y$ is regarded via the 
functor above as a sheaf on $\cC$ and take the quotient sheaf).   We 
do not allow objects of the form 
$Y/H$ for arbitrary subgroup $H$ of $\Aut(Y)$.    There is a technical 
point  
that works fine with the restriction, and which may or may not work for 
general quotients.

\subsection{Preliminary on quotient objects}
We would like to apply the two comparison lemmas 
(Proposition \ref{prop:comp1} and Proposition \ref{prop:comp2}) 
to the case when $(\cC,J)$ is a $Y$-site.
As a preliminary, we recall in this paragraph the notion 
of a \quotobj in a general category and prove some basic facts 
on quotient objects.

\subsubsection{ }
First let us recall the notion of quotient object.
Let $\cC$ be a category, $Y$ an object in $\cC$,
and $G$ a subgroup of $\Aut_{\cC}(Y)$.
A \quotobj of $Y$ by $G$ is an object
in $\cC$, which we denote by $\quot{Y}{G}$, equipped with
a morphism $c\colon Y \to \quot{Y}{G}$ in $\cC$ satisfying the
following universal property: For any object $Z$
in $\cC$ and for any morphism
$f \colon  Y \to Z$ in $\cC$ satisfying $f\circ g=f$ for
all $g \in G$, there exists a unique morphism
$\overline{f} \colon  \quot{Y}{G} \to Z$ such that $f=\overline{f}\circ c$.
In other words, the \quotobj $\quot{Y}{G}$ is an
object in $\cC$ which co-represents the covariant functor
from $\cC$ to the category of sets which associates,
with each object $Z \in \cC$, the $G$-invariant part
$\Hom_{\cC}(Y,Z)^G$ of the set $\Hom_{\cC}(Y,Z)$.
We call the morphism $c\colon Y \to \quot{Y}{G}$ 
the canonical quotient morphism.

A \quotobj of $Y$ by $G$ is unique up to
unique isomorphism in the following sense. Suppose that
both $Y'_1$ and $Y'_2$ are \quotobjs of $Y$ by $G$.
We denote by $c_1\colon Y \to Y'_1$ and $c_2\colon Y\to Y'_2$ the
canonical quotient morphisms. Then there exists a unique
isomorphism $\alpha \colon  Y'_1 \xto{\cong} Y'_2$ satisfying
$\alpha \circ c_1 = c_2$.
This claim follows easily from the universality of
\quotobjs.

\begin{lem}\label{lem:quot_fullsub}
Let $\cC$ be a category and 
let $\cC' \subset \cC$ be a full subcategory.
Let $Y$ be an object in $\cC'$,
and $G$ be a subgroup of $\Aut_{\cC}(Y)$.
Suppose that a \quotobj $\quot{Y}{G}$
of $Y$ by $G$ in $\cC$ exists and that
$\quot{Y}{G}$ is isomorphic in $\cC$ to an object $Z$
in $\cC'$. Then $Z$ is a \quotobj
of $Y$ by $G$ in $\cC'$.
\end{lem}

\begin{proof}
The universality of $Z$ can be checked easily.
\end{proof}

\begin{lem}\label{lem:quot_presheaf}
Let $\cC$ be an essentially small category.
Let $F$ be an object in $\Presh(\cC)$ and let $G \subset \Aut_{\Presh(\cC)}(F)$
be a subgroup. We define an object $F'$ in $\Presh(\cC)$ by setting
$F'(Y)=\quot{F(Y)}{G}$ for each object $Y$ in $\cC$.
Let $c\colon F\to F'$ be the morphism which consists of the quotient
maps $F(Y) \surj \quot{F(Y)}{G} =F'(Y)$ for all objects $Y$ in $\cC$.
Then the presheaf $F'$, together with the morphism $c\colon F \to F'$
is a \quotobj of $F$ by $G$ in the category $\Presh(\cC)$.
\end{lem}

\begin{proof}
This is a formal consequence of the fact that a small colimit
in the category $\Presh(\cC)$ can be taken sectionwisely.
\end{proof}

\begin{lem}\label{lem:quot_adjunction}
Let $\cC$, $\cD$ be categories and let $L\colon \cC \to \cD$ be
a functor which admits a right adjoint $R\colon \cD \to \cC$.
Let $Y$ be an object in $\cC$ and $H\subset \Aut_{\cC}(Y)$ be a
subgroup. Suppose that there exists a \quotobj
$\quot{Y}{G}$ of $Y$ by $G$ in $\cC$. Let $c\colon Y \to \quot{Y}{G}$
denote the canonical quotient morphism.
Let $G' \subset \Aut_{\cD}(L(Y))$ denote the image of $G$
under the homomorphism $\Aut_{\cC}(Y) \to \Aut_{\cD}(L(Y))$
induced by $L$.
Then $L(\quot{Y}{G})$ together with $L(c)\colon L(Y)\to L(\quot{Y}{G})$
is the \quotobj of $L(Y)$ by $G'$ in the
category $\cD$.
\end{lem}

\begin{proof}
It suffices to show that, for any object $Z \in \cD$, the map
$$
\Hom_{\cD}(L(\quot{Y}{G}),Z) \to \Hom_{\cD}(L(Y),Z)
$$
given by the composite with $L(c)$ is injective and its
image is equal to the $G'$-invariant part of 
$\Hom_{\cD}(L(Y),Z)$. 
Let $\iota_1\colon  \Hom_{\cD}(L(\quot{Y}{G}),Z)\xto{\cong} 
\Hom_{\cC}(\quot{Y}{G},R(Z))$
and $\iota_2\colon  \Hom_{\cD}(L(Y),Z) \xto{\cong} \Hom_{\cC}(Y,R(Z))$
denote the bijections given by the adjunction.
It follows from the functoriality of the adjunction
that the diagram
$$
\begin{CD}
\Hom_{\cD}(L(\quot{Y}{G}),Z) @>{\iota_1}>{\cong}> 
\Hom_{\cC}(\quot{Y}{G},R(Z)) \\
@V{-\circ L(c)}VV @V{-\circ c}VV \\
\Hom_{\cD}(L(Y),Z) @>{\iota_2}>{\cong}> \Hom_{\cC}(Y,R(Z))
\end{CD}
$$
is commutative and the isomorphism $\iota_2$ is $G$-equivariant
where $G$ acts on $\Hom_{\cD}(L(Y),Z)$ via the surjection $G \to G'$
induced by $L$. Since $\quot{Y}{G}$ is a \quotobj of $Y$ by $G$,
the right vertical map is injective and its image is
equal to the $G$-invariant part of $\Hom_{\cC}(Y,R(Z))$.
Hence the claim follows.
\end{proof}

\subsubsection{ }
Let $Y$ be an object in a category $\cC$ and let $G$ be a subgroup of
$\Aut_{\cC}(Y)$. Let $H$ be a normal subgroup of $G$ and
suppose that a \quotobj $\quot{Y}{H}$ of $Y$ by $H$
exists. Let $c\colon Y \to \quot{Y}{H}$ denote the canonical quotient morphism.
Let $g \in G$. Since $H$ is a normal subgroup of $G$, 
the composite $c \circ g \circ h$ is equal to $c \circ g$
for any $h \in H$. Hence it follows from the universality of
the \quotobj that there exists a unique morphism
$\alpha_g \colon  \quot{Y}{H} \to \quot{Y}{H}$ 
such that $c \circ g = \alpha_g \circ c$.
It is then easy to check that $\alpha_{g_1 \circ g_2} 
= \alpha_{g_1} \circ \alpha_{g_2}$ for $g_1,g_2 \in G$ 
and that $\alpha_{h} = \id_{\quot{Y}{H}}$ for $h \in H$.
Hence the map $g \mapsto \alpha_g$ gives a group homomorphism
$G \to \Aut_{\cC}(\quot{Y}{H})$ which factors through the quotient map 
$G \surj G/H$. Let $\overline{G} \subset \Aut_{\cC}(\quot{Y}{H})$
denote the image of this homomorphism.

\begin{lem}\label{lem:quot_general}
Under the notation as above, suppose that a
\quotobj $\quot{(\quot{Y}{H})}{\overline{G}}$ of 
$\quot{Y}{H}$ by $\overline{G}$
exists. Let $c'\colon \quot{Y}{H} \to \quot{(\quot{Y}{H})}{\overline{G}}$ denote the
canonical quotient map. Then $\quot{(\quot{Y}{H})}{\overline{G}}$ together with
the composite $c' \circ c \colon  Y \to \quot{(\quot{Y}{H})}{\overline{G}}$
is a \quotobj of $Y$ by $G$.
\end{lem}

\begin{proof}
It suffices to check that, for any object $Z$ in $\cC$,
the composite
$\Hom_{\cC}(\quot{(\quot{Y}{H})}{\overline{G}},Z) 
\xto{-\circ c'}
\Hom_{\cC}((\quot{Y}{H}),Z) \xto{-\circ c}
\Hom_{\cC}(Y,Z)$
is injective and its image is equal to the $G$-invariant
part of $\Hom_{\cC}(Y,Z)$.
We know that the first map $-\circ c'$ is injective
and its image is equal to the $\overline{G}$-invariant
part of $\Hom_{\cC}((\quot{Y}{H}),Z)$, and that
the second map $-\circ c$ is injective
and its image is equal to the $H$-invariant
part of $\Hom_{\cC}((\quot{Y}{H}),Z)$.
Hence it suffices to prove that the second map
$-\circ c$ is $G$-equivariant, where $G$ acts on
$\Hom_{\cC}((\quot{Y}{H}),Z)$ via the surjection $G \surj \overline{G}$.
The last claim follows from the equality 
$c \circ g = \alpha_g \circ c$. This completes the proof.
\end{proof}

\subsection{Adding quotient objects}
Let $(\cC, J)$ be a $\frU$-site. We assume that
$J$ is an $A$-topology in the sense of \cite[2.4]{Grids}.

\begin{defn}
We say that a subgroup $H \subset \Aut_{\cC}(X)$
is moderate if there exists a morphism
$f\colon X \to Y$ in $\cC$ such that $f$ belongs to
$\cT(J)$ and that $H$ is contained in $\Aut_Y(X)$.
\end{defn}

\subsubsection{}\label{sec:Ctil}
Let $X$ be an object in $\cC$ and
$H$ be a subgroup of $\Aut_{\cC}(X)$.
We denote by $\quot{X}{H}$ 
the sheaf associated with the presheaf
$\quot{\frh_{\cC}(X)}{H}$.

Let $\wt{\cC^m}$ 
denote the full subcategory of the category 
$\Shv(\cC,J)$ of sheaves on $(\cC,J)$ 
whose objects are sheaves of the form $\quot{X}{H}$ 
with $X$ in $\cC$
and $H$ a moderate subgroup of $\Aut_{\cC}(X)$.
We note that, for two pairs $(X_1,H_1)$ and $(X_2,H_2)$
such that $X_i$ is an object of $\cC$ 
and $H_i \subset \Aut_{\cC}(X_i)$ is a moderate subgroup
for $i=1,2$, the equality $\quot{X_1}{H_1}
= \quot{X_2}{H_2}$ does not imply that
$X_1$ and $X_2$ are isomorphic.
To avoid such a problem, we will sometimes use 
the following category $\wt{\cC^m}^\str$ instead of $\wt{\cC^m}$.
The objects of $\wt{\cC^m}^\str$ is a pair
$(X,H)$ of an object $X$ of $\cC$ 
and a moderate subgroup $H \subset \Aut_{\cC}(X)$.
For two objects $(X_1,H_1)$ and $(X_2,H_2)$ of $\wt{\cC^m}^\str$,
a morphism from $(X_1,H_1)$ to $(X_2,H_2)$ in $\wt{\cC^m}^\str$ is
a morphism $\quot{X_1}{H_1} \to \quot{X_2}{H_2}$ of
sheaves on $(\cC,J)$.
Let $\imath^{m\str}\colon  \cC \to \wt{\cC^m}^\str$ denote
the functor that associates, with each object $X$ of $\cC$,
the object $(X,\{\id_X\})$ of $\wt{\cC^m}^\str$.
For an object $(X,H)$ of $\wt{\cC^m}^\str$, we let
$q_{X,H}\colon  \imath^{m\str}(X) \to (X,H)$ denote the
morphism in $\wt{\cC^m}^\str$ induced by the
quotient morphism $\frh_\cC(X) \to \quot{\frh_\cC(X)}{H}$.
It is clear that the functor $\wt{\cC^m}^\str \to \wt{\cC^m}$
that associates, with each object $(X,H)$ of $\wt{\cC^m}^\str$,
the object $\quot{X}{H}$ is an equivalence of categories.
We let $\imath^m\colon \cC \to \wt{\cC}$ denote the
composite $\cC \xto{\imath^{m\str}} \wt{\cC^m}^\str \to \wt{\cC^m}$.
We sometimes regard $q_{X,H}$ as a morphism
$\imath^m(X) \to \quot{X}{H}$ in $\wt{\cC^m}^\str$.

We remark that if $\cC$ is an essentially small
category, then so are $\wt{\cC^m}^\str$ and $\wt{\cC^m}$.

\subsubsection{}
We also consider full subcategories $\wt{\cC}$ of 
$\wt{\cC^m}$,
that contain objects (isomorphic to 
objects of the form) $\quot{X}{\{\id_X\}}$
for all $X \in \cC$.
The category $\wt{\cC}^\str$ is defined in a similar manner
as above.  
We impose that if $(Z, H_Z) \cong (X, H)$
and $(X, H)$ belongs to $\wt{\cC}^\str$,
then $(Z, H_Z) \in \wt{\cC}^\str$.
We have functors
$\imath^\str\colon \cC \to \wt{\cC}^\str$
and
$\imath\colon \cC \to \wt{\cC}$.

\begin{lem}\label{lem:quot_FCd}
Let $X$ be an object in $\cC$ and let $H$ be a subgroup of
$\Aut(X)$. 
Then the sheaf $\quot{X}{H}$ on $(\cC,J)$
together with the morphism $\imath(X) \to \quot{X}{H}$
is a \quotobj of $\quotid{X}$ by $H$ in the category
$\Shv(\cC,J)$.
In particular the object $\quot{X}{H}$ of $\wt{\cC}$
together with the morphism $\imath(X) \to \quot{X}{H}$
is a \quotobj of $\imath(X)$ by $H$.
\end{lem}

\begin{proof}
It follows from Lemma \ref{lem:quot_presheaf} that
the presheaf $\quot{\frh_\cC(X)}{H}$ on $\cC$ together with the
surjection $\frh_\cC(X) \surj \quot{\frh_\cC(X)}{H}$
is a \quotobj of the presheaf $\frh_\cC(X)$ by $H$.
It follows from Lemma \ref{lem:quot_adjunction} that
$\quot{X}{H}$ together with the
morphism $\quotid{X} \to \quot{X}{H}$ is a \quotobj
of $\quotid{X}$ by $H$
in the category $\Shv(\cC,J)$. 
Hence by Lemma \ref{lem:quot_fullsub},
it is also a \quotobj of $\quotid{X}$ by $H$
in the category $\wt{\cC}$.
This proves the claim.
\end{proof}

\subsection{The semi-localizing collection $\wt{\cT}^\str$}
Let $(\cC, J)$ be a $\frU$-site. 
Let $\wt{\cC}^\str$ be as in Section \ref{sec:Ctil}.
We construct a semi-localizing collection 
$\wt{\cT}^\str$ of morphisms in $\wt{\cC}^\str$.
We show that when $(\cC, J)$ is a $B$-site, 
so is $(\wt{\cC}^\str, J_{\wt{\cT}^\str})$.
We will show later that this topology is
equal to the pushforward topology.
We assume that the topology 
$J$ is an $A$-topology 
(\cite[Definition 2.4.1]{Grids}).

\begin{defn}
Let $f\colon (X,H) \to (Y,K)$ be a morphism in $\wt{\cC}^\str$. 
A representative of $f$ in $\cC$ is a diagram 
$X \xleftarrow{m} X' \xto{f'} Y$ in $\cC$ 
such that $m$ belongs to $\cT(J)$ and that 
the diagram
\begin{equation} \label{diagram:pentagon}
\xymatrix{
\imath^\str (X') \ar[dr]^{\imath^\str(f')} 
\ar[d]_{\imath^\str(m)} & \\
\imath^\str(X) \ar[d]_{q_{X,H}} 
& \imath^\str(Y) \ar[d]_{q_{Y,K}} \\
(X,H) \ar[r]_{f} 
& (Y,K)
}
\end{equation}
in $\wt{\cC^\str}$ is commutative.
We use the notation $(f; m, f')$.
\end{defn}

\begin{lem} \label{lem:pentagon}
Suppose $(\cC, J)$ is a site with $A$-topology.
Let $f\colon (X,H) \to (Y,K)$ be
a morphism in $\wt{\cC}^\str$. Then there exists
a representative of $f$ in $\cC$.
\end{lem}

\begin{proof}
Let us consider the isomorphisms
given by adjunctions:
\[
\begin{array}{l}
\Hom_{\wt{\cC}^\str}(\imath^\str(X),
(Y,K))
= 
\Hom_{\Shv(\cC,J)}(a_J(\frh_\cC(X)), \quot{Y}{K})
\\
\cong
\Hom_{\Presh(\cC)}(\frh_\cC(X), \quot{Y}{K})
\cong (\quot{Y}{K})(X).
\end{array}
\]
Note that the sieves $R_f$ where $f$ is a 
morphism in $\cT(J)$ with target $X$ 
are cofinal in $J(X)$.
From the explicit description in \cite[4.4.2]{Grids} of the sheafification
functor $a_J$, we further have, using 
\cite[Lem 2.5.2]{Grids},
\[
\begin{array}{l}
(\quot{Y}{K})(X)
=(a_J(\quot{h_\cC(Y)}{K})(X)
\\
=\varinjlim_{(m\colon X' \to X)\in \cT(J)}
\quot{
[\Hom_\cC(X',Y) \rightrightarrows
\Hom_{\Presh(\cC)}
(h_\cC(X') 
\times_{h_\cC(X)}
h_\cC(X')
,Y)]}{K}.
\end{array}
\]
Let $a \in (\quot{Y}{K})(X)$ denote the element
corresponding to the composite
$\quot{X}{\{\id_{X}\}} \to \quot{X}{H}
\xto{f} \quot{Y}{K}$.
Take any pair $(m\colon X' \to X) \in \cT(J)$
and $f' \in \Hom_\cC(X',Y)$ which corresponds
to $a$ in the equality above.
Then the diagram $X \xleftarrow{m} X' \xto{f'} Y$
is a representative of $f$ in $\cC$.
This completes the proof.
\end{proof}

\subsubsection{ }\label{sec:cTmu}
Let $\wt{\cT}^\str$ denote the following set of morphisms
in $\wt{\cC}^\str$.   
Let $f\colon  X \to Y$ be a morphism
in $\wt{\cC}^\str$.  
Then $f$ belongs to $\wt{\cT}^\str$
if and only if there exists a representative $(f; m,f')$
with $f' \in \cT(J)$.

\subsubsection{ }
In this paragraph we assume that $(\cC,J)$ is a $B$-site.
\begin{lem} \label{lem:omega_faithful1}
The functor $\cC \to \Shv(\cC,J)$ that
associates an object $X$ of $\cC$ with
$a_J(\frh_\cC(X))$ is faithful.
\end{lem}
\begin{proof}
Let $X$ and $Y$ be two objects of $\cC$.
We prove that the map
$\Hom_\cC(X,Y) \to \Hom_{\Shv(\cC,J)}
(a_J(\frh_\cC(X)),a_J(\frh_\cC(Y)))$ is injective.
Since we have isomorphisms
\[
\begin{array}{l}
\Hom_{\Shv(\cC,J)}
(a_J(\frh_\cC(X)),a_J(\frh_\cC(Y)))
\\
\cong 
\Hom_{\Presh(\cC)}
(\frh_\cC(X),a_J(\frh_\cC(Y)))
\cong a_J(\frh_\cC(Y))(X),
\end{array}
\]
It suffices to show that the map
$(\frh_\cC(Y))(X) \to a_J(\frh_\cC(Y))(X)$
induced by the adjunction morphism
$\frh_\cC(Y) \to a_J(\frh_\cC(Y))$ is injective.
It follows from an explicit
description of the sheafification 
functor given in \cite[4.4.2]{Grids}
we have, using \cite[Lem 2.5.2]{Grids},
\begin{equation} \label{eq:aJ1}
a_J(\frh_\cC(Y))(X) \\
\cong 
\varinjlim_{(X',f)} 
\quot{[\Hom_\cC(X',Y)
\rightrightarrows
\Hom_{\Presh(\cC)}
(h_\cC(X') 
\times_{h_\cC(X)}
h_\cC(X'),
Y)]}
{K}
\end{equation}
where $(X',f)$ runs over the morphisms
in $\cT(J)$ with target $X$.
(The set of morphisms may not be 
$\frU$-small, but the set of 
isomorphism classes of such morphisms
is $\frU$-small.  Hence, the 
right-hand side above is $\frU$-small.)
Since $\cC$ is an $E$-category, it follows that the
transition maps in the colimit \eqref{eq:aJ1} are injective.
This proves that
$(\frh_\cC(Y))(X) \to a_J(\frh_\cC(Y))(X)$ is injective.
\end{proof}

\begin{lem} \label{lem:exists_g}
Let $f\colon  (X,H) \to (Y,K)$ be a morphism
in $\wt{\cC}^\str$. Let 
$X \xleftarrow{m} X' \xto{f'_1} Y$
and $X \xleftarrow{m} X' \xto{f'_2} Y$
be two representatives of $f$ with common 
$X'$ and $m$. Then there exists an element
$g \in K$ satisfying $f'_2 = g \circ f'_1$.
\end{lem}

\begin{proof}
We have $q_{Y,K} \circ\imath^\str(f'_1) 
= q_{Y,K} \circ \imath^\str(f'_2)$.
The two morphisms $q_{Y,K} \circ\imath(f'_1)$
and $q_{Y,K} \circ \imath(f'_2)$ give two
elements of $(\quot{Y}{K})(X')$ which belong
to the image of the map 
\begin{equation} \label{eq:YKX'}
(K \backslash \frh_\cC(Y))(X') \to (\quot{Y}{K})(X').
\end{equation}
Since $\cC$ is an $E$-category, it follows that
the group $K$ acts freely on the set
$(\frh_\cC(Y))(Z)$ for any object $Z$ of $\cC$.
Hence it follows from the description of the
sheafification functor $a_J$ given in
\eqref{eq:aJ1} that the map \eqref{eq:YKX'} is injective.
This shows that the two elements
of $K\backslash (\frh_\cC(Y))(X')$ 
given by $f'_1$ and $f'_2$ coincide.
Hence there exists an element $g \in K$
satisfying $f'_2 = g \circ f'_1$.
This completes the proof.
\end{proof}

\begin{lem}
\label{lem:any representative}
Let $f\colon (X,G) \to (Y,H)$ be a morphism which belongs to $\wt{\cT}^\str$.
Then for any representative 
$(f; m', f')$, the morphism $f'$
belongs to $\cT(J)$.
\end{lem}
\begin{proof}
Let $(f; m'\colon X'_1 \to X, f'\colon X'_1 \to Y)$ 
be a representative 
where $f' \in \cT(J)$.
Take another representative
$(f;m''\colon X'_2 \to X, f''\colon  X_2' \to Y)$.

We can complete the diagram 
$X_1' \xto{m'} X \xleftarrow{m''} X_2'$ 
to a square
\[
\begin{CD}
X_3' @>{g}>>   X_2'
\\
@V{h}VV     @VV{m''}V
\\
X_1'  @>{m'}>> X
\end{CD}
\]
with $h \in \cT(J)$ 
using Property (3) of semi-localizing 
collection. 
Note that both $(f;m'\circ h,f'\circ h)$ and 
$(f;m''\circ g, f''\circ g)$ are representatives
of $f$. It follows from Lemma \ref{lem:exists_g} that
we have $\sigma \circ f' \circ h = f'' \circ g$
for some $\sigma \in H$.
The morphism $\sigma$ belongs to $\cT(J)$ since
it is an isomorphism.
Hence the composite $\sigma \circ f' \circ h
= f'' \circ g$ belongs to $\cT(J)$.
By using that $\cT(J) = \wh{\cT}$,
we see that $f'' \in \cT(J)$.
This proves the claim.
\end{proof}

\begin{lem}
\label{lem:tilde semi-localizing}
The set $\wt{\cT}^\str$ is semi-localizing.
Moreover, we have $\wh{\wt{\cT}^\str} = \wt{\cT}^\str$.
\end{lem}
\begin{proof}
Let $(X,G)$ be an object of $\wt{\cC}^\str$ and 
let $\id_{(X,G)}$ denote the identity morphism.
Then $(\id_{(X,G)};\id_X, \id_X)$ is a representative
of $\id_{(X,G)}$.  Since $\id_X$ belongs to 
$\cT(J)$, it follows that $\id_{(X,G)}$ belongs to 
$\wt{\cT}^\str$.

Let $f\colon (X,G) \to (Y,H)$ and 
$g\colon (Y,H) \to (Z,K)$ be two morphisms
in $\wt{\cT}^\str$.
Let $(f; m'\colon X' \to X, f'\colon X' \to Y)$ 
be a representative of $f$.
Let $(g; n'\colon Y' \to Y, g'\colon Y' \to Z)$ 
be a representative of $g$.
We can complete the diagram 
$X' \xto{f} Y \xleftarrow{n'} Y'$
to a commutative square
\[
\begin{CD}
Z'  @>{f''}>> Y'    
\\
@V{n''}VV   @V{n'}VV
\\
X' @>{f'}>> Y
\end{CD}
\]
with $n'' \in \cT(J)$
using Property (3) of semi-localizing collection 
(\cite[Def 2.3.1]{Grids}).  
Then 
$(g\circ f; m' \circ n'', g' \circ f'')$ is a representative
of the composite $g\circ f$.
Now suppose that $f$ and $g$ belong to $\wt{\cT}^\str$.
It follows from Lemma \ref{lem:any representative} that
both $f'$ and $g'$ belong to $\cT(J)$.
Then using Property (3) of B-site
(\cite[Def 4.2.1]{Grids}), we also have $f'' \in \cT(J)$. 
Since $g' \circ f'' \in \cT(J)$, the composite
$g \circ f$ 
hence belongs to $\wt{\cT}^\str$.
So $\wt{\cT}^\str$ is closed under composition.
Conversely, suppose that $g \circ f$ belongs to 
$\wt{\cT}^\str$. 
It follows from Lemma \ref{lem:any representative} that
$g' \circ f''$ belongs to $\cT(J)$.
Since $\wh{\cT(J)} = \cT(J)$, we have $g' \in \cT(J)$. 
Hence $g$ belongs to $\wt{\cT}^\str$.
This proves that $\wh{\wt{\cT}^\str} = \wt{\cT}^\str$.

It remains to check that Property (3) holds for $\wt{\cT}^\str$.
Given morphisms $f\colon (X, G) \to (Y,H)$ and 
$g\colon (Z,K) \to (Y,H)$, suppose $f \in \wt{\cT}^\str$.
Take representatives 
$(f; m', f')$ and $(g; m'', f'')$ with $f' \in \cT(J)$.
We can complete the diagram 
$X' \xto{f'} Y \xleftarrow{f''} Y'$ 
to a commutative square
\[
\begin{CD}
Z' @>{\ell}>>   Y'
\\
@V{h}VV @VV{f''}V
\\
X' @>{f'}>> Y
\end{CD}
\]
with $\ell \in \cT(J)$.
We obtain a morphism $i\colon \imath^\str(Z') \to (Z,K)$
as the composite 
$q_{Z,K} \circ 
\imath^\str(m'' \circ \ell)$.
We can take a representative 
$(i; \id_{Z'}, m'' \circ \ell)$.
As $m'' \circ \ell \in \cT(J)$, we have that 
$i \in \wt{\cT}^\str$.
Then we have the desired commutative square: 
\[
\begin{CD}
Z' @>{i}>>  (Z,K) \\
@VVV   @VV{g}V   \\
(X,G) @>>{f}> (Y,H).
\end{CD}
\]
\end{proof}

\subsection{ }
\label{sec:cT}
Let $(\cC, J)$ be a $B$-site. 
Let $\wt{\cC}$ and $\wt{\cC}^\str$ be as in Section \ref{sec:Ctil}.
We let $\wt{\cT}$ denote the set of morphisms
$f\colon F_1 \to F_2$ in $\wt{\cC}$ satisfying the following conditions:
There exist a morphism $f'\colon (X_1,H_1) \to (X_2,H_2)$ in $\wt{\cC}^\str$
and isomorphisms $\alpha_i \colon  \quot{X_i}{H_i} \xto{\cong} F_i$ for $i=1,2$
such that $f'$ belongs to $\wt{\cT}^\str$ and that 
the morphism $f'$ regarded as a morphism $\quot{X_1}{H_1} \to \quot{X_2}{H_2}$
in $\wt{\cC}$ is equal to $\alpha_2^{-1} \circ f  \circ \alpha_1$.
It follows from Lemma \ref{lem:tilde semi-localizing} that
$\wt{\cT}$ is semi-localizing.
We show that $(\wt{\cC}, J_{\wt{\cT}})$ 
is a $B$-site.
\begin{prop} \label{prop:Ctil_Bsite}
Let $(\cC, J)$ be a $B$-site.
Then the site $(\wt{\cC}, J_{\wt{\cT}})$
is a B-site.
\end{prop}
\begin{proof}
We need to check the three conditions
in \cite[Def.~4.2.1]{Grids}.
The topology
$J_{\wt{\cT}}$ is by 
definition an $A$-topology.
It is shown that $\wt{\cC}$ 
is an $E$-category in
Lemma~\ref{lem:tilde E-category}.
We show that the third condition 
in Lemma~\ref{lem:tilde cond 3}
holds below.
\end{proof}

\begin{lem}
\label{lem:tilde E-category}
The category $\wt{\cC}$ is an $E$-category,
i.e., all the morphisms are epimorphisms.
\end{lem}

\begin{proof}
%
It suffices to prove that $\wt{\cC}^\str$ is an $E$-category.
Let us look at the following 
diagram~\eqref{diagram:trapezoid} in $\wt{\cC}^\str$.
We are given morphisms
$f, g_1, g_2$ such that 
$g_1 \circ f=g_2 \circ f$.
We want to show that $g_1=g_2$.
We have taken representatives
$(f; f', f'')$, $(g_1;g', g_1'')$ and 
$(g_2;g', g_2'')$.   Note that we can 
take representatives so that we 
have the same $g'$ for $g_1$ and $g_2$.
We complete the diagram
$B \xto{f''} Y \xleftarrow{g'} A$
to the square $g' \circ i= f'' \circ h$
with $h \in \cT(J)$.

We then obtain representatives
$(g_i\circ f; f' \circ h, i \circ g_i'')$
for $i=1,2$.
As $g_1 \circ f=g_2 \circ f$,
it follows from Lemma \ref{lem:exists_g} that
there exists $k \in K$ satisfying
$g_1' \circ i = k \circ g_2' \circ i$.
Since $i$ is an epimorphism,  
we have $g_1' = k \circ g_2'$.
Note that 
$(g_1; f'' \circ h, g_1'' \circ i)$
and $(g_2; f'' \circ h, k \circ g_2'' \circ i)$
are representatives.
Hence we obtain 
$g_1=g_2$ as desired.
\begin{equation} 
\label{diagram:trapezoid}
\xymatrix{
\imath^\str (C) 
\ar[dr]^{\imath^\str(i)} 
\ar[d]_{\imath^\str(h)} & 
\\
\imath^\str(B) 
\ar[d]_{\imath^\str(f')}
\ar[dr]_{\imath^\str(f'')}
&
\imath^\str(A) 
\ar[d]_{\imath^\str(g')}
\ar@<-.5ex>[dr]_{\imath^\str(g_1'')}
\ar@<.5ex>[dr]^{\imath^\str(g_2'')}
\\
\imath^\str(X)
\ar[d]_{q_{X,G}} 
& 
\imath^\str(Y)
\ar[d]_{q_{Y,H}}
& 
\imath^\str(Z)
\ar[d]_{q_{Z,K}}
\\
(X,G)
\ar[r]_{f}
&
(Y,H)
\ar@<-.5ex>[r]_{g_1}
\ar@<.5ex>[r]^{g_2}
&
(Z,K)
}
\end{equation}
\end{proof}

\begin{lem}
\label{lem:tilde cond 3}
The condition (3) in \cite[Def.~4.2.1]{Grids} 
holds:
For any diagram
$Z\backslash K \xto{f} Y \backslash H \xto{g} X \backslash G$ in $\wt{\cC}$,
the composite $g \circ f$ belongs
to $\wt{\cT}$ if and 
only if both $f$ and $g$ belong to
$\wt{\cT}$.
\end{lem}
\begin{proof}
We have seen 
in 
Lemma~\ref{lem:tilde semi-localizing}
the if direction.

Suppose $g \circ f$ 
belongs to $\wt{\cT}$.
We show that $f, g \in \wt{\cT}$.
Let us use the notation in
the proof of Lemma~\ref{lem:tilde semi-localizing}.
From Lemma~\ref{lem:any representative}, 
we know that 
$g' \circ f'' \in \cT(J)$.
Hence $g', f'' \in \cT(J)$.   Then it follows from 
$n'' \circ f' =f'' \circ n' \in \cT(J)$ that $f' \in \cT(J)$.
By the definition of $\wt{\cT}$, we see that $f, g \in \wt{\cT}$.
\end{proof}

\subsection{Comparing the pushforward topology and $J_{\wt{\cT}}$}

\begin{lem} \label{lem:compare_JT}
The topology $J_{\wt{\cT}}$ associated with
the semi-localizing collection $\wt{\cT}$
equals the pushforward topology
$\iota_* J$.
\end{lem}
\begin{proof}
Let $(Z,K) \in \Obj(\wt{\cC}^\str)$.
A sieve $R$ of $(Z,K)$ belongs to 
$J_{\wt{\cT}^\str}((Z,K))$ if and only if there 
exists a morphism $(X,G) \to (Z,K) \in \wt{\cT}^\str$
which belongs to $R$.
A sieve $R$ of $(Z,K)$ belongs to $(\iota^\str_*J)((Z,K))$
if and only if 
for any morphism $\iota^\str(Y) \to (Z,K)$,
there exists a morphism
$W \to Y \in \cT(J)$
such that $\iota^\str(W) \to \iota^\str(Y) \to Z$ 
belongs to $R$.

It is then clear that 
$(\iota^\str_*J)((Z,K)) \subset J_{\wt{\cT}^\str}((Z,K))$.
Let us prove the other inclusion.
Let $R \in J_{\wt{\cT}^\str}((Z,K))$.
Take $f\colon  (X,G) \to (Z,K) \in \wt{\cT}^\str$.
Note that this means $R$ contains the sieve
$R_f$.   Let $Y \in \cC$ and 
$g\colon (Y, \{\id_Y\}) \to (Z,K)$ be an arbitrary morphism.
Using Property (3) of the semi-localizing 
collection $\wt{\cT}^\str$,
we obtain a commutative square:
\[
\begin{CD}
(V, S) @>>>   (X, G)  \\
@V{i}VV      @VV{f}V  \\
(Y, \{\id_Y\}) @>{g}>> (Z,K)
\end{CD}
\]
with $i \in \wt{\cT}^\str$.
Take representatives 
$(i; i_1\colon U \to V, i_2\colon U \to Y)$
and 
$(g; g_1\colon A \to Y, g_2\colon A \to Z)$.
We complete the diagram
$U \xto{i_2} Y \xleftarrow{g_1}A$ 
to a commutative square
\[
\begin{CD}
W @>{b_1}>>  A \\
@V{b_2}VV    @V{g_1}VV \\
U  @>{i_2}>> Y
\end{CD}
\]
with $b_1 \in \cT(J)$.   
We consider the object $W \in \cC$
with the morphism $j=g_1 \circ b_1=i_2\circ b_2$
to $Y$.   We need to show that 
$g \circ \iota^\str(j)$ belongs 
to $R$.   This follows from that $R$ contains 
$R_f$ and that 
$g \circ \iota^\str(j)$ factors through
$(X,G)$ by construction.
\end{proof}

\subsection{$Y$-sites}
Let $(\cC,J)$ be a $B$-site such that $\cT(J)$ has enough
Galois coverings.

\subsubsection{}

\begin{lem} \label{lem:aJ_properties}
The functor $a_J\colon  \Presh(\cC) \to \Shv(\cC,J)$
commutes with finite limits and arbitrary small colimits.
\end{lem}

\begin{proof}
$a_J$ commutes with arbitrary small colimits since
it is left adjoint to the inclusion functor
$j\colon \Shv(\cC,J) \inj \Presh(\cC)$.
To prove that $a_J$ commutes with finite limits,
it suffices to prove that the composite $j \circ a_J$
commutes with finite limits.
Hence it suffices to prove that, for any object $X$ of $\cC$,
the functor $s_X$ from $\Presh(\cC)$ to the category of sets
which associates $a_J(F)(X)$ to each presheaf $F$ on $\cC$
commutes with finite limits.
Let us fix an object $X$ of $\cC$.
Let us consider the cofiltered category $\Gal/X$ introduced
in Section 4.4.2 of \cite{Grids}.
For an object $(Y,f)$ of $\Gal/X$, let $s_{Y,f}$ denote the
functor from $\Presh(\cC)$ to the category of sets
which associates $F(Y)^{\Aut_X(Y)}$ to each presheaf $F$ on $\cC$.
One can check easily that the functor $s_{Y,f}$ commutes with
finite limits.
It follows from the argument in Section 4.4.2 of \cite{Grids}
that the functor $s_X$ is isomorphic to the functor
that associates $\varinjlim_{(Y,f)} s_{Y,f}(F)$
with any presheaf $F$ on $\cC$.
Since a filtered colimit commutes with finite limits
(cf.\ \cite{MacLane} Ch.\ IX), it follows that the functor
$s_X$ commutes with finite limits.
This proves the claim.
\end{proof}

\begin{lem} \label{lem:omega_faithful}
The functor $\cC \to \Shv(\cC,J)$ thatociates an object $X$ of $\cC$ to
$a_J(\frh_\cC(X))$ is faithful.
\end{lem}

\begin{proof}
Let $X$ and $Y$ be two objects of $\cC$.
We prove that the map
$\Hom_\cC(X,Y) \to \Hom_{\Shv(\cC,J)}
(a_J(\frh_\cC(X)),a_J(\frh_\cC(Y)))$ is injective.
Since we have isomorphisms
$$
\Hom_{\Shv(\cC,J)}
(a_J(\frh_\cC(X)),a_J(\frh_\cC(Y)))
\cong 
\Hom_{\Presh(\cC)}
(\frh_\cC(X),a_J(\frh_\cC(Y)))
\cong a_J(\frh_\cC(Y))(X),
$$
It suffices to show that the map
$(\frh_\cC(Y))(X) \to a_J(\frh_\cC(Y))(X)$
induced by the adjunction morphism
$\frh_\cC(Y) \to a_J(\frh_\cC(Y))$ is injective.
It follows from the argument in Section 4.4.2 of \cite{Grids}
that we have
\begin{equation} \label{eq:aJ}
a_J(\frh_\cC(Y)) \cong 
\varinjlim_{(Z,f)} ((\frh_\cC(Y))(Z))^{\Aut_Y(Z)}
\end{equation}
where $(Z,f)$ runs over the object of $\Gal/Y$
introduced in Section 4.4.2 of \cite{Grids}.
Since $\cC$ is an $E$-category, it follows that the
transition maps in the colimit \eqref{eq:aJ} are injective.
This proves that
$(\frh_\cC(Y))(X) \to a_J(\frh_\cC(Y))(X)$ is injective.
\end{proof}

\begin{lem} \label{lem:i_Galois}
The functor $\imath$ sends a Galois covering in $\cC$
to a Galois covering in $\wt{\cC}$.
Moreover, for any Galois covering $f\colon X\to Y$ 
in $\cC$, the functor $\imath$ induces an isomorphism
$\Aut_{Y}(X) \xto{\cong} \Aut_{\quotid{Y}}(\quotid{X})$
of groups.
\end{lem}

\begin{proof}
Let $f\colon  X \to Y$ be a Galois covering in $\cC$.
Let us write $G = \Aut_Y(X)$. Then
the commutative diagram
$$
\begin{CD}
\coprod_{g \in G} X @>{\coprod_g g}>> X \\
@V{\coprod_g \id_X}VV @VV{f}V \\
X @>{f}>> Y
\end{CD}
$$
in the category $\Presh(\cC)$ is cartesian.
We apply the sheafification functor to the above diagram.
It follows from Lemma \ref{lem:aJ_properties} that
the diagram
$$
\begin{CD}
\coprod_{g \in G} \iota(X) @>{\coprod_{g} \iota(g)}>> \iota(X) \\
@V{\coprod_{g} \id_{\iota(X)}}VV @VV{\iota(f)}V \\
\iota(X) @>{\iota(f)}>> \iota(Y)
\end{CD}
$$
in $\Shv(\cC,J)$ is cartesian.
This implies that $\iota(f)$ is a Galois covering in $\wt{\cC}$ 
and its Galois group is isomorphic to $G$. This proves the claim.
\end{proof}

\begin{lem} \label{lem:Galois2}
Let $\quot{X}{H}$ be an object of
$\wt{\cC}$. Then the morphism $q_{X,H}\colon \imath({X})
\to \quot{X}{H}$ in $\wt{\cC}$ is a Galois covering
which belongs to $\cT(\imath_* J)$,
and its Galois group is isomorphic to $H$
via the composite $H \subset \Aut(X) \to
\Aut(\quot{X}{\{\id_X\}})$.
\end{lem}

\begin{proof}
It suffices to prove a similar statement for 
the morphism $q_{X,H}\colon \imath^\str(X) \to (X,H)$
regarded as a morphism in $\wt{\cC}^\str$.
Let $f\colon  Z \to (X,H)$ be a morphism in $\wt{\cC}^\str$
and let $h_1,h_2\colon  Z \to \imath^\str(X)$ be two morphisms
in $\wt{\cC}^\str$ over $(X,H)$. It suffices to
prove that there exists a unique element $g \in H$
satisfying $h_2 = \imath^\str(g) \circ h_1$.
The uniqueness of such $g \in H$ follows from 
Lemma \ref{lem:omega_faithful}
and from that $\wt{\cC}$ is an $E$-category.
Hence it suffices to prove the existence of $g$.

Let us write $Z = (Y,K)$.
For $i \in\{1,2\}$, let us choose a representative 
$Y \xleftarrow{m_i} Y'_i \xto{h'_i} X$ of $h_i$ in $\cC$.
Since $\cC(\cT(J))$ is semi-cofiltered, we may choose the
representatives of $h_1$ and $h_2$ in such a way that
$Y'_1 = Y'_2$ and $m_1 = m_2$ holds.
Since the two representatives can be regarded as two 
representatives of $f$, it follows from Lemma \ref{lem:exists_g} that
there exists $g \in H$ satisfying $h'_2 = g \circ h'_1$.
Since $\wt{\cC}^\str$ is an $E$-category,
this implies that $h_2 = \imath^\str(g) \circ h_1$.
This proves the claim.
\end{proof}

\subsubsection{ }
From now on until the end of this section,
we assume that $(\cC, J)$ is a $Y$-site.

\begin{lem}\label{lem:imath_enough_Galois}
$\cT(\imath_* J)$ contains enough Galois coverings.
\end{lem}

\begin{proof}
It suffices to show that
$\cT((\imath^\str)_* J)$ contains enough Galois coverings.
Let $f\colon  (X,H) \to (Y,K)$ be a morphism in $\wt{\cC}$
which belongs to $\cT((\imath^\str)_* J)$.
Let us choose a representative 
$X \xleftarrow{m} X' \xto{f'} Y$ of $f$ in $\cC$.
It follows from Lemma \ref{lem:any representative} that
$f'$ belongs to $\cT(J)$. 
Let us choose a morphism $m' \colon  Y \to Z$ in $\cC$
such that $m'$ belongs to $\cT(J)$ and that $K$
is a subgroup of $\Aut_Z(Y)$.
Since $\cT(J)$ contains enough Galois coverings,
we may assume, by changing the representative of $f$ in $\cC$
if necessary, that the composite $m' \circ f'$ is a Galois
covering in $\cC$.
It follows from Lemma \ref{lem:i_Galois} that
$\imath(m' \circ f')$ is a Galois covering in $\wt{\cC}$.
Hence $\imath^\str(m' \circ f')$ is a Galois covering 
in $\wt{\cC}^\str$.

Since $\imath^\str(m')$ factors through the morphism
$q_{Y,K}$, by Lemma 4.1.3 of \cite{Grids},
the morphism $f \circ q_{X,H} \circ
\imath^\str(m) = q_{Y,K} \circ \imath^\str(f')$
is a Galois covering in $\wt{\cC}^\str$.
This proves that $\cT((\imath^\str)_* J)$ contains
enough Galois coverings.
\end{proof}

\begin{cor} \label{cor:Ctil_Y_site}
The site $(\wt{\cC}, \imath_* J)$ is a $Y$-site
in the sense of \cite[Definition 5.4.2]{Grids}.
\end{cor}

\begin{proof}
The category $\wt{\cC}$ is essentially small
since $\cC$ is essentially small.
It follows from Lemma \ref{lem:imath_enough_Galois}
that $\cT(\imath_* J)$ contains enough Galois coverings.
It follows from Lemma \ref{lem:any representative} that
the morphism $q_{X,H}$ belongs to $\cT((\imath^\str)_*J)$ 
for any object $(X,H)$ of $\wt{\cC}^\str$.
By using this and the $\Lambda$-connectivity 
(see Definition 5.4.1 in \cite{Grids}) 
of $\cC(\cT(J))$,
one can check easily that $\wt{\cC}^\str(\cT((\imath^\str)_* J))$
is $\Lambda$-connected. Hence
$(\wt{\cC}, \imath_* J)$ is a $Y$-site.
\end{proof}

\subsection{Inheriting the finiteness condition}
\begin{lem}
Suppose the cardinality condition (1) in Section 5.8.1
of \cite{Grids} holds for the overcategory
$\cC_{/X}$ 
for any object $X$ of $\cC$.
Then the same cardinality condition 
is satisfied for the overcategory 
$\wt{\cC}(\cT(\imath_* J))_{/Z}$
for any object $Z$ of $\wt{\cC}$.
\end{lem}

\begin{proof}
Let us write $Z = \quot{X}{H}$, where $X$ is
an object of $\cC$ and $H$ is a moderate subgroup 
of $\Aut_\cC(X)$.
It suffices to prove a similar statement for
the overcategory 
$\wt{\cC}^\str(\cT((\imath^\str)_*J))_{/(X,H)}$.
Let $f_1 \colon (Y_1,K_1) \to (X,H)$ and
$f_2\colon  (Y_2,K_2) \to (X,H)$ be two morphisms
in $\wt{\cC}^\str$ which belong to $\cT((\imath^\str)_*J)$. 
We will prove that the set 
$\Hom_{(\wt{\cC}^\str)_{/(X,H)}}(f_1,f_2)$ is a finite set.

Let us choose a morphism $h\colon X \to W$ in $\cC$
which belongs to $\cT(J)$ such that
$H$ is a subgroup of $\Aut_W (X)$.
Then the morphism $\imath^\str(h)$ factors through 
$q_{X,H}$. Let $\alpha\colon \quot{X}{H} \to \imath^\str(W)$
denote the induced morphism.
Observe that $\Hom_{(\wt{\cC}^\str)_{/(X,H)}}(f_1,f_2)$ is a 
subset of $\Hom_{(\wt{\cC}^\str)_{/\imath^\str(W)}}
(\alpha \circ f_1, \alpha \circ f_2)$.
Hence by replacing $(X,H)$ with $(Y,\{\id_Y\})$ if necessary,
we may and will assume that $H=\{\id_X\}$, i.e.,
$(X,H) = \imath^\str(X)$.

For each $i \in \{1,2\}$, let us choose a representative
$Y_i \xleftarrow{m_i} Y'_i \xto{f'_i} X$ of 
$f_i\colon  (Y_i,K_i) \to \imath^\str(X)$ in $\cC$.
It follows from 
Lemma \ref{lem:any representative} that $f'_i$ belongs to $\cT(J)$.
Since $\wt{\cC}^\str$ is an $E$-category,
the map 
$\Hom_{(\wt{\cC}^\str)_{/\imath^\str(X)}}(f_1,f_2)
\to \Hom_{(\wt{\cC}^\str)_{/\imath^\str(X)}}(f'_1,f_2)$
given by the composition with 
$q_{Y_1,K_1} \circ \imath^\str(m_1)$ is injective.
Hence by replacing $f_1$ with $f'_1$ if necessary,
we may and will assume that $K_1=\{\id_{Y_1}\}$ and
$f_1 = \imath^\str(f_0)$ for some morphism
$f_0\colon  Y_1 \to X$ in $\cC$ which belongs to $\cT(J)$.

Let $t\colon  \imath^\str(Y_1) \to \quot{Y_2}{K_2}$
be an arbitrary morphism in $\wt{\cC}^\str$.
Let us choose a representative
$Y_1 \xleftarrow{m'_1} U \xto{t'} Y_2$
of $t$. Since $\cT(J)$ is semi-localizing,
one can extend the diagram
$U \xto{t'} Y_2 \xleftarrow{m_2} Y'_2$
to a commutative diagram
$$
\begin{CD}
U' @>{t''}>> Y'_2 \\
@V{m'_2}VV @VV{m_2}V \\
U @>{t'}>> Y_2
\end{CD}
$$
in $\cC$ such that the morphism $m'_2$ belongs to $\cT(J)$.
Since $\cT(J)$ contains enough Galois coverings,
there exists a morphism $V \to U$ in
$\cC$ which belongs to $\cT(J)$ such that the 
composite $f'_0 \colon  V \to U' \xto{m'_2} U \xto{f_0\circ m'_1} X$ is 
a Galois covering in $\cC$.
By replacing $f_0 \circ m'_1$ with $f'_0$ if necessary, 
we may will assume that $f_0 \circ m'_1 \colon  U \to X$ is a Galois
covering in $\cC$.
In summary, we have the following: 
For any morphism $t\colon  \imath^\str(Y_1) \to \quot{Y_2}{K_2}$
in $\wt{\cC}^\str$,
there exists a diagram
\begin{equation} \label{eq:rep_U}
Y_1 \xleftarrow{m'_1} U \xto{t''} Y'_2
\end{equation}
in $\cC$ such that 
$Y_1 \xleftarrow{m'_1} U \xto{m_2 \circ t''} Y_2$
is a representative of $t$ in $\cC$
and that $f_0 \circ m'_1$ is a Galois covering 
in $\cC$. 

If $\Hom_{(\wt{\cC}^\str)_{/\imath^\str(X)}}(f_1,f_2)$ 
is an empty set, then there is nothing to prove. 
Suppose that there exists a morphism 
$t_0\colon \imath^\str(Y_1) \to (Y_2,K_2)$ in $\wt{\cC}^\str$
over $\imath^\str(X)$.
Let us choose a diagram
$Y_1 \xleftarrow{m'_0} U_0 \xto{t''_0} Y'_2$
in $\cC$ such that 
$Y_1 \xleftarrow{m'_0} U_0 \xto{m_2 \circ t''_0} Y_2$
is a representative of $t_0$ in $\cC$
and that $f_0 \circ m'_0$ is a Galois covering 
in $\cC$. 
For an arbitrary morphism 
$t\colon  \imath^\str(Y_1) \to \quot{Y_2}{K_2}$
in $\wt{\cC}^\str$, let us choose a diagram \eqref{eq:rep_U}
as above. Since $\cC(\cT(J))$ is semi-cofiltered
and $\cT(J)$ contains enough Galois coverings,
we may assume, by changing $U$ if necessary, 
that there exists a morphism $\alpha\colon  U \to U_0$
satisfying $m'_1 = m'_0 \circ \alpha$.
It follows from Corollary 4.2.4 in \cite{Grids} that
there exists a morphism $t'''\colon  U_0 \to Y'_2$ satisfying
$t''=t''' \circ \alpha$.
Since $\wt{\cC}^\str$ is an $E$-category, $\imath^\str(\alpha)$
is an epimorphism in $\wt{\cC}^\str$.
This implies that the diagram
$Y_1 \xleftarrow{m'_0} U_0 \xto{m_2 \circ t'''} Y_2$
is a representative of $t$ in $\cC$.
Hence by replacing the diagram \eqref{eq:rep_U} with the
diagram $Y_1 \xleftarrow{m'_0} U_0 \xto{t'''} Y'_2$,
we may assume that $U=U_0$ and $m'_1 = m'_0$.
One can check, by using that $\wt{\cC}^\str$ is an $E$-category,
that the morphism $t''$ uniquely determines the morphism $t$.
It follows from Lemma \ref{lem:omega_faithful} that
$t''\colon  U_0 \to Y'_2$ is a morphism over $X$.
Since both $f_0 \circ m'_0$ and $f'_2$ belongs to $\cT(J)$
and $(\cC,J)$ is a $B$-site, any morphism $U_0 \to Y'_2$ 
over $X$ belongs to $\cT(J)$.
Since the set of morphisms $U_0 \to Y'_2$ over $X$ is a finite set,
it follows that the set 
$\Hom_{\wt{\cC}^\str(\cT((\imath^\str)_*J))_{/\imath^\str(X)}}
(f_1,f_2)$ is a finite set.
This completes the proof.
\end{proof}

\subsubsection{ }
Let $\Cip$ be a grid of the $Y$-site $(\cC,J)$.
Suppose that the cardinality condition (1) in \cite[5.8.1]{Grids}
is satisfied for the overcategory $\cC(\cT(J))_{/X}$ for any
object $X$ of $\cC$.
Let $\cD$ denote the category of smooth left $M_{(\cC_0,\imath_0)}$-sets.
It follows from \cite[Thm.~5.8.1]{Grids} that the functor
$\omega_\Cip$ gives an equivalence of categories.
Let $\omega\colon \cC \to \cD$ denote the composite of the functor
$\omega_\Cip\colon \Presh(\cC) \to \cD$ and the Yoneda
embedding $\frh_\cC \colon \cC \to \Presh(\cC)$.
In this paragraph, we give a description
of the Grothendieck topology $\omega_* J$
in terms of the category $\cD$ 
without any reference to the categories
$\cC$, $\Shv(\cC,J)$ or 
the functors $\omega$, $\omega_\Cip$.

\begin{lem} \label{lem:pushforward_canonical}
$\omega_* J$ is equal to the canonical Grothendieck
topology on $\cD$.
\end{lem}
\begin{proof}
It follows from \cite[Lemma 10.1.2]{Grids} that
the functor $\omega_\Cip\colon \Presh(\cC) \to \cD$ 
is isomorphic to the
composite of the sheafification functor
$a_J$ with the restriction of $\omega_\Cip$
to $\Shv(\cC,J)$.
It follows from Theorem 5.8.1 in \cite{Grids} that
the restriction of $\omega_\Cip$
to $\Shv(\cC,J)$ is an equivalence of categories.
Hence the claim follows from Proposition \ref{prop:canonical}.
\end{proof}

\begin{lem}\label{lem:Galoiscovering}
Let $f\colon F \to F'$ be a Galois covering in $\wt{\cC}$.
Then the object $F'$  together with
the map $f\colon F \to F'$ 
is a \quotobj of $F$ by $\Aut_{F'}(F)$
in the category $\Shv(\cC,J)$.
In particular it is a \quotobj of $F$ 
by $\Aut_{F'}(F)$
in the category $\wt{\cC}$.
\end{lem}

\begin{proof}
This follows from Lemma \ref{lem:pushforward_canonical} and
\cite[Lemma 3.2.4]{Grids}.
\end{proof}

Let $f\colon X\to X'$ be a morphism in $\cC$ which belongs to $\cT$.
We denote by $\Aut^{[X]}(X')$ the set of automorphisms of
$X'$ in $\cC$ which ascends (see \cite[4.2.1]{Grids} for definition)
 to an automorphism of $X$ via $f$.
Then $\Aut^{[X]}(X')$ is a subgroup of $\Aut_{\cC}(X')$.
We denote by $\Aut_{[X']}(X)$ the set of automorphisms of
$X$ in $\cC$ which descends (see \cite[4.2.1]{Grids} for definition)
to an automorphism of $X'$ via $f$.
Then $\Aut_{[X']}(X)$ is a subgroup of $\Aut_{\cC}(X)$.
There exists a group homomorphism $\Aut_{[X']}(X) \to \Aut_{\cC}(X')$
which sends $\wt{\alpha}$ to the automorphism $\alpha$ of $X'$ to
which $\wt{\alpha}$ descends. The automorphism $\alpha$
is uniquely determined by $\wt{\alpha}$ since $f$ is an epimorphism.
By definition the image of this homomorphism is equal to
$\Aut^{[X]}(X')$. Hence we have a short exact sequence
\begin{equation}\label{eq:3_automs}
1 \to \Aut_{X'}(X) \to \Aut_{[X']}(X) \to \Aut^{[X]}(X') \to 1
\end{equation}
of groups.

\begin{lem}\label{lem:double_quot}
Let $f\colon X \to X'$ be a Galois covering in $\cC$
which belongs to $\cT$.
Let $H \subset \Aut^{[X]}(X')$ be a subgroup and let
$\wt{H} \subset \Aut_{[X']}(X)$ denote the inverse image of $H$
under the second homomorphism in \eqref{eq:3_automs}.
Then the morphism $f$ in $\cC$ induces an isomorphism
$\quot{X}{\wt{H}} \cong \quot{X'}{H}$ in $\wt{\cC}$ 
such that the diagram
$$
\begin{CD}
\quotid{X} @>{f}>> \quotid{X'} \\
@VVV @VVV \\
\quot{X}{\wt{H}} @>{\cong}>> \quot{X'}{H},
\end{CD}
$$
in $\wt{\cC}$
is commutative, 
where the vertical arrows are the canonical
quotient maps.
\end{lem}
\begin{proof}
It suffices to prove that the object $\quot{X'}{H}$ together with
the composite $\quotid{X} \to \quotid{X'} \to \quot{X'}{H}$
is a \quotobj of $\quotid{X}$ by $\wt{H}$
in $\wt{\cC}$. Since we have an exact sequence
$1 \to \Aut_{X'}(X) \to \wt{H} \to H \to 1$ and
since $\quotid{X'}$ together with the morphism
$\quotid{X} \to \quotid{X'}$ is the \quotobj
of $\quotid{X}$ by $\Aut_{X'}(X)$, the claim
follows from Lemma \ref{lem:quot_general}, 
Lemma \ref{lem:Galois2}
and Lemma \ref{lem:Galoiscovering}.
\end{proof}

\begin{lem}\label{lem:mono}
Let $f\colon  Y \to X$ be a morphism in $\wt{\cC}$ which belongs
to $\wt{\cT}$. Suppose that $f$ is a monomorphism.
Then $f$ is an isomorphism.
\end{lem}

\begin{proof}
Let us choose a morphism $g\colon Z \to Y$ in $\wt{\cC}$
such that $g$ belongs to $\wt{\cT}$ and that the composite
$f \circ g$ is a Galois covering in $\wt{\cC}$.
It follows from Lemma 4.1.3 of \cite{Grids} that $g$ is 
also a Galois covering in $\wt{\cC}$.
Let us consider the following commutative diagram
$$
\begin{CD}
\Hom_{\cC}(Z,Z) @>{g_*}>> \Hom_{\cC}(Z,Y) \\
@| @VV{f_*}V \\
\Hom_{\cC}(Z,Z) @>{(f\circ g)_*}>> \Hom_{\cC}(Z,X)
\end{CD}
$$
of sets. Since $f$ is monomorphism, the map $f_*$ is
injective. This implies that $g_*^{-1}(g)$ and 
$(f \circ g)_*^{-1}(f \circ g)$ are equal as subsets
of $\Hom_{\cC}(Z,Z)$. Since $g_*^{-1}(g)$ is a $\Gal(g)$-torsor
and $(f \circ g)_*^{-1}(f \circ g)$ is a $\Gal(f \circ g)$-torsor,
the canonical map $\Gal(f \circ g) \to \Gal(g)$ is an isomorphism.
It follows from Lemma \ref{lem:Galoiscovering} that
$g\colon Z \to Y$ is a quotient object of $Z$ by $\Gal(g)$,
and that $f \circ g\colon Z \to X$ is a quotient object of $Z$ by 
$\Gal(f \circ g)$. The uniqueness of quotient objects 
implies that $f$ is an isomorphism.
\end{proof}

\section{Adding finite coproducts}
\label{sec:add coproduct}
The example case of a $Y$-site of a Galois extensions of a field is given in Section~\ref{sec:intro Y-sites}.

At this point, we have in our mind a $Y$-site where we have added quotients.   
The next step is to add finite coproducts.   Again, we have an embedding 
$\cC \to \Shv(\cC, J)$.     There are coproducts of objects in $\Shv(\cC,J)$
so we consider such subcategory.    We note that the resulting category is no longer
a $Y$-site, but we still have control over this category in that 
when equipped with the pushforward topology, the new site is 
still equivalent to the original one.
The main merit or the result of this construction is that the fiber 
products exist
in this larger category.

\subsection{$F$-category}
\subsubsection{ }
Let $\cF$ be a category which has an initial object and
finite coproducts.
We say that an object $Y$ of $\cF$ is connected 
if for any finite set $I$, 
and for any family $(N_i)_{i \in I}$ of objects
of $\cF$ indexed by the elements in $I$,
the map $\Hom_{\cF}(Y,N_i) \to \Hom_{\cF}(Y,
\coprod_{i \in I} N_i)$ given by the composition
with the inclusion morphism 
$N_i \to \coprod_{i \in I} N_i$
for each $i \in I$ induces a bijection
$\coprod_{i \in I} \Hom_{\cF}(Y,N_i) \to \Hom_{\cF}(Y,
\coprod_{i \in I} N_i)$.
\begin{defn}
We say that a category $\cF$ is an
{\em $F$-category} if the following two conditions are
satisfied:
\begin{enumerate}
\item $\cF$ has an initial object and
finite coproducts.
\item Any object of $\cF$ is isomorphic to a coproduct 
of a family $(N_i)_{i\in I}$ of connected objects of $\cF$
indexed by a finite set $I$.
\end{enumerate}
When $\cF$ is an $F$-category, 
we let $\cF^{(0)} \subset \cF$ denote 
the full subcategory of connected objects.
\end{defn}

\subsubsection{ }
\label{sec:fin_conn_morph}
Let $\cF$ be an $F$-category.
Let $M$, $N$ be two objects of $\cF$.
Let us write $M$ and $N$ as coproducts
$M= \coprod_{i\in I} M_i$ and 
$N= \coprod_{j\in J} N_j$ of connected 
objects of $\cF$ in such a way that both
$I$ and $J$ are finite sets.
\begin{lem}\label{lem:fin_conn}
We have an isomorphism
$$
\Hom_{\cF}(M,N) \cong \prod_{i \in I}
\coprod_{j \in J} \Hom_{\cF^{(0)}}(M_i,N_j).
$$
\end{lem}
\begin{proof}
It follows from the definition of coproducts that
we have $\Hom_{\cF}(M,N) \cong \prod_{i \in I}
\Hom_{\cF^{(0)}}(M_i,N)$.
Since $M_i$ is a connected object for each $i \in I$, 
we have $\Hom_{\cF^{(0)}}(M_i,N) 
= \coprod_{j \in J} \Hom_{\cF^{(0)}}(M_i,N_j)$.
This proves the claim.
\end{proof}

\subsubsection{ }
\label{sec:cond12}
Let $\cF$ be a category which has an initial object and
finite coproducts.
We consider the following two conditions for
a presheaf $F$ on $\cF$.
\begin{enumerate}
\item $F$ sends the initial object $\emptyset$ in $\cC$
to a final object $*$ in the category of sets in $\frU$.
\item For any finite set $I$ 
and for any family $(N_i)_{i \in I}$ of 
objects in $\cC$ indexed by the set $I$,
the map $F(\coprod_{i \in I} N_i) \to F(N_i)$
given by the inclusion morphism $N_i \to
\coprod_{i \in I} N_i$ in $\cC$ induces a
bijection $F( \coprod_{i\in I} N_i) \to
\prod_{i \in I} F(N_i)$.
\end{enumerate}

\subsubsection{ }\label{sec:pi_0}
Let $\cF$ be an $F$-category.
Let $N$ be an object of $\cF$. 
Then $N$ is of the form $N = \coprod_{i \in I} N_i$
where $I$ is a finite set and $N_i$ is an object of
$\cF^{(0)}$ for each $i \in I$.
We set $\pi_0(N) = I$ and call it the set of connected
components of $N$.
For $i \in I = \pi_0(N)$,
the object $N_i$ of $\cF^{(0)}$ is called the
component at $i$ of $N$.

It can be checked easily, 
by using Lemma \ref{lem:fin_conn}
that the set $\pi_0(N)$ is well-defined 
in the sense that it does not
depend on the choice of the presentation
$N = \coprod_{i \in I} N_i$ of $N$, 
up to canonical isomorphisms.
%
%
%
%
It is immediate from 
the definition that an object $N$ of
$\cF$ is an initial object (\resp connected) 
if and only if $\pi_0(N)$ is an empty set 
(\resp $\pi_0(N)$ consists of a single element).

A morphism $f\colon M \to N$ induces a map $\pi_0(M) \to \pi_0(N)$
which we denote by $\pi_0(f)$.
Using Lemma \ref{lem:fin_conn} 
we can regard a morphism $f\colon M \to N$
as a data $(\pi_0(f), (f_i)_{i \in \pi_0(M)})$
where 
$f_i \colon  M_i \to N_{\pi_0(f)(i)}$ is a morphism in
$\cF^{(0)}$ for each $i \in \pi_0(M)$.
%
%
For $i \in I$, the morphism $f_i$ in $\cF^{(0)}$ is called the
component at $i$ of $f$.

%
\begin{lem}\label{lem:aux1}
Let $\cF$ be an $F$-category.
Let $f\colon M \to N$ be a morphism in $\cF$. 
Suppose that $N$ is a coproduct $N = N_1 \amalg N_2$ 
of two objects $N_1$, $N_2$ of $\cF$.
For $j \in \pi_0(M)$, we denote by $M_j$ the component at $j$ of $M$.
For $i=1,2$, let $\iota_i\colon N_i \to N$ denote the $i$-th inclusion and
let $S_i \subset \pi_0(M)$ denote the inverse image
under the map $\pi_0(f)\colon \pi_0(M) \to \pi_0(N)$ of the image
of the map $\pi_0(\iota_i) \colon  \pi_0(N_i) \to \pi_0(N)$.
We set $M_i = \coprod_{j \in S_i} M_j$.
Let $f_i\colon M_i \to N_1$ denote the morphism 
whose component at $j \in S_i$ is equal to
the component at $j$ of $f$.
Then for $i=1,2$, the fiber product $N_i \times_N M$ exists and
the commutative diagram
\begin{equation}\label{eq:cartesian_Mi}
\begin{CD}
M_i @>{\iota'_i}>> M \\
@V{f_i}VV @V{f}VV \\
N_i @>{\iota_i}>> N,
\end{CD}
\end{equation}
where $\iota'_i \colon M_i \to M$ denotes the inclusion,
gives an isomorphism $M_i \cong N_i \times_N M$.
\end{lem}

\begin{proof}
Let $i \in \{1,2\}$.
For each object $M'$ of $\cF$, it follows immediately
from the definition of the morphisms in $\cF$ that
the map that sends the morphism $h\colon M' \to M_i$
to the pair $(\iota'_i \circ h, f_i \circ h)$ gives
a bijection from $\Hom_{\cF}(M',M_i)$ to the set of pairs
$(\iota'',f')$ of morphisms such that the diagram
$$
\begin{CD}
M' @>{\iota''}>> M \\
@V{f'}VV @V{f}VV \\
N_i @>{\iota_i}>> N,
\end{CD}
$$
is commutative. This shows that the diagram 
(\ref{eq:cartesian_Mi}) is cartesian.
This proves the claim.
\end{proof}

\begin{lem}\label{lem:fiber_prod_basic}
Let $\cF$ be an $F$-category.
\begin{enumerate}
\item Let $N_1 \xto{f_1} N' \xleftarrow{f_2} N_2$ 
be a diagram in $\cF$. Suppose that
the images of $\pi_0(f_i)\colon \pi_0(N_i) \to \pi_0(N')$
for $i=1,2$ do not intersect.
Then the fiber product $N_1 \times_{N'} N_2$
exists in $\cF$ and is isomorphic to an
initial object of $\cF$.
\item Let $N_1, N_2, N_3$ and $N'$ be objects of
$\cF$. For $i=1,2,3$, let $f_i \colon  N_i \to N'$
be a morphism in $\cF$. We set $N=N_1 \amalg N_2$.
Let $f\colon N \to N'$ denote the morphism in $\cF$ 
given by $f_1$ and $f_2$.
Then the fiber product $N \times_{N'} N_3$
of $f$ and $f_3$ exists if and only if
the fiber product $N_i \times_{N'} N_3$
exists for each $i=1, 2$.
Moreover if fiber product $N \times_{N'} N_3$
of $f$ and $f_3$ exists, then the morphism 
$(N_1 \times_{N'} N_3) \amalg
(N_2 \times_{N'} N_3) \to N \times_{N'} N_3$
induced by the commutative diagrams
\begin{equation}\label{eq:diag_pentagon}
\xymatrix{
& N_i \times_{N'} N_3 \ar[r] \ar[dl] 
& N_3 \ar[d]^{f_3} \\
N_i \ar[r]^{\subset} & N \ar[r]^{f} & N'
}
\end{equation}
for $i=1,2$ is an isomorphism in $\cF$.
\end{enumerate}
\end{lem}

\begin{proof}
Let the notation be as in the claim (1).
Let
\begin{equation}\label{eq:comm_diagram}
\begin{CD}
N @>{g_2}>> N_2 \\
@V{g_1}VV @VV{f_2}V \\
N_1 @>{f_1}>> N'
\end{CD}
\end{equation}
be a commutative diagram in $\cF$.
We then have $\pi_0(f_2) \circ \pi_0(g_2)
= \pi_0(f_1) \circ \pi_0(g_1)$.
Since the images of $\pi_0(f_1)$ and
$\pi_0(f_2)$ are disjoint, it follows that
$\pi_0(N)$ is an empty set.
This shows that, for an object $N$ of $\cF$,
there does not exist a pair $(g_1,g_2)$ of 
morphisms $g_1\colon N \to N_1$
and $g_2\colon N \to N_2$ such that the diagram
(\ref{eq:comm_diagram}) is commutative
if $N$ is not an initial object of $\cF$.
If $N$ is an initial object 
of $\cF$, then it is easy to check that
there exists a unique such pair $(g_1,g_2)$ 
of morphisms.
On the other hand, it follows from
Lemma \ref{lem:fin_conn} that, 
for an object $N$ of $\cF$,
there does not exist a morphism
from $N$ to an initial object of $\cF$
if $N$ is not an initial object of $\cF$.
Hence the fiber product $N_1 \times_{N'} N_2$
exists in $\cF$ and is isomorphic to an
initial object of $\cF$.
This proves the claim (1).

We prove the claim (2).
For an object $M$ of $\cF$, we let 
$h_M$ denote the presheaf $\cF$ represented
by $M$. It follows from Lemma \ref{lem:fin_conn}
that $h_M$ satisfies Conditions (1) and (2)
in Section \ref{sec:cond12}.
Let the notation be as in the claim (2).
First suppose that the fiber product
$N_i \times_{N'} N_3$ exists for each $i=1,2$.
Let $\phi\colon  h_{N_1 \times_{N'} N_3}
\amalg h_{N_2 \times_{N'} N_3}
\to h_{N} \times_{h_{N'}} h_{N_3}$
denote the morphism of presheaves on $\cF$
induced by the diagram (\ref{eq:diag_pentagon})
for $i=1,2$.
Here $h_{N_1 \times_{N'} N_3}
\amalg h_{N_2 \times_{N'} N_3}$ and
$h_{N} \times_{h_{N'}} h_{N_3}$
denote the coproduct and the fiber product,
respectively, in the category of presheaves on $\cF$.
It then follows from Lemma \ref{lem:fin_conn}
that $\phi(M)\colon  h_{N_1 \times_{N'} N_3}(M)
\amalg h_{N_2 \times_{N'} N_3}(M)
\to h_{N}(M) \times_{h_{N'}(M)} h_{N_3}(M)$
is bijective.
Since any representable presheaf on $\cF$ satisfies
Conditions (1) and (2) in Section \ref{sec:cond12},
it follows that the morphism $\phi$ of presheaves
on $\cF$ induces an isomorphism 
$h_{N_1 \times_{N'} N_3} \amalg 
h_{N_2 \times_{N'} N_3}
\xto{\cong} h_{N} \times_{h_{N'}} h_{N_3}$
of presheaves on $\cF$.
This shows that the fiber product $N \times_{N'} N_3$
exists in $\cF$ and the diagrams 
(\ref{eq:diag_pentagon}) for $i=1,2$ induce 
an isomorphism $(N_1 \times_{N'} N_3) \amalg 
(N_2 \times_{N'} N_3) \xto{\cong}
N \times_{N'} N_3$.

Next suppose that the fiber product
$N \times_{N'} N_3$ of $f$ and $f_3$ exists.
It follows from Lemma \ref{lem:aux1} that for $i=1,2$
the fiber product $N_i \times_{N} (N \times_{N'} N_3)$ exists.
This shows that the fiber product
$N_i \times_{N'} N_3$ exists for $i=1,2$.
This completes the proof.
\end{proof}

\begin{cor}\label{cor:fiber_product}
Let $\cF$ be an $F$-category and let $N_1 \xto{f_1} N' \xleftarrow{f_2} N_2$ 
be a diagram in $\cF$.
For $i =1,2$ and for $j \in \pi_0(N_i)$,
let $N_{i,j}$ denote the component at $j$ of $N_i$.
For $j \in \pi_0(N')$, let $N'_j$ denote the component
at $j$ of $N'$.
Then the fiber product $N_1 \times_{N'} N_2$ exists
if and only if the fiber product
$N_{1,j} \times_{N'_{\pi_0(f_1)(j)}} N_{2,j'}$ exists
for any pair $(j,j') \in \pi_0(N_1) \times \pi_0(N_2)$
satisfying $\pi_0(f_1)(j) = \pi_0(f_2)(j)$.
Moreover if the fiber product $N_1 \times_{N'} N_2$ exists,
then it is isomorphic to the coproduct
$\coprod_{(j,j')} N_{1,j} \times_{N'_{\pi_0(f_1)(j)}} N_{2,j'}$,
where $(j,j')$ runs over the pairs 
$(j,j') \in \pi_0(N_1) \times \pi_0(N_2)$
satisfying $\pi_0(f_1)(j) = \pi_0(f_2)(j)$.
\end{cor}

\begin{proof}
This can be proved by induction on the cardinality of 
$\pi_0(N_1) \amalg \pi_0(N_2)$ using Lemma \ref{lem:aux1}
and Lemma \ref{lem:fiber_prod_basic}.
\end{proof}

\subsection{ }
\begin{thm} \label{thm:sheaf_F-conn}
Let $\cF$ be an $F$-category and let
%
Let $\jmath\colon \cC:= \cF^{(0)} \to \cF$ denote 
the inclusion functor.
Suppose that a Grothendieck topology $J$ on $\cC$
is given.
Let $F$ be a presheaf on $\cF$.
Then $F$ is a sheaf with respect to $\jmath_* J$
if and only if its restriction to $\cC$ is a sheaf
with respect to $J$ and Conditions (1) and (2)
in Section \ref{sec:cond12} are satisfied.
\end{thm}

\begin{proof}
Let $X$ be an object of $\cF$ and $R$ a sieve on $X$.
It follows from the definition of pushforwards given in
Section \ref{sec:pushforward}, 
the sieve $R$ belongs to $\jmath_* J(X)$ if and only
if the following condition is satisfied: For any object
$Y$ of $\cC$ and for any morphism $f\colon Y \to X$ in $\cF$,
the full subcategory of $\cC_{/Y}$ whose objects are
the morphisms $g\colon Z \to Y$ in $\cC$ such that the
composite $Z \xto{g} Y \xto{f} X$ is a sieve on $Y$
which belongs to $J(Y)$.
For $i \in \pi_0(X)$, let $X_i$ denote the component of $X$
at $i$. Then latter condition is equivalent to the following
condition: For any $i \in \pi_0(X)$, 
the full subcategory of $\cC_{/X_i}$ whose objects are
the morphisms $g\colon Z \to X_i$ in $\cC$ such that the
composite $Z \xto{g} X_i \xto{f} X$ is a sieve on $Y$
which belongs to $J(X_i)$.
From this one can see easily that if $F$ is a sheaf on
$(\cF, \jmath_* J)$, then 
its restriction to $\cC$ is a sheaf
with respect to $J$ and Conditions (1) and (2)
in Section \ref{sec:cond12} are satisfied.

It remains to prove the converse. Suppose that a presheaf
$F$ on $\cF$ satisfies Conditions (1) and (2)
in Section \ref{sec:cond12} and that 
its restriction to $\cC$ is a sheaf
with respect to $J$.
Let $X$ be an object of $\cF$ and let $R$ be
a sieve on $X$ which belongs to $(\jmath_* J)(X)$.
We will prove that the natural map
$F(X) \cong \Hom_{\cF}(\frh_{\cF}(X),F)
\to \Hom_{\cF}(\frh_{\cF}(R),F)$ is bijective.

To do this let us introduce some more notation.
For $i \in \pi_0(X)$,
 let $X_i$ denote the component of $X$ at $i$.
Let $R_i$ denote the full subcategory of $\cC_{/X_i}$
whose objects are the morphisms $h\colon Y \to X_i$ such that
the composite $Y \to X_i \to X$ belongs to $R$.
Then $R_i$ is a sieve on $X_i$ regarded as an object of $\cC$.
%
%
%
For any object $Y$ of $\cC$, we have a natural bijection
$$
\frh_{\cF}(R)(Y)
\xto{\cong} \coprod_{i \in \pi_0(X)}\frh_{\cC}(R_i)(Y)
$$
and for any object $Y$ of $\cF$ the map
$$
\frh_{\cF}(R)(Y)
\to \prod_{j \in \pi_0(Y)}
\frh_{\cF}(R) (Y_j)
$$
is injective, 
where $Y_j$ denotes the component of $Y$ at $j$.
Since $F$ satisfies Conditions (1) and (2) in Section
\ref{sec:cond12}, we have a natural bijection
$$
\Hom_{\Presh(\cF)}(\frh_{\cF}(R),F) 
\cong \prod_{i \in \pi_0(X)} 
\Hom_{\Presh(\cC)}(\frh_{\cC}(R_i),F|_{\cC}).
$$
Since $F|_{\cC}$ is a sheaf with respect to $J$,
we have
$$
\prod_{i \in \pi_0(X)} \Hom_{\Presh(\cC)}(\frh_{\cC}(R_i),F|_{\cC})
\cong \prod_{i \in \pi_0(X)} F(X_i)
\cong F(X).
$$
This shows that $F$ is a sheaf with respect to $\jmath_*J$.
This proves the claim.
\end{proof}

\subsection{ } \label{sec:Galois_F}
Let $\cF$ be an $F$-category and
Let $f\colon Y \to X$ be a morphism in $\cF$.
We say that $f$ is a Galois covering if there exists a group
$G$ and a homomorphism $\rho\colon G \to \Aut_X(Y)$ such that the
following condition is satisfied: for any connected object
$Z$ of $\cF$, the map $\Hom_\cF(Z,Y) \to \Hom_\cF(Z,X)$
given by the composition with $f$ is a pseudo G-torsor.
When $f$ is a Galois covering and a homomorphism $\rho$ in the
above definition of Galois covering is specified, 
we say that $f$ is a Galois covering with respect to $\rho$.

The authors would like to apologize for that this notion
of Galois covering is not equivalent to that given in
\cite[Def.\ 3.1.2]{Grids}. The difference lies in that,
in our new notion of Galois covering, only connected
objects are allowed as test objects $Z$.
For an $E$-category, \cite[Def.\ 3.1.2]{Grids} gives an appropriate 
notion of Galois covering. However, it is not suitable
for an $F$-category.

We remark that, unlike the notion in \cite[Def.\ 3.1.2]{Grids}, 
an analogous statement of \cite[Lemma 3.1.3]{Grids} is false 
for our new notion of Galois covering for an $F$-category. 
When $f\colon Y \to X$ is a Galois covering in an $F$-category,
the homomorphism $\rho\colon  G \to \Aut_X(Y)$ is not necessarily bijective.
Moreover $\Aut_X(Y)$ may not be equal to $\End_X(Y)$ in general.
The morphism $\rho$ is always injective unless $Y$ is not an initial object.
However, the image of $\rho$ is not uniquely determined by $f$.
The same morphism $f$ may simultaneously be a Galois covering
with respect to two homomorphism $\rho_1,\rho_2$ with different images.

\begin{lem}\label{lem:Galcov}
Let $\cF$ be an $F$-category and
let $f\colon  N \to N'$ be a morphism in $\cF$.
Let us write $N$ and $N'$ as finite coproducts
$N = \coprod_{i \in \pi_0(N)} N_i$ and
$N' = \coprod_{j \in \pi_0(N')} N'_j$
of connected objects in $\cF$.
Let $G \subset \Aut_{N'}(N)$ be a subgroup.
Let us choose a complete set $S \subset \pi_0(N)$
of representatives of $\quot{\pi_0(N)}{G}$.
For each $s \in S$, let $G_s \subset G$ denote
the stabilizer of $s$.
Then $f$ is a Galois covering in $\cF$ with
Galois group $G$ if and only if the following
conditions are satisfied:
\begin{enumerate}
\item For each $s \in S$, the homomorphism
$G_s \to \Aut_{N'_{\pi_0(f)(s)}}(N_s)$ induced by
the action of $G$ on $N$ is injective.
\item For each $s \in S$, the component
$N_s \to N'_{\pi_0(f)(s)}$ at $s$ of the morphism $f$ 
is a Galois covering in $\cC$ with Galois group $G_s$.
Here we regard $G_s$ as a subgroup of 
$\Aut_{N'_{\pi_0(f)(s)}}(N_s)$ via the injective
homomorphism in (1).
\end{enumerate}
\end{lem}

\begin{proof}
Let $M$ be a connected object in $\cC$.
We have isomorphisms 
$\Hom_{\cF}(M,N) \cong \coprod_{i \in \pi_0(N)}
\Hom_{\cF^{(0)}}(M,N_i)$
and
$\Hom_{\cF}(M,N') \cong \coprod_{j \in \pi_0(N')}
\Hom_{\cF^{(0)}}(M,N'_j)$.
Via these isomorphisms, the map
$\phi \colon  \Hom_{\cF}(M,N) \to \Hom_{\cF}(M,N')$
given by the composition with $f$ is
identified with a map
$\coprod_{i \in \pi_0(N)}
\Hom_{\cF^{(0)}}(M,N_i) \to
\coprod_{j \in \pi_0(N')}
\Hom_{\cF^{(0)}}(M,N'_j)$
which we denote by $\phi'$.
For each $i \in \pi_0(M)$, let
$f_i \colon  N_i \to N'_{\pi_0(f)(i)}$ denote
the component at $i$ of $N$ and
let $\phi_i \colon  \Hom_{\cF}(M,N_i) \to 
\Hom_{\cF}(M,N'_{\pi_0(f)(i)})$
denote the map given by the composition with $f_i$.
Then the map $\phi'$ is given by the map $\phi_i$
for each $i \in \pi_0(N)$.
Hence the map $\phi$ is a pseudo $G$-torsor
if and only if for each $s \in S$, the homomorphism
$G_s \to \Aut_{N'_{\pi_0(f)(s)}}(N_s)$ induced by
the action of $G$ on $N$ is injective and $\phi_s$
is a pseudo $G_s$-torsor.
Hence the claim follows from the definition of
Galois covering.
\end{proof}

\subsection{ }\label{sec:wtcFC}

Let $(\cC,J)$ be a $B$-site. Let $\wt{\cC}$ be as in Section \ref{sec:Ctil}.
Let $\wt{\cFC}$ denote the full subcategory of $\Shv(\cC,J)$ 
whose objects are the finite coproducts 
of objects of $\wt{\cC}$.
For a finite collection $(F_i)_{i \in I}$ of sheaves on 
$(\cC,J)$, we denote by $\coprod_{i \in I} F_i$ a coproduct
in the category $\Shv(\cC,J)$, i.e., 
an object isomorphic to the sheafification
of a coproduct in the category $\Presh(\cC)$.

\begin{lem} \label{lem:wtcC_conn}
Any object of $\wt{\cC}$ is a connected object of the category $\Shv(\cC,J)$.
\end{lem}

\begin{proof}
Let $F$ be an object of $\wt{\cC}$.
Then $F$ is of the form $F = \quot{X}{H}$ 
where $X$ is an object of $\cC$
and $H$ is a moderate subgroup of $\Aut_{\cC}(X)$.
Then for any two presheaves $F_1$, $F_2$ on $\cC$, 
it follows from the explicit description, given in (4.2) of \cite{Grids},
of the sheafification functor $a_J$ that we have
$a_J(F_1)(X) \amalg a_J(F_2)(X) \cong a_J(F_1 \amalg F_2)(X)$.
Hence we have $(a_J(F_1)(X))^{H} \amalg (a_J(F_2)(X))^{H}
\cong (a_J(F_1 \amalg F_2)(X))^{H}$.
Hence for any two sheaves $F_1$, $F_2$ on $(\cC,J)$, we have
$\Hom_{\Shv(\cC,J)}(\quot{X}{H},F_1)
\amalg \Hom_{\Shv(\cC,J)}(\quot{X}{H},F_2)
\cong \Hom_{\Shv(\cC,J)}(\quot{X}{H},F_1 \amalg F_2)$.
%
Hence $F \cong \quot{X}{H}$ is a connected object in the
category $\Shv(\cC,J)$. This completes the proof.
\end{proof}

\begin{cor} \label{cor:FC is an F-category}
The category $\wt{\cFC}$ is an $F$-category. Moreover, an object
of $\wt{\cFC}$ is connected if and only if it is isomorphic in $\wt{\cFC}$ 
to an object of $\wt{\cC}$.
\end{cor}

\begin{proof}
It follows from the definition of $\wt{\cFC}$ that, 
any finite coproduct in $\Shv(\cC,J)$ of objects of $\wt{\cFC}$ 
is isomorphic in $\Shv(\cC,J)$ to an object of $\wt{\cFC}$.
Hence $\wt{\cFC}$ has an initial object and finite coproducts.
Any object of $\wt{\cC}$ is connected as an object of $\wt{\cFC}$
since it follows from Lemma \ref{lem:wtcC_conn} that it is
connected as an object of $\Shv(\cC,J)$.
Since any object of $\wt{\cFC}$ is a coproduct in $\Shv(\cC,J)$
of objects of $\wt{\cC}$, it is also a coproduct in $\wt{\cFC}$
of objects of $\wt{\cFC}$.
This proves that the category $\wt{\cFC}$ is an $F$-category.
As we have remarked in Section \ref{sec:pi_0}, an object $F$ of
$\wt{\cFC}$ is connected if and only if $\pi_0(F)$ consists of
a single element. Hence the second claim follows.
\end{proof}

Let $F$ be an object of $\wt{\cFC}$. 
For any $i \in \pi_0(F)$, we may and will regard,
by Corollary \ref{cor:FC is an F-category}, the component $F_i$ at $i$
as an object of $\wt{\cC}$. Moreover, for any morphism 
$f\colon F \to G$ in $\wt{\cFC}$, we regard the component of $f$ at $i$
as a morphism in $\wt{\cC}$.

\newcommand{\cFT}{\mathcal{FT}}
Let $\wt{\cFT}$ denote the set of morphisms
$f\colon F_1 \to F_2$ in $\wt{\cFC}$ satisfying the following conditions:
For any $i \in \pi_0(F_1)$, the component of $f$ at $i$,
regarded as a morphism in $\wt{\cC}$, belongs to $\wt{\cT}$.
One can show that the set $\wt{\cFT}$ is semi-localizing in $\wt{\cFC}$
in the sense of \cite[Definition 2.3.1]{Grids}.
However, contrary to the case of $\wt{\cT}$, the $A$-topology on
$\wt{\cFC}$ given by $\wt{\cFT}$ is not interesting, since
any morphism from the initial object of $\wt{\cFC}$ belongs to $\wt{\cFT}$.

\subsection{}
We study fiber products.   The main proposition here is that 
fiber products exist in $\wt{\cFC^m}$.   Of course, if we consider general
full subcategories $\wt{\cC} \subset \wt{\cC^m}$,
fiber products may not exist.
\begin{prop} \label{prop:FC_product}
Suppose that $(\cC,J)$ satisfies the following condition:
For any morphism $X \to Y$ in $\cT(J)$, the group $\Aut_Y(X)$ is
a finite group.
Let $F_1 \xto{f_1} F \xleftarrow{f_2} F_2$ be a diagram in $\wt{\cFC^m}$.
Suppose that $f_2$ belongs to $\wt{\cFT^m}$.
Then the fiber product $F_1 \times_F F_2$ exists in $\wt{\cFC^m}$ and
the morphism $F_1 \times_F F_2 \to F_1$ belongs to $\wt{\cFT^m}$.
\end{prop}

To prove Proposition \ref{prop:FC_product}, we need
some preliminaries. In the preliminaries we do not impose
the finiteness condition on $(\cC,J)$ in 
Proposition \ref{prop:FC_product}.

\begin{lem} \label{lem:ZHK}
Let $f\colon (X,H) \to (Y,K)$ be a morphism in $\wt{\cC}^\str$
which belongs to $\wt{\cT}$.
Then there exist an object $Z$ of $\cC$, moderate subgroups
$H', K' \subset \Aut_{\cC}(Z)$ with $H' \subset K'$, 
and isomorphisms $\alpha\colon  (Z,H') \cong (X,H)$ and 
$\beta\colon  (Z,K') \xto{\cong} (Y,K)$ in $\wt{\cC}^\str$ such that
the diagram
$$
\begin{CD}
(Z,H') @>{c}>> (Z,K') \\\
@V{\alpha}VV @VV{\beta}V \\
(X,H) @>{f}>> (Y,K)
\end{CD}
$$
in $\wt{\cC}^\str$, where the upper horizontal arrow $c$ is the morphism 
in $\wt{\cC}^\str$ induced by the identity $Z \xto{=} Z$, is commutative.
Moreover, one can choose $Z$ and $K'$ so that there exists
a morphism $h\colon Z \to Y_0$ in $\cT$ such that $h$ is a Galois
covering and $K' \subset \Aut_{Y_0}(Z)$.
\end{lem}

\begin{proof}
Let us take a representative $X \xleftarrow{m} Z \xto{f'} Y$ of $f$.
Let us choose a morphism $t\colon Y \to Y_0$ in $\cC$ such that
$t$ belongs to $\cT(J)$ and that $K$ is a subgroup of $\Aut_{Y_0}(Y)$.
It follows from Lemma \ref{lem:any representative} that $f'$ belongs
to $\cT(J)$. Since $\cT(J)$ has enough Galois coverings, we may assume,
by replacing $m$ and $f'$ with their composites with a suitable
morphism to $Z$ which belongs to $\cT(J)$, that the composite 
$t \circ f'$ is a Galois covering in $\cC$. 
It follows from Lemma \ref{lem:i_Galois} that $\imath^\str(t \circ f)$
is a Galois covering in $\wt{\cC}^\str$.
Moreover, Lemma 3.1.3 of \cite{Grids} and the argument 
of the proof of Lemma \ref{lem:i_Galois} shows that the 
homomorphism
$\phi\colon \Aut_{Y_0}(Z) \to \Aut_{\imath^\str(Y_0)}(\imath^\str(Z))$
of groups induced by the functor $\imath^\str$ is an isomorphism.
Since the morphism $\imath^\str(t \circ f)$ factors through
both $q_{X,H} \circ \imath^\str(m)$ and
$q_{Y,K} \circ \imath^\str(f)$,
it follows that
the two morphisms 
$q_{X,H} \circ \imath^\str(m)$ and
$q_{Y,K} \circ \imath^\str(f)$
are Galois coverings in $\wt{\cC}^\str$.
Let $H' = \phi^{-1}(\Aut_{(X,H)}(\imath^\str(Z)))$
and $K' = \phi^{-1}(\Aut_{(Y,K)}(\imath^\str(Z)))$.
The two groups $H'$ and $K'$ are moderate subgroups of
$\Aut(Z)$ since they are subgroups of $\Aut_{Y_0}(Z)$.
Since $q_{Y,K} \circ \imath^\str(f)$ factors through
$q_{X,H} \circ \imath^\str(m)$, we have $H' \subset K'$.
It is clear that $(Z,H')$ and $(Z,K')$ are
quotient objects of $Z$ by $H'$ and $K'$, respectively,
and that the morphisms $q_{Z,H'}$ and $q_{Z,K'}$
are the canonical quotient morphisms.
Hence it follows from Lemma 3.2.4 of \cite{Grids} 
and the universality of quotient objects 
that there exist unique isomorphisms
$\alpha\colon  (Z,H') \xto{\cong} (X,H)$ and
$\beta\colon  (Z,K') \xto{\cong} (Y,K)$ satisfying
$q_{X,H} \circ \imath^\str(m) = \alpha \circ q_{Z,H'}$
and $q_{Y,K} \circ \imath^\str(f) = \beta \circ q_{Z,K'}$.
This implies that the equality
$c \circ \alpha^{-1} \circ \imath^\str(m)
= \beta^{-1} \circ f \circ \imath^\str(m)$ holds.
Since $\imath^\str(m)$ is an epimorphism, 
we have $c \circ \alpha^{-1} = \beta^{-1} \circ f$, as desired.
This completes the proof of the first claim.
The second claim follows from our construction of 
$Z$ and $K'$.
\end{proof}

\begin{lem} \label{lem:quot_Ctil}
Let $F$ be an object of $\wt{\cC}$ and let $H$ be
a subgroup of $\Aut_{\wt{\cC}}(F)$. Suppose that
there exists a morphism $f\colon F \to F_0$ in $\wt{\cC}$
such that $f$ belongs to $\wt{\cT}$ and that
$H$ is a subgroup of $\Aut_{F_0}(F)$.
Let $\quot{F}{H}$ denote the sheaf on $(\cC,J)$ associated 
with the presheaf that
sends an object $X$ of $\cC$ to the set $\quot{(F(X))}{H}$.
Then $\quot{F}{H}$ is isomorphic in $\Shv(\cC,J)$ 
to an object of $\wt{\cC^m}$.
\end{lem}

\begin{proof}
By using Lemma \ref{lem:ZHK},
we may and will assume that $F = \quot{X}{K}$,
$F_0 = \quot{X}{L}$ for some object $X$ of $\cC$
and moderate subgroups $K,L \subset \Aut_{\cC}(X)$
with $K \subset L$, that $f\colon \quot{X}{K} \to \quot{X}{L}$
is the morphism induced by the identity morphism of $X$,
and that there exists a morphism $h\colon X \to X_0$ in $\cC$
such that $h$ is a Galois covering which belongs to $\cT$ 
and that $L$ is a subgroup of $\Aut_{X_0}(X)$.

The morphism $\iota(h)\colon \quotid{X} \to \quotid{X_0}$ in $\wt{\cC}$
factors through the morphism $q_{X,L}$ regarded as a morphism
$\quotid{X} \to \quot{X}{L}$ in $\wt{\cC}$,
and the latter morphism factors through the morphism $q_{X,K}$
regarded as a morphism
$\quotid{X} \to \quot{X}{K}$ in $\wt{\cC}$.
This implies that $H$ is a subgroup of $\Aut_{\quotid{X_0}}(\quot{X}{K})$.
It follows from Lemma \ref{lem:i_Galois} that 
$\iota(h)\colon \quotid{X} \to \quotid{X_0}$
is a Galois covering in $\wt{\cC}$.
Let $H'$ denote the set of elements of 
$\Aut_{\quotid{X_0}}(\quotid{X})$ which descends to an element of $H$
via $q_{X,K}$. Clearly $H'$ is a subgroup of 
$\Aut_{\quotid{X_0}}(\quotid{X})$ which contains $K$.
It is clear that the map $H' \to H$ that sends $h' \in H'$ 
to the unique element to which $h$ descends is a homomorphism of
groups. It follows from Lemma 4.2.6 (1) of \cite{Grids} that
this map is surjective. Thus, we have the exact sequence
\begin{equation} \label{eq:KH'H}
1 \to K \to H' \to H \to 1.
\end{equation}

It follows from Lemma \ref{lem:i_Galois} that
the functor $\iota$ induces an isomorphism
$\Aut_{X_0}(X) \xto{\cong} \Aut_{\quotid{X_0}}(\quotid{X})$
of groups. Let $H''$ denote the subgroup of $\Aut_{X_0}(X)$
that corresponds to the subgroup $H'$ of 
$\Aut_{\quotid{X_0}}(\quotid{X})$
via this isomorphism.
Clearly $H''$ is a moderate subgroup of $\Aut_{\cC}(X)$.
By the exact sequence of \eqref{eq:KH'H}, we see that
$\quot{F}{H}$ is isomorphic to $\quot{X}{H''}$
in $\Shv(\cC,J)$. This proves the claim.
\end{proof}

For a sheaf $F$ on $(\cC,J)$ and for a subgroup
$H \subset \Aut_{\Shv(\cC,J)}(F)$, let
$\quot{F}{H}$ denote the sheaf on $(\cC,J)$ associated 
with the presheaf that
sends an object $X$ of $\cC$ to the set $\quot{(F(X))}{H}$.

\begin{lem}\label{lem:quot_FCtil}
Let $F$ be an object of $\wt{\cFC}$ and let $H$ be
a subgroup of $\Aut_{\wt{\cFC}}(F)$. Suppose that
there exists a morphism $f\colon F \to F'$ in $\wt{\cFC}$
such that $f$ belongs to $\wt{\cFT}$ and that
$H$ is a subgroup of $\Aut_{F'}(F)$.
Then $\quot{F}{H}$ is an object of $\wt{\cFC^m}$.
\end{lem}

\begin{proof}
Since $F$ is an object of $\wt{\cFC}$, it is of the form
$F = \coprod_{i \in I} F_i$, where $I$ is a finite set
and $F_i$ is an object of $\wt{\cC}$ for each $i \in I$.
The action of $H$ on $F$ induces an action of $H$ on $I$.
Let us choose a complete set $S \subset I$ 
of representatives of $\quot{I}{H}$.
For $i \in S$ let $H_i \subset H$ denote the stabilizer of
$i \in I$. Then
the action of $H$ on $F$ induces an action of $H_i$ on $F_i$.
It is easy to check that $\quot{F}{H}$ is a coproduct 
$\coprod_{i \in S} \quot{F_i}{H_i}$ 
in the category $\Shv(\cC,J)$.
Hence the claim follows from Lemma \ref{lem:quot_Ctil}.
\end{proof}

\begin{lem} \label{lem:aJ_quot}
Let $F$ be a presheaf on $\cC$
and let $H$ be a subgroup of $\Aut_{\Presh(\cC)}(F)$.
Suppose that $F$ satisfies the following two conditions:
\begin{enumerate}
\item For any morphism $f\colon Y\to X$, the map
$F(f) \colon  F(X) \to F(Y)$ is injective.
\item For any object $X$ of $\cC$, there exists
a Galois covering $f\colon Y \to X$ in $\cC$ which belongs to $\cT(J)$
such that $H$ acts faithfully on the set $F(Y)$.
\end{enumerate}
Let $\quot{(F(-))}{H}$ denote the presheaf on $\cC$
that associates $\quot{F(X)}{H}$ with any object $X$ of $\cC$.
Then the morphism $\quot{(F(-))}{H} \to \quot{(a_J(F)(-))}{H}$ 
in $\Presh(\cC)$ given by the adjunction morphism 
$F \to a_J(F)$ induces an isomorphism 
$a_J(\quot{(F(-))}{H})\xto{\cong} a_J(\quot{(a_J(F)(-))}{H})$
in $\Shv(\cC,J)$.
\end{lem}

\begin{proof}
Let $X$ be an object of $\cC$.
It suffices to prove that the map
\begin{equation} \label{eq:aJ_quot}
a_J(\quot{(F(-))}{H})(X) \to a_J(\quot{(a_J(F)(-))}{H})(X)
\end{equation}
is bijective.

First let us assume that $H$ acts faithfully on the
set $F(X)$. 
Let us consider the 
cofiltered category $\Gal/X$ introduced
in Section 4.4.2 of \cite{Grids}.
Let $(Y,f)$ be an object of $\Gal/X$. Suppose that
$\sigma \in \Aut_X(Y)$ and $h \in H$ satisfying
that the composite $F(Y) \xto{F(\sigma)} F(Y) \xto{h} F(Y)$
is equal to the identity.
This and Condition (2) on $F$ implies that $F(X) \xto{h} F(X)$
is equal to the identity.
Hence it follows from our assumption on $X$ that $h=1$.
This implies that the map $\quot{(F(Y)^{\Aut_X(Y)})}{H}
\to (\quot{F(Y)}{H})^{\Aut_X(Y)}$ is bijective.
Passing to the colimit with respect to $(Y,f)$ we conclude
that the map $a_J(\quot{(F(-))}{H})(X) \to (\quot{(a_J(F)(-))}{H})(X)$
is bijective when $H$ acts faithfully on the set $F(X)$. 
Condition (1) on $F$ implies that
$H$ acts faithfully on the set $F(Y)$ for any object $(Y,f)$
of $\Gal/X$.
This shows that the map \eqref{eq:aJ_quot} is bijective as desired
when $H$ acts faithfully on the set $F(X)$.

Now we let $X$ be an arbitrary object of $\cC$.
By Condition (2) on $F$, we can find a Galois covering $f\colon Y \to X$
such that $f$ belongs to $\cT(J)$ and that $H$ acts faithfully on $F(Y)$.
Let $\alpha$ denote the map  \eqref{eq:aJ_quot} with $X$ 
replaced by $Y$. The argument in the last paragraph shows that
the map $\alpha$ is bijective.
It is clear that the map $\alpha$ is $\Aut_Y(X)$-equivariant.
Since both $a_J(\quot{(F(-))}{H})$ and $a_J(\quot{(a_J(F)(-))}{H})$
are sheaves, we conclude, by taking the $\Aut_Y(X)$-invariant parts
of the domain and the codomain of the map $\alpha$,
that the map \eqref{eq:aJ_quot} is bijective.
\end{proof}

We use the symbol $\times^\Presh$ to denote fiber products
in the category $\Presh(\cC)$.

\begin{lem} \label{lem:cFC_limit}
Let $F_1 \xto{f_1} F \xleftarrow{f_2} F_2$ be a diagram in $\wt{\cFC}$.
Suppose that a fiber product $F_1 \times_F F_2$ exists in $\wt{\cFC}$.
Then the natural morphism
$F_1 \times_F F_2 \to F_1 \times^\Presh_F F_2$
in $\Presh(\cC)$ is an isomorphism.
\end{lem}

\begin{proof}
For any object $F'$ of $\wt{\cFC}$ and for any object $X$ of $\cC$, we have
isomorphisms
$$
F'(X) \cong \Hom_{\Presh(\cC)}(\frh_X,F')
\cong \Hom_{\Shv(\cC,J)}(a_J(\frh_X),F')
\cong \Hom_{\wt{\cFC}}(a_J(\frh_X),F').
$$
Hence
\begin{align*}
(F_1 \times_F F_2)(X)
\cong & \Hom_{\wt{\cFC}}(a_J(\frh_X),F_1 \times_F F_2) \\
\cong & \Hom_{\wt{\cFC}}(a_J(\frh_X),F_1)
\times_{\Hom_{\wt{\cFC}}(a_J(\frh_X),F)}
\Hom_{\wt{\cFC}}(a_J(\frh_X),F_2) \\
\cong & F_1(X) \times_{F(X)} F_2(X).
\end{align*}
This shows that $F_1 \times_F F_2$ is a fiber product
of the diagram $F_1 \xto{f_1} F \xleftarrow{f_2} F_2$ 
in the category $\Presh(\cC)$,
which proves the claim.
\end{proof}

\begin{cor}
The natural inclusion $\wt{\cFC} \subset \wt{\cFC^m}$
preserves fiber products.
\end{cor}
\begin{proof}
Immediate from the lemma.
\end{proof}

\begin{lem} \label{lem:limit_of_quotient}
Let $F_1 \xto{f_1} F \xleftarrow{f_2} F_2$ be a diagram in $\wt{\cFC}$.
Let $H$ be a subgroup of $\Aut_{\wt{\cFC}}(F_2)$.
Suppose that $H \subset \Aut_{F}(F_2)$ and that
there exists morphism $f_2'\colon  F_2 \to F_3$ in $\wt{\cFT}$
such that $H$ is a subgroup of $\Aut_{F_3}(F_2)$.
Suppose that the sheaf $\quot{H}{F_2}$ on $(\cC, J)$
is an object of $\wt{\cFC}$.
Suppose that a fiber product $F_1 \times_F F_2$ exists in $\wt{\cFC}$
and that the quotient $\quot{(F_1 \times_F F_2)}{H}$ is
an object of $\wt{\cFC}$.
Then a fiber product $F_1 \times_F (\quot{F_2}{H})$
exists in $\wt{\cFC}$ and is isomorphic to 
$\quot{(F_1 \times_F F_2)}{H}$.
\end{lem}

\begin{proof}
The presheaf $F' = F_1 \times^\Presh_F F_2$
satisfies Conditions (1) (2) in Lemma \ref{lem:aJ_quot}.
It follows from Lemma \ref{lem:cFC_limit}
that
the natural morphism
$F_1 \times_F F_2 \to F'$
in $\Presh(\cC)$ is an isomorphism.
In particular we have
$\quot{F'}{H} \cong \quot{(F_1 \times_F F_2)}{H}$. 
By definition, we have 
$$
\quot{F'}{H} = a_J(\quot{(F'(-))}{H})
\cong a_J(F_1 \times^\Presh_F \quot{(F_2(-))}{H}).
$$
Hence it suffices to show that the morphism
$a_J(F_1 \times^\Presh_F \quot{(F_2(-))}{H})
\to a_J(F_1 \times^\Presh_F \quot{F_2}{H})$
is an isomorphism.
Since $\wt{\cC}$ is an $E$-category and
since 
$$F_2(X) \cong \Hom_{\wt{\cFC}}(\quotid{X},F_2)$$
for any object $X$ of $\cC$, it follows that
the map $F_2(X) \to F_2(Y)$ is injective for any
morphism $f\colon Y \to X$ in $\cC$.
This implies that the map $\quot{F_2(X)}{H}
\to (\quot{F_2}{H})(X)$ is injective for any object $X$
of $\cC$. Hence the map
$$
(F_1 \times^\Presh_F \quot{(F_2(-))}{H})(X)
\to (F_1 \times^\Presh_F \quot{F_2}{H})(X)
$$
is injective for any object $X$ of $\cC$.

Let $X$ be an arbitrary object of $\cC$.
It suffices to show that the map 
\begin{equation} \label{F_1FF_2}
a_J(F_1 \times^\Presh_F \quot{(F_2(-))}{H})(X)
\to a_J(F_1 \times^\Presh_F \quot{F_2}{H})(X)
\end{equation}
is surjective.
Let $x$ be an element of 
$a_J(F_1 \times^\Presh_F \quot{F_2}{H})(X)$
then there exist an object $(Y,f)$ of $\Gal/X$ and
an object $(Z,h)$ of $\Gal/Y$ such that $f \circ h$
is a Galois covering and that $x$ belongs to
the image of the map 
$(F_1(Y) \times_{F(Y)} (\quot{F_2(Z)}{H})^{\Aut_Y(Z)})^{\Aut_X(Y)}
\to a_J(F_1 \times^\Presh_F \quot{F_2}{H})(X)$.
Lemma 4.2.6 of \cite{Grids} gives an isomorphism
$(F_1(Y) \times_{F(Y)} (\quot{F_2(Z)}{H})^{\Aut_Y(Z)})^{\Aut_X(Y)}
\cong (F_1(Z) \times_{F(Z)} \quot{F_2(Z)}{H})^{\Aut_X(Z)}$.
This shows that $x$ belongs to the image of the map \eqref{F_1FF_2}.
Hence the map \eqref{F_1FF_2} is surjective.
This completes the proof.
\end{proof}

\begin{proof}[Proof of Proposition \ref{prop:FC_product}]
We use Corollaries \ref{cor:FC is an F-category} and
\ref{cor:fiber_product}. It suffices to prove the claim
when $F_1$, $F_2$, $F$ are objects of $\wt{\cC^m}$.
By using Lemma \ref{lem:ZHK}, we may assume that 
$F_2= \quot{X}{H}$, $F=\quot{X}{K}$ for some object $X$ of $\cC$
and for some moderate subgroups $H,K \subset \Aut_\cC(X)$ with
$H \subset K$, and that $f_2$ is the morphism induced by the
identity morphism of $X$.
Moreover, by using Lemmas \ref{lem:quot_FCtil} and 
\ref{lem:limit_of_quotient}, we are
reduced to the case where $H=\{1\}$.
Let us take an object $(Y,L)$ of $\wt{\cC^m}^\str$ satisfying 
$F_1 = \quot{Y}{L}$. Let us regard $f_1$ as a morphism
$(Y,L) \to (X,K)$ in $\wt{\cC^m}^\str$. Let us take a representative
$Y \xleftarrow{m} Z \xto{f} X$ of $f_1$.
Let us choose a morphism $h\colon Y \to Y_0$ in $\cC$ which belongs to $\cT$
such that $L$ is a subgroup of $\Aut_{Y_0}(Y)$.
By replacing $m$ and $f$ with their composites with a suitable
morphism to $Z$ which belongs to $\cT$, we may assume that
the composite $h \circ m$ is a Galois covering in $\cC$ 
which belongs to $\cT$.
Let $L'$ denote the set of elements $\sigma \in \Aut_{Y_0}(Z)$ such
that $\sigma$ descends to an element of $L$ via $m$ in the sense of
Section 4.2.1 of \cite{Grids}. It is clear that
$L'$ is a subgroup of $\Aut_{Y_0}(Z)$.
The last sentence in the first paragraph of
Section 4.2.1 of \cite{Grids} gives a map 
$\phi\colon  L' \to L$ characterized
by the following property: Any $\sigma \in L'$ 
descends to $\phi(\sigma) \in L$ via $m$.
It is easy to check that the map $\phi$ is a homomorphism 
of groups.
It follows from Lemma 4.2.6 (1) of \cite{Grids} that
$\phi$ is surjective.
From this one can see easily that 
the morphism $m\colon Z \to Y$ induces an isomorphism
$\quot{Z}{L'} \xto{\cong} \quot{Y}{L} = F_1$.
Let $\sigma \in L'$. Since $q_{Y,L} \circ \imath^\str(m \circ \sigma^{-1})
= q_{Y,L} \circ \imath^\str(\phi(\sigma^{-1}) \circ m) 
= q_{Y,L} \circ m$, the triple
$(f_1;m,f\circ \sigma)$ is a representative of $f_1$.
Hence it follows from Lemma \ref{lem:exists_g} that
there exists an element $k \in K$ satisfying $f \circ \sigma = k \circ f$,
i.e., $\sigma$ descends to $k$ via $f$.
Since $f$ is an epimorphism in $\cC$, such an element $k$ is unique.
By sending $\sigma$ to $k$, we obtain a map $\psi\colon L' \to K$.
It is easy to check that this map is a homomorphism of groups.

Let $G = \coprod_{k \in K} \quotid{Z}$.
It follows from the finiteness condition on $(\cC,J)$
that $G$ is an object of $\wt{\cFC^m}$.
For $\sigma \in L'$, let $\sigma'\colon G \to G$
be the morphism in $\wt{\cFC^m}$ defined as follows:
$\pi_0(\sigma')\colon  K \to K$ is the map that sends $k$ for 
$k \in K$ to $k\psi(\sigma)^{-1}$, and for any $k \in K$ the component of
$\sigma'$ at $k$ is equal to $\iota(\sigma)\colon  \quotid{Z} \to \quotid{Z}$.
It is clear that $\sigma'$ is an automorphism of $G$
in $\wt{\cFC^m}$ and the map $\eta\colon L' \to \Aut_{\wt{\cFC^m}}(G)$
which sends $\sigma$ to $\sigma'$ is a homomorphism of groups.
Via this map we let $L'$ act from the left on the object 
$G$ of $\wt{\cFC^m}$.
Let $\quot{G}{L'}$ denote the sheaf on $(\cC,J)$ associated 
with the presheaf which
sends an object $X$ of $\cC$ to the set $\quot{(G(X))}{L'}$.
It is an object of $\wt{\cFC^m}$ since it is a coproduct
in $\Shv(\cC,J)$ of the family $(\quot{Z}{\ker \psi})_{k \in S}$,
where $S \subset K$ is a complete set of representatives of 
$K/\psi(L')$.
Let $m'\colon G \to \quotid{Y}$ denote the map all of whose components
are the map $\iota(m)\colon \quotid{Z} \to \quotid{Y}$.
The morphism $m'$ is $L'$-equivariant, where $L'$
acts on $\quotid{Y}$ via the homomorphism $\phi\colon L' \to L$.
Hence $m'$ induces the morphism 
$f'_2\colon \quot{G}{L'} \to \quot{Y}{L} = F_1$.

In this paragraph we construct a morphism
$\quot{G}{L'} \to \quotid{X}$. 
Let $G \to \quotid{X}$ be a morphism in $\wt{\cFC^m}$
whose component at $k \in K$ is equal to the morphism
$\imath(k\circ f)$.
We let the group $K$ act on $G$ as follows:
For any $k \in K$ and the action of $k$ on $\pi_0(G)=K$
is the multiplication by $k$ from the left and 
the action of $k$ on each component of $G$ is
the identity morphism of $\quotid{Z}$.
Then the morphism $\wt{f}$ is $K$-equivariant.
Observe that the action of $L'$ on $G$ commutes with the action of $K$ 
on $G$. One can check easily from the definition of the
homomorphism $\wt{\eta}$ that $\wt{\eta}(L')$ is a subgroup of 
$\Aut_{\quotid{X}}(G)$.
This implies that $\eta(L')$ is a subgroup of $\Aut_{\quotid{X}}(G)$.
Hence the morphism $\wt{f}_1$ induces a morphism 
$f'_1\colon  \quot{G}{L'} \to \quotid{X}=F_2$.
Thus, we obtain a diagram
$$
\begin{CD}
\quot{G}{L'} @>{f'_1}>> F_2 \\
@V{f'_2}VV @VV{f_2}V \\
F_1 @>{f_1}>> F
\end{CD}
$$
\label{diag:cartesian}
in $\wt{\cFC^m}$. Let us consider the map
$p'\colon  G \to \quot{G}{L'}$ induced by the identity map $G \to G$.
One can check that $p'$ is an epimorphism in $\wt{\cFC^m}$ 
by using that $\wt{\cC^m}$ is an $E$-category and that
$\pi_0(p')$ is surjective.
Since $f_1 \circ f'_2 \circ p' = f_2 \circ f'_1 \circ p'$,
it follows that the diagram above is commutative.

Since we have an isomorphism
$\quot{Z}{L'} \xto{\cong} F_1$
and since $G$ is a fiber product 
$\quotid{Z} \times_F F_2$ in $\wt{\cFC^m}$, 
it follows from Lemma \ref{lem:limit_of_quotient}
that the diagram above is cartesian.

Now suppose $f_2$ belongs to $\wt{\cFT^m}$.
Then by the following lemma, we 
see that the morphism
$F_1 \times_F F_2 \to F_1$
belongs to $\wt{\cFT^m}$.
\begin{lem}
Let $f\colon  Z \to Y$ be a morphism in $\wt{\cC}$
that belongs to $\wt{\cT}$.   
Let $H_Z \subset \Aut(Z)$ and 
$H_Y \subset \Aut(Y)$ be subgroups.
Suppose there is a group homomorphism $H_Z \to H_Y$
and that the morphism $f$ is equivariant with respect to the $H_Z$-action.
Then the induced morphism 
$\quot{H_Z}{Z} \to \quot{H_Y}{Y}$
belongs to $\wt{\cT}$. 
\end{lem}
\begin{proof}
The claim readily follows by taking the representative of the induced map
$(Z, H_Z) \to (Y, H_Y)$ to be
$Z \xleftarrow{=} Z \xto{f} Y$. 
\end{proof}

This completes the proof.
\end{proof}

\begin{rmk}
In a category where fiber products exist,
there is a simpler and more familiar definition 
of a Galois covering.   Since by 
Proposition \ref{prop:FC_product}
fiber products exist, we have the following 
criterion.

Let $f\colon F_1 \to F_2$ be a morphism in $\wt{\cFC^m}$ and
let $\rho\colon  G \to \Aut_{F_2}(F_1)$ be a homomorphism of groups. 
Then, 
the morphism $f$ is a Galois covering if and only if the diagram
$$
\begin{CD}
\coprod_{g \in G} F_1 @>{\coprod_g \rho(g)}>> F_1 \\
@V{\coprod_g \id_{F_1}}VV @VV{f}V \\
F_1 @>{f}>> F_2
\end{CD}
$$
in $\wt{\cFC}$ is cartesian.
\end{rmk}

\begin{lem} \label{lem:FC_product_isom}
Suppose that, for each object $X$ of $\cC$, 
the overcategory $\cC(\cT(J))_{/X}$
satisfies at least one of the two conditions 
in Section 5.8.1 of \cite{Grids}.
Let us fix a grid $(\cC_0,\iota_0)$ of $(\cC,J)$.
Let $\omega = \omega_{(\cC_0,\iota_0)}$ denote the
associated fiber functor.
Let $F_1 \xto{f_1} F \xleftarrow{f_2} F_2$ be a diagram in $\wt{\cFC}$.
Suppose that a fiber product $F_1 \times_F F_2$ exists in $\wt{\cFC}$.
Then the natural morphism
$\omega(F_1 \times_F F_2) \to \omega(F_1) \times_{\omega(F)} \omega(F_2)$
of sets, where the right-hand side is the fiber product in the category
of sets, is an isomorphism.
\end{lem}

\begin{proof}
It follows from Lemma \ref{lem:cFC_limit} that the 
$F_1 \times_F F_2$ is a fiber product in the category $\Presh(\cC)$.
In particular it is a fiber product in the category $\Shv(\cC,J)$.
It follows from \cite[Thm.~5.8.1]{Grids} 
that the functor $\omega$ induces the equivalence of categories
between $\Shv(\cC,J)$ and the category of smooth left 
$M_{(\cC_0,\iota_0)}$-sets.
Since fiber products in the latter category 
can be computed set theoretically, the claim follows.
\end{proof}

\begin{lem}\label{lem:last}
Let
\begin{equation}\label{eq:diagram_last}
\begin{CD}
F_1 @>{h}>> F_2 \\
@V{f_1}VV @V{f_2}VV \\
F'_1 @>{h'}>> F'_2
\end{CD}
\end{equation}
be a commutative diagram in $\wt{\cC}$.
Suppose that $f_1$ and $f_2$ are Galois coverings
in $\wt{\cC}$ with finite Galois groups $G_1$ and $G_2$,
respectively, which belong to $\wt{\cT}$, and that
there exists an isomorphism
$\phi\colon G_1 \xto{\cong} G_2$
satisfying $h \circ g = \phi(g) \circ h$
for all $g \in G_1$.
Then the diagram \eqref{eq:diagram_last}
is cartesian in $\wt{\cFC}$.
\end{lem}

\begin{proof}
We may assume that $\wt{\cC}=\wt{\cC^m}$.
It follows from Proposition \ref{prop:FC_product}
that the fiber product $F'_1 \times_{F'_2} F_2$
exists in $\wt{\cFC}$ and the first projection
morphism $\pr_1 \colon  F'_1 \times_{F'_2} F_2 \to F'_1$
belongs to $\wt{\cFT}$.
Hence it follows from Lemma \ref{lem:tilde semi-localizing}
that the morphism $h\colon F_1 \to
F'_1 \times_{F'_2} F_2$
induced by the commutative diagram \eqref{eq:diagram_last}
belongs to $\wt{\cFT}$.

We claim that $F'_1 \times_{F'_2} F_2$ is a connected object of
$\wt{\cFC}$.
Let us consider the homomorphism
$\psi\colon G_2 \to 
\Aut_{F'_1}(F'_1 \times_{F'_2} F_2)$
of groups that sends $g \in G_2$
to $\id_{F'_1} \times g$.
It can be checked easily that the the
homomorphism $\psi$ is injective and that
the morphism $\pr_1$ is a Galois covering with
the Galois group $\psi(G_2)$.
It follows from Lemma \ref{lem:Galcov},
that the maps $\pi_0(f_1)\colon \pi_0(F_1) \to \pi_0(F'_1)$ 
and $\pi_0(\pr_1)\colon \pi_0(F'_1 \times_{F'_2} F_2)
\to \pi_0(F'_1)$ induce bijections
$\quot{\pi_0(F_1)}{G_1} \xto{\cong} \pi_0(F'_1)$
and $\quot{\pi_0(F'_1 \times_{F'_2} F_2)}{\psi(G_2)}
\xto{\cong} \pi_0(F'_1)$.
By assumption,
we have $h \circ g = \psi(\phi(g)) \circ h$
for any $g \in G_1$,
Hence the map
$\quot{\pi_0(F_1)}{G_1} \to
\quot{\pi_0(F'_1 \times_{F'_2} F_2)}{\psi(\phi(G_1))}$
induced by the map $\pi_0(h)\colon \pi_0(F_1) \to
\pi_0(F'_1 \times_{F'_2} F_2)$ is bijective.
This shows that the map $\pi_0(h)$ is surjective,
which proves the claim.

Let us regard $F'_1 \times_{F'_2} F_2$ as an object of $\wt{\cC}$.
We apply Lemma \ref{lem:mono} to the morphism $h$
in $\wt{\cC}$ which belongs to $\wt{\cT}$.
To prove that $h$ is an isomorphism,
it suffices to show that $h$
is a monomorphism in $\wt{\cC}$.

Let $G$ be an object in $\wt{\cC}$.
To prove that $h$ is a monomorphism,
it suffices to show that the map
$\varphi\colon  \Hom_{\wt{\cC}}(G,F_1) \to
\Hom_{\cCotu{d}}(G,F'_1)
\times_{\Hom_{\wt{\cC}}(G,F'_2)}
\Hom_{\wt{\cC}}(G,F_2)$
induced by the commutative 
diagram \eqref{eq:diagram_last} is injective.
Since both $f_1$ and $\pr_1$ are Galois coverings,
the maps 
$\Hom_{\wt{\cC}}(G,F_1) \to \Hom_{\wt{\cC}}(G,F'_1)$
and
$\Hom_{\wt{\cC}}(G,F'_1 \times_{F'_2} F_2) 
\to \Hom_{\wt{\cC}}(G,F'_1)$
given by the compositions with $f_1$ and $\pr_1$,
respectively, are pseudo $G_1$-torsors.
Here the group $G_1$ acts on the set
$\Hom_{\wt{\cC}}(G,F'_1 \times_{F'_2} F_2)$
via the composite $\psi \circ \phi$.
Since $f_1 = \pr_1 \circ h$, this shows
that the map $\varphi$ is injective.
This completes the proof.
\end{proof}

\subsection{On the topology $\iota'_* J$} \label{sec:iota'}
Let $\iota'$ denote the composite $\cC \xto{\iota} \wt{\cC}
\inj \wt{\cFC}$. It follows from Corollary \ref{cor:5563} that
the functor $\Presh(\wt{\cFC}) \to \Presh(\cC)$ given by the
composition with $\iota'$ induces an equivalence
\begin{equation} \label{eq:iota'}
\Shv(\wt{\cFC},\iota'_* J) \xto{\cong} \Shv(\cC,J)
\end{equation}
of categories.

\begin{prop} \label{prop:cFT_main}
Let $F$ be a presheaf on $\wt{\cFC}$.
Then $F$ is a sheaf with respect to $\iota'_* J$
if and only if its restriction to $\wt{\cC}$ is a sheaf
with respect to $\iota_* J$ and Conditions (1) and (2)
in Section \ref{sec:cond12} are satisfied.
\end{prop}

\begin{proof}
Let $\wt{J} = J_{\wt{\cT}}$ and let 
$\jmath\colon \wt{\cC} \inj \wt{\cFC}$ denote the natural inclusion.
By Lemma \ref{lem:compare_JT} we have $\wt{J} = \imath_* J$.
Hence it follows from Lemma \ref{lem:pushforward_transitivity}
that we have $\iota'_* J = \jmath_* \wt{J}$.
Hence the claim follows from Theorem \ref{thm:sheaf_F-conn}
\end{proof}

One can construct a quasi-inverse 
$\nu\colon  \Shv(\cC,J) \to \Shv(\wt{\cFC},\iota'_* J)$ 
to the functor \eqref{eq:iota'}
explicitly follows: Let $F$ be a sheaf on $(\cC,J)$,
then for an object $\coprod_{i\in I} \quot{X_i}{H_i}$
of $\wt{\cFC}$, we set
$$
\nu(F)(\coprod_{i\in I} \quot{X_i}{H_i})
= \prod_{i \in I} F(X_i)^{H_i}.
$$
It is then easy to check that the assignment 
of $\nu(F)(\coprod_{i\in I} \quot{X_i}{H_i})$ to each
object $\coprod_{i\in I} \quot{X_i}{H_i}$ of $\wt{\cFC}$
yields a sheaf $\nu(F)$, well-defined and unique up to 
canonical isomorphisms, on $\wt{\cFC}$.
By sending $\nu(F)$ to any sheaf $F$ on $(\cC,J)$, we obtain
a functor $\nu$.
It is straightforward to check that $\nu$ is a quasi-inverse
to the functor \eqref{eq:iota'}.

\subsection{A complement on the topology $\iota'_* J$}
\label{sec:cFT*}
Let $\wt{\cFT}^* \subset \wt{\cFT}$ denote the subset of
morphisms $f$ which belongs to $\wt{\cFT}$ such that
$\pi_0(f)$ is surjective.

\begin{prop} \label{prop:cFCot_enoughGalois}
Suppose that $Aut_Y(X)$ is a finite group for any 
morphism $f\colon X \to Y$ of $\wt{\cC}$ which belongs to $\wt{\cC}$.
Then $\wt{\cFT}^*$ has enough Galois coverings in the following sense.
For any morphism $f\colon F_1 \to F_2$ in $\wt{\cFC}$ which belongs to 
$\wt{\cFT}^*$, there exists a morphism 
$g\colon F_0 \to F_1$ in $\wt{\cFC}$ which belongs to $\wt{\cFT}^*$
such that the composite $f \circ g$ is a Galois covering in 
$\wt{\cFC}$ in the sense of Section \ref{sec:Galois_F}.
\end{prop}

\begin{proof}
For $j \in \{1,2\}$ and for $i \in \pi_0(F_j)$, let
$F_{j,i}$ denote the component at $i$ of $F_j$.
For $i \in \pi_0(F_1)$, let $f_i \colon  F_{1,i} \to F_{2,\pi_0(f)(i)}$
denote the component at $i$ of $f$.
Let $j \in \pi_0(F_2)$.
Since $\wt{\cT}$ has enough Galois coverings, it follows
from \cite[Cor.\ 4.3.2]{Grids} that there exists
an object $G_j$ of $\wt{\cC}$ and a morphism
$g_i \colon  G_j \to F_{1,i}$ in $\wt{\cC}$ which belongs to
$\wt{\cT}$ for each $i \in \pi_0(f)^{-1}(j)$
such that $h_j = f_i \circ g_i$ does not depend on $i$ and
$g_i$ and $h_j$ are Galois coverings in $\wt{\cC}$.
Let $H_j =\Aut_{F_{2,j}}(G_j)$ and fix a finite group
$I_j$ whose order is equal to the cardinality of 
the set $\pi_0(f)^{-1}(j)$. Fix a bijection
$\psi_j \colon  I_j \xto{\cong} \pi_0(f)^{-1}(j)$.
Set $F_{1,(j)} = \coprod_{i \in \pi_0(f)^{-1}(j)} F_{1,i}$
and $G_{(j)} = \coprod_{s \in H_j} G_j$.
Let $f_{(j)} \colon  F_{1,(j)} \to F_{2,j}$ denote the
morphism induced by $f$.
Let $q_j \colon  G_{(j)} \to F_{1,(j)}$ denote the morphism
in $\wt{\cFC}$ such that $\pi_0(q_j) = \psi_j$
and that for any $s \in H_j$, the component at $s$ of
$q_j$ is equal to $g_{\psi_j(s)}$.
Let $H^j = \prod_{k \in \pi_0(F_2) \setminus \{j\}} 
(H_k \times I_k)$ and set 
$G^{(j)} = \coprod_{t \in H^j} G_{(j)}$.
Set $q^j \colon  G^{(j)} \to F_{1,(j)}$ to be
$\prod_{t \in H^j} q_j$.
We let the group $H = \prod_{j \in \pi_0(F_2)}
(H_j \times I_j)$ act on $G^{(j)}$
over $F_{2,j}$ from the left via the bijections $\psi_j$.
Let $\rho_j \colon  H \to \Aut_{F_{2,j}}(G^{(j)})$
denote the homomorphism given by this action.
Then it is easy to see that $f_{(j)} \circ q^j$ 
is a Galois covering 
in $\wt{\cFC}$ with respect to $\rho_j$.
Let $G = \coprod_{j \in \pi_0(F_2)}$ and
let $q = \coprod_{j \in \pi_0(F_2)} q_j \colon  G \to F_1$.
By construction $q$ belongs to $\wt{\cFT}^*$.
Then $\rho_j$ for each $j \in \pi_0(F_2)$ gives a natural
action of the group $H$ on $G$ over $F_2$ and
$f \circ q$ is a Galois covering with respect to this action.
This shows that $\wt{\cFT}^*$ has enough Galois coverings.
\end{proof}

\begin{lem}\label{lem:cFT semi-localizing}
The set $\wt{\cFT}^*$ is a semi-localizing collection of morphisms in $\wt{\cFC}$
in the sense of \cite[Def.\ 2.3.1]{Grids}.
Moreover, we have $\wh{\wt{\cFT}^*} = \wt{\cFT}^*$.
\end{lem}

\begin{proof}
Let $Y_1 \xto{f_1} X \xleftarrow{f_2} Y_2$ be a diagram in $\wt{\cFC}$.
Suppose that $f_1$ belongs to $\wt{\cFT}^*$.
For each $i \in \pi_0(Y_2)$, set $k_i = \pi_0(f_2)(i)$.
Let us choose an element $j_i \in \pi_0(Y_1)$
satisfying $\pi_0(f_1)(j_i) = k_i$. 
Let $Y_{1,j_i}$ (\resp $Y_{2,i}$) denote the component at $j_i$
(\resp $i$) of $Y_1$ (\resp $Y_2$).
It follows from Lemma \ref{lem:tilde semi-localizing} that
$\wt{\cT}$ is semi-localizing. Hence for each $i \in \pi_0(Y_2)$,
there exists a diagram $Y_{1,j_i} 
\xleftarrow{g_{1,i}} Z_i \xto{g_{2,i}} Y_{2,i}$
in $\wt{\cC}$ with $g_{2,i} \in \wt{\cT}$ satisfying 
$f_{1,j_i} \circ g_{1,i} = f_{2,i} \circ g_{2,i}$.
We set $Z = \coprod_{i \in \pi_0(Y_2)} Z_i$. 
For $j=1,2$ let $g_j \colon  Z \to Y_j$ denote the morphism in $\wt{\cFC}$
whose component at $i \in \pi_0(Y_2)$ is given by $g_{j,i}$.
Then $g_2$ belongs to $\wt{\cFT}^*$ 
and we have $f_1 \circ g_1 = f_2 \circ g_2$.
This proves that $\wt{\cFT}^*$ is semi-localizing.
One can check easily that $\wh{\wt{\cFT}^*} = \wt{\cFT}^*$.
\end{proof}

By Lemma \ref{lem:cFT semi-localizing}, one can consider the $A$-topology
$J_{\wt{\cFT}^*}$ on $\wt{\cFC}$ given by $\wt{\cFT}^*$. 
This Grothendieck topology on $\wt{\cFC}$ is coarser 
than $\iota'_* J$, and does not play an important role in the sequel. 
Proposition \ref{prop:cFT_main} implies that $\iota'_* J$ is the Grothendieck
topology generated by the covering families which are either singleton sets in
$\wt{\cFC}$ or of the form
$\{ X_i \to X \}_{i \in \pi_0(X)}$ for some object $X$ of $\wt{\cFC}$.
In particular a presheaf $F$ on $(\wt{\cFC},\iota'_* J)$ is a sheaf if and only if it is a sheaf on $(\wt{\cFC},J_{\wt{\cFT}^*})$ and Conditions (1) and (2)
in Section \ref{sec:cond12} are satisfied.

\section{Degree} \label{sec:degree}
In this section, we assume certain finiteness condition on the 
hom sets of the category $\cC$.   For example, all the hom sets
in $\cC$ are finite.    Then we have the notion of the degree 
of a morphism.   The aim of this section
is to study this notion.
The degree of a morphism
is used in the definition (Definition~\ref{def:transfer}) 
of presheaves with transfers.

\subsection{}
Let $(\cC,J)$ be a $Y$-site. We assume that
for any diagram $X \xto{f} Z \xleftarrow{g} Y$ in $\cC$
with $f,g$ in $\cT = \cT(J)$, there exist only finitely many
morphisms from $X$ to $Y$ over $Z$.

Let $f\colon X \to Y$ be a morphism in $\cT$.
Since $\cT$ has enough Galois coverings, there exists
a morphism $g\colon  Z \to X$ in $\cT$ such that the composite
$f \circ g$ is a Galois covering.
We choose such a morphism $g$ and consider the
diagram $Z \xto{f\circ g} Y \xleftarrow{f} X$.
We define the degree $\deg f$ of $f$ to be the number of 
the morphisms from $Z$ to $X$ over $Y$.
It follows from Lemma 4.1.3 of \cite{Grids} that
$g$ is also a Galois covering.
It follows from Lemma 4.2.3 (2) of \cite{Grids} that
the group $\Aut_Y(Z)$ acts transitively on $\Hom_Y(Z,X)$
and the stabilizer of $g$ is equal to $\Aut_X(Z)$.
Hence by definition, we have
\begin{equation} \label{eq:deg_defn}
\deg f = \sharp \Aut_Y(Z)/ \sharp \Aut_X(Z).
\end{equation}
We claim that the degree of $f$ does not depend on the
choice of $g$.
Let $g'\colon Z' \to X$ be another morphism in $\cT$ such
that $f \circ g'$ is a Galois covering.
To prove the claim, we may assume that there exists
a morphism $h\colon Z' \to Z$ over $Y$.
Since it follows from Lemma 4.2.3 (2) of \cite{Grids}
that $\Aut_Y(Z')$ acts transitively on $\Hom_Y(Z',X)$, 
we may assume, by replacing $h$ with
its composite with a suitable element in $\Aut_Y(Z')$,
that the morphism $h$ is over $X$.
By Lemma 4.2.6 (4) of \cite{Grids} it follows that, we have
short exact sequences:
$$
1 \to \Aut_Z(Z') \to \Aut_X(Z') \to \Aut_X(Z) \to 1
$$
and
$$
1 \to \Aut_Z(Z') \to \Aut_Y(Z') \to \Aut_Y(Z) \to 1.
$$
From this we have
$$
\sharp \Aut_X(Z') / \sharp  \Aut_X(Z)
= \sharp \Aut_Z(Z') =
\sharp \Aut_Y(Z') / \sharp \Aut_Y(Z).
$$
Hence $\sharp \Aut_Y(Z') / \sharp  \Aut_X(Z')
= \sharp \Aut_Y(Z) / \sharp \Aut_X(Z)$.
This shows that the degree of $f$ does not depend on the
choice of $g$.

\begin{lem}\label{lem:degree is multiplicative}
\begin{enumerate}
\item If $f\colon X \to Y$ is an isomorphism in $\cC$, then
we have $\deg f =1$.
\item Let $f\colon X \to Y$ and $g\colon Y \to Z$ be morphisms in $\cC$
which belongs to $\cT$. Then we have 
$\deg g \circ f = \deg f \deg g$.
\end{enumerate}
\end{lem}

\begin{proof}
The claim (1) is obvious from the definition of $\deg f$.
We prove the claim (2). Let us choose a morphism
$h\colon  W \to X$ in $\cC$ which belongs to $\cT$ such that
the composite $g \circ f \circ h$ is a Galois covering in $\cC$.
Then by \eqref{eq:deg_defn} we have
$$
\deg g \circ f = \sharp \Aut_Z(W)/ \sharp \Aut_X(W)
= (\sharp \Aut_Z(W)/ \sharp \Aut_Y(W))
\cdot (\sharp \Aut_Y(W)/ \sharp \Aut_X(W))
= \deg g \deg f,
$$
which proves the claim (2).
\end{proof}

\subsubsection{}
Let us consider the site $(\wt{\cC},\iota_*J)$
introduced in Section 6.
By Corollary~\ref{cor:Ctil_Y_site} this is a $Y$-site.
One then has the notions of degree 
both for morphisms in $\cT$ and for those in $\wt{\cT}$.
Lemma \ref{lem:i_Galois} implies that these two notions
of degree are compatible with respect to the functor 
$\iota\colon  \cC \to \wt{\cC}$.
Let us consider the category $\wt{\cFC}$ and the class
$\wt{\cFT}$ of morphisms in $\wt{\cFC}$ introduced 
in Section \ref{sec:wtcFC}.
We can naturally extend the notion of
degree to the morphisms in $\wt{\cFT}$ as follows.
Let $f \colon  X \to Y$ be a morphism in $\wt{\cFT}$.
For $i \in \pi_0(X)$, let $f_i$ denote the component
at $i$ of $f$.
Then the degree of $f$ is the $\Z$-valued function on
$\pi_0(Y)$ whose value at $j \in \pi_0(Y)$
is equal to the sum 
$$
\sum_{\genfrac{}{}{0pt}{}{i \in \pi_0(X) }{\pi_0(f)(i)=j}}
\deg\, f_i.
$$

Let us fix a grid $(\cC_0,\imath_0)$ of $(\cC,J)$ and set
$M=M_\Cip$. 
%
Let $\omega \colon  \Shv(\cC,J) \to (M\text{-set})_\mathrm{sm}$ 
denote the fiber functor, introduced in Section 5.7.3 of \cite{Grids}, 
associated with the grid $(\cC_0,\imath_0)$.

\begin{lem} \label{lem:epi_surj}
Let $f\colon S \to T$ be a morphism of smooth left $M$-sets.
Then $f$ is an epimorphism in $(M\text{-set})_\mathrm{sm}$
if and only if $f$ is surjective as a map of sets.
\end{lem}

\begin{proof}
If $f$ is surjective, then $f$ is an epimorphism
since the forgetful functor from $\cD$
to the category of sets is faithful.
Suppose that $f$ is not surjective.
We will prove that $f$ is not an epimorphism.
The claim is clear when $S$ is an empty set.
Let us assume that $S$ is non-empty and choose
an element $s \in S$.
Let $R \subset T \times T$ denote the union
of the diagonal subset $T \subset T \times T$ and
$f(S) \times f(S)$.
Then $R$ is an equivalence relation on the set $T$.
Let $U$ denote the quotient of the set $T$ under
the equivalence relation $R$.
Since $R$ is stable under the diagonal action of
$M_\Cip$ on $T \times T$, the action of $M_\Cip$ on $T$
induces a structure of a smooth left $M_\Cip$-set on $U$. 
Let $g\colon T \to U$ denote the quotient map and let
$g'\colon T \to U$ denote the constant map whose value
is the class of $f(s)$ in $U$.
Then both $g$ and $g'$ are morphisms in $\cD$
and we have $g \circ f = g' \circ f$.
However $g$ is not equal to $g'$ since $f$ is not surjective.
This shows that the map $f$ is not an epimorphism in $\cD$.
\end{proof}

\begin{lem} \label{lem:TJ_surj}
Let $f$ be a morphism in $\wt{\cC}$
which belongs to $\wt{\cT}$.
Then $\omega(f)$ is surjective as a map of sets.
\end{lem}

\begin{proof}
By the definition of $\wt{\cT}$, we may and will assume that
$f$ is either of the form $\iota(f')$ for some morphism
$f'\colon  X \to Y$ in $\cC$ which belongs to $\cT(J)$, 
where $\iota \colon  \cC \to \wt{\cC}$
is as in Section \ref{sec:Ctil}, or of the form
$f\colon  \iota(X) \to \quot{X}{H}$ for some object $X$ in $\cC$
and a subgroup $H \subset \Aut_{\cC}(X)$.
For the latter case the claim follows immediately from
the definition of $\quot{X}{H}$ and Lemma 10.1.2 of \cite{Grids}.
From now on we assume that we are in the former case.
We set $f''=\frh_\cC(f')\colon \frh_\cC(X) \to \frh_\cC(Y)$.
Then $F(f')\colon  F(Y) \to F(X)$ is injective for any sheaf 
$F$ on $(\cC,J)$.
Hence $a_J(f'') \colon  a_J(\frh_\cC(X)) \to a_J(\frh_\cC(Y))$ is an 
epimorphism in the category $\Shv(\cC,J)$.
It follows from Theorem 5.8.1 of \cite{Grids} and
Lemma 10.1.2 of \cite{Grids} that 
$\omega(f'')$ is an epimorphism in the category 
of smooth $M$-sets.
By Lemma 10.1.2 of \cite{Grids}, we can identify $\omega(f)$
with $\omega_\Cip(f')$.
Hence it follows from Lemma \ref{lem:epi_surj} that
$\omega(f)$ is surjective as a map of sets.
\end{proof}

\subsubsection{}

\begin{lem} \label{lem:deg_omega}
Let $f\colon X \to Y$ be a morphism in $\wt{\cT}$. Then for any element 
$y$ of the set $\omega(Y)$, the cardinality of its fiber
$\omega(f)^{-1}(y)$ is equal to $\deg f$.
\end{lem}

\begin{proof}
First we prove the case where $f$ is a Galois covering.
It follows from Lemma \ref{lem:TJ_surj} that the fiber
$\pi_0(f)^{-1}(y)$ is non-empty.
Let $G=\Aut_Y(X)$ and let $x,x' \in \pi_0(f)^{-1}(y)$.
It suffices to show that there exists a unique element
$g \in G$ satisfying $x' = gx$.
Let us choose an object $Z$ of $\cC_0$ such that
both $x$ and $x'$ belong to the image of
the map $X(\imath_0(Z)) \to \omega(X)$.
Note that $X(\imath_0(Z)) \cong \Hom_{\wt{\cC}}(\iota(Z),X)$.
Since $f$ is a Galois covering, the map 
$X(\imath_0(Z)) \to Y(\imath_0(Z))$ between the sections on 
$\imath_0(Z)$ is a pseudo $G$-torsor.
Hence there exists a unique element
$g \in G$ satisfying $x' = gx$.
It follows from Lemma \ref{lem:tilde E-category} that
the map $X(\imath_0(Z)) \to \omega(X)$ is injective.
This proves the claim when $f$ is a Galois covering.

For general $f$, let us take a morphism $g\colon Z \to X$
in $\wt{\cT}$ such that the composite $h = f \circ g$ is
a Galois covering. Then by applying the claim
for the Galois coverings $g$ and $h$, we obtain the
claim for $f$. This completes the proof.
\end{proof}

\begin{cor} \label{cor:deg_omega_F}
Let $f\colon  F_1 \to F_2$ be a morphism in $\wt{\cFC}$.
For $i \in \pi_0(F_2)$, let $F_{2,i}$ denote the component
of $F_2$ at $i$. Let us identify $\omega(F_2)$ with 
$\coprod_{i \in \pi_0(F_2)} \omega(F_{2,i})$.
Then for any $i \in \pi_0(F_{2,i})$ and for any
$y \in \omega(F_{2,i})$ the degree $\deg(f)(i)$ of $f$ at $i$ is
equal to the cardinality of the fiber 
$\omega(f)^{-1}(y)$.
\end{cor}

\begin{proof}
This follows immediately from Lemma \ref{lem:deg_omega}.
\end{proof}

\begin{cor} \label{cor:deg_limit}
Let
$$
\begin{CD}
F'_1 @>>> F_1 \\
@V{f'}VV @VV{f}V \\
F'_2 @>{g}>> F_2
\end{CD}
$$
be a cartesian diagram in $\wt{\cFC}$.

Then for any $i \in \pi_0(F'_2)$, we have
$$
\deg(f')(i) = (\deg(f))(\pi_0(g)(i)).
$$
\end{cor}

\begin{proof}
This follows from Corollary \ref{cor:deg_omega_F} and Lemma \ref{lem:FC_product_isom}.
\end{proof}

\section{Presheaves with transfers}
We introduce the notion of presheaves with transfers. 
We have in mind the presheaf of (motivic) cohomology of moduli spaces.
The change-of-level morphisms 
(usually) induce pullback and pushforward 
morphisms between the 
cohomology groups of the moduli spaces.
The axioms we state here 
are the expected properties of such pushforwards.
We refer to Section~\ref{sec:Drinfeld Euler} 
for the example of
motivic cohomology groups of Drinfeld modular schemes.
The easier example of the multiplicative groups of 
cyclotomic fields (i.e., the $K_1$ of fields) is 
treated in Section~\ref{sec:cyclotomic Euler}.

The transfer structure is not unrelated to the topology.
We show that a sheaf of abelian groups is equipped with a canonical structure 
of presheaf with transfers (Proposition~\ref{prop:transfer}).    In the 
other direction, we see that a presheaf 
with transfer with values in 
$\Q$-vector spaces is a sheaf 
(Corollary~\ref{cor:Q-coefficient is sheaf}).

With the transfer structure, we are able to define Hecke operators
as a pullback followed by a pushforward.
We also introduce a presheaf of rings with transfers to
be a presheaf of rings that satisfies the projection formula.

\subsection{Definition and basic properties}
Let $(\cC,J)$ be a $Y$-site. Let $\wt{\cC}$ be as in Section \ref{sec:Ctil}.
Let $\wt{\cFC}$ be as in Section \ref{sec:wtcFC}.
We assume that for any object $X$ of $\cC$, the overcategory $\cC(\cT(J))_{/X}$
satisfies Condition (1) of \cite[5.8.1]{Grids}, that is, 
for any objects $F_1, F_2$, the hom set 
$\Hom_{\cC(\cT(J))_{/X}}(F_1, F_2)$ is a finite set.

\begin{defn}\label{def:transfer}
An {\em abelian presheaf with transfers} 
on $\wt{\cFC}$
is a presheaf $G$ of abelian groups on $\wt{\cFC}$
satisfying  Conditions (1) and (2) in Section \ref{sec:cond12},
equipped with, for each morphism
$f\colon F\to F'$ in $\wt{\cFC}$ which belongs to $\wt{\cFT}$,
a homomorphism $f_{*}\colon G(F)\to G(F')$ satisfying the
following properties:
\begin{enumerate}
\item For any two composable morphisms
$f$ and $f'$ which belong to $\wt{\cFT}$, we have
$(f \circ f')_{*} = f_{*}\circ f'_{*}$.
\item For any cartesian diagram
$$
\xymatrix{
F_1 \ar[r]^{g_1} \ar@{}[dr]|\square \ar[d]_{f}
& F'_1 \ar[d]^{f'} \\
F_2 \ar[r]^{g_2} 
& F'_2
}
$$
in $\wt{\cFC}$
such that $f'$ belongs to $\wt{\cFT}$,
we have $g_2^{*} \circ f'_* = f_{*} \circ g_1^{*}$.
\item The composite $f_*\circ f^*$ is 
the multiplication by $\deg f$, which means that,
if we write $F' = \coprod_{i \in \pi_0(F')} F'_i$, then
the endomorphism of $\prod_{i \in \pi_0(F')}
G(F'_i)$ induced by the endomorphism 
$f_* \circ f^*$ of $G(F')$ is equal to the
homomorphism that sends $(x_i)_{i \in \pi_0(F')} \in 
\prod_{i \in \pi_0(F')} G(F'_i)$ to
$(\deg(f)(i) x_i)_{i \in \pi_0(F')}$.
\end{enumerate}
\end{defn}


The following lemma assures us that Condition (2) in 
Definition \ref{def:transfer} for $\wt{\cFC}$ makes sense.

We note that 
by our assumption $(\cC,J)$ satisfies the condition in 
Proposition \ref{prop:FC_product}.
Hence for any diagram 
$F_1 \xto{f_1} F' \xleftarrow{f_2} F_2$ diagram in $\wt{\cFC}$
with $f_2$ in $\wt{\cFT}$, the fiber product
$F_1 \times_{F'} F_2$ exists and
the first projection 
$\pr_1 \colon  F_1 \times_{F'} F_2 \to F_1$ belongs to $\wt{\cFT}$.

\begin{lem}\label{lem:transfer_isom}
Let $f\colon F \to F'$ be an isomorphism in the category
$\wt{\cFC}$.
Let $G$ be an abelian presheaf with transfers on 
$\wt{\cFC}$.
Then the composite $f_* f^*$ is the identity on $G(F')$ and 
the composite $f^* f_*$ is the identity on $G(F)$.
\end{lem}

\begin{proof}
Since $f$ is an isomorphism, the map $f^*$ is an isomorphism.
%
Note that any isomorphism in $\wt{\cFC}$ belongs to $\wt{\cFT}$.
It follows from Lemma \ref{lem:degree is multiplicative} (1) 
that we have $\deg f=1$. 
Hence the composite $f_* f^*$ is equal to the identity on $G(F')$.
Since $f^* f_* f^* = f^*$ and $f^*$ is an isomorphism, the
composite $f^* f_*$ is equal to the identity on $G(F)$.
This proves the claim.
\end{proof}

\begin{lem}\label{lem:sum_formula}
Let $f\colon F \to F'$ be a morphism
in $\wt{\cFC}$ which belongs to $\wt{\cFT}$.
Suppose that $F$ is isomorphic to 
the coproduct $F \cong F_1 \amalg \cdots \amalg F_n$
for some objects $F_1, \ldots, F_n$ in $\wt{\cFC}$.
For $i=1,\cdots,n$ let 
$\iota_i\colon F_i \to F$ denote the $i$-th inclusion morphism
and set $f_i = f \circ \iota_i$.
Let $G$ be an abelian presheaf with transfers on $\wt{\cFC}$.
Then we have $f_*=\sum_{i=1}^n (f_i)_* \iota_i^* \colon 
G(F) \to G(F')$.
\end{lem}

\begin{proof}
It suffices to prove the claim when $F=F'$ and $f\colon F\to F$
is the identity morphism.
Suppose that $f=\id_F$. By Lemma \ref{lem:transfer_isom},
$f_*$ is the identity map on $G(F)$.
Since $(f_i)_* \iota_i^* = (\iota_i)_* \iota_i^*$
is equal to the multiplication by $\deg \iota_i$ for each $i$,
it suffices to prove that $\sum_{i=1}^n \deg \iota_i = 1$.
For $i=1,\ldots,n$, let $\pi_0(F)_i$ denote the image
of the map $\pi_0(\iota_i)\colon \pi_0(F_i) \to \pi_0(F)$.
Then the set $\pi_0(F)$ is equal to the disjoint union
$\pi_0(F) = \coprod_{i=1}^n \pi_0(F)_i$.
It can be checked that the degree of the morphism
$\iota_i$ is equal to $1$ on $\pi_0(F)_i$, and is
zero on the complement $\pi_0(F) \setminus \pi_0(F)_i$.
Hence we have $\sum_{i=1}^n \deg \iota_i = 1$.
This proves the claim.
\end{proof}

\begin{lem}\label{lem:transfer_nonGal}
Let $f\colon F \to F'$ be a morphism in $\wt{\cC}$
and let $g\colon F_1\to F'$ be a Galois covering in $\wt{\cC}$
which belongs to $\wt{\cT}$.
Suppose that there exists a morphism $h\colon F_1\to F$ in $\wt{\cC}$
satisfying $g=f\circ h$.
Let $\Hom_{F'}(F_1,F) \subset \Hom_{\wt{\cFC}}(F_1,F)$ 
denote the subset of such morphisms $h$.
Then for any abelian presheaf $G$ with transfers on $\wt{\cFC}$ 
we have $g^* f_* = \sum_{h \in \Hom_{F'}(F_1,F)} h^*\colon G(F) \to G(F_1)$.
\end{lem}
\begin{proof}
It follows from Proposition \ref{prop:Ctil_Bsite}, 
Lemma \ref{lem:aJ_properties},
and Lemma 4.2.3 (2) of \cite{Grids} that the diagram
$$
\begin{CD}
\coprod_{h \in \Hom_{F'}(F_1,F)} F_1 @>{f'_1}>> F_1 \\
@V{f'_2}VV @V{g}VV \\
F @>{f}>> F'
\end{CD}
$$
is cartesian. Here $f'_1$ is the morphism whose component at
$h$ is the identity morphism of $F_1$ for each $h \in \Hom_{F'}(F_1,F)$, 
and $f'_2$ is the morphism whose component at $h \in \Hom_{F'}(F_1,F)$ is
the morphism $h\colon F_1 \to F$.
For $h \in \Hom_{F'}(F_1,F)$, let $i_h\colon F_1 \inj 
\coprod_{h \in \Hom_{F'}(F_1,F)} F_1$ denote
the inclusion to the component at $h$.
We have $f'_1 \circ i_h = \id_{F_1}$ 
and $f'_2 \circ i_h = h$ for each $h \in \Hom_{F'}(F_1,F)$.
Hence, by Lemma \ref{lem:sum_formula}, we have
$$
g^* f_* = (f'_1)_* {f'_2}^*
= \sum_{h \in \Hom_{F'}(F_1,F)} (f'_1 \circ i_h)_*
i_h^* {f'_2}^* = \sum_{h \in \Hom_{F'}(F_1,F)} {\id_{F_1}}_* h^*.
$$
Therefore
$$
g^* f_* = \sum_{h \in \Hom_{F'}(F_1,F)} h^*.
$$
This proves the claim.
\end{proof}

\begin{cor}
\label{cor:Q-coefficient is sheaf}
Let $F$ be a presheaf with transfers with values in $\Q$-vector spaces.   Then $F$ is a sheaf.
\end{cor}
\begin{proof}
Let $m\colon  M \to N$ be a Galois covering of group $G$.
It suffices to prove that the map 
$F(N) \to F(M)^G$ induced by $m^*$ is an isomorphism.
By the definition of transfers, we have 
$\deg m=m_*m^*$.   
By Lemma~\ref{lem:transfer_nonGal},
we have 
$m^*m_*=\sum_{\sigma \in G} \sigma$.
Using these, one can check that 
$\frac{1}{|G|}m_*$ is the inverse of the map above
induced by $m^*$.
\end{proof}

\subsection{The transfer structure on sheaves}
\label{sec:with transfers}

\begin{prop}\label{prop:transfer}
Let $G$ be an abelian sheaf on $\wt{\cFC}$.
Then $G$ has the structure of presheaf with transfers.
The structure is uniquely determined.
\end{prop}

Before proving Proposition \ref{prop:transfer},
we give a candidate for the transfer
homomorphism for each morphism in $\wt{\cFC}$
which belongs to $\wt{\cFT}$.

Let $f\colon  F \to F'$ be a morphism in $\wt{\cC}$
which belongs to $\wt{\cT}$.
It follows from Lemma \ref{lem:imath_enough_Galois}
that there exists an object $F_1$ in $\wt{\cC}$ 
and a morphism $f_1 \colon  F_1 \to F$ which belongs to
$\wt{\cT}$ such that the composite $f \circ f_1$
is a Galois covering in $\wt{\cC}$.
We then define the transfer map $f_*\colon  G(F) \to G(F')$ to be
the composite of the map $G(F) \to G(F_1)^{\Aut_{F'}(F_1)}$
that sends $x \in G(F)$ to 
$\sum_{f'_1 \in \Hom_{F'}(F_1,F)} f'^*_1(x)$
and the inverse of the isomorphism $G(F') \xto{\cong} 
G(F_1)^{\Aut_{F'}(F_1)}$.
The following lemma shows that the map
$f_*$ does not depend on the choice of $F_1$.

\begin{lem}
Let $\cF$ be an $F$-category and let $\cF^{(0)} \subset \cF$ 
denote the full subcategory of connected objects.
Suppose that the category $\cF^{(0)}$ is semi-cofiltered
in the sense of \cite[Def.~2.4.4]{Grids}
and that any morphism in $\cF^{(0)}$ is an epimorphism.
Let $N \to N' \xleftarrow{f_1} N_1
\xleftarrow{f_2} N_2$ be a diagram in $\cF^{(0)}$
such that $f_1$ and $f_1 \circ f_2$ are Galois coverings
in $\cF^{(0)}$.
Suppose that the set $\Hom_{N'}(N_1,N)$ is non-empty.
Then the map $\Hom_{N'}(N_1,N) \to \Hom_{N'}(N_2,N)$
given by the composition with $f_2$ is bijective.
\end{lem}

\begin{proof}
The injectivity follows since $f_2$ is an epimorphism 
in $\cF^{(0)}$.
We prove the surjectivity.
Let $h\colon N_2 \to N$ be a morphism over $N'$.
We apply Lemma 4.2.3 (1) of \cite{Grids} to the diagram
$$
\begin{CD}
N_2 @>{f_2}>> N_1 \\
@V{h}VV @V{f_1}VV \\
N @>>> N'.
\end{CD}
$$
There exists an automorphism $g\in \Aut_{N'}(N_1)$
such that $h = f_1 \circ g \circ f_2$.
Hence the claim follows.
\end{proof}

Let $f\colon  F \to F'$ be a morphism in $\wt{\cFC}$
which belongs to $\wt{\cFT}$.
For $i \in \pi_0(F)$, let $F_i$ denote the
component at $i$ of $F$ and let $f_i$ denote the
component at $i$ of $f$.
For $j \in \pi_0(F')$ we let $F'_j$ denote the
component at $j$ of $F'$.
We let $f_* \colon  G(F) \to G(F')$ denote the
unique homomorphism such that the diagram
$$
\begin{CD}
G(F) @>{f_*}>> G(F') \\
@V{\cong}VV @V{\cong}VV \\
\prod_{i \in \pi_0(F)}G(F_i)
@>>> \prod_{j \in \pi_0(F')}G(F'_j)
\end{CD}
$$
is commutative.
Here the bottom horizontal map is the homomorphism
that sends $(x_i)_{i \in \pi_0(F)} \in
\prod_{i \in \pi_0(F)}G(F_i)$ to
$(y_j)_{j \in \pi_0(F')} \in \prod_{j \in \pi_0(F')}G(F'_j)$,
where $y_j = \sum_{\pi_0(f)(i) = j} (f_i)_* x_i$ for each
$j \in \pi_0(F')$.

\begin{proof}[Proof of Proposition \ref{prop:transfer}]
It follows from Lemma \ref{lem:sum_formula} and
Lemma \ref{lem:transfer_nonGal} that
for any structure of presheaf with transfers 
on $G$, the transfer homomorphism for a morphism $f$ is 
equal to the homomorphism $f_*$ introduced above.
This proves the uniqueness of the structure of
presheaf with transfers on $G$.

It remains to prove that the collection $(f_*)$ 
of maps $f_*$ introduced above satisfies the 
three properties in Definition \ref{def:transfer}.

We prove that the collection $(f_*)$ satisfies
the property (1) in Definition \ref{def:transfer}.
Let $F \xto{f} F' \xto{f'} F''$ be a diagram
in $\wt{\cFC}$ such that $f$ and $f'$ belong to $\wt{\cFT}$.
We prove that $(f' \circ f)_*$ is equal to 
$f'_* \circ f_*$.
We are easily reduced to the case when $f$ and $f'$
are morphisms in $\wt{\cC}$.
Let us take a morphism 
$h \colon  F_1 \to F$ in $\wt{\cC}$ which belongs to $\wt{\cT}$
such that the composite $f' \circ f \circ h$ is 
Galois covering in $\wt{\cC}$.
Let us consider the three sets
$S=\Hom_{F''}(F_1,F)$, $S'=\Hom_{F''}(F_1,F')$,
and $S_1=\Hom_{F'}(F_1, F)$.
By definition $S_1$ is a subset of $S$ and
the composition with $f$ gives a map $S \to S'$.
It follows from \cite[Lemma 4.2.3]{Grids}
that the group $\Aut_{F''}(F_1)$ acts transitively
on the sets $S$ and $S'$ and that the group
$\Aut_{F'}(F_1)$ acts transitively
on the set $S_1$.
Since the stabilizer of $h \in S$ and $f \circ h \in S'$ 
in $\Aut_{F''}(F_1)$ are equal to $\Aut_{F'}(F_1)$ and
$\Aut_{F}(F_1)$ respectively, we have
$(f \circ h)^* \circ f_* (x) = 
\sum_{g \in \quot{\Aut_{F'}(F_1)}{\Aut_{F}(F_1)}}
g^* \circ h^* (x)$
and
$(f' \circ f \circ h)^* \circ (f'\circ f)_* (x) = 
\sum_{g \in \quot{\Aut_{F''}(F_1)}{\Aut_{F}(F_1)}}
g^* \circ h^* (x)$
for any $x \in G(F)$ and
$(f' \circ f \circ h)^* \circ f'_* (y) = 
\sum_{g \in \quot{\Aut_{F''}(F_1)}{\Aut_{F'}(F_1)}}
g^* \circ (f\circ h)^* (y)$
for any $y \in G(F')$.
Hence the equality $(f' \circ f)_* =f'_* \circ f_*$ holds.
This proves that the collection $(f_*)$ satisfies
the property (1) in Definition \ref{def:transfer}.

Let the notation be as in the 
property (2) in Definition \ref{def:transfer}.
We prove the equality $g_2^* \circ f'_* = f_* \circ g_1^*$.
By using Corollary \ref{cor:fiber_product} we are easily
reduced to the case when $F_2$, $F'_1$ and $F'_2$
are objects in $\wt{\cC}$.
Let us take a morphism
$h' \colon  F' \to F'_1$ in $\wt{\cC}$ which belong to $\wt{\cT}$
such that $f' \circ h'$ is a Galois covering in 
$\wt{\cC}$.
It follows from Proposition \ref{prop:Ctil_Bsite} 
and \cite[Lemma 4.2.3]{Grids} that
there exist an object $F$ in $\wt{\cC}$
a morphism $g\colon F \to F'$ in $\wt{\cC}$
and a morphism $h\colon F \to F_2$ in
$\wt{\cC}$ which belongs to $\wt{\cT}$ satisfying
$g_2 \circ h = f' \circ h' \circ g$.
For each $i \in \pi_0(F_1)$, let 
$F_{1,i}$ denote the component at $i$ of $F_1$
and let $f_i\colon F_{1,i} \to F_2$ and $g_{1,i}\colon F_{1,i} \to F'_2$
denote the component at $i$ of $f$ and $g_1$,
respectively.
It follows from Corollary 4.3.2 of \cite{Grids} that
there exist an object $F_3$ in $\wt{\cC}$ 
a morphism $f_3 \colon  F_3 \to F$ in
$\wt{\cC}$ which belongs to $\wt{\cT}$, and a morphism
$f_{3,i}\colon F_3 \to F_{1,i}$ which belongs to $\wt{\cT}$ for
each $i \in \pi_0(F_1)$ such that
$h \circ f_3 = f_i \circ f_{3,i}$ holds for each 
$i \in \pi_0(F)$ and that $h \circ f_3$ is a Galois
covering in $\wt{\cC}$.
It follows from the construction of the transfer 
homomorphisms that we have
$(h \circ f_3)^* \circ f_* \circ g_1^*
= \sum_{i \in \pi_0(F)} \sum_{h_i \in \Hom_{F_2}(F_3,F_{1,i})} 
h_i^* \circ g_{1,i}^*$.
Since the diagram in the 
property (2) in Definition \ref{def:transfer}
is cartesian, the map
$\coprod_{i \in \pi_0(F)} \Hom_{F_2}(F_3, F_{1,i})
\to \Hom_{F'_2}(F_3, F'_1)$
given by the composition with $g_{1,i}$ for each $i \in \pi_0(F)$
is bijective.
Hence we have
$(h \circ f_3)^* \circ f_* \circ g_1^*
= \sum_{h'' \in \Hom_{F'_2}(F_3, F'_1)} h''^* $.
It follows from Lemma 4.2.3 (1) of \cite{Grids} that
for any $h'' \in \Hom_{F'_2}(F_3, F'_1)$, there exists
an automorphism $\alpha \in \Aut_{F'_2}(F')$
satisfying $h'' = h' \circ \alpha \circ g \circ f_3$.
Since it follows from Proposition \ref{prop:Ctil_Bsite} that 
$g \circ f_3$ is an epimorphism in $\wt{\cC}$ 
the map $\Hom_{F'_2}(F',F'_1) \to \Hom_{F'_2}(F_3,F'_1)$
given by the composition with $g \circ f_3$ is bijective.
Therefore we have
$(h \circ f_3)^* \circ f_* \circ g_1^*
= (g \circ f_3)^* \circ (\sum_{h''' \in \Hom_{F'_2}(F', F'_1)} h'''^*)
= (g \circ f_3)^* \circ h'^* \circ f'_*
= (h \circ f_3)^* \circ g_2^* \circ f'_*$.
This proves that the collection $(f_*)$ satisfies
the property (2) in Definition \ref{def:transfer}.

We prove that the collection $(f_*)$ satisfies
the property (3) in Definition \ref{def:transfer}.
Let $f\colon F \to F'$ be a morphism 
in $\wt{\cFC}$ which belongs to $\wt{\cFT}$.
We prove that the composite $f_* \circ f^*$
is equal to the multiplication by $\deg f$.
We are easily reduced to the case when $f$ is
a morphism in $\wt{\cC}$.
Let us take a morphism
$h \colon  F_1 \to F$ in $\wt{\cC}$ which belongs to $\wt{\cT}$
such that the composite $f \circ h$ is 
Galois covering in $\wt{\cC}$.
Let $s$ denote the cardinality of the set
$\Hom_{F'}(F_1,F)$.
It follows from the construction of the transfer
homomorphism $f_*$ that the composite
$(f \circ h)^* \circ f_*$ is equal to
the homomorphism $(f \circ h)^*$ multiplied by $s$.
It follows from Corollary \ref{cor:deg_limit} 
and Lemma 4.2.3 of \cite{Grids} that we have $s=\deg f$.
This proves that the composite 
is equal to the multiplication by $\deg f$.
This completes the proof.
\end{proof}

In a similar manner, one can show that
any abelian sheaf 
on $\wt{\cFC}$ 
has a unique structure of abelian presheaf with
transfers on $\wt{\cFC}$.

\subsection{Hecke operators}

\subsubsection{ }
A homomorphism of abelian presheaves with transfers
is a homomorphism of abelian presheaves compatible
with $f_*$. If $F$ is an abelian sheaf
on $\wt{\cFC}$,
any homomorphism of 
abelian presheaves
from an abelian presheaf with transfers
to $F$ is compatible with $f_*$.

\subsubsection{Hecke operators}
\label{sec:def Hecke}
Let $D =[X \xleftarrow{f}Z \xto{g} Y]$ be a diagram in $\wt{\cFC}$
such that $g$ belongs to $\wt{\cFT}$.
Let $G$ be an abelian presheaf with transfer on $\wt{\cFC}$. We define
the Hecke operator $T_D$ for $G$ to be the composite
$$
T_D = g_* \circ f^* \colon  G(X) \to G(Y).
$$

\subsubsection{Examples of Hecke operators}
\label{sec:exam Hecke}
We will give an interpretation of the Hecke operators
with the usual Hecke operators in terms of 
double cosets in Section~\ref{sec:Hecke double}.

Let $d \ge 1$ be an integer. Let $X$ be as in Section \ref{sec:4},
i.e., $X$ is a regular noetherian scheme of pure Krull dimension one
such that the residue field at each closed point is finite.
Let us consider the category $\cC^d$ introduced in Section \ref{sec:Cd_defn}
and the Grothendieck topologies $J^d$ and $J^d_m$ on $\cC^d$
introduced in the statement of Theorem \ref{thm:section2}.
It follows from Theorem \ref{thm:section2} that 
$(\cC^d,J^d)$ and $(\cC^d,J^d_m)$ are $Y$-sites.
Let us denote by $\cFCot{d}$ (\resp by $\cFCotu{d}$) 
the category $\wt{\cFC}$ introduced in Section \ref{sec:wtcFC}
constructed from the $Y$-site 
$(\cC^d,J^d)$ (\resp $(\cC^d,J^d_m)$).

Let $N$ be an object in $\cCo{d}$, 
and $F'$ be an object in $\cFCot{d}$
(\resp in $\cFCotu{d}$). 
Suppose that $F'$ is of the form
$F' = \coprod_j \quot{N'_j}{H_j}$
(\resp $F' = \coprod_j \quotu{N'_j}{H_j}$)
such that $N'_j \oplus N$ is an object
in $\cCo{d}$ for every $j$. 
We define an object $F'\oplus[N]$ in $\cFCot{d}$
(\resp in $\cFCotu{d}$)
by
$$
F'\oplus [N] = \coprod_j 
\quot{(N'_j\oplus N)}{(H_j\times \Aut_{\cO_X}(N))}.
$$
The two morphisms
$N'_j = N'_j \hookrightarrow N'_j \oplus N$
and $N'_j
\twoheadleftarrow N'_j \oplus N = N'_j \oplus N$
induce the morphisms
$$
r_{F'\oplus [N],F'}, 
m_{F'\oplus [N],F'}\colon  F'\oplus [N] \to F'
$$
in $\cFCot{d}$
(\resp in $\cFCotu{d}$).
When $F'$ is an object in $\cFCotu{d}$,
the morphism $m_{F'\oplus [N],F'}$ is a \fibr
in $\cFCotu{d}$.

Let $G$ be an abelian presheaf with
transfers on $\cFCot{d}$
(\resp on $\cFCotu{d}$). 
The composite 
$$
(m_{F'\oplus [N],F'})_{*} 
r_{F'\oplus [N],F'}^* \colon  G(F')\to G(F')
$$ 
is called the Hecke operator
for $[N]$ (\resp $[N]^{\uppr}$) 
and is denoted by $T_{[N]}$
(\resp $T_{[N]^{\uppr}}$).


\subsection{Presheaf of rings with transfers}
\begin{defn}
A {\em presheaf of rings with transfers} on $\wt{\cFC}$
is a presheaf $G$ of rings on $\wt{\cFC}$
equipped with a structure of abelian presheaf 
with transfers 
satisfying the following property:
\begin{itemize}
\item For any morphism
$f\colon F\to F'$ in $\wt{\cFT}$, 
for any $x\in G(F)$, and for any $y\in G(F')$, we have
$f_*(x\cdot f^* y) = f_*(x)\cdot y$ and $f_*(f^* y \cdot x)
= y \cdot f_*(x)$.
\end{itemize}
\end{defn}
\label{def:ring transfer}
Any sheaf of rings on $\wt{\cFC}$ 
has a unique structure of presheaf of rings with
transfers on $\wt{\cFC}$.
This can be proved as follows. 
Suppose that $G$ is a sheaf of rings on $\wt{\cFC}$. 
It follows from Proposition \ref{prop:transfer} that
$G$, regarded as an abelian sheaf with respect to the
addition, has a unique structure of presheaves
with transfers on $\wt{\cFC}$.
We claim that the collection $(f_*)$ of the transfer homomorphisms 
satisfies the condition above.
To prove the claim, we may and will assume that $f$ is
a morphism in $\wt{\cC}$ which belongs to $\wt{\cT}$.
Let us take a morphism $h\colon F_1 \to F$ in $\wt{\cC}$ which belongs
to $\wt{\cT}$
such that the composite $f \circ h$ is a Galois covering in $\wt{\cC}$.
Then for any $x\in G(F)$ and for any $y\in G(F')$, we have
$(f\circ h)^* f_*(x \cdot f^* y) = 
\sum_{h' \in \Hom_{F'}(F_1,F)} h'^* (x \cdot f^* y)
= \sum_{h' \in \Hom_{F'}(F_1,F)} (h'^* x \cdot (f \circ h)^* y)
= (\sum_{h' \in \Hom_{F'}(F_1,F)} h'^* x)\cdot (f \circ h)^* y
= (f\circ h)^* f_*(x) \cdot (f \circ h)^* y
= (f \circ h)^* (f_*(x) \cdot y)$.
Since $(f \circ h)^*$ is injective, we have
one of the desired equalities 
$f_*(x \cdot f^* y) = f_*(x) \cdot y$.
In a manner similar to that, we can prove
the other equality $f_*(f^* y \cdot x) = y \cdot f_*(x)$.

%

\begin{lem}\label{lem:ring_transfer}
Let $G$ be a presheaf of rings with transfers on $\wt{\cFC}$.
Let $f_1\colon F_1 \to F'$ and $f_2\colon F_2 \to F'$ be morphisms 
in $\wt{\cFC}$ which belong to $\wt{\cFT}$.
For $i=1,2$, let $\pr_i\colon  F_1 \times_{F'} F_2 \to F_i$
denote the $i$-th projection. We set $f=f_1 \circ \pr_1
= f_2 \circ \pr_2$.
Then for any $x \in G(F_1)$ and for any $y \in G(F_2)$,
we have $f_*(\pr_1^* x \cdot \pr_2^* y) = (f_1)_* x
\cdot (f_2)_* y$.
\end{lem}

\begin{proof}
Since $f= f_1 \circ \pr_1$, we have
$f_*(\pr_1^* x \cdot \pr_2^* y) = (f_1)_*
(x \cdot (\pr_1)_* \pr_2^* y)$.
It follows from Condition (2) in Definition \ref{def:transfer}
that we have $(\pr_1)_* \circ \pr_2^* = f_1^* (f_2)_*$.
Hence
$f_*(\pr_1^* x \cdot \pr_2^* y) = (f_1)_*
(x \cdot f_1^* (f_2)_* y)
= (f_1)_* x \cdot (f_2)_* y$.
This proves the claim.
\end{proof}

\section{The compact induction functor}
\label{sec:descent}
Let $(\cC,J)$ be a $Y$-site. Let $\wt{\cC}$ be as in Section \ref{sec:Ctil}.
Let $\wt{\cFC}$ be as in Section \ref{sec:wtcFC}.
We assume that for any object $X$ of $\cC$, the overcategory $\cC(\cT(J))_{/X}$
satisfies the condition (1) of \cite[5.8.1]{Grids}.
We note that the last condition is equivalent to the following condition:
For any morphism $f\colon Y \to X$ in $\cC$ which belongs to
$\cT(J)$, there exists only finitely many endomorphisms $Y \to Y$
over $X$.

Suppose that $\bM$ is a topological monoid and $\bK, \bK' \subset \bM^\times$
are open subgroups such that $\bK'$ is a normal subgroup of $\bK$.
Set $G = \bK/\bK'$ and let $\cB(G)$ denote the category of
finite left $G$-sets. Then by sending an object $S$ of $\cB(G)$ to
the smooth $\bM$-set $\bM \times^{\bK} S := \bM \times S/\sim$,
where $\sim$ is the equivalence relation given by $(mk,s) \sim (m,ks)$
for any $m\in \bM$, $k \in \bK$, and $s \in S$, we obtain a functor
from $\cB(G)$ to the category of smooth left $\bM$-sets.
Below we construct a similar functor to $\wt{\cFC}$, which we call
the compact induction functor.

For example, we take $\bM=\GL_d(L)$ 
where $L$ is a nonarchimedean local field,
$\bK=\GL_d(\cO_L)$ for the ring of integer $\cO_L$ of $L$,
and $\bK'=\Ker[\GL_d(\cO_L) \to \GL_d(\cO_L/\varpi \cO_L)]$
where $\varpi$ is a uniformizer.   

\subsection{Preliminary on the overcategory $\wt{\cFC}_{/F'}$}
\subsubsection{ }
Let $F'$ be an object in $\wt{\cFC}$.
Let us consider the overcategory $\wt{\cFC}_{/F'}$.
We let $\wt{\cC}_{/F'} \subset \wt{\cFC}_{/F'}$ denote the full subcategory
whose objects are the morphisms $F \to F'$ in $\wt{\cFC}$ such that
$F$ is an object in $\wt{\cC}$.
This notation may be confused with that of overcategories.
However there is no risk of serious confusion, since
if $F'$ is an object of $\wt{\cC}$, then
$\wt{\cC}_{/F'}$ is equal to the overcategory of morphisms 
to $F'$ in $\cC$.
It then can be checked easily that the category
$\wt{\cFC}_{/F'}$ is an $F$-category, and an object 
of $\wt{\cFC}_{/F'}$ is connected if and only if 
it is isomorphic to an object of $\wt{\cC}_{/F'}$.

Let us write $F' = \coprod_{i \in \pi_0(F')} F'_i$.
We regard each $F'_i$ as an object of $\wt{\cC}$.
Observe that $\wt{\cC}_{/F'}$ is equal to the categorical sum
$\coprod_{i \in \pi_0(F')} \wt{\cC}_{/F'_i}$.
For $i \in \pi_0(F')$, let $\jmath_i\colon  \wt{\cC}_{/F'_i} \inj \wt{\cC}$
denote the inclusion functor.
Let $\wt{\cT}_{/F'_i}$ denote the set of morphisms $f$ in
$\wt{\cC}_{/F'_i}$ such that $\jmath_i(f)$ belongs to $\wt{\cT}$.
We saw in Proposition \ref{prop:Ctil_Bsite}
that $(\wt{\cC},\iota_* J)$ is a $B$-site.
It follows that $\wt{\cT}_{/F'_i}$ is semi-localizing
and any morphism in $\wt{\cC}_{/F'_i}$ is an epimorphism.
It follows from Corollary \ref{cor:bicovering} that
the Grothendieck topology $\jmath_i^* \iota_* J$ on 
$\wt{\cC}_{/F'_i}$ is equal to the $A$-topology given by
$\wt{\cT}_{/F'_i}$.
Moreover, it follows from Lemma \ref{lem:imath_enough_Galois}
that the category $\wt{\cC}_{/F'_i}(\wt{\cT}_{/F'_i})$ 
has enough Galois coverings.
Let $\wt{\cFT}_{/F'}$ denote the set of morphisms
in $\wt{\cFC}_{/F'}$ which belongs to $\wt{\cFT}$ as
a morphism in $\wt{\cFC}$. 

Let $\iota' \colon  \cC \inj \wt{\cFC}$ denote the composite 
$\cC \xto{\iota} \wt{\cC} \inj \wt{\cFC}$ and let
$\jmath\colon  \wt{\cFC}_{/F'} \inj \wt{\cFC}$ 
denote the inclusion functor.
Let us consider the Grothendieck topology $\jmath^* \iota'_* J$.
Then it is not hard to check, by using Corollary \ref{cor:bicovering} that
a contravariant functor $G$ from the category $\wt{\cFC}_{/F'}$ 
to the category of sets is a sheaf
on $(\wt{\cFC}_{/F'}, \jmath^* \iota'_* J)$
if and only if $G$
satisfies Conditions (1) and (2) in Section \ref{sec:cond12}
and the following condition:
\begin{enumerate}
\setcounter{enumi}{2}
\item For each $i \in \pi_0(F')$, the restriction of $G$ to
$\wt{\cC}_{/F'_i}$ is a sheaf on $(\wt{\cC}_{/F'_i}, \jmath_i^* \iota_* J)$.
\end{enumerate}

\subsubsection{ }
Let $f\colon  F \to F'$ be a morphism in $\wt{\cFC}$.
Let $\iota_f \colon  \wt{\cFC}_{/F} \to \wt{\cFC}_{/F'}$ denote the functor
given by the composition with $f$.
It is clear that the functor $\iota_f$ commutes with
 fiber products. Hence it follows from the definition
of sheaves on $\wt{\cFC}_{/F}$ that
for any sheaf $G$ on $\wt{\cFC}_{/F'}$, its composite with
$\iota_f$ is a sheaf on $\wt{\cFC}_{/F}$.
We denote the sheaf on $\wt{\cFC}_{/F}$ by
$f^* G$ and call it the pullback of $G$ with respect to
the morphism $f$.
By associating $f^*G$ to each sheaf 
$G$ on $\wt{\cFC}_{/F'}$, we obtain a functor $f^*$
from the category of sheaves on 
$\wt{\cFC}_{/F'}$ to the category of sheaves
on $\wt{\cFC}_{/F}$.

\begin{prop}\label{prop:Gal_descent}
Let $f\colon  F \to F'$ be a morphism in $\wt{\cC}$
which belongs to $\wt{\cT}$.
Let $\phi\colon G_1 \to G_2$ be a morphism of sheaves 
on $\wt{\cFC}_{/F'}$.
Then $\phi$ is an isomorphism if its pullback 
$f^*(\phi)\colon  f^* G_1 \to f^* G_2$ with respect to $f$
is an isomorphism.
\end{prop}

\begin{proof}
We may and will assume that $f$ is a Galois covering which
belongs to $\wt{\cT}$.
Let $g'\colon F'_1 \to F'$ be an object in $\wt{\cFC}_{/F'}$.
We prove that the map $\phi(F'_1) \colon  G_1(F'_1) \to G_2(F'_1)$
given by $\phi$ is bijective.
Since $G_1$ and $G_2$ are sheaves on $\wt{\cFC}_{/F'}$,
they satisfy Conditions (1) and (2) in Section \ref{sec:cond12}.
Hence to prove $\phi(F'_1)$ is bijective, we may assume that 
$g'\colon F'_1 \to F'$ is an object in $\cC_{/F'}$.
It follows from Proposition \ref{prop:Ctil_Bsite} that
there exist an object $F_1$ in $\wt{\cC}$ and morphisms
$f_1 \colon F_1 \to F'_1$ and $g\colon F_1 \to F$ satisfying
$g' \circ f_1 = f \circ g$.
We may and will assume,
by using Lemma \ref{lem:tilde semi-localizing}, 
that $f_1$ belongs to $\wt{\cT}$.
It then follows from Lemma \ref{lem:imath_enough_Galois}
that there exist an object $F_2$ in $\wt{\cC}$ 
and a morphism $h\colon F_2 \to F_1$ which belongs to $\wt{\cT}$
such that $f_1 \circ h\colon F_2 \to F'_1$ is a Galois covering
in $\wt{\cC}$ which belongs to $\wt{\cT}$.
Since $f^*(\phi)$ is an isomorphism,
it follows that $\phi(F_2)\colon G_1(F_2) \to G_2(F_2)$ 
is bijective. Hence $\phi(F_2)$ induces a
bijection $G_1(F_2)^{\Aut_{F'_1}(F_2)} 
\to G_2(F_2)^{\Aut_{F'_1}(F_2)}$.
Since $G_1$ and $G_2$ are sheaves on $\wt{\cFC}_{/F'}$,
it follows that the map
$\phi(F'_1)\colon G_1(F'_1) \to G_2(F'_1)$
is bijective. This proves the claim.
\end{proof}

\begin{cor}\label{cor:cartesian_descent}
Let $f\colon  F \to F'$ be a morphism in $\wt{\cC}$
which belongs to $\wt{\cT}$.
Then a commutative square in the category
$\wt{\cFC}_{/F'}$ is cartesian if its base change 
to the square in the category $\wt{\cFC}_{/F}$ is cartesian.
\end{cor}

\begin{proof}
For an object $G$ in $\wt{\cFC}$, we denote by
$h_G$ the presheaf on $\wt{\cFC}$ represented by $G$.
One can check easily that $h_G$ satisfies Conditions (1) and (2)
in Section \ref{sec:cond12}.
It follows from Corollary 3.1.5 of \cite{Grids} and
Lemma \ref{lem:Galoiscovering} that
the restriction of $h_G$ to $\wt{\cC}$ is a sheaf on $\wt{\cC}$.
Hence it follows from Proposition \ref{prop:cFT_main} that 
$h_G$ is a sheaf on $\wt{\cFC}$.
Thus, its pullback $h_G|_{\wt{\cFC}_{/F'}}$ to
$\wt{\cFC}_{/F'}$ is a sheaf on $\wt{\cFC}_{/F'}$.

Let
\begin{equation}\label{eq:comm_sq}
\begin{CD}
F_1 @>{g}>> F_2 \\
@V{f_1}VV @V{f_2}VV \\
F'_1 @>{g'}>> F'_2
\end{CD}
\end{equation}
be a commutative square in the category
$\wt{\cFC}_{/F'}$.
Suppose that its base change
$$
\begin{CD}
F_1 \times_{F'} F @>>> F_2 \times_{F'} F \\
@VVV @VVV \\
F'_1 \times_{F'} F @>>> F'_2 \times_{F'} F
\end{CD}
$$
is cartesian in the category $\wt{\cFC}_{/F}$.
The commutative diagram (\ref{eq:comm_sq})
gives a morphism
$\phi \colon  h_{F_1}|_{\wt{\cFC}_{/F'}} \to (h_{F'_1} 
\times_{h_{F'_2}} h_{F_2})|_{\wt{\cFC}_{/F'}}$
of sheaves on $\wt{\cFC}_{/F'}$.
Here $h_{F'_1} \times_{h_{F'_2}} h_{F_2}$
denotes the fiber product in the category
of presheaves on $\wt{\cFC}$. 
It follows from Lemma \ref{lem:aJ_properties} and 
Lemma \ref{lem:cFC_limit}
that $h_{F'_1} \times_{h_{F'_2}} h_{F_2}$ is a sheaf on $\wt{\cFC}$
and hence $(h_{F'_1} 
\times_{h_{F'_2}} h_{F_2})|_{\wt{\cFC}_{/F'}}$
is a sheaf on $\wt{\cFC}_{/F'}$.
It follows from our assumption that the
pullback $f^*(\phi)$ with respect to $f$
is an isomorphism. Hence it follows from
Proposition \ref{prop:Gal_descent} that
$\phi$ is an isomorphism.
Hence the diagram (\ref{eq:comm_sq}) is cartesian.
This proves the claim.
\end{proof}

\subsection{The category $\cB(G)$}
For a finite group $G$, we denote by $\cB(G)$
the following category: An object in $\cB(G)$ 
is a finite set with a left $G$-action,
and a morphism in $\cB(G)$ is a $G$-equivariant map.
It can be checked easily that the category
$\cB(G)$ is an $F$-category and an object $S$ in
$\cB(G)$ is connected if and only if $S$ is non-empty and
the group $G$ acts transitively on $S$.

\subsection{The compact induction functor from $\cB(G)$}
\label{sec:functor M_G}
Let $c\colon F \to F'$ be a Galois covering in $\wt{\cC^m}$
which belongs to $\wt{\cT^m}$.
Let $G = \Aut_{F'}(F)$ be its Galois group.
We note that, as a consequence of Lemma \ref{lem:Galoiscovering},
$c\colon F \to F'$ is a quotient object of $F$ by $G$.
In this paragraph we construct a functor $M_G$ from 
$\cB(G)$ to $\wt{\cFC^m}$. 

Let $S$ be an object in $\cB(G)$.
We let the group $G$ act on $\coprod_{s \in S} F$
in such a way that for each $g \in G$, the map
$\pi_0(g)\colon  S \to S$ 
is given by the action of $g$ on $S$, and that
the component at $s$ of $g\colon \coprod_{s \in S} F \to \coprod_{s \in S} F$
is equal to the morphism $g\colon F \to F$.
In this way, we can regard $G$ as a subgroup of 
$\Aut_{\wt{\cFC}}(\coprod_{s \in S} F)$
when $S$ is non-empty.
We set $M_G(S) = M_{G,F}(S) = \quot{\coprod_{s \in S} F}{G}$.
Here for the notation $\quot{\coprod_{s \in S} F}{G}$ we refer 
to the statement of Lemma \ref{lem:quot_Ctil}.
When $S$ is empty, the symbol $\quot{\coprod_{s \in S} F}{G}$
stands for an initial object
in $\wt{\cFC}$.

The argument in the proof of Lemma \ref{lem:quot_FCd} shows that
$M_G(S)$ is a \quotobj in $\Shv(\cC,J)$ 
of $\coprod_{s \in S} F$ by $G$.
Hence the morphism $\coprod_{s \in S} F \to F$ in $\wt{\cFC^m}$,
whose component at $s$ is equal to the identity morphism $\id_F$
for each $s\in S$, 
induces a morphism $M_G(S) \to F'$ in $\Shv(\cC,J)$.
Hence it follows from Lemma \ref{lem:quot_FCtil} that
$M_G(S)$ is an object of $\wt{\cFC^m}$ and
that $M_G(S)$ is a \quotobj in $\wt{\cFC^m}$ 
of $\coprod_{s \in S} F$ by $G$.
We will frequently use this fact in this section.
The morphism $M_G(S) \to F'$ constructed above 
gives an object in $\wt{\cFC^m}_{/F'}$,
which we denote by $\wt{M}_G(S)$ by abuse of notation.
For a morphism $f\colon S \to T$ in $\cB(G)$, we construct
a morphism $M_G(f)\colon M_G(S) \to M_G(T)$ in $\wt{\cFC^m}_{/F'}$ as follows.
Let $\alpha$ denote the morphism
$\coprod_{s \in S} F \to \coprod_{t \in T} F$
characterized by the following property: 
The map $\pi_0(\alpha)\colon  S \to T$ is equal to 
$f$ and the component $\alpha$ at each $s \in S$ is
the identity morphism of $F$.
Let us consider the composite 
$\coprod_{s \in S} F \to M_G(T)$ 
of the morphism $\alpha$
with the canonical quotient map
$\coprod_{t \in T} F \to M_G(T)$.
It follows from the universality of
the \quotobjs that this composite factors through
the canonical morphism 
$\coprod_{s \in S} F \to M_G(S)$.
Hence we obtain a morphism $M_G(S) \to M_G(T)$
which we denote by $M_G(f)$.
It can be checked easily that the morphism $M_G(f)$
in $\wt{\cFC^m}$ is over $F'$. Hence $M_G(f)$ gives a morphism
$\wt{M}_G(S) \to \wt{M}_G(T)$ in $\wt{\cFC^m}_{/F'}$ which we denote by
$\wt{M}_G(f)$.
We thus obtain functors $M_G\colon \cB(G) \to \wt{\cFC^m}$ and
$\wt{M}_G\colon \cB(G) \to \wt{\cFC^m}_{/F'}$ such that
$M_G$ is equal to the composite of $\wt{M}_G$
with the forgetful functor $\wt{\cFC^m}_{/F'} \to \wt{\cFC^m}$ that
sends an object $F_1 \to F'$ in $\wt{\cFC^m}_{/F'}$ to
the object $F_1$ in $\wt{\cFC^m}$.

The following lemma follows immediately from
the definition of the functor $M_G$.

\begin{lem}\label{lem:M_G_2}
For an object $S$ in $\cB(G)$, we have a bijection
$\pi_0(S) \xto{\cong} \pi_0(M_G(S))$ which is functorial 
in $S$.
\qed
\end{lem}

\begin{lem}\label{lem:M_base_change}
For an object $S$ in $\cB(G)$, the diagram
$$
\begin{CD}
\coprod_{s \in S} F @>{(1)}>> F \\
@V{(2)}VV @VV{c}V \\
M_G(S) @>{(3)}>> F'
\end{CD}
$$
in $\wt{\cFC^m}$ is cartesian.
Here (1) is the morphism 
whose component at $s$ is the identity morphism
$\id_F$ for each $s \in S$,
the morphism (2) is the canonical morphism
$\coprod_{s \in S} F \to M_G(S)$,
and (3) is the structure morphism of the object
$\wt{M}_G(S)$ in $\wt{\cFC}_{/F'}$.
\end{lem}

\begin{proof}
First suppose that $S$ is a connected object in $\cB(G)$.
Let us choose $s \in S$ and let $G_s \subset G$
denote the stabilizer of $s$.
Then the inclusion $F \inj \coprod_{s' \in S} F$
of the component at $s$
induces an isomorphism $\quot{F}{G_s} \xto{\cong} M_G(S)$.
Let $f_1\colon  \quot{F}{G_s} \xto{f_1} F' \xleftarrow{f_2} F$ 
be the morphism in $\wt{\cFC^m}$
induced by the identity 
morphism on $F$.

It follows from Lemma \ref{lem:Galois2} that $c$ is a Galois covering
and that we have 
$$
\Hom_{F'}(F,F) \cong \Aut_{F'}(F) = G.
$$
It follows from Lemma \ref{lem:Galois2}
that the canonical morphism $c_s\colon F \to \quot{F}{G_s}$
is a Galois covering.
It follows from Lemma 4.2.3 (2) of \cite{Grids} that the map
$$
\Hom_{F'}(F,F) 
\to \Hom_{F'}(F,\quot{F}{G_s})
$$
given by the composition with $c_s$ is surjective.
Hence it follows from the definition of a Galois covering that
we have a bijection 
$\Hom_{F'}(F,\quot{F}{G_s})
\cong \quot{G}{G_s}$.
Thus it follows from Lemma 4.2.3 (2) of \cite{Grids} that 
we have an isomorphism
$M_G(S) \times_{F'} F \cong \coprod_{s \in S} F$.
This proves the claim when $S$ is a connected object
in $\cB(G)$.

Since $\cB(G)$ is an $F$-category, 
we obtain, by using Lemma \ref{lem:fiber_prod_basic},
an isomorphism
$\phi_S \colon  M_G(S) \times_{F'} F 
\xto{\cong} \coprod_{s \in S} F$
for any object $S$ in $\cB(G)$.
This proves the claim.
\end{proof}

\begin{lem}\label{lem:M_fiber_products}
The functors $\wt{M}_G$ and $M_G$ commute with fiber products.
\end{lem}
\begin{proof}
Since the forgetful functor
$\wt{\cFC^m}_{/F'} \to \wt{\cFC^m}$ commutes with fiber products,
it suffices to show that 
the functor $\wt{M}_G$ commutes with fiber products.

By Lemma \ref{lem:M_base_change}, we have an isomorphism
$\phi_S \colon  M_G(S) \times_{F'} F 
\xto{\cong} \coprod_{s \in S} F$
in $\wt{\cFC^m}$ for any object $S$ in $\cB(G)$.
It can be checked easily that for a morphism $f\colon S \to T$
in $\cB(G)$, the morphism 
$M_G(S) \times_{F'} F \to 
M_G(T) \times_{F'} F$
induced by $M_G(f)$ is identified, via the
isomorphisms $\phi_S$ and $\phi_T$,
with the morphism 
$\phi_f \colon  \coprod_{s \in S} F \to 
\coprod_{t \in T} F$
such that the map $\pi_0(\phi_f)$ 
is equal to the map $f$ and for each $s \in S$, the component
at $s$ of $\phi_f$ is equal to the identity morphism on $F$.
From this we can see easily that the functor
$\cB(G) \to \wt{\cFC^m}_{/F}$
which associates to each object $S$ in $\cB(G)$
the object $M_G(S) \times_{F'} F$
commutes with fiber products.
Hence the claim follows from Corollary 
\ref{cor:cartesian_descent}.
\end{proof}

\begin{lem}\label{lem:M_G_fibr}
For any morphism $f\colon S \to T$ in $\cB(G)$, the
morphism $M_G(f)$ in $\wt{\cFC^m}_{/F'}$ belongs to
$\wt{\cFT^m}_{/F'}$.
\end{lem}

\begin{proof}
Since any morphism in $\cB(G)$ is written as a composite 
$i \circ j$ of morphisms in $\cB(G)$ with $i$ injective and
$j$ surjective, we may assume that $f$ is either injective 
or surjective.
Suppose that $f$ is injective. Then $T$ is a coproduct in $\cB(G)$
of $S$ and $T\setminus S$. Then $M_G(T)$ is identified with a
coproduct of $M_G(S)$ and $M_G(T \setminus S)$ and
under this identification the morphism $M_G(f)$ is equal to
the canonical morphism $M_G(S) \inj M_G(S) \amalg M_G(T \setminus S)$.
This in particular implies that $M_G(f)$ belongs to $\wt{\cFT^m}_{/F'}$
in this case.

Suppose that $f$ is surjective.
Let $c_S\colon  \coprod_{s \in S} F \to M_G(S)$ and
$c_T\colon \coprod_{t \in T} F \to M_G(T)$
denote the canonical morphisms.
Let us denote by $f'$ the morphism
$$
\coprod_{s \in S} F
\to
\coprod_{t \in T} F
$$
characterized by the following property: 
the map $\pi_0(f')\colon  S \to T$ is equal to 
$f$ and the component of $f'$ at each $s \in S$ is
the identity morphism of $F$.
We then have $M_G(f) \circ c_S = c_T \circ f'$.
Since $c_S$, $c_T$, and $f'$ belong to $\wt{\cFT^m}^*$, 
it follows from Lemma \ref{lem:cFT semi-localizing} 
that $M_G(f)$ belongs to $\wt{\cFT^m}^*$.
This proves the claim.
\end{proof}

\begin{lem}\label{lem:M_G_deg}
Let $f\colon S \to T$ be a morphism in $\cB(G)$.
The map $S \to \Z$ which sends $s \in S$
to the cardinality of $f^{-1}(s)$ factors
through the surjective map $S \surj
\quot{S}{G} = \pi_0(S)$.
We denote the induced map $\pi_0(S) \to \Z$
by $\deg f$.
Then $\deg f$ is equal to the composite
of the bijection $\pi_0(S)\xto{\cong}
\pi_0(M_G(S))$ in Lemma \ref{lem:M_G_2}
with the map $\deg M_G(f)\colon  \pi_0(M_G(f)) \to \Z$.
\end{lem}

\begin{proof}
%
The argument in the proof of Lemma \ref{lem:M_fiber_products}
shows that the diagram
$$
\begin{CD}
\coprod_{s \in S} F @>>> M_G(S) \\
@VVV @VV{M_G(f)}V \\\
\coprod_{t \in T} F @>>> M_G(T)
\end{CD}
$$
in $\wt{\cFC^m}$ is cartesian. Hence the claim follows from
Corollary \ref{cor:deg_limit}.
\end{proof}

\subsection{Composition law}
\label{sec:M_G_H}
Let $H \subset G$ be a subgroup.
We set $F''=\quot{F}{H}$.
It follows from Lemma \ref{lem:quot_FCd} that
the canonical morphism $F \to F'$
factors through the canonical morphism
$F \to F''$.
We denote by $c^{/H}_{/G}$ the induced morphism 
$F''\to F'$.

For an object $S$ in $\cB(G)$, we denote by $r^G_H(S)$ the
object in $\cB(H)$ obtained by restricting the action of $G$
on $S$ to $H$. By definition we have $r^G_H(S) =S$ as a set.
%
By associating $r^G_H(S)$ for each object $S$ in $\cB(G)$,
we obtain a functor $\cB(G) \to \cB(H)$ which we denote by $r^G_H$.
For an object $S$ in $\cB(G)$, we let
$c_{/G,S}\colon \coprod_{s \in S} F \to M_G(S)$
denote the canonical morphism.
Then the morphism $c_{/G,S}$
factors through the canonical morphism $c_{/H,r^G_H(S)}$.
We denote by $c^{/H}_{/G,S}\colon M_H(r^G_H(S)) \to M_G(S)$ the induced morphism.

For an object $S$ in $\cB(H)$,
we denote by $G \times^H S$ the quotient
of the direct product $G \times S$ as sets 
under the following equivalence relation:
we say that two elements $(g,s), (g',s')
\in G \times S$ are equivalent if there exists
an element $h \in H$ which satisfies
$(g',s')=(gh^{-1},hs)$.
For $g \in G$ and for an element $x \in G \times^H S$
represented by $(g',s) \in G \times S$, we define
$gx$ to be the class of $(gg',s)$.
This gives an action from the left of the group $G$
on the set $G \times^H S$.
Since $G$ and $H$ are finite groups, 
the $G$-set $G \times^H S$ is an object in $\cB(G)$.
We thus obtain a functor $G \times^H - \colon 
\cB(H) \to \cB(G)$ which sends
an object $S$ in $\cB(H)$ to the object
$G \times^H S$ in $\cB(G)$.
For an object $S$ in $\cB(H)$, we
let $j_{G,S} \colon  S \to G \times^H S$ denote the map
of sets which sends $s \in S$ to the class of
$(1,s) \in G \times S$.

Let $S$ be an object in $\cB(H)$ and let
$T$ be an object in $\cB(G)$. Suppose that
a morphism $f\colon S \to r^G_H(T)$ in $\cB(H)$
is given.
The map $G \times S \to T$ which sends
$(g,s) \in G \times S$ to $g f(s)$ factors through
the quotient $G \times S \surj G \times^H S$.
We denote the induced map $G \times^H S \to T$
by $i^G_H(f)$.
It is easy to see that the map $i^G_H(f)$ is a
morphism in $\cB(G)$.
The map $\Hom_{\cB(H)}(S,r^G_H(T)) \to \Hom_{\cB(G)}(G \times^H S,T)$
that sends $f \in \Hom_{\cB(H)}(S,r^G_H(T))$ to $i^G_H(f)$
is bijective. In fact, its inverse map is given by sending
a morphism $f'\colon  G \times^H S \to T$ in $\cB(G)$ to its composite
with the map $j_{G,S}\colon S \to G \times^H S$.
Hence the functor 
$G \times^H - \colon  \cB(H) \to \cB(G)$ is left
adjoint to the functor $r^G_H \colon  \cB(G) \to \cB(H)$.

Let $c^{/H}_{/G} \circ-\colon  \wt{\cFC^m}_{/F''} \to \wt{\cFC^m}_{/F'}$ 
denote the functor
that sends an object $f\colon F_1 \to F''$
in $\wt{\cFC^m}_{/F''}$ to the object
$c^{/H}_{/G} \circ f\colon F_1 \to F'$ in $\wt{\cFC^m}_{/F'}$.
\begin{lem}\label{lem:M_G_H}
Let the notation be as above.
Then the functor $(c^{/H}_{/G} \circ -)\circ \wt{M}_H$ is 
canonically isomorphic to the functor 
$\wt{M}_G \circ (G \times^H -)$.
\end{lem}

\begin{proof}
Let $S$ be an object in $\cB(H)$.
Let $j_{G,S,*} \colon  \coprod_{s \in S} F
\to \coprod_{t \in G \times^H S} F$ denote the
morphism in $\wt{\cFC}$ such that the map
$\pi_0(j_{G,S,*})$ is equal to $j_{G,S}\colon S \to G \times^H S$
and that for each $s \in S$, the component of $j_{G,S,*}$
at $s$ is the identity morphism $\id_F$.
It is then easy to check that the morphism
$j_{G,S,*}$ induces a bijection
$\quot{(\coprod_{s \in S} F(-))}{H}
\xto{\cong} 
\quot{(\coprod_{t \in G \times^H S} F(-))}{G}$
of presheaves on $\cC$.
By applying the sheafifcation functor, 
we obtain an isomorphism
$\quot{(\coprod_{s \in S} F)}{H} \xto{\cong} 
\quot{(\coprod_{t \in G \times^H S} F)}{G}$ 
in $\wt{\cFC^m}$.
It is easy to check that this isomorphism induces an isomorphism
$M_H(S) \xto{\cong} M_G(G \times^H S)$ 
in $\wt{\cFC}_{/F'}$.
By construction, this isomorphism is functorial in $S$.
Thus, we obtain an isomorphism  $(c^{/H}_{/G} \circ -)\circ \wt{M}_H
\xto{\cong} \wt{M}_G \circ (G \times^H -)$ of functors.
This proves the claim.
\end{proof}

\begin{lem}\label{lem:M_G_H2}
Let $S$ be an object in $\cB(G)$.
Let $a^G_{H,S}\colon  G \times^H r^G_H(S) \to S$ denote the map
which sends the class of $(g,s) \in G \times S$
in $G \times^H r^G_H(S)$ to $g s \in S$.
Then $a^G_{H,S}$ is a morphism in $\cB(G)$ and
the morphism $c^{/H}_{/G,S} \colon  M_H(r^G_H(S))
\to M_G(S)$ is equal to the composite
of the isomorphism 
$\psi_S \colon M_H(r^G_H(S)) \cong M_G(G \times^H r^G_H(S))$ 
given by Lemma \ref{lem:M_G_H}
with the morphism $M_G(a^G_{H,S})\colon M_G(G \times^H r^G_H(S))
\to M_G(S)$.
\end{lem}

\begin{proof}
It is easy to check that the map $a^G_{H,S}$ is a morphism
in $\cB(G)$. 
It follows from Lemma 3.2.2 of \cite{Grids} that
the canonical morphism $c_{/H,r^G_H(S)}$ is an epimorphism. 
By the definition of the morphism $c^{/H}_{/G,S}$,
we have $c^{/H}_{/G,S} \circ c_{/H,r^G_H(S)}= c_{/G,S}$.
Hence it suffices to prove that $c_{/G,S}$ is equal to
the composite $M_G(a^G_{H,S}) \circ \psi_S \circ c_{/H,r^G_H(S)}$.
It can be checked easily that the composite
$\psi_S \circ c_{/H,r^G_H(S)}$ is equal to the composite
of the morphism $\iota(j_{H,r^G_H(S),*})$
with $c_{/G,G \times^H r^G_H(S)}$.
%
Since the composite 
$$
S = r^G_H(S) \xto{j_{H,r^G_H(S)}}
G \times^H r^G_H(S) \xto{a^G_{H,S}} S
$$
is equal to the identity, we obtain the claim.
\end{proof}

\begin{lem}\label{lem:M_G_H3}
Let $S$ be an object in $\cB(H)$ and let
$T$ be an object in $\cB(G)$. Let
$f\colon S \to r^G_H(T)$ be a morphism in $\cB(H)$.
Then the composite of $M_H(f) \colon M_H(S) \to M_H(r^G_H(T))$
with $c^{/H}_{/G,T} \colon  M_H(r^G_H(T)) \to M_G(T)$
is equal to the composite of the isomorphism
$M_H(S) \cong M_G(G \times^H S)$ given by Lemma \ref{lem:M_G_H}
with $M_G(i^G_H(f))\colon  M_G(G \times^H S) \to M_G(T)$.
\end{lem}

\begin{proof}
It follows from Lemma \ref{lem:M_G_H2} that
$c^{/H}_{/G,T}$ is equal to the composite
of the isomorphism 
$\psi_T \colon M_H(r^G_H(T)) \cong M_G(G \times^H r^G_H(T))$ 
given by Lemma \ref{lem:M_G_H}
with the morphism $M_G(a^G_{H,T})\colon M_G(G \times^H r^G_H(T))
\to M_G(T)$.
Since the composite
$$
G \times^H S \xto{G \times^H f} G \times^H r^G_H(T)
\xto{a^G_{H,T}} T
$$
is equal to $i^G_H(f)$, the claim follows.
\end{proof}

\subsection{}
Given full subcategories $\wt{\cC} \subset \wt{\cC^m}$
and $\wt{\cFC} \subset \wt{\cFC^m}$,
we consider the full subcategory
$\cB(G)_m \subset \cB(G)$ 
consisting of objects 
$S$ such that 
$\wt{M}_G(S) \in \wt{\cFC^m}$ 
is isomorphic to an object of 
$\wt{\cFC}$.
We will use the same notations 
such as $M_G$, $\wt{M}_G, c^{/H}_{/G,T}$, etc.

\chapter{The category $\cC^d$ and finite adeles}
\label{cha:Cd}
The aim of this chapter is to introduce 
$Y$-sites whose underlying category is
$\cC^d$.    
We give the definition of 
$\cC^d$ in Section~\ref{sec:Cd_defn} 
and introduce some Grothendieck topologies 
on it.
We give a proof that they are indeed $Y$-sites.
Then, we construct grids for those $Y$-sites.
The computation of the absolute Galois monoids
is given in the following Section~\ref{sec:4}.

We present one of our main theorems 
(Theorem~\ref{main theorem}) 
in Section~\ref{sec:Euler distributions}.
This is about the norm relation of elements
in some sections of 
presheaves with transfers on the $Y$-site
introduced in this chapter.   Its applications are
presented in Chapter~\ref{cha:applications}.

\section{The category $\cCo{d}$ and some Grothendieck topologies}
\label{sec:3}
We introduce the category $\cCo{d}$ and some Grothendieck 
topologies on it.  Then we show that they are $Y$-sites and 
construct grids for them.   

Let us recall the analogous statements in the 
case of the Galois theory of a field $F$.
The objects of the category are separable extensions
of the field $F$.   The morphisms are 
$F$-morphisms.    With the atomic topology,
it becomes a $Y$-site.
A grid can be constructed by introducing a fixed separable closure of $F$.
An object of the grid is a finite extension contained in
this separable closure.   Note that it is a poset.   
The Galois group is computed
as the limit of the (finite) Galois groups of the Galois 
extensions $L/F$ as $L$ runs over finite Galois extensions
in the fixed separable closure.

For our category $\cC^d$,
we do not introduce an analogue of a fixed separable closure,
but we introduce the category (poset) of lattices 
contained in $\cO_X^{\oplus d}$.    
Corresponding to the (poset of) 
finite Galois extensions is the category (poset) of 
pairs of lattices.    

The actual computation of the absolute Galois monoid
is done in the next section.   See its introduction 
for more detail.

Let $X$ be a regular noetherian scheme of
pure Krull dimension one.
%
We do not assume that $X$ is separated.

\subsection{Definition of the category $\cCo{d}$}\label{subsec:21}
Let $d \ge 1$ be a positive integer.

\subsubsection{ }\label{sec:Cd_defn}
We define the category $\cCo{d} =\cCXo{d}$ as
follows.
An object in $\cCo{d}$ is a coherent
$\cO_X$-module of finite length which admits a
surjection from $\cO_X^{\oplus d}$.
For two objects $N$ and $N'$ in $\cCo{d}$,
the set $\Hom_{\cCo{d}}(N,N')$ of morphisms from
$N$ to $N'$ is the set of isomorphism classes of diagrams
\[
N' \twoheadleftarrow N''  \hookrightarrow N
\]
in the category of coherent $\cO_X$-modules
where the left arrow is surjective and
the right arrow is injective. 
Here two diagrams
$N' \twoheadleftarrow N''  \hookrightarrow N$
and $N' \twoheadleftarrow N'''  \hookrightarrow N$
are considered to be isomorphic if there exists
an isomorphism $N''\xto{\cong} N'''$ of 
$\cO_X$-modules such that the diagram
$$
\begin{array}{ccccc}
N' & \twoheadleftarrow & N'' & \inj & N \\
{\Large \parallel} & & 
{\Large \downarrow} \cong & & 
{\Large \parallel} \\
N' & \twoheadleftarrow & N''' & \inj & N \\
\end{array}
$$
is commutative.
The composition of two morphisms 
$N' \twoheadleftarrow M \hookrightarrow N$
and 
$N'' \twoheadleftarrow M' \hookrightarrow N'$
is seen in the following diagram:
\[
\begin{array}{ccccc}
 &&&& N \\
 &&&& \uparrow \\
 &&N' &\twoheadleftarrow & M \\
 &&\uparrow & \ \begin{small}\square\end{small} & \uparrow \\
 N'' & \twoheadleftarrow & M' & \twoheadleftarrow & M\times_{N'} M'
 \end{array}
\]
where the box means that the square is cartesian.
This definition of
morphisms is due to Quillen (\cite{Qu})
except that here we take morphisms in the opposite
direction.

An object of the category $\cCo{d}$ can be regarded as an object of
the category of $\cO_X$-modules. Since these two categories have
quite different types of morphisms, it may cause serious confusion
when we discuss the morphisms of the category $\cCo{d}$.
To avoid such confusion, we use the following terminology:
For two objects $N$ and $N'$ of $\cCo{d}$, a morphism
from $N$ to $N'$ in $\cCo{d}$ is referred to as a morphism from
$N$ to $N'$, while a morphism from $N$ to $N'$ in the category
of $\cO_X$-modules is referred to as a homomorphism from
$N$ to $N'$.

\begin{rmk} \label{rmk1}
For a coherent $\cO_X$-module $M$ of finite length,
the condition that $M$ admits a
surjection from $\cO_X^{\oplus d}$ is equivalent to the
condition that $M$ admits a surjection from a locally
free $\cO_X$-module of rank $d$.
It is clear that the former condition implies the latter condition.
Let us prove that the latter condition implies the former condition.
To do this we need to introduce some notation.
For an effective Cartier divisor $Z \subset X$,
we let $i_Z \colon  Z \inj X$ denote the canonical inclusion
and $\sI_Z \subset \cO_X$ the sheaf of ideals
defining $Z$.
We note that $Z$ is affine and artinian.
We say that an $\cO_X$-module $M$ is annihilated by $\sI_Z$
if the adjunction homomorphism $M \to {i_Z}_* i_Z^* M$ 
is an isomorphism.
Let $M$ be a coherent $\cO_X$-module of finite length,
the condition that $M$ admits a surjection 
$\phi \colon  \cF \surj M$ from a locally
free $\cO_X$-module $\cF$ of rank $d$.
By our assumption, there exists an effective Cartier divisor
$Z \subset X$ such that $M$ is annihilated by $\sI_Z$.
Since $Z$ is artinian,
there exists an isomorphism 
$\alpha\colon  \cO_Z^{\oplus d} \xto{\cong} i_Z^* \cF$
of $\cO_Z$-modules. Since $i_Z^*$ is right exact and ${i_Z}_*$ is
exact, the composite
$$
\cO_X^{\oplus d} \to {i_Z}* i_Z^* \cO_X^{\oplus d} =
{i_Z}_* \cO_Z^{\oplus d}
\xto{{i_Z}_*(\alpha)} {i_Z}_* i_Z^* \cF 
\xto{{i_Z}_* i_Z^*(\phi)} {i_Z}_* i_Z^* M \xto{\cong} M
$$
is a surjective homomorphism of $\cO_X$-modules.
Hence $M$ admits a
surjection from $\cO_X^{\oplus d}$.
\end{rmk}

\subsubsection{ }
Suppose that $X=\Spec(A)$ is an affine scheme. Then
the functor that sends $N$ to the $A$-module of the global
section $\Gamma(X,N)$ gives an equivalence of categories 
from the category of coherent $\cO_X$-modules of finite length 
to the category of $A$-modules of finite length.
For an $A$-module $M$ of finite length,
(\resp a morphism $f$ of $A$-modules of finite length),
we denote, by abuse of notation, the corresponding 
coherent $\cO_X$-module of finite length 
(\resp the corresponding morphism of coherent $\cO_X$-modules 
of finite length)
by the same symbol $M$ (\resp $f$) if there is 
no risk of confusion.

\subsubsection{ }
\label{sec:notationz}
Let $N$ be an object in $\cCo{d}$. For an automorphism
$\alpha\colon N \xto{\cong} N$ of $N$ as an $\cO_X$-module,
we let $z(\alpha)$ denote the morphism $N \to N$ 
in $\cCo{d}$ represented by the diagram
$$
N \xleftarrow{\alpha} N \xto{=} N
$$
of $\cO_X$-modules.
It is easy to see that any endomorphism of $N$ in $\cCo{d}$
is an automorphism and the map 
$z \colon  \Aut_{\cO_X}(N) \to \Aut_{\cCo{d}}(N)$
that sends $\alpha$ to $z(\alpha)$ is an isomorphism 
of groups.
For an automorphism $\alpha\colon N \xto{\cong} N$ of $N$ as an $\cO_X$-module,
we denote, by abuse of notation, the automorphism
$z(\alpha)$ of $N$ in $\cCo{d}$ by the same symbol $\alpha$
if there is no risk of confusion.

\subsubsection{ }
We often consider the following two types
of morphisms in $\cCo{d}$.
Let $N$ be an object in $\cCo{d}$.
For an $\cO_X$-submodule $N'$ of $N$,
the morphism $N' = N' \inj N$ in $\cCo{d}$
is denoted by $r_{N,N'} \colon  N \to N'$.
For a quotient $\cO_X$-module $N''$ of $N$,
the morphism $N'' \twoheadleftarrow N =N$ in $\cCo{d}$
is denoted by $m_{N,N''} \colon  N \to N''$.

\subsection{The set $C(N,N')$}

\subsubsection{ }\label{sec:triple}
In Section \ref{sec:Cd_defn}, 
we defined the morphisms in $\cCo{d}$
as some equivalences classes of diagrams.
However, the definition is often inconvenient for practical use
since we are forced to make some efforts in checking if
two diagrams belong to the same equivalence class.
In this paragraph we will give an alternative
description of the morphisms in $\cCo{d}$.
Although the latter description is a little more ad-hoc,
it is more suitable for the actual computation.

Let $N$ and $N'$ be objects in $\cCo{d}$.
%
Let $C(N,N')$ denote the set of triples
$(N_1,N_2,\alpha)$ of 
$\cO_X$-submodules $N_1 \subset N_2 \subset N$
and an isomorphism $\alpha\colon N_2/N_1 \xto{\cong} N'$ of
$\cO_X$-modules.
For $(N_1,N_2,\alpha) \in C(N,N')$, we denote by
$\wt{\alpha} \colon  N_2 \to N'$ the
composite of $\alpha$ with the surjection $N_2 \surj N_2/N_1$.
%
To each triple $(N_1,N_2,\alpha) \in C(N,N')$ we
associate the morphism $N \to N'$ in $\cCo{d}$ represented
by the diagram
$$
N' \xleftarrow{\wt{\alpha}} N_2 \xto{\subset} N.
$$
This gives a map $C(N,N') \to \Hom_{\cCo{d}}(N,N')$.

\begin{lem}\label{lem:4_0}
The map $C(N,N') \to \Hom_{\cCo{d}}(N,N')$
is bijective.
\end{lem}
\begin{proof}
We construct the inverse map.
Let $f \colon  N \to N'$ be a morphism represented by a diagram
$$
N' \xleftarrow{q} N'' \xto{i} N.
$$
We set $N_1=i(\Ker\, q)$ and $N_2 = i(N'')$.
The homomorphism $i$ induces an isomorphism
$N''/\Ker\, q \cong N_2/N_1$. Let 
$\alpha\colon N_2/N_1 \xto{\cong} N'$ denote the composite of the
inverse of this isomorphism with the
isomorphism $N''/\Ker\, q \cong N'$ induced by $q$.
%
%
By sending $f$ to the triple $(N_1,N_2,\alpha)$
we obtain a map $\Hom_{\cCo{d}}(N,N') \to C(N,N')$.
%
It is then easy to check that this map is the desired
inverse.
%
\end{proof}

\begin{cor}\label{cor:finite}
Suppose that the residue fields of all closed points
of $X$ are finite fields.
Then for any two objects $N$, $N'$ in $\cCo{d}$,
the set $\Hom_{\cCo{d}}(N,N')$ is a finite set.
\end{cor}

\begin{proof}
Since the set $C(N,N')$ is a finite set,
the claim follows from Lemma \ref{lem:4_0}.
\end{proof}

We describe how the composition of morphisms in $\cCo{d}$
is expressed via the bijection in Lemma~\ref{lem:4_0}.

\begin{lem} \label{lem:4_0_2}
Let $N \xto{f} N' \xto{g} N''$ be a diagram in $\cCo{d}$.
Let $(N_1,N_2,\alpha) \in C(N,N')$ and
$(N'_1,N'_2,\beta) \in C(N',N'')$ be triples corresponding
to the morphisms $f$ and $g$, respectively.
Then the triple in $C(N,N'')$ corresponding to 
$g \circ f$ is equal to
$(\wt{\alpha}^{-1}(N'_1), \wt{\alpha}^{-1}(N'_2),\beta \circ
\overline{\alpha})$. Here 
$\overline{\alpha} \colon  \wt{\alpha}^{-1}(N'_2)/\wt{\alpha}^{-1}(N'_1)
\xto{\cong} N'_2/N'_1$ is the isomorphism induced by
$\wt{\alpha}$.
\end{lem}

\begin{proof}
The diagram
$$
\begin{array}{ccc}
N' & \xleftarrow{\wt{\alpha}} & N_2 \\
\uparrow & & \uparrow \\
N'_2 & \xleftarrow{\wt{\alpha}} & \wt{\alpha}^{-1}(N'_2),
\end{array}
$$
where the vertical maps are the inclusions,
is commutative and cartesian.
It then follows from the definition of 
the composite of the morphisms in $\cCo{d}$
that $g \circ f$ is represented by the diagram
$$
N'' \xleftarrow{\wt{\beta} \circ \wt{\alpha}}
\wt{\alpha}^{-1}(N'_2) \xto{\subset} N.
$$
Since the kernel of the homomorphism
$\wt{\beta} \circ \wt{\alpha} \colon  \wt{\alpha}^{-1}(N'_2)\to N''$
is equal to $\wt{\alpha}^{-1}(N'_1)$, the claim follows.
\end{proof}

Let $N$ be an object of $\cCo{d}$. Let $\Pair(N)$
denote the following poset. The elements of $\Pair(N)$
are the pairs $(N_1,N_2)$ of $\cO_X$-submodules of $N$
with $N_1 \subset N_2$. For two elements 
$(N_1,N_2)$ and $(N'_1,N'_2)$ of $\Pair(N)$, we have
$(N_1,N_2) \le (N'_1,N'_2)$ if and only if
$N'_1 \subset N_1 \subset N_2 \subset N'_2$.
By associating, with each element $(N_1,N_2)$ of
$\Pair(N)$, the equivalence class of the diagram
$$
N_2/N_1 \twoheadleftarrow N_2 \inj N,
$$
we obtain a contravariant functor from $\Pair(N)$
regarded as a category to the undercategory 
$\cCo{d}_{N/}$. Let us denote this functor by 
$\jmath_N$

\begin{lem} \label{lem:epi}
The functor $\jmath$ is an anti-equivalence of categories.
\end{lem}

\begin{proof}
It follows from Lemma \ref{lem:4_0} and the bijectivity of
the map $z$ in Section \ref{sec:notationz} that
the functor $\jmath$ is essentially surjective.

Let $(N_1,N_2)$ and $(N'_1,N'_2)$ be two
elements of $\Pair(N)$. Let $(M_1,M_2,\alpha)$ be
an element of $C(N_2/N_1,N'_2/N'_1)$ and let
$f\colon  N_2/N_1 \to N'_2/N'_1$ denote the corresponding 
morphism in $\cCo{d}$. Then it follows from Lemma \ref{lem:4_0_2}
that the equality $\jmath_N(N_1,N_2) = f \circ \jmath_N(N'_1,N'_2)$
holds if and only if $(N'_1,N'_2) \le (N_1,N_2)$,
$M_1 = N'_1/N_1$, $M_2 = N'_2/N_1$, and $\alpha$ is equal to
the inverse of the isomorphism $N'_2/N'_1 \xto{\cong} 
(N'_2/N_1)/(N'_1/N_1)$ induced by the quotient homomorphism
$N'_2 \to N'_2/N_1$.
This shows that the functor $\jmath$ is fully faithful. 
\end{proof}

\begin{ex}\label{ex:automorphisms}
We explain how to compute the group $\Aut_{N'}(N)$ for
a morphism $f \colon  N \to N'$ in $\cCo{d}$.
Let $(N_1,N_2,\alpha) \in C(N,N')$ be
the triple corresponding to $f$.
If $\beta\colon N\xto{\cong} N$ is an $\cO_X$-linear
automorphism, then it follows from Lemma \ref{lem:4_0_2}
that the composite of the automorphism
$N \cong N$ corresponding to $(N,0,\beta)$ 
with $f$ is given by the triple
$(\beta^{-1}(N_1),\beta^{-1}(N_2), 
\alpha\circ \overline{\beta})$,
where $\overline{\beta}\colon \beta^{-1}(N_2)/\beta^{-1}(N_1)
\xto{\cong} N_2/N_1$ is the isomorphism induced by $\beta$.
Hence the group $\Aut_{N'}(N)$ is isomorphic to
the subgroup of elements $\beta \in \Aut_{\cO_X}(N)$
such that $\beta(N_i) = N_i$ for $i=1, 2$ and 
that the automorphism of $N_2/N_1$ induced by $\beta$
is equal to the identity.
\end{ex}

\subsection{\Fibrs and \cofibrs}\label{sec:modelstr}

\subsubsection{ }
A morphism $f\colon N \to N'$ in $\cCo{d}$ is called a {\em \fibr} 
(\resp a {\em \cofibr})
if it is represented by a diagram
$$
N' \stackrel{p}{\twoheadleftarrow} N''
\stackrel{i}{\inj} N
$$
of $\cO_X$-modules, in which $i$ (\resp $p$) is an isomorphism.
In particular for a quotient $\cO_X$-module (\resp an $\cO_X$-submodule) 
$N'$ of an object $N \in \cCo{d}$, the morphism $m_{N,N'}$ (\resp $r_{N,N'}$)
is a \fibr (\resp a \cofibr).
We will use the following properties satisfied by \fibrs and \cofibrs.
\begin{lem}\label{lem:pre_model}
\begin{enumerate}
\item Let
$$
\begin{CD}
M @>{f}>> N \\
@V{r}VV @V{m}VV \\
M' @>{f'}>> N'
\end{CD}
$$
be a commutative diagram in $\cCo{d}$,
where $r$ is a \cofibr and
$m$ is a \fibr. Then there exists a morphism
$h\colon M' \to N$ in $\cCo{d}$ satisfying
$f = h\circ r$ and $f' = m \circ h$.
\item Any morphism $f$ in $\cCo{d}$ is of the form
$f = m \circ r$ where $r$ is a \cofibr and 
$m$ is a \fibr.
\item Any isomorphism in $\cCo{d}$ is a \fibr.
\item 
Let $M \xto{f} N \xto{g} N'$ be morphisms in $\cCo{d}$.
Then the composite $g \circ f$ is a \fibr if and only if
$f$ and $g$ are \fibrs.
\end{enumerate}
\end{lem}

\begin{proof}
Let the notation be as in the claim (1) and
let $(M_1,M_2,\alpha_1) \in C(M,N)$,
$(M_3,M_4,\alpha_2) \in C(M,M')$ and
$(M_5,M_6,\alpha_3) \in C(M,N')$ denote the
triple corresponding to the morphisms
$f$, $r$, and $m \circ f$.
Since $m \circ f$ factors through both $f$ and $r$, we have
$M_1 \subset M_5 \subset M_6 \subset M_2$ and
$M_3 \subset M_5 \subset M_6 \subset M_4$.
Since $r$ is a \cofibr, we have $M_3=0$.
Since $m$ is a \fibr, we have $M_6=M_2$.
This shows that $M_3 \subset M_1 \subset M_2 \subset M_4$.
For $i=1,2$, let $M'_i$ denote the image of $M_i$ under the
isomorphism $\alpha_2\colon M_4 = M_4/M_3 \xto{\cong} M'$.
We define the isomorphism $\beta\colon M'_2/M'_1 \to N$ to be
the composite of the inverse of the isomorphism
$M_2/M_1 \xto{\cong} M'_2/M'_1$ 
induced by $\alpha_2$ with the isomorphism $\alpha_1$.
Let $h\colon M' \to N$ be the morphism in $\cCo{d}$
corresponding to the triple $(M'_1,M'_2,\beta)$.
Then it is easy to check that $f = h \circ r$.
Since $f' \circ r = m \circ f = m \circ h \circ r$
and $r$ is an epimorphism in $\cCo{d}$, we have
$f' = m \circ h$. This proves the claim (1).

Let $f\colon N \to N'$ be a morphism in $\cCo{d}$
represented by a diagram $N' \stackrel{p}{\twoheadleftarrow}
N'' \stackrel{i}{\inj} N$. Then $f$ is equal to the
composite of the \cofibr represented by
$N'' \xleftarrow{=} N'' \stackrel{i}{\inj} N$
and the \fibr
$N'  \stackrel{p}{\twoheadleftarrow} N'' \xto{=} N''$.
This proves the claim (2).

The claim (3) is obvious.

Let the notation be as in the claim (4).
Let $(M_1,M_2,\alpha_1) \in C(M,N)$ and
$(M_3,M_4,\alpha_2) \in C(M,N')$ denote the triples
corresponding to the morphisms $f$ and $g \circ f$,
respectively. Since $g \circ f$ factors through $f$, we have
$M_1 \subset M_3 \subset M_4 \subset M_2$.
It can be checked that $f$ (\resp $g$, \resp $g\circ f$) 
is a \fibr if and only if $M_2=M$ (\resp $M_4=M$, \resp $M_4=M_2$).
Hence the claim (4) follows.

\end{proof}

\begin{rmk}
We adopt the terminologies ``fibration" and ``cofibration"
from the model category theory.
These terminologies are justified by the following observation.
Let us define the weak equivalences in $\cC^d$ to be
all the morphisms in $\cC^d$. Then
the three subcategories of $\cC^d$ that consist
of the weak equivalences, the fibrations, and the cofibrations
give a model structure of $\cC^d$ in the sense of 
Hovey~\cite[Def.~1.1.3]{Hovey}.
However, the category $\cC^d$ is not a model category,
since it is neither complete nor cocomplete.

While the model category theory considers the localization
with respect to the weak equivalences, we consider the
Grothendieck topology whose covering morphisms are given 
by the weak equivalences.
\end{rmk}

\begin{rmk}
The statements of Lemma \ref{lem:pre_model} remain valid if
we exchange the roles of fibrations and those of cofibrations.
We can check this by using the functor $\bD$ which will be
introduced later.
For example, in the notation of the
claim (4), $f$ and $g$ are \cofibrs if $g \circ f$ is a \cofibr.
We did not include this statement as a part of Lemma \ref{lem:pre_model}
since we will not use this in the sequel.
\end{rmk}

\subsection{Lattices}
\label{sec:lattices}
Let $\pi_0(X)$ denote the set of connected components of $X$.
The set $\pi_0(X)$ is a finite set and
each $X' \in \pi_0(X)$ is an integral scheme.
We let $\eta_{X'}$ denote the generic point of $X'$.
Set $\eta_{X} = \coprod_{X' \in \pi_0(X)} \eta_{X'}$ and let
$j_X \colon  \eta_X \to \coprod_{X' \in \pi_0(X)} X' =X$ 
denote the canonical morphism.
We set $\cK_X = {j_X}_* j_X^* \cO_X$ and call it the
sheaf of total ring of fractions on $X$.
We regard $\cO_X$ as an $\cO_X$-submodule of $\cK_X$
via the adjunction homomorphism $\cO_X \to \cK_X$.

\subsubsection{ }
We set 
$$
V = \cK_X^{\oplus d}.
$$
We say that an $\cO_X$-submodule $L$ of $V$ is
an $\cO_X$-lattice of $V$ if 
$L$ is coherent and $j_X^* V/L = 0$.
For example, the $\cO_X$-submodule $\cO_X^{\oplus d}$ 
of $\cK^{\oplus d}$ is an $\cO_X$-lattice of $V$.

\begin{lem}\label{lem:loc_freeness}
Let $L$ be an $\cO_X$-lattice of $V$.
Then $L$ is a locally free $\cO_X$-module of rank $d$.
\end{lem}

\begin{proof}
Since the question is local, we may assume that $X = \Spec\, R$
is a spectrum of discrete valuation ring. Let $K$ denote the
field of fractions of $R$. It suffices to show that any
finite generated $R$-submodule $N$ of $K^{\oplus d}$ 
such that $(K^{\oplus d}/N) \otimes_R K = 0$ is free of rank $d$.
The last claim follows from the well-known fact that 
any finitely generated torsion-free $R$-module is free.
\end{proof}

\begin{lem}\label{lem:lattices_V}
Let $L_1$, $L_2$ be $\cO_X$-lattices in $V$.
Then both $L_1 \cap L_2$ and $L_1+L_2$ are
$\cO_X$-lattices of $V$.
\end{lem}

\begin{proof}
The $\cO_X$-module $L_1 \cap L_2$ is coherent since
it is a quasi-coherent $\cO_X$-submodule of $L_1$
and $X$ is noetherian.
We have $j_X^* V/(L_1 \cap L_2) = 0$ since it is 
isomorphic to an
$\cO_{\eta_X}$-submodule
of $j_X^* (V/L_1 \oplus V/L_2) = 0$. 
This shows that $L_1 \cap L_2$ is an $\cO_X$-lattice of $V$.

The $\cO_X$-module $L_1 + L_2$ is coherent since it is
quasi-coherent and finitely generated, and $X$ is noetherian.
We have $j_X^* V/(L_1+L_2) = 0$ since it is a quotient
of $j_X^* V/L_1 = 0$. This shows that $L_1 + L_2$ is
an $\cO_X$-lattice of $V$.
\end{proof}

\begin{lem} \label{lem:lattice_V}
Let $L$ be an $\cO_X$-lattice of $V$.
Then for an $\cO_X$-submodule $L'$ of $V$,
the following conditions are equivalent:
\begin{enumerate}
\item Both $(L+L')/L$ and $(L+L')/L'$ are coherent $\cO_X$-modules 
of finite length.
\item Both $L/(L \cap L')$
and $L'/(L \cap L')$ are coherent $\cO_X$-modules 
of finite length.
\item $L'$ is an $\cO_X$-lattice of $V$.
\end{enumerate}
\end{lem}

\begin{proof}
The equivalence of the first two conditions is obvious since
$L/(L \cap L') \cong (L+L')/L'$
and
$L'/(L \cap L') \cong (L+L')/L$.

Suppose that $L$ satisfies Condition (2).
The $\cO_X$-module $L \cap L'$ is coherent since
both $L$ and $L/(L \cap L')$ are coherent.
Since $L'/(L \cap L')$ are coherent, it follows
that $L'$ is coherent.
Observe that $j_X^* V/L'$ is a quotient of
$j_X^* V/(L \cap L')$.
Since both $j_X^* V/L$ and
$j_X^* L/(L \cap L')$ vanish,
we have $j_X^* V/(L \cap L') = 0$.
Hence $j_X^* V/L' = 0$. This proves that
Condition (2) implies Condition (3).

Suppose that $L'$ satisfies Condition (3).
Since both $j_X^* V/L'$ and $j_X^* V/L$
vanish, we have 
\begin{equation} \label{jX}
j_X^* (L+L')/L =
j_X^* (L+L')/L' = 0.
\end{equation}
Note that the $\cO_X$-modules $(L+L')/L$ and $(L+L')/L'$ are coherent
since they are quasi-coherent and finitely generated,
and $X$ is noetherian.
Since $X$ is noetherian of Krull dimension one,
\eqref{jX} shows that both $(L+L')/L$ and
$(L+L')/L'$ are of finite length.
Hence Condition (3) implies Condition (1).
\end{proof}

\subsection{Enough Galois coverings}
In this paragraph, we prove that the set $\Mor(\cCo{d})$
of the morphisms in $\cCo{d}$ contains enough Galois coverings.

\begin{prop}\label{prop:cofinality0}
The set of the morphisms in $\cCo{d}$ contains enough Galois coverings.
Moreover, if $f\colon N \to N'$ is a \fibr in $\cCo{d}$,
then there exists a \fibr $h\colon M \to N$ in $\cCo{d}$
such that the composite $f \circ h$ is
a Galois covering in $\cCo{d}$.
\end{prop}

To prove Proposition \ref{prop:cofinality0} we
use the following lemma.

\begin{lem}\label{lem:Gal}
Let $f\colon N \to N'$ be a morphism in $\cCo{d}$
given by the diagram 
$N'\stackrel{p}{\twoheadleftarrow} N''
\stackrel{i}{\hookrightarrow} N$.
Suppose that there exist $\cO_X$-submodules
$N_1, N_2 \subset N'$ and effective Cartier divisors
$Z_1, Z_2 \subset X$ such that
$p^{-1}(N_1) \cong {i_{Z_1}}_* \cO_{Z_1}^{\oplus d}$ and
$N/i(p^{-1}(N_2)) \cong {i_{Z_2}}_* \cO_{Z_2}^{\oplus d}$.
Then $f$ is a Galois covering in $\cCo{d}$.
\end{lem}

\begin{proof}
Let $M$ be an object in $\cCo{d}$. 
It suffices to
show that the group $\Aut_{N'}(N)$ acts freely on
the set $\Hom_{\cCo{d}}(M,N)$ and that the map 
$\alpha_M\colon \Hom_{\cCo{d}}(M,N) \to \Hom_{\cCo{d}}(M,N')$
given by the composition with $f$ induces an injective
map from $\quot{\Hom_{\cCo{d}}(M,N)}{\Aut_{N'}(N)}$ 
to $\Hom_{\cCo{d}}(M,N')$.

Lemma \ref{lem:epi} shows that any morphism in $\cCo{d}$
is an epimorphism. It follows that the group $\Aut_{N'}(N)$ 
freely acts on the set $\Hom_{\cCo{d}}(M,N)$.
Hence it suffices to prove that,
for any $x \in \Hom_{\cCo{d}}(M,N')$, the group $\Aut_{N'}(N)$
acts transitively on the set $\alpha_{M}^{-1}(x)$ if 
$\alpha_{M}^{-1}(x)$ is non-empty.

Take an element $x\in \Hom_{\cCo{d}}(M,N')$ and let us
consider the set $\alpha_M^{-1}(x)$. Suppose 
$y\in \alpha_M^{-1}(x)$ is given by the diagram
$N \stackrel{s'}{\twoheadleftarrow} F
\stackrel{s}{\hookrightarrow} M$.
We let $F_1'=s^{\prime -1}(i(p^{-1}(N_1)))$,
$F_2'=s^{\prime -1}(i(p^{-1}(N_2)))$, and 
$F''=\Ker\, s'$.

Let $\sI_1, \sI_2 \subset \cO_X$ denote the sheaves of ideals 
defining the closed subschemes $Z_1$, $Z_2$, respectively.
Since $F'_1/F'' \cong {i_{Z_1}}_* \cO_{Z_1}^{\oplus d}$, 
$F/F'_2\cong {i_{Z_2}}_* \cO_{Z_2}^{\oplus d}$,
and $M$ is generated by $d$ elements, it follows that
$F'_1/F'' \cong {i_{Z_1}}_* i_{Z_1}^* F'_1$ 
and $F/F'_2$ is equal to the maximal $\cO_X$-submodule of
$M/F'_2$ annihilated by $\sI_2$.
Hence $F''=\sI_1 F'_1$ and $F$ is equal to the maximal $\cO_X$-submodule 
$M'$ of $M$ satisfying $\sI_2 M' \subset F'_2$.
Since $s(F'_1)$ and $s(F'_2)$ 
depend only on $x$, 
the $\cO_X$-submodules
$s(F)$ and $s(F'')$ of 
$M$ are uniquely determined by $x$ and are
independent of the choice of 
$y$. Note that $y$ is the composite of 
the canonical morphism $s(F)/s(F'') 
\twoheadleftarrow s(F) \hookrightarrow M$ 
and an isomorphism $s(F)/s(F'') \cong N$. 
Thus the set $\alpha_M^{-1}(x)$ is
canonically isomorphic to the subset of 
the set of isomorphisms $s(F)/s(F'')
\cong N$ such that the composite $M \to s(F)/s(F'')
\cong N\to N'$ equals the morphism $x$.
Hence the group $\Aut_{N'}(N)$ acts transitively on the set
$\alpha_M^{-1}(x)$. This proves the claim.
\end{proof}

\begin{proof}[Proof of Proposition \ref{prop:cofinality0}]
Let $f\colon N \to N'$ be a morphism in $\cCo{d}$.
Let $(N_1,N_2,\alpha) \in C(N,N')$ be the triple
corresponding to the morphism $f$.
Let us choose a surjection $\phi\colon \cO_X^{\oplus d} \surj N$
and let $L=\phi^{-1}(N_2)$.
Let us take two effective Cartier divisors $Z, W \subset X$
such that $N_2$ and $N/N_2$ are annihilated by
$\sI= \sI_Z$ and $\sJ=\sI_W$, respectively.
We note that $\sI$ and $\sJ$ are invertible $\cO_X$-modules.
%
%
Let us consider the product $\sI\sJ$ of ideals.
We set $M=L/\sI\sJ L$.
Since $N/N_2$ is annihilated by $\sJ$,
we have $\sJ^{\oplus d} \subset L$. We denote by 
$i'\colon \sJ^{\oplus d}/\sI \sJ L \inj M$
the injective homomorphism induced by the inclusion
$\sJ^{\oplus d} \subset L$.
Since $\sJ$ is an invertible $\cO_X$-module,
there exists an isomorphism $b\colon  \cO_X/\sI\sJ \xto{\cong} \sJ/\sI\sJ^2$ of 
$\cO_X$-modules. Let $i\colon  \cO_X^{\oplus d}/\sI L \inj M$ denote 
the composite of $i'$ with the isomorphism 
$\cO_X^{\oplus d}/\sI L \cong \sJ^{\oplus d}/\sI \sJ L$ induced by $b$.
Since $N_2$ is annihilated by $\sI$, the submodule 
$\sI L$ is contained in the kernel of $\phi$. Hence the
homomorphism $\phi$ induces a surjective homomorphism
$p\colon \cO_X^{\oplus d}/IL \surj N$.
%
Let $h\colon M \to N$ denote the morphism represented by the
diagram
$$
N \stackrel{p}{\leftarrow} \cO_X^{\oplus d}/\sI L
\stackrel{i}{\to} M.
$$
It then follows from Lemma \ref{lem:Gal} that the composite
$f \circ h$ is a Galois covering.
This proves the claim.
\end{proof}

\subsection{The category $\cC_0^d$}

\subsubsection{ }
\label{sec:Pair}
We let $\Pair^d$ denote the following poset.
The elements of $\Pair^d$ are the pairs 
$(L_1,L_2)$ of $\cO_X$-lattices of $V$ with $L_1 \subset L_2$.
For two elements $(L_1,L_2)$ and $(L'_1,L'_2)$
in $\Pair^d$, we have $(L_1,L_2) \le (L'_1,L'_2)$
if and only if $L'_1 \subset L_1 \subset L_2 \subset L'_2$.

%

We let $\cC_0^d$ denote the opposite category of the poset
$\Pair^d$ regarded as a small category. By definition,
the set of objects in $\cC_0^d$ is that of $\Pair^d$
and for two elements $(L_1,L_2)$ and $(L'_1,L'_2)$, the set
$\Hom_{\cC_0^d}((L_1,L_2),(L'_1,L'_2))$ consists
of a single element if $(L_1,L_2) \ge (L'_1,L'_2)$,
and is the empty set otherwise. Here the symbol $\ge$
denotes the ordering of the poset $\Pair^d$.
To avoid confusion, we never use the symbols $\ge$ and 
$\le$ to denote the ordering of the dual poset $\cC_0^d$.

\subsubsection{ }\label{sec:ratio_V}
Let $\iota_0^d \colon \cC_0^d \to \cCo{d}$ denote the following functor.
For $(L_1,L_2) \in \Pair^d$, we set $\iota_0^d((L_1,L_2))=L_2/L_1$,
which, by Lemma \ref{lem:lattice_V}, Lemma \ref{lem:loc_freeness}
and Remark \ref{rmk1}, is an object in $\cCo{d}$.
For two elements $(L_1,L_2), (L'_1,L'_2) \in \Pair^d$ 
with $(L_1,L_2) \le (L'_1,L'_2)$, the functor $\iota_0^d$ sends
the unique morphism $f\colon (L'_1,L'_2) \to (L_1,L_2)$
in $\cC_0^d$ to the morphism in $\cCo{d}$ represented by
the diagram
$$
L_2/L_1 \twoheadleftarrow L_2/L'_1 \inj L'_2/L'_1.
$$

We introduce the following full subcategories $\cC_{0,m}^d$ and
$\cC_{0,r}^d$ of $\cC_0^d$: The objects of $\cC_{0,m}^d$ 
(\resp $\cC_{0,r}^d$) are the objects $(L_1,L_2)$ of $\cC_0^d$
satisfying $L_1 \subset L_2 \subset \cO_X^{\oplus d}$ 
(\resp $\cO_X^{\oplus d} \subset L_1 \subset L_2$).
Let $\iota_{0,m}^d$ and $\iota_{0,r}^d$ denote the restriction
of $\iota_0^d$ to $\cC_{0,m}^d$ and $\cC_{0,r}^d$, respectively.

\begin{lem} \label{lem:pair_directed_V}
The categories $\cC_0^d$, $\cC_{0,m}^d$ and
$\cC_{0,r}^d$ are $\Lambda$-connected.
\end{lem}

\begin{proof}
Let $\cC_0$ denote one of the three categories 
$\cC_0^d$, $\cC_{0,m}^d$ and $\cC_{0,r}^d$.
Let $(L_1,L_2)$ and $(L'_1,L'_2)$ be objects of $\cC_0$.
We set $L''_1 = L_1 \cap L'_1$ and
$L''_2 = L_2 + L'_2$.
By Lemma \ref{lem:lattices_V}, we have
$(L''_1,L''_2) \in \Pair^d$.
One can check easily that $(L''_1,L''_2)$ is an object
of $\cC_0$.
Since $L''_1 \subset L_1 \subset L_2 \subset L''_2$
and $L''_1 \subset L'_1 \subset L'_2 \subset L''_2$,
we have a diagram 
$(L_1,L_2) \leftarrow (L''_1,L''_2) \to (L'_1,L'_2)$
in $\cC_0$.
This proves the claim.
\end{proof}

\subsection{The first main theorem}

Let $\cT^d$ (\resp $\cT^d_m$, \resp $\cT^d_r$) denote the
set of morphisms (\resp fibrations, \resp cofibrations)
in $\cC^d$. The main result of this section is the following:
\begin{thm} \label{thm:section2}
The sets $\cT^d$, $\cT^d_m$, and $\cT^d_r$ are semi-localizing 
collections of morphisms in $\cCo{d}$.
Let $J^d$, $J_m^d$,and  $J_r^d$ denote the Grothendieck topologies on
$\cCo{d}$ associated with $\cT^d$, $\cT^d_m$, and $\cT^d_r$, respectively,
in the sense of \cite[Lemma 2.3.4]{Grids}.
Then $(\cCo{d},J^d)$, $(\cCo{d},J_m^d)$, and $(\cCo{d},J_r^d)$
are $Y$-sites in the sense of \cite[Def.\ 5.4.2]{Grids}.
Moreover, the pairs $(\cC_0^d,\iota_0^d)$,
$(\cC_{0,m}^d,\iota_{0,m}^d)$, and $(\cC_{0,r}^d,\iota_{0,r}^d)$
are grids of $(\cCo{d},J^d)$, $(\cCo{d},J_m^d)$, and $(\cCo{d},J_r^d)$,
respectively in the sense of \cite[Definition 5.5.3]{Grids}.
\end{thm}

\subsection{Proof of Theorem \ref{thm:section2} for $\cT^d$}

First, we prove Theorem \ref{thm:section2} for $\cT^d$.
We use the following:

%
%
%
%
\begin{prop} \label{prop:grid_atomic}
Let $\cC$ be an essentially small category.
%
Let $\Cip$ be a pair of a small category $\cC_0$ 
and a functor $\iota_0 \colon  \cC_0 \to \cC$. 
Suppose that the following conditions are satisfied.
\begin{enumerate}
\item The category $\cC_0$ is a poset and is $\Lambda$-connected,
\item The set of the morphisms in $\cC$ contains 
enough Galois coverings,
\item The functor $\iota_0$ is essentially surjective,
\item For any object $X$ of $\cC_0$,
$\iota_{0,/X} \colon  \cC_{0,/X} \to \cC_{/\iota_0(X)}$ is essentially
surjective,
\item For any object $X$ of $\cC_0$, the functor
$\iota_{0,X/} \colon  \cC_{0,X/} \to \cC_{\iota_0(X)/}$ is an
equivalence of categories.
\end{enumerate}
Then the category $\cC$ is semi-cofiltered, the pair
$(\cC,J)$, where $J$ denotes the atomic topology on $\cC$,
is a $Y$-site, and the pair $\Cip$ is a grid of $(\cC,J)$.
\end{prop}

\begin{rmk}
The conditions (1), (3), (4), and (5) above are identical to the
four conditions, given in \cite[Definition 5.5.1]{Grids}, 
for the pair $\Cip$ to be a pregrid
when $\cC$ is a $\Lambda$-connected category.
The reason why we restate these four conditions instead of imposing that
$\Cip$ is a pregrid is that we would like to apply 
Proposition \ref{prop:grid_atomic} to the situation where 
we do not a priori know if $\cC$ is $\Lambda$-connected,
and where $\cC$ turns out to be $\Lambda$-connected
as a consequence of Proposition \ref{prop:grid_atomic}.
\end{rmk}

\begin{proof}
First, we prove that $\cC$ is semi-cofiltered.
Let $Y_1 \xto{f_1} X \xleftarrow{f_2} Y_2$ be a diagram
in $\cC$. By Condition (3), one can take an object $X'$ 
of $\cC_0$ and an isomorphism $\iota_0(X') \cong X$ in $\cC$.
By Condition (4), one can choose a diagram
$Y'_1 \xto{f'_1} X' \xleftarrow{f'_2} Y'_2$ and isomorphisms
$\beta_1 \colon  \iota_0(Y'_1) \cong Y_1$,
$\beta_2 \colon  \iota_0(Y'_2) \cong Y_2$ such that the diagram
$$
\begin{CD}
\iota_0(Y'_1) @>{\iota_0(f'_1)}>> 
\iota_0(X') @<{\iota_0(f'_2)}<< \iota_0(Y'_2) \\
@V{\beta_1}VV @V{\alpha}VV @VV{\beta_2}V \\
Y_1 @>{f_1}>> X @<{f_2}<< Y_2
\end{CD}
$$
is commutative. Since the category $\cC_0$ is $\Lambda$-connected,
there exists a diagram $Y'_1 \xleftarrow{g_1} Z' \xto{g_2} Y'_2$
in $\cC_0$. Since $\cC_0$ is a poset, we have $f'_1 \circ g_1
= f'_2 \circ g_2$. Set $Z = \iota_0(Z')$. We have a commutative diagram
$$
\begin{CD}
Z @>{\beta_2 \circ \iota_0(g_2)}>> Y_2 \\
@V{\beta_1 \circ \iota_0(g_1)}VV @VV{f_2}V \\
Y_1 @>{f_1}>> X
\end{CD}
$$
in $\cC$. This shows that the category $\cC$ is semi-cofiltered.

Let $J$ denote the atomic topology on $\cC$.
Since $\cC_0$ is a poset, any morphism in $\cC_0$ is an epimorphism.
Hence Condition (5) implies that any morphism in $\cC$ is an epimorphism.
This shows that $(\cC,J)$ is a $B$-site in the sense of \cite[Def.\ 4.2.1]{Grids}.
Since $\cC_0$ is $\Lambda$-connected and $\iota_0$ is essentially
surjective, the category $\cC$ is also $\Lambda$-connected.
This shows that $(\cC,J)$ is a $Y$-site.
The claim that $\Cip$ is a grid of $(\cC,J)$ follows
immediately from Conditions (1)-(5).
\end{proof}

\begin{rmk}
Let $\Cip$ be a grid of a $Y$-site $(\cC,J)$, and let
$\iota_0^* \cT(J)$ denote the set of morphisms in $\cC_0$ of type $J$
(see \cite[Def.\ 5.5.2]{Grids}).
Then $\iota_0^*\cT(J)$ is semi-localizing in the sense of \cite[Def.\ 2.3.1]{Grids}.
Let $\iota_0^* J$ denote the Grothendieck topology on $\cC_0$
associated with $\iota_0^* \cT(J)$ in the sense of \cite[Lemma 2.3.4]{Grids}.
Then $(\cC_0, \iota_0^* J)$ is a $B$-site, and satisfies
Conditions (1) and (3) of \cite[Definition 5.4.2]{Grids}.
However $(\cC_0, \iota_0^* J)$ is not necessarily a
$Y$-site, since $\cC_0(\iota_0^* J)$ is not necessarily 
$\Lambda$-connected.
\end{rmk}

\begin{proof}[Proof of Theorem \ref{thm:section2} for $\cT^d$]
It suffices to show that the category $\cC = \cC^d$ and the
pair $\Cip = (\cC_0^d,\iota_0^d)$ satisfy 
Conditions (1)-(5) in Proposition \ref{prop:grid_atomic}.

Condition (1) follows from Lemma \ref{lem:pair_directed_V}.
Condition (2) follows from Proposition \ref{prop:cofinality0}.

Let $N$ be an arbitrary object of $\cC$. Let us take
a surjection $p\colon  \cO_X^{\oplus d} \surj N$ and let $L$
denote the kernel of $p$. Then the pair
$(L,\cO_{X}^{\oplus d})$ is an object of $\cC_0$
and $p$ induces an isomorphism 
$\cO_X^{\oplus d}/L \cong N$. This proves that
Condition (3) is satisfied.

Condition (4) follows from Proposition \ref{prop:sec2_below} below.

Let us prove that Condition (5) is satisfied.
Let $X$ be an object of $\cC_0$ and let $N = \iota_0(X)$.
Then the undercategory $\cC_{0,X/}$ is isomorphic to the
dual of the poset $\Pair(N)$. Hence it follows from
Lemma \ref{lem:epi} that the functor
$\iota_{0,X/} \colon  \cC_{0,X/} \to \cC_{\iota_0(X)/}$ is an
equivalence of categories. This proves that Condition (5)
is satisfied.
\end{proof}

The following statement is the technical heart of this section.
\begin{prop} \label{prop:sec2_below}
Let $(L_1,L_2)$ be an object of $\cC_0^d$, and let
$N \to L_2/L_1$ be a morphism in $\cCo{d}$ represented by
a diagram
$$
L_2/L_1 \stackrel{p}{\twoheadleftarrow} M 
\stackrel{j}{\inj} N.
$$
Then there exists an object $(L'_1,L'_2)$ of $\cC_0^d$
with $(L'_1,L'_2) \ge (L_1,L_2)$ and isomorphisms
$f\colon  L_2/L'_1 \xto{\cong} M$ and
$g\colon  L'_2/L'_1 \xto{\cong} N$ of $\cO_X$-modules such that 
the diagram
$$
\begin{CD}
L_2/L_1 @<<< L_2/L'_1 @>{\subset}>> L'_2/L'_1 \\
@| @V{f}VV @VV{g}V \\
L_2/L_1 @<{p}<< M @>{j}>> N
\end{CD}
$$
of $\cO_X$-modules is commutative.
\end{prop}

\begin{proof}
Let us take an effective Cartier divisor
$Z \subset X$ such that $M$ is annihilated by the sheaf $\sI_Z$ of
ideals defining $Z$.
Let $\alpha$ denote the quotient homomorphism $\alpha \colon  L_2 \to L_2/L_1$.
Since $Z$ is artinian, one can take projective covers
$\beta_1 \colon  P_1 \surj i_Z^* (L_2/L_1)$ and $\beta_2 \colon  P_2 \surj i_Z^* M$
in the category of quasi-coherent $\cO_Z$-modules.
Then there exist homomorphisms
$\gamma_1 \colon  i_Z^* L_2 \to P_1$
and $\gamma_2\colon  P_2 \to P_1$ satisfying 
$i_Z^* (\alpha) = \beta_1 \circ \gamma_1$
and $i_Z^*(p) \circ \beta_2 = \beta_1 \circ \gamma_2$.
Since projective covers are essential surjections,
the homomorphisms $\gamma_1$
and $\gamma_2$ are surjective.
Hence by the projectivity of $P_1$ and $P_2$, 
one can choose a right inverse $s_i$ of $\gamma_i$ for $i \in \{1,2\}$.
This in particular implies that $\Ker\, \gamma_i$ is locally free.
Since $M$ admits a surjection from $\cO_X^{\oplus d}$,
one can show, by comparing the local ranks of
$\Ker\, \gamma_1$ and $\Ker\, \gamma_2$, that
there exists a surjective homomorphism 
$f_1\colon  \Ker\, \gamma_1 \surj \Ker\, \gamma_2$.
We have a decomposition of $i_Z^* L_2$ into the direct sum
$i_Z^* L_2 = s_1(P_1) \oplus \Ker\, \gamma_1$.
We define 
$f_2 \colon  i_Z^* L_2 \to P_2$.
to be the following homomorphism: The restriction
of $f_2$ to $s_1(P_1)$ is equal to the composite
$$
s_1(P_1) \cong P_1 \xto{s_2} P_2
$$
and the restriction of $f_2$ to $\Ker\, \gamma_1$ is
equal to the composite
$$
\Ker\, \gamma_1 \xto{f_1} \Ker\, \gamma_2 \inj P_2.
$$
It is easy to check that $f_2$ is a surjective homomorphism.
We let $f'$ denote the composite
$$
L_2 \to {i_Z}_*i_Z^* L_2 
\xto{{i_Z}_* (f_2)} {i_Z}_* P_2
\xto{{i_Z}_* (\beta_2)} {i_Z}_* i_Z^* M \cong M.
$$
Then the homomorphism $f'$ is surjective.

Let $L'_1$ denote the kernel of $f'$ and $f$ the
isomorphism $L_2/L'_1 \xto{\cong} M$ induced by $f'$.
It follows from Lemma \ref{lem:lattice_V} 
that $L'_1$ is an $\cO_X$-lattice in $V$.
Then $f'$ induces an injective homomorphism
$\alpha' \colon  M \inj V/L'_1$.
%
%
Let us take injective hulls $\beta'_1\colon  M \inj J_1$ and $\beta'_2\colon N \inj J_2$
of $M$ and $N$, respectively, in the category of 
$\cO_X$-modules. 
Then there exist homomorphisms
$\gamma'_1 \colon  J_1 \to V/L'_1$
and $\gamma'_2\colon  J_1 \to J_2$ satisfying $\alpha' = \gamma'_1 \circ \beta'_1$
and $\beta'_2 \circ j = \gamma'_2 \circ \beta'_1$.
Since injective hulls are essential injections,
the homomorphisms $\gamma'_1$ and $\gamma'_2$ are injective.
Hence by the injectivity of $J_1$ and $J_2$, 
one can choose a left inverse $t_i$ of $\gamma'_i$ for $i \in \{1,2\}$.
The two $\cO_X$-modules $\Ker\, t_1$ and
$\Ker\, t_2$ are torsion divisible $\cO_X$-modules
which are co-finitely generated.
One can show, by comparing the local coranks of
$\Ker\, t_1$ and $\Ker\, t_2$, that
there exists an injective homomorphism
$g_1\colon  \Ker\, t_2 \surj \Ker\, t_1$.
We have a decomposition of $J_2$ into the direct sum
$J_2 = \gamma'_2(J_1) \oplus \Ker\, t_2$.
We define $g_2 \colon  J_2 \to V/L'_1$
to be the following homomorphism: The restriction of 
$g_2$ to $\gamma'_2(J_1)$ is equal to the composite
$$
\gamma'_2(J_1) \cong J_1 \xto{\gamma'_1} V/L'_1
$$
and the restriction of $g_2$ to $\Ker\, t_2$ is
equal to the composite
$$
\Ker\, t_2 \xto{g_1}
\Ker\, t_1 \inj V/L'_1.
$$
Then the homomorphism $g_2$ is injective.
We let $g'\colon  N \to V/L'_1$ denote the composite 
$N \xto{\beta'_2} J_2 \xto{g_2} V/L'_1$.
Let $L'_2$ denote the inverse image under the homomorphism
$V \surj V/L'_1$ of the image of $g'$.
We let $g$ denote the inverse of  isomorphism
$N \xto{\cong} L'_2/L_1$ induced by $g'$.
By construction, we have the following 
commutative diagram
$$
\begin{CD}
L_2/L_1 @<<< L_2/L'_1 @>{\subset}>> L'_2/L'_1 \\
@| @V{f}V{\cong}V @V{\cong}V{g}V \\
L_2/L_1 @<{p}<< M @>{j}>> N
\end{CD}
$$
of $\cO_X$-modules, where the vertical arrows are
isomorphisms. This proves the claim.
\end{proof}

\subsection{Proof of Theorem \ref{thm:section2} for $\cT_m^d$}

\begin{lem}\label{lem:Galois_fibration}
Let $f\colon N \to N'$ be a Galois covering in $\cCo{d}$.
Suppose that $f$ is the composite $f=m\circ r$ of a
\cofibr $r$ and a \fibr $m$. Then $m$ is a Galois
covering in $\cCo{d}$
\end{lem}

\begin{proof}
We apply Lemma 4.2.6 of \cite{Grids} to the $B$-site 
$(\cC^d,J^d)$.
%
It suffices to show that any automorphism in $\Aut_{N'}(N)$
descends to an automorphism in $\Aut_{N'}(M)$, where $M$ denotes the
domain of the morphism $m$. Let $\wt{g} \in \Aut_{N'}(N)$ and
consider the commutative diagram
$$
\begin{CD}
N @>{r \circ \wt{g}}>> M \\
@V{r}VV @V{m}VV \\
M @>{m}>> N'
\end{CD}
$$
in $\cCo{d}$. It follows from Lemma \ref{lem:pre_model} (1) that
there exists a morphism $g\colon  M \to M$ over $N'$ such that
$r \circ \wt{g} = g\circ r$. Since any endomorphism of $M$
is an automorphism, we have $g \in \Aut_{N'}(M)$.
This shows that $\wt{g}$ descends to the automorphism $g$.
This proves the claim.
\end{proof}

\begin{lem}
$\cT_m^d$ is a semi-localizing collection of morphisms in $\cCo{d}$. 
\end{lem}

\begin{proof}
It follows from (3) and (4) of
Lemma \ref{lem:pre_model} that $\cT_m^d$ contains all the
identity morphisms in $\cCo{d}$ and is closed under the
composition. Let $Y_1 \xto{f_1} X \xleftarrow{f_2} Y_2$
be a diagram in $\cCo{d}$ such that the morphism $f_2$ belongs to
$\cT_m^d$. Since we have shown that $\cCo{d}$ is semi-cofiltered,
there exists a diagram $Y_1 \xleftarrow{g_1} Z \xrightarrow{g_2} Y_2$ 
in $\cCo{d}$ satisfying $f_1 \circ g_1 = f_2 \circ g_2$.
It follows from (2) of Lemma \ref{lem:pre_model} that 
one can write $g_1$ as the composite $g_1 = m \circ r$
of a fibration $m\colon W \to Y_1$ with a cofibration $r\colon Z \to W$.
By applying (1) of Lemma \ref{lem:pre_model} to the diagram
$$
\begin{CD}
Z @>{g_2}>> Y_2 \\
@V{r}VV @VV{f_2}V \\
W @>{f_1 \circ m}>> X,
\end{CD}
$$
we obtain a morphism $h\colon  W \to Y_2$ satisfying
$f_1 \circ m = f_2 \circ h$, i.e, the diagram
$$
\begin{CD}
W @>{h}>> Y_2 \\
@V{m}VV @VV{f_2}V \\
Y_1 @>{f_1}>> X,
\end{CD}
$$
is commutative. This proves that $\cT_m^d$ is a semi-localizing 
collection of morphisms in $\cCo{d}$.
\end{proof}

Let $J_m^d$ denote the $A$-topology on $\cC^d$ associated with $\cT_m^d$.
Recall that we have shown that any morphism in $\cCo{d}$ 
is an epimorphism. Hence it follows from (4) of Lemma \ref{lem:pre_model} 
that the site $(\cCo{d},J_m^d)$ is a $B$-site and we have
$\cT_m^d = \cT(J_m^d)$ where in the right-hand side we used
the notation in \cite[Def.\ 2.4.2]{Grids}.

\begin{defn}
Let $\cC_0$, $\cC$ be categories and 
let $\iota_0\colon  \cC_0 \to \cC$ be a functor.
Let $\cT$ be a set of morphisms in $\cC$.

A morphism $f$ in $\cC_0$ is called of type $\cT$
if $\iota_0(f)$ belongs to $\cT$.
An object $X$ of $\cC_0$ is called an edge object of $\cC_0$
with respect to $\cT$ if any morphism $f\colon Y \to X$ in $\cC_0$ 
is of type $\cT$.
\end{defn}

\begin{lem} \label{lem:explicit_edge}
Let $(L_1,L_2)$ be an object of $\cC_{0,m}^d$. Then $(L_1,L_2)$
is an edge object of $\cC_{0,m}^d$ with respect to $\cT_m^d$
if and only if $L_2 = \cO_X^{\oplus d}$.
\end{lem}

\begin{proof}
Suppose that $L_2 = \cO_X^{\oplus d}$. Then any object
$(L'_1,L'_2)$ of $\cC_{0,m}^d$ with $(L_1,L_2) \le (L'_1,L'_2)$
satisfies $L'_2 = \cO_X^{\oplus d}$. This shows that
$(L_1,L_2)$ be an edge object of $\cC_{0,m}^d$.

Suppose that $L_2 \neq \cO_X^{\oplus d}$. Then the
unique morphism from $(L_1,\cO_X^{\oplus d})$ to $(L_1,L_2)$
in $\cC_{0,m}^d$ is not of type $\cT^d_m$. Hence
$(L_1,L_2)$ is not an edge object of $\cC_{0,m}^d$.
This completes the proof.
\end{proof}

\begin{cor} \label{cor:0me}
The full subcategory $\cC_{0,m,e}^d \subset \cC_{0,m}^d$ of 
edge objects with respect to $\cT_m^d$ is $\Lambda$-connected.
The restriction of the functor $\iota_{0,m}^d$ to 
$\cC_{0,m,e}^d$ is essentially surjective.
\end{cor}

\begin{proof}
For any two objects $(L_1,\cO_X^{\oplus d})$ and
$(L'_1,\cO_X^{\oplus d})$ of $\cC_{0,m,e}^d$, we have
a diagram $(L_1,\cO_X^{\oplus d}) \leftarrow 
(L_1 \cap L'_1,\cO_X^{\oplus d}) \to (L'_1,\cO_X^{\oplus d})$
in $\cC_{0,m,e}^d$. Hence $\cC_{0,m,e}^d$ is $\Lambda$-connected.
The second claim follows since any object $N$ of $\cCo{d}$
admits a surjective homomorphism from $\cO_X^{\oplus d}$.
\end{proof}

\begin{lem}
$(\cCo{d},J_m^d)$ is a $Y$-site.
\end{lem}

\begin{proof}
First, we prove that $\cT^d_m$ contains enough Galois coverings.
Let $f\colon Y\to X$ be a morphism which belongs to $\cT^d_m$.
Since $\cT^d$ contains enough Galois coverings, there exists
a morphism $g\colon  Z \to Y$ in $\cCo{d}$ such that the composite
$f \circ g$ is a Galois covering in $\cCo{d}$.
Let us write, by using (2) of Lemma \ref{lem:pre_model}, 
the morphism $g$ as a composite $g = m \circ r$ of a
fibration $m\colon  W \to Y$ with a cofibration $r\colon Z \to W$.
Then it follows from Lemma \ref{lem:Galois_fibration} that
$f \circ m$ is a fibration which is a Galois covering.
This proves that  $\cT^d_m$ contains enough Galois coverings.

To complete the proof,
it suffices to prove that the category $\cCo{d}(\cT^d_m)$ is 
$\Lambda$-connected. Let $X$ and $Y$ be two objects of 
$\cCo{d}(\cT^d_m)$. By Corollary \ref{cor:0me} one can
take edge objects $X'$ and $Y'$ with respect to $\cT_m^d$ and
isomorphisms $\iota_{0,m}(X') \cong X$ and
$\iota_{0,m}(Y') \cong Y$ in $\cCo{d}$.
It follows again from Corollary \ref{cor:0me} that
there exists a diagram $X' \leftarrow Z' \to Y'$.
We set $Z = \iota_{0,d}(Z')$. We thus obtain a diagram 
$X \leftarrow Z \to Y$ in $\cCo{d}$.
This proves that the category $\cCo{d}(\cT^d_m)$ is 
$\Lambda$-connected and the proof is complete.
\end{proof}

\begin{proof}[Proof of Theorem \ref{thm:section2} for $\cT_m^d$]
It suffices to prove that the $Y$-site $(\cC,J) = (\cCo{d},J^d_m)$
and the pair $\Cip = (\cC_{0,m}^d,\iota_{0,m})$
satisfy the four conditions in Definition 5.5.3 of \cite{Grids}.

Condition (1) follows from Lemma \ref{lem:pair_directed_V}, and
Condition (2) follows from Corollary \ref{cor:0me}.

We check Condition (3).
Let $X=(L_1,L_2)$ be an object of $\cC_{0,m}^d$ and set 
$X' = \iota_{0,m}^d(X)$. Let $f'\colon  Y' \to X'$ be a morphism in
$\cCo{d}$ which belongs to $\cT_m^d$.
It follows from Proposition \ref{prop:sec2_below} that
there exist an object $Y=(L'_1,L'_2)$ of $\cC_0^d$,
a morphism $f\colon  Y \to X$ in $\cC_0^d$ and an isomorphism
$\iota_0^d(Y) \cong Y'$ such that the diagram
$$
\begin{CD}
\iota_0^d(Y) @>{\iota_0^d(f)}>> \iota_{0,m}^d(X) \\
@V{\cong}VV @| \\
Y' @>{f'}>> X'
\end{CD}
$$
is commutative. Since $f'$ belongs to $\cT_m^d$,
it follows that $L'_2 = L_2$. This implies that
$Y$ is an object of $\cC_{0,m}^d$.
Hence $f$ is a morphism in $\cC_{0,m}^d$ and 
Condition (3) is satisfied.

Let $X$ be an object of $\cC_{0,m}^d$.
One can check easily that the functor
$\cC_{0,m,X/}^d \to \cC_{0,X/}^d$ induced by the inclusion
$\cC_{0,m}^d \inj \cC_{0}^d$ is an equivalence of categories.
We set $X' = \iota_{0,m}^d(X)$.
It follows from (5) of Lemma \ref{lem:pre_model} that
$\cC^d_{m, X'/} \to \cC^d_{X'/}$ induced by the inclusion
$\cC^d_m \inj \cCo{d}$ is an equivalence of categories.
Hence Condition (4) follows from the corresponding condition for
$\cT^d$.
\end{proof}

\subsection{The dual functors $\bD$ and $\bD_0$}\label{sec:dual}

We introduce two dual functors $\bD$ and $\bD_0$.
The first functor $\bD$ is an anti-auto-equivalence of
the category $\cCo{d}$, which is induced by
a variant of the theory of duality over artinian local rings.
The second functor $\bD_0$ is an anti-auto-equivalence of
the category $\cC_0^d$, which is given by taking 
the $\cO_X$-linear duals of locally free $\cO_X$-modules.
%
%
\subsubsection{ }
We set $D = \cK_X/\cO_X$.
For an $\cO_X$-module $N$ of finite length we set
$\bD(N) = \cHom_{\cO_X}(N,D)$.
For a homomorphism $f\colon N \to N'$ of $\cO_X$-modules of finite
length, we let $\bD(f)\colon  \bD(N') \to \bD(N)$ denote
the homomorphism of $\cO_X$-modules given by the
composite with $f$.
It follows from the theory of duality over artinian local rings
that $\bD(N)$ is an $\cO_X$-module of finite length and the
natural homomorphism $N \to \bD(\bD(N))$ is an isomorphism.
Moreover, if the homomorphism $f$ is injective
(\resp surjective), then the homomorphism $\bD(f)$ is
surjective (\resp injective).
We can also check that for each $\cO_X$-module $N$ of finite length,
the $\cO_X$-module $\bD(N)$ is (non-canonically) isomorphic to $N$.
In particular if $N$ is an object in $\cCo{d}$, then
$\bD(N)$ is also an object in $\cCo{d}$.

\subsubsection{ }
Let $N$ and $N'$ be two objects of $\cCo{d}$ and let
$f$ be a morphism from $N$ to $N'$ in $\cCo{d}$.
Let us choose a diagram
\begin{equation} \label{choice1}
N' \xleftarrow{p} N'' \xto{i} N
\end{equation}
of $\cO_X$-modules which represents $f$. By applying
the functor $\bD$ to this diagram, we obtain
$$
\bD(N') \xto{\bD(p)} \bD(N'') \xleftarrow{\bD(i)} \bD(N).
$$
We note that $\bD(p)$ is injective and $\bD(i)$ is
surjective. Let $M \subset \bD(N)$ denote the inverse image under
$\bD(i)$ of the image of $\bD(p)$. Then the homomorphism
$\bD(i)$ induces the surjective homomorphism
$q\colon  M \to \bD(N')$ of $\cO_X$-modules. 
Let $\bD(f)$ denote the equivalence class of the diagram
$$
\bD(N') \xleftarrow{q} M \xto{\subset} \bD(N).
$$
One then can check that $\bD(f)$ is a morphism from $\bD(N)$
to $\bD(N')$ in $\cC^d$ and is independent of the
choice of a diagram \eqref{choice1} which represents $f$.
By associating $\bD(f)$ for each morphism $f$ in $\cCo{d}$
we obtain a covariant functor $\bD \colon \cCo{d} \to \cCo{d}$.
Moreover, the isomorphism $N \to \bD(\bD(N))$ 
for each object of $\cCo{d}$ induces an isomorphism 
$\id_{\cCo{d}} \to \bD \circ \bD$ of functors.

\begin{prop} \label{prop:bD}
Let $f$ be a morphism in $\cCo{d}$.
Then $f$ is a cofibration (\resp a fibration) 
if and only if $\bD(f)$ is a fibration (\resp a cofibration).
\end{prop}

\begin{proof}
It is immediate from the definition of the functor $\bD$.
\end{proof}

\subsubsection{ } \label{sec:bD_0}
For a locally free $\cO_X$-module $L$ of finite rank,
we set $L^\vee = \cHom_{\cO_X}(L,\cO_X)$.
If $L = \cO_X^{\oplus d}$, then $L^\vee$ is a free
$\cO_X$-module of rank $d$.
Let us fix an isomorphism $\alpha \colon  
(\cO_X^{\oplus d})^\vee \xto{\cong} \cO_X^{\oplus d}$
of $\cO_X$-modules.

Let $L$ be an $\cO_X$-lattice of $V$.
We regard $L^\vee$ as an $\cO_X$-lattice of $V$ in a way
depending on $\alpha$ as follows.
We set 
$$
V^* = {j_X}_* \cHom_{\cO_{\eta_X}}(j_X^* V,\cO_{\eta_X}).
$$
For any $\cO_X$-lattice $L'$, the inclusion
$L' \inj V$ induces a morphism
$j_X^* L' \cong j_X^* V$. 
This gives an isomorphism
\begin{equation} \label{Kvee}
\cHom_{\cO_{\eta_X}}(j_X^* V,\cO_{\eta_X}) 
\cong j_X^* L'^\vee.
\end{equation}
The isomorphism \eqref{Kvee} applied to 
$L' = \cO_X^{\oplus d}$ implies that the
isomorphism $\alpha$ induces an isomorphism
$V^* \cong V$.
The inverse of the isomorphism \eqref{Kvee} 
for $L' = L$ induces, by adjunction,
a homomorphism $L'^\vee \to V^*$.
By composing this with the inverse of the isomorphism
$V^* \cong V$ above, we obtain a homomorphism
$\beta_L \colon  L^\vee \to V$.
It is easy to check that $\beta_L$ is injective and
its image is an $\cO_X$-lattice of $V$.
We regard $L^\vee$ as an $\cO_X$-lattice of $V$ by
identifying $L^\vee$ with its image under $\beta_L$.

By associating $(L_2^\vee,L_1^\vee)$ to each object $(L_1,L_2)$
of $\cC_0^d$, we obtain a contravariant functor $\bD_0$
from $\cC_0^d$ to itself. One can check easily that
$\bD_0 \circ \bD_0$ is equal to the identity functor.
It is clear that an object $X$ of $\cC_0^d$ is an object 
of $\cC_{0,m}^d$ (\resp of $\cC_{0,r}^d$) if and only if $\bD_0(X)$ 
is an object $\cC_{0,r}^d$ (\resp of $\cC_{0,m}^d$).

\subsubsection{ }

Let $(L_1,L_2)$ be an object of $\cC_0^d$.
We construct an isomorphism 
$c(L_1,L_2) \colon L_1^\vee/L_2^\vee \xto{\cong} \bD(L_2/L_1)$
as follows.
Let us consider the functor $\bD' = R \cHom_{\cO_X}^\bullet(-,\cO_X)$
from the category of $\cO_X$-modules to the derived category $\cD_X$ 
of $\cO_X$-modules. 
We regard the category of $\cO_X$-modules 
as a full subcategory of $\cD_X$.
For a locally free $\cO_X$-module $L$ of finite rank, 
we have a natural isomorphism $\bD'(L) \cong L^\vee$.
Observe that the complex $0 \to \cK_X \to D \to 0$ 
gives an injective resolution of the complex $\cO_X$. 
Hence for a coherent
$\cO_X$-module $N$ of finite length we have a canonical
isomorphism $\bD'(N) \cong \bD(N)[-1]$.

By applying the functor $\bD'$ to the short exact sequence
$0 \to L_1 \to L_2 \to L_2/L_1 \to 0$, we obtain
a distinguished triangle 
$\bD'(L_2) \to \bD'(L_1) \to \bD'(L_2/L_1)[1] \xto{+1}$
which can be regarded as a short exact sequence
$0 \to L_2^\vee \to L_1^\vee \to \bD(L_2/L_1) \to 0$.
The last exact sequence induces an isomorphism 
$L_1^\vee/L_2^\vee \xto{\cong} \bD(L_2/L_1)$
which we denote by $c(L_1,L_2)$.

\subsection{Proof of Theorem \ref{thm:section2} for $\cT_r^d$}

\begin{proof}[Proof of Theorem \ref{thm:section2} for $\cT_r^d$]
The isomorphism $c(L_1,L_2)$ for each object $(L_1,L_2)$
of $\cC_0^d$ induces an isomorphism
$\iota_0^d \circ \bD_0 \cong \bD \circ \iota_0^d$
of functors. In particular, we have a diagram
$$
\begin{CD}
\cC_{0,r}^d @>{\bD_0}>> \cC_{0,m}^d \\
@V{\iota_{0,r}^d}VV @VV{\iota_{0,m}^d}V \\
\cCo{d} @>{\bD}>> \cCo{d}
\end{CD}
$$
of categories and functors, which is commutative up to isomorphisms
of functors.
Hence the assertions of Theorem \ref{thm:section2} for $\cT_r^d$
follows from those for $\cT_m^d$
\end{proof}

This completes the proof of Theorem \ref{thm:section2}.

\section{Computation of the absolute Galois monoids}
\label{sec:4}
We introduced $Y$-sites with 
underlying category $\cC^d$ and grids
in the previous Section~\ref{sec:3}.
The goal of this section 
is to compute  
the associated Galois monoid, 
which is a topological monoid.
The answer (Theorem~\ref{cor:M_isom})
is, in the case of 
the $Y$-site with atomic topology,
that the Galois monoid is isomorphic to 
$\GL_d(\A_X)$.

The grid for $\cC^d$ is described 
in terms of $\cO_X$-lattices, but 
the adelic objects are closer to 
$\wh{\cO}_X$-lattices. Therefore,
we introduce the poset of 
$\wh{\cO}_X$-lattices
and the poset of pairs of such lattices.
The comparison of these two types of lattices,
hence of the grids and abolute Galois monoids,
is given in the earlier part of this section.

The steps in proving the theorem are as follows.
We introduced the category of 
 $\cO_X$-lattices in $\cK_X^d$
(see Section~\ref{sec:lattices})
and poset of pairs of lattices.
In the analogy with the classical Galois theory of a base field,
a lattice corresponds to a finite separable extension
of the base field contained in a fixed separable closure
of the base field and a pair of lattices corresponds 
to an inclusion of two finite separable extensions.
In order to compare with the finite adeles,
we introduce the category of $\wh{\cO}_X$-lattices
and the poset of pairs of lattices.
Using these, we form a $Y$-site and a grid,
and a morphism from the original category.
The absolute Galois monoid of this 
latter pair of a $Y$-site and a grid 
can be computed to be $\GL_d(\A_X)$
(or something similar depending on the
Grothendieck topology).

In Section~\ref{sec:adeles}, we recall the topological ring
structure of the ring of finite adeles.
There is nothing new, but we supplied details 
since it was not so easy to find them in the 
literature
and we use them for the proof of the main statements 
of this section.

Let $X$ be as in the previous section, i.e., $X$ is a 
regular noetherian scheme of pure Krull dimension one.
In this section we assume that the residue field at 
each closed point is finite.
We let $|X|$ denote the set of closed points of $X$.


\subsection{The ring $\A_X$ of finite adeles on $X$}
\label{sec:adeles}

\subsubsection{ }
For a closed point $x \in |X|$,
let $\wh{\cO}_{X,x}$ denote the completion 
of the local ring $\cO_{X,x}$ at $x$, and $\Frac\, \wh{\cO}_{X,x}$ 
denote the field of fractions.
We let $\A_X=\prod'_{x\in |X|} \Frac\, \wh{\cO}_{X,x}$ 
denote the ring of (finite) adeles on $X$,
and $\wh{\cO}_X=\prod_{x\in |X|} \wh{\cO}_{X,x} \subset \A_X$ 
its ring of integers.
Here $\prod'$ means the restricted direct product
with respect to the inclusions $\wh{\cO}_{X,x} \subset 
\Frac\, \wh{\cO}_{X,x}$.
For each $x \in |X|$, we regard the ring $\wh{\cO}_{X,x}$
as a topological ring with respect to the $\fram_{X,x}$-adic topology,
where $\fram_{X,x} \subset \cO_{X,x}$ denotes the maximal ideal.
We regard the ring $\wh{\cO}_X$ as a topological ring
with respect to the product topology.
We endow $\A_X$ with the finest topology such that
for any $a \in \A_X$, the map $\wh{\cO}_X \to \A_X$
which sends $b$ to $a+b$ is continuous.
In other words, we say that a subset 
$U \subset \A_X$ is open if the subset
$\{ b \in \wh{\cO}_X \ |\ a+ b \in U \}$
of $\wh{\cO}_X$ is open for any $a \in \A_X$.
%
We will prove in Corollary \ref{cor:top_ring} below 
that $\A_X$ is a topological ring with respect to
this topology.
When $X$ is the spectrum of the ring of integers
of a finite extension $F$ of $\Q$,
the topological ring $\A_X$ is the usual
(topological) ring of finite adeles of $F$.

\begin{lem}\label{lem:topology A_X}
\begin{enumerate}
\item $\wh{\cO}_X$ is an open subset of $\A_X$.
\item If $U \subset \A_X$ is an open subset,
then $a + U$ is an open subset of $\A_X$ for
any $a \in \A_X$.
\item Let $U \subset \wh{\cO}_X$ be a subset.
Then $U$ is an open subset of $\A_X$
if and only if $U$ is an open subset of $\wh{\cO}_X$.
\item Let $U \subset \A_X$ be a subset.
Then $U$ is an open subset of $\A_X$ if and only
if for any $a \in U$ there exists an open neighborhood
$U_a$ of $0$ in $\wh{\cO}_X$ 
satisfying $a + U_a \subset U$.
\end{enumerate}
\end{lem}

\begin{proof}
The claim (1) follows from the definition of
the topology of $\A_X$, since the subset
$\{ b \in \wh{\cO}_X\ |\ a+b \in \wh{\cO}_X \}$
of $\wh{\cO}_X$ is equal to $\wh{\cO}_X$
if $a \in \wh{\cO}_X$, and is empty otherwise.

The claim (2) is obvious from the definition
of the topology of $\A_X$.

The ``only if" part of the claim (3) is
obvious since the inclusion map
$\wh{\cO}_X \inj \A_X$ is continuous.
We prove the ``if" part.
Suppose that $U \subset \wh{\cO}_X$ is an
open set of $\wh{\cO}_X$.
Since $\wh{\cO}_X$ is a topological ring,
$a + U$ is an open set of $\wh{\cO}_X$ for any
$a \in \wh{\cO}_X$.
The claim follows from the definition of
the topology of $\A_X$, since the subset
$\{ b \in \wh{\cO}_X\ |\ a+b \in U \}$
of $\wh{\cO}_X$ is equal to $(-a) + U$
if $a \in \wh{\cO}_X$, and is empty otherwise.
This proves the claim (3).

The ``if" part of the claim (4) follows from
the claim (3). We prove the ``only if" part.
Let $U \subset \A_X$ be an open subset and
let $a \in U$. Then $U_a = \{ b \in \wh{\cO}_X\ |\ 
a+b \in U \}$ is an open subset of $\wh{\cO}_X$
containing $0$.
Since $a + U_a \subset U$, the claim follows.
\end{proof}

\begin{cor}\label{cor:nbd A_X}
The set of open ideals of $\wh{\cO}_X$
forms a fundamental system of neighborhoods of
$0$ both in $\wh{\cO}_X$ and in $\A_X$.
\end{cor}

\begin{proof}
Let $U$ be an open set of $\wh{\cO}_X$
containing $0$.
Since the open ideals of $\wh{\cO}_{X,x}$
forms a fundamental system of neighborhoods
of $0$ in $\wh{\cO}_{X,x}$, there exist
a finite subset $S \subset |X|$ and
an open ideal $I_s \subset \wh{\cO}_{X,s}$
for each $s \in S$ such that $U$ contains
the open ideal $\prod_{x \not\in S} \wh{\cO}_{X,x}
\times \prod_{s \in S} I_s$.
This proves that the open ideals of
$\wh{\cO}_X$ form a fundamental system of
neighborhoods of $0$ in $\wh{\cO}_X$.
It follows from Lemma \ref{lem:topology A_X}
that they also form a fundamental system of
neighborhoods of $0$ in $\A_X$.
\end{proof}

\begin{lem}\label{lem:topology A_X_2}
For any $a \in \A_X$ the map
$\A_X \to \A_X$ which sends $b$ to $ab$
is continuous.
\end{lem}

\begin{proof}
Let $a = (a_x)_{x \in |X|} \in \A_X$ and
let $\I \subset \wh{\cO}_X$ be an open ideal.
By Corollary \ref{cor:nbd A_X}, it suffices to
prove that the set $\{b \in \A_X\ |\ ab \in \I\}$
contains an open ideal of $\wh{\cO}_X$.
There exists a finite subset $S_1 \subset |X|$
such that $a_x \in \wh{\cO}_{X,x}$ for $x \not\in S_1$.
There exist a finite subset $S \subset |X|$ containing
$S_1$ and an open ideal $I_s \subset \wh{\cO}_{X,s}$
for each $s \in S$ such that $\I$ contains
the open ideal $\prod_{x \not\in S} \wh{\cO}_{X,x}
\times \prod_{s \in S} I_s$.
Let $s \in S$.
It is easy to see that
there exists an open ideal $J_s \subset \wh{\cO}_{X,s}$
satisfying $a_s J_s \subset I_s$.
We set $\bJ = \prod_{x \not\in S} \wh{\cO}_{X,x}
\times \prod_{s \in S} J_s$.
Then $\bJ$ is an open ideal of $\wh{\cO}_X$
which satisfies $a \bJ \subset \I$.
This proves the claim.
\end{proof}

\begin{cor}\label{cor:top_ring}
The addition $\A_X \times \A_X \to \A_X$
and the multiplication $\A_X \times \A_X \to \A_X$
are continuous.
\end{cor}

\begin{proof}
The claim for the addition follows immediately
from Lemma \ref{lem:topology A_X} (2).
We prove the claim for the multiplication.
Let $U\subset \A_X$ be an open set and
let $(a,b) \in \A_X \times \A_X$ be a pair
satisfying $ab \in U$.
By Lemma \ref{lem:topology A_X} (2) and
Corollary \ref{cor:nbd A_X},
it suffices to prove that there exist
open ideals $\bJ_1, \bJ_2 \subset \wh{\cO}_X$
satisfying $(a+\bJ_1)(b+\bJ_2) \subset U$.
It follows from Lemma \ref{lem:topology A_X} (2) and
Corollary \ref{cor:nbd A_X}
that there exists an open ideal
$\I \subset \A_X$ satisfying $ab + \I \subset U$.
It follows from Lemma \ref{lem:topology A_X_2}
that there exist open ideals $\bJ'_1, \bJ'_2 \subset
\wh{\cO}_X$ satisfying $a \bJ'_2 \subset \I$ and
$b \bJ'_1 \subset \I$.
We set $\bJ_i =\bJ'_i \cap \I$ for $i=1,2$.
Then $\bJ_1$ and $\bJ_2$ are open ideals of
$\wh{\cO}_X$ and we have
$(a+\bJ_1)(b+\bJ_2) \subset
ab + a \bJ_2 + b \bJ_1 + \bJ_1 \bJ_2
\subset ab + \I + \I + \I^2 \subset \I$.
This proves the claim.
\end{proof}

\begin{lem}\label{lem:topology whO_X}
\begin{enumerate}
\item An element $a=(a_x)_{x \in |X|} 
\in \wh{\cO}_X$ is invertible in $\A_X$
if and only if $a_x \neq 0$ for all $x \in |X|$
and $a_x \in \wh{\cO}_{X,x}^\times$ for
all but finitely many $x \in |X|$.
\item If $a \in \wh{\cO}_X$ is invertible in
$\A_X$, then the ideal $a \wh{\cO}_X$ generated
by $a$ is an open ideal of $\wh{\cO}_X$.
\item Any open ideal $\I \subset \wh{\cO}_X$ is
generated by an element $a \in \wh{\cO}_X$
which is invertible in $\A_X$.
\item The open ideals of $\wh{\cO}_X$ form a
fundamental system of neighborhoods of 
$0 \in \wh{\cO}_X$.
\end{enumerate}
\end{lem}

\begin{proof}
The claim (1) is obvious from the definition of
$\A_X$.

Let $a=(a_x)_{x \in |X|} \in \wh{\cO}_X$
be an element which is invertible in $\A_X$. The ideal $a \wh{\cO}_X$
generated by $a$ is equal, as a subset
of $\A_X$, to the direct product of the
ideals $a_x \wh{\cO}_{X,x}$ of $\wh{\cO}_{X,x}$.
It follows from the claim (1) that
$a_x \wh{\cO}_{X,x}$ is a non-zero ideal of
$\wh{\cO}_{X,x}$ for all $x \in |X|$ and
$a_x \wh{\cO}_{X,x} = \wh{\cO}_{X,x}$ for
all but finitely many $x \in |X|$.
Since any non-zero ideal of $\wh{\cO}_{X,x}$ is open,
the claim (2) follows from the definition of
the topology on $\wh{\cO}_X$.

Let $\I \subset \wh{\cO}_X$ be an open ideal.
Since the open ideals of $\wh{\cO}_{X,x}$ 
form a fundamental system of neighborhoods
of $0 \in \wh{\cO}_{X,x}$ for each $x \in |X|$,
there exist a finite subset $S \subset |X|$
and an open ideal $J_s \subset \wh{\cO}_{X,s}$
for each $s \in S$ such that $\I$ contains
the direct product $\prod_{x \not\in S}
\wh{\cO}_{X,x} \times \prod_{s \in S} J_s$.
This in particular implies that 
$\I$ is of the form $\I = \prod_{x \not\in S}
\wh{\cO}_{X,x} \times I_S$
for some open ideal 
$I_S \subset \prod_{s \in S} \wh{\cO}_{X,s}$.
Since $\prod_{s \in S} \wh{\cO}_{X,s}$
is a direct product of finitely many rings,
there exists an ideal $I_s$ for each
$s \in S$ such that $I_S = \prod_{s \in S} I_s$
as a subset of $\prod_{s \in S} \wh{\cO}_{X,s}$.
Since $I_s \supset J_s$, the ideal $I_s$ is
open in $\wh{\cO}_{X,s}$ and hence is generated
by a non-zero element $a_s \in \wh{\cO}_{X,s}$.
We set $a_x = 1$ for $x \in |X| \setminus S$.
Then the element $a =(a_x)_{x \in |X|}$ is
invertible in $\A_X$ and the ideal $\I$ is generated
by the element $a$. This proves the claim (3).

The claim (4) follows from the definition of
the topology of $\wh{\cO}_X$ and the well-known
fact that for each $x \in |X|$, the open ideals
of $\wh{\cO}_{X,x}$ form a fundamental system
of neighborhood of $0 \in \wh{\cO}_{X,x}$.
\end{proof}

\begin{cor}\label{cor:compact}
\begin{enumerate}
\item Any open ideal of $\wh{\cO}_X$ is compact.
\item For any open ideals $\I_1,\I_2 \subset \wh{\cO}_X$,
with $\I_1 \subset \I_2$, the quotient $\I_2/\I_1$
is finite as an abelian group.
\item The natural map $\wh{\cO}_X
\to \varprojlim_{\I} \wh{\cO}_X/\I$ is bijective,
where $\I$ runs over the open ideals of $\wh{\cO}_X$.
\end{enumerate}
\end{cor}

\begin{proof}
Let $\I \subset \wh{\cO}_X$ be an open ideal.
It follows from Lemma \ref{lem:topology whO_X}
that $\I$ is equal to the direct product of
open ideals of $\wh{\cO}_{X,x}$ where
$x$ runs over the closed points of $X$.
Since any open ideal of $\wh{\cO}_{X,x}$ is
compact, the claim (1) follows.

Let the notation be as in the claim (2).
Then it follows from the claim (1) that
$\I_2/\I_1$ with the quotient topology is discrete
and compact. Hence $\I_2/\I_1$ is a finite set.
This proves the claim (2).

It follows from Lemma \ref{lem:topology whO_X} (3) 
that any open ideals of $I \subset \wh{\cO}_X$ 
is of the form $\I = \prod_{x \in |X|} I_x$,
where $I_x \subset \wh{\cO}_{X,x}$ is an open ideal
for each $x \in |X|$ satisfying $I_x = \wh{\cO}_{X,x}$
for all but finitely many $x \in |X|$.
Hence the limit
$\varprojlim_{\I} \wh{\cO}_X/\I$ is equal to
the direct product
$\prod_{x \in |X|} \varprojlim_{I_x}
\wh{\cO}_{X,x}/I_x$, where for each $x \in |X|$,
$I_x$ runs over the open ideals of $\wh{\cO}_{X,x}$.
Since $\wh{\cO}_{X,x}$ is the completion
of $\cO_{X,x}$, we have
$\wh{\cO}_{X,x} \cong \varprojlim_{I_x}
\wh{\cO}_{X,x}/I_x$.
This proves the claim (3).
\end{proof}

Let $\I \subset \wh{\cO}_X$ be an open ideal.
It then follows from Lemma \ref{lem:topology whO_X} (3)
that $\I = a \wh{\cO}_X$ for some element
$a \in \wh{\cO}_X$ which is invertible in $\A_X$.
We let $\I^{-1}$ denote the $\wh{\cO}_X$-submodule
of $\A_X$ generated by $a^{-1}$. It is clear
that $\I^{-1}$ does not depend on the choice of $a$.

For an $\wh{\cO}_X$-submodule $\bL$ of an
$\A_X$-module $W$, we let $\I^{-1} \bL$ denote
the $\wh{\cO}_X$-submodule $a^{-1} \bL$ of $W$.

\subsection{The posets $\Lat^d$ and $\Lat^d_\A$}
\label{sec:poset Lat}
Let us fix an integer $d \ge 1$.
We use the notation in the previous section.
Let $\Lat^d$ denote the set of $\cO_X$-lattices
in $V = \cK_X^{\oplus d}$.

The aim of this paragraph is to introduce
the notion of an $\wh{\cO}_X$-lattice in $\A_X^{\oplus d}$
and relate $\wh{\cO}_X$-lattices in $\A_X^{\oplus d}$
with $\cO_X$-lattices in $V$.
\subsubsection{ }
Let $\A_X^{\oplus d}$ be the free $\A_X$-module of
rank $d$ over $\A_X$, equipped
with the product topology.
An $\wh{\cO}_X$-lattice of $\A_X^{\oplus d}$
is an $\wh{\cO}_X$-submodule of $\A_X^{\oplus d}$
which is compact and open in $\A_X^{\oplus d}$.
Let $\Lat^d_\A$ denote the set of the $\wh{\cO}_X$-lattices
in $\A_X^{\oplus d}$. We regard $\Lat^d_\A$ as a poset
with respect to the inclusions.

\begin{lem}\label{lem:lattice_basics}
\begin{enumerate}
\item $\wh{\cO}_X^{\oplus d} \subset
\A_X^{\oplus d}$ is an $\wh{\cO}_X$-lattice.
\item The set $\Lat^d_\A$ forms a fundamental system
of neighborhoods of $0$ in $\A_X^{\oplus d}$.
\item Let $\bL$ be an $\wh{\cO}_X$-lattice in
$\A_X^{\oplus d}$. Then for any open
ideal $\I \subset \wh{\cO}_X$, the $\wh{\cO}_X$-submodules
$\I \bL$ and $\I^{-1}\bL$ of $\A_X^{\oplus d}$ are
$\wh{\cO}_X$-lattices in $\A_X^{\oplus d}$.
\item Let $\bL$ be an $\wh{\cO}_X$-lattice in
$\A_X^{\oplus d}$. Then for any finite set 
$S \subset \A_X^{\oplus d}$,
the $\wh{\cO}_X$-submodule of $\A_X^{\oplus d}$
generated by $\bL$ and $S$ is an 
$\wh{\cO}_X$-lattice in $\A_X^{\oplus d}$.
\end{enumerate}
\end{lem}

\begin{proof}
The claim (1) follows from Lemma \ref{lem:topology A_X} (1)
and Corollary \ref{cor:compact} (1).

It follows from Corollary \ref{cor:nbd A_X} that
the set 
$$
\{ \I_1\oplus \cdots \oplus \I_d\ |\ 
\I_1,\ldots,\I_d \subset \wh{\cO}_X
\text{ are open ideals}\}
$$
forms a fundamental system of neighborhoods of $0$
in $\A_X^{\oplus d}$.
It follows from Corollary \ref{cor:compact} (1)
that $\I_1\oplus \cdots \oplus \I_d \in \Lat^d_\A$
for any open ideals
$\I_1,\ldots,\I_d \subset \wh{\cO}_X$.
This proves the claim (2).

It follows from Lemma \ref{lem:topology whO_X} (3)
that $\I$ is generated by an element $a \in \wh{\cO}_X$
which is invertible in $\A_X$.
It then follows from Lemma \ref{lem:topology A_X_2}
that the map $\A_X^{\oplus d} \to \A_X^{\oplus d}$
which sends $x$ to $a x$ is a homeomorphism.
Hence both $\I\bL = a\bL$ and $\I^{-1} \bL = a^{-1} \bL$
is an $\wh{\cO}_X$-lattice in $\A_X^{\oplus d}$.
This proves the claim (3).

We prove the claim (4).
Let $\bL'$ denote 
the $\wh{\cO}_X$-submodule of $\A_X^{\oplus d}$
generated by $\bL$ and $S$ is an 
$\wh{\cO}_X$-lattice in $\A_X^{\oplus d}$.
Since $\bL'$ contains $\bL$ as an $\wh{\cO}_X$-submodule,
$\bL'$ is an open submodule of $\A_X^{\oplus d}$.
It follows from Corollary \ref{cor:top_ring}
that the map $L \times \prod_{s \in S} \wh{\cO}_X
\to \A_X^{\oplus d}$ 
which sends $(x,(a_s)_{s \in S})$
to $x + \sum_{s \in S} a_s s \in \A_X^{\oplus d}$
is continuous.
Since $\bL \times \prod_{s \in S} \wh{\cO}_X$
is compact, its image $\bL'$ is a compact
subset of $\A_X^{\oplus d}$.
Hence $\bL'$ is an $\wh{\cO}_X$-lattice in $\A_X^{\oplus d}$.
This proves the claim (4).
\end{proof}

\begin{lem}\label{lem:finite_index}
For $\bL_1$, $\bL_2$ in $\Lat^d_\A$ with $\bL_1 \le \bL_2$, the
quotient $\bL_2/\bL_1$ equipped with the quotient topology 
is a discrete $\wh{\cO}_X$-module of finite length.
\end{lem}

\begin{proof}
Since $\bL_1$ is open, the quotient $\bL_2/\bL_1$ equipped with the
quotient topology is a discrete $\wh{\cO}_X$-module.
Let us consider the open covering 
$\bL_2 = \bigcup_{a \in \bL_2}(a+\bL_1)$ of $\bL_2$.
Since $\bL_2$ is compact, there exists a finite subset $S \subset \bL_2$
such that $\bL_2 = \bigcup_{a \in S} (a+\bL_1)$. This shows that
$\bL_2/\bL_1$ is a finite set. Hence $\bL_2/\bL_1$ is of finite length
as an $\wh{\cO}_X$-module.
\end{proof}

\begin{lem} \label{lem:lattices}
For $\bL_1$, $\bL_2$ in $\Lat^d_\A$, the intersection
$\bL_1 \cap \bL_2$ and the sum $\bL_1+\bL_2$ in $\A_X^{\oplus d}$
are in $\Lat^d_\A$.
\end{lem}

\begin{proof}
The subset $\bL_1 \subset \A_X^{\oplus d}$
is closed since its complement is a union of translations
of $\bL_1$, each of which is an open subset of $\A_X^{\oplus d}$.
Hence $\bL_1 \cap \bL_2$ is a closed subset of the compact set $\bL_2$.
This shows that $\bL_1 \cap \bL_2$ is compact.
Since $\bL_1 \cap \bL_2$ is open and compact, we have 
$\bL_1 \cap \bL_2 \in \Lat^d_\A$.

By Lemma \ref{lem:finite_index} the quotient
$(\bL_1+\bL_2)/\bL_1 \cong \bL_2/(\bL_1 \cap \bL_2)$ is finite
as an abelian group. Hence $\bL_1 + \bL_2$ is a union of
a finite number of translations of $\bL_1$.
This shows that $\bL_1 + \bL_2$ is open and compact.
Hence $\bL_1 + \bL_2 \in \Lat^d_\A$.
\end{proof}

\begin{lem}\label{lem:I1I2}
Let $\bL_1,\bL_2 \in \Lat^d_\A$. Then there exist
open ideals $\I_1,\I_2 \subset \wh{\cO}_X$
satisfying $\I_1 \bL_2 \subset \bL_1 \subset \I_2^{-1} \bL_2$.
\end{lem}

\begin{proof}
It follows from Lemma \ref{lem:lattices}
and Lemma \ref{lem:finite_index} that
$(\bL_1+ \bL_2)/\bL_1$ and $(\bL_1+\bL_2)/\bL_2$ 
are discrete $\wh{\cO}_X$-modules of finite length.
Hence there exist 
open ideals $\I_1,\I_2 \subset \wh{\cO}_X$
such that $(\bL_1+\bL_2)/\bL_1$ is annihilated by
$\I_1$ and that $(\bL_1+\bL_2)/\bL_2$ is annihilated by $\I_2$.
Hence we have $\I_1 \bL_2 \subset \bL_1 \subset \I_2^{-1} \bL_2$.
This proves the claim.
\end{proof}

We regard an element in $\A_X^{\oplus d}$
as a row vector. Hence the group $\GL_d(\A_X)$ acts
on $\A_X^{\oplus d}$ via the multiplication
from the right. 

Let $g \in \GL_d(\A_X)$.
For an $\wh{\cO}_X$-lattice 
$\bL \subset \A_X^{\oplus d}$, we set
$\bL g = \{ vg\ |\ v \in \bL \}$.

\begin{lem}\label{lem:lattice_translation}
Let $\bL \in \Lat^d_\A$ and $g \in \GL_d(\A_X)$.
Then we have $\bL g \in \Lat^d_\A$.
\end{lem}

\begin{proof}
We set $\bL_0 = \wh{\cO}_X^{\oplus d}$.
It follows from Lemma \ref{lem:I1I2}
that there exist open ideals 
$\I_1, \I_2 \subset \wh{\cO}_X$ satisfying
$\I_1 \bL_0 \subset \bL \subset \I_2^{-1} \bL_0$.
It follows from Lemma \ref{lem:lattice_basics} (4)
and Lemma \ref{lem:I1I2} applied for $d=1$ that
there exist open ideals 
$\bJ_1, \bJ_2 \subset \wh{\cO}_X$ such that
all entries of the matrix $g$ belong to $\bJ_1^{-1}$
and that all entries of the matrix 
$g$ belong to $\bJ_2^{-1}$.
Since $\bL g \subset \I_2^{-1} \bL_0 g
\subset \I_2^{-1} \bJ_1^{-1} \bL_0$
and $\I_1 \bL_0 g^{-1} \subset
\I_1 \bJ_2^{-1} \bL_0 \subset \bJ_2^{-1} L$,
we have 
$\I_1 \bJ_2 \bL_0 \subset \bL g \subset \I_2^{^1} \bJ_1^{-1} \bL_0$.
It follows from Lemma \ref{lem:finite_index}
that $\I_2^{^1} \bJ_1^{-1} \bL_0/ \I_1 \bJ_2 \bL_0$
is a discrete $\wh{\cO}_X$-module of finite
length. 
Hence $\bL_g$ is generated by
$\I_1 \bJ_2 \bL_0$ and a finite set.
Hence it follows from Lemma \ref{lem:lattice_basics}
that $\bL g$ is an $\wh{\cO}_X$-lattice
in $\A_X^{\oplus d}$.
This proves the claim.
\end{proof}

\begin{cor}\label{cor:translation}
For $g \in \GL_d(\A_X)$, the map
$\A_X^{\oplus d} \to \A_X^{\oplus d}$ given
by the right multiplication by $g$
is a homeomorphism.
\end{cor}

\begin{proof}
This follows from Lemma \ref{lem:lattice_basics}
and Lemma \ref{lem:lattice_translation}.
\end{proof}

\begin{lem}\label{lem:freeness1}
Any open $\wh{\cO}_X$-submodule of $\wh{\cO}_X^{\oplus d}$ 
is free of rank $d$ as an $\wh{\cO}_X$-module.
\end{lem}

\begin{proof}
We let $\bL_0 = \wh{\cO}_X^{\oplus d}$.
Let $\bL$ be an open $\wh{\cO}_X$-submodule of $\bL_0$.
Set $N = \wh{\cO}_X^{\oplus d}/\bL$. Then $N$ is discrete and compact.
Hence $N$ is a discrete $\wh{\cO}_X$-module of finite length.
Let $S \subset |X|$ denote the support of the
$\cO_X$-module of finite length corresponding to $N$.
Then $S$ is a finite set. 
We set $\wh{\cO}_X^S = \prod_{x \in |X| \setminus S} 
\wh{\cO}_{X,x}$. Since the topological ring 
$\wh{\cO}_X$ is equal to the direct product
$\wh{\cO}_X = \wh{\cO}_X^S \times \prod_{s \in S} \cO_{X,s}$,
the $\wh{\cO}_X$-modules $\bL_0$ and $\bL$
have the corresponding decompositions 
$\bL_0 = \bL^S \times \prod_{s \in S} \bL_{0,s}$ and
$\bL = \bL^S \times \prod_{s \in S} \bL_s$.
It follows from the definition of $S$ that
we have $\bL^S = \bL_0^S = (\wh{\cO}_X^S)^{\oplus d}$.
For $s \in S$, the module $\bL_s$ is 
an $\wh{\cO}_{X,s}$-submodule of $\wh{\cO}_{X,s}^{\oplus d}$
of finite index.
Hence $\bL_s$ is free of rank $d$ over $\wh{\cO}_X$.
This proves that $\bL$ is a free $\wh{\cO}_X$-module of 
rank $d$ for any $s \in S$.
This proves the claim.
\end{proof}

\begin{lem} \label{lem:freeness}
Any $\bL \in \Lat^d_\A$ is free of rank $d$ 
as an $\wh{\cO}_X$-module.
Moreover the homomorphism 
$i_\bL \colon  \bL \otimes_{\wh{\cO}_X} \A_X \to \A_X^{\oplus d}$ 
induced by the inclusion $\bL \subset \A_X^{\oplus d}$ 
is an isomorphism of $\A_X$-modules.
\end{lem}

\begin{proof}
We set $\bL_0=\wh{\cO}^{\oplus d} \in \Lat^d_\A$.
It follows from Lemma \ref{lem:lattices}
that $\bL + \bL_0$ is in $\Lat^d_\A$.
It follows from Lemma \ref{lem:finite_index}
that $N = (\bL+\bL_0)/\bL_0$ is a discrete
$\wh{\cO}_X$-module of finite length.
Since the annihilator of the
$\wh{\cO}_X$-module $N$ is an open ideal of 
$\wh{\cO}_X$, it follows that there
exists an element in $\wh{\cO}_X \cap \A_X^\times$
such that $a N = 0$.
Hence we have $a\bL \subset \bL_0$.
Since $a\bL$ is an open submodule of 
$\bL_0 = \wh{\cO}_X^{\oplus d}$, it follows from
Lemma \ref{lem:freeness1} that $\bL$
is a free $\wh{\cO}_X$-module of rank $d$.

The ring $\A_X$ can be written as a colimit
$\A_X = \varinjlim_{\I} \I^{-1}\wh{\cO}_X$,
where $\I$ runs over the ideals of $\wh{\cO}_X$
of the form $\I = a \wh{\cO}_X$ for some
$a \in \wh{\cO}_X \cap \A_X^{\times}$,
and $\I^{-1} \wh{\cO}_X$ denotes the 
$\wh{\cO}_X$-submodule of $x \in \A_X$ consisting of
elements $x$ satisfying 
$\I x \subset \wh{\cO}_X$.
Since $\I \subset \wh{\cO}_X$ is a principal ideal 
generated by a non-zero-divisor, $\I^{-1} \wh{\cO}_X$
is isomorphic to $\wh{\cO}_X$ as an $\wh{\cO}_X$-module.
Hence $\A_X$ is a flat $\wh{\cO}_X$-module.

Since $\A_X$ is a flat $\wh{\cO}_X$-module,
we have, for two lattices $\bL_1, \bL_2 \in \Lat^d_\A$
with $\bL_1 \le \bL_2$, a short exact sequence
$0\to \bL_1\otimes_{\wh{\cO}_X} \A_X 
\to \bL_2\otimes_{\wh{\cO}_X} \A_X
\to (\bL_2/\bL_1) \otimes_{\wh{\cO}_X} \A_X \to 0$.
Since $\bL_2/\bL_1$ is a discrete $\wh{\cO}_X$-modules
of finite length, it is annihilated by an
element $a \in \wh{\cO}_X \cap \A_X^\times$,
which shows that $(\bL_2/\bL_1) \otimes_{\wh{\cO}_X} \A_X = 0$.
Hence the homomorphism $\bL_1\otimes_{\wh{\cO}_X} \A_X 
\to \bL_2\otimes_{\wh{\cO}_X} \A_X$ is an isomorphism.

It follows from the commutativity of the diagram
$$
\begin{CD}
\bL\otimes_{\wh{\cO}_X} \A_X @>{\cong}>>
(\bL+\bL_0) \otimes_{\wh{\cO}_X} \A_X
@<{\cong}<< 
\bL_0 \otimes_{\wh{\cO}_X} \A_X \\
@V{i_\bL}VV @V{i_{\bL+\bL_0}}VV 
@V{i_{\bL_0}}V{\cong}V \\
\A_X^{\oplus d} @= \A_X^{\oplus d}
@= \A_X^{\oplus d}
\end{CD}
$$
that the homomorphism
$i_\bL \colon  \bL \otimes_{\wh{\cO}_X} \A_X \to \A_X^{\oplus d}$ 
is an isomorphism.
This proves the claim.
\end{proof}

\subsubsection{ }\label{sec:nuX}

We set $\wh{X} = \Spec\, \wh{\cO}_X$.

Let us construct a morphism $\nu_X\colon \wh{X} \to X$
of schemes. For a finite subset $S \subset |X|$ let 
$\wh{\cO}_X^S = \prod_{x \in |X| \setminus S} \wh{\cO}_{X,x}$.
The projection homomorphism $\wh{\cO}_X \to \wh{\cO}_X^S$
induces the morphism $\Spec \wh{\cO}_X^S \to \wh{X}$ which
is an open immersion. The schemes $\Spec \wh{\cO}_X^S$ form
an open basis of the underlying topological space of $\wh{X}$.
Let $U \subset X$ be an affine open dense subscheme 
and set $S = |X| \setminus |U|$.
Then $S$ is a finite set and the natural map $\Gamma(U,\cO_X) \to 
\prod_{x \in |U|} \cO_{X,x} \to \wh{\cO}_X^S$ induces a morphism
$\nu_U \colon \Spec\, \wh{\cO}_X^S \to U$ of schemes. By patching $\nu_U$'s 
for various $U$, we obtain a morphism $\nu_X \colon  \wh{X} \to X$
of schemes.

\begin{lem}\label{lem:nu_flat}
The morphism $\nu_X \colon  \wh{X} \to X$ is flat.
\end{lem}

\begin{proof}
Let $U = \Spec\, R \subset X$ be an affine open dense subscheme
and let $S = |X| \setminus |U|$. Since $\wh{\cO}_X^S$ is
a torsion-free $R$-module, the claim follows.
\end{proof}

For an $\cO_X$-module $\cF$, we set 
$\Gamma_\A(\cF) = \Gamma(\wh{X},\nu_X^* \cF)$.
Then $\Gamma_\A(\cF)$ is a $\wh{\cO}_X$-module.
This construction gives a functor $\Gamma_\A$ from
the category of $\cO_X$-modules to the category of
$\wh{\cO}_X$-modules.

\begin{lem} \label{lem:Gamma_exact}
The functor $\Gamma_\A$ is exact and commutes with filtered colimits.
\end{lem}

\begin{proof}
Since $\wh{X}$ is affine, the first claim follows from
Lemma \ref{lem:nu_flat}. The second claim can be checked easily.
\end{proof}

\begin{lem} \label{lem:bI_open}
If $Z \subset X$ is an effective Cartier divisor,
then $\Gamma_\A(\sI_Z)$ is an open ideal of $\wh{\cO}_X$.
Here we regard $\Gamma_\A(\sI_Z)$ as an ideal of $\wh{\cO}_X$
via the homomorphism $\Gamma_\A(\sI_Z) \to \Gamma_\A(\cO_X)
= \wh{\cO}_X$ which is injective due to Lemma \ref{lem:Gamma_exact}.
Conversely, any open ideal of $\wh{\cO}_X$ is of the form
$\Gamma_\A(\sI_Z)$ for some effective Cartier divisor on $X$.
\end{lem}

\begin{proof}
Let $Z \subset X$ be an effective Cartier divisor.
For each closed point $x \in |X|$, let us choose an
element $a'_x \in \cO_{X,x}$ such that the ideal $I_{Z,x}$
of $\cO_{X,x}$ is generated by $a'_x$. Let $a_x$ denote the
image of $a'_x$ in the completion $\wh{\cO}_{X,x}$.
Then by definition the ideal $\Gamma_\A(\sI_Z)$ of $\wh{\cO}_X$
is generated by a single element $a=(a_x)_{x \in |X|} 
\in \wh{\cO}_X$. Since $Z$ is artinian, $a'_x$ is a non-zero-divisor
for any $x \in |X|$ and $a'_x$ is a unit of $\wh{\cO}_{X,x}$
except for finitely many $x$'s. This proves that the
ideal $\Gamma_\A(\sI_Z)$ of $\wh{\cO}_X$ is open.

Let $\I \subset \wh{\cO}_X$ be an open ideal.
It follows from Lemma \ref{lem:topology whO_X} (3)
that the ideal $\I$ is generated by a single element $a =(a_x)_{x \in |X|}
\in \wh{\cO}_X$ which is invertible in $\A_X$.
Let us take a finite subset $S \subset |X|$ such that
$a_x$ is a unit of $\wh{\cO}_{X,x}$ for any $x \in |X| \setminus S$.
Let $R = \prod_{x\in S} \wh{\cO}_{X,x}/a_x \wh{\cO}_{X,x}$
and set $Z =\Spec\, R$. Since the composite
$\cO_{X,x} \inj \wh{\cO}_{X,x} \surj \wh{\cO}_{X,x}/a_x \wh{\cO}_{X,x}$
is surjective for any $x \in S$, one can regard $Z$ as a
closed subscheme of $X$. It is then easy to check that
the ideal $\Gamma_\A(\sI_Z)$ of $\wh{\cO}_X$ is equal to $\I$.
This completes the proof.
\end{proof}

Let $Z \subset X$ be an effective Cartier divisor.
Since $Z$ is artinian, $Z$ is affine and 
$\Gamma(Z,\cO_Z) \cong \prod_{z \in Z} \cO_{Z,z}$.
The inclusion $Z \inj X$ induces a surjective homomorphism
$\cO_{X,z} \surj \cO_{Z,z}$ for each $z \in Z$. 
Hence we obtain a surjective homomorphism
$\pi_Z\colon \wh{\cO}_X \to \Gamma(Z,\cO_Z)$.

Let us consider the diagram
\begin{equation} \label{ZX}
\begin{CD}
Z @>>> \wh{X} \\
@| @VV{\nu_X}V \\
Z @>{i_Z}>> X
\end{CD},
\end{equation}
where the upper horizontal arrow is the 
closed immersion $Z \inj \wh{X}$ given by $\pi_Z$.
It is easy to check that the diagram 
\eqref{ZX} is commutative.

\begin{lem} \label{lem:ZX_cartesian}
The diagram \eqref{ZX} is cartesian.
\end{lem}
\begin{proof}
It suffices to prove the claim Zariski locally on $X$.
Hence we may assume that $X = \Spec\, R$ is affine
and the ideal sheaf $\sI_Z$ is generated by a single element
$a \in \Gamma(X,\cO_X) =R$.
Then for any closed point $x \in |X|$ we have
$\cO_{Z,x} = \cO_{X,x}/a \cO_{X,x}$.
Since $\Gamma(Z,\cO_Z) \cong \prod_{x \in Z} \cO_{Z,x}$,
we have $\Gamma(Z,\cO_Z) \cong \wh{\cO}_X/a \wh{\cO}_X$
and the claim is proved.
\end{proof}

\begin{lem} \label{lem:XN}
Let $Z \subset X$ be an effective Cartier divisor and
let $N$ be an $\cO_X$-module annihilated by $\sI_Z$.
Then there exists a natural isomorphism
\begin{equation} \label{XN}
\Gamma_\A(N) \cong \Gamma(X, N)
\end{equation}
of $\wh{\cO}_X$-modules.
Here we regard the right-hand side as an $\wh{\cO}_X$-module
via the isomorphism $\Gamma(X,N) \cong \Gamma(X,{i_Z}_* i_Z^* N)
= \Gamma(Z,i_Z^*N)$ and the homomorphism 
$\pi_Z\colon  \wh{\cO}_X \to \Gamma(Z,\cO_Z)$.
\end{lem}

\begin{proof}
Since $\nu_X$ is flat and the diagram \eqref{ZX} is cartesian, 
we obtain the desired isomorphism \eqref{XN}
by applying 
\cite[Lemme (2.3.1)]{EGAIV} to $i_Z^* N$.
\end{proof}

\begin{cor} \label{cor:submod}
Let $N$ be an $\cO_X$-module of finite length.
Then the functor $\Gamma_\A$ gives a bijection from
the set of $\cO_X$-submodules of $N$ to
the set of $\wh{\cO}_X$-submodules of $\Gamma_\A(N)$.
\end{cor}

\begin{proof}
Let us take an effective Cartier divisor $Z \subset X$
such that $N$ is annihilated by $\sI_Z$.
Then the functor $i_Z^*$ gives a bijection from
the set of $\cO_X$-submodules of $N$ to
the set of $\cO_Z$-submodules of $i_Z^* N$.
Since $Z$ is affine and the underlying topological space of $Z$
is discrete, the global section functor gives
a bijection from the set of $\cO_Z$-submodules of $i_Z^* N$
to the set of $\Gamma(Z,\cO_Z)$-submodules of 
$\Gamma(Z,i_Z^* N) = \Gamma(X, {i_Z}_* i_Z^* N)
\cong \Gamma(X,N)$. Hence the claim follows from Lemma \ref{lem:XN}.
\end{proof}

\begin{lem} \label{lem:isom_A}
We have a canonical isomorphism 
$\Gamma_\A({j_X}_* \cO_{\eta_X}) \cong \A_X$.
\end{lem}

\begin{proof}
To prove the claim, we may assume that $X = \Spec\, R$ 
is affine and integral.
Then $R$ is a Dedekind domain.
Let $F$ denote the field of fractions of $R$.
Then we have $\Gamma_\A({j_X}_* \cO_{\eta_X}) =
\wh{\cO}_X \otimes_R F$, where we regard
$\wh{\cO}_X$ as an $R$-algebra via the diagonal
map $R \to \prod_{x \in |X|} \cO_{X,x} \to \wh{\cO}_X$.
Since for any element $a \in F$, the composite
$F \to \Frac\, \cO_{X,x} \to \Frac\, \wh{\cO}_{X,x}$
maps $a$ to an element of $\wh{\cO}_{X,x}$ except for
finitely many closed points $x \in |X|$.
Hence the diagonal map 
$F \to \prod_{x \in |X|} \Frac\, \wh{\cO}_{X,x}$
induces a homomorphism $F \to \A_X$ which makes the
diagram
$$
\begin{CD}
R @>{\subset}>> F \\
@VVV @VVV \\
\wh{\cO}_X @>{\subset}>>\A_X
\end{CD}
$$
commutative. This diagram induces a ring homomorphism
$\lambda \colon  \Gamma_\A({j_X}_* \cO_{\eta_X}) \to \A_X$.
The homomorphism $\lambda$ is surjective since 
for any element $a$ of $\A_X$, there exists a non-zero
element $b\in R$ such that $ba$ belongs to 
$\wh{\cO}_X$.
The homomorphism $\lambda$ is injective since it is
equal to the composite of the injective homomorphism
$\Gamma_\A({j_X}_* \cO_{\eta_X}) =
\wh{\cO}_X \otimes_R F \inj \A_X \otimes_{R} F$
with the inverse of the isomorphism 
$\A_X \cong \A_X \otimes_{R} F$.
\end{proof}

Lemma \ref{lem:isom_A} provides an isomorphism
$\Gamma_\A(V) \cong \A_X^{\oplus d}$. We denote this
isomorphism by $\lambda^d$.

Let us construct a map $T_\A\colon  \Lat^d \to \Lat^d_\A$.
Later we will prove that this map is bijective.
Let $L \in \Lat^d$. Since the functor $\Gamma_\A$
is exact, the inclusion $L \inj V$ induces
an injective homomorphism $\Gamma_\A(L) \inj \A_X^{\oplus d}$
of $\wh{\cO}_X$-modules.
Via this injection we regard 
$\Gamma_\A(L)$ as an $\wh{\cO}_X$-submodule of $\A_X^{\oplus d}$.
\begin{lem}
$\Gamma_\A(L)$ is an $\wh{\cO}_X$-lattice of $\A_X^{\oplus d}$.
\end{lem}

\begin{proof}
Let us choose an effective Cartier divisor $Z \subset X$
such that $\cO_X^{\oplus d}/(L\cap \cO_X^{\oplus d})$
is annihilated by $\sI_Z$.
Then we have $L \supset \sI_Z^{\oplus d}$.
It follows from Lemma \ref{lem:bI_open} that
$\Gamma_\A(\sI_Z)^{\oplus d}$ is an $\wh{\cO}_X$-lattice of $\A_X^{\oplus d}$. 
Since $\Gamma_\A(L) \supset \Gamma_\A(\sI_Z)^{\oplus d}$,
it follows from Lemma \ref{lem:bI_open} that $\Gamma_\A(L)$ is an
open $\wh{\cO}_X$-submodule of $\A_X^{\oplus d}$.
It follows from Lemma \ref{lem:XN} that we have
isomorphisms $\Gamma_\A(L)/\Gamma_\A(\sI_Z)^{\oplus d}
\cong \Gamma_\A(L/\sI_Z^{\oplus d})
\cong \Gamma(X, L/\sI_Z^{\oplus d})$.
Since $L/\sI_Z^{\oplus d}$ is of finite length,
it follows that $\Gamma_\A(L)/\Gamma_\A(\sI_Z)^{\oplus d}$
is a finite abelian group.
Since $\Gamma_\A(\sI_Z)^{\oplus d}$ is compact, 
it follows that $\Gamma_\A(L)$ is compact.
This proves that $\Gamma_\A(L)$ is an $\wh{\cO}_X$-lattice of 
$\A_X^{\oplus d}$.
\end{proof}

\begin{lem} \label{lem:T_A_injective}
The map $T_\A \colon  \Lat^d \to \Lat^d_\A$ given by $\Gamma_\A$ is injective.
\end{lem}

\begin{proof}
Let $L,L' \in \Lat^d$ and suppose that $\Gamma_\A(L) = \Gamma_\A(L')$
as $\wh{\cO}_X$-lattices of $\A_X^{\oplus d}$.
Since $\Gamma_\A$ is exact, we have 
$\Gamma_\A(L/(L\cap L')) \cong \Gamma_\A(L'/(L\cap L')) \cong 0$.
It follows from Lemma \ref{lem:XN} that
$L/(L\cap L') \cong L/(L\cap L') \cong 0$.
Hence we have $L = L'$.
\end{proof}

\begin{lem} \label{lem:T_A_surjective}
The map $T_\A \colon  \Lat^d \to \Lat^d_\A$ given by $\Gamma_\A$ is surjective.
\end{lem}

\begin{proof}
It suffices to prove the surjectivity.
Let $\bL$ be an arbitrary $\wh{\cO}_X$-lattice of $\A_X^{\oplus d}$.
Then there exists open ideals $\I, \bJ \subset \wh{\cO}_X$ satisfying
$(\I^{-1})^{\oplus d} \subset \bL \subset \bJ^{\oplus d}$.
It follows from Lemma \ref{lem:bI_open} that there exist effective
Cartier divisors $Z,W \subset X$ satisfying 
$\I = \Gamma_\A(\sI_Z)$ and 
$\bJ = \Gamma_\A(I_W)$ as ideals of $\wh{\cO}_X$.
It follows from Corollary \ref{cor:submod} that there exists a unique
$\cO_X$-submodule $N$ of $(\sI_Z^{-1})^{\oplus d}/I_W^{\oplus d}$ such that
$\Gamma_\A(N) = \bL/\bJ^{\oplus d}$ as $\wh{\cO}_X$-submodules 
of $(\I^{-1})^{\oplus d}/\bL \subset \bJ^{\oplus d}$.
Let $L$ denote the inverse image of $N$ under the
quotient homomorphism 
$(\sI_Z^{-1})^{\oplus d} \surj (\sI_Z^{-1})^{\oplus d}/I_W^{\oplus d}$.
It follows from Lemma \ref{lem:lattice_V} that 
$L$ is an $\cO_X$-lattice of $V$.
It is then straightforward to check that
$\Gamma_\A(L) = \bL$ as an $\wh{\cO}_X$-submodule of $\A_X^{\oplus d}$.
This prove that the map $T_\A$ is surjective.
\end{proof}

It follows from Lemma \ref{lem:T_A_injective} and 
Lemma \ref{lem:T_A_surjective}, that the map $T_\A$ is bijective.
it is easy to check that the map $T_\A$ preserves the orderings.
Hence the map $T_\A$ is an isomorphism of posets.

\subsubsection{}
\label{sec:Pair adele}
We let $\Pair^d_\A$ denote the following poset.
The elements of $\Pair^d_\A$ are the pairs 
$(\bL_1,\bL_2)$ of elements in $\Lat^d_\A$ with $\bL_1 \le \bL_2$.
For two elements $(\bL_1,\bL_2)$ and $(\bL'_1,\bL'_2)$
in $\Pair^d_\A$, we have $(\bL_1,\bL_2) \le (\bL'_1,\bL'_2)$
if and only if $\bL'_1 \le \bL_1 \le \bL_2 \le \bL'_2$.
The isomorphism $T_\A$ induces an isomorphism 
$\Pair^d \xto{\cong} \Pair^d_\A$ of posets.
%

We let $\cC_{\A,0}^d$ denote the opposite category of the poset
$\Pair^d_\A$ regarded as a small category. By definition,
the set of objects in $\cC_{\A,0}^d$ is that of $\Pair^d_\A$
and for two elements $(\bL_1,\bL_2)$ and $(\bL'_1,\bL'_2)$, the set
$\Hom_{\cC_0^d}((\bL_1,\bL_2),(\bL'_1,\bL'_2))$ consists
of a single element if $(\bL_1,\bL_2) \ge (\bL'_1,\bL'_2)$,
and is the empty set otherwise. Here the symbol $\ge$
denotes the ordering of the poset $\Pair^d_\A$.
To avoid confusion, we never use the symbols $\ge$ and 
$\le$ to denote the ordering of the dual poset $\cC_{\A,0}^d$.

The isomorphism $T_\A$ induces an isomorphism 
$\cC_0^d \xto{\cong} \cC_{\A,0}^d$ of categories.
We let $\iota_{\A,0} \colon  \cC_{\A,0}^d \to \cC^d$ denote
the composite of the inverse of this isomorphism with
the functor $\iota_0^d$.

We introduce the following full subcategories $\cC_{\A,0,m}^d$ and
$\cC_{\A,0,r}^d$ of $\cC_{\A,0}^d$: the objects of $\cC_{\A,0,m}^d$ 
(\resp $\cC_{\A,0,r}^d$) are the objects $(\bL_1,\bL_2)$ of $\cC_{\A,0}^d$
satisfying $\bL_1 \subset \bL_2 \subset \wh{\cO}_X^{\oplus d}$ 
(\resp $\wh{\cO}_X^{\oplus d} \subset \bL_1 \subset \bL_2$).
Let $\iota_{\A,0,m}^d$ and $\iota_{\A,0,r}^d$ denote the restriction
of $\iota_{\A,0}^d$ to $\cC_{\A,0,m}^d$ and $\cC_{\A,0,r}^d$, respectively.

Since the isomorphism $T_\A$ induces isomorphisms 
$\cC_0^d \xto{\cong} \cC_{\A,0}^d$,
$\cC_{0,m}^d \xto{\cong} \cC_{\A,0,m}^d$,
and
$\cC_{0,r}^d \xto{\cong} \cC_{\A,0,r}^d$
of categories,
it follows from Theorem \ref{thm:section2} that
the pairs $(\cC_{\A,0}^d,\iota_{\A,0}^d)$,
$(\cC_{\A,0,m}^d,\iota_{\A,0,m}^d)$, and $(\cC_{\A,0,r}^d,\iota_{\A,0,r}^d)$
are grids of $(\cCo{d},J^d)$, $(\cCo{d},J_m^d)$, and $(\cCo{d},J_r^d)$,
respectively in the sense of \cite[Definition 5.5.3]{Grids}.

For an object $(L_1,L_2)$ of $\cC_0^d$,
the isomorphism in Lemma \ref{lem:XN} induces an
isomorphism $T_\A(L_2)/T_\A(L_1) \cong \Gamma(X,L_2/L_1)$ of
$\wh{\cO}_X$-modules, which is functorial
in the sense that it induces an isomorphism of functors
from $\cC_0^d$ to the category of $\wh{\cO}_X$-modules.
Here we regard $\Gamma(X,L_2/L_1)$ as an 
$\wh{\cO}_X$-module via the isomorphism 
$\Gamma(X,L_2/L_1) \cong \Gamma(X,{i_Z}_* i_Z^* (L_2/L_1))
= \Gamma(Z,i_Z^*(L_2/L_1))$ and the homomorphism
$\pi_Z\colon  \wh{\cO}_X \to \Gamma(Z,\cO_Z)$,
where $Z \subset X$ is an effective Cartier divisor such that
$L_2/L_1$ is annihilated by $\sI_Z$.

Let $(\cC,J)=(\cCo{d},J^d)$ 
(\resp $(\cCo{d},J^d_m)$, \resp $(\cCo{d},J^d_r)$).
Let $\Cip = (\cC_0^d,\iota_0^d)$
(\resp $(\cC_{0,m}^d,\iota_{0,m}^d)$, \resp 
$(\cC_{0,r}^d,\iota_{0,r}^d)$)
and let $(\cC_{\A,0},\iota_{\A,0}) = (\cC_{\A,0}^d,\iota_{\A,0}^d)$
(\resp $(\cC_{\A,0,m}^d,\iota_{\A,0,m}^d)$, \resp 
$(\cC_{\A,0,r}^d,\iota_{\A,0,r}^d)$).
Then the isomorphism $T_\A$ induces an isomorphism 
$M_\Cip \cong M_{(\cC_{\A,0},\iota_{\A,0})}$
of monoids.
In the following paragraphs we will compute the monoid 
$M_{(\cC_{\A,0},\iota_{\A,0})}$.

\subsubsection{ } \label{sec:hom_phi}
When $\cC=\cCo{d}$ (\resp $\cC_m^d$, \resp $\cC_r^d$),
we set $\bbM = \GL_d(\A_X)$ 
(\resp $\{g \in \GL_d(\A_X)\ |\ g^{-1} \in \Mat_d(\wh{\cO}_X) \}$,
$\GL_d(\A_X) \cap \Mat_d(\wh{\cO}_X)$).
Then $\bbM$ is a submonoid of $\GL_d(\A_X)$.

Let $(\bL,g) \in \Lat^d_\A \times \GL_d(\A_X)$.
It follows from Lemma \ref{lem:lattice_translation}
that $\bL g$ is an element of $\Lat^d_\A$.
The map $\Lat^d_\A \times \GL_d(\A_X)$ which sends
$(\bL,g)$ to $\bL g^{-1}$ induces an action from the left
of the group $\GL_d(\A_X)$ on the set $\Lat^d_\A$.
It is clear that for any $g \in \GL_d(\A_X)$, the
action of $g^{-1}$ on $\Lat^d_\A$ gives an automorphism
$\Lat^d_\A \to \Lat^d_\A$ of posets.
Hence it induces an automorphism
of the category $\cC_{\A,0}^d$.
Moreover if $g$ belongs to the monoid $\bbM$, then it induces
an automorphism of the full subcategory $\cC_{\A,0}$
of $\cC_{\A,0}^d$. We denote the last automorphism
by $\alpha_g\colon \cC_{\A,0} \to \cC_{\A,0}$.

\subsubsection{ }

We construct a map 
$\phi \colon  \bbM \to M_\Cip$ as follows.
Let $g \in \bbM$. Then for any object
$(\bL_1,\bL_2)$ of $\cC_{\A,0}$, the automorphism
$\bL_2 \xto{\cong} \bL_2 g^{-1}$ given by the right multiplication
by $g^{-1}$ induces an isomorphism
\begin{equation} \label{rightg}
\bL_2/\bL_1 \xto{\cong} \bL_2 g^{-1}/\bL_1 g^{-1}
\end{equation}
of $\wh{\cO}_X$-modules.
It follows from Lemma \ref{lem:XN} that we have
isomorphisms 
$$
\Gamma(X,\iota_{\A,0}(\bL_1,\bL_2))
\cong \bL_2/\bL_1
$$
and
$$
\Gamma(X,\iota_{\A,0}(\alpha_g(\bL_1,\bL_2)))
\cong \bL_2 g^{-1}/ \bL_1 g^{-1}.
$$
Hence the isomorphism \eqref{rightg} induces
an isomorphism 
$\iota_{\A,0}(\bL_1,\bL_2)
\cong \iota_{\A,0}(\alpha_g(\bL_1,\bL_2))$.
It is straightforward to check that this isomorphism
is functorial in $(\bL_1,\bL_2)$. We thus obtain
an isomorphism $\gamma_g\colon \iota_{\A,0} \to \iota_{\A,0} \circ \alpha_g$ 
of functors. By sending $g \in \bbM$ to the pair
$\phi(g) = (\alpha_g,\gamma_g)$ we obtain a map 
$\phi\colon \bbM \to M_{(\cC_{\A,0},\iota_{\A,0})}$.
It is easy to check that the map $\phi$ is a homomorphism
of monoids. 

We will prove that the homomorphism $\phi$ is an isomorphism.
First let us introduce the object 
$$
Y = (\wh{\cO}_X^{\oplus d}, \wh{\cO}_X^{\oplus d})
$$ of 
$\cC_{\A,0}^d$.

\begin{lem}
The object $Y$ is an edge object of $\cC_{\A,0}$. 
\end{lem}

\begin{proof}
The claim is clear when $\cC = \cC^d$. 
When $\cC = \cC^d_m$,
the claim follows from Lemma \ref{lem:explicit_edge}.
When $\cC = \cC^d_r$, it follows from Lemma \ref{lem:explicit_edge}
and the construction of the functor $\bD_0$ introduced in
Section \ref{sec:bD_0} that the object $(L_1,L_2)$ of the category
$\cC_{0,r}^d$ is an edge object if and only if $L_1= \cO_X^{\oplus d}$.
This implies that $Y$ is an edge object of $\cC_{\A,0}$ when
$\cC = \cC^d_r$.
\end{proof}

\begin{lem} \label{lem:phi_injective}
The homomorphism $\phi$ is injective.
\end{lem}

\begin{proof}
Let $g \in \bbM$ and suppose that $\phi(g)=(\alpha_g,\gamma_g)$ is 
the unit element of $M_{(\cC_{\A,0},\iota_{\A,0})}$.
Since $\alpha_g(Y) = Y$, we have 
$\wh{\cO}_X^{\oplus d} g^{-1} = \wh{\cO}_X^{\oplus d}$.
It follows that $g$ belongs to $\GL_d(\wh{\cO}_X)$.
Let $\I \subset \wh{\cO}_X$ be an arbitrary open ideal.
When $\cC=\cCo{d}$ (\resp $\cC_m^d$, \resp $\cC_r^d$),
we set $Y(\I) = (\I^{\oplus d}, \wh{\cO}_X^{\oplus d})$
(\resp $(\I^{\oplus d}, \wh{\cO}_X^{\oplus d})$,
\resp $(\wh{\cO}_X^{\oplus d}, (\I^{-1})^{\oplus d})$.
Then $Y(\I)$ is an object of $\cC_{\A,0}$.
Since $\phi(g)$ is 
the unit element of $M_{(\cC_{\A,0},\iota_{\A,0})}$,
we have $\alpha_g(Y(\I)) = Y(\I)$ and
the automorphism of $\iota_{\A,0}(Y(\I))$
given by $\gamma_g$ is equal to the identity.
Hence the automorphism $(\wh{\cO}_X/\I)^{\oplus d}$
(\resp $(\wh{\cO}_X/\I)^{\oplus d}$, \resp
$(\I^{-1}/\wh{\cO}_X)^{\oplus d}$) given by the
right multiplication by $g^{-1}$ is equal to the identity.
This shows that $g$ belongs to the kernel of the
map $\GL_d(\wh{\cO}_X) \to \GL_d(\wh{\cO}_X/\I)$ 
induced by the quotient homomorphism
$\wh{\cO}_X \surj \wh{\cO}_X/\I$.
Since $\I$ is arbitrary and the map
$\GL_d(\wh{\cO}_X) \to \varprojlim_{\I} 
\GL_d(\wh{\cO}_X/\I)$ is bijective, it follows that
$g$ is equal to the unit element of $\bbM$.
This proves that the map $\phi$ is injective.
\end{proof}

\begin{lem}
The group $\GL_d(\A_X)$ acts transitively on the set
$\Lat^d_\A$.
\end{lem}

\begin{proof}
Let $\bL \in \Lat^d_\A$. It follows from Lemma \ref{lem:freeness}
that $\bL$ is free of rank $d$ over $\wh{\cO}_X$.
Let us choose an $\wh{\cO}_X$-basis $b_1,\ldots,b_d$ of $\bL$.
Recall that  $\bL$ is an $\wh{\cO}_X$-submodule of $\A_X^{\oplus d}$.
Hence we may regard $b_1,\ldots,b_d$ as row vectors with coefficients
in $\A_X$. Let $h$ denote the $d \times d$ matrix whose $i$-th row
is equal to $b_i$ for $i=1,\ldots,d$.
It then follows from the second statement of Lemma \ref{lem:freeness}
that the matrix $h$ belongs to $\GL_d(\A_X)$.
We set $g = h^{-1}$.
By construction we have $\bL = \wh{\cO}_X^{\oplus d} g^{-1}$.
This proves that the action of $\GL_d(\A_X)$ on
$\Lat^d_\A$ is transitive.
\end{proof}

\begin{lem} \label{lem:phi_surjective}
The homomorphism $\phi$ is surjective.
\end{lem}

\begin{proof}
Let $(\alpha,\gamma_\alpha)$ be an element of $M_\Cip$.
Since $\iota(\alpha(Y))$ is isomorphic to $\iota(Y)$ in $\cC$,
the object $\alpha(Y)$ is of the form $(\bL,\bL)$ for some
$\bL \in \Lat^d_\A$. Moreover if $\cC = \cC_m^d$ (\resp $\cC = \cC_r^d$),
then we have $\bL \subset \wh{\cO}_X^{\oplus d}$ (\resp 
$\wh{\cO}_X^{\oplus d} \subset \bL$).
Since the group $\GL_d(\A_X)$ acts transitively on $\Lat^d_\A$,
there exists an element $g \in \GL_d(\A_X)$ such that
$\bL = \wh{\cO}_X^{\oplus d} g^{-1}$. 
If $\cC = \cC_m^d$ (\resp $\cC = \cC_r^d$),
then we have $g^{-1} \in \Mat_d(\wh{\cO}_X)$ 
(\resp $g \in \Mat_d(\wh{\cO}_X)$).
Hence in any case we have $g \in \bbM$.
Since $Y$ is an edge object, it follows from Lemma 8.1.5 of \cite{Grids}
that there exists an element $h \in \bK_Y$ satisfying 
$(\alpha,\gamma_\alpha) = \phi(g) h$.
Hence the claim follows from Lemma \ref{lem:image_phi}
below.
\end{proof}

\begin{lem}
\label{lem:image_phi}
Any element of $\bK_Y$ belongs to the image of $\phi$.
\end{lem}

\begin{proof}
Let $(\alpha,\gamma_\alpha) \in \bK_Y$.
Let $\I \subset \wh{\cO}_X$ be an arbitrary open ideal.
and let us consider the object $Y(\I)$ of $\cC_{\A,0}$.
Let $f_\I$ denote the unique morphism from $Y(\I)$
to $Y$ in $\cC_{\A,0}$.
Since $(\alpha,\gamma_\alpha) \in \bK_Y$, we have 
$\alpha(\wh{\cO}_X^{\oplus d})
= \wh{\cO}_X^{\oplus d}$. 
It follows from Lemma \ref{lem:Gal} 
that the morphism $\iota_{\A,0}(f)$ in $\cC$
is a Galois covering in $\cC$.
Since the diagram
$$
\begin{CD}
\iota_{\A,0}(Y(\I)) @>{\cong}>> \iota_{\A,0}(\alpha(Y(\I))) \\
@V{\iota_{\A,0}(f)}VV @VV{\iota_{\A,0}(\alpha(f))}V \\
\iota_{\A,0}(Y) @>{\id_{\iota_{\A,0}(Y)}}>> \iota_{\A,0}(Y)
\end{CD},
$$
where the upper horizontal arrow is the isomorphisms in $\cC$
given by $\gamma_\alpha$, is commutative, it follows from Corollary 7.1.2
of \cite{Grids} that $\alpha(Y(\I))=Y(\I)$ and
$\iota_{\A,0}(\alpha(f)) = \iota_{\A,0}(f)$.
Let $\alpha_\I$ denote the unique element of
$\Aut_{\cO_X}(\iota_{\A,0}(Y(\I))) $ such that
$z(\alpha)$ is equal to the automorphism
of $\iota_{\A,0}(Y(\I))$ given by $\gamma_\alpha$,
where $z$ is the isomorphism of groups introduced
in Section \ref{sec:notationz}.
It follows from the construction of the functor
$\iota_{\A,0}$ that $\Gamma_\A(\iota_{\A,0}(Y(\I)))$
is canonically isomorphic to $(\wh{\cO}_X/\I)^{\oplus d}$
or $(\I^{-1}/\wh{\cO}_X)^{\oplus d}$.
Hence the automorphism $\Gamma_\A(\alpha_\I)$ of the
$\wh{\cO}_X$-module $\Gamma_\A(\iota_{\A,0}(Y(\I)))$
can naturally be regarded as the right multiplication
by the inverse of an element of
$\GL_d(\wh{\cO}_X/\I)$.
Since $\gamma_\alpha$ is an isomorphism of functors,
it follows that the family
$(\Gamma_\A(\alpha_\I))_\I$, where $\I$ runs over the
open ideals of $\wh{\cO}_X$, gives an element of
the projective limit $\varprojlim_{\I} 
\GL_d(\wh{\cO}_X/\I)$. Since the map
$\GL_d(\wh{\cO}_X) \to \varprojlim_{\I} 
\GL_d(\wh{\cO}_X/\I)$ is bijective, there exists
an element $g \in \GL_d(\wh{\cO}_X)$ such that
for any open ideal $\I$ of $\wh{\cO}_X$,
the automorphism $\Gamma_\A(\alpha_\I)$ is given by
the right multiplication by the inverse of
the image of $g$ in $\GL_d(\wh{\cO}_X/\I)$.
It is then straightforward to check that
the element $(\alpha,\gamma_\alpha)$ of
$M_{\cC_{\A,0},\iota_{\A,0}}$ is equal to $\phi(g)$.
This proves the claim.
\end{proof}

\subsubsection{ }
We endow the group $\GL_d(\A_X)$ with the
coarsest topology such that the two subsets
$$
U^1_{K,P} = \{ g\in \GL_d(\A_X)\ |\ K g \subset P \}
$$
and
$$
U^2_{K,P} = \{ g\in \GL_d(\A_X)\ |\ K g^{-1} \subset P \}
$$
of $\GL_d(\A_X)$ are open for
any compact subset $K \subset \A_X^{\oplus d}$ 
and for any open subset $P \subset \A_X^{\oplus d}$.

When $X$ is the spectrum of the ring of integers
of a finite extension $F$ of $\Q$, this topology on 
$\GL_d(\A_X)$ coincides with the usual topology 
on the group $\GL_d(\A_X)$.

\begin{lem}\label{lem:continuous}
\begin{enumerate}
\item The map $\GL_d(\A_X) \to \GL_d(\A_X)$
that sends $g$ to $g^{-1}$ is a homeomorphism.
\item If $U\subset \GL_d(\A_X)$ is an open subset,
then $Ug$ and $gU$ are open subsets of $\GL_d(\A_X)$
for any $g \in \GL_d(\A_X)$.
\item The map $m\colon  \A_X^{\oplus d} \times \GL_d(\A_X)
\to \A_X^{\oplus d}$ thath sends $(x,g)$ to $xg$
is continuous.
\end{enumerate}
\end{lem}

\begin{proof}
The map $\GL_d(\A_X) \to \GL_d(\A_X)$
which sends $g$ to $g^{-1}$
is bijective and sends $U^1_{K,P}$, $U^2_{K,P}$
to $U^2_{K,P}$, $U^1_{K,P}$, respectively.
Hence the claim (1) follows.

For any $g \in \GL_d(\A_X)$, we have
$g U^1_{K,P} = U^1_{K g^{-1},P}$,
$g U^2_{K,P} = U^2_{K, P g^{-1}}$,
$U^1_{K,P} g = U^1_{K,Pg}$,
and $U^2_{K,P} g = U^2_{Kg,P}$.
Hence the claim (2) follows from 
Corollary \ref{cor:translation}.

Let $P \subset \A_X^{\oplus d}$ be an open set.
Let $(x,g) \in \A_X^{\oplus d} \times \GL_d(\A_X)$
be a pair satisfying $xg \in P$.
Then it follows from Corollary \ref{cor:translation}
that $P g^{-1}$ is an open neighborhood of
$x$ in $\A_X^{\oplus d}$.
Since any open subgroup of $\wh{\cO}_X^{\oplus d}$ 
is compact, it follows from the definition of the
topology on $\A_X$ that there exists a lattice 
$\bL \in \Lat^d_\A$ satisfying $x + \bL \subset P g^{-1}$.
Hence $(x+\bL) \times U^1_{x+\bL,P}$
is an open neighborhood of $(x,g)$ contained
in $m^{-1}(P)$.
This shows that $m^{-1}(P)$ is an open
subset of $\A_X^{\oplus d} \times \GL_d(\A_X)$.
This proves the claim (3).
\end{proof}

For an element $(\bL_1,\bL_2) \in \Pair^d_\A$,
we let $\bK_{\bL_1,\bL_2} \subset \GL_d(\A_X)$ 
denote the subgroup of the elements 
$g\in \GL_d(\A_X)$ such that
$\bL_i g = \bL_i$ for $i=1,2$ and
the map induced by $g$ on $\bL_2/\bL_1$ is the identity.

\begin{lem}\label{lem:basic_bK}
\begin{enumerate}
\item For any $(\bL_1,\bL_2) \in \Pair^d_\A$,
the group  $\bK_{\bL_1,\bL_2}$ is an open
subgroup of $\GL_d(\A_X)$.
\item For two elements $(\bL_1,\bL_2), (\bL'_1,\bL'_2) \in \Pair^d_\A$
with $(\bL_1,\bL_2) \le (\bL'_1,\bL'_2)$, the group
$\bK_{\bL'_1,\bL'_2}$ is a subgroup of $\bK_{\bL_1,\bL_2}$
of finite index.
\item For $(\bL_1,\bL_2) \in \Pair^d_\A$ and
for $g \in \GL_d(\A_X)$, we have
$g^{-1} \bK_{\bL_1,\bL_2} g = \bK_{\bL_1 g, \bL_2 g}$.
\item For any finite number of
elements $(\bL_{1,1},\bL_{1,2}), (\bL_{2,1},\bL_{2,2}),
\ldots, (\bL_{r,1},\bL_{r,2}) \in \Pair^d_\A$,
there exists an element $(\bL'_1,\bL'_2) \in  \Pair^d_\A$
satisfying $\bK_{\bL'_1,\bL'_2}
\subset \bigcap_{i=1}^r \bK_{\bL_{i,1},\bL_{i,2}}$.
\item For any open subset 
$U \subset \GL_d(\A_X)$ and for any
$g \in U$, there exists an element
$(\bL_1,\bL_2) \in \Pair^d_\A$ satisfying
$\bK_{\bL_1,\bL_2} g \subset U$ and
$g \bK_{\bL_1,\bL_2} \subset U$.
\end{enumerate}
\end{lem}

\begin{proof}
Let us choose a complete set $S \subset \bL_2$
of representatives of $\bL_2/\bL_1$.
It follows from Lemma \ref{lem:finite_index}
that $S$ is a finite set.
We note that $s + \bL_1$ is a compact open
subset of $\A_X^{\oplus d}$ for each $s \in S$.
By definition, the group 
$\bK_{\bL_1,\bL_2}$ is equal to the intersection
$$
\bK_{\bL_1,\bL_2} =
\bigcap_{s \in S} 
(U^1_{s+\bL_1,s+\bL_1}
\cap U^2_{s+\bL_1,s+\bL_1}).
$$
Hence $\bK_{\bL_1,\bL_2}$ is an open subset of $\GL_d(\A_X)$.
This proves the claim (1).

Let $(\bL_1,\bL_2), (\bL'_1,\bL'_2) \in \Pair^d_\A$ be
two elements with $(\bL_1,\bL_2) \le (\bL'_1,\bL'_2)$.
Let $g \in \bK_{\bL'_1,\bL'_2}$.
Since $\bL_i$ is the inverse image of 
$\bL_i/\bL'_1$ under the surjection 
$\bL'_2 \surj \bL'_2/\bL'_1$
for $i=1,2$, it follows from the definition
of $\bK_{\bL'_1,\bL'_2}$ that we have 
$\bL_1 g = \bL_1$ and $\bL_2 g = \bL_2$.
Since the right multiplication by $g$ induces
the identity map on $\bL'_2/\bL'_1$,
it also induces the identity map on
$\bL_2/\bL_1$. Hence we have $g \in \bK_{\bL_1,\bL_2}$.
This proves that $\bK_{\bL'_1,\bL'_2}
\subset \bK_{\bL_1,\bL_2}$.
It follows from Lemma \ref{lem:I1I2} that
there exist open ideals 
$\I_1,\I_2 \subset \wh{\cO}_X$
satisfying $\I_1 \bL_2 \subset \bL'_1 \subset \bL'_2
\subset \I_2^{-1} \bL_2$.
We let $C'$
denote the set of pairs $((\bL''_1,\bL''_2),\alpha)$
of an element $(\bL''_1,\bL''_2) \in \Pair^d_\A$
with $(\bL''_1,\bL''_2) \le (\I_1 \bL_2, \I_2^{-1} \bL_2)$
and an isomorphism $\alpha \colon 
\bL''_2/\bL''_1 \xto{\cong} \bL'_2/\bL'_1$ of 
$\wh{\cO}_X$-modules.
It follows from Lemma \ref{lem:finite_index}
that there exist only finitely many 
$\wh{\cO}_X$-lattices $\bL \in \Lat^d_\A$ satisfying
$\I_1 \bL_2 \subset \bL \subset \I_2^{-1} \bL_2$ and that
the group $\Aut_{\wh{\cO}_X}(\bL'_2/\bL'_1)$
is a finite group.
Hence the set $C'$ is a finite set.
The group $\bK_{\bL_1,\bL_2}$ acts on the set
$C'$ by setting $g \cdot ((\bL''_1,\bL''_2),\alpha)
= ((\bL''_1 g^{-1}, \bL''_2 g^{-1}),\alpha \circ \beta_g)$,
where $\beta_g \colon  \bL''_2 g^{-1}/\bL''_1 g^{-1}
\xto{\cong} \bL''_2/\bL''_1$ is the homomorphism
of $\wh{\cO}_X$-modules induced by the right
multiplication by $g$.
It can be checked easily that 
the subgroup $\bK_{\bL'_1,\bL'_2}\subset \bK_{\bL_1,\bL_2}$ 
is equal to the stabilizer of the element
$((\bL'_1,\bL'_2),\id_{\bL'_2/\bL'_1})$.
Hence $\bK_{\bL'_1,\bL'_2}$ is a subgroup of
$\bK_{\bL_1,\bL_2}$ of finite index.
This proves the claim (2).

The claim (3) is clear from the definition
of $\bK_{\bL_1 g, \bL_2 g}$.

The claim (4) follows from the claim (2)
since it follows from Lemma \ref{lem:pair_directed_V}
that the poset $\Pair^d_\A$ is filtered.

We prove the claim (5).
By definition of the topology on $\GL_d(\A_X)$,
there exists an open neighborhood $U'$ of $g$
such that $U' \subset U$ and that
$U'$ is an intersection of finitely many
subsets of $\GL_d(\A_X)$, each of which
is either of the form $U^1_{K,P}$ or of the
form $U^2_{K,P}$ for some compact subset
$K \subset \GL_d(\A_X)$ and for 
some open subset $P \subset \GL_d(\A_X)$.
Hence it follows from the claim (4) that,
to prove the claim (5), we may assume that
$U = U^1_{K,P}$ or $U = U^2_{K,P}$.
First suppose that $U=U^1_{K,P}$.
Since any open subgroup of $\wh{\cO}_X^{\oplus d}$ 
is compact, it follows from the definition of the
topology on $\A_X$ that for any $y \in U$,
there exists a lattice $\bL_y \in \Lat^d_\A$
satisfying $y + \bL_y \subset U$.
Let $g \in U^1_{K,P}$.
Since $K \subset \bigcup_{x \in K} (x + \bL_{x g} g^{-1})$
and $K$ is compact, there exists a finite subset
$S \subset K$ satisfying $K \subset 
\bigcup_{s \in S} (s + \bL_{s g}g^{-1})$.
We set $\bL'_1 = \bigcap_{s \in S} (\bL_{s g} g^{-1})$
and let $\bL'_2$ be the $\wh{\cO}_X$-submodule
of $\A_X^{\oplus d}$ generated by 
$\sum_{s \in S} (\bL_{s g} g^{-1})$ and $S$.
Then we have $(\bL'_1,\bL'_2) \in \Pair^d_\A$.
Let $k \in \bK_{\bL'_1,\bL'_2}$.
It follows from the definition of
the pair $(\bL'_1,\bL'_2)$ that 
$(s + \bL_{s g} g^{-1}) k =s + \bL_{s g} g^{-1}$
for any $s \in S$.
Hence we have $K k g \subset
\bigcup_{s \in S} (sg + \bL_{s g})
\subset P$. This shows that
$\bK_{\bL'_1,\bL'_2} g \subset U^1_{K,P}$.
By the claim (3), we have 
$g^{-1} \bK_{\bL'_1,\bL'_2} g = \bK_{\bL'_1 g, \bL'_2 g}$.
Hence $g \bK_{\bL'_1 g, \bL'_2 g} \subset U^1_{K,P}$.
It follows from the claim (4) that
there exists an element $(\bL_1,\bL_2) \in \Pair^d_\A$
satisfying $\bK_{\bL_1,\bL_2} \subset \bK_{\bL'_1,\bL'_2}
\cap \bK_{\bL'_1 g, \bL'_2 g}$. We then have
$\bK_{\bL_1,\bL_2} g \subset U^1_{K,P}$
and $g \bK_{\bL_1,\bL_2} \subset U^1_{K,P}$.
Next suppose that $U = U^2_{K,P}$.
Let $g \in U^2_{K,P}$.
Since $g^{-1} \in U^1_{K,P}$, there
exists an element $(\bL_1,\bL_2) \in \Pair^d_\A$
satisfying $\bK_{\bL_1,\bL_2} g^{-1} \subset U^1_{K,P}$
and $g^{-1} \bK_{\bL_1,\bL_2} \subset U^1_{K,P}$.
Hence we have 
$g \bK_{\bL_1,\bL_2} \subset U^2_{K,P}$
and $\bK_{\bL_1,\bL_2} g \subset U^2_{K,P}$.
This proves the claim (5).
\end{proof}

\begin{cor}\label{cor:bK_continuous}
\begin{enumerate}
\item The set $\{\bK_{\bL_1,\bL_2}\ |\ 
(\bL_1,\bL_2) \in \Pair^d_\A \}$ forms a
fundamental system of neighborhoods of $1$
in $\GL_d(\A_X)$.
\item The multiplication map 
$\GL_d(\A_X) \times \GL_d(\A_X) \to \GL_d(\A_X)$
is continuous.
\end{enumerate}
\end{cor}

\begin{proof}
The claim (1) is an immediate consequence of
Lemma \ref{lem:basic_bK} (5).

Let $U \subset \GL_d(\A_X)$ be an open set
and let $g_1,g_2 \in \GL_d(\A_X)$ be elements
satisfying $g_1 g_2 \in U$.
It follows from Lemma \ref{lem:basic_bK} (5)
that there exists $(\bL'_1,\bL'_2) \in \Pair^d_\A$
such that $g_1 g_2 \bK_{\bL'_1,\bL'_2} \subset U$.
We set $(\bL_1,\bL_2) = (\bL'_1 g_2^{-1}, \bL'_2 g_2^{-1})$.
Then $g_1 \bK_{\bL_1,\bL_2} g_2 \subset U$.
It follows from Lemma \ref{lem:continuous} (2)
that $g_1 \bK_{\bL_1,\bL_2} \times
\bK_{\bL_1,\bL_2} g_2$ is an open subset of
$\GL_d(\A_X) \times \GL_d(\A_X)$
containing $(g_1,g_2)$ whose image under
the multiplication map is contained in $U$.
This shows that the multiplication map is
continuous. This proves the claim (2).
\end{proof}

\begin{thm} \label{cor:M_isom}
The map $\phi\colon  \bbM \to M_{(\cC_{\A,0},\iota_{\A,0})}$ is 
an isomorphism of topological monoids. Here we
regard $M_{(\cC_{\A,0},\iota_{\A,0})}$ as a topological
monoid as in Corollary 11.1.3 of \cite{Grids}.
\end{thm}

\begin{proof}
It follows from Lemma \ref{lem:phi_injective} and 
Lemma \ref{lem:phi_surjective} that
$\phi$ is a bijective homomorphism of monoids.
It follows from Corollary \ref{cor:bK_continuous} (1)
that the map $\phi$ is a homeomorphism. Hence the claim follows.
\end{proof}

For a presheaf $F$ of sets on $\cC$, we set
$$
\omega_\Cip(F)
= \varinjlim_{(L_1,L_2) \in \Obj \cC_{0}} F(L_2/L_1).
$$
As is explained in Section 5.7.3 of \cite{Grids},
the set $\omega_\Cip(F)$ has a natural structure 
of a left $M_\Cip$-set. Via the isomorphisms
$$
M_\Cip \cong M_{(\cC_{\A,0},\iota_{\A,0})} 
\xleftarrow[\cong]{\phi} \bbM
$$
of topological monoids, we regard $\omega_\Cip(F)$
as a left $\bbM$-set.
By associating $\omega_\Cip(F)$ to each presheaf
$F$ of sets on $\cC$, we obtain a functor $\omega_\Cip$
from the category of presheaves on $\cC$ to the category of
left $\bbM$-sets.

\begin{cor} \label{cor:main_section3}
The restriction of the functor $\omega_\Cip$
to the category $\Shv(\cC,J)$ of sheaves on $(\cC,J)$
gives an equivalence from the category $\Shv(\cC,J)$ to
the category of smooth left $\bbM$-sets. Here a left $\bbM$-set
$S$ is called smooth if for any $s \in S$, there is an open
subgroup $\bK \subset \bbM$ such that $gs =s$ holds for
every $g \in \bK$.
\end{cor}

\begin{proof}
This follows from Theorem~\ref{cor:M_isom} and 
Theorem 5.8.1 of \cite{Grids}.
\end{proof}

\subsection{}
We record the following for later use.
\begin{prop}
\label{prop:old cat equiv}
Let $(L_1, L_2)$ be an object of $\Pair^d$.
Let $H$ be a subgroup of 
$\Aut_{\cO_X}(L_2/L_1)$.
Let $\bK_{L_1, L_2, H} \subset
\GL_d(\A_X)$
denote the 
compact open subgroup of the 
elements $g \in \GL_d(\A_X)$
such that 
$L_i g=L_i$ for $i=1,2$ and the action
of $g$ on $L_2/L_1$ lies in $H$.
Then the following assertions hold.
\begin{enumerate}
\item 
There is a unique $\GL_d(\A_X)$-equivariant
isomorphism
$\omega(H \backslash (L_2/L_1)) 
\cong \GL_d(\A_X)/\bK_{L_1,L_2, H}$
that sends the element in 
$\omega(H \backslash(L_2/L_1))$
represented by the element $\id_{L_2/L_1}$
in $H \backslash \Hom(L_2/L_1, L_2/L_1)$
to the class of the identity matrix in 
$\GL_d(\A_X)/\bK_{L_1, L_2, H}$.
\item
For any presheaf $F$ on $\cFC^d$,
the map
$
\Hom(H \backslash (L_2/L_1), F^a)
\to F^a(L_2/L_1) \to
\omega(F^a)
\cong \omega(F)$
induces an isomorphism
$\Hom(H\backslash (L_2/L_1), F^a)
\cong \omega(F)^{\bK_{L_1,L_2, H}}$.
\end{enumerate}
\end{prop}
\begin{proof}
First, we prove the case where $H$ is a singleton.
Let $H_{L_1/L_2}$ be the group introduced in 
Section 7.6.1 of \cite{Grids} for $X=L_1/L_2$.
One can compute directly that $H_{L_1/L_2}$ is
isomorphic to $\bK_{L_1, L_2}$.
Then \cite[Cor. 7.6.3, Lem. 7.6.2]{Grids}
imply the assertions by unfolding the definitions 
of the functor $\omega$ there.

Let $h \in H$.  Then by \cite[Lem 7.7.1]{Grids}
(or one can show directly),
there exists an element $g$ of $\GL_d(\A_X)$
such that $g \in \bK_{L_1, L_2, H}$ inducing an action
$h$ on $L_2/L_1$.   
Then the sheaf 
$H\backslash (L_2/L_1)$ 
corresponds via the functor $\omega$
to $\GL_d(\A_X)/\bK_{L_1,L_2,H}$
as quotients of $L_2/L_1$
and $\GL_d(\A_X)/\bK_{L_1, L_2}$
via the isomorphism above.

The rest of the assertions follow from the 
construction of the functor $\omega$.
\end{proof}

\section{An Euler system of distributions}\label{sec:Euler distributions}
In this section, we prove our main (abstract) theorem
 (Theorem~\ref{main theorem}).
This theorem is a realization of the slogan 
``tensor product of $d$ distributions is an Euler system for $\GL_d$''.   
An application in the context of Drinfeld modular schemes is given in 
Theorem~\ref{thm:Drinfeld Euler}.

The assumptions or the setup for the theorem (namely, 
Situations I and II)
may look complicated.   
The reason for this complication is that in 
the Drinfeld modular setup, 
the theta functions are defined only up to 
$(q_\infty^d-1)$-st roots of unity.
Things would be a lot easier and become simpler 
if we are to deal only with the elements that are
$(q_\infty^d-1)^{d-1}$ times what we consider.
Then the simpler setup explained 
in~\ref{sec:intro Euler tech} applies, but the result
is less sharp.

The result of this section for the case when $d=1$
may also be applied to proving the norm compatibility 
of cyclotomic units (see Section~\ref{sec:cyclotomic Euler}). 
The case when $d=2$ is closely related to the Euler system of Kato
(but in the motivic cohomology rather than in K-theory).
We warn that the normalization is different from \cite{Kato}; we choose
$(m_{*,*})_*$ (see text for the notation) for our norm maps, while
Kato chooses $(r_{*,*})_*$ as his norm maps.  

\subsection{Distributions}

\subsubsection{}
Let $d \ge 1$ be an integer. Let $X$ be as in Section \ref{sec:4},
i.e., $X$ is a regular noetherian scheme of pure Krull dimension one
such that the residue field at each closed point is finite.
Let us consider the category $\cC^d$ introduced in Section \ref{sec:Cd_defn}
and the Grothendieck topology $J^d_m$ on $\cC^d$
introduced in the statement of Theorem \ref{thm:section2}.
It follows from Theorem \ref{thm:section2} that 
$(\cC^d,J^d_m)$ is a $Y$-site.
We denote by $\cCot{d}_m$, $\wt{\cC}^{d,\mu}_m$ and
$\cFCot{d}_m$ the categories $\wt{\cC}$, $\wt{\cC}^\mu$ and 
$\wt{\cFC}$ introduced in Sections \ref{sec:Ctil}
and \ref{sec:wtcFC} for $(\cC,J) = (\cC^d,J^d_m)$.
We say that a morphism $f$ in $\cCot{d}_m$ (\resp in $\wt{\cC}^{d,\mu}_m$, 
\resp in $\cFCot{d}_m$)
is a fibration if it belongs to $\wt{\cT}$ (\resp $\wt{\cT}^\mu$ 
\resp $\wt{\cFT}^*$) introduced in Section \ref{sec:cT}
(\resp Section \ref{sec:cTmu}, \resp Section \ref{sec:cFT*})
for $(\cC,J) = (\cC^d,J^d_m)$.

\subsubsection{ }
We define two special abelian presheaves $\BS'=\BS'_d$,
and ${\BS^*}'={\BS^*}'_d$ on $\cCo{d}$.
For an object $N$ in $\cCo{d}$, let $\BS'(N)$ (\resp ${\BS^*}'(N)$)
be the free abelian group generated by the set $\Gamma(X,N)$
of global sections of $N$
(\resp $\Gamma(X,N) \setminus \{0\}$).
For a morphism $N' \stackrel{p}{\twoheadleftarrow} N''
\stackrel{i}{\hookrightarrow} N$
from $N$ to $N'$ in $\cCo{d}$, define a homomorphism
$\BS'(N') \to \BS'(N)$ (\resp ${\BS^*}'(N') \to {\BS^*}'(N)$)
by sending $x \in \Gamma(X,N')$ 
(\resp $x \in \Gamma(X,N')\setminus\{0\}$)
to the element $\sum_{p(y)=x}i(y)$. 
%
Let us consider the sheaves $(\BS')^a$ and 
$({\BS^*}')^a$ 
on $(\cC^d,J^d_m)$
associated with the presheaves $\BS'$ 
and ${\BS^*}'$, respectively
Let us consider the functor $\iota' \colon  \cC^d \to \wt{\cFC}^d_m$
introduced in Section \ref{sec:iota'} for $(\cC,J) = (\cC^d,J^d_m)$.
By applying the functor $\nu\colon \Shv(\cC,J) \to \Shv(\wt{\cFC}^d_m,\iota'_*J)$ 
in Section \ref{sec:iota'}
we obtain the sheaves $\nu((\BS')^a)$ and 
$\nu(({\BS^*}')^a)$ on $\cFCotu{d}$.
We denote the sheaves $\nu((\BS')^a)$ and 
$\nu(({\BS^*}')^a)$ on $\cFCotu{d}$
by $\BS$ and $\BS^{*}$,
respectively and call
them the Schwartz-Bruhat sheaf (of rank $d$) and
the punctured Schwartz-Bruhat sheaf (of rank $d$)
respectively.

\subsubsection{ }
\label{sec:BS functions}
Let $(\cC_0,\iota_0) = (\cC_{0,m}^d,\iota^d_{0,m})$, where
the notation is as in Section \ref{sec:Pair}.
Set $\omega = \omega_{(\cC_0,\iota_0)}$.
It follows from Corollary \ref{cor:5563} and
Corollary \ref{cor:main_section3} that the composite
$\omega \circ \iota'^*$ gives an equivalence from the category of
abelian sheaves on $\cFCot{d}$ to the category of
smooth left $\Z[\bbM]$-module, where
$$
\bbM = \{ g \in \GL_d(\A_X)\ |\ g^{-1} \in \Mat_d(\wh{\cO}_X) \}.
$$
Let us compute the smooth left $\Z[\bbM]$-modules
$\omega\circ \iota'^* (\BS)$ and $\omega \circ \iota'^* (\BS^*)$.
As we discussed in Section \ref{sec:iota'}, the functor
$\nu$ is a quasi-inverse to the functor $\iota'^*$.
Hence it follows from Lemma 10.1.2 of \cite{Grids} that
for a presheaf $F$ on $\cCo{d}$, we have an
$\bbM$-equivariant isomorphism
$\omega(F) \xto{\cong} \omega \circ \iota'^* (\nu(F^a))$ 
which is functorial in $F$.
In particular we have isomorphisms
$\omega(\BS') \cong \omega\circ \iota^* (\BS)$ 
and $\omega({\BS^*}') \cong \omega \circ \iota^* (\BS^*)$
of left $\Z[\bbM]$-modules.

Let $\cS(\wh{\cO}_X^{\oplus d})$
denote the set of locally constant, compactly supported
$\Z$-valued functions on $\wh{\cO}_X^{\oplus d}$.
We regard $\cS(\wh{\cO}_X^{\oplus d})$ as the abelian group
whose addition is defined pointwisely.
\begin{lem}\label{lem:SB}
For any $f \in \cS(\A_X^{\oplus d})$,
there exists an $\wh{\cO}_X$-lattice $\bL \subset \wh{\cO}_X^{\oplus d}$ 
of $\A_X^{\oplus d}$ such that $f$ 
factors through the quotient map $\wh{\cO}_X^{\oplus d} 
\surj \wh{\cO}_X^{\oplus d}/\bL$.
\end{lem}

\begin{proof}
Let $x \in \wh{\cO}_X^{\oplus d}$.
Since $f$ is locally constant, it follows from 
Lemma~\ref{lem:lattice_basics}~(1) 
that there exists
a lattice $\bL_x \in \Lat^d_\A$ with 
$\bL_x \subset \wh{\cO}_X^{\oplus d}$ 
such that $f$ is constant
on 
$x + \bL_x$. Since 
$\wh{\cO}_X^{\oplus d}$ is compact, there exists a
finite subset $S \subset \wh{\cO}_X^{\oplus d}$ 
satisfying $\wh{\cO}_X^{\oplus d} = \bigcup_{s \in S}
(s + \bL_s)$.
We set $\bL =\bigcap_{s \in S} \bL_s$.
It follows from Lemma~\ref{lem:lattices} that $\bL \in \Lat^d_\A$.
It is clear from the construction that $\bL$
satisfies the desired property.
This completes the proof.
\end{proof}

For a lattice $\bL \in \Lat^d_\A$ with 
$\bL \subset \wh{\cO}_X^{\oplus d}$,
let $\cS(\wh{\cO}_X^{\oplus d})_{\bL}
\subset \cS(\wh{\cO}_X^{\oplus d})$ denote the
subset of functions $f$ such that
$f$ factors through the quotient map $\wh{\cO}_X^{\oplus d} 
\surj \wh{\cO}_X^{\oplus d}/\bL$.
For each $\bL \in \Lat^d_\A$ with 
$\bL \subset \wh{\cO}_X^{\oplus d}$,
$$
\xi_{\bL} \colon \cS(\A_X^{\oplus d})_{\bL} \to \BS(\wh{\cO}_X^{\oplus d}/\bL)
$$
denote the following map. For $f \in \cS(\wh{\cO}_X^{\oplus d})_{\bL}$,
let $f'\colon \wh{\cO}_X^{\oplus d}/\bL \to \Z$ denote the map induced by $f$.
Then we set $\xi_{\bL}(f) = \sum_{y \in \wh{\cO}_X^{\oplus d}/\bL} f'(y) y
\in \BS(\wh{\cO}_X^{\oplus d}/\bL)$.
It can be checked easily that the map $\xi_{\bL}$
is bijective and that for two lattices $\bL$, $\bL'$
with $\bL' \subset \bL \subset \wh{\cO}_X^{\oplus d}$, 
the diagram
$$
\begin{CD}
\cS(\wh{\cO}_X^{\oplus d})_{\bL} @>{\subset}>>
\cS(\wh{\cO}_X^{\oplus d})_{\bL'} \\
@V{\xi_{\bL}}VV @V{\xi_{\bL'}}VV \\
\BS'(\wh{\cO}_X^{\oplus d}/\bL_1) @>{\BS'(\iota_{\A,0}^d(g))}>> 
\BS'(\wh{\cO}_X^{\oplus d}/\bL')
\end{CD}
$$
is commutative. Here $g\colon (\wh{\cO}_X^{\oplus d},\bL') 
\to (\wh{\cO}_X^{\oplus d},\bL)$ denotes
the unique morphism in $\cC^d_{0,m}$.
It follows from Lemma \ref{lem:SB} that
$\cS(\wh{\cO}_X^{\oplus d})$ is equal to the union
$\cS(\wh{\cO}_X^{\oplus d}) = 
\bigcup_{\bL \subset \wh{\cO}_X^{\oplus d}}
\cS(\wh{\cO}_X^{\oplus d})_{\bL}$.
Hence by taking the colimit of the bijections $\xi_{\bL}$ 
with respect to $\bL$, we obtain a bijection
$$
\xi\colon  \cS(\wh{\cO}_X^{\oplus d}) \xto{\cong} \omega(\BS').
$$

Let us regard an element $f \in \cS(\wh{\cO}_X^{\oplus d})$
as a function on $\A_X^{\oplus d}$ which is identically 
zero outside $\wh{\cO}_X^{\oplus d}$.
For $f \in \cS(\wh{\cO}_X^{\oplus d})$ and for
$g \in \bbM$, we let $gf$ denote the
function on $\wh{\cO}_X^{\oplus d}$ whose value at
$x \in \wh{\cO}_X^{\oplus d}$ is equal to $f(xg)$.
It follows from Corollary \ref{cor:translation}
that $gf$ belongs to $\cS(\wh{\cO}_X^{\oplus d})$.
The map $\bbM \times \cS(\wh{\cO}_X^{\oplus d})
\to \cS(\wh{\cO}_X^{\oplus d})$ which sends $(g,f)$ to $gf$
gives $\cS(\wh{\cO}_X^{\oplus d})$ a structure of
a left $\Z[\bbM]$-module.
It then follows from the definition of the action of
$\bbM$ on $\omega(\BS')$ that the bijection
$\xi\colon  \cS(\wh{\cO}_X^{\oplus d}) \xto{\cong} \omega(\BS')$
is an isomorphism of left $\Z[\bbM]$-modules.

Since ${\BS^*}'$ is a subpresheaf of $\BS'$,
it follows from the construction of $\omega$ that
the $\bbM$-equivariant map 
$\omega({\BS^*}') \to \omega(\BS')$ induced
by the inclusion ${\BS^*}' \inj \BS'$ is injective.
It can be checked easily that the image of
the map $\omega({\BS^*}') \to \omega(\BS')$
is equal to the image under $\xi$ of
the submodule of $\cS(\wh{\cO}_X^{\oplus d})$
of the functions $f$ with $f(0)=0$.

\begin{defn}
Let $F$ be an abelian sheaf 
on $\cFCotu{d}$.
A {\it{distribution}}
(\resp {\it{punctured distribution}})
with values in $F$ is a homomorphism
$\BS \to F$ (\resp $\BS^{*}\to F$ )
of abelian sheaves
on $\cFCotu{d}$.
\end{defn}

\subsection{Construction of Euler systems}

\subsubsection{}
\label{sec:Situations}

Let $G$ be a presheaf of rings with transfers 
on $\cFCotu{d}$. Suppose for each $i=1,\ldots,d$, the
following data are given.
\begin{enumerate}
\item An abelian sheaf $F_i$ on $\cFCotu{d}$,
and an abelian subpresheaf $F'_i \subset F_i$ equipped with a structure
of an abelian presheaf with transfers such that the inclusion
$F'_i \subset F_i$ is a homomorphism of abelian presheaves with transfers.
\item A homomorphism $\alpha_i\colon F'_i\to G$ of abelian presheaves 
with transfers.
\item An object $N_i$ in $\cCo{1}$ and a 
generator $b_i\in \Gamma(X,N_i)$ as 
$\wh{\cO}_X$-module.
\item A quotient $\cO_X$-module $N'_i$
of $N_i$.
\end{enumerate}
Let $b'_i$ denote
the image of $b_i$ in $N'_i$.
We write $\bN=\bigoplus_{j=1}^d N_j$, 
$\bN' =\bigoplus_{j=1}^d N'_j$
and $N''_i=\Ker(N_i\surj N'_i)$.

We consider the following two settings which 
we call Situation~I and Situation~II.
\begin{description}
\item[Situation I] we have a morphism 
$g'_i \colon  \BS' \to \iota'^* F'_i$ of 
abelian presheaves for each $i$.
Let us consider the composite 
$\BS' \to \iota'^* F'_i \subset \iota'^* F_i$.
The adjunction property of the sheafification functor
$(-)^a$ induces the morphism
$(\BS')^a \to \iota'^* F_i$.
By applying $\nu$ in Section \ref{sec:iota'},
we obtain a distribution $g_i\colon \BS \to F_i$ with 
values in $F_i$.   This morphism $g_i$ 
factors through $F_i'$ as $\BS \to F_i' \to F_i$.
\item[Situation II] we assume $N'_i\neq 0$
and we have 
a morphism $g'_i \colon  {\BS^*}' \to \iota'^* F'_i$ of 
abelian presheaves for each $i$.
Let us consider the composite ${\BS^*}' \to \iota'^* F'_i 
\subset \iota'^* F_i$.
The adjunction property of the sheafification functor
$(-)^a$ induces the morphism
$({\BS^*}')^a \to \iota'^* F_i$.
By applying $\nu$ in Section \ref{sec:iota'},
we obtain a punctured distribution $g_i\colon \BS^* \to F_i$ with 
values in $F_i$.
This morphism $g_i$ factors through $F_i'$ as
$\BS^* \to F_i' \to F_i$.
\end{description}

We set
\[
\kappa_{\bN, (b_j)}=
\prod_{j=1}^{d}
\pr_j^*(\alpha_j g'_j(N_j)([b_j]))
\in G(\quotid{\bN})
\]
where $\pr_j$ $(j=1,\dots, d)$ is the morphism
$N_j = N_j \inj \bN$
from $\bN$ to $N_j$ 
given by the inclusion of $N_j$ into the $j$-th factor
of $\bN$.  The product notation $\prod_{i=1}^d a_i$ 
means the product $a_1 \cdots a_d$ taken in this order.

In the two settings above, if $F'_i$ is a sheaf, 
then we may take $F_i=F'_i$.  This would make the 
exposition much clearer, but in application 
we have a case where $F'_i$ may not be a sheaf.
\begin{thm}\label{main theorem}
Suppose that we are either in Situation
I or in II.

Then the following statements hold.
\begin{enumerate}
\item
If $\Supp(N''_i) \subset \Supp(N'_j)$
for any $1 \le i,j \le d$, then
\[
(m_{\bN,\bN'})_*
\kappa_{\bN,(b_j)}
=\kappa_{\bN',(b'_j)}.
\]
\item
Let $\wp$ be a closed point of $X$.
Suppose that 
$\Supp(N''_i)\subset \{\wp\}
\subset \Supp(N_i)$
for every $i$. 
Let $e$ denote the number of $i$'s
with $\wp\not\in \Supp(N'_i)$.
Then
\[
(m_{\bN, \bN'})_*
\kappa_{\bN,(b_j)}
=\sum_{r=0}^e (-1)^r q_\wp^{r(r-1)/2}
T_{[\kappa(\wp)^{\oplus r}]}
\kappa_{\bN', (b'_j)},
\]
where $\kappa(\wp)$ is the residue field at $\wp$ and 
$q_{\wp}$ is the cardinality of $\kappa(\wp)$.
\end{enumerate}
\end{thm}

\subsection{Proof of Theorem~\ref{main theorem} (1)}
\subsubsection{ }
We set $\bN''= \bigoplus_i N''_i$.
By decomposing the surjection $N_i \surj N'_i$ 
into the composite of two surjections and 
by using induction, we may assume that 
$\length_{\cO_{X,x}} N'_i \otimes_{\cO_X} \cO_{X,x} \ge 
\length_{\cO_{X,x}}N''_j \otimes_{\cO_X} \cO_{X,x}$
for any $x\in X$ and for any $i,j$.

Under the assumption, it follows from 
Lemma \ref{lem:Gal}
that the morphism $m_{\bN,\bN'}\colon  \bN \to \bN'$ 
is a Galois covering in $\cCo{d}$ which is a \fibr.
Let us compute the group $H = \Aut_{\bN'}(\bN)$.
As we have explained in Example \ref{ex:automorphisms},
the group $H$ consists of the automorphisms
of the form $z(\sigma)$ (see Section~\ref{sec:notationz}
for the notation $z(-)$),
where $\sigma \colon  \bN \to \bN$ is an automorphism 
of the $\cO_X$-module $\bN$ such that
$\sigma(\bN'') = \bN''$ and that the automorphism of 
the $\cO_X$-module $\bN'$ induced by $\sigma$ is 
equal to the identity.
For such an automorphism $\sigma \colon  \bN \to \bN$,
the difference $\sigma - \id_N$ gives an $\cO_X$-linear
homomorphism $\bN \to \bN''$.
Conversely, for an $\cO_X$-linear
homomorphism $\bN \to \bN''$,
we set $\sigma_f = \id_N + \iota \circ f$,
where $\iota\colon  \bN'' \inj \bN$ denotes the
inclusion homomorphism.
We can check, by using the assumption that
$\Supp(N''_i) \subset \Supp(N'_j)$ holds for any $1 \le i,j \le d$
and by Nakayama's lemma, that
$\sigma_f$ is an $\cO_X$-linear automorphism of $\bN$.
This shows that the map $\Hom_{\cO_X}(\bN,\bN'')
\to H$ which sends $f \in \Hom_{\cO_X}(\bN,\bN'')$ to 
$z(\sigma_f)$ is bijective.
By the assumption on the length of $N_i'$ and of $N_i''$, 
the map $\Hom_{\cO_X}(\bN',\bN'') \to \Hom_{\cO_X}(\bN,\bN'')$
given by the composition of the quotient map $\bN \surj \bN'$
is an isomorphism of groups.
By using this, we can check that the composite
$\Hom_{\cO_X}(\bN',\bN'') \xto{\cong} \Hom_{\cO_X}(\bN,\bN'')
\xto{\cong} H$ is an isomorphism of groups.
In particular, $H$ is an abelian group.
The abelian group $\Hom_{\cO_X}(\bN',\bN'')$
is canonically isomorphic to the direct product
$\Hom_{\cO_X}(\bN',\bN'') \cong \prod_{i=1}^d \Hom_{\cO_X}(N'_i, \bN'')$. 
Via the isomorphism $\Hom_{\cO_X}(\bN',\bN'') \xto{\cong} H$
constructed above, we obtain the direct product decomposition
$H = \prod_{i=1}^d H_i$ of the group $H$.

\subsubsection{ }
Given an object $M$ in $\cCo{d}$, by abuse of notation, 
we denote simply by $M$ the sheaf $\quotid{M}$.  
In this subsection, we define the objects and the morphisms 
of the following diagram in $\cFCotu{d}$
and show that it is commutative and that the middle square
is cartesian:
\begin{equation}
\label{eqn:diagram1}
\xymatrix{
\bN \ar[r]^{\beta_i}
\ar@(ur,ul)[rr]^{r_{\bN,\wt{N}_i}}
\ar[dr]_{m_{\bN,\bN'}}
& \overline{N}_i 
\ar[r]^{\gamma_i}
\ar[d]^{\delta_i}
\ar@{}[dr]|\square
& \wt{N}_{i} 
\ar[d]^{\delta'_i}
\ar[r]^{r_{\wt{N}_i,N_i}} 
\ar[dr]^{m_{\wt{N}_i,N'_i}}
& N_i \\
& \bN' 
\ar[r]^{\gamma'_i}
\ar@(dr,dl)[rr]_{r_{\bN',N'_i}} 
& \quot{\wt{N}_i}{\wt{H}_i}
\ar[r]^{\epsilon_i}
& N'_i.
}
\end{equation}

Consider the subgroup 
$H'_i= \prod_{j \neq i} H_j \subset H$ for each $i=1,\dots, d$.
We set $\overline{N}_i=\quot{\bN}{H_i'}$
and let $\beta_i\colon \bN \to \overline{N}_i$ denote the canonical morphism.
We can check, by using Lemma \ref{lem:M_fiber_products} that
the morphism
$$
\beta_1 \times \cdots\times \beta_d \colon 
\bN \to \overline{N}_1\times_{\bN'}
\cdots\times_{\bN'} \overline{N}_d
$$
in $\cFCotu{d}$ is an isomorphism.

Let $\wt{N}_i$ denote the inverse image 
of $N'_i \subset \bN'$
under the $\cO_X$-homomorphism $\bN
\to \bN'$. Let $\wt{H}_i$ denote the group
of automorphisms of the object $\wt{N}_i$ in $\cCo{d}$
of the form $z(\sigma)$, where $\sigma\colon \wt{N}_i \xto{\cong}
\wt{N}_i$ is an automorphism of the $\cO_X$-module $\wt{N}_i$
such that $\sigma(\bN'')= \bN''$ and that the endomorphisms
of $\bN''$ and $N'_i = \wt{N}_i/\bN''$ induced by $\sigma$
are the identity homomorphisms.
Let $\delta'_i \colon  \wt{N}_i
\to \quot{\wt{N}_i}{\wt{H}_i}$ denote the canonical morphism.
By using Lemma \ref{lem:4_0_2} we can check that the equality
$m_{\wt{N}_i,N_i} \circ h = m_{\wt{N}_i,N_i}$ holds for any $h \in \wt{H}_i$.
Hence it follows from Lemma \ref{lem:quot_FCd}
that the morphism $m_{\wt{N}_i,N_i}$ factors through 
the canonical morphism $\delta'_i$.
We denote the induced morphism $\quot{\wt{N}_i}{\wt{H}_i}\to N_i$ by $\epsilon_i$.



%
%

Let us consider the morphism
$r_{\bN,\wt{N}_i}\colon \bN \to \wt{N}_i$.
By using Lemma \ref{lem:4_0_2} we can check that the equality
$r_{\bN,\wt{N}_i} \circ h' = r_{\bN,\wt{N}_i}$ holds for any $h' \in H'_i$.
Hence it follows from Lemma \ref{lem:quot_FCd}
that the morphism $r_{\bN,\wt{N}_i}$ factors 
through the canonical morphism $\beta_i\colon  \bN \to \overline{N}_i$.
We denote the induced morphism $\overline{N}_i \to \wt{N}_i$ by $\gamma_i$.

It follows from the construction of the isomorphism
$\Hom_{\cO_X}(\bN',\bN'') \xto{\cong} H$ that
for any automorphism $\sigma$ of the $\cO_X$-module
$\bN$ such that $z(\sigma) \in H$,
we have $\sigma (\wt{N}_i) = \wt{N}_i$ and
the automorphism $\sigma_i$ of $\wt{N}_i$ 
induced by $\sigma$ satisfies $z(\sigma_i) \in \wt{H}_i$.
Hence it follows from Lemma \ref{lem:Galoiscovering}
that the composite
$$
\bN \xto{r_{\bN,\wt{N}_i}} \wt{N}_i
\to \quot{\wt{N}_i}{\wt{H}_i}
$$
factors through the morphism
$m_{\bN,\bN'}\colon \bN \to \bN'$.
We denote by $\gamma'_i$ the induced morphism 
$\gamma'_i\colon \bN' \to \quot{\wt{N}_i}{\wt{H}_i}$
in $\cCotu{d}$.

We have $\delta'_i \circ \gamma_i 
\circ \beta_i = \delta'_i \circ r_{\bN,\wt{N}_i}
= \gamma'_i \circ m_{\bN,\bN'}
= \gamma'_i \circ \delta_i \circ \beta_i$.
It follows from 
Proposition \ref{prop:Ctil_Bsite} that 
$\beta_i$ is an epimorphism in $\cCotu{d}$.
Hence we have $\delta'_i \circ \gamma_i 
= \gamma'_i \circ \delta_i$.

Observe that $m_{\wt{N}_i,N'_i} \circ r_{\bN,\wt{N}_i}
= r_{\bN',N'_i} \circ m_{\bN,\bN'}$.
Hence we have
$\epsilon_i \circ \gamma'_i \circ m_{\bN,\bN'}
= \epsilon_i \circ \delta'_i \circ r_{\bN,\wt{N}_i}
= m_{\wt{N}_i,N'_i} \circ r_{\bN,\wt{N}_i}
= r_{\bN',N'_i} \circ m_{\bN,\bN'}$.
It follows from 
Proposition \ref{prop:Ctil_Bsite}
that $m_{\bN,\bN'}$ is an epimorphism in $\cCotu{d}$.
Hence we have $\epsilon_i \circ \gamma'_i
= r_{\bN',N_i}$.
We thus obtain the commutative 
diagram \eqref{eqn:diagram1}.

We prove that the middle square of the
diagram \eqref{eqn:diagram1} is cartesian.
The homomorphism 
$\Hom_{\cO_X}(N'_i,\bN'') \to \Aut_{\cO_X}(\wt{N}_i)$
which maps $f$ to $(n \mapsto n+f(n))$ induces an
isomorphism $\phi\colon  H_i \cong \wt{H}_i$ satisfying
$r_{\bN,\wt{N}_i} \circ \alpha =
\phi(\alpha) \circ r_{\bN,\wt{N}_i}$. 
Since $H'_i$ is a normal subgroup of $H$,
it follows from Lemma 4.2.6 (2) and (3) of \cite{Grids}
that the morphism $\delta_i$ is a Galois
covering with Galois group isomorphic to
$H/H'_i$.
Since the inclusion $H_i \inj H$
induces an isomorphism $H_i \cong H/H'_i$,
it follows from Lemma \ref{lem:last}
that the middle square in \eqref{eqn:diagram1} 
is cartesian.

%
%

\subsubsection{ }
By the definition of 
$\kappa_{\bN,(b_j)}$, we have
\[
\kappa_{\bN,(b_j)}
=(\beta_1 \times \cdots\times \beta_d)^* \prod_{j=1}^d
\check{\pr}_j^*
(r_{\wt{N}_j,N_j} \gamma_j)^* 
(\alpha_j g'_j (N_j) [b_j])
\]
where
$\check{\pr}_j\colon  \overline{N}_1 \times_{\bN'} 
\cdots \times_{\bN'} \overline{N}_d \to \overline{N}_j$
denotes the $j$-th projection.  Hence,
\[
\begin{array}{ll}
(m_{\bN, \bN'})_*
\kappa_{\bN,(b_j)} &
=(\delta_1 \times \cdots \times \delta_d)_*
\displaystyle\prod_{j=1}^d
\check{\pr}_j^*
(r_{\wt{N}_j,N_j} \gamma_j)^*
(\alpha_j g'_j (N_j) [b_j])
\\
& = \displaystyle\prod_{j=1}^d
(\gamma'_j)^*
(\delta'_j)_{*} r_{\wt{N}_j,N_j}^* 
(\alpha_j g'_j (N_j) [b_j]).
\end{array}
\]

Let $y$ be an element in $\Gamma(X,\wt{N}_i)$,
and $y'$ be its image in $\Gamma(X,N'_i)$.  Suppose that 
$N'_i$ is generated by $y'$.  
Let 
$[y]\in \BS^*(\wt{N}_i)\subset \BS(\wt{N}_i)$, 
$[y']\in \BS^*(N'_i)\subset \BS(N'_i)$
be the sections corresponding to $y$, $y'$.  Then we have
$$
\delta^{\prime *} \delta'_* ([y])
= \sum_{x\in \Ker(\wt{N}\to N')} [y+x]
= m_{\wt{N},N'}^*([y']).
$$
The first equality is because 
the group $\wt{H}_i$ is isomorphic to $\Hom(N'_i,\Ker(\wt{N}_i \to N'_i))$,
and the second equality follows from the definition of the
Schwartz-Bruhat sheaf.  Since $\delta'$ is a \fibr,
it follows that $\delta'_*([y]) = \epsilon^*[y']$.
Applying this, we then have
\[
\begin{array}{ll}
(m_{\bN, \bN'})_*
\kappa_{\bN,(b_j)} &
= \prod_{j=1}^d
r_{\bN',N'_j}^*
(\alpha_j g'_j (N'_j) [b'_j])
\\
&
=\kappa_{\bN',(b'_j)}.
\end{array}
\]
This proves Theorem~\ref{main theorem}(1).
\qed



\subsection{Reduction to the local case}
\label{sec:3_4}
The first step in our proof of Theorem~\ref{main theorem} (2)
is to reduce the case when $X$ is a local scheme.
In the succeeding paragraphs, 
we will prepare arguments necessary for this reduction step.

\subsubsection{ }
We will consider the category
$\cCo{d}$, $\cCotu{d}$, $\wt{\cC}_m^{d,\mu}$ and $\cFCotu{d}$ for
various $X$. To avoid confusion, 
we will write $\cCo{d}$, $\cCotu{d}$, $\wt{\cC}_m^{d,\mu}$
and $\cFCotu{d}$ for $X$ by $\cCo{d}_X$,
$\cCot{d}_{m,X}$, $\cCot{d,\mu}_{m,X}$ and $\cFCot{d}_{m,X}$, respectively.

\subsubsection{ }
Let $\wp$ be a closed point of $X$.
We set $X_{\wp} = \Spec\, \cO_{X,\wp}$.

From now on until the end of Section \ref{sec:3_4}
we assume that $|X| \setminus \{\wp\}$ is non-empty.
We let $U \subset X$ denote the unique open subscheme
satisfying $|U| = |X| \setminus \wp$.
%
We let $j^\wp \colon  U \inj X$ denote the canonical inclusion
and $i_\wp \colon  X_\wp \inj X$ denote the canonical morphism.
%
%
%
%
Let us consider the functors $(j^\wp)^*\colon  \cCo{d}_X
\to \cCo{d}_{U}$ and $i_\wp^* \colon \cCo{d}_{X} \to \cCo{d}_{X_\wp}$
which associate, to each object $N$ in $\cCo{d}_X$,
the pullbacks $(j^\wp)^*N$ and $i_\wp^* N$, respectively.
The pair $((j^\wp)^*,i_\wp^*)$ gives
an equivalence of categories from
the category $\cCo{d}_{X}$ to the product category 
$\cCo{d}_{U} \times \cCo{d}_{X_\wp}$.
The following lemma can be checked easily.
We omit the proof.

\begin{lem}\label{lem:reduction1_1}
Let $f\colon N \to N'$ be a morphism in $\cCo{d}_X$.
Then $f$ is a Galois covering 
(\resp a \fibr) in $\cCo{d}_X$ if and only if
the morphism $(j^\wp)^*(f)$ is a Galois covering
(\resp a \fibr) in $\cCo{d}_U$ and 
the morphism $i_\wp^*(f)$ is a Galois covering 
(\resp a \fibr) in $\cCo{d}_{X_\wp}$.
Moreover the functors $(j^\wp)^*$ and $i_\wp^*$
give an isomorphism $\Aut_{N'}(N) \xto{\cong}
\Aut_{N'|_U}(N|_U) \times \Aut_{N'_\wp}(N'_\wp)$
of groups.
\qed
\end{lem}

\subsubsection{ }
For a presheaf $F^\wp$ on $\cCo{d}_{U}$ 
and for a presheaf $F_\wp$ on $\cCo{d}_{X_\wp}$,
we define $F^\wp \boxtimes F_\wp$ as the presheaf 
on $\cCo{d}_X$ which associates, 
to each object $N$ in $\cCo{d}$, the
set $F^\wp(N|_U) \times F_\wp(N_\wp)$.

\begin{lem}\label{lem:reduction1_2}
Let $F^\wp$ be a presheaf on $\cFCo{d}_{U}$ and let
$F_\wp$ be a presheaf on $\cFCo{d}_{X_\wp}$.
Then there exists an isomorphism
$\theta_{F^\wp,F_\wp} \colon  (F^\wp \boxtimes F_\wp)^a
\xto{\cong} (F^\wp)^a \boxtimes (F_\wp)^a$
of presheaves on $\cCo{d}_X$
such that for each object $N$ in $\cCo{d}_X$,
the diagram
\begin{equation}\label{eq:boxtimes}
\begin{CD}
(F^\wp \boxtimes F_\wp)(N)
@=
F^\wp((j^\wp)^*(N)) \times 
F_\wp(i_\wp^*(N_\wp)) \\
@VVV @VVV \\
(F^\wp \boxtimes F_\wp)^a(N)
@>{\theta_{F^\wp,F_\wp}(N)}>>
(F^\wp)^a((j^\wp)^*(N)) \times 
(F_\wp)^a(i_\wp^*(N_\wp))
\end{CD}
\end{equation}
is commutative.
Here the vertical arrows in the diagram are the maps
induced by the morphisms 
$F^\wp \to (F^\wp)^a$,
$F_\wp \to (F_\wp)^a$, and
$F^\wp \boxtimes F_\wp \to 
(F^\wp \boxtimes F_\wp)^a$.
In particular, the presheaf $(F^\wp)^a \boxtimes (F_\wp)^a$
on $\cFCo{d}_X$ is a sheaf.
\end{lem}

\begin{proof}
Let $N$ be an object in $\cCo{d}$.
Let us consider the category $\Gal/N$ introduced in
Section 4.4.2 of \cite{Grids}.
It follows from Lemma \ref{lem:reduction1_2}
that the functor which associates, to each
object $f\colon M \to N$ in $\Gal/N$, the pair
$((j^\wp)^*(f),i_\wp^*(f))$ gives an equivalence
of categories from $\Gal/N$ to 
the product category $\Gal/(j^\wp)^*N
\times \Gal/i_\wp^* N$.
Hence, by the last statement in
Lemma \ref{lem:reduction1_2} and by the description of the 
functor $(-)^a$ given in
Section 4.4.2 of \cite{Grids}, we have a bijection
$\theta_{F^\wp,F_\wp}(N) \colon  (F^\wp \boxtimes F_\wp)^a(N)
\xto{\cong} (F^\wp)^a((j^\wp)^*(N)) \times 
(F_\wp)^a(i_\wp^*(N_\wp))$ such that the diagram \eqref{eq:boxtimes}
is commutative.
We can check, by using the construction of 
the pullback maps, given in Section 4.4.4 of \cite{Grids},
with respect to a morphism in $\cCo{d}$,
that the bijection $\theta_{F^\wp,F_\wp}(N)$ is functorial in $N$.
Hence the collection of bijections $\theta_{F^\wp,F_\wp}(N)$
for all objects $N$ in $\cCo{d}_X$ gives an isomorphism
of presheaves on $\cCo{d}_X$. This completes the proof.
\end{proof}

\subsubsection{ }
Let us fix an object $N^\wp$ in $\cCo{d}_{U}$.
Let us consider the functor $\eta_{N^\wp}\colon \cCo{d}_{X_\wp}
\to \cCo{d}_X$ which associates, to each 
object $N_\wp$ of $\cCo{d}_{X_\wp}$,
an object $N^\wp \oplus N_\wp$ of $\cCo{d}_X$.
Here the symbol $N^\wp \oplus N_\wp$ stands for the object 
$j^\wp_* N^\wp \oplus (i_\wp)_* N_\wp$ of $\cCo{d}_X$.

Let $F$ be an object in $\cCot{d}_{m,X_\wp}$.
Let us write $F=\quot{N_\wp}{H}$, where
$N_\wp$ is an object in $\cCo{d}_{X_\wp}$
and $H \subset \Aut_{\cCo{d}_{X_\wp}}(N_\wp)$ 
is a subgroup.
Via the functor $\eta_{N^\wp}$, 
we sometimes regard $H$ as a subgroup of $\eta_{N^\wp}(N_\wp)$.
We set $\wt{\eta}_{N^\wp}(F)
= \quot{\eta_{N^\wp}(N_\wp)}{H}$. 
This is an object in $\cCot{d}_{m,X}$.
By using Lemma \ref{lem:reduction1_1}, we can check that
the presheaf 
$\quot{\Hom_{\cCo{d}_X}(-,\eta_{N^\wp}(N_\wp))}{H}$ on $\cCo{d}_{X}$
is isomorphic to the presheaf 
$(\quot{\Hom_{\cCo{d}_{X_\wp}}(-,N_\wp)}{H})
\boxtimes \Hom_{\cCo{d}_U}(-, N^\wp)$.
Hence it follows from Lemma \ref{lem:reduction1_2}
that $\wt{\eta}_{N^\wp}(F)$ is isomorphic,
as a presheaf on $\cCo{d}_{X}$, 
to the presheaf $\quot{N_\wp}{H} \boxtimes \quotid{N^\wp}$.
Via this isomorphism, one can construct
a morphism $\wt{\eta}_{N^\wp}(f) \colon  
\wt{\eta}_{N^\wp}(F) \to \wt{\eta}_{N^\wp}(F')$
in $\cCot{d}_{m,X}$ for each morphism $f\colon F \to F'$
in $\cCot{d}_{m,X_\wp}$, which gives a functor $\wt{\eta}_{N^\wp}$ 
from $\cCot{d}_{m,X_\wp}$ to $\cCot{d}_{m,X}$.
By abuse of notation, we denote by the same symbol
$\wt{\eta}_{N^\wp}$ the functor $\cFCot{d}_{m,X_\wp} \to \cFCot{d}_{m,X}$
induced by $\wt{\eta}_{N^\wp}$.

\begin{lem}\label{lem:reduction1_3}
Let $N^\wp$ be an object in $\cCo{d}_{U}$.
\begin{enumerate}
\item For any diagram $F_1 \to F' \leftarrow F_2$ in
$\cFCot{d}_{m,X_\wp}$, the map
$\Hom_{F'}(F_1,F_2) \to \Hom_{\wt{\eta}_{N^\wp}(F')}
(\wt{\eta}_{N^\wp}(F_1), \wt{\eta}_{N^\wp}(F_2))$
given by the functor $\wt{\eta}_{N^\wp}$ is bijective.
\item The functor $\wt{\eta}_{N^\wp}$ commutes with 
fiber products.
\item The functor $\wt{\eta}_{N^\wp}$ sends any \fibr in 
$\cFCot{d}_{m,X_\wp}$ to a \fibr in $\cFCot{d}_{m,X}$.
\end{enumerate}
\end{lem}

\begin{proof}
For any diagram $F_{\wp,1} \to F'_\wp \leftarrow F_{\wp,2}$
of presheaves in $\cCo{d}_{X_\wp}$ and for any presheaf
$F^\wp$ on $\cCo{d}_U$, one can check easily that the map
$\Hom_{F'_\wp}(F_{\wp,1},F_{\wp,2}) \to 
\Hom_{F'_\wp \boxtimes F^\wp}(F_{\wp,1}\boxtimes F^\wp,
F_{\wp,2} \boxtimes F^\wp)$, which sends 
a morphism $f\colon F_{\wp,1} \to F_{\wp,2}$ over $F'_\wp$
to the morphism induced by $f$ and $\id_{F^\wp}$,
is bijective.
Applying this to the case when the diagram
$F_{\wp,1} \to F'_\wp \leftarrow F_{\wp,2}$ is
equal to the diagram $F_1 \to F' \leftarrow F_2$
and when $F^\wp = \quotid{N^\wp}$,
we obtain the claim (1).

The claim (2) is an immediate consequence of the claim (1).

We prove the claim (3). It suffices to show that
the functor $\wt{\eta}_{N^\wp}$ sends any \fibr in 
$\cCot{d}_{m,X_\wp}$ to a \fibr in $\cCot{d}_{m,X}$.
Suppose that $f\colon F \to F'$ is a \fibr in $\cCot{d}_{m,X_\wp}$.
Let us write $F=\quot{N_\wp}{H}$ and $F'=\quot{N'_\wp}{H'}$, 
where $N_\wp$ and $N'_\wp$ are objects in $\cCo{d}_{X_\wp}$
and $H \subset \Aut_{\cFCo{d}_{X_\wp}}(N_\wp)$ and
$H' \subset \Aut_{\cFCo{d}_{X_\wp}}(N'_\wp)$
are subgroups.
Let us denote by $f^\mu$ the morphism $(N_\wp,H) \to (N'_\wp,H')$
in $\wt{\cC}^{d,\mu}_{m,X_\wp}$ corresponding to $f$.
Let us take a model $(f^\mu;m,f')$ of $f^\mu$.
Then $f'$ is a \fibr in $\cCo{d}_{X_\wp}$.
Let us denote by $\eta_{N^\wp}^\mu$ the functor
$\wt{\cC}^{d,\mu}_{m,X_\wp} \to \wt{\cC}^{d,\mu}_{m,X}$
induced by $\eta_{N^\wp}$.
Then the triple $(\eta_{N^\wp}^\mu(f^\mu);\eta_{N^\wp}(m),\eta_{N^\wp}(f'))$
is a model of the morphism $\wt{\eta}_{N^\wp}^\mu(f^\mu)$ in
$\wt{\cC}^{d,\mu}_{m,X}$. It follows from Lemma \ref{lem:reduction1_1} (2)
that $\eta_{N^\wp}(f')$ is a \fibr.
Hence $\wt{\eta}_{N^\wp}^\mu(f^\mu)$ is a \fibr in
$\wt{\cC}^{d,\mu}_{m,X}$, which proves the claim (3).
\end{proof}

\begin{lem}\label{lem:reduction1_4}
Let $N^\wp$ be an object in $\cCo{d}_{U}$.
For a presheaf $G$ on $\cFCot{d}_{m,X}$,
let $(\wt{\eta}_{N^\wp})^* G$ denote the presheaf on 
$\cFCot{d}_{m,X_\wp}$ given by the composite 
of $G$ with $\wt{\eta}_{N^\wp}$.
\begin{enumerate}
\item Suppose that $G$ is a sheaf on $\cFCot{d}_{m,X}$.
Then $(\wt{\eta}_{N^\wp})^* G$ is a sheaf on
$\cFCot{d}_{m,X_\wp}$.
\item Suppose that $G$ is an abelian presheaf
(\resp a presheaf of rings) with transfers on $\cFCot{d}_{m,X}$.
Then the presheaf $(\wt{\eta}_{N^\wp})^* G$
together with the map $\wt{\eta}_{N^\wp}(f)_*$ 
for each \fibr $f$ in $\cFCot{d}_{m,X_\wp}$
is an abelian presheaf (\resp a presheaf of rings) 
with transfers on $\cFCot{d}_{m,X_\wp}$.
\end{enumerate}
\end{lem}

\begin{proof}
The claim (1) is obvious.

We prove the claim (2).
The claim when $G$ is a presheaf of rings with transfers 
immediately follows from the claim when $G$ is an abelian presheaf
with transfers. Hence it suffices to show that, when
$G$ is an abelian presheaf with transfers,
the presheaf $(\wt{\eta}_{N^\wp})^* G$
together with the map $\wt{\eta}_{N^\wp}(f)_*$ 
for each \fibr $f$ in $\cFCot{d}_{m,X_\wp}$
satisfies the three conditions in Definition \ref{def:transfer}.
It is clear that Condition (1) is satisfied.
Condition (2) is satisfied since the functor
$\wt{\eta}_{N^\wp}$ commutes with fiber products.
To prove that Condition (3) is satisfied,
it suffices to prove that, for any \fibr
$f\colon F \to F'$ in $\cCot{d}_{m,X_\wp}$, we have
$\deg f = \deg \wt{\eta}_{N^\wp}(f)$.
It follows from Proposition \ref{prop:cFCot_enoughGalois},
the category $\cFCot{d}_{m,X_\wp}$ has enough Galois coverings
which are \fibrs. 
Hence there exists a Galois covering $g\colon F_1 \to F'$
in $\cCot{d}_{m,X_\wp}$ such that $\Hom_{F'}(F_1,F)$ is non-empty.
It follows from Lemma \ref{lem:reduction1_3} (2) that
the morphism $\wt{\eta}_{N^\wp}(g)$ is a Galois
covering in $\cCot{d}_{m,X}$.
It follows from Lemma \ref{lem:reduction1_3} (1) that
the functor $\wt{\eta}_{N^\wp}$ induces a bijection
$\Hom_{F'}(F_1,F) \xto{\cong} 
\Hom_{\wt{\eta}_{N^\wp}(F')}(\wt{\eta}_{N^\wp}(F_1),
\wt{\eta}_{N^\wp}(F))$.
An argument in the proof of Proposition \ref{prop:transfer} 
shows that $\deg f$ and $\deg \wt{\eta}_{N^\wp}(f)$
are equal to the cardinalities of the sets
$\Hom_{F'}(F_1,F)$ and
$\Hom_{\wt{\eta}_{N^\wp}(F')}(\wt{\eta}_{N^\wp}(F_1),
\wt{\eta}_{N^\wp}(F))$, respectively.
Thus we have $\deg f = \deg \wt{\eta}_{N^\wp}(f)$.
This completes the proof.
\end{proof}

\subsection{Proof of Theorem~\ref{main theorem} (2)}
\label{sec:3.4}
We now start the proof of Theorem~\ref{main theorem} (2). 
We finish the proof at the end of Section~\ref{sec:3.4}.
By using Lemma \ref{lem:reduction1_4} if necessary,
we may and will assume that $X=X_\wp=\Spec(\cO_{X,\wp})$
and that we are in Situation I.
By decomposing the surjection $N_i \surj N'_i$ 
into the composite of two surjections and 
by using Theorem~\ref{main theorem} (1), 
we may and will assume that
$e \ge 1$, that $N_1=\cdots =N_e =N \cong \kappa(\wp)$,
$N'_i=0$ for $i=1,\ldots,e$, and that
$N_i=N'_i \neq 0$ for $i=e+1,\ldots,d$.
%
%
We set $\bF = \quotid{\bN}$.

\subsubsection{ }
\label{f_r}
For $i=1,\ldots,d$, we have either
$(N'_i,N''_i) =(0,N_i)$ or $(N'_i,N''_i) =(N_i,0)$. 
Hence we have a canonical isomorphism
$N_i \cong N'_i \oplus N''_i$.
We set $\bN'' =\bigoplus_{i=1}^{e} N''_i
\cong \kappa(\wp)^{\oplus e}$.
The isomorphisms $N_i \cong N'_i \oplus N''_i$ for 
$i=1,\ldots,d$ give a canonical isomorphism $\bN \cong \bN' \oplus \bN''$.

For $r=0,\ldots,e$, we define
two objects $E_r$, $E'_r$ in $\cFCotu{d}$
as follows. We set $\bM''_r=\bigoplus_{i=1}^r N''_i$.
Let $H''_r$ denote the group of
the automorphisms of the $\cO_X$-module $\bN''$
that stabilize $\bM''_r$. 
%
For $h''_r\in H''_r$, the restriction of $h''_r$
to $\bM''_r$ induces an $\cO_X$-automorphism
of $\bM''_r$, which we denote by
$h''_r|_{\bM''_r}\colon  \bM''_r \to \bM''_r$.
%
For $h''_r \in H''_r$, we regard $\id_{\bN'} \oplus h''_r\colon 
\bN' \oplus \bN'' \xto{\cong} \bN' \oplus \bN''$ as
an automorphism of the $\cO_X$-module $\bN$
via the canonical isomorphism 
$\bN \cong \bN' \oplus \bN''$ above.
We use the symbol $z(-)$ introduced in Section \ref{sec:notationz}.
We set 
$H_r = \{ z(\id_{\bN'} \oplus h''_r) \ |\ h''_r \in H''_r \}
\subset \Aut_{\cCo{d}}(\bN)$.
We define $E'_r$ to be $\quot{\bN}{H_r}$.
We denote by $c_r$ the canonical morphism
$\bF \to E'_r$ in $\cCotu{d}$.
It follows from Lemma \ref{lem:Galois2} the morphism 
$c_r$ is a Galois covering in $\cCotu{d}$ 
with Galois group $H_r$.

Let us consider the functor 
$M_{H_r}\colon  \cB(H_r) \to \cFCot{d}_{m,/E'_r}$ 
constructed in Section \ref{sec:functor M_G}, 
that sends an object $S$ in $\cB(H_r)$
to the object $\quot{\coprod_{s \in S} \bN}{H_r}$.
We set $S_r =\Hom_{\cO_X}(N,\bM''_r)$.
We regard the set $S_r$ as an object in $\cB(H_r)$ by setting
$z(\id_{\bN'} \oplus h''_r) \cdot s_r = h''_r|_{\bM''_r}\circ s$ for 
$h''_r \in H''_r$ and for $s \in S_r$.
We denote the object $\wt{M}_{H_r}(S_r)$
in $\cFCot{d}_{m,/E'_r}$ by $\eta_r \colon  E_r \to E'_r$.

Let $\bj''_r\colon  \bM''_r \inj \bM''_e = \bN''$ denote the inclusion 
homomorphism of $\cO_X$-modules.
The map $\bj''_r \circ -\colon S_r \to r^{H_e}_{H_r}(S_e)$
given by the composition with $\bj''_r$ is a morphism in $\cB(H_r)$.
Applying the functor $M_{H_r}$, we obtain a morphism
$E_r \to M_{H_r}(r^{H_e}_{H_r}(S_e))$ in $\cFCotu{d}$ over $E'_r$.
%
Since $H_r$ is a subgroup of $H_e$,
we have the morphisms
$c^{/H_r}_{/H_e}\colon E'_r \to E'_e$,
which we denote by $f_r'$,
and $c^{/H_r}_{/H_e, S_e}\colon M_{H_r}(r^{H_e}_{H_r}(S_e)) \to E_e$
introduced in Section \ref{sec:M_G_H}.
We denote by $f_r\colon E_r\to E_e$
the composite $E_r \to M_{H_r}(r^{H_e}_{H_r}(S_e)) 
\xto{c^{/H_r}_{/H_e, S_e}} E_e$.
%
We have the following commutative diagram in $\cFCotu{d}$:
$$
\begin{CD}
E_r @>{\eta_r}>> E'_r \\
@V{f_r}VV @VV{f'_r}V \\
E_e @>{\eta_e}>> E'_e.
\end{CD}
$$

\subsubsection{ }
It follows from the construction of the functor
$M_{H_e}$ that the inclusion $\bN \inj \coprod_{h \in H_e} \bN$
of the component at $1 \in H_e$ induces an isomorphism
\begin{equation}\label{eq:q-calculus_2}
\bF \xto{\cong} M_{H_e}(H_e)
\end{equation}
in $\cFCotu{d}$.
For each $i=1, \dots,e$, let $\mu_i\colon N = N''_i \inj \bN''=\bM''_e$
be the $\cO_X$-homomorphism given by the inclusion into the $i$-th factor.
Let us consider the map $w_i\colon H_e \to S_e$
that sends $z(\id_{\bN'} \oplus h''_e) \in H_e$ to the composite
$h''_e \circ \mu_i$.
Then $w_i$ is a morphism in $\cB(H_e)$.
%
By applying the functor $M_{H_e}$ and composing with the isomorphism
\eqref{eq:q-calculus_2}, we obtain the morphism
$\alpha'_i\colon \bF \to E_e$ in $\cFCotu{d}$ over $E'_e$.
We let $\bE_e=E_e \times_{E'_e} \cdots \times_{E'_e} E_e$ 
be the $e$-fold fiber product of $\eta_e$, and 
$\bi\colon  \bF \to \bE_e$ be the morphism
$(\alpha_1', \ldots, \alpha_e')$.

\subsubsection{ }

For each $r=0,\dots,e$, we let $\bE_r$ denote the
$e$-fold fiber product 
$E_r\times_{E'_r} \cdots \times_{E'_r} E_r$
of $E_r$ over $E'_r$, and let
$\mathbf{f}_r=f_r\times \cdots \times f_r\colon  \bE_r \to \bE_e$.

\begin{prop}\label{q-calculus}
Let the notations be as above. 
The morphisms $\bi$ and $\mathbf{f}_{j}$ for each $j=0, \dots, e$
are \fibrs, and we have 
\[
\deg\bi =\sum_{r=0}^e (-1)^r
q_\wp^{r(r-1)/2} \deg \mathbf{f}_{e-r}.
\]
\end{prop}

Before giving a proof of Proposition \ref{q-calculus},
we introduce some notation.
%
For $r=0,\ldots,e$ and for $i=1,\ldots,e$, let
$\pr_{r,i} \colon  \bE_r \to E_r$
denote the projection to the $i$-th factor.
%

For $r=0,\ldots,e$,
%
%
we set
$$
t_r = \iota^{H_e}_{H_r}(\bj''_r \circ -)\colon  H_e \times^{H_r} S_r \to S_e.
$$
Here the notation $\iota^{H_e}_{H_r}(-)$ be as in Section \ref{sec:M_G_H}.
Let $\bS_r = \prod_{i=1}^e S_r$.
We endow $\bS_r$ with the diagonal left action of the
group $H_r$ and regard it as an object in $\cB(H_r)$.
The map $\bj''_r \circ - \colon  S_r \to r^{H_e}_{H_r}(S_e)$ induces a map
$\bS_r \to r^{H_e}_{H_r}(\bS_e)$, 
which we denote by $(\bj''_r \circ -)_{1 \le i \le e}$.
We set
$$
\bt_r =  \iota^{H_e}_{H_r}((\bj''_r \circ -)_{1 \le i \le e}) \colon 
H_e \times^{H_r} \bS_r \to \bS_e.
$$ 
%
%
We denote by $\bw \colon  H_e \to \bS_e$ the morphism in $\cB(H_e)$ 
which sends $h_e \in H_e$ to $(w_i(h_e))_{1 \le i \le e}$.
Then $\bw$ is a morphism in $\cB(H_e)$.

\begin{lem}\label{lem:diag_q-calculus}
There exist isomorphisms 
$\bE_e \xto{\cong} M_{H_e}(\bS_e)$ and
$\bE_r \xto{\cong} M_{H_e}(H_e \times^{H_r} \bS_r)$
for $r=0,\ldots,e$
such that the diagrams
\begin{equation}\label{diag:q-calculus_1}
\begin{CD}
\bE_r @>{\mathbf{f}_r}>> \bE_e \\
@V{\cong}VV @VV{\cong}V \\
M_{H_e}(H_e \times^{H_r} \bS_r)
@>{M_{H_e}(\bt_r)}>>
M_{H_e}(\bS_e)
\end{CD}
\end{equation}
for $r=0,\ldots,e$ and
\begin{equation}\label{diag:q-calculus_2}
\begin{CD}
\bF @>{\bi}>> \bE_e \\
@V{\cong}VV @VV{\cong}V \\
M_{H_e}(H_e)
@>{M_{H_e}(\bw)}>>
M_{H_e}(\bS_e)
\end{CD}
\end{equation}
are commutative. 
Here 
the left vertical arrow in the diagram \eqref{diag:q-calculus_2}
is the isomorphism \eqref{eq:q-calculus_2}.
\end{lem}

\begin{proof}
Let $r$ be an integer with $0 \le r \le e$.
Since $E_r = M_{H_r}(S_r)$,
it follows from Lemma \ref{lem:M_fiber_products}
and Lemma \ref{lem:M_G_H} that we have isomorphisms
\begin{equation}\label{eq:q-calculus_1}
\bE_r
\cong M_{H_r}(\bS_r)
\cong
M_{H_e}(H_e \times^{H_r} \bS_r)
\end{equation}
in $\cFCotu{d}$. When $r=e$, we in particular obtain
an isomorphism
\begin{equation}\label{eq:q-calculus_3}
\bE_e
\cong
M_{H_e}(\bS_e).
\end{equation}
%

We prove that the isomorphisms
\eqref{eq:q-calculus_3} and
\eqref{eq:q-calculus_1}
satisfy the desired properties.

Let us fix an integer $r$ with $0 \le r \le e$.
We prove that the diagram \eqref{diag:q-calculus_1}
is commutative.
By Lemma \ref{lem:M_G_H}, we have an isomorphism
$\gamma\colon  M_{H_e}(H_e \times^{H_r} S_r) \cong E_r$.
It follows from Lemma \ref{lem:M_G_H3} that
the composite of $\gamma$ with $f_r \colon  E_r
\to E_e$ is equal to $M_{H_e}(t_r)$.
%
Since the composite $\pr_{e,i} \circ \mathbf{f}_r$ is equal to
the composite $f_r \circ \pr_{r,i}$ for $i=1,\ldots,e$,
the morphism $\mathbf{f}_r$ is equal to
the composite of 
the isomorphism \eqref{eq:q-calculus_1} 
with the morphism $M_{H_e}(\bt_r)$.
Hence the diagram \eqref{diag:q-calculus_1}
is commutative.

It follows from the construction of the morphism $\bi$
that the composite $\pr_{e,i} \circ \bi \colon \bF \to E_e$
is equal to the composite of the isomorphism \eqref{eq:q-calculus_2}
with $M_{H_e}(w_i)$ for $i=1,\ldots,r$.
Hence the morphism $\bi$ is equal to
the composite of 
the isomorphism \eqref{eq:q-calculus_2} 
with the morphism $M_{H_e}(\bw)$.
This shows that the diagram \eqref{diag:q-calculus_2}
is commutative, which completes the proof.
\end{proof}

\begin{proof}[Proof of Proposition \ref{q-calculus}]
By Lemma \ref{lem:diag_q-calculus},
we have $\deg \mathbf{f}_r = \deg M_{H_e}(\bt_r)$
for $r=0,\ldots,e$ and $\deg \bi = \deg M_{H_e}(\bw)$.
We apply Lemma \ref{lem:M_G_deg} to the morphisms
$\bt_r$ for $r=0,\ldots,e$ and to the morphism $\bw$.
Then, to prove the claim, 
it suffices to show the equality
\begin{equation}\label{eq:cardinality}
\sharp \bw^{-1}(\bs) = \sum_{r=0}^e (-1)^r
q_\wp^{r(r-1)/2} \sharp \bt_{e-r}^{-1}(\bs)
\end{equation}
holds for any $\bs \in \bS_e$.

By assumption $\bN''$ is an $e$-dimensional $\kappa(\wp)$-vector space.
Let $\mathrm{Gr}(\bN'',r)$ be the set of $r$-dimensional subspaces
of $\bN''$.  
Let $\bT_r$ denote the set of pairs 
$(W,(\nu_i)_{1 \le i \le e})$ of
$W \in \mathrm{Gr}(\bN'',r)$ and
$(\nu_i)_{1 \le i \le e} \in \prod_{i=1}^e \Hom_{\cO_X}(N,W)$.
For $h''_e \in H''_e$, for $W \in \mathrm{Gr}(\bN'',r)$,
and for $(\nu_i)_{1 \le i \le e} \in \prod_{i=1}^e \Hom_{\cO_X}(N,W)$,
we set $h''_e \cdot W = h''_e(W)$ and 
$h''_e \cdot (W, (\nu_i)_{1 \le i \le e})
=(h''_e(W), (h''_e \circ \nu_i))$. 
Here $h''_e \circ \nu_i$ denotes the composite of $\nu_i$ with
the isomorphism $W \xto{\cong} h''_e(W)$ induced by $h''_e$.
This gives an action of the group $H''_e$ 
on the sets $\mathrm{Gr}(\bN'',r)$ and $\bT_r$.
It can be checked easily that the group
$H''_e$ acts transitively on the set $\mathrm{Gr}(\bN'',r)$
and the stabilizer of $\bM''_r \in \mathrm{Gr}(\bN'',r)$
is equal to $H''_r$.
This shows that the map
$$
\psi_r \colon  H_e \times^{H_r} \bS_r \to \bT_r
$$
which sends the class of
$(z(\id_{\bN'}\oplus h''_e),(s_i)_{1\le i \le e})
\in H_e \times \bS_r$
to the element $(h''_e(\bM''_r), (h''_e \circ \bj''_r \circ s_i)_{1 \le i \le e})$
is an isomorphism in $\cB(H_e)$.
Composing the isomorphism \eqref{eq:q-calculus_1}
with $M_{H_e}(\psi_r)$, we obtain an isomorphism
$$
\bE_r \cong M_{H_e}(\bT_r).
$$
Observe that the composite 
$\bt_r \circ \psi_r^{-1} \colon \bT_r \to \bS_e$ 
is equal to the
map which sends $(W,(\nu_i)_{1 \le i \le e}) \in \bT_r$
to $(j_W \circ \nu_i)_{1 \le i \le e}$. 
Here 
$j_W\colon W \inj \bN''=\bM''_e$ denotes the inclusion homomorphism
of $\cO_X$-modules.
Hence for each $\bs = (s_i)_{1 \le i \le e} \in \bS_e$, 
the cardinality of
$\bt_r^{-1}(\bs)$ is equal to the number of
$W \in \mathrm{Gr}(\bN'',r)$ satisfying
$W \supset V$, where $V=\sum_{i=1}^e \mathrm{Im}\, s_i$.

We recall
some notations in $q$-calculus.
For non-negative integers $j, m, n \ge 0$ with $m \le n$, we let
$[j]=\frac{{q_{\wp}}^j-1}{q_{\wp}-1}$,
$[j]!=[j][j-1]\cdots [1]$,
and
$\displaystyle{n \brack m}
=\dfrac{[n]!}{[m]![n-m]!}$.
Since $\displaystyle{n \brack m}$ 
is the number of $m$-dimensional subspaces 
in an $n$-dimensional $\kappa(\wp)$-vector space,
we have
\[
\begin{array}{l}
\sharp \bt_r^{-1}(\bs)
=
\left\{
\begin{array}{ll}
\displaystyle{e - \dim_{\kappa(\wp)} V 
\brack r - \dim_{\kappa(\wp)} V},
& \text{ if } \dim_{\kappa(\wp)}V \le r
\\
0, & \text{ if } \dim_{\kappa(\wp)}V > r.
\end{array}
\right.
\end{array}
\]
Applying Gauss' binomial formula
\cite[(5.5)]{Kac},
we have
\[
\sum_{r=0}^e (-1)^r q_\wp^{r(r-1)/2} 
\bt_{e-r}^{-1} (\bs)
=\left\{
\begin{array}{ll}
0, & \text{ if }V \neq \bN'' \\
1, & \text{ if }V = \bN''.
\end{array}
\right.
\]
On the other hand, it can be checked easily that
for each $\bs = (s_i)_{1 \le i \le e} \in \bS_e$,
the cardinality of $\bw^{-1}(\bs)$ is
\[
\sharp \bw^{-1}(\bs)
=\left\{
\begin{array}{ll}
0, & \text{ if }V \neq \bN'' \\
1, & \text{ if }V = \bN''.
\end{array}
\right.
\]
We thus obtain the desired equality 
\eqref{eq:cardinality}
for each $\bs \in \bS_e$. This completes the proof.
\end{proof}

\subsubsection{ }
\label{sec:3.4.4}
Let $b \in \Gamma(X,N)$. Recall that
$S_e = \Hom_{\cO_X}(N,\bN'')$. 
For an element $s \in S_e$, let us
consider $s(b) \in \Gamma(X,\bN'')$. 
We regard $s(b)$ as an element in $\Gamma(X,\bN)$
via the embedding $\Gamma(X,\bN'') \inj
\Gamma(X,\bN)$ induced by the inclusion
$\bN'' \inj \bN' \oplus \bN'' \cong \bN$.
The element 
\begin{equation}\label{eq:sb}
([s(b)])_{s \in S_e}
\in \BS(\coprod_{s\in S_e} \bF)
\end{equation}
is invariant under the action of $H_e$.
Here we let the group $H_e$ act on
$\coprod_{s \in S_e} \bF$
in such a way that for each 
$h_e =z(\id_{\bN'} \oplus h''_e) \in H_e$,
the action $h_e\colon 
\coprod_{s \in S_e} \bF \to \coprod_{s \in S_e} \bF$
of $h_e$ is the morphism characterized by the following
property: the map $\pi_0(h_e)\colon S_e \to S_e$ 
is given by the composition with $h''_e$,
and the component at $s$ of $h_e$ is
equal to the automorphism $h_e\colon \bF \to \bF$.
%
Since $\BS$ is a sheaf, the element \eqref{eq:sb}
defines an element in 
$\BS(E_e)$, which we denote by $\wt{b}$.
For $j=1,\ldots,e$, 
let us consider the element
$\wt{g}'_j = g_j(E_e)(\wt{b_j}) \in F_j(E_e)$.

\begin{lem}\label{lem:wtg_i}
The element $\wt{g}'_j \in F_j(E_e)$
belongs to the subset $F'_j(E_e) \subset F_j(E_e)$.
\end{lem}

\begin{proof}
The set
$S_e$ consists of two
$H_e$-orbits $\{0\}$ and
$S_e \setminus \{0\}$.
Let $E_{e,0}$ and $E^0_e$ denote the components of $E_e$ 
corresponding to the $H_e$-orbits
$\{0\}$ and $S_e \setminus \{0\}$, respectively.
We then have $E_e = E_{e,0} \amalg E^0_e$.
For $j=1,\ldots,e$, let $\wt{g}'_{j,0} \in F_j(E_{e,0})$ and 
$\wt{g}'^0_j \in F_j(E^0_e)$ denote the
pullback of $\wt{g}'_j$ with respect to the
inclusions $E_{e,0} \to E_e$ and 
$E^0_e \to E_e$, respectively.
It then suffices to show 
$\wt{g}'_{j,0} \in F'_j(E_{e,0})$ and
$\wt{g}'^0_j \in F'_j(E^0_e)$.

Since the set $\{0\}$ consists of the single element $0$,
the inclusion $\bN \to \coprod_{0 \in \{0\}} \bN$
is an isomorphism and induces 
an isomorphism $\quot{\bN}{H_e} \cong E_{e,0}$.
Since $r_{\bN,\{0\}} \circ h = r_{\bN,\{0\}}$
for any $h \in H_e$, it follows from
Lemma \ref{lem:quot_FCd} the morphism 
$r_{\bN,\{0\}} \colon  \bF \to \quotid{\{0\}}$
in $\cCotu{d}$ factors through the canonical morphism
$\bF \to \quot{\bN}{H_e}$.
We denote by $r_{E_{e,0}, \{0\}} \colon E_{e,0} \to \{0\}$
the composite of the induced morphism
$\quot{\bN}{H_e} \to \{0\}$ with the inverse of
the isomorphism $\quot{\bN}{H_e} \cong E_{e,0}$ above.
Let $j$ be an integer with $1 \le j \le e$.
It is easy to check that the pullback
of $\wt{b_j} \in \BS(E_e)$ with respect to
the inclusion $E_{e,0} \to E_e$ is equal to the
pullback of the image of $[0] \in \BS(\quotid{\{0\}})$
with respect to $r_{E_{e,0}, \{0\}}$.
Hence $\wt{g}'_{j,0}$ is equal to the
pullback of $g'(\{0\})([0]) \in F'_j(\quotid{\{0\}})$
with respect to $r_{E_{e,0}, \{0\}}$.
This shows that the element $\wt{g}'_{j,0} \in F_j(E_{e,0})$ 
belongs to the subset $F'_j(E_{e,0}) \subset F_j(E_{e,0})$.

For $j=1,\ldots,e$, let $H_{e,j} \subset H_e$ denote the
stabilizer of $\mu_j\in S_e$.
Then the map $H_e \to S_e$ that sends
$z(\id_{\bN'} \oplus h''_e) \in H_e$ to $h''_e \cdot \mu_j$ gives an
isomorphism $H_e/H_{e,j} \xto{\cong} S_e \setminus \{0\}$.
By applying the functor $M_{H_e}$, we obtain an
isomorphism $\phi_j\colon \quot{\bN}{H_{e,j}} \xto{\cong} E^0_e$.
We denote by $\bc_e \colon \bF \to \quot{\bN}{H_{e,j}}$ the canonical morphism.
It can be checked easily that we have
$\pr_j \circ h = \pr_j$ for any $h \in H_{e,j}$.
Hence the morphism $\pr_j\colon \bF \to \quotid{N_j} = \quotid{N}$
factors through the morphism $\bc_e$. Taking the composite
with the inverse of $\phi_j$, we obtain
the morphism $\overline{\pr}_j \colon  E^0_e \to \quotid{N}$.
Let $b \in \Gamma(X,N)$.
It follows from the definition of the element $\wt{b}
\in \BS(E_e)$ that the pullback of $\wt{b}$
with respect to the inclusion
$E^0_e \to E_e$ is equal to pullback of 
$[b] \in \BS(\quotid{N})$ with respect to
the morphism $\overline{\pr}_j$.
Hence $\wt{g}'^0_j$ is equal to the pullback of
$g'_j(\quotid{N})([b]) \in F'_j(\quotid{N})$ 
with respect to the morphism $\overline{\pr}_j$.
This shows that the element $\wt{g}'^0_j \in F_j(E^0_e)$ 
belongs to the subset $F'_j(E^0_e) \subset F_j(E^0_e)$.
This completes the proof.
\end{proof}
We set $\bF'=\quotid{\bN'}$.
We set $\wt{g}_{j}=\alpha_j(\wt{g}'_j) \in G(E_e)$.
We also set 
$$
\kappa' = \prod_{j=e+1}^d 
r_{\bN',N'_j}^{*}\alpha_j g'_j(N_j)([b_j])
\in G(\bF').
$$

Let us consider the morphisms
$r_{\bN,\bN'},m_{\bN,\bN'}\colon \bF \to \bF'$.
Since we have $r_{\bN,\bN'} \circ h = r_{\bN,\bN'}$
and $m_{\bN,\bN'} \circ h = m_{\bN,\bN'}$
for any $h \in H_e$, it follows from
Lemma \ref{lem:quot_FCd} that the morphisms 
$r_{\bN,\bN'}$ and $m_{\bN,\bN'}$
induce morphisms $E'_e \to \bF'$ which are denoted by
$r_{E'_e,\bN'}$, $m_{E'_e,\bN'}$ respectively.
For $r=1,\ldots,e$, we set
$r_{E'_r,\bN'}= r_{E'_e,\bN'} \circ f'_r$ and 
$m_{E'_r,\bN'}= m_{E'_e,\bN'} \circ f'_r$.
Let $r_{\bE_e,\bN'}$,
$m_{\bE_e,\bN'}$ denote the 
morphisms 
$r_{E'_e,\bN'}\circ(\eta_e \times\cdots \times \eta_e)$,
$m_{E'_e,\bN'}\circ(\eta_e \times\cdots \times \eta_e)
\colon \bE_e \to \bF'$ respectively.

\begin{lem}\label{lem:3_5_4_1}
We have
\begin{align}
& (m_{\bN,\bN'})_* 
\kappa_{\bN,(b_j)} \\
= & \sum_{r=0}^e (-1)^r q_\wp^{r(r-1)/2}
(m_{\bE_e,\bN'})_*
(\mathbf{f}_{e-r})_* \mathbf{f}_{e-r}^*
((\prod_{j=1}^e
\pr_{e,j}^* \wt{g}_j)\cdot 
r_{\bE_e,\bN'}^* \kappa').
\notag
\end{align}
\end{lem}

\begin{proof}
Let the notation be as in the proof of Lemma \ref{lem:wtg_i}.
It follows from the construction of the morphism 
$\bi\colon \bF \to \bE_e$ that the composite
$\bF \xto{\bc_e} \quot{\bN}{H_{e,j}} \xto{\phi_j} E^0_e \to E_e$,
where the last morphism is the inclusion,
is equal to the composite $\pr_{e,j} \circ \bi$.
Hence, by the argument in the proof of Lemma \ref{lem:wtg_i},
we have the equality
$$
\bi^* \pr_{e,j}^* \wt{g}_j = \pr_j^* (\alpha_j g'_j(N_j)([b_j]))
$$
for $j=1,\ldots, e$.

For $j=e+1,\ldots,d$, the composite
$r_{\bN',N'_j} \circ r_{\bE_e,\bN'} \circ \bi$
is equal to $\pr_j\colon \bF \to \quotid{N_j}$.
Hence we have
$$
\kappa_{\bN,(b_j)}
= \bi^* ((\prod_{j=1}^e
\pr_{e,j}^* \wt{g}_j)
\cdot r_{\bE_e,\bN'}^* \kappa').
$$
Since $m_{\bN,\bN'} = m_{\bE_e,\bN'} \circ \bi$,
we have
$$
(m_{\bN,\bN'})_* 
\kappa_{\bN,(b_j)}
=
(m_{\bE_e,\bN'})_*
\bi_* \bi^*
((\prod_{j=1}^e
\pr_{e,j}^* \wt{g}_j)
\cdot r_{\bE_e,\bN'}^* \kappa').
$$
%
By Proposition~\ref{q-calculus},
it is equal to
\[
\sum_{r=0}^e (-1)^r q_\wp^{r(r-1)/2}
(m_{\bE_e,\bN'})_*
(\mathbf{f}_{e-r})_* \mathbf{f}_{e-r}^*
((\prod_{j=1}^e
\pr_{e,j}^* \wt{g}_j)\cdot 
r_{\bE_e,\bN'}^* \kappa').
\]
This proves the claim.
\end{proof}

\begin{lem}\label{lem:3_5_4_2}
Let $r$ be an integer with $0 \le r \le e$.
Then we have
\begin{align}
\label{eqn:by parts}
& (m_{\bE_e,\bN'})_*
(\mathbf{f}_r)_* \mathbf{f}_r^*
((\displaystyle\prod_{j=1}^e
\pr_{e,j}^*\wt{g}_j) \cdot
r_{\bE_e,\bN'}^{*}\kappa') \\
= & (m_{E'_r,\bN'})_* 
((\displaystyle\prod_{j=1}^e
(\eta_r)_* f_r^*
\wt{g}_j) \cdot
r_{E'_r,\bN'}^{*}\kappa').
\notag
\end{align}
\end{lem}

\begin{proof}
For each $i=1,\dots, e$,
consider the following diagram:
$$
\xymatrix{
E_r
\ar[d]_{f_r}
& \bE_r 
\ar[l]_-{\pr_{r,i}}
\ar[rr]^-{\eta_r \times\cdots\times\eta_r}
\ar[d]^{\mathbf{f}_r}
& & E'_r 
\ar[d]^{f'_r}
\ar[dr]^{m_{E'_r, \bN'}}
& 
\\
E_e
& \bE_e
\ar[l]^{\pr_{e,i}}
\ar[rr]^{\eta_e \times\cdots\times\eta_e}
\ar@(dr,dl)[rrr]_{m_{\bE_e,\bN'}}
& & E'_e
\ar[r]^{m_{E'_e,\bN'}}
& \bF'
\\
}
$$
where 
%
$\eta_r$ and $f'_r$ are as in Section~\ref{f_r}.
It can be checked easily that the diagram is commutative.
Hence we have
$$
\begin{array}{rl}
& (m_{\bE_e,\bN'})_*
(\mathbf{f}_r)_* \mathbf{f}_r^*
((\displaystyle\prod_{j=1}^e
\pr_{e,j}^*\wt{g}_j) \cdot
r_{\bE_e,\bN'}^{*}\kappa')
\\
=& (m_{\bE_e,\bN'})_*
(((\mathbf{f}_r)_* \mathbf{f}_r^*
\displaystyle\prod_{j=1}^e
\pr_{e,j}^*\wt{g}_j) \cdot
r_{\bE_e,\bN'}^{*}\kappa')
\\
=& (m_{E'_e,\bN'})_* 
(((f'_r)_*
(\eta_r \times \cdots \times \eta_r)_*
\displaystyle\prod_{j=1}^e
\pr_{r,j}^* f_r^*
\wt{g}_j) \cdot
r_{E'_e,\bN'}^{*}\kappa').
\end{array}
$$
By Lemma \ref{lem:ring_transfer} we have
$$
(\eta_r \times \cdots \times \eta_r)_*
\displaystyle\prod_{j=1}^e
\pr_{r,j}^* f_r^*
= \displaystyle\prod_{j=1}^e
(\eta_r)_* f_r^*
\wt{g}_j.
$$
Hence
\begin{align}
& (m_{\bE_e,\bN'})_*
(\mathbf{f}_r)_* \mathbf{f}_r^*
((\displaystyle\prod_{j=1}^e
\pr_{e,j}^*\wt{g}_j) \cdot
r_{\bE_e,\bN'}^{*}\kappa')
\notag
\\
= &(m_{E'_e,\bN'})_* 
(((f'_r)_*
\displaystyle\prod_{j=1}^e
(\eta_r)_* f_r^*
\wt{g}_j) \cdot
r_{E'_e,\bN'}^{*}\kappa')
\notag
\\
= & (m_{E'_r,\bN'})_* 
((\displaystyle\prod_{j=1}^e
(\eta_r)_* f_r^*
\wt{g}_j) \cdot
r_{E'_r,\bN'}^{*}\kappa').
\notag
\end{align}
This proves the claim.
\end{proof}

\subsubsection{ }
For $r=0,\ldots,e$, we set
$\bMbar''_r=\bigoplus_{i=r+1}^{e} N''_i$.
Let $c'_r\colon  \quotid{\bN' \oplus \bMbar''_r} \to \bF' \oplus [\bMbar''_r]$
denote the canonical morphism (see Section~\ref{sec:exam Hecke} for notation).
We regard $\bN' \oplus \bMbar''_r$ as a quotient $\cO_X$-module
of $\bN$ via the projection homomorphism $p_r\colon \bN \surj  
\bigoplus_{j=e+1}^d N_j \oplus \bigoplus_{i=r+1}^e N_i
= \bN' \oplus \bMbar''_r$.
Observe that the kernel of $p_r$ is equal to $\bM''_r$.
Hence for any $h''_r \in H''_r$, there exists an
automorphism $\overline{h''}_r$ of the $\cO_X$-module $\bMbar''_r$
satisfying $p_r \circ (\id_{\bN'} \oplus h''_r) = 
(\id_{\bN'} \oplus \overline{h''}_r) \circ p_r$.
This shows that the equality $c'_r \circ m_{\bN,\bN' \oplus \bMbar''_r}
\circ h_r = c'_r \circ m_{\bN,\bN' \oplus \bMbar''_r}$
holds for any $h_r \in H_r$.
Hence it follows from Lemma \ref{lem:quot_FCd} that
the composite $c'_r \circ m_{\bN,\bN' \oplus \bMbar''_r}$
factors through the canonical morphism $\bF \to E'_r$.
We denote by $\delta_r\colon  E'_r \to \bF' \oplus [\bMbar''_r]$
the induced morphism.
We then have the following commutative diagram
$$
\xymatrix{
E_e &
E_r \ar[l]_{f_r}
\ar[r]^{\eta_r} &
E'_r \ar[r]^-{\delta_r}
\ar[d]_{r_{E'_r,\bN'}}
\ar@(ur,ul)[rrr]^{m_{E'_r,\bN'}}
& \bF' \oplus[\bMbar''_r]
\ar[rr]_-{m_{\bF'\oplus[\bMbar''_r],\bF'}}
\ar[d]^{r_{\bF'\oplus[\bMbar''_r],\bF'}}
& & \bF' \\
& & 
\bF'
& \bF'. & &
}
$$

\begin{lem}\label{Lemma 8_1}
Let the notations be as above.
Then for $j=1,\ldots,e$, we have
$(\eta_r)_* {f_r}^* (\wt{b_j}) = 
\delta_r^* 
r_{\bF'\oplus [\bMbar''_r],\bF'}^* ([0])$ 
in $\BS(E'_r)$.
\end{lem}

\begin{proof}
It follows from Lemma \ref{lem:M_base_change} that
the following diagram in $\cFCotu{d}$ is cartesian:
$$
\xymatrix{
\displaystyle{\coprod_{s \in S_r}} \bF 
\ar[r]^{\ (1)}
\ar[d]_{(2)}
\ar@{}[dr]|\square
& \bF
\ar[d]^{c_r} \\
E_r
\ar[r]_{\eta_r}
& E'_r.
}
$$
Here (1) is the morphism in $\cFCotu{d}$ 
every component of which is the
identity morphism on $\bF$, and (2) is the composite of the
canonical isomorphism 
$\coprod_{s \in S_r} \bF \cong 
\quotid{\coprod_{s \in S_r} \bN}$
with the canonical morphism 
$$
\quotid{\coprod_{s \in S_r} \bN} \to E_r.
$$
By Condition (2) of Definition \ref{def:transfer} we have 
$$
\begin{array}{rl}
& c_r^* (\eta_r)_{*}{f_r}^* (\wt{b}_j)
= 
\displaystyle{\sum_{s \in S_r}} [s(b_j)]
= \displaystyle{\sum_{x\in \bM''_r}} [x] \\
= & m_{\bN, \bMbar''_r\oplus \bN'}^{*}([0])
= c_r^* \delta_r^{*}r_{\bF'\oplus[\bMbar''_r],\bF'}^* ([0]).
\end{array}
$$
Since $c_r$ is a fiber and $\BS$ is a sheaf on $\cFCotu{d}$,
the claim follows.
\end{proof}

\subsubsection{ }
Let $G'_r$ be the subgroup of $\Aut_{\cO_X}(\bN)$ of
elements $g'$ such that the equalities $g'(\bN'')=\bN''$,
$g'(\bM''_r)=\bM''_r$, 
and $g'(\bN' \oplus \bM''_r) = \bN' \oplus \bM''_r$ hold, 
and that the automorphism of $(\bN' \oplus \bM''_r)/\bM''_r$
induced by $g'$ is equal to the identity.
%
%
%
We set $G_r = \{ z(g')\ |\ g' \in G'_r \} \subset \Aut_{\cCo{d}}(\bN)$.
It follows from Lemma \ref{lem:Gal} that
the morphism $m_{\bN,\bN'\oplus \bMbar''_r}\colon  \bN \to \bN'\oplus \bMbar''_r$ 
is a Galois covering in $\cCo{d}$ which is a \fibr.
Hence it follows from Lemma \ref{lem:i_Galois} that
the morphism $m_{\bN,\bN'\oplus \bMbar''_r}\colon \bF \to 
\quotid{\bN'\oplus \bMbar''_r}$ is a Galois
covering in $\cFCotu{d}$ which is a \fibr.
We can check, by using Lemma \ref{lem:4_0_2}, that
the group $\Aut_{\bN'\oplus \bMbar''_r}(\bN)$ is equal to 
$z(G''_r)$, where $G''_r$ denotes the group of automorphisms 
$g''_r$ of the $\cO_X$-module $\bN$ satisfying $p_r \circ g''_r = p_r$.
This shows that $\Aut_{\bN'\oplus \bMbar''_r}(\bN)$ is a
subgroup of $G_r$.
For $g'_r \in G'_r$, let $g'_r|_{\bMbar''_r}$ denote the
the automorphism of the $\cO_X$-module 
$\bN''/\bM''_r$ induced by $g'_r$, regarded as an
automorphism of $\bMbar''_r$ via the isomorphism
$\bMbar''_r \cong \bN''/\bM''_r$ given by the composite
$\bMbar''_r \inj \bN'' \surj \bN''/\bM''_r$.
Then the map $G_r \to \Aut_{\cCo{d}}(\bMbar''_r)$
which sends $z(g'_r) \in G_r$ to $z(g'_r|_{\bMbar''_r})$
gives a short exact sequence
$$
1 \to \Aut_{\bN'\oplus \bMbar''_r}(\bN) \to G_r
\to \Aut_{\cCo{d}}(\bMbar''_r) \to 1.
$$
Hence by Lemma \ref{lem:double_quot}, we obtain an isomorphism
$\quot{\bN}{G_r} \xto{\cong} \bF'\oplus[\bMbar''_r]$.

Observe that $H_r$ is a subgroup of $G_r$.
Hence the canonical morphism $\bF \to \quot{\bN}{G_r}$
induces a morphism $E'_r = \quot{\bN}{H_r} \to \quot{\bN}{G_r}$.
Then $\delta_r\colon E'_r \to \bF' \oplus [\bMbar''_r]$ is equal to
the composite of the morphism $E'_r \to \quot{\bN}{G_r}$
with the isomorphism $\quot{\bN}{G_r} \xto{\cong} \bF'\oplus[\bMbar''_r]$.

\subsubsection{ }
Let us consider the functor $M_{G_r}\colon \cB(G_r) \to \cFCotu{d}$.
It follows from Lemma \ref{lem:M_G_H} that we have
an isomorphism $M_{G_r}(G_r/H_r) \xto{\cong} E'_r$.
Let us consider the abelian group $\Hom_{\cO_X}(\bN', \bM''_r)$.
For $g'_r \in G'_r$ the restriction of
$g'_r$ to the $\cO_X$-submodule $\bM''_r$ gives an automorphism
of $\bM''_r$, which we denote by $g'_r|_{\bM''_r}$.
Let $\bj'\colon  \bN' \inj \bN' \oplus \bN'' \cong \bN$ denote the
inclusion homomorphism of $\cO_X$-modules.
For $g'_r \in G'_r$, the image of the difference $g'_r|_{\bN'} - \bj'\colon 
\bN' \to \bN$ is contained in $\bM''_r$.
Hence it induces a homomorphism $\bN' \to \bM''_r$
of $\cO_X$-modules, which we denote by $D(g'_r)$.
For $z(g'_r) \in G_r$ and for $\mu \in \Hom_{\cO_X}(\bN', \bM''_r)$,
we set $z(g'_r) \cdot \mu = g'_r|_{\bM''_r} \circ \mu + D(g'_r)$.
This gives an action from the left of $G_r$ on 
$\Hom_{\cO_X}(\bN', \bM''_r)$.
In this way, we regard $\Hom_{\cO_X}(\bN', \bM''_r)$ as an object
in $\cB(G_r)$.
For a homomorphism $\mu\colon \bN' \to \bM''_r$ of $\cO_X$-modules,
we denote by $\alpha_\mu \colon  \bN'\oplus \bN'' \to \bN' \oplus \bN''$
the endomorphism of the $\cO_X$-module $\bN' \oplus \bN''$
such that the restriction of $\alpha_\mu$ to $\bN''$ is equal to the inclusion
$\bN'' \inj \bN' \oplus \bN''$ and that the restriction
$\alpha_\mu$ to $\bN'$ is equal to the sum of the inclusion
$\bN' \inj \bN' \oplus \bN''$ and the composite
$\bN' \xto{\mu} \bM''_r \inj \bN'' \inj \bN' \oplus \bN''$.
Since $\alpha_\mu \circ \alpha_{-\mu} = \alpha_{-\mu}
\circ \alpha_\mu = \id_{\bN' \oplus \bN''}$,
the endomorphism $\alpha_\mu$ is an automorphism of 
the $\cO_X$-module $\bN' \oplus \bN''$.
We regard $\alpha_\mu$ as an automorphism of the
$\cO_X$-module $\bN$ via the isomorphism 
$\bN \cong \bN' \oplus \bN''$.
The map 
\begin{equation}\label{eq:alpha}
\alpha \colon  \Hom_{\cO_X}(\bN', \bM''_r) \to G_r
\end{equation}
which sends $\mu$ to $z(\alpha_\mu)$ is an injective homomorphism
of groups.
It can be checked easily that the composite
of $\alpha$ with the quotient map $G_r \to G_r/H_r$
gives an isomorphism 
\begin{equation}\label{eq:alpha2}
\Hom_{\cO_X}(\bN', \bM''_r) \xto{\cong} G_r/H_r
\end{equation}
in $\cB(G_r)$.

\subsubsection{ }
Let $i$ be an integer with $e+1 \le i \le d$.
Let us consider the abelian group
$S_{r,i} = \Hom_{\cO_X}(N'_i, \bM''_r)$. 
For $z(g'_r) \in G_r$ and for $\mu \in S_{r,i}$,
we set $z(g'_r) \cdot \mu = g'_r|_{\bM''_r} \circ \mu + D(g'_r)_i$.
Here $D(g'_r)_i$ denote the restriction of $D(g'_r)$ to the
$i$-th component $N'_i \subset \bN'$.
This gives an action from the left of $G_r$ on $S_{r,i}$ and
$S_{r,i}$ is an object in $\cB(G_r)$.
We set $E'_{r,i}=M_{G_r}(S_{r,i})$.
We denote by $\epsilon'_i \colon  E'_{r,i} \to \quot{\bN}{G_r} 
\cong \bF' \oplus [\bMbar''_r]$ the structure morphism
of $\wt{M}_{G_r}(S_{r,i})$.
%
%
Since $S_{r,i}$ is isomorphic to a quotient of
$\Hom_{\cO_X}(\bN', \bM''_r) \cong G_r/H_r$, the group $G_r$ acts
transitively on $S_{r,i}$.
Let $G_{r,i} \subset G_r$ denote the stabilizer
of $0 \in S_{r,i}$.
Since $S_{r,i} \cong G_r/G_{r,i}$, we have
$E'_{r,i} \cong \quot{\bN}{G_{r,i}}$. We denote by
$\beta_{r,i}\colon \bF \to E'_{r,i}$ the composite of
the canonical morphism $\bF \to \quot{\bN}{G_{r,i}}$ with this
isomorphism.
It follows from the definition of the action of $G_r$ on
$S_{r,i}$ that we have $r_{\bN,N_i}
\circ g_{r,i} = r_{\bN,N_i}$ for any $g_{r,i} \in G_{r,i}$.
Hence by Lemma \ref{lem:quot_FCd}, the morphism $r_{\bN,N_i}$
in $\cFCotu{d}$ factors through the morphism
$\beta_{r,i} \colon  \bN \to \quot{\bN}{G_{r,i}} \cong E'_{r,i}$. We denote
by $c_{r,i} \colon  E'_{r,i} \to \quotid{N_i'}$ the induced morphism.


\begin{lem}\label{Lemma 8_2}
For $i=e+1,\ldots,d$, we have
$(\epsilon'_i)_* c_{r,i}^* ([b_i])
= r_{\bF'\oplus [\bMbar''_r],\bF'}^* r_{\bN',N'_i}^* 
([b_i])$
in $\BS(\bF' \oplus[\bMbar''_r])$.
\end{lem}

\begin{proof}
It follows from the definition of the group $G_{r,i}$
that the composite
$$
\Hom_{\cO_X}(N_i,\bM''_e) \xto{(1)} 
\Hom_{\cO_X}(\bN',\bM''_e)
\xto{(2)} G_r/H_r \xto{(3)} G_r/G_{r,i}
$$
is an isomorphism in $\cB(G_r)$. 
Here the map (1) is the homomorphism
given by the composition with the projection
$\bN' \to N_i$, the map (2) is the isomorphism
\eqref{eq:alpha2}, and the map (3) is the natural
quotient map.
Let $H$ denote the image of
the composite of the homomorphism (1)
with the homomorphism $\alpha$ in \eqref{eq:alpha}.
Then the two quotient maps $G_r \to G_r/G_{r,i}$ and
$G_r \to G_r/H$ gives an isomorphism
$G_r \xto{\cong} G_r/G_{r,i} \times G_r/H$ in $\cB(G_r)$.
Hence it follows from Lemma \ref{lem:M_fiber_products} that
%
the following diagram is cartesian:
$$
\xymatrix{
\bF \ar[r]^{(4)}
\ar[d]_{\beta_{r,i}}
\ar@{}[dr]|\square
& \quot{\bN}{H}
\ar[d]^{(5)} \\
E'_{r,i}
\ar[r]_-{\epsilon'_i}
& \bF'\oplus [\bMbar''_r].
}
$$
Here (4) is the canonical morphism and
(5) is the morphism induced by the canonical
morphism $\bF \to \quot{\bN}{G_r} \cong \bF'\oplus [\bMbar''_r]$.
Hence 
$$
\begin{array}{rl}
\beta_{r,i}^* (\epsilon'_i)^*
(\epsilon'_i)_* c_{r,i}^{*}([b_i])
& = \displaystyle{\sum_{g\in H}} 
g^* r_{\bN,N_i}^*[b_i] \\
& = m_{\bN,\bN/\bM''_r}^* r_{\bN/\bM''_r,N_i }^* ([b_i]) \\
& = \beta_{r,i}^* (\epsilon'_i)^* r_{\bN'\oplus[\bMbar''_r],\bN'}^*
r_{\bF',N_i}^* ([b_i]).
\end{array}
$$
Since $\epsilon'_i \circ \beta_{r,i}$ is a \fibr
and $\BS$ is a sheaf on $\cFCotu{d}$, 
the claim follows.
\end{proof}

\begin{proof}[Proof of Theorem~\ref{main theorem}(2)]
Let $r$ be an integer with $0 \le r \le e$.
By Lemma \ref{lem:3_5_4_1} and Lemma \ref{lem:3_5_4_2},
it suffices to prove that the equality
$$
(m_{E'_r,\bN'})_* 
((\displaystyle\prod_{j=1}^e
(\eta_r)_* f_r^*
\wt{g}_j) \cdot
r_{E'_r,\bN'}^{*}\kappa')
= (m_{\bF'\oplus [\bMbar''_r],\bF'})_*
(r_{\bF'\oplus [\bMbar''_r],\bF'})^* \kappa_{\bN',(b'_j)}
$$
holds.

By Lemma~\ref{Lemma 8_1}, we have
$$
\begin{array}{l}
(m_{E'_r,\bN'})_* 
((\displaystyle\prod_{j=1}^e
(\eta_r)_* f_r^*
\wt{g}_j) \cdot
r_{E'_r,\bN'}^{*}\kappa')
\\
=(m_{\bF'\oplus [\bMbar''_r],\bF'})_*
((\displaystyle\prod_{j=1}^e
r_{\bF' \oplus [\bMbar''_r],\bF'}^* 
\alpha_j g_j(\bN')([0]))
\cdot (\delta_r)_* r_{E'_r,\bN'}^{*}\kappa').
\end{array}
$$
Hence it suffices to prove that the equality
$(\delta_r)_* r_{E'_r,\bN'}^{*}\kappa'
= r_{\bF'\oplus [\bMbar''_r],\bF'}^* \kappa'$
holds.

Let $i$ be an integer with $e+1 \le i \le d$.
Let us consider the homomorphism 
$\beta''_i \colon  \Hom_{\cO_X}(\bN', \bM''_r) \to S_{r,i}$ 
given by the composition with  the 
inclusion $N'_i \inj \bN'$ of the $i$-th component.
The homomorphism $\beta''_i$ is a morphism in $\cB(G_r)$.
Hence we obtain, by applying the functor $M_{G_r}$ to $\beta''_i$, 
a morphism $E'_r \to E'_{r,i}$ in $\cFCotu{d}$, which we denote by
$\beta'_i$.
The homomorphism $\beta''_i$ for each $i=e+1,\dots,d$
induces an isomorphism 
$\Hom_{\cO_X}(\bN', \bM''_r) \xto{\cong} 
\prod_{i=e+1}^d S_{r,i}$.
Hence it follows from Lemma \ref{lem:M_fiber_products} that
the morphism $\beta_{e+1}'\times \cdots \times \beta_d'\colon E'_r \to
E'_{r,e+1}\times_{\bF'\oplus[\bMbar''_r]}\cdots 
\times_{\bF'\oplus[\bMbar''_r]} E'_{r,d}$
is an isomorphism.

Let us consider the diagram
$$
\xymatrix{
E'_r \ar[r]^{\beta'_i}
\ar[d]_{r_{E'_r,\bN'}}
& E'_{r,i}
\ar[r]^-{\epsilon'_i}
\ar[d]^{c_{r,i}}
& \bF'\oplus[\bMbar''_r] 
\ar[rr]^-{r_{\bF'\oplus [\bMbar''_r],\bF'}}
& & \bF'
\ar[d]_{r_{\bN',N'_i}}
\\
\bF' \ar[r]_-{r_{\bN',N'_i}}
& \quotid{N'_i}  & & & \quotid{N'_i}.
}
$$
The left square of this diagram is commutative since
we have $r_{\bN', N'_i} \circ r_{E'_r,\bN'} \circ c_r
= c_{r,i} \circ \beta'_i \circ c_r$
and since $c_r$ is an epimorphism in $\cCotu{d}$.

By Lemma~\ref{Lemma 8_2}, we have
$$
\begin{array}{l}
(\delta_r)_* r_{E'_r,\bN'}^{*}\kappa'
\\
= (\epsilon'_{e+1}\times \cdots \times \epsilon'_{d})_*
(\beta'_{e+1}\times \cdots \times \beta'_{d})_*
r_{E'_r,\bN'}^{*}\kappa'
\\
= \displaystyle{\prod_{i=e+1}^{d}}
(\epsilon'_i)_* c_{r,i}^*
(\alpha_i g'_i(N'_i)([b_i]))
\\
= \displaystyle{\prod_{i=e+1}^{d}}
r_{\bF'\oplus [\bMbar''_r],\bF'}^* r_{\bN',N'_i}^*
(\alpha_i g'_i(N'_i)([b_i]))
\\
= r_{\bF'\oplus [\bMbar''_r],\bF'}^* \kappa'.
\end{array}
$$
Hence Theorem~\ref{main theorem}(2) 
is proved.
\end{proof}

\chapter{Other examples of $Y$-sites}
\label{cha:examples}
The purpose of this chapter is to present three 
kinds of examples of $Y$-sites.    
In Section~\ref{sec:classical groups}, 
we construct $Y$-sites and grids 
whose absolute Galois monoids are classical groups,
in the spirit similar to the category $\cC^d$ of 
Chapter~\ref{cha:Cd}.

In Section~\ref{sec:conn loc noeth},
we construct a $Y$-site for a connected locally noetherian 
scheme over a base scheme.    When the scheme equals the 
base, we recover the \'etale fundamental group.   
In general, the absolute Galois monoids may not be profinite.
We give several examples.   

In Section~\ref{sec:Riemannian}, 
we have an example for Riemannian symmetric spaces.
We see that we recover the group structure but not the topology.
In order to recover the topology, we suspect that enriching 
the category to topological spaces might help.   

\section{Generalities}

\subsection{The absolute Galois monoid as a monoid of endomorphisms} \label{sec:hyp_End}

Let $(\cC,J)$ be a $Y$-site and let $(\cC_0,\iota_0)$ be a grid.
Let $I$ be the order dual of the poset $\cC_0$ and
let us consider the pro-object
$$
Y = (\iota_0(X))_{X \in I}
$$
in the category $\cC$ 
indexed by the filtered poset $I$. 

We construct a homomorphism
\begin{equation} \label{eq:hyp_End}
\Phi \colon  M_\Cip \to \End_{\Pro(\cC)} (Y)
\end{equation}
from the absolute Galois monoid $M_\Cip$
to the monoid $\End_{\Pro(\cC)} (Y)$ of endomorphism of 
$Y$ in the category $\Pro(\cC)$ of filtered pro-objects of $\cC$
as follows. Let $(\alpha, \gamma_\alpha) \in M_\Cip$.
For each object $X$ of $\cC_0$, the composite of the projection
$Y \to \iota_0(\alpha(X))$ and the inverse of the isomorphism
$\gamma_\alpha(X) \colon  \iota_0(X) \xto{\cong} \iota_0(\alpha(X))$
gives an element $f_X \in \Hom_{\Pro(\cC)}(Y,\iota_0(X))$.
For any morphism $h\colon  X' \to X$ in $\cC_0$, we have a commutative
diagram
$$
\begin{CD}
Y @>>> \iota_0(\alpha(X')) @<{\gamma_\alpha(X')}<{\cong}< \iota_0(X') \\
@| @V{\iota_0(\alpha(h))}VV @VV{\iota_0(h)}V \\
Y @>>> \iota_0(\alpha(X)) @<{\gamma_\alpha(X)}<{\cong}< \iota_0(X).
\end{CD}
$$
This shows that $(f_X)_{X \in I}$ is an element of
$\varprojlim_{X \in I}\Hom_{\Pro(\cC)}(Y,\iota_0(X))
= \End_{\Pro(\cC)}(Y)$.
We define $\Phi$ in \eqref{eq:hyp_End} to be the map
that sends $(\alpha,\gamma_\alpha)$ to $(f_X)_{X \in I}$.
It is easy to see that the map $\Phi$
is a homomorphism of monoids.

\begin{prop} \label{prop:hyp_top}
The homomorphism $\Phi$ is an isomorphism of
topological monoids. Here we equip
$\Hom_{\Pro(\cC)}(Y, \iota_0(X)) = \varinjlim_{X' \in I} 
\Hom_{\cC}(\iota_0(X'), \iota_0(X))$
with the discrete topology for each $X \in I$, 
and $\End_{\Pro(\cC)}(Y) = \varprojlim_{X \in I}
\Hom_{\Pro(\cC)}(Y,\iota_0(X))$
with the limit topology.
\end{prop}

To prove Proposition \ref{prop:hyp_top}, we need
the following lemma.

\begin{lem} \label{lem:hyp_aux1}
Let $Z$ be an object of $\cC$.
Then for any $f \in \Hom_{\Pro(\cC)}(Y,Z)$,
there exists a unique pair $(X,\beta)$ of
$X \in I$ and an isomorphism $\beta \colon  \iota_0(X) \xto{\cong} Z$
in $\cC$ such that $f$ is equal to the class of
$\beta \in \Hom_{\cC}(\iota_0(X),Z)$ in $\Hom_{\Pro(\cC)}(Y,Z)$.
\end{lem}

\begin{proof}
First, we prove the existence.
Let us choose an element $X' \in I$ and a morphism
$f'\colon  \iota_0(X') \to Z$ which represents $f$.
Since $\cC_{0,X'/} \to \cC_{\iota_0(X')/}$ is an equivalence
of categories, there exists an object
$h\colon  X' \to X$ of $\cC_{0,X'/}$ such that $\iota_0(h)$
is isomorphic to $f'$ as an object of $\cC_{\iota_0(X')/}$,
i.e., there exists an isomorphism $\beta \colon  \iota_0(X) \xto{\cong} Z$
satisfying $f' = \beta \circ \iota_0(h)$.
Then the pair $(X,\beta)$ satisfies the desired property.

Next we prove the uniqueness.
Suppose that $f$ is the class of an isomorphism
$\beta\colon  \iota_0(X) \xto{\cong} Z$ and
an isomorphism $\beta' \colon \iota_0(X') \xto{\cong} Z$
at the same time.
Then there exists
a diagram $X \xleftarrow{h} X'' \xto{h'} X'$
in $\cC_0$ satisfying 
$\beta \circ \iota_0(h) = \beta' \circ \iota_0(h')$.
The composite $\beta'^{-1} \circ \beta$ gives
an isomorphism from $\iota_0(h)\colon  \iota_0(X'') \to \iota_0(X)$
to $\iota_0(h') \colon \iota_0(X'') \to \iota_0(X')$ in
the undercategory $\cC_{\iota_0(X'')/}$.
Since $\cC_{0,X''/} \to \cC_{\iota_0(X'')/}$ is an equivalence
of categories, we have $X=X'$ and $h=h'$.
Since $\cC$ is an $E$-category, the equality
$\beta \circ \iota_0(h) = \beta' \circ \iota_0(h')$
implies $\beta = \beta'$.
Thus we have $(X,\beta) = (X',\beta')$, which proves the claim.
\end{proof}

\begin{proof}[Proof of Proposition \ref{prop:hyp_top}]
First we prove that $\Phi$ is injective.
Suppose that two elements 
$(\alpha,\gamma_\alpha), (\alpha',\gamma'_{\alpha'})$
are sent to the same element of $\End_{\Pro(\cC)}(Y)$.
Let $X \in I$. Then the composite
$Y \to \iota_0(\alpha(X)) \xto{\gamma_\alpha(X)^{-1}} \iota_0(X)$
is equal to the composite
$Y \to \iota_0(\alpha'(X)) \xto{\gamma'_{\alpha'}(X)^{-1}} \iota_0(X)$.
By Lemma \ref{lem:hyp_aux1}, we have
$\alpha(X) = \alpha'(X)$ and $\gamma_\alpha(X) = \gamma'_{\alpha'}(X)$.
Hence we have $(\alpha,\gamma_\alpha) = (\alpha',\gamma'_{\alpha'})$,
which proves the injectivity of $\Phi$.

Next we prove that $\Phi$ is surjective.
Let $f = (f_X)_{X \in I} \in \End_{\Pro(\cC)}(Y)$.
For each $X \in I$, it follows from Lemma \ref{lem:hyp_aux1} that
there exists a unique pair $(Z,\beta)$ 
of an object $Z$ of
$\cC_0$ and an isomorphism $\beta\colon  \iota_0(Z) \xto{\cong} \iota_0(X)$
such that $f_X$ is equal to the composite
$Y \to \iota_0(Z) \xto{\beta} \iota_0(X)$.
Let us write $Z =\alpha(X)$ and $\beta^{-1} = \gamma_\alpha(X)$.
Let $h\colon  X' \to X$ be a morphism in $\cC_0$.
Since $f_X = \iota_0(h) \circ f_{X'}$, the composite
$Y \to \iota_0(\alpha(X')) \xto{\gamma_\alpha(X')^{-1}} 
\iota(X') \xto{\iota_0(h)} \iota_0(X)$
is equal to the composite
$Y \to \iota_0(\alpha(X)) \xto{\gamma_\alpha(X)^{-1}} 
\iota_0(X)$.
Hence there exists 
a diagram $\alpha(X) \xleftarrow{g} X'' \xto{g'} \alpha(X')$
in $\cC_0$ satisfying 
$\gamma_\alpha(X)^{-1} \circ \iota_0(g) 
= \iota_0(h) \circ \gamma_\alpha(X')^{-1} \circ \iota_0(g')$.
Then $\gamma_\alpha(X) \circ \iota_0(h) 
\circ \gamma_\alpha(X')^{-1}$ is a morphism from
$\iota_0(g')$ to $\iota_0(g)$ in the undercategory
$\cC_{\iota_0(X'')/}$.
Since $\cC_{0,X''/} \to \cC_{\iota_0(X'')/}$ is an equivalence
of categories, there exists a morphism
$h' \colon  \alpha(X') \to \alpha(X)$ satisfying
$\gamma_\alpha(X) \circ \iota_0(h) 
\circ \gamma_\alpha(X')^{-1} = \iota_0(h')$.
Let us denote this $h'$ by $\alpha(h)$.
Then we have $\gamma_\alpha(X) \circ \iota_0(h) 
= \iota_0(\alpha(h)) \circ \gamma_\alpha(X')$.
Thus, we obtain an element $(\alpha,\gamma_\alpha) \in M_\Cip$
satisfying $\Phi(\alpha,\gamma_\alpha) = f$,
which proves the surjectivity of $\Phi$.

Finally we prove that $\Phi$ is a homeomorphism.
Let $I'$ denote the subset of $I$ that consists of
edge objects of $\cC_0$.
Since edge objects are cofinal in $\cC_0$, we have
$\End_{\Pro(\cC)}(Y) 
= \varprojlim_{X \in I'} \Hom_{\Pro(\cC)}(Y,\iota_0(X))$.
For $X \in I'$ and $g \in \Hom_{\Pro(\cC)}(Y,\iota_0(X))$,
let $U_{X,g} \subset \End_{\Pro(\cC)}(Y)$ 
denote the subset of elements $f \in \End_{\Pro(\cC)}(Y)$
such that the composite $Y \xto{f} Y \to \iota_0(X)$
is equal to $g$. Then the family $(U_{X,g})_{X,g}$ forms
an open basis of $\End_{\Pro(\cC)}(Y)$.
It follows from Lemma \ref{lem:hyp_aux1}
that there exists a unique pair
$(X',\beta)$ of an object $X'$ of $\cC_0$ and
an isomorphism $\beta\colon  \iota_0(X') \xto{\cong} \iota_0(X)$
such that $g$ is equal to the composite $Y \to \iota_0(X')
\xto{\beta} \iota_0(X)$.
Then the subset $\Phi^{-1}(U_{X,g})$ of $M_\Cip$ is
equal to the set of elements $(\alpha,\gamma_\alpha) \in M_\Cip$
satisfying $\alpha(X) = X'$ and $\gamma_\alpha(X) = \beta^{-1}$.
Suppose that $\Phi^{-1}(U_{X,g})$ is nonempty and choose
$x \in \Phi^{-1}(U_{X,g})$. Then it follows from
Lemma 8.1.5 of \cite{Grids} that $\Phi^{-1}(U_{X,g})$
is equal to $x \bK_X$.
Since the family $(x \bK_X)_{X,x}$ forms an open basis
of $M_\Cip$, it follows that the map $\Phi$ is
a homeomorphism. This completes the proof.
\end{proof}

\begin{lem} \label{lem:hyp_surj}
Let $f\colon Z' \to Z$ be a morphism in $\cC$. Then the map
$$
\Hom_{\Pro(\cC)}(Y,Z') \to \Hom_{\Pro(\cC)}(Y,Z)
$$
is surjective.
\end{lem}

\begin{proof}
Let $g \in \Hom_{\Pro(\cC)}(Y,Z)$.
Let us choose an object $X$ of $\cC_0$
and a morphism $g_0\colon  \iota_0(X) \to Z$ in $\cC$
whose class in $\Hom_{\Pro(\cC)}(Y,Z)$
is equal to $g$.
Since $\cC$ is semi-cofiltered, there exist
an object $W$ of $\cC$ and morphisms
$f'\colon  W \to \iota_0(X)$ and $g'_0 \colon  W \to Z'$
in $\cC$ satisfying $f \circ g'_0 = g_0 \circ f'$.
Since $(\cC_0,\iota_0)$ is a grid, 
it follows from the axiom of a grid that
there exist a morphism $h\colon  X' \to X$ and
an isomorphism $\beta\colon  \iota_0(X') \xto{\cong} W$
in $\cC$ satisfying $\iota_0(h) = f' \circ \beta$.
Let $g' \in \Hom_{\Pro(\cC)}(Y,Z)$ 
denote the class of the composite
$\iota_0(X') \xto{\beta} W \xto{g'_0} Z'$.
Then the map 
$\Hom_{\Pro(\cC)}(Y,Z') \to \Hom_{\Pro(\cC)}(Y,Z)$
sends $g'$ to $g$.
Since $g$ is arbitrary, the claim is proved.
\end{proof}

\begin{cor} \label{cor:hyp_pro_surj}
Let $X$ be an edge object of $\cC_0$. 
Suppose that the overcategory
$\cC(\cT(J))_{/\iota_0(X)}$ satisfies one of the
two cardinality conditions in Section 5.8.1 of \cite{Grids}.
Then the composite
$$
M_\Cip \cong \End_{\Pro(\cC)}(Y) \to 
\Hom_{\Pro(\cC)}(Y,\iota_0(X))
$$
is surjective.
\end{cor}

\begin{proof}
Let $I'$ denote the poset of elements $X' \in I$ satisfying
$X \le X'$. Then we have $\End_{\Pro(\cC)}(Y) =
\varprojlim_{X' \in I'} \Hom_{\Pro(\cC)}(Y,\iota_0(X'))$.
It follows from Lemma \ref{lem:hyp_surj} that
the transition maps in this limit are surjective.
By our assumption on $\cC(\cT(J))_{/\iota_0(X)}$,
we have either that the fibers of the map
$\Hom_{\Pro(\cC)}(Y,\iota_0(X')) \to 
\Hom_{\Pro(\cC)}(Y,\iota_0(X))$ are finite sets
for every $X' \in I'$, or that $I'$ has
a cofinal subset that is at most countable.
Hence the map $\End_{\Pro(\cC)}(Y) \to 
\Hom_{\Pro(\cC)}(Y,\iota_0(X))$ is surjective,
which proves the claim.
\end{proof}

\subsection{Finite products of $Y$-sites}

\subsubsection{Finite product of $A$-topologies.}
Let $r \ge 1$ be an integer.
For $i=1,\ldots,r$, let $\cC_i$ be a category equipped
with an $A$-topology $J_i$.
Let us consider the direct product 
$$
\cC = \cC_1 \times \cdots \times \cC_r
$$ 
of categories.
It is clear that the class
$\cT = \cT(J_1) \times \cdots \times \cT(J_r)$
of morphisms of $\cC$ is a semi-localizing 
collection of morphisms in $\cC$ in the sense of
\cite[Definition 2.3.1]{Grids}.
Hence $\cT$ gives an $A$-topology $J_\cT$ on $\cC$.
As is easily checked, we have $\wh{\cT} = \cT$,
where $\wh{\cT}$ is as in Definition 2.3.5 of \cite{Grids}.
Hence it follows from Proposition 2,4,3 of \cite{Grids}
that we have $\cT = \cT(J_\cT)$.

\subsubsection{ }

\begin{lem}
Suppose that $(\cC_i, J_i)$ is a $B$-site (\resp a $Y$-site)
for $i=1,\ldots,r$. Then so is $(\cC,J_{\cT})$.
\end{lem}

\begin{proof}
The claim for a $B$-site is clear since
$(\cC,J_\cT)$ obviously satisfies the three conditions
in Definition 4.2.1 of \cite{Grids}.

We prove the claim of a $Y$-site.
It is easy to check that $(\cC,J_\cT)$ satisfies the
first two conditions in Definition 5.4.2 of \cite{Grids}.
Observe that,
if $f = (f_1,\ldots,f_r)$ is a morphism of $\cC$,
then $f$ is a Galois covering in $\cC$
if and only if $f_i$ is a Galois covering in $\cC_i$
for $i=1,\ldots,r$.
This shows that $(\cC,J_\cT)$ satisfies the last
condition in Definition 5.4.2 of \cite{Grids}.
\end{proof}

\begin{lem}
Suppose that $(\cC_i, J_i)$ is a $Y$-site 
for $i=1,\ldots,r$. Let $(\cC_{i,0},\iota_{i,0})$
be a grid of $(\cC_i,J_i)$. Set
$\cC_0 = \cC_{1,0} \times \cdots \times \cC_{i,r}$
and $\iota_0 = (\iota_{1,0},\cdots,\iota_{r,0})
\colon  \cC_0 \to \cC$. Then the pair
$(\cC_0,\iota_0)$ is a grid of $(\cC,J_\cT)$.
\end{lem}

\begin{proof}
First observe that, if $X = (X_1,\ldots,X_r)$ is an object of
$\cC_0$, then $X$ is an edge object of $\cC_0$ if and only if
$X_i$ is an edge object of $\cC_{i,0}$ for $i=1,\ldots,r$.
This in particular shows that $(\cC,J_\cT)$ satisfies the second
condition in Definition 5.3.3 of \cite{Grids}.

It is easy to see that $(\cC,J_\cT)$ satisfies the remaining 
three conditions in Definition 5.3.3 of \cite{Grids}.
\end{proof}

\subsection{Functors between $Y$-sites}

Let $(\cC,J)$ and $(\cD,K)$ be Grothendieck sites.
As we have recalled in Chapter 1, Section \ref{sec:top_functor}
a functor $F\colon  \cC \to \cD$ is said to 
be cocontinuous
if for any object $X$ of $\cC$ and any sieve $S$ on
$F(X)$ that belongs to $K(F(X))$, there exists a sieve
$R$ on $X$ that belongs to $J(X)$ such that for any object
$f\colon  Y \to X$ of $R$, its image $F(f)$ under $F$ is an object
of $S$. Although the following statement is obvious from the
definition, we record to it as a lemma since we will refer it
several times.

\begin{lem} \label{lem:A_covering-lifting}
Let $(\cC,J)$ and $(\cD,K)$ be categories equipped with $A$-topologies, 
and $F \colon  \cC \to \cD$ a functor. Then $F$ 
is cocontinuous
if and only if
for any object $X$ of $\cC$ and for any
morphism $f'\colon  Y' \to F(X)$ in $\cD$ that belongs to $\cT(K)$,
there exist a morphism $f\colon  Y \to X$ in $\cC$ that belongs to
$\cT(J)$ and a morphism $g'\colon  F(Y) \to Y'$ in $\cD$
satisfying $F(f) = f' \circ g'$.
\qed
\end{lem}

Let $(\cC,J)$ and $(\cD,K)$ be $Y$-sites, and
let $(\cC_0,\iota_0)$ and $(\cD_0,\jmath_0)$ be grids
of $(\cC,J)$ and $(\cD,K)$, respectively.
Let $F\colon \cC \to  \cD$ and $F_0\colon  \cC_0 \to \cD_0$
be functors. We say that $F_0$ is a functor above
$F$ if $F \circ \iota_0$ is naturally isomorphic to
$\jmath_0 \circ F_0$.

\begin{prop} \label{prop:Y_covering-lifting}
Let $(\cC,J)$ and $(\cD,K)$ be $Y$-sites, and
let $(\cC_0,\iota_0)$ and $(\cD_0,\jmath_0)$ be grids
of $(\cC,J)$ and $(\cD,K)$, respectively.
Let $F\colon \cC \to  \cD$ be a functor, and 
$F_0\colon  \cC_0 \to \cD_0$ a functor above $F$.
Suppose that $F_0$ satisfies the following two conditions:
\begin{enumerate}
\item $F_0$ sends any edge object of $\cC_0$
to an edge object of $\cD_0$.
\item For any edge object $X_0$ of $\cC_0$
and for any morphism $Y'_0 \to F_0(X_0)$ in $\cD_0$,
there exists a morphism $Y_0 \to X_0$ in $\cC_0$
such that there exists a morphism from $F_0(Y_0)$ to
$Y'_0$ in $\cD_0$.
\end{enumerate}
Then the functor $F$ is cocontinuous.
\end{prop}

\begin{proof}
Let $X$ be an object of $\cC$ and $f'\colon Y' \to F(X)$
a morphism in $\cD$ that belongs to $\cT(K)$.
Choose an edge object $X_0$ of $\cC_0$ such that
$\iota_0(X_0)$ is isomorphic to $X$ in $\cC$.
Let us choose an isomorphism $\alpha \colon  \iota_0(X_0) \xto{\cong} X$
in $\cC$.
Set $X'_0 = F_0(X_0)$. Then $X'_0$ is an edge object of $\cD_0$.
Since $F_0$ is above $F$, the object $\jmath_0(X'_0)$ of $\cD$
is isomorphic to $F(X)$. Let us choose a functorial isomorphism
$\gamma \colon  \jmath_0 \circ F_0 \xto{\cong} F \circ \iota_0$.
Then the composite $\alpha'= F(\alpha)\circ \gamma(X_0)$
is an isomorphism from $\jmath_0(X_0)$ to $F(X)$ in $\cD$.
Since $(\cD_0, \jmath_0)$ is a grid of $(\cD,K)$,
there exist a morphism $f'_0\colon Y'_0 \to X'_0$ in $\cD_0$
and an isomorphism $\beta' \colon  \jmath_0(Y'_0) \xto{\cong} Y'$
satisfying $\alpha' \circ \jmath_0(f'_0) = f' \circ \beta'$.
It follows from our assumption that
there exists a morphism $f_0\colon  Y_0 \to X_0$ in $\cC_0$
such that there exists a morphism $g'_0 \colon  F_0(Y_0)
\to Y'_0$ in $\cD_0$.
Set $Y=\iota_0(Y_0)$ and $f=\alpha \circ \iota_0(f_0)\colon  Y \to X$.
Since $X_0$ is an edge object of $\cC_0$, the morphism $f$ of
$\cC$ belongs to $\cT(J)$.
By construction, we have $F(f) = f' \circ g'$, where
$g'$ is the composite $g'= \beta \circ \jmath_0(g'_0) \circ \gamma(Y_0)^{-1}\colon 
F(Y) \to Y'$.
Hence it follows from Lemma \ref{lem:A_covering-lifting}
that the functor $F$ is cocontinuous.
\end{proof}

\section{A $Y$-site for a parabolic subgroup of $\GL_d$}

\subsubsection{ }

Let $X$ be a regular noetherian scheme of pure Krull dimension one.
Let $d \ge 1$ be an integer.
Let us consider the category $\cC^d = \cC^d_X$ introduced in 
\ref{sec:Cd_defn}.
The aim of this paragraph is to introduce, for a partition $\bd$ of $d$,
a $Y$-site and its grid whose absolute Galois monoid is isomorphic to
the group of $\A_X$-valued points of the standard parabolic subgroup of
$\GL_d$ corresponding to $\bd$.

\subsubsection{ } 

Let $f\colon M \to N$ be a morphism in $\cC^d$.
Suppose that a decreasing filtration $\Fil^\bullet M$
of $M$ by $\cO_X$-submodules of $M$ is given. 
We define a decreasing filtration 
$(f_* \Fil)^\bullet N$ of $N$ by setting
$$
(f_* \Fil)^j N = p(i^{-1}(\Fil^j M))
$$
for each integer $j$, where we take a representative
$N \stackrel{p}{\twoheadleftarrow} M' 
\stackrel{i}{\hookrightarrow} M$ of the morphism $f$.
As is easily seen, the filtration $(f_* F)^\bullet N$
of $N$ is independent of the choice of a representative of $f$.

\begin{lem} \label{lem:201808_1}
Let
$$
\begin{CD}
M' @>{i'}>> N' \\
@V{p'}VV @VV{p}V \\
M @>{i}>> N
\end{CD}
$$
be a cartesian commutative diagram of coherent $\cO_X$-modules
of finite length.
Then for any $\cO_X$-submodule $N'_1$ of $N'$, we have
$$
i^{-1}(p(N'_1)) = p'(i'^{-1}(N'_1))
$$
as $\cO_X$-submodules of $M$.
\end{lem}

\begin{proof}
Since
$$
i(p'(i'^{-1}(N'_1)))
= p(i'(i'^{-1}(N'_1)))
\subset p(N'_1),
$$
we have $i^{-1}(p(N'_1)) \supset p'(i'^{-1}(N'_1))$.

It remains to prove $i^{-1}(p(N'_1)) \subset p'(i'^{-1}(N'_1))$.
It suffices to show that 
$\Gamma(X,i^{-1}(p(N'_1))) \subset 
\Gamma(X,p'(i'^{-1}(N'_1)))$.
Let us take $x \in \Gamma(X,i^{-1}(p(N'_1)))$.
Then there exists $y \in \Gamma(X,N'_1)$
satisfying $i(x) = p(y)$.
Then the pair $(x,y)$ gives an element
$z \in \Gamma(X,M')$ satisfying $p'(z)=x$ and
$i'(z) = y$. The existence of such $z$ shows that
$x$ belongs to $\Gamma(X,p'(i'^{-1}(N'_1)))$, which
completes the proof.
\end{proof}

\begin{lem} \label{lem:Fil_transitivity}
Let $M_1 \xto{f} M_2 \xto{g} M_3$ be a diagram in $\cC^d$
and let $\Fil^\bullet M_1$ be a decreasing filtration on $M_1$.
Then the two filtrations $(g_*(f_* \Fil))^\bullet M_3$
and $((g\circ f)_* \Fil)^\bullet M_3$ on $M_3$ coincide.
\end{lem}

\begin{proof}
Let $M_2 \stackrel{p}{\twoheadleftarrow} M'_1 
\stackrel{i}{\hookleftarrow} M_1$ and
$M_3 \stackrel{q}{\twoheadleftarrow} M'_2 
\stackrel{j}{\hookleftarrow} M_1$
be diagrams of $\cO_X$-modules representing
$f$ and $g$, respectively.
Then for any integer $k$, we have
$(g_*(f_* \Fil))^k M_3 = q(j^{-1}(p(i^{-1}(\Fil^k M_1))))$.

Let $N$ denote the fiber product of the diagram 
$M'_2 \xto{j} M_2 \xleftarrow{p} M'_1$ in the category of $\cO_X$-modules.
Let $p'\colon N \to M'_2$ and $j'\colon N \to M'_1$ denote the canonical
projection. Then the composite $g \circ f$ is represented by
a diagram $M_3 \stackrel{q \circ p'} {\twoheadleftarrow}
N \stackrel{i \circ j'}{\hookrightarrow} M_1$ of $\cO_X$-modules.
Then for any integer $k$, we have
$((g \circ f)_* \Fil)^k M_3 = q(p'(j'^{-1}(i^{-1}(\Fil^k M_1))))$.
It follows from Lemma \ref{lem:201808_1} that we have
$j^{-1}(p(i^{-1}(\Fil^k M_1))) = p'(j'^{-1}(i^{-1}(\Fil^k M_1)))$.
Hence we have $(g_*(f_* \Fil))^k M_3 = ((g \circ f)_* \Fil)^k M_3$,
which proves the claim.
\end{proof}

\subsubsection{}
For a partition $\bd=(d_1,\ldots,d_m)$,
$d=d_1 +\cdots + d_m$, $d_1,\ldots,d_m\ge 1$
of $d$, let $\cE^{\bd}$ denote the following
category. An object in $\cE^{\bd}$ is an
object $M$ in $\cC^{d}$ endowed with a decreasing
filtration 
$$
M = \Fil^1 M \supset \Fil^2 M \supset \cdots \supset 
\Fil^{m+1}M=0
$$
of $M$ by $\cO_X$-submodules such that
for each $i=1,\ldots,m$, $\gr^i M =\Fil^i M/\Fil^{i+1}M$
is an object in $\cC^{d_i}$. For two objects
$(M,\Fil^{\bullet})$, $(N,\Fil^{\bullet})$ in $\cE^{\bd}$,
a morphism from $(M,\Fil^{\bullet} M)$ to $(N,\Fil^{\bullet})$
is a morphism $f$ from $M$ to $N$ in $\cC^d$ 
such that
the filtration $\Fil^{\bullet} N$ coincides with
the the filtration on $(f_* \Fil)^\bullet N$.

We have the following diagram of categories
$$
\cC^{d_1}\times\cdots\times\cC^{d_m}
\xleftarrow{\gr} \cE^{\bd} \xrightarrow{\ffor} \cC^d,
$$
where, $\gr$ (\resp $\ffor$) denotes the functor
that sends an object $(M,\Fil^{\bullet})$ in $\cE^{\bd}$
to the object $(\Gr^1 M,\ldots,\Gr^m M)$ in 
$\cC^{d_1}\times\cdots\times\cC^{d_m}$ 
(\resp the object $M$ in $\cC^d$).

\subsection{Properties of the functor $\ffor$}
\label{sec:property_ffor}
In this paragraph, we give some basic properties of the
functor $\ffor$. Using some of the properties, we prove that
the category $\cE^\bd$ equipped with the atomic topology is
a $Y$-site.

\begin{lem} \label{lem:Ed_Ecat}
The category $\cE^\bd$ is an $E$-category.
\end{lem}

\begin{proof}
It follows from the construction of the functor 
$\ffor$ is faithful. Since $\cC^d$ is an $E$-category,
the claim follows.
\end{proof}

\begin{lem} \label{lem:for_under_equiv}
Let $(M,\Fil^\bullet M)$ be an arbitrary object of $\cE^\bd$.
Then the functor $\cE^\bd_{(M,\Fil^\bullet M)/}
\to \cC^d_{M/}$ induced by $\ffor$
is an equivalence of categories.
\end{lem}

\begin{proof}
Any morphism $f\colon M \to N$ in $\cC^d$ gives
a morphism $(M,\Fil^\bullet M) \to
(N, (f_*\Fil)^\bullet N)$ in $\cE^\bd$. Let us denote
the latter morphism by $\alpha(f)$. It follows from
Lemma \ref{lem:Fil_transitivity} that any morphism
$g \colon  f_1 \to f_2$ in $\cC^d_{M/}$ is a morphism
from $\alpha(f_1)$ to $\alpha(f_2)$ in $\cE^\bd_{(M,\Fil^\bullet M)/}$. 
Thus, by sending $f$ to $\alpha(f)$ for each object $f$ of
$\cC^d_{M/}$, we obtain a functor $\alpha\colon  \cC^d_{M/}
\to \cE^\bd_{(M,\Fil^\bullet M)/}$. It is easy to check 
that the functor $\alpha$ is the inverse of the functor
$\cE^\bd_{(M,\Fil^\bullet M)/}
\to \cC^d_{M/}$ induced by $\ffor$.
\end{proof}

\begin{lem} \label{lem:for_Galois_lifts}
Let $f\colon (M,\Fil^\bullet M)
\to (N,\Fil^\bullet N)$ be a morphism in $\cE^\bd$.
Then $f$ is a Galois covering in $\cE^\bd$ if and only if 
$\ffor(f)$ is a Galois covering in $\cC^d$.
\end{lem}

\begin{proof}
Let $(M',\Fil^\bullet M')$ be an object of $\cE^\bd$
and $g_1, g_2 \colon  (M', \Fil^\bullet M') \to
(M,\Fil^\bullet M)$ two morphisms in $\cE^\bd$
satisfying $f \circ g_1 = f \circ g_2$.
Since $f$ is a Galois covering in $\cC^d$,
we have $g_2 = \alpha \circ g_1$ for some
$\alpha \in \Gal(f)$.
Since $g_1$ is an epimorphism in $\cE^\bd$,
such a morphism $\alpha$ is unique.
It follows from Lemma \ref{lem:for_under_equiv} that
$\alpha$ is an isomorphism in $\cE^\bd$.
This shows that $f$ is a Galois covering in $\cE^\bd$.
\end{proof}

Let $(\cC^d_0, \iota_0^d)$ be the grid of $\cC^d$
introduced in Section \ref{sec:Pair}.
By definition, an object of $\cC_0^d$ is a pair $(L_1,L_2)$
of $\cO_X$-lattices of $V=\cK_X^{\oplus d}$ 
satisfying $L_1 \subset L_2$, and the functor $\iota_0^d$
sends an object $(L_1,L_2)$ of $\cC_0^d$ to the object
$L_2/L_1$.

Let us consider the following decreasing filtration 
$\Fil^\bullet V$ on $V$: for $i=1,\ldots, r+1$,
we set $\Fil^i V = \cK^{\oplus d_i + \cdots + d_r}$,
which we regard as the direct sum of the
last $d_i + \cdots + d_r$ direct summands of $V$. 
For an object $(L_1,L_2)$ of $\cC_0^d$ and
for $i=1,\ldots, r+1$, we let $\Fil^i (L_2/L_1)$ denote the image of
$\Fil^i V \cap L_2$ under the surjection $L_2 \to L_2/L_1$.
Then the pair $\jmath_0^d(L_1,L_2) 
= (L_2/L_1, \Fil^\bullet (L_2/L_1))$ is an object of
$\cE^\bd$ and, by sending $(L_1,L_2)$ to $\jmath_0^d(L_1,L_2)$
for each object $(L_1,L_2)$ of $\cC_0^d$, we obtain a functor
$\jmath_0^d \colon  \cC_0^d \to \cE^\bd$. By construction
we have $\iota_0^d = \ffor \circ \jmath_0^d$.

In Section \ref{sec:Pair adele},  we introduced a category $\cC^d_{\A,0}$
and an isomorphism $\cC_0^d \xto{\cong} \cC^d_{\A,0}$ of
categories. By taking the composite of the inverse of this isomorphism
with $\jmath_0^d$, we obtain a functor $\cC_{\A,0}^d \to \cE^\bd$
that we denote by $\jmath^d_{\A,0}$. 
Recall that an object of $\cC^d_{\A,0}$ is a pair $(\bL_1,\bL_2)$
of $\wh{\cO}_X$-lattices of $V_\A = \A^d$ with $\bL_1 \subset\bL_2$.
For $i=1,\ldots, r+1$,
we set $\Fil^i V_\A = \A^{\oplus d_i + \cdots + d_r}$,
which we regard as the direct sum of the
last $d_i + \cdots + d_r$ direct summands of $V_\A$.
Let $\wh{X} =\Spec\, \wh{\cO}_X$ and
let $\nu_X \colon  \wh{X} \to X$ be as in 
Chapter 2, Section \ref{sec:nuX}.
For an $\wh{\cO}_X$-module $M$, let $M^\sim$
denote the quasi-coherent $\cO_{\wh{X}}$-module
associated with $M$.
Then. by construction, the functor $\jmath_{\A,0}^d$
is isomorphic to the functor that
sends an object $(\bL_1,\bL_2)$ of $\cC_{\A,0}^d$ to the object
$(\nu_{X,*} (\bL_2/\bL_1)^\sim, \Fil^\bullet \nu_{X,*} (\bL_2/\bL_1)^\sim)$,
where for $i=1,\ldots,r+1$, the $\cO_X$-submodule
$\Fil^i \nu_{X,*} (\bL_2/\bL_1)^\sim$ of $\nu_{X,*} (\bL_2/\bL_1)^\sim$
is obtained by applying $\nu_{X,*}$ to the image of
the composite $(\bL_2 \cap \Fil^i V_\A)^\sim
\to \bL_2^\sim \to (\bL_2/\bL_1)^\sim$.

\begin{lem} \label{lem:jmath_surj}
The functor $\jmath_0^d$ is essentially surjective.
\end{lem}

\begin{proof}
Let $(M,\Fil^\bullet M)$ be an arbitrary object of $\cE^\bd$.
For each $i=1,\ldots,r$, let us choose a surjective homomorphism
$\phi_i \colon  \cO_X^{\oplus d_i} \surj \gr^i M$. Since $M$
is of finite length, the map $\Gamma(X, \Fil^i M) \to \Gamma(X, \gr^i M)$
is surjective. Hence we can take a lift $\wt{\phi}_i \colon  \cO_X^{\oplus d_i}
\to \Fil^i M$ of $\phi_i$.  Let $L_1$ denote the kernel of
the homomorphism $\wt{\phi} = (\wt{\phi}_1, \ldots, \wt{\phi}_r) \colon  
\cO_X^d \to M$
and set $L_2 = \cO_X^{\oplus d}$. Then the pair $(L_1,L_2)$ is
an object of $\cC_0^d$ and the homomorphism $\wt{\phi}$
induces an isomorphism from $\jmath_0^d((L_1,L_2))$ to
$(M,\Fil^\bullet M)$ in $\cE^\bd$.
\end{proof}

The following proposition will be used to
show that the functor $\ffor$ is cocontinuous
with respect to the atomic topologies.

\begin{prop} \label{prop:for_over_surj}
For any object $Y$ of $\cE^\bd$, the functor
$\cE^\bd_{/Y} \to \cC^d_{/\ffor(Y)}$
induced by $\ffor$ is essentially surjective.
\end{prop}

\begin{proof}
Let $f\colon  N \to \ffor(Y)$ be an arbitrary object of $\cC^d_{/\ffor(Y)}$.
It follows from Lemma \ref{lem:jmath_surj} that
there exists an object $Y_0$ of $\cC_0^d$
and an isomorphism $\beta \colon  Y \xto{\cong} \jmath_0^d (Y_0)$
in $\cE^\bd$. 
Since $(\cC_0^d, \iota_0^d)$ is a grid of $\cC^d$, it follows from 
Condition (3) in the definition of grid in \cite[Definition 5.5.3]{Grids} that
there exists a morphism $g_0\colon  Z_0 \to Y_0$ in $\cC_0^d$ such that
$\iota_0^d(g_0)$ is isomorphic to $\ffor(\beta) \circ f$ in
$\cC^d_{/\iota_0(Y_0)}$.
Set $Z = \jmath_0^d(Z_0)$ and
$g = \beta^{-1} \circ \jmath_0^d(g_0) \colon  Z \to Y$.
Then $\ffor(g)$ is isomorphic to $f$ in $\cC^d_{/\ffor(Y)}$.
\end{proof}

\begin{cor} \label{cor:Ed_enough_Galois}
The category $\cE^\bd$ has enough Galois coverings.
\end{cor}

\begin{proof}
Let $f\colon  Z \to Y$ be an arbitrary morphism in $\cE^\bd$.
Since $\cC^d$ has enough Galois coverings, there exists
a morphism $g\colon  W \to \ffor(Z)$ such that the composite
$\ffor(f) \circ g$ is a Galois covering in $\cC^d$.
It follows from Proposition \ref{prop:for_over_surj} that
there exists a morphism $g' \colon  W' \to Z$ in $\cE^\bd$
such that $\ffor(g')$ is isomorphic to $g$ in the category
$\cC^d_{/\ffor(Z)}$. Then $\ffor(f\circ g')$ is a Galois
covering in $\cC^d$ since it is isomorphic to $\ffor(f) \circ g$.
In $\cC^d_{/\ffor(Y)}$.
Hence it follows from Lemma \ref{lem:for_Galois_lifts} that
$f \circ g'$ is a Galois covering in $\cE^\bd$.
\end{proof}

\begin{prop} \label{prop:jmath_over_surj}
For any object $Y_0$ of $\cC_0^d$, the functor
$\cC^d_{/Y_0} \to \cE^\bd_{/\jmath_0^d(Y_0)}$
induced by $\jmath_0^d$ is essentially surjective.
\end{prop}

To prove Proposition \ref{prop:jmath_over_surj}, we need some lemmas.
For an object $M$ of $\cC^d$, we say that $M$ is generated by
$e_1,\ldots,e_s \in \Gamma(X,M)$ if the homomorphism $\cO_X^{\oplus s} \to M$ 
given by $e_1,\ldots, e_s$ is surjective.

\begin{lem} \label{lem:jmath_aux_1}
\begin{enumerate}
\item Let $M$ be an object of $\cC^d$
and $p \colon  M \surj N$ a surjective homomorphism of $\cO_X$-modules.
Suppose that elements $e_1,\ldots, e_d \in \Gamma(X,N)$
which generate $N$ over $\cO_X$ are given.
Then there exist elements $\wt{e}_1, \ldots, \wt{e}_d \in \Gamma(X,M)$
such that $p(\wt{e}_i) = e_i$ for $i=1,\ldots,d$ and
that $M$ is generated over $\cO_X$ by $\wt{e}_1, \ldots, \wt{e}_d$.
\item Let $L$ be an $\wh{\cO}_X$-module generated by at most
$d$ elements, and $p \colon  L \surj N$ a surjective homomorphism of 
$\wh{\cO}_X$-modules.
Suppose that elements $e_1,\ldots, e_d \in N$
which generate $N$ over $\wh{\cO}_X$ are given.
Then there exist elements $\wt{e}_1, \ldots, \wt{e}_d \in L$
such that $p(\wt{e}_i) = e_i$ for $i=1,\ldots,d$ and
that $L$ is generated over $\wh{\cO}_X$ by $\wt{e}_1, \ldots, \wt{e}_d$.
\end{enumerate}
\end{lem}

\begin{proof}
We only give a proof of the claim (1). One can prove the claim (2) 
in a similar manner.

We may assume that $X = \Spec\,R$ for some discrete valuation ring $R$.
We regard $M$ and $N$ as $R$-modules of finite length.
Let us choose a uniformizer $\pi \in R$.
Set $k = R/\pi R$, $s =\dim_k N/\pi N$, and $t = \dim_k M/\pi M$.
By changing the order of $e_1,\ldots,e_d$, we may and will assume that
$e_1 \mod{\pi N}, \ldots, e_s \mod{\pi N}$ is a basis of
$N/\pi N$ over $k$.

Let $M'$ denote the kernel of $p \colon  M \to N$.
We have an exact sequence $M'/\pi M' \to M/\pi M \to N/\pi N \to 0$.
Let us choose $\delta_{s+1}, \ldots, \delta_t \in M'$ such that
the images of $\delta_{s+1} \mod{\pi M'}, \ldots,
\delta_t \mod{\pi M'}$ under the homomorphism
$M'/\pi M' \to M/\pi M$ form a basis over $k$ of the kernel of
$M/\pi M \to N/\pi N$.

It follows from Nakayama's lemma that $N$ is generated over $R$ by
$e_1, \ldots, e_s$. For $i=s+1, \ldots, d$, let us choose
$a_{i,1}, \ldots, a_{i,s} \in R$ satisfying
$e_i = \sum_{j=1}^s a_{i,j} e_j$.

For $i=1,\ldots,s$, choose an arbitrary lift $\wt{e}_i \in M$ of $e_i$.
For $i=s+1, \ldots, t$, set
$\wt{e}_i = \delta_i + \sum_{j=1}^s a_{i,j}\wt{e}_j$.
For $i=t+1, \ldots, d$, set
$\wt{e}_i = \sum_{j=1}^s a_{i,j}\wt{e}_j$.
Then it is straightforward to check that the elements
$\wt{e}_1,\ldots,\wt{e}_d$ have the desired property.
\end{proof}

\begin{lem} \label{lem:jmath_aux_2}
Let $M$ be an object of an abelian category $\cA$,
and $M_1, M_2 \subset M$ two subobjects of $M$.
Then the morphism $M \to (M/M_1) \times_{M/(M_1+M_2)}
(M/M_2)$ induced by the commutative diagram
$$
\begin{CD}
M @>>> M/M_1 \\
@VVV @VVV \\
M/M_1 @>>> M/(M_1+M_2)
\end{CD}
$$
is an epimorphism.
\end{lem}

\begin{proof}
We may and will assume that $\cA$ is the category of
left $R$-modules for some associative ring $R$.
Let $(x,y)$ be an element of $(M/M_1) \times_{M/(M_1+M_2)}
(M/M_2)$. Let us choose representatives $\wt{x}, \wt{y} \in M$
of $x$ and $y$, respectively.
Since $\wt{x} \equiv \wt{y}$ modulo $M_1 + M_2$,
there exists $x' \in M_1$ and $y' \in M_2$ satisfying
$\wt{x} - \wt{y} = x'+y'$. Set $z = \wt{x} - x' = \wt{y} + y' \in M$.
Then the image of $z$ under the morphism
$M \to (M/M_1) \times_{M/(M_1+M_2)}
(M/M_2)$ is equal to $(x,y)$. This completes the proof.
\end{proof}

\begin{lem} \label{lem:jmath_aux3}
\begin{enumerate}
\item Let $M$ be an object of $\cC^d$
and $p \colon  M \surj N$ a surjective homomorphism of $\cO_X$-modules.
Let $M' \subset M$ and $N' \subset N$ be $\cO_X$-submodules
satisfying $p(M') =N'$.
Let $s$ be an integer with $1 \le s \le d$.
Suppose that elements $\wt{e}_1,\ldots, \wt{e}_s \in \Gamma(X,M')$
and $e_{s+1}, \ldots, e_d \in \Gamma(X,N)$ are given in such a way that
$M'$ is generated over $\cO_X$ by 
$\wt{e}_1, \ldots, \wt{e}_s$.
and that $N$ is generated over $\cO_X$ by $N'$ and $e_{s+1}, \ldots, e_d$.
Then there exist elements $\wt{e}_{s+1}, \ldots, \wt{e}_d \in \Gamma(X,M)$
such that $p(\wt{e}_i) = e_i$ for $i=s+1,\ldots,d$ and
that $M$ is generated over $\cO_X$ by $\wt{e}_1, \ldots, \wt{e}_d$.
\item Let $L$ be an $\wh{\cO}_X$-module generated by at most $d$ elements
and $p \colon  M \surj N$ a surjective homomorphism of $\wh{\cO}_X$-modules.
Let $L' \subset L$ and $N' \subset N$ be $\wh{\cO}_X$-submodules
satisfying $p(L') =N'$.
Let $s$ be an integer with $1 \le s \le d$.
Suppose that elements $\wt{e}_1,\ldots, \wt{e}_s \in L'$
and $e_{s+1}, \ldots, e_d \in N$ are given in such a way that
$L'$ is generated over $\wh{\cO}_X$ by 
$\wt{e}_1, \ldots, \wt{e}_s$.
and that $N$ is generated over $\wh{\cO}_X$ by $N'$ and $e_{s+1}, \ldots, e_d$.
Then there exist elements $\wt{e}_{s+1}, \ldots, \wt{e}_d \in L$
such that $p(\wt{e}_i) = e_i$ for $i=s+1,\ldots,d$ and
that $L$ is generated over $\wh{\cO}_X$ by $\wt{e}_1, \ldots, \wt{e}_d$.
\end{enumerate}
\end{lem}

\begin{proof}
We only give a proof of the claim (1). One can prove the claim (2) 
in a similar manner.

By applying Lemma \ref{lem:jmath_aux_1} to the surjective
homomorphism $p' \colon  M/M' \to N/N'$ induced by $p$, we may find elements
$\wt{e}'_{s+1}, \ldots, \wt{e}'_d \in \Gamma(X,M/M')$ 
such that $p'(\wt{e}'_i) = e_i \mod{\Gamma(X,N')}$
for $i=s+1,\ldots, d$ and that
$M/M'$ is generated over $\cO_X$ by $\wt{e}_{s+1}, \ldots, \wt{e}_d$.

For $i=s+1, \ldots, d$, the pair $(e_i, \wt{e}'_i)$ gives
an element of $\Gamma(X, N\times_{N/N'} (M/M'))$,
where $N\times_{N/N'} (M/M')$ is the fiber product of
$N$ and $M/M'$ over $N/N'$ in the category of $\cO_X$-modules.
It follows from Lemma \ref{lem:jmath_aux_2} that
the homomorphism $M \to N\times_{N/N'} (M/M')$
of $\cO_X$-modules is surjective. Hence there exists
an element $\wt{e}_i \in \Gamma(X,M)$ satisfying
$p(\wt{e}_i) = e_i$ and $\wt{e}_i \mod{\Gamma(X,M')} = \wt{e}'_i$.
Then the elements $\wt{e}_{s+1}, \ldots,\wt{e}_d$
of $\Gamma(X,M)$ have the desired property.
\end{proof}

\begin{proof}[Proof of Proposition  \ref{prop:jmath_over_surj}]
It suffices to prove similar statements for $(\cC_{\A,0}^d,
\jmath_{\A,0}^d)$.
Let $Y_0$ be an object of $\cC_{\A,0}^d$, and let
$f\colon  Z \to \jmath_0^d(Y_0)$ be a morphism in $\cE^\bd$.
We prove that there exists a morphism $f_0 \colon  Z_0 \to Y_0$
in $\cC_{\A,0}^d$ and an isomorphism $\beta\colon  \jmath_0^d (Z_0)
\xto{\cong} Z$ satisfying $\jmath_0^d (f_0) = f \circ \beta$.
We are easily reduced to the cases where $f$ is either 
a fibration or a cofibration in the sense of 
Chapter 2, Section \ref{sec:modelstr}.
Using the dual functor $\bD_0$ introduced in
Chapter 2, Section \ref{sec:dual},
one can reduce the case where $f$ is 
a cofibration to the case where $f$ is a fibration.
Thus we may assume that $f$ is a fibration.
Let us write $Z=(M,\Fil^\bullet M)$, $Y_0 = (\bL_1,\bL_2)$
and $\jmath_0^d(Y_0) = (N, \Fil^\bullet N)$.
Then $f$ is represented by a diagram $N \stackrel{p}{\twoheadleftarrow}
M \stackrel{\id_M}{\hookrightarrow} M$
of $\cO_X$-modules for some surjective homomorphism
$p\colon  M \to N$ of $\cO_X$-modules.

Let us choose $e_1,\ldots,e_d \in \Gamma(X,N)$ in such a way
that for $i=1,\ldots,r$, the elements
$e_{d_1+\cdots+d_{i-1}+1} \mod{\Gamma(X,\Fil^{i+1}N)}$,
$\ldots$, $e_{d_1+\cdots+d_{i}} \mod{\Gamma(X,\Fil^{i+1}N)}$
generate $\Gr^{i} N$ over $\cO_X$.
By applying Lemma \ref{lem:jmath_aux3} repeatedly,
one can find elements
$\wt{e}_1,\ldots,\wt{e}_d \in \Gamma(X,M)$ such that
$p(\wt{e}_i)= e_i$ for $i=1,\ldots,d$ and that
for $i=1,\ldots,r$, the elements
$\wt{e}_{d_1+\cdots+d_{i-1}+1} \mod{\Gamma(X,\Fil^{i+1}M)}$,
$\ldots$, $\wt{e}_{d_1+\cdots+d_{i}} \mod{\Gamma(X,\Fil^{i+1}M)}$
generate $\Gr^{i} M$ over $\cO_X$.

By construction, we have a canonical isomorphism
$\gamma \colon  \Gamma(X,N) \cong \bL_2/\bL_1$ of $\wh{\cO}_X$-modules,
where we regard $\Gamma(X,N)$ as an $\wh{\cO}_X$-module
in the way as was explained in Lemma \ref{lem:XN}.
For $i=1,\ldots,r$, let $\Fil^i(\bL_2/\bL_1)$ denote
the image of $\bL_2 \cap \Fil^i V_\A$ under the
homomorphism $\bL_2 \to \bL_2/\bL_1$.
Then $\Gr^i (\bL_2/\bL_1)$ is generated over $\wh{\cO}_X$
by $\gamma(e_{d_1+\cdots+d_{i-1}+1}) \mod \Fil^{i+1}(\bL_2/\bL_1)$, 
$\ldots$, $\gamma(e_{d_1+\cdots+d_{i-1}+1}) \mod \Fil^{i+1}(\bL_2/\bL_1)$.

By applying Lemma \ref{lem:jmath_aux3} repeatedly,
one can find elements $\delta_1,\ldots,\delta_d \in \bL_2$
such that $\gamma(e_i) = \delta_i \mod{\bL_1}$ for
$i=1,\ldots,d$ and that 
for $i=1,\ldots,r$, the elements
$\delta_{d_1+\cdots+d_{i-1}+1} \mod{\bL_2 \cap \Fil^{i+1} V_\A}$,
$\ldots$, $\delta_{d_1+\cdots+d_{i}} \mod{\bL_2 \cap \Fil^{i+1} V_\A}$
generate $(\bL_2 \cap \Fil^i V_\A)/(\bL_2 \cap \Fil^{i+1} V_\A)$ 
over $\wh{\cO}_X$.
This in particular implies that
$\bL_2$ is a free $\wh{\cO}_X$-module of rank $d$ with basis
$\delta_1, \ldots, \delta_d$.
Let us consider the homomorphism $\phi\colon \bL_2 \to \Gamma(X,M)$
of $\wh{\cO}_X$-modules that sends $\delta_i$ to
$\wt{e}_i$ for $i=1,\ldots,d$.
Let $\bL'_1$ denote the kernel of $\phi$ and set
$Z_0 = (\bL'_1,\bL_2)$. Then $Z_0$ is an object of $\cC_{\A,0}^d$
and there exists a unique morphism $f_0\colon Z_0 \to Y_0$.
Since $\phi$ is surjective, $\phi$ induces
an isomorphism $\bL_2/\bL'_1 \cong \Gamma(X,M)$ of $\wh{\cO}_X$-modules,
that gives an isomorphism $\beta\colon \iota_{\A,0}^d(Z_0) \cong M$ in $\cC^d$.
It follows from the construction that 
for $i=1,\ldots,d+1$, the image of $\bL_2 \cap \Fil^i V_\A$
under $\phi$ is equal to $\Fil^i M$.
This implies that $\beta$ is an isomorphism from $\jmath_{\A,0}^d(Z_0)$
to $Z$ in $\cE^\bd$. Since it is straightforward to check that
$\jmath_{\A,0}^d(f_0) = f \circ \beta$, the claim is proved.
\end{proof}

\begin{cor} \label{cor:Ed_semi-cofiltered}
The category $\cE^\bd$ is semi-cofiltered.
\end{cor}

\begin{proof}
Let $Y_1 \xto{f_1} Y_3 \xleftarrow{f_2} Y_2$ be
a diagram in $\cE^\bd$. It follows from 
Lemma \ref{lem:jmath_surj} that there exists
an object $Y_{0,3}$ and an isomorphism
$\beta \colon  \jmath_0^d(Y_{0,3}) \xto{\cong} Y_3$ in
$\cE^\bd$. It then follows from Proposition \ref{prop:jmath_over_surj}
that, for $i=1,2$, there exist a morphism $f_{0,i} \colon  Y_{0,i} \to Y_{0,3}$
in $\cC^d_0$ and an isomorphism $\beta_i \colon  \jmath_0^d(Y_{0,i}) \xto{\cong} Y_i$
in $\cE^\bd$ satisfying $f_i \circ \beta_i = \beta \circ \jmath_0^d (f_{0,i})$. 

Since $\cC_0^d$ is $\Lambda$-connected, there exist an object
$Z_0$ of $\cC_0^d$ and morphisms $g_{0,i} \colon  Z_0 \to Y_{0,i}$
for $i=1,2$. We set $Z = \jmath_0^d(Z_0)$ and
$g_i = \beta_i \circ \jmath_0^d(g_{0,i}) $ for $i=1,2$.
Since $\cC_0^d$ is a poset, we have
$f_{0,1} \circ g_{0,1} = f_{0,2} \circ g_{0,2}$.
This implies that $f_1 \circ g_1 = f_2 \circ g_2$.
Thus $\cE^\bd$ is semi-cofiltered. 
\end{proof}

\begin{prop} \label{prop:Ed_Y-site}
$\cE^\bd$ equipped with the atomic topology is a $Y$-site.
\end{prop}

\begin{proof}
It follows from Corollary \ref{cor:Ed_semi-cofiltered} that
one can consider the atomic topology on $\cE^\bd$.
It follows from Lemma \ref{lem:Ed_Ecat} that
$\cE^\bd$, equipped with the atomic topology, is a $B$-site.

It is clear from the definition that the category $\cE^\bd$
is essentially small.
Let $\Fil^\bullet 0$ denote the unique decreasing
filtration on the zero object $0$ in the category of $\cO_X$-modules.
Since any object $Y$ of $\cE^\bd$ admits a
(not necessarily unique) morphism to $(0, \Fil^\bullet 0)$,
It follows from Corollary \ref{cor:Ed_semi-cofiltered} that
$\cE^\bd$ is $\Lambda$-connected. 
Hence it follows from Corollary \ref{cor:Ed_enough_Galois}
that $\cE^\bd$, equipped with the atomic topology, is a $Y$-site.
\end{proof}

\begin{lem} \label{lem:Ed_grid}
The pair $(\cC_0^d, \jmath_0^d)$ is a grid of $\cE^\bd$
equipped with the atomic topology.
\end{lem}

\begin{proof}
It suffices to show that $(\cC_0^d,\jmath_0^d)$ satisfies the four
conditions in the definition of a pregrid given in
\cite[Definition 5.5.1]{Grids}. Since $(\cC_0^d, \iota_0^d)$ is a pregrid
of $\cC^d$, Condition (1) is satisfied and Condition (4) follows from
Lemma \ref{lem:for_under_equiv}. 
It follows from Lemma \ref{lem:jmath_surj}
that Condition (2) is satisfied. It follows from Proposition \ref{prop:jmath_over_surj}
that Condition (3) is satisfied.
\end{proof}

\begin{prop} \label{prop:for_covering_lifting}
The functor $\ffor$ is cocontinuous (cf.\ Section \ref{sec:top_functor})
with respect to the atomic topologies.
\end{prop}

\begin{proof}
This follows from Proposition \ref{prop:for_over_surj} 
and Proposition \ref{prop:Y_covering-lifting}.
\end{proof}

Let $(-) \circ \ffor \colon  \Presh(\cC^d) \to \Presh(\cE^\bd)$
denote the functor given by the composite with $\ffor$.
Since $\cE^\bd$ is essentially small and the category of sets
has small limits and colimits, it follows from \cite[Expose I]{SGA4}
that the functor $(-) \circ \ffor$ has both a left
adjoint and a right adjoint.

Let $J^d$, $J^\bd$ denote the atomic topologies on
$\cC^d$, and $\cE^\bd$, respectively.
Let $\ffor^* \colon  \Shv(\cC^d,J^d) \to \Shv(\cE^\bd,J^\bd)$
denote the composite of the inclusion functor
$\Shv(\cC^d,J^d) \inj \Presh(\cC^d)$, the functor
$(-)\circ \ffor$, and the sheafification functor
$a_{J^\bd}\colon  \Presh(\cE^\bd) \to \Shv(\cE^\bd,J^\bd)$.

The following statements are consequences of 
Proposition \ref{prop:for_covering_lifting} 
together with \cite[III, Prop.\ 2.3 and IV, 4.7]{SGA4}
or \cite[C2.3.18]{Johnstone}.

\begin{cor} \label{cor:geometric1}
\begin{enumerate}
\item The right adjoint of the functor 
$(-) \circ \ffor$ sends a sheaf to a sheaf.
\item Let $\ffor_* \colon  \Shv(\cE^\bd,J^\bd) \to \Shv(\cC^d,J^d)$
denote the functor induced by a right adjoint of $(-)\circ\ffor$.
Then $\ffor_*$ is right adjoint to the functor $\ffor^*$.
\item The pair $(\ffor^*,\ffor_*)$ is a geometric morphism
from $\Shv(\cE^\bd,J^\bd)$ to $\Shv(\cC^d,J^d)$.
\end{enumerate}
\qed
\end{cor}

\subsection{Properties of the functor $\gr$}
In this paragraph, we give some basic properties of the
functor $\gr$. 

\begin{lem} \label{lem:gr_surj}
The functor $\gr$ is essentially surjective.
\end{lem}

\begin{proof}
Let $(M_1, \ldots, M_r)$ be an arbitrary object of 
$\cC^{d_1} \times \cdots \times \cC^{d_r}$.
We set $M = M_1 \oplus \cdots \oplus M_r$ and
$\Fil^i M = M_i \oplus \cdots \oplus M_r$ for
$i=1,\ldots, r+1$. Then
$(M,\Fil^\bullet M)$ is an object of $\cE^\bd$ and
$\gr((M,\Fil^\bullet M))$ is isomorphic to
$(M_1,\ldots,M_r)$ in $\cC^{d_1} \times \cdots
\times \cC^{d_r}$. This proves the claim.
\end{proof}
Let
$\gr_0 \colon  \cC_0^d \to \cC_0^{d_1} \times \cdots \times \cC_0^{d_r}$
denote the functor that sends
$(L_1,L_2)$ to $((\Gr^1 L_1, \Gr^1 L_2),
\ldots, (\Gr^r L_1, \Gr^r L_2))$,
where $\Gr^i L_j = (\Fil^i V \cap L_j)/(\Fil^{i+1} V \cap L_j)$.
We regard $\Gr^i L_j$ as 
an $\cO_X$-lattice of $\cK_X^{d_i}$ via the injection $\Gr^i L_j \inj \Fil^i V/\Fil^{i+1} V \cong \cK_X^{d_i}$.

The diagram
$$
\begin{CD}
\cC_0^d @>{\gr_0}>> \cC_0^{d_1} \times \cdots \times \cC_0^{d_r} \\
@V{\jmath_0^d}VV @VV{\iota_0^{d_1} \times \cdots \times \iota_0^{d_r}}V \\
\cE^\bd @>{\gr}>> \cC^{d_1}\times \cdots \times \cC^{d_r}
\end{CD}
$$
of categories is commutative up to
natural equivalences.

\begin{lem} \label{prop:gr0_over_surj}
The functor $\gr_0$ satisfies the two conditions in
Proposition \ref{prop:Y_covering-lifting}.
\end{lem}

\begin{proof}
Condition (1) is automatically satisfied 
since the topology on $\cC^d$ is atomic.
We prove that Condition (2) is satisfied.

%
Let $Y_0$ be an arbitrary object of $\cC^d_0$.
Let $Z'_0 \to \gr_0(Y_0)$ be a morphism in 
$\cC_0^{d_1} \times \cdots \times \cC_0^{d_r}$.
It suffices to show that there exists a morphism
$Z_0 \to Y_0$ in $\cC^d_0$ such that there
exists a morphism from $\gr_0(Z_0)$ to
$Z'_0$ in 
$\cC_0^{d_1} \times \cdots \times \cC_0^{d_r}$.

Let us write $Y_0 = (L_1,L_2)$ and
$Z'_0 = ((L'_{1,1},L'_{1,2}), \ldots,
(L'_{r,1},L'_{r,2}))$.
For each $i=1,\ldots,r$, let us choose a sufficiently small
$\cO_X$-lattice $L_{i,1} \subset \cK_X^{\oplus d_i}$
and a sufficiently large
$\cO_X$-lattice $L_{i,2} \subset \cK_X^{\oplus d_i}$
satisfying 
$$
L_{1,1} \oplus \cdots \oplus L_{r,1}
\subset L_1 \subset L_2
\subset
L_{1,2} \oplus \cdots \oplus L_{r,2}.
$$
Then the pair
$$
Z_0 = ((L_{1,1}\cap L'_{1,1})\oplus
\cdots \oplus (L_{r,1} \cap L'_{r,1}),
(L_{1,2} + L'_{1,2}) \oplus
\cdots \oplus (L_{r,1}+L'_{r,2}))
$$
satisfy the desired property.
\end{proof}

\begin{prop} \label{prop:gr_covering_lifting}
The functor $\gr$ is cocontinuous (cf.\ Section \ref{sec:top_functor})
with respect to the atomic topologies.
\end{prop}

\begin{proof}
This follows from Proposition \ref{prop:gr0_over_surj} and
Proposition \ref{prop:Y_covering-lifting}.
\end{proof}

Let $(-) \circ \gr \colon  \Presh(\cC^{d_1}\times \cdots \times \cC^{d_r}) 
\to \Presh(\cE^\bd)$
denote the functor given by the composite with $\gr$.
Since $\cE^\bd$ is essentially small and the category of sets
has small limits and colimits, it follows from \cite[Expose I]{SGA4}
that the functor $(-) \circ \gr$ has both a left
adjoint and a right adjoint.

Let $J$ denote the atomic topology on
$\cC^{d_1} \times \cdots \times \cC^{d_r}$.
Let $\gr^* \colon  \Shv(\cC^{d_1} \times \cdots \times \cC^{dr},J) 
\to \Shv(\cE^\bd,J^\bd)$
denote the composite of the inclusion functor
$\Shv(\cC^{d_1} \times \cdots \times \cC^{d_r},J) 
\inj \Presh(\cC^d)$, the functor
$(-)\circ \gr$, and the sheafification functor
$a_{J^\bd}\colon  \Presh(\cE^\bd) \to \Shv(\cE^\bd,J^\bd)$.

The following statements are consequences of 
Proposition \ref{prop:for_covering_lifting} 
together with \cite[III, Prop.\ 2.3 and IV, 4.7]{SGA4}
or \cite[C2.3.18]{Johnstone}.

\begin{cor} \label{cor:geometric2}
\begin{enumerate}
\item The right adjoint of the functor 
$(-) \circ \gr$ sends a sheaf to a sheaf.
\item Let $\gr_* \colon  \Shv(\cE^\bd,J^\bd) \to 
\Shv(\cC^{d_1} \times \cdots \times \cC^{d_r},J)$
denote the functor induced by a right adjoint of $(-)\circ\gr$.
Then $\gr_*$ is right adjoint to the functor $\gr^*$.
\item The pair $(\gr^*,\gr_*)$ is a geometric morphism
from $\Shv(\cE^\bd,J^\bd)$ to 
$\Shv(\cC^{d_1} \times \cdots \times \cC^{d_r},J)$.
\end{enumerate}
\qed
\end{cor}

\begin{defn}
Let $f \colon  M \to N$ be a morphism in $\cC^d$ represented by
the diagram $N \stackrel{p}{\twoheadleftarrow} M' \stackrel{i}{\inj} M$
of $\cO_X$-modules. We say that $f$ is Galois of special type if
the following two conditions are satisfied:
\begin{enumerate}
\item As an $\cO_X$-module, 
$M$ is isomorphic to $M_1^{\oplus d}$ for some object $M_1$ of $\cC^1$.
\item There exists a coherent ideal sheaf $I \subset \cO_X$ satisfying
$i(\Ker\, p) \subset IM \subset i(M')$.
\end{enumerate}
\end{defn}

We note that, if $f \colon  M \to N$ is Galois of special type, then
it follows from Lemma \ref{lem:Gal} that $f$ is a Galois covering in $\cC^d$.

\begin{prop} \label{prop:gr_Galois_surj}
Let $f\colon  Z \to Y$ be a morphism in $\cE^\bd$
such that $\ffor(f)$ is Galois of special type.
Then $f$ and $\gr(f)$ are Galois coverings in 
$\cE^\bd$ and $\cC^{d_1} \times
\cdots \times \cC^{d_r}$, respectively, and the
homomorphism $\Gal(f) \to \Gal(\gr(f))$ is
surjective.
\end{prop}

To prove Proposition \ref{prop:gr_Galois_surj}, we need 
the following lemma.

\begin{lem} \label{lem:gr_Galois_surj_aux1}
Let $f\colon Z \to Y$ be a morphism in $\cE^\bd$. 
Suppose that $\ffor(f)$ is a cofibration and is Galois of special type.
Then the assertion of Proposition \ref{prop:gr_Galois_surj}
is true, i.e., $f$ and $\gr(f)$ are Galois coverings in 
$\cE^\bd$ and $\cC^{d_1} \times
\cdots \times \cC^{d_r}$, respectively, and the homomorphism 
$\Gal(f) \to \Gal(\gr(f))$ is surjective.
\end{lem}

\begin{proof}
First, we note that Lemma \ref{lem:for_Galois_lifts}
implies that $f$ is a Galois covering in $\cE^\bd$.
Let us write $\gr(f) = (f_1,\ldots,f_r)$.
Since $f$ is Galois of special type, it follows from the
definition that $f_i$ is Galois of special type as a morphism in
$\cC^{d_i}$ for $i=1,\ldots,d$.
This in particular implies that 
$\gr(f)$ is a Galois covering in $\cC^{d_1} \times
\cdots \times \cC^{d_r}$,

It remains to prove that the homomorphism 
$\Gal(f) \to \Gal(\gr(f))$ is surjective.
We are easily reduced to the case where $X$ is the spectrum
of a discrete valuation ring.
By induction on the length $r$ of the filtration, we may 
moreover assume that $r=2$.
Then the claim follows from Lemma \ref{lem:gr_Galois_surj_aux2} below.
\end{proof}

\begin{lem} \label{lem:gr_Galois_surj_aux2}
Let $R$ be a discrete valuation ring, $\wp \subset R$ the maximal ideal,
$M$ an $R$-module of finite length, 
and $N, M_2 \subset M$ two submodules. Set $M_1 = M/M_2$,
$N_2 = N \cap M_2$, and let $N_1$ denote the image of $N$
under the surjection $M \surj M_1$.
Let $d_1,d_2 \ge 1$ be two integers and set $d=d_1 + d_2$.
Suppose that there exists an integer $n \ge 0$ such that
$M_1$, $M_2$, and $M$ are isomorphic to $(R/\wp^n)^{\oplus d_1}$,
$(R/\wp^n)^{\oplus d_2}$, and $(R/\wp^n)^{\oplus d}$, respectively.
For $i=1,2$, let $\alpha_i$ be an automorphism of $M_i$
satisfying $\alpha_i(x) =x$ for any $x \in N_i$.
Then there exists an automorphism $\alpha$ of $M$
such that $\alpha(M_2) = M_2$, $\alpha(x)=x$ for
any $x \in N$, and that for $i=1,2$,
the automorphism of $M_i$ induced by $\alpha$ is
equal to $\alpha_i$.
\end{lem}

\begin{proof}
Let us choose a uniformizer $\pi \in \wp$.
It follows from the theory of elementary divisors that
there exist a basis $x_1, \ldots, x_{d_1}$ of
the $R/\wp^n$-module $M_1$ and integers
$0 \le m_1, \ldots, m_{d_1} \le n$ such that
$\pi^{m_1}x_1 ,\ldots, \pi^{m_{d_1}} x_{d_1}$
generate $N_1$ as an $R$-module.
For $i=1,\ldots,d_1$, choose an element $y_i \in N$
whose image under the surjection $N \surj N_1$ is
equal to $\pi^{m_i} x_i$.

It follows from the assumption that the short exact sequence
$$
0 \to M_2 \to M \to M_1 \to 0
$$
splits. Let us fix a splitting $M_1 \to M$ and
regard $M$ as the direct sum $M_1 \oplus M_2$.
For $i=1,\ldots,d_1$, let us write
$y_i = (\pi^{m_i} x_i, y'_i)$ where $y'_i \in M_2$.
Then, since $\pi^n x_i =0$, we have $\pi^{n-m_i} y'_i \in N_2$.
Hence we have $\alpha_2(\pi^{n-m_i} y'_i) = \pi^{n-m_i} y'_i$.
This implies that $y'_i - \alpha_2(y'_i) \in \pi^{m_i} M_2$.
Choose an element $z_i \in M_2$ satisfying 
$y'_i - \alpha_2(y'_i) \in \pi^{m_i} z_i$.
Let $\beta \colon  M_1 \to M_2$ denote the homomorphism of $R$-modules
that sends $x_i$ to $z_i$ for $i=1,\ldots,d_1$.
Let $\alpha \colon  M_1 \oplus M_2 \to M_1 \oplus M_2$ 
denote the endomorphism of the $R$-module
that sends $(x,y) \in M_1 \oplus M_2$ to 
$(\alpha_1(x),\alpha_2(y)+\beta(x))$.
Then $\alpha$ is an automorphism since
$\alpha(M_2) \subset  M_2$ and the endomorphisms of
$M_2$ and $M_1 = (M_1 \oplus M_2)/M_2$ induced by $\alpha$
are equal to $\alpha_2$ and $\alpha_1$, respectively.
It follows from the definition of $\beta$ that
we have $\alpha(\pi^{m_i} x_i, y'_i) = (\pi^{m_i} x_i, y'_i)$
for $i=1,\ldots,d_1$.
Thus, if we regard $\alpha$ as an automorphism of $M$, then
$\alpha$ satisfies the desired property.
\end{proof}

\begin{proof}[Proof of Proposition \ref{prop:gr_Galois_surj}]
By the argument in the proof of Lemma \ref{lem:gr_Galois_surj_aux1},
we are reduced to showing that
$\Gal(f) \to \Gal(\gr(f))$ is surjective in the case when
$X$ is local and $r=2$.
Then the claim follows from Lemma \ref{lem:gr_Galois_surj_aux4} below.
\end{proof}

\begin{lem} \label{lem:gr_Galois_surj_aux4}
Let $R$ be a discrete valuation ring, $\wp \subset R$ the maximal ideal,
$M$ an $R$-module of finite length, 
and $M',N, M_2 \subset M$ three submodules satisfying
$N \subset M'$. Set $M_1 = M/M_2$, $M'_1 = M \cap M_2$
$N_2 = N \cap M_2$, and let $M'_1$, $N_1$ denote the images of 
$M'$ and $N$, respectively, under the surjection $M \surj M_1$.
Let $d_1,d_2 \ge 1$ be two integers and set $d=d_1 + d_2$.
Suppose that there exist integers $n \ge n' \ge 0$ such that
$N \subset \wp^{n'} M \subset M'$ and that
$M_1$, $M_2$, and $M$ are isomorphic to $(R/\wp^n)^{\oplus d_1}$,
$(R/\wp^n)^{\oplus d_2}$, and $(R/\wp^n)^{\oplus d}$, respectively.
For $i=1,2$, let $\alpha_i$ be an automorphism of $M_i$
satisfying $\alpha(M'_i) = M'_i$,
$\alpha(N_i) = N_i$, and $\alpha_i(x) \equiv x$ modulo 
$N_i$ for any $x \in M'_i$.
Then there exists an automorphism $\alpha$ of $M$
such that $\alpha(M_2) = M_2$, $\alpha(M')=M'$,
$\alpha(N)=N$ and $\alpha(x) \equiv x$ modulo $N$ for
any $x \in M'$, and that for $i=1,2$,
the automorphism of $M_i$ induced by $\alpha$ is
equal to $\alpha_i$.
\end{lem}

\begin{proof}
Set $\Mbar = M/\wp^{n'} M$,
$\Mbar_1 = M_1/\wp^{n'} M_1$, and
$\Mbar_2 = M_2/\wp^{n'} M_2$.
We then have the short exact sequence
$$
0 \to \Mbar_2 \to \Mbar \to \Mbar_1 \to 0.
$$
Via this short exact sequence, we regard $\Mbar_2$
as an $R$-submodule of $\Mbar$.
By applying Lemma \ref{lem:gr_Galois_surj_aux2} to
the $R$-module $\Mbar$ and its two submodules
$M'+\wp^{n'} M/\wp^{n'} M$ and 
$\Mbar_2$, we obtain an automorphism
$\overline{\alpha}\colon  \Mbar \to \Mbar$
of the $R$-module $\Mbar$ such that
$\alpha'(\Mbar_2) = \Mbar_2$ and that the
automorphisms of $\Mbar_1$ and $\Mbar_2$
induced by $\overline{\alpha}$ are equal to
$\alpha_1$ modulo $\wp^{n'} M_1$ and
$\alpha_2$ modulo $\wp^{n'} M_2$, respectively.

It follows from the assumption that the short exact sequence
$$
0 \to M_2 \to M \to M_1 \to 0
$$
splits. Let us fix a splitting $M_1 \to M$ and
regard $M$ as the direct sum $M_1 \oplus M_2$.
Then $\Mbar$ is regarded as the direct sum 
$\Mbar_1 \oplus \Mbar_2$.
Let $\overline{\beta} \colon  \Mbar_1 \to \Mbar_2$ 
denote the composite
$$
\Mbar_1 \inj \Mbar_1 \oplus \Mbar_2 \cong
\Mbar \xto{\overline{\alpha}} \Mbar
\cong \Mbar_1 \oplus \Mbar_2 \surj \Mbar_2
$$
and choose an arbitrary $R$-linear lift $\beta'\colon  M_1 \to M_2$
of $\overline{\beta}$. Then for any $(x,y) \in M' 
\subset M = M_1 \oplus M_2$, we have 
$(\alpha_1(x), \alpha_2(y) + \beta'(x)) \equiv (x,y)$
modulo $\pi^{n'} M$.
Using an argument similar to that in the beginning of the 
proof of Lemma \ref{lem:gr_Galois_surj_aux2},
we choose a basis $x_1, \ldots, x_{d_1}$ of
the $R/\wp^n$-module $M_1$ and integers
$0 \le m_1, \ldots, m_{d_1} \le n'$ such that
$\pi^{m_1}x_1 ,\ldots, \pi^{m_{d_1}} x_{d_1}$
generates $M'_1$ as an $R$-module.
For $i=1,\ldots,d_1$, choose an element $y_i \in M'$
whose image under the surjection $M' \surj M'_1$ is
equal to $\pi^{m_i} x_i$.
Let us write
$y_i = (\pi^{m_i} x_i, y'_i)$ where $y'_i \in M_2$.
Then we have $\alpha_2(y'_i) + \beta'(\pi^{m_i} x_i)
\equiv y'_i$ modulo $\pi^{n'} M_2$.
Let us choose $w_i \in M_2$ satisfying
$y'_i - \alpha_2(y'_i) = \beta'(\pi^{m_i} x_i) + \pi^{n'}w_i$
By assumption, we have $\pi^{m_i} x_i - \alpha_1(\pi^{m_i} x_i)
\in N_1$. Let us choose an element $z_i \in N$
whose image under the surjection $N \surj N_1$ is
equal to $\pi^{m_i} x_i - \alpha_1(\pi^{m_i} x_i)$.
Let us write
$z_i = (\pi^{m_i} x_i- \alpha_1(\pi^{m_i} x_i), z'_i)$ 
where $z'_i \in M_2$.
Since $z_i \in N$, we have $z'_i \in \pi^{n'} M_2$.
Let us choose $z''_i \in M_2$ satisfying 
$z'_i = \pi^{n'} z''_i$.

Let $\beta''\colon M_1 \to M_2$ denote the homomorphism
of $R$-modules that sends $x_i$ to
$\pi^{n'-m_i}(w_i - z''_i)$ for $i=1,\ldots,d_1$.
Set $\beta = \beta' + \beta''$.
Then, by definition, we have
\begin{equation} \label{eq:gr_eq1}
(\alpha_1(\pi^{m_i} x_i), \alpha_2(y'_i) + \beta(\pi^m_i x_i))
=(\pi^{m_i} x_i, y'_i) - (\pi^{m_i} x_i- \alpha_1(\pi^{m_i} x_i), z'_i)
= y_i - z_i
\end{equation}
for $i=1,\ldots,d_1$.
Let $\alpha \colon  M_1 \oplus M_2 \to M_1 \oplus M_2$ 
denote the endomorphism of the $R$-module $M_1 \oplus M_2$ 
that sends $(x,y) \in M_1 \oplus M_2$ to 
$(\alpha_1(x),\alpha_2(y)+\beta(x))$.
Then $\alpha$ is an automorphism since
$\alpha(M_2) \subset  M_2$ and the endomorphisms of
$M_2$ and $M_1 = (M_1 \oplus M_2)/M_2$ induced by $\alpha$
are equal to $\alpha_2$ and $\alpha_1$, respectively.
It follows from the equality \eqref{eq:gr_eq1} that
we have $\alpha(\pi^{m_i} x_i, y'_i) \equiv (\pi^{m_i} x_i, y'_i)$
modulo $N$ for $i=1,\ldots,d_1$.
Thus, if we regard $\alpha$ as an automorphism of $M$, then
$\alpha$ satisfies the desired property.
\end{proof}

\begin{prop} \label{prop:gr_sheaf}
Let $F$ be a sheaf of sets on $\cC^{d_1} \times \cdots \times \cC^{d_r}$
with respect to the atomic topology, Then the presheaf $F\circ \gr$ on
$\cE^\bd$ is a sheaf with respect to the atomic topology.
\end{prop}

\begin{proof}
Let $f\colon Z \to Y$ be a morphism in $\cE^\bd$ that is a Galois covering.
It suffices to prove that the map 
$F(\gr(Y)) \to F(\gr(Z))^{\Gal(f)}$ is bijective.
It follows from Proposition \ref{prop:for_over_surj} that 
there exists a morphism $g\colon W \to Z$
in $\cE^\bd$ such that $\ffor(f\circ g)$ is Galois of special type.
Then $\ffor(g)$ is also Galois of special type.
It then follows from Proposition \ref{prop:gr_Galois_surj} that
$f \circ g$ and $\gr(f \circ g)$ are Galois coverings in $\cE^\bd$
and in $\cC^{d_1} \times \cdots \times \cC^{d_r}$, respectively, 
and the homomorphism $\Gal(f \circ g) \to \Gal(\gr(f\circ g))$
is surjective.
This implies that the map $F(\gr(Y)) \to F(\gr(W))^{\Gal(f \circ g)}$
is bijective since $F$ is a sheaf and $F(\gr(W))^{\Gal(f \circ g)}
= F(\gr(W))^{\Gal(\gr(f \circ g))}$.
Similarly, the morphism $g$ is a Galois covering in $\cE^\bd$ and the
map $F(\gr(Z)) \to F(\gr(W))^{\Gal(g)}$ is bijective.
Thus, the short exact sequence
$$
1 \to \Gal(g) \to \Gal(f \circ g) \to \Gal(f) \to 1
$$
shows that the map 
$F(\gr(Y)) \to F(\gr(Z))^{\Gal(f)}$ is bijective, as desired.
\end{proof}

\begin{cor}
The functor $\gr^*$ has a left adjoint.
\end{cor}

\begin{proof}
Since $\cE^\bd$ is essentially small and the category of sets
has small limits and colimits, it follows from \cite[Expose I]{SGA4}
that the functor 
$(-)\circ \gr \colon  \Presh(\cC^{d_1} \times \cdots \times \cC^{d_r}) 
\to \Presh(\cE^\bd)$ given by the composite with $\gr$ has
a left adjoint, which we denote by $\gr'_!$.
Let $\gr_!\colon  \Shv(\cE^\bd,J^\bd) \to 
\Shv(\cC^{d_1} \times \cdots \times \cC^{d_r}, J)$
denote the functor that sends a sheaf $F$ on
$(\cE^\bd,J^\bd)$ to the 
sheaf on $(\cC^{d_1} \times \cdots \times \cC^{d_r},J)$
associated with the
presheaf $\gr'_! F$.
Then for any sheaf $F$ on $(\cE^\bd,J^\bd)$ and for any
sheaf $G$ on $(\cC^{d_1} \times \cdots \times \cC^{d_r},J)$, we have
\begin{align*}
& \Hom_{\Shv(\cC^{d_1}\times \cdots \times \cC^{d_r},J)}(\gr_! F, G) \\
\cong &
\Hom_{\Presh(\cC^{d_1}\times \cdots \times \cC^{d_r})}(\gr'_! F, G) \\
\cong & \Hom_{\Presh(\cE^\bd)}(F, G\circ \gr)
= \Hom_{\Shv(\cE^\bd,J^\bd)}(F,\gr^* G)
\end{align*}
since if follows from Proposition \ref{prop:gr_sheaf} that we have
$\gr^* G = G \circ \gr$.
This shows that the functor $\gr_!$ is left adjoint to the functor
$\gr^*$.
\end{proof}

\subsection{Computation of absolute Galois monoids}

Let us compute the absolute Galois monoid $M_{(\cC_0^d,\jmath_0^d)}$.
We note that $M_{(\cC_0^d,\jmath_0^d)}$ is a topological group
since $J^\bd$ is the atomic topology.
Since $\iota_0^d = \ffor \circ \jmath_0^d$, we have
$(\alpha,\ffor(\gamma_\alpha)) \in M_{(\cC_0^d,\iota_0^d)}$
for any $(\alpha,\gamma_\alpha) \in M_{(\cC_0^d,\jmath_0^d)}$.
Here we denote by $\ffor(\gamma_\alpha)$ the natural
isomorphism $\iota_0^d \xto{\cong} \iota_0^d \circ \alpha$
that associates the isomorphism $\ffor(\gamma_\alpha(Y_0))$
to any object $Y_0$ of $\cC_0^d$.

Since the functor $\ffor$ is faithful, it follows that
the map $M_{\ffor} \colon  M_{(\cC_0^d,\jmath_0^d)} \to M_{(\cC_0^d,\iota_0^d)}$
that sends $(\alpha,\gamma_\alpha) \in M_{(\cC_0^d,\jmath_0^d)}$
to $(\alpha,\ffor(\gamma_\alpha)) \in M_{(\cC_0^d,\iota_0^d)}$
is an injective homomorphism of groups.
For each object $Y_0$ of $\cC_0^d$, we have two different groups
denoted by the same symbol $\bK_{Y_0}$. 
One is a compact open subgroup of $M_{(\cC_0^d,\jmath_0^d)}$ 
and the other is a compact open subgroup of $M_{(\cC_0^d,\iota_0^d)}$.
To avoid confusion, we denote the former group by $\bK^\bd_{Y_0}$
and the latter group by $\bK^d_{Y_0}$.
Then, by definition, we have $\bK^\bd_{Y_0} =
M_{\ffor}^{-1}(\bK^d_{Y_0})$.
This shows that the homomorphism $M_{\ffor}$ is continuous and
the topology on $M_{(\cC_0^d,\jmath_0^d)}$ is equal to the
topology induced from that on $M_{(\cC_0^d,\iota_0^d)}$
via $M_{\ffor}$.

It is clear from the definition of
$M_{\ffor}$ that an element 
$(\alpha,\gamma_\alpha) \in M_{(\cC_0^d,\iota_0^d)}$
belongs to the image of $M_{\ffor}$ if and only if,
for any object $Y_0$ of $\cC_0^d$, the isomorphism
$\gamma_\alpha(Y_0) \colon  \iota_0^d(Y_0) \xto{\cong}
\iota_0^d(\alpha(Y_0))$ are compatible with the
decreasing filtrations on $\iota_0^d(Y_0)$ and
$\iota_0^d(\alpha(Y_0))$ given by 
$\jmath_0^d(Y_0)$ and
$\jmath_0^d(\alpha(Y_0))$.

Let us recall from Chapter 2, Section \ref{sec:hom_phi}
that we have constructed an isomorphism
$\phi \colon  \GL_d(\A_X) \xto{\cong} M_{(\cC_0^d,\iota_0^d)}$
of topological groups.
Let $P_\bd \subset \GL_d$ denote the standard parabolic 
subgroup scheme (over $\cK_X$) of block-upper-triangular
invertible matrices corresponding to the partition $\bd$ of $d$.
It is then straightforward from the definition of
$\phi$ that the isomorphism $\phi$ induces an isomorphism
from $P_\bd(\A_X)$ to the image of $M_{\ffor}$.
Hence the restriction of $\phi$ to $P_{\bd}(\A_X)$
induces an isomorphism
$$
\phi_{P_\bd} \colon  P_{\bd}(\A_X) \xto{\cong} 
M_{(\cC_0^d,\jmath_0^d)}
$$
of topological groups.
As a consequence, we have the following.

\begin{prop}
Let us consider the equivalences
$\omega_{(\cC_0^d,\jmath_0^d)} \colon  \Shv(\cE^\bd,J^\bd)
\xto{\cong} (P_\bd(\A_X)\text{-sets})_\sm$ and
$\omega_{(\cC_0^d,\iota_0^d)} \colon  \Shv(\cC^d,J^d)
\xto{\cong} (\GL_d(\A_X)\text{-sets})_\sm$ of categories
given by the fiber functors with respect to the 
grids $(\cC_0^d,\jmath_0^d)$ and $(\cC_0^d,\iota_0^d)$, respectively.
Then the diagram
$$
\begin{CD}
\Shv(\cC^d,J^d) @>{\omega_{(\cC_0^d,\iota_0^d)}}>> 
(\GL_d(\A_X)\text{-sets})_\sm \\
@V{\ffor^*}VV @VV{\res}V \\
\Shv(\cE^\bd,J^\bd) @>{\omega_{(\cC_0^d,\jmath_0^d)}}>>
(P_\bd(\A_X)\text{-sets})_\sm
\end{CD}
$$
is commutative up to natural isomorphisms,
where $\res$ is the restriction functor.
\qed
\end{prop}

By definition, we have a canonical isomorphism
\begin{equation} \label{eq:canisom}
M_{(\cC_0^{d_1}\times \cdots \times \cC_0^{d_r}, \iota_0^{d_1}
\times \cdots \times \iota_0^{d_r})}
\cong M_{(\cC_0^{d_1},\iota_0^{d_1})} \times
\cdots \times M_{(\cC_0^{d_r},\iota_0^{d_r})}
\end{equation}
of topological groups.
Let $M_\bd = \GL_{d_1} \times \cdots \times \GL_{d_r}$
and let $\phi_\bd \colon  M_\bd(\A_X) \xto{\cong} 
M_{(\cC_0^{d_1}\times \cdots \times \cC_0^{d_r}, \iota_0^{d_1}
\times \cdots \times \iota_0^{d_r})}$ denote
the composite of the isomorphism
$$
M_\bd(\A_X) \xto{\cong} 
M_{(\cC_0^{d_1},\iota_0^{d_1})} \times
\cdots \times M_{(\cC_0^{d_r},\iota_0^{d_r})}
$$
given by the isomorphisms
$\phi \colon  \GL_{d_i}(\A_X) \xto{\cong} M_{(\cC_0^{d_i},
\iota_0^{d_i})}$ in Chapter 2, Section \ref{sec:hom_phi}
for $i=1,\ldots,r$
with the inverse of the isomorphism \eqref{eq:canisom}.

Let $q_\bd \colon  P_\bd \to M_\bd$
denote the projection to the diagonal blocks.
Let $g \in P_\bd(\A_X)$.
Observe that the automorphism of
$\A_X^{\oplus d}$ given by the right multiplication by $g^{-1}$
preserves the decreasing filtration
$\Fil^\bullet V_\A$ of $V_\A = \A_X^{\oplus d}$.
This implies that, if we set
$(\alpha,\gamma_\alpha) = \phi_{P_\bd}(g) \in 
M_{(\cC_0^d,\jmath_0^d)}$ and
$(\beta,\gamma_\beta) = \phi_\bd(q_\bd(g))
\in M_{(\cC_0^{d_1}\times \cdots \times \cC_0^{d_r}, \iota_0^{d_1}
\times \cdots \times \iota_0^{d_r})}$, then
the diagram
$$
\begin{CD}
\cC_0^d @>{\alpha}>> \cC_0^d \\
@V{\gr_0}VV @VV{\gr_0}V \\
\cC_0^{d_1}\times \cdots \times \cC_0^{d_r} 
@>{\beta}>> 
\cC_0^{d_1}\times \cdots \times \cC_0^{d_r}
\end{CD}
$$
is commutative, and moreover for any object
$Y_0$ of $\cC_0^d$, we have
$\gr(\gamma_\alpha(Y_0)) = \gamma_\beta(\gr_0(Y_0))$.
Hence by sending $(\alpha,\gamma_\alpha)$ to 
$(\beta,\gamma_\beta)$, we obtain
a map 
$\gr_M \colon  M_{(\cC_0^d,\jmath_0^d)}
\to M_{(\cC_0^{d_1}\times \cdots \times \cC_0^{d_r}, \iota_0^{d_1}
\times \cdots \times \iota_0^{d_r})}$
that makes the diagram
$$
\begin{CD}
P_\bd(\A_X) @>{\phi_{P_\bd}}>> 
M_{(\cC_0^d,\jmath_0^d)} \\
@V{q_\bd}VV @VV{\gr_M}V \\
M_\bd(\A_X)
@>{\phi_\bd}>> 
M_{(\cC_0^{d_1}\times \cdots \times \cC_0^{d_r}, \iota_0^{d_1}
\times \cdots \times \iota_0^{d_r})}
\end{CD}
$$
commutative.
As a consequence, we have the following.

\begin{prop}
Let us consider the equivalences
$\omega_{(\cC_0^d,\jmath_0^d)} \colon  \Shv(\cE^\bd,J^\bd)
\xto{\cong} (P_\bd(\A_X)\text{-sets})_\sm$ and
$\omega_{(\cC_0^{d_1}\times \cdots \times \cC_0^{d_r}, \iota_0^{d_1}
\times \cdots \times \iota_0^{d_r})} \colon  
\Shv(\cC_0^{d_1}\times \cdots \times \cC_0^{d_r},J)
\xto{\cong} (M_\bd(\A_X)\text{-sets})_\sm$ of categories
given by the fiber functors with respect to the 
grids $(\cC_0^d,\jmath_0^d)$ and $(\cC_0^{d_1}\times \cdots \times \cC_0^{d_r}, \iota_0^{d_1}\times \cdots \times \iota_0^{d_r})$, respectively.
Then the diagram
$$
\begin{CD}
\Shv(\cC_0^{d_1}\times \cdots \times \cC_0^{d_r},J)
@>{\omega_{(\cC_0^{d_1}\times \cdots \times \cC_0^{d_r}, \iota_0^{d_1}
\times \cdots \times \iota_0^{d_r})}}>> 
(M_\bd(\A_X)\text{-sets})_\sm \\
@V{\gr^*}VV @VV{\mathrm{Inf}}V \\
\Shv(\cE^\bd,J^\bd) @>{\omega_{(\cC_0^d,\jmath_0^d)}}>>
(P_\bd(\A_X)\text{-sets})_\sm
\end{CD}
$$
is commutative up to natural isomorphisms,
where $\mathrm{Inf}$ is the inflation functor.
\qed
\end{prop}


\section{$Y$-sites for classical similitude groups}
\label{sec:classical groups}

\subsection{Notation and terminology} \label{sec:classicalgp_notation}
Let $K$ be a non-archimedean local field whose
residue field is finite.
Let $L$ be either $K$ or a separable quadratic extension of $K$.
Let $\cO_K$ and $\cO_L$ denote the ring of integers of $K$ and $L$,
respectively. Let $\fram_L \subset \cO_L$ denote the maximal ideal
and $* \colon  L \to L$ the generator of $\Gal(L/K)$.

Let $D$ be a one-dimensional $L$-vector space,
and $\eps \colon  D \to D$ a $*$-semilinear map satisfying 
$\eps \circ \eps = \id_D$.

\begin{lem}
Let $I$ be an $\cO_L$-lattice of $D$.
Then we have $\eps(I)=I$.
If moreover the extension $L/K$ is at most tamely ramified,
then there exists an $\cO_L$-basis $e$ of $I$ satisfying
$\eps(e) \in \{e,-e\}$.
\end{lem}

\begin{proof}
Let us choose an arbitrary $\cO_L$-basis $e'$ of $I$.
Since $e'$ is an $L$-basis of $D$, there exists
$u \in L$ satisfying $\eps(e')= u e'$.
The equality $\eps \circ \eps = \id_D$ implies
that $u u^* =1$.
In particular we have $u \in \cO_L^\times$.
Hence we have $\eps(I) = I$.

Suppose moreover that
the extension $L/K$ is at most tamely ramified.
Then we can choose $v \in \cO_L^\times$ satisfying
$uv^*/v \in \{\pm1\}$.
Then the $\cO_L$ basis $e = ve'$ of $I$ satisfies
the desired property.
\end{proof}

Let $V$ be a finite dimensional $L$-vector space,
and $\psi\colon  V \times V \to D$ a $K$-bilinear map.
We say that $\psi$ is a non-degenerate 
$\eps$-hermitian form on $V$ if $\psi$ satisfies the
following conditions:
\begin{itemize}
\item $\psi$ is $*$-sesquilinear, i.e., 
$\psi(av,w) = \psi(v,a^*  w) = a \psi(v,w)$ for any $v,w \in V$
and for any $a \in L$,
\item $\psi(w,v) = \eps(\psi(v,w))$ for any $v,w \in D$,
\item $\{v \in V\ |\ \psi(v,w) = 0 \text{ for all }w \in V \}=\{0\}$.
\end{itemize}

\subsection{Aim}
Let $K$, $L$, $*$, $D$, and $\eps$ be as in 
Section \ref{sec:classicalgp_notation}.
Let us assume that the characteristic of the residue field
of $K$ is odd.
Suppose that a finite dimensional $L$-vector space
$V$ and non-degenerate 
$\eps$-hermitian form $\psi\colon V \times V \to D$ 
on $V$ are given.

The goal of this section is to construct a $Y$-site whose 
absolute Galois monoid is isomorphic to the classical similitude group
$$
G(\psi) = \{ (g,\lambda) \in \GL_K(V)
\times K^\times \ |\ \psi(gv,gw)= \lambda \psi(v,w)
\text{ for all }v,w \in V \}.
$$

First, we have a simple candidate $\wt{\cC}_0^\psi$ for the grid of the $Y$-site.
However, to construct the $Y$-site, for technical reasons, 
we do not directly construct the $Y$-site.
We use a subcategory $\cC_0^\psi\subset \wt{\cC}_0^\psi$
and construct a $Y$-site $(\Bil^\psi, J)$ such that a grid is given by $\cC_0^\psi$.
Then construct a $Y$-site $(\wh{\Bil}^\psi, J)$ such that a grid is given by $\wt{\cC}_0^\psi$.

One justification for using $\wt{\cC}_0^\psi$ as a grid is given in Section~\ref{sec:Okazaki}.
We consider the case where $G(\psi) \cong \GSp_4(K)$ and compute the compact
open subgroups associated with a strongly cyclic object 
(to be defined there) of $\wt{\cC}_0^\psi$.
We see that they are conjugates of the exactly quasi-paramodular groups 
introduced by Okazaki \cite{Okazaki} in his study of local new forms.

The aim of this section is to introduce
a semi-cofiltered category $\Bil^\psi$,
a poset category $\cC_0^\psi$, and a functor
$\iota_0 \colon  \cC_0^\psi \to \Bil^\psi$ in a way
as ``less artificial" as possible such that
$\Bil^\psi$ equipped with the atomic topology $J$
is a $Y$-site, that the pair $(\cC_0^\psi,\iota_0)$
is a grid for $(\Bil^\psi,J)$, and the monoid
$M_{(\cC_0^\psi,\iota_0)}$ is isomorphic to the group
$$
G(\psi) = \{ (g,\lambda) \in \GL_K(V)
\times K^\times \ |\ \psi(gv,gw)= \lambda \psi(v,w)
\text{ for all }v,w \in V \}.
$$

\subsection{The category $\Bil$}
Let $K$, $L$ and $*$ be as in Section \ref{sec:classicalgp_notation}.
In this section we introduce a category $\Bil$.

In a later section we will construct the desired
category $\Bil^\psi$ as a full subcategory of $\Bil$
under the assumption that a finite dimensional $L$-vector space
$V$ and non-degenerate 
$\eps$-hermitian form $\psi\colon V \times V \to D$ 
on $V$ are given and that $K$ is of odd residue characteristic.
We note that we do not require the datum $(V,\psi)$ 
in the definition of $\Bil$.

Let $\cC'$ denote the following category.
The objects of $\cC'$ are $\cO_L$-modules of finite length.
For two objects $M$, $M'$ of $\cC'$, the morphisms from
$M$ to $M'$ in $\cC'$ are the isomorphism classes of the
diagrams
$$
M' \stackrel{p}{\twoheadleftarrow} M'' 
\stackrel{i}{\inj}  M
$$
in the category of $\cO_L$-modules such that
$i$ is injective and $p$ is surjective.
Here two such diagrams are regarded to be isomorphic
if they satisfy the condition in Section \ref{sec:Cd_defn}.
For an integer $d \ge 1$, we let $\cC'^d$ denote the full subcategory of
$\cC'$ whose objects are the $\cO_L$-modules of finite length
generated by at most $d$ elements.

Let $\Bil$ denote the following category.
The objects of $\Bil$ are the quadruples
$(M,N,\vep,\phi)$, where $M$ is an object of $\cC'$, 
$N$ is an object of $\cC'^1$, $\vep\colon  N \to N$
is a $*$-semilinear map
satisfying $\vep \circ \vep = \id_N$,
and $\phi$ is
an $\cO_K$-bilinear map $\phi\colon  M \times M\to N$ 
satisfying $\phi(am,n) = \phi(m,a^* n) = a\phi(m,n)$
and $\phi(n,m) = \vep(\phi(m,n))$ for any
$a \in \cO_L$, $m,n \in M$.
For two objects $(M,N,\vep,\phi)$ and
$(M',N',\vep',\phi')$ of $\Bil$, 
the morphisms from $(M,N,\vep,\phi)$ to
$(M',N',\vep',\phi')$ in $\Bil$ are the 
pairs $(\alpha,\beta) \in \Hom_{\cC'}(M,M')
\times \Hom_{\cC'^1}(N,N')$ satisfying the
following condition: if 
$M' \stackrel{p}{\twoheadleftarrow} M'' 
\stackrel{i}{\inj}  M$ and
$N' \stackrel{q}{\twoheadleftarrow} N'' 
\stackrel{j}{\inj}  N$ are diagrams of
$\cO_L$-modules which represent $\alpha$
and $\beta$ respectively, then $\phi$ satisfies
\begin{itemize}
\item $\phi(i(M'')\times i(M'')) \subset j(N'')$,
\item $\phi(i(\Ker\, p)\times i(M'')) \subset j(\Ker\, q)$,
\item $\phi(i(M'')\times i(\Ker\, p)) \subset j(\Ker\, q)$,
\item for any $x,x' \in N''$ satisfying
$j(x') = \vep(j(x))$, we have $q(x') = \vep'(q(x))$,
\end{itemize}
and the $\cO_K$-bilinear map
$$
i(M'')/i(\Ker\, p) \times i(M'')/i(\Ker\, p)
\to j(N'')/j(\Ker\, q)
$$
induced by $\phi$ corresponds to $\phi'$
via the isomorphisms $M' \cong
M''/(\Ker\, p) \cong i(M'')/i(\Ker\, p)$
and $N' \cong
N''/(\Ker\, q) \cong j(N'')/i(\Ker\, q)$.

Now let us consider the functor 
$F \colon  \Bil \to \cC' \times \cC'^1$
that sends $(M,N,\vep,\phi)$ to $(M,N)$.
It follows from the definition of morphisms in $\Bil$ that
the functor $F$ is faithful.

\begin{lem} \label{lem:classicalgp_epi}
$\Bil$ is an $E$-category, i.e, any morphism in $\Bil$
is an epimorphism.
\end{lem}
\begin{proof}
Since $\cC' \times \cC'^1$ is an $E$-category, 
the claim follows from the faithfulness of $F$.
\end{proof}

The following statement is obvious from the
definition of morphisms in $\Bil$. 
We state it as a lemma since we would like to 
use it for reference.
\begin{lem} \label{lem:classicalgp_aux1}
Let $X=(M,N,\vep,\phi)$ be an object of $\Bil$. 
Let $f=(\alpha,\beta)\colon  (M,N) \to (M',N')$
be a morphism in $\cC' \times \cC'^1$.
Let
$M' \stackrel{p}{\twoheadleftarrow} M'' 
\stackrel{i}{\inj}  M$ and
$N' \stackrel{q}{\twoheadleftarrow} N'' 
\stackrel{j}{\inj}  N$ be diagrams of
$\cO_L$-modules which represent $\alpha$
and $\beta$ respectively.
Then there exists an object $Y$ of $\Bil$
such that $F(Y) = (M',N')$ and
$f$ is a morphism from $X$ to $Y$ in $\Bil$
if and only if $\phi(i(M''),i(M'')) \subset j(N'')$
and $\phi(i(M''),i(\Ker\, p)) \subset j(\Ker\, q)$.
Moreover, such an object $Y$ is unique if it exists.
\qed
\end{lem}

\begin{lem} \label{lem:classicalgp_under}
Let $X$ be an object of $\Bil$. Then the
functor $\Bil_{X/} \to (\cC'\times \cC'^1)_{F(X)/}$
induced by $F$ between the undercategories is
fully faithful.
\end{lem}

\begin{proof}
It follows from Lemma \ref{lem:classicalgp_epi}
that the categories $\Bil_{X/}$ and $(\cC'\times \cC'^1)_{F(X)/}$
are quasi-posets. Hence it suffices to prove that
the functor $\Bil_{X/} \to (\cC'\times \cC'^1)_{F(X)/}$
is full.

Let $f\colon  X \to Y$ and $g \colon  X \to Z$ be morphisms in
$\Bil$ and let $h\colon  F(Y) \to F(Z)$ be a morphism
in $\cC'\times \cC'^1$ satisfying
$F(g) = h \circ F(f)$. It suffices to prove that
$h$ comes from  a morphism $Y \to Z$  in $\Bil$.
For $S \in \{X,Y,Z\}$, let us write 
$S=(M_S,N_S,\vep_S,\phi_S)$.

For $t \in \{f,g,h\}$, let us write
$t = (\alpha_t,\beta_t)$.
Let
$M_T \stackrel{p_t}{\twoheadleftarrow} M_t 
\stackrel{i_t}{\inj}  M_S$ and
$N_T \stackrel{q_t}{\twoheadleftarrow} N_t 
\stackrel{j_t}{\inj}  N_S$ be diagrams of
$\cO_L$-modules which represent $\alpha_t$
and $\beta_t$ respectively.
Here $S$ and $T$ denote the domain 
and codomain of $t$.
Since $g = h \circ f$, it follows that
$i_g(M_g)$, $i_g(\Ker\, p_g)$, $j_g(N_g)$ 
and $j_g(\Ker\, q_g)$ are equal to 
$i_f(p_f^{-1}(i_h(M_h)))$,
$i_f(p_f^{-1}(i_h(\Ker\, p_h)))$,
$j_f(q_f^{-1}(j_h(N_h)))$ and
$j_f(q_f^{-1}(j_h(\Ker\, q_h)))$,
respectively.

Hence we have
\begin{align*}
& \phi_Y(i_h(M_h) \times i_h(M_h)) \\
= & q_f(j_f^{-1}(\phi_X(i_f(p_f^{-1}(i_h(M_h))) \times
i_f(p_f^{-1}(i_h(M_h)))))) \\
= & q_f(j_f^{-1}(\phi_X(i_g(M_g) \times i_g(M_g)))) \\
\subset & q_f(j_f^{-1}(j_g(N_g))) \\
= & q_f(j_f^{-1}(j_f(q_f^{-1}(j_h(N_h))))) \\
= & j_h(N_h)
\end{align*}
and
\begin{align*}
& \phi_Y(i_h(M_h) \times i_h(\Ker\, p_h)) \\
= & q_f(j_f^{-1}(\phi_X(i_f(p_f^{-1}(i_h(M_h))) \times
i_f(p_f^{-1}(i_h(\Ker\, p_h)))))) \\
= & q_f(j_f^{-1}(\phi_X(i_g(M_g) \times i_g(\Ker\, p_g)))) \\
\subset & q_f(j_f^{-1}(j_g(\Ker\, q_g))) \\
= & q_f(j_f^{-1}(j_f(q_f^{-1}(j_h(\Ker\, q_h))))) \\
= & j_h(\Ker\, q_h).
\end{align*}
Hence it follows from Lemma \ref{lem:classicalgp_aux1}
that there exists an object $Z'$ of $\Bil$ with
$F(Z') = (M_Z,N_Z)$ such that $h$ is a morphism
from $Y$ to $Z'$ in $\Bil$.
Since $g = h \circ f$ is a morphism from $Y \to Z$
in $\Bil$ and a morphism from $Y$ to $Z'$ in $\Bil$
at the same time.
Hence it follows from \ref{lem:classicalgp_aux1}
that we have $Z'=Z$.
This completes the proof.
\end{proof}

Let $d \ge 1$ be an integer and let
$\Bil^d$ denote the full-subcategory of $\Bil$
whose objects are objects $X=(M,N,\vep,\phi)$
of $\Bil$ such that $M$ is an object of $\cC'^d$.

\begin{lem} \label{lem:classicalgp_Galois}
Let $X=(M,N,\vep,\phi)$, $Y=(M_0,N_0,\vep_0,\phi_0)$
be objects of $\Bil^d$ and let $f = (\alpha,\beta)$ be
a morphism from $X$ to $Y$ in $\Bil$.
Suppose that $\phi_0$ is not identically zero and that
the morphism $\alpha \colon  M \to M_0$ is a Galois covering in 
$\cC'^d$. Then $f$ is a Galois covering in $\Bil^d$ in the
sense of Definition 3.1.2 of \cite{Grids}. 
\end{lem}

\begin{proof}
It follows from Lemma \ref{lem:classicalgp_epi} that
$\Bil^d$ is an $E$-category.

Let $X'=(M',N',\vep',\phi')$ be an arbitrary object of
$\Bil^d$. To prove that $f$ is a Galois covering
in $\Bil^d$,
it suffices to show that for any two morphisms $g,g' \colon  X' \to X$
in $\Bil^d$ with $f \circ g = f \circ g'$, there exists
an automorphism $h\colon  X \to X$ in $\Bil^d$ 
satisfying $g' = h \circ g$.
In fact, since $\Bil^d$ is an $E$-category, such an automorphism
$h$ is unique if exists, and is automatically an automorphism over $Y$.

Let us write $g = (\gamma,\delta)$ and 
$g'=(\gamma',\delta')$. Since
$\alpha \circ \gamma = \alpha \circ \gamma'$ and
$\alpha$ is a Galois covering in $\cC'^d$, there exists
a unique automorphism $\eta\colon M \to M$ in $\cC'^d$ over $M_0$
satisfying $\gamma' = \eta \circ \gamma$.
It suffices to show that there exists an automorphism
$\kappa\colon  N \to N$ in $\cC'^1$ such that
$\delta' = \kappa \circ \delta$ and
$(\eta,\kappa)$ is an automorphism of $X$ in $\Bil^d$.
Let $N' \stackrel{q}{\twoheadleftarrow} N'' 
\stackrel{j}{\inj}  N$ and
$N' \stackrel{q'}{\twoheadleftarrow} N''' 
\stackrel{j'}{\inj}  N$ be diagrams of $\cO_L$-modules
which represent $\delta$ and $\delta'$, respectively.
Since $N$ is an object of $\cC'^1$, there exists
an integer $r \in \Z$ satisfying the following
property:
\begin{itemize}
\item If $r \ge 0$, then we have $j'(N''') = \fram_L^r j(N'')$
and $j'(\Ker\, q') = \fram_L^r j(\Ker\, q)$,
\item If $r \le 0$, then we have $\fram_L^{-r} j'(N''') = j(N'')$
and $\fram_L^{-r} j'(\Ker\, q') = j(\Ker\, q)$.
\end{itemize}
Let
$M' \stackrel{p}{\twoheadleftarrow} M'' 
\stackrel{i}{\inj}  M$ and
$M' \stackrel{p'}{\twoheadleftarrow} M''' 
\stackrel{i'}{\inj}  M$ be diagrams of $\cO_L$-modules
which represent $\gamma$ and $\gamma'$, respectively.
The existence of the automorphism $\eta$ implies
that we have $i(M'') = i'(M''')$ and
$i(\Ker\, p) = i'(\Ker\, p')$.

Let $N_1$ denote the image of the map $\phi\colon M \times M \to N$.
We then have $N_1 = \fram_L^r N_1$ if $r \ge 0$
and $\fram_L^{-r} N_1 = N_1$ if $r \le 0$. 
By our assumption on $\phi$, we have $N_1 \neq 0$.
This implies $r=0$.
Since $j'(N''') = j(N'')$ and $j'(\Ker\, q') = j(\Ker\, q)$,
there exists an automorphism
$\kappa\colon  N \to N$ in $\cC'^1$ such that
$\delta' = \kappa \circ \delta$.

It follows from Lemma \ref{lem:classicalgp_under}
that $(\eta,\kappa)$ is an automorphism of $X$ in $\Bil^d$.
This proves that $f$ is a Galois covering in $\Bil^d$.
\end{proof}

We introduce some terminology which we will use in later
paragraphs. For a finitely generated $\cO_L$-module $M$
and an integer $d$, we say that $M$ is generated exactly by
$d$ elements if $M/\fram_L M$ is a $d$-dimensional
$\cO_L/\fram_L$-vector space.

\begin{defn}
Let $X = (M,N,\vep,\phi)$ be an object of $\Bil$.
For $m \in M$, let $\phi(m,-)$
denote the $*$-semilinear map 
from $M$ to $N$ that sends $n \in M$ to $\phi(m,n)$.
Let $\Hom_{*}(M,N)$ denote the set of $*$-semilinear
maps from $M$ to $N$. We endow $\Hom_{*}(M,N)$ with a
structure of an $\cO_L$-module by setting
$(af)(m) = a(f(m))$ for any $a \in \cO_L$, 
$f \in \Hom_{*}(M,N)$.
Let us consider the map $b_\phi \colon  M \to \Hom_{*}(M,N)$
that sends $m \in M$ to $\phi(m,-)$.
It is easy to see that $b_\phi$ is $\cO_L$-linear.
\begin{enumerate}
\item We say that $X$ is non-degenerate if
the map $b_\phi$ is injective.
\item Let $d$ be an integer. We say that $X$ is
$d$-good if both $M$ and the image of $b_\phi$ 
are generated exactly by $d$ elements.
\end{enumerate}
\end{defn}

\begin{lem} \label{lem:classicalgp_lift}
Suppose that $K$ is of odd residue characteristic.
Let $d \ge 1 $ an integer and
$X = (M,N,\vep,\phi)$ an object 
of $\Bil^d$ which is $d$-good and non-degenerate.
Let $p \colon  \wt{M} \to M$ (\resp $q\colon  \wt{N} \to N$)
be a surjective homomorphism
of $\cO_L$-modules such that
$\wt{M}$ (\resp $\wt{N}$) is generated exactly by
$d$ elements (\resp one element) and
$\Ker\, p$ (\resp $\Ker\, q$) is isomorphic to
$(\cO_L/\fram_L \cO_L)^{\oplus d}$
(\resp $\cO_L/\fram_L \cO_L$) as an $\cO_L$-module.
Let $\alpha \colon  \wt{M} \to M$ and
$\beta \colon  \wt{N} \to N$ be morphisms in $\cC'$
represented by the diagrams
$M \stackrel{p}{\twoheadleftarrow} \wt{M} 
\xto{\id}  \wt{M}$ and
and
$N \stackrel{q}{\twoheadleftarrow} \wt{N} 
\xto{\id} \wt{N}$,
respectively.
Let $\wt{X}$ and $\wt{X}'$ be objects of $\Bil^d$
such that $F(\wt{X}) = F(\wt{X}') = (\wt{M},\wt{N})$ and
$f=(\alpha,\beta)$ is a morphism from $\wt{X}$ to
$X$ and a morphism from $\wt{X}'$ to $X$ at the
same time.
Then there exists a morphism $g \colon \wt{X} \to \wt{X}'$
in $\Bil^d$ satisfying $f = f \circ g$.
\end{lem}

\begin{proof}
Let us choose a uniformizer $\varpi \in \cO_L$
and $c \in \{\pm 1\}$ satisfying $\varpi^* = c \varpi$.
We may assume that 
$M = \bigoplus_{i=1}^d \cO_L/\fram_L^{n_i}\cO_L$,
$N = \cO_L/\fram_L^{n} \cO_L$, for positive integers
$n_i$ ($1\le i\le d$) and $n$, and that there exist
$c' \in \{\pm 1\}$ such that 
$\vep(a \mod{\fram_L^{n} \cO_L})
= c' a^* \mod{\fram_L^{n} \cO_L}$ for any $a \in \cO_L$.
By assumption we have $1 \le n_1,\ldots,n_d \le n$.

Let us write $\wt{X} = (\wt{M}, \wt{N}, \wt{\vep}, \wt{\phi})$
and $\wt{X}' = (\wt{M}, \wt{N}, \wt{\vep}', \wt{\phi}')$.
We may further assume that
$\wt{M} = \bigoplus_{i=1}^d \cO_L/\fram_L^{n_i+1}\cO_L$,
$N = \cO_L/\fram_L^{n+1} \cO_L$, that
$p$ and $q$ are canonical surjections, and that
$\wt{\vep}(a \mod{\fram_L^{n+1} \cO_L})
= \wt{\vep}'(a \mod{\fram_L^{n+1} \cO_L})
= c' a^* \mod{\fram_L^{n+1} \cO_L}$ for any $a \in \cO_L$.
Let $e_1,\ldots,e_d \in \wt{M}$ denote the standard
generators of $\wt{M}$ and let 
$G = (\wt{\phi}(e_i,e_j)), G'= (\wt{\phi}'(e_i,e_j))
\in \Mat_d(\cO/\fram_L^{n+1} \cO)$ denote the
matrix representation of $\wt{\phi}$, $\wt{\phi}'$
respectively.
Let $S = \diag(\varpi^{n-n_1},\cdots, \varpi^{n-n_d})$
and $T = \diag(\varpi^{n_1},\ldots,\varpi^{n_d})$.
It suffices to prove that there exists $X \in \Mat_d(\cO_L)$
such that $G' = (1 + {}^t (T X)^*) G (1+TX)$.
Let us choose lifts $\wt{G}, \wt{G}' \in \Mat_d(\cO_L)$
of $G$, $G'$ in such a way that
${}^t \wt{G}^* = c' \wt{G}$ and ${}^t \wt{G}'^* = c \wt{G}'$.
We have $\wt{G}' \equiv \wt{G} \mod{\fram_L^{n} \Mat_d(\cO)}$.
Let us write $\wt{G}' - \wt{G} = \varpi^{n} Y$ for some $Y \in \Mat_d(\cO)$.
Then we have ${}^t Y^* = c^{n} c' Y$.

By assumption there exists $H \in \Mat_d(\cO_L)$
such that $\wt{G} = S H$. Moreover by 
the non-degeneracy of $\phi$, we have $H \in \GL_d(\cO_L)$.
Since ${}^t \wt{G}^* = c' \wt{G}$, we have
$\wt{G} = c' S^* {}^t H^*$.
Then for any $X \in \Mat_d(\cO)$, we have
$(1 + {}^t (T X)^*) G (1+TX)
= \wt{G} + \varpi^{n} (c^{n} c' {}^t(HX)^* + HX) \mod{\fram_L^{n} \Mat_d(\cO_L)}$.
Hence if we set $X = \frac{1}{2} H^{-1} Y$, then
we have the desired equality
$G' = (1 + {}^t (T X)^*) G (1+TX)$.
\end{proof}

%

\subsection{A pair $(\wt{\cC}_0^\psi, \wt{\iota}_0)$}
Let $K$, $L$, $*$, $D$, and $\eps$ be as in 
Section \ref{sec:classicalgp_notation}.
Suppose that a finite dimensional $L$-vector space
$V$ and a non-degenerate $\eps$-hermitian
form $\psi \colon  V \times V \to D$ on $V$ are given.

In this section we introduce a poset category
$\wt{\cC}_0^\psi$ and a functor 
$\wt{\iota}_0 \colon  \wt{\cC}_0^\psi \to \Bil$.
In a later section we construct the desired 
category $\cC_0^\psi$ as a full subcategory of
$\wt{\cC}_0^\psi$, and the desired functor $\iota_0$ 
as a functor induced by the restriction of
$\wt{\iota}_0$ to $\cC_0^\psi$ so that
$(\cC_0^\psi, \iota_0)$ is a grid.

Let $\Pair^\psi$ denote the following poset.
The elements of $\Pair^\psi$ are the quadruples
$(L_1,L_2,I_1,I_2)$ of $\cO_L$-lattices
$L_1,L_2 \subset V$ and $I_1,I_2 \subset L$
satisfying $L_1 \subset L_2$, $I_1 \subset I_2$,
$\psi(L_2 \times L_2) \subset I_2$, and
$\psi(L_1 \times L_2) \subset I_1$.
For two elements $(L_1,L_2,I_1,I_2)$ and
$(L'_1,L'_2,I'_1,I'_2)$ of $\Pair^\psi$, we have
$(L_1,L_2,I_1,I_2) \le (L'_1,L'_2,I'_1,I'_2)$
if and only if $L'_1 \subset L_1 \subset L_2
\subset L'_2$ and
$I'_1 \subset I_1 \subset I_2 \subset I'_2$.
Let $\wt{\cC}_0^\psi$ denote the poset category
corresponding to the order dual of $\Pair^\psi$.
Observe that, for an object $X=(L_1,L_2,I_1,I_2)$
of $\wt{\cC}_0^\psi$, the $*$-semilinear map $\eps$ 
induces a $*$-semilinear map $\eps_X \colon  I_2/I_1 
\to I_2/I_1$ and the $\eps$-hermitian form $\psi$
induces an $\cO_K$-bilinear map
$L_2/L_1 \times L_2/L_1 \to I_2/I_1$
which we denote by $\psi_X$.
Let $\wt{\iota}_0\colon  \wt{\cC}_0^\psi \to \Bil$ denote
the functor that sends an object
$X=(L_1,L_2,I_1,I_2)$ of $\wt{\cC}_0^\psi$ to 
$(L_2/L_1,I_2/I_1,\eps_X,\psi_X)$, and that sends a
morphism from $(L_1,L_2,I_1,I_2)$ to 
$(L'_1,L'_2,I'_1,I'_2)$
in $\wt{\cC}_0^\psi$
to the morphism in $\Bil$ represented by the diagrams
$L'_2/L'_1 \twoheadleftarrow L'_2/L_1 \inj L_2/L_1$
and $I'_2/I'_1 \twoheadleftarrow I'_2/I_1 \inj I_2/I_1$.

\begin{defn} \label{defn:classicalgp_expander}
Let $X = (L_1,L_2,I_1,I_2)$ be an object of $\wt{\cC}_0^\psi$ 
and $n \ge 0$ be an integer. We set
$S^n(X) = (\fram_L^n L_1, L_2, \fram_L^n I_1, I_2)$ 
and $T^n(X) = (\fram_L^{2n} L_1, 
\fram_L^{-n} L_2, \fram_L^n I_1, \fram_L^{-2n} I_2)$.
Then $S^n(X)$ and $T^n(X)$ are objects of $\wt{\cC}_0^\psi$
and we have a commutative diagram
$$
\begin{CD}
\cdots @>>> T^2(X) @>>> T^1(X) @>>> T^0(X) @= X \\
@. @VVV @VVV @| \\
\cdots @>>> S^2(X) @>>> S^1(X) @>>> S^0(X) @= X.
\end{CD}
$$
in $\wt{\cC}_0^\psi$.
\end{defn}

\subsection{The category $\Bil^\psi$}
In this section we construct the category $\Bil^\psi$.

Let us denote by $\wt{\Bil}^\psi$ the full subcategory
of $\Bil$ whose collection of objects is the 
essential image of the functor $\wt{\iota}_0$.

\begin{lem} \label{lem:classicalgp_under2}
Let $X$ be an object of $\wt{\cC}_0^\psi$.
Then the functor 
$\wt{\iota}_{0,X/}\colon 
\wt{\cC}^\psi_{0,X/} \to \wt{\Bil}^\psi_{\wt{\iota}_0(X)/}$ 
between the undercategories induced by $\wt{\iota}_0$
is an equivalence of categories.
\end{lem}

\begin{proof}
It is easy to see that the composite of
$\wt{\iota}_{0,X/}$ with the functor
$F_{\wt{\iota}_0(X)/}\colon  \wt{\Bil}^\psi_{\wt{\iota}_0(X)/}
\to (\cC' \times \cC'^1)_{F(\wt{\iota}_0(X))/}$
induced by $F$ is fully faithful.
Hence it follows from Lemma \ref{lem:classicalgp_under}
that $\wt{\iota}_{0,X/}$ is fully faithful.

Lemma \ref{lem:classicalgp_aux1} implies that the
functor $\wt{\iota}_{0,X/}$ is essentially surjective.
Thus $\wt{\iota}_{0,X/}$ gives an equivalence of 
categories.
\end{proof}

\begin{lem} \label{lem:classicalgp_Galois2}
Let $X=(M,N,\vep,\phi)$ be an object of $\wt{\Bil}^\psi$
such that $\phi$ is not identically zero.
Then, for any morphism
$f\colon  Y \to X$ in $\wt{\Bil}^\psi$,
there exists a morphism $g\colon  Z \to Y$
in $\wt{\Bil}^\psi$ such that the composite
$f \circ g$ is a Galois covering in $\wt{\Bil}^\psi$.
\end{lem}

\begin{proof}
We may assume that $Y = \wt{\iota}_0(Y_0)$ for some
object $Y_0$ of $\wt{\cC}_0^\psi$.
It follows from Lemma \ref{lem:classicalgp_under2} that
there exists a morphism $f_0 \colon  Y_0 \to X_0$ in
$\wt{\cC}_0^\psi$ such that $f$ is equal to the
composite of $\wt{\iota}_0(f_0)$ with an isomorphism
$\wt{\iota}_0(X_0) \xto{\cong} X$ in $\wt{\Bil}^\psi$.
We may assume that $X = \wt{\iota}_0(X_0)$
and $f = \wt{\iota}_0(f_0)$.

Let us write $X_0 =(L_1,L_2,I_1,I_2)$
and $Y_0=(L'_1,L'_2,I'_1,I'_2)$.
Choose a sufficiently large integer $n \ge 1$
so that $\fram_L^{2n} L_2 \subset L'_1$,
$L'_2 \subset \fram_L^{-n} L_2$,
$\fram_L^n I_2 \subset I'_1$, and
$I'_2 \subset \fram_L^{-2n} I_2$.
Then $Z_0 = T^n(X_0)$
is an object of $\wt{\cC}_0^\psi$
and there exists a morphism $g_0\colon  Z_0 \to Y_0$ in $\wt{\cC}_0^\psi$.
We set $Z = \wt{\iota}_0(Z_0)$ and
$g = \wt{\iota}_0(g_0)$.

Set $d =\dim_L V$.
It follows from Lemma \ref{lem:Gal} that the morphism
from $\fram_L^{-n} L_2 /\fram_L^{2n} L_2$ to $L_2/L_1$
give by the diagram
$$
L_2/L_1 \twoheadleftarrow L_2/\fram_L^{2n} L_2
\inj  \fram_L^{-n} L_2/\fram_L^{2n} L_2
$$
is a Galois covering in $\cC'^d$.
Since $\phi$ is not identically zero, it follows from
Lemma \ref{lem:classicalgp_Galois} that $g \circ f$
is a Galois covering in $\Bil^d$.
This in particular shows that $g \circ f$ is a Galois covering
in $\wt{\Bil}^\psi$, which completes the proof.
\end{proof}

We note that the category $\wt{\Bil}^\psi$ is
not necessarily semi-cofiltered. The following lemma
gives a counterexample.

\begin{lem} \label{lem:classicalgp_Lambda}
Let $X=(M,N,\vep,\phi)$ be an object of
$\wt{\Bil}^\psi$. Suppose that $\phi$ is identically zero.
Then the overcategory $\wt{\Bil}^\psi_{/X}$
is not $\Lambda$-connected.
\end{lem}

\begin{proof}
We may assume that $X = \wt{\iota}_0(X_0)$
for some object $X_0$ of $\wt{\cC}_0^\psi$.
Let us write $X_0 = (L_1,L_2,I_1,I_2)$.
Since $\phi$ is identically zero, we have
$\psi(L_2,L_2) \subset I_1$.
Let us choose a uniformizer $\varpi_K \in \cO_K$.
Let us choose an integer $n \ge 1$
satisfying $\fram_L^n I_2 \subset I_1$ and
$\psi(L_2,L_2) \not\subset \fram_L^n I_2$.
Then $Y_0 = S^n(X_0)$ is an object
of $\wt{\cC}_0^\psi$ and there exists a morphism
$f_0 \colon  Y_0 \to X_0$ in $\wt{\cC}_0^\psi$.

We also set 
$X'_0 = (L_1, L_2, \varpi_K^{-1} I_1, \varpi_K^{-1} I_2)$
and $Y'_0 = S^n(X_0)$.
Then $X'_0$ and $Y'_0$ are objects of $\wt{\cC}_0^\psi$
and there exists a morphism
$f'_0 \colon  Y'_0 \to X'_0$ in $\wt{\cC}_0^\psi$.
Set $Y = \wt{\iota}_0(Y_0)$, $Y' = \wt{\iota}_0(Y'_0)$,
and $X' = \wt{\iota}_0(X'_0)$.
The pair of the identity map on $L_2/L_1$ and the
isomorphism $\fram_K^{-1} I_2/ \fram_K^{-1} I_1
\xto{\cong} I_2/I_1$ given by the multiplication by $\varpi_K$
induces an isomorphism 
$\gamma \colon  X' \xto{\cong} X$ in $\wt{\Bil}^\psi$.
Set $f = \wt{\iota}_0(f_0)$ and 
$f' = \gamma \circ \wt{\iota}_0(f'_0)$.
Let us consider the diagram
$Y \xto{f} X \xleftarrow{f'} Y'$ in $\wt{\Bil}^\psi$.

Assume that there exist an object $Z$ of $\wt{\Bil}^\psi$
and morphisms $g \colon  Z \to Y$, $g' \colon  Z \to Y'$
satisfying $f \circ g = f' \circ g'$.
Let us write $Z = (M',N',\vep',\phi')$
and $f \circ g = (\alpha,\beta)$.
Let us choose diagrams
$$
M \stackrel{p}{\twoheadleftarrow} M''
\stackrel{i}{\inj} M'
$$
and
$$
N \stackrel{q}{\twoheadleftarrow} N''
\stackrel{j}{\inj} N'
$$
of $\cO_L$-modules which represent 
$\alpha$ and $\beta$, respectively.

Then we have $\phi'(i(M''),i(M'')) \subset j(N'')$.
Let $N_1$ denote the $\cO_L$-submodule of
$j(N'')/\fram_L^n j(N'')$
generated by the image of the composite
$i(M'') \times i(M'') \xto{\phi'} j(N'') \to j(N'')/\fram_L^n j(N'')$.
Since $f \circ g = (\alpha,\beta)$, we have
$N_1 \neq \{0\}$. Since $f' \circ g' = (\alpha,\beta)$,
we have $N_1 = \varpi_K N_1$.
Hence by Nakayama's lemma we have $N_1 = \{0\}$,
which leads to contradiction.
This proves that 
$\wt{\Bil}^\psi_{/X}$ is not $\Lambda$-connected.
\end{proof}

Let $d = \dim_L V$.
We let $\Bil^\psi$ denote the
full subcategory of $\wt{\Bil}^\psi$ whose
objects are the objects $X=(M,N,\vep,\phi)$ of 
$\wt{\Bil}^\psi$ such that $\phi$ is $d$-good and
the overcategory $\wt{\Bil}^\psi_{/X}$ is $\Lambda$-connected.

\begin{prop} \label{prop:classicalgp_etale}
Suppose that $K$ is of odd residue characteristic.
Let $X$ be an object 
in $\wt{\Bil}^\psi$ which is $d$-good and non-degenerate.
Then $X$ is an object of $\Bil^\psi$.
\end{prop}

\begin{proof}
Let $f \colon  Y \to X$ and $f' \colon  Y' \to X$ be 
two morphisms in $\wt{\Bil}^\psi$.
We prove that there exist an object $Z$ of $\wt{\Bil}^\psi$
and two morphisms $g\colon  Z \to Y$ and $g' \colon  Z \to Y'$
satisfying $f\circ g = f' \circ g'$.
We may assume that $Y = \wt{\iota}_0(Y_0)$ and
$Y' =\wt{\iota}_0(Y'_0)$ for some objects
$Y_0$, $Y'_0$ of $\wt{\cC}_0^\psi$.

It follows from Lemma \ref{lem:classicalgp_under2}
that there exists a unique tuple $(f_0,f'_0,\alpha,\alpha')$
of morphisms $f_0 \colon  Y_0 \to X_0$, $f'_0\colon  Y_0 \to X'_0$ in
$\wt{\cC}_0^\psi$ and isomorphisms
$\alpha\colon  \wt{\iota}_0(X_0) \xto{\cong} X$,
$\alpha' \colon  \wt{\iota}_0(X'_0) \xto{\cong} X$
satisfying $f = \alpha \circ \wt{\iota}_0(f_0)$
and $f' = \alpha' \circ \wt{\iota}_0(f'_0)$.
Let us write 
$X_0 = (L_1,L_2,I_1,I_2)$,
$X'_0 = (L'_1,L'_2,I'_1,I'_2)$,
$Y_0 = (M_1,M_2,J_1,J_2)$, and
$Y'_0 = (M'_1,M'_2,J'_1,J'_2)$.
For an integer $n \ge 1$, set
$X_n = S^n(X_0)$
and 
$X'_n = S^n(X'_0)$.
Then $X_n$ and $X'_n$ are objects of $\wt{\cC}_0^\psi$
and we have two projective systems
$\cdots \to X_2 \to X_1 \to X_0$ and
$\cdots \to X'_2 \to X'_1 \to X'_0$
in the category $\wt{\cC}_0^\psi$.

By successively applying Lemma \ref{lem:classicalgp_lift},
we have an isomorphism 
$\alpha_n\colon  \wt{\iota}_0(X_n) \xto{\cong} \wt{\iota}_0(X'_n)$ 
in $\wt{\Bil}^\psi$ for each $n \ge 1$ such that the
diagram
$$
\begin{CD}
\cdots @>>> \wt{\iota}_0(X_2) @>>> 
\wt{\iota}_0(X_1) @>>> \wt{\iota}_0(X_0) \\
@. @V{\alpha_2}VV @V{\alpha_1}VV @VV{\alpha'^{-1} \circ \alpha}V \\
\cdots @>>> \wt{\iota}_0(X'_2) @>>> 
\wt{\iota}_0(X'_1) @>>> \wt{\iota}_0(X'_0)
\end{CD}
$$
in $\wt{\Bil}^\psi$ is commutative.
The family $(\alpha_n)_{n \ge 1}$ of isomorphisms
induces an $\cO_L$-linear isomorphisms 
$\gamma \colon  L_2 \xto{\cong} L'_2$
and $\delta \colon  I_2 \xto{\cong} I'_2$
satisfying the equalities 
$\psi(\gamma(x),\gamma(x'))
= \delta(\psi(x,x'))$ 
and $\eps(\delta(y)) = \delta(\eps(y))$
for any $x,x' \in L_2$ and $y \in I_2$.
By extending $\gamma$ and $\delta$ by
$L$-linearity, we obtain
$L$-linear isomorphisms 
$\wt{\gamma} \colon  V \xto{\cong} V$
and $\wt{\delta} \colon  D \xto{\cong} D$
satisfying the equalities 
$\psi \circ (\wt{\gamma}, \wt{\gamma})
= \wt{\delta} \circ \psi$ 
and $\eps \circ \wt{\delta} = \wt{\delta} \circ \eps$.

For an integer $n \ge 1$, set
$Z_n = T^n(X_0)$
and 
$Z'_n = T^n(X'_0)$.
Then $Z_n$ and $Z'_n$ are objects of $\wt{\cC}_0^\psi$
and the pair $(\wt{\gamma},\wt{\delta})$ induces
an isomorphism $\gamma_n \colon  \wt{\iota}_0(Z_n) \xto{\cong}
\wt{\iota}_0(Z'_n)$ in $\wt{\Bil}^\psi$.
Choose a sufficiently large integer $m \ge 1$ so that
there exist morphisms $g_0\colon  Z_m \to Y_0$ and $g'_0 \colon  Z'_m \to Y'_0$
in $\wt{\cC}_0^\psi$. Then the triple 
$(Z, g,g') 
= (\wt{\iota}_0(Z_m),\wt{\iota}_0(g_0), \wt{\iota}_0(g'_0) \circ \gamma_m)$
has the desired property.
\end{proof}

For $\cO_L$-lattices $L' \subset V$ and $I \subset D$,
we let $L'^\perp_{/ I}$ denote the subset
$$
L'^\perp_{/ I} = \{ x \in V\ |\ \psi(x,y) \in I
\text{ for all } y \in L' \}
$$
of $V$. Then $L'^\perp_{/ I}$ is an $\cO_L$-lattice
of $V$.

Let $L' \subset V$ be an $\cO_L$-lattice and
$I_1, I_2 \subset D$ be $\cO_L$-lattices.
We say that the triple $(L',I_1,I_2)$ is
admissible if $\psi(L',L') \subset I_2$,
$I_1 \subset I_2$, and 
$L'^\perp_{/I_1} \subset \fram_L L'$.
For an admissible triple $(L',I_1,I_2)$, we set
$X_{L',I_1,I_2} = (L'^\perp_{/I_1},L',I_1,I_2)$.

Let $\cC_0^\psi \subset \wt{\cC}_0^\psi$ denote 
the full subcategory whose objects are the objects
$X$ of $\wt{\cC}_0^\psi$ such that 
$\wt{\iota}_0^\psi(X)$ is an object of $\Bil^\psi$.
\begin{lem} \label{lem:classicalgp_triple}
Let $(L',I_1,I_2)$ be an admissible triple.
Then $X_{L',I_1,I_2}$ is an object of $\cC_0^\psi$.
\end{lem}

\begin{proof}
$X_{L',I_1,I_2}$ is an object of
$\wt{\cC}_0^\psi$ such that
$\wt{\iota}_0(X_{L',I_1,I_2})$ is $d$-good and non-degenerate.
Hence it follows from Proposition \ref{prop:classicalgp_etale}
that $X_{L',I_1,I_2}$ is an object of $\cC_0^\psi$.
\end{proof}

\begin{cor} \label{cor:classicalgp_aux2}
Let $X$ be an object of $\wt{\Bil}^\psi$.
Then there exists a morphism
$f\colon  Y \to X$ in $\wt{\Bil}^\psi$
such that $Y$ is an object of $\Bil^\psi$.
\end{cor}

\begin{proof}
We may assume that $X = \wt{\iota}_0(X_0)$
for some object $X_0$ of $\wt{\cC}_0^\psi$.
Let us write $X_0=(L_1,L_2,I_1,I_2)$.
Choose a sufficiently large integer $n \ge 1$
so that the $\cO_L$-submodule
$L'_1 = L_{2,/\fram_L^n I_2}^\perp$
of $V$ is contained in $\fram_L L_2 \cap L_1$.
Then $(L_2, \fram_L^n I_2,I_2)$ is an admissible triple.
Set $Y_0 = X_{L_2, \fram_L^n I_2,I_2}$.
Then $Y_0$ is an object of $\cC_0^\psi$ 
and there exists a morphism
$f_0 \colon  Y_0 \to X_0$ in $\wt{\cC}_0^\psi$.
This proves the claim.
\end{proof}

\begin{cor} \label{cor:classicalgp_cofiltered}
Suppose that $K$ is of odd residue characteristic.
Then the category $\Bil^\psi$ is $\Lambda$-connected 
and semi-cofiltered.
\end{cor}

\begin{proof}
First we prove that $\Bil^\psi$ is $\Lambda$-connected.
It is easy to see that $\wt{\cC}_0^\psi$ is a cofiltered poset.
This implies that $\wt{\Bil}^\psi$ is $\Lambda$-connected.
Let $X,Y$ be two objects of $\Bil^\psi$.
Since $\wt{\Bil}^\psi$ is $\Lambda$-connected, there exists
a diagram $X \leftarrow Z' \to Y$ in $\wt{\Bil}^\psi$.
It follows from Corollary \ref{cor:classicalgp_aux2} that
there exists a morphism $Z \to Z'$ in $\wt{\Bil}^\psi$ such that
$Z$ is an object of $\Bil^\psi$.
Thus we obtain a diagram $X \leftarrow Z \to Y$ in $\Bil^\psi$.
This proves that $\Bil^\psi$ is $\Lambda$-connected.

Next we prove that $\Bil^\psi$ is semi-cofiltered.
Let $Y_1 \xto{f_1} X \xleftarrow{f_2} Y_2$ be a
diagram in $\Bil^\psi$.
Since $\wt{\Bil}^\psi_{/X}$ is $\Lambda$-connected,
there exist an object $Z'$ of $\wt{\Bil}^\psi$
and morphisms $g'_1\colon  Z' \to Y_1$ and $g'_2\colon Z' \to Y_2$
in $\wt{\Bil}^\psi$ satisfying $f_1 \circ g'_1 = f_2 \circ g'_2$.
It follows from Corollary \ref{cor:classicalgp_aux2} that
there exists a morphism $h\colon  Z \to Z'$ in $\wt{\Bil}^\psi$ such that
$Z$ is an object of $\Bil^\psi$.
By setting $g_1 = g'_1 \circ h$ and $g_2 = g'_2 \circ h$,
we obtain a diagram $Y_1 \xleftarrow{g_1} Z \xto{g_2} Y_2$ in $\Bil^\psi$
satisfying $f_1 \circ g_1 = f_2 \circ g_2$.
This proves that $\Bil^\psi$ is semi-cofiltered.
\end{proof}

\begin{cor} \label{cor:classicalgp_B}
Suppose that $K$ is of odd residue characteristic.
Then the category $\Bil^\psi$
equipped with the atomic topology $J$ is a $B$-site.
\end{cor}

\begin{proof}
Since it follows from Corollary \ref{cor:classicalgp_cofiltered}
that $\Bil^\psi$ is semi-cofiltered,
one can consider the atomic topology $J$ on $\Bil^\psi$.
It follows from Lemma \ref{lem:classicalgp_epi} 
that $\Bil^\psi$ is an $E$-category.
Hence $(\Bil^\psi,J)$ is a $B$-site.
\end{proof}

\subsection{The pair $(\cC_0^\psi,\iota_0)$}

The restriction of the functor $\wt{\iota}_0$ 
to $\cC_0^\psi$ factors through the inclusion functor $\Bil^\psi
\inj \wt{\Bil}^\psi$. Let $\iota_0\colon  \cC_0^\psi \to \Bil^\psi$
denote the induced functor.

\begin{prop} \label{prop:classicalgp_Y}
Suppose that $K$ is of odd residue characteristic.
Then 
\\
(1) The category $\Bil^\psi$
equipped with the atomic topology $J$ is a $Y$-site,
\\
(2) The pair $(\cC_0^\psi,\iota_0)$ is a grid of
$(\Bil^\psi,J)$,
\\
(3) The topological monoid $M_{(\cC_0^\psi,\iota_0)}$ 
is canonically isomorphic to the locally profinite group 
$$
G(\psi) = \{ (g,\lambda) \in \GL_L(V)
\times K^\times \ |\ \psi(gv,gw)= \lambda \psi(v,w)
\text{ for all }v,w \in V \}.
$$
Here we equip $G(\psi)$ with the topology
induced from the topology on $\GL_L(V)
\times K^\times$.
\end{prop}

\begin{proof}[Proof of (1)]
First we prove that $(\Bil^\psi,J)$ is a $Y$-site.
It follows from Corollary \ref{cor:classicalgp_B}
that $(\Bil^\psi,J)$ is a $B$-site.
It follows from Corollary \ref{cor:classicalgp_cofiltered} 
that $\Bil^\psi$ is $\Lambda$-connected.
It is obvious that $\Bil^\psi$ is essentially small since
$\wt{\Bil}^\psi$ is essentially small.
Hence to prove that $(\Bil^\psi,J)$ is a $Y$-site,
it suffices to prove that the collection of morphisms
in $\Bil^\psi$ has enough Galois covering.
Let $f\colon Y \to X = (M,N,\vep,\phi)$ be a morphism in $\Bil^\psi$.
We prove that there exists a morphism $g\colon Z \to Y$ in $\Bil^\psi$
such that the composite $f \circ g$ is a Galois covering in $\Bil^\psi$.
By Lemma \ref{lem:classicalgp_under2}, we may assume that
$f = \iota_0(f_0)$ for some morphism $f_0\colon  Y_0 \to X_0$ in
$\cC_0^\psi$.
Let us write $X_0 = (L_1,L_2,I_1,I_2)$ and
$Y_0 = (L'_1,L'_2,I'_1,I'_2)$.
Since $X$ is $d$-good,
the triple $(L_2, I_1, I_2)$ is admissible.
Set $Z_0 = X_{L_2, I_1, I_2}$. By definition
$Z_0 =(L''_1,L_2,I_1,I_2)$ where
$L''_1 = L^\perp_{2,/I_1}$.
It follows from Lemma \ref{lem:classicalgp_triple}
that $Z_0$ is an object of $\cC_0^\psi$.
Let us choose an integer $n \ge 0$ satisfying
$\fram_L^{2n} L''_1 \subset L'_1$, $L'_2 \subset \fram_L^{-n} L_2$, 
$\fram_L^n I_1 \subset I'_1$, and $I'_2 \subset \fram_L^{-2n} I_2$.
Then 
$Z_n = T^n(Z_0)$ is an object of $\cC_0^\psi$.
Let $g_0$ denote the unique morphism from $Z_n$ to $Y_0$ in $\cC_0^\psi$
and set $g=\iota_0(g_0)$.
It follows from Lemma \ref{lem:classicalgp_Galois} that
$f \circ g$ is a Galois covering in $\Bil^\psi$.
This proves that $(\Bil^\psi,J)$ is a $Y$-site.
\end{proof}

\begin{proof}[Proof of (2)]
Next we prove that $(\cC_0^\psi,\iota_0)$ is a grid of
$(\Bil^\psi,J)$. Since $J$ is the atomic topology, it suffice
to show that $(\cC_0^\psi,\iota_0)$ is a pregrid of $\Bil^\psi$.
We will check that $(\cC_0^\psi,\iota_0)$ satisfies the four conditions
in Definition 5.5.1 of \cite{Grids}.
It is easy to check that the poset $\cC_0^\psi$ is $\Lambda$-connected.
It is clear from the definition that the functor 
$\wt{\iota}_0 \colon  \wt{\cC}_0^\psi \to \wt{\Bil}^\psi$
is essentially surjective.
Hence the functor $\iota_0 \colon  \cC_0^\psi \to \Bil^\psi$
is essentially surjective.
Let $X$ be an object of $\cC_0^\psi$.
It follows from Lemma \ref{lem:classicalgp_under2} that
the functor $\wt{\cC}^\psi_{0,X/} \to \wt{\Bil}^\psi_{\wt{\iota}_0(X)/}$
induced by $\wt{\iota}_0$ is an equivalence of categories.
This implies that the functor
$\cC^\psi_{0, X/} \to \Bil^\psi_{\iota_0(X)/}$
induced by $\iota_0$ is an equivalence of categories.
Hence $(\cC_0^\psi,\iota_0)$ satisfies (1), (2), and (4)
in Definition 5.5.1 of \cite{Grids}.

To prove that $(\cC_0^\psi,\iota_0)$ is a grid of
$(\Bil^\psi,J)$, 
it remains to show that $(\cC_0^\psi,\iota_0)$ satisfies (3)
in Definition 5.5.1 of \cite{Grids}.
Let $X_0$ be an object of $\cC_0^\psi$ and set $X = \iota_0(X_0)$.
Let $f\colon  Y \to X$ be a morphism in $\Bil^\psi$.
Let us write $X_0 = (L_1,L_2,I_1,I_2)$ 
and $Y=(M,N,\vep,\phi)$.
Let us take an integer $n \ge 0$ satisfying
$\fram_L^{2n} (L_{2,/I_2}^\perp) \subset L_1$,
$\fram_L^n M =\{0\}$, and $\fram_L^n N = \{0\}$.
Set $Z_0 = T^n(X_{L_2,I_1,I_2})$.
Then it follows from Lemma \ref{lem:classicalgp_triple} that
$Z_0$ is an object of $\cC_0^\psi$.
By construction there exists a morphism
$g_0 \colon  Z_0 \to X_0$. 
Set $Z = \iota_0(Z_0)$ and $g = \iota_0(g_0)$.
Since $\Bil^\psi$ is
semi-cofiltered, there exist an object
$W$ of $\Bil^\psi$ and
morphisms $f'\colon  W \to Y$, $g'\colon  W \to Z$
in $\Bil^\psi$ satisfying $f \circ f' = g \circ g'$.
It follows from the choice of $n$ that
the morphism $f'$ factors through $g'$, i.e,
there exists a morphism $h\colon  Z \to Y$
satisfying $f' = h \circ g'$.
We have already proved that
$\cC^\psi_{0, Z_0/} \to \Bil^\psi_{Z/}$ is an
equivalence of categories.
Hence there exist an object $Y_0$ of $\cC_0^\psi$,
morphisms $h_0\colon  Z_0 \to Y_0$ and 
$f_0\colon  Y_0 \to X_0$ in $\cC_0^\psi$,
and an isomorphism $\beta\colon  \iota_0(Y_0) \xto{\cong} Y$
in $\Bil^\psi$ satisfying $h = \beta \circ \iota_0(h_0)$.
Since $\Bil^\psi$ is an $E$-category, this implies that
$\iota_0(f_0) = f \circ \beta$.
This proves that the functor
$\cC^\psi_{0, /X_0} \to \Bil^\psi_{/X}$
induced by $\iota_0$ is essentially surjective.
Hence $(\cC_0^\psi,\iota_0)$ is a grid of
$(\Bil^\psi,J)$.
\end{proof}

\begin{proof}[Proof of (3)]
Now let us compute the monoid
$M_{(\cC_0^\psi,\iota_0)}$.
For $(g,\lambda) \in G(\psi)$, let 
$\wt{\alpha}_{g,\lambda} \colon  \wt{\cC}_0^\psi \to \wt{\cC}_0^\psi$
denote the functor that sends an object $(L_1,L_2,I_1,I_2)$
of $\wt{\cC}_0^\psi$ to $(g(L_1), g(L_2), \lambda I_1, \lambda I_2)$.
Then for each object $(L_1,L_2,I_1,I_2)$
of $\wt{\cC}_0^\psi$, we have isomorphisms 
$L_2/L_1 \xto{\cong} g(L_2)/g(L_1)$
and $I_2/I_1 \xto{\cong} \lambda I_2/\lambda I_1$ of
$\cO_L$-modules given by $g$ and the multiplication by 
$\lambda$, respectively.
These isomorphisms induce a natural isomorphism
$\wt{\iota}_0 \xto{\cong} \wt{\iota}_0 \circ \wt{\alpha}_{g,\lambda}$
which we denote by $\wt{\gamma}_{g,\lambda}$.
Since the full subcategory $\Bil^\psi$ of $\wt{\Bil}^\psi$
is strict, the functor $\wt{\alpha}_{g,\lambda}$ and
the natural isomorphism $\wt{\gamma}_{g,\lambda}$ induce
a functor $\alpha_{g,\lambda} \colon  \cC_0^\psi \to \cC_0^\psi$
and a natural isomorphism
$\gamma_{g,\lambda} \colon  \iota_0 \xto{\cong} 
\iota_0 \circ \alpha_{g,\lambda}$.
By sending $(g,\lambda)$ to the pair 
$(\alpha_{g,\lambda},\gamma_{g,\lambda})$, we obtain
a map $\Phi \colon  G(\psi) \to M_{(\cC_0^\psi,\iota_0)}$.
It is easy to see that $\Phi$ is a homomorphism of monoids.

Next we prove that $\Phi$ is bijective.
Let us fix $\cO_L$-lattices $L_0 \subset V$
and $I_0 \subset D$ satisfying $\psi(L_0 \times L_0) \subset I_0$ 
and set $L'_1 = L_{0,/I_0}^\perp$.
Let us choose an integer $n_0 \ge 0$ such that
$\fram_L^{n_0} L'_1 \subset \fram_L L_0$.
Then for any $n \ge n_0$, the triple
$(L_0,\fram_L^n I_0, I_0)$ is an admissible triple.
Hence it follows from Lemma \ref{lem:classicalgp_triple} that 
$X_n = X_{L_0,\fram_L^n I_0, I_0}$ is
an object of $\cC_0^\psi$.
Let $(g,\lambda), (g',\lambda') \in G(\psi)$ and suppose
that $\Phi(g,\lambda) = \Phi(g',\lambda')$.
Then for any $n \ge n_0$ we have
$g(L_0) = g'(L_0)$, $g(\fram_L^n L'_1)= g'(\fram_L^n L'_1)$,
$\lambda I_0 = \lambda' I_0$, and the pair of 
isomorphisms
$L_0/\fram_L^n L'_1 \xto{\cong} g(L_0)/g(\fram_L^n L'_1)$
given by $g$ 
and $I_0/\fram_L^n I_0 \xto{\cong} \lambda I_0/\fram_L^n \lambda I_0$
given by $\lambda$ 
is equal to the pair
of isomorphisms
$L_0/\fram_L^n L'_1 \xto{\cong} g(L_0)/g(\fram_L^n L'_1)$
given by $g'$ 
and $I_0/\fram_L^n I_0 \xto{\cong} \lambda I_0/\fram_L^n \lambda I_0$
given by $\lambda'$.
Since $\bigcap_{n \ge n_0} g(\fram_L^n L'_1) = \{0\}$
and $\bigcap_{n \ge n_0} \fram_L^n \lambda I_0 = \{0\}$,
we have $(g,\lambda) = (g',\lambda')$.
This proves that $\Phi$ is injective.
Let $(\alpha,\gamma_\alpha)$ be an arbitrary element of 
$M_{(\cC_0^\psi,\iota_0)}$. For any $n \ge n_0$,
set $\alpha(X_n) = (L_{1,n},L_{2,n},I_{1,n},I_{2,n})$.
Then $\gamma_\alpha$ induces $\cO_L$-linear isomorphisms
$g_n \colon  L_0/\fram_L^n L'_1 \cong L_{2,n}/L_{1,n}$
and $\lambda_n \colon  I_0/\fram_L^n I_0 \cong I_{2,n}/I_{1,n}$.
The naturality of $\gamma_\alpha$ implies 
$L_{2,n} = L_{2,n_0}$,
$L_{1,n} = \fram_L^{n-n_0} L_{1,n_0}$,
$I_{2,n} = I_{2,n_0}$,
and $I_{1,n} = \fram_L^n I_{2,n_0}$.
Hence these $\cO_L$-linear isomorphisms induces
$\cO_L$-linear isomorphisms $L_0 \xto{\cong} L_{2,n_0}$
and $I_0 \xto{\cong} I_{2,n_0}$.
By extending these isomorphisms by $L$-linearity,
we obtain $L$-linear isomorphisms
$g\colon  V \to V$ and $\lambda \colon  D \to D$.
Since each $\lambda_n$ is compatible with 
$\eps$, we have $\lambda \in K^\times$.
Since the pair $(g_n,\lambda_n)$ gives an isomorphism
from $\iota_0(X_n)$ to $\iota_0(\alpha(X_n))$ in
$\Bil^\psi$, it follows that the pair
$(g,\lambda)$ belongs to $G(\psi)$.
For any $n \ge n_0$, the triple 
$(\fram_L^{-n} L_0,\fram_L^n I_0, \fram_L^{-2n} I_0)$
is an admissible triple.
Hence it follows from Lemma \ref{lem:classicalgp_triple} that 
$Y_n = X_{\fram_L^{-n} L_0,\fram_L^n I_0, \fram_L^{-2n} I_0}$
is an object of $\cC_0^\psi$.
One then can check that we have
$\alpha(Y_n) = \alpha_{g,\lambda}(Y_n)$
and the isomorphism 
$\gamma_\alpha(Y_n) \colon  \iota_0(Y_n) \xto{\cong}
\iota_0(\alpha(Y_n))$ is equal to
$\gamma_{g,\lambda}(Y_n)$.
Since the family $(Y_n)_{n \ge 0}$ of objects of $\cC_0^\psi$
is cofinal in $\cC_0^\psi$, it follows that
$\alpha = \alpha_{g,\lambda}$ and $\gamma_{\alpha} = \gamma_{g,\lambda}$.
This proves that $\Phi$ is surjective.

Finally we prove that $\Phi$ is an isomorphism of
topological monoids.
For $\cO_L$-lattices $L_1,L_2 \subset V$ and
$I_1,I_2 \subset D$ satisfying 
$L_1 \subset L_2$ and $I_1 \subset I_2$,
let $H(L_1,L_2,I_1,I_2)$ denote the subgroup
of $\GL_L(V) \times K^\times$ that consists of
the elements $(g,\lambda) \in \GL_L(V) \times K^\times$
satisfying $(g(L_1),g(L_2), \lambda I_1, \lambda I_2)
= (L_1,L_2,I_1,I_2)$ and the $\cO_L$-linear automorphisms of
$L_2/L_1$ and $I_2/I_1$ given by $g$ and $\lambda$
respectively are equal to the identity morphisms.
If $X = (L_1,L_2,I_1,I_2)$ is an object of $\cC_0^\psi$,
then the subgroup $\Phi^{-1}(\bK_X)$ of $G(\psi)$
is equal to $G(\psi) \cap H(L_1,L_2,I_1,I_2)$.
In particular $\Phi^{-1}(\bK_X)$ 
is an open subgroup of $G(\psi)$.
Conversely, Let $H \subset G(\psi)$ be an open subgroup.
Then there exists $\cO_L$-lattices $L_1,L_2 \subset V$ and
$I_1,I_2 \subset D$ satisfying 
$L_1 \subset L_2$ and $I_1 \subset I_2$
and $G(\psi) \cap H(L_1,L_2,I_1,I_2) \subset H$.
Let us choose a sufficiently large $n \ge n_0$ satisfying
$\fram_L^{2n} L_0 \subset L_1 \subset L_2 \subset \fram_L^{-n} L_0$
and $\fram_L^n I_0 \subset I_1 \subset I_2 \subset \fram_L^{-2n} I_0$.
Then we have $\Phi(H) \supset \bK_{Y_n}$, which implies that
$\Phi(H)$ is an open subgroup of $M_{(\cC_0^\psi,\iota_0)}$.
This proves that $\Phi$ is an isomorphism of topological
groups. 
\end{proof}

\begin{rmk}
One also obtains a monoid version of the statements above
by introducing an appropriate non-atomic topology on $\Bil^\psi$.
\end{rmk}

\subsection{A variant}
In some applications, it seems more convenient to introduce
a variant of the pair $((\Bil^\psi,J),(\cC_0^\psi,\iota_0))$.
In this section, we will give a definition of the variant,
which will be denoted by 
$((\wh{\Bil}^\psi,J),(\wt{\cC}_0^\psi,\wh{\iota}_0))$.

To this end, we will introduce a notion of a blowup of a category.

\subsection{Blowup of a category}
Let $\cC$ be a category 
and let $\cD$ be its full subcategory. 
Suppose that $\cD$ is essentially small and semi-cofiltered.
Under this situation we will introduce a category
$\Bl(\cC,\cD)$ and call it the blowup of $\cC$ at 
the complement of $\cD$. The category 
$\Bl(\cC,\cD)$ has the following property:
the inclusion functor $\cD \to \cC$ factors as
$$
\cD \xto{\iota} \Bl(\cC,\cD) \to \cC
$$
with $\iota$ fully faithful.
\begin{defn}
We say that a morphism $f\colon X\to Y$ in $\cC$ is 
a {\em move from $\cD$ in $\cC$} if its domain $X$ is an
object of $\cD$. 
We say that two moves $f\colon  X\to Y$ and $f'\colon X' \to Y'$
from $\cD$ in $\cC$ are {\em equivalent} if $Y=Y'$ and
there exists a diagram $f \leftarrow g \to f'$ in the
overcategory $\cC_{/Y}$.
\end{defn}%
By using the assumption that $\cD$ is essentially small and semi-cofiltered,
one sees that this notion gives an equivalence relation on
the collection of moves from $\cD$ in $\cC$, and that
for any object $Y$ of $\cC$, 
the equivalence classes of moves
whose codomain is $Y$ forms a small set.

We let $\Bl(\cC,\cD)$ denote the following category.
\begin{defn}
The objects of $\Bl(\cC,\cD)$ are the equivalence classes
of moves from $\cD$ in $\cC$. 
Let $Z$ and $Z'$ be two objects of $\Bl(\cC,\cD)$. 
Let us choose moves $f\colon X \to Y$
and $f'\colon X' \to Y'$ from $\cD$ in $\cC$ which represent $Z$ and $Z'$,
respectively. Then the morphisms from $Z$ to $Z'$ in $\Bl(\cC,\cD)$
are the morphisms $g\colon Y \to Y'$ in $\cC$ such that the two moves 
$g \circ f$ and $f'$ are equivalent.
\end{defn}

One can check easily that the category $\Bl(\cC,\cD)$ is well-defined
and is semi-cofiltered. We call the category $\Bl(\cC,\cD)$ 
the {\em blowup of $\cC$ at the complement of $\cD$}.

By associating the class of a move $f$ to its codomain,
we obtain a functor $\Bl(\cC,\cD) \to \cC$.
Moreover by associating $\id_X\colon  X \to X$ to
each object $X$ of $\cD$, we obtain a functor 
$\cD \to \Bl(\cC,\cD)$ which is fully faithful. 

\begin{rmk}
Later we will use this notion when
$\cC = \wt{\Bil}^\psi$ and $\cD = \Bil^\psi$ and 
define the category  $\wh{\Bil}^\psi$ as the blowup
$\Bl(\wt{\Bil}^\psi,\Bil^\psi)$.

When constructing a $Y$-site for classical groups,
it was not so convenient to use $\wt{\Bil}^\psi$
and we used the full subcategory $\Bil^\psi$
as the underlying category.
However, in this process, we eliminate basic 
objects
like $(\{0\}, \{0\}, 0\colon \{0\} \times \{0\} \to \{0\})$.
In applications, we suspect that this type of objects should
be included.   The category $\wh{\Bil}^\psi$
is introduced to circumvent this type of problem.
In $\wh{\Bil}^\psi$, such objects regenerate in a 
proper manner, and moreover, we can use this 
as the underlying category of a $Y$-site such 
that the toposes are equivalent.

We called this new category a blowup, because the 
sequence 
$\cD \to \Bl(\cC, \cD) \to \cC$
resembles the sequence 
$U \to \Bl_Z(X) \to X$
of schemes where $Z \subset X$ is a closed 
subscheme, $U=X \setminus Z$,
 and $\Bl_Z(X)$ is the blowup of $X$ at $Z$.
\end{rmk}

\subsubsection{}
We define the category $\wh{\Bil}^\psi$ 
to be the blowup $\Bl(\wt{\Bil}^\psi,\Bil^\psi)$.
%
One then can check that the fully faithful functor
$\Bil^\psi \to \wh{\Bil}^\psi$ induces an equivalence
$$
\Shv(\wh{\Bil}^\psi,J) \xto{\cong} \Shv(\Bil^\psi,J)
$$
of categories.
We note that, any move from $\Bil^\psi$ in $\wt{\Bil}^\psi$ 
is a morphism 
in $\wt{\Bil}^\psi$. Hence we have 
$\wh{\Bil}^\psi = \Bl(\wt{\Bil}^\psi,\Bil^\psi)$.
The functor $\wt{\iota}_0$ induces a functor 
$\Bl(\wt{\cC}_0^\psi, \cC_0^\psi) \to \wh{\Bil}^\psi$.
Since the functor $\Bl(\wt{\cC}_0^\psi, \cC_0^\psi) \to \wt{\cC}_0^\psi$
is an isomorphism of categories, we obtain a functor
$\wt{\cC}_0^\psi \to \wh{\Bil}^\psi$
which we denote by $\wh{\iota}_0^\psi$.
Let $J$ denote the atomic topology on the semi-cofiltered category
$\wh{\Bil}^\psi$. Then one can check that $(\wh{\Bil}^\psi,J)$ is a $Y$-site
and $(\wt{\cC}_0^\psi,\wh{\iota}_0)$ is a grid for $(\wh{\Bil}^\psi,J)$.
Moreover the commutative diagram
$$
\begin{CD}
\cC_0^\psi @>>> \wt{\cC}_0^\psi \\
@V{\iota_0}VV @VV{\wh{\iota}_0}V \\
\Bil^\psi @>>> \wh{\Bil}^\psi
\end{CD}
$$
of categories and functors induces an isomorphism
$M_{(\wt{\cC}_0^\psi,\wh{\iota}_0)} \to M_{(\cC_0^\psi,\iota_0)}$ 
of groups.

\begin{rmk}
Let $\wt{\Psi}\colon G(\psi) \to M_{(\wt{\cC}_0^\psi, \wh{\iota}_0)}$ denote the composite of the isomorphism $\Psi \colon G(\psi) \to M_{(\cC_0^\psi,\iota_0)}$ with the inverse of the isomorphism $M_{(\wt{\cC}_0^\psi,\wh{\iota}_0)} \to M_{(\cC_0^\psi,\iota_0)}$ above.
Let $X=(L_1,L_2,I_1,I_2)$ be an object of $\wt{\cC}_0^\psi$. Let us consider the subgroup $H(L_1,L_2,I_1,I_2)$ of $\GL_L(V) \times K^\times$ introduced in the proof of Proposition~\ref{prop:classicalgp_Y}.
Then the subgroup $\wt{\Psi}^{-1}(\bK_X)$ of $G(\psi)$ is equal to $G(\psi) \cap H(L_1,L_2,I_1,I_2)$.
\end{rmk}

\subsubsection{An example: compact open subgroups of symplectic similitude groups}
\label{sec:Okazaki}
In this paragraph, we assume $L=K$, $D=K$, and $\epsilon=-1$. 
Then $\psi \colon V \times V \to K$ is a non-degenerate skew-symmetric
bilinear form. This implies $V$ is even-dimensional.
Let us write $\dim_K V = 2n$. One can find a $K$-basis 
$e_1,\ldots,e_{2n}$ of $V$ with respect to which
the Gram matrix of $\psi$ is given by the matrix
\[
\left(\begin{array}{ccc|ccc}
\multicolumn{3}{c|}{\multirow{3}{*}{${\Huge O}_n$}} & & & -1 \\
 & & & & \reflectbox{$\ddots$} & \\
 & & & -1 & & \\
\hline
 & & 1 & \multicolumn{3}{c}{\multirow{3}{*}{${\Huge O}_n$}} \\
 & \reflectbox{$\ddots$} & & & \\
1 & & & & &
\end{array}\right).
\]
Fix such a basis $e_1,\ldots,e_{2n}$ 
and regard $V$ as the space $K^{\oplus 2n}$ of row vectors
via the isomorphism $K^{\oplus 2n} \xto{\cong} V$
that sends $(a_1,\ldots,a_{2n}) \in K^{\oplus 2n}$
to $\sum_i a_i e_i \in V$.
The group $\GL_{2n}(K)$ acts on $V$ from the left by the map
$\GL_{2n}(K) \times V \to V$ that sends $(g,v)$ to $v g^{-1}$.
This action gives an isomorphism 
$\jmath \colon \GL_{2n}(K) \xto{\cong} \GL_K(V)$.
By sending $(\jmath(g),\lambda) \in G(\psi)$
to $(g,\lambda^{-1})$, we obtain an isomorphism
$G(\psi) \xto{\cong} \GSp_{2n}(K)$.

Let $X = (L_1,L_2,I_1,I_2)$ be an object of
$\wt{C}_0^\psi$. We say that $X$ is strongly cyclic
if the $\cO_K$-module $L_2/L_1$ is generated by at most one
element and the homomorphism $L_1 \otimes_{\cO_K} L_1 \to I_2$
induced by $\psi$ is surjective.

We expect that, for any irreducible generic smooth complex
representation $\pi$ of $G(\psi)$, there exists a 
strongly cyclic object $X$ of $\wt{\cC}_0^\psi$
such that the $\wt{\Psi}^{-1}(\bK_X)$-invariant part 
of $\pi$ is
one-dimensional.

In the rest of this paragraph we assume $n=2$.
We will give a complete list, up to conjugation, 
of the subgroups of $G(\psi) \cong \GSp_4(K)$
of the form $\wt{\Psi}^{-1}(\bK_X)$ for a
strongly cyclic object $X$ of $\wt{\cC}_0^\psi$.

Let $X=(L_1,L_2,I_1,I_2)$ be a strongly cyclic object of
$\wt{\cC}_0^\psi$. Let us fix a uniformizer $\varpi \in \cO_K$.
Then one can choose an $\cO_K$-basis
$f_1,f_2,f_3,f_4$ of $L_1$ with respect to which the
Gram matrix of $\psi$ is given by the matrix
$$
\begin{pmatrix}
0 & 0 & 0 & -\varpi^{b_1} \\
0 & 0 & -\varpi^{b_2} & 0 \\
0 & \varpi^{b_2} & 0 & 0 \\
\varpi^{b_1} & 0 & 0 & 0
\end{pmatrix}
$$
for some integers $b_1,b_2$ satisfying $b_1 \ge b_2$.
Then we have $I_1=I_2= \varpi^{b_2} \cO_K$.
Let $\fram_K^e$ denote the annihilator of the
$\cO_K$-module $L_2/L_1$. Then one can choose
$v \in L_1$ such that $L_2$ is generated by
$L_1$ and $\varpi^{-e} v$.
Let us write $v = \sum_{i=1}^4 a_i f_i$.
We note that at least one of $a_1,\ldots, a_4$ is a
unit of $\cO_K$.
We set $\ell = b_1 - b_2$.
The equality $I_1= \varpi^{b_2} \cO_K$ shows that
$\psi(\varpi^{-e} v, L_1) \subset \varpi^{b_2} \cO_K$.
This implies $0 \le e \le \ell$ and 
$a_2,a_3 \in \varpi^e \cO_K$.
By replacing $v$ with an element of $v + \varpi^e L_1$,
we may assume $a_2 = a_3 = 0$.
By replacing $v$ with its multiple by a unit of $\cO_K$,
we may also assume that either $a_1=1$ or $a_4=1$.
By replacing $f_1, f_4$ with $-f_4, v$ if $a_1=1$, and
by replacing $f_4$ with $v$ otherwise,
we may assume, without changing the Gram matrix of $\psi$
with respect to the basis $f_1,\ldots,f_4$, 
that $v = f_4$.

Let $h$ denote the $4$-by-$4$ matrix
whose first row is $\varpi^{-\ell} f_1$, and
whose $i$-th row is $f_i$ for $i=2,3,4$.
Then $h \in \GL_4(K)$ and the pair
$\wt{h} = (\jmath(h), \varpi^{-b_2})$ is an element 
of $G(\psi)$.

We set $m = \ell + e$.
Since $\ell \ge e$, we have $m \ge 2e$.
For $i=1,2$ let $\bK_i \subset \GL_4(K)$ denote the subgroup of elements $g \in \GL_4(K)$ satisfying $L_i g = L_i$.
Then $h \bK_1 h^{-1}$ consists of the $4$-by-$4$ matrices $g$ such that
\[
g \in \begin{pmatrix}
\cO_K & \fram_K^{-\ell} & \fram_K^{-\ell} & \fram_K^{-\ell} \\
\fram_K^\ell & \cO_K & \cO_K & \cO_K \\
\fram_K^\ell & \cO_K & \cO_K & \cO_K \\
\fram_K^\ell & \cO_K & \cO_K & \cO_K
\end{pmatrix}
\]
and $\det(g) \in \cO_K^\times$, and
$h \bK_2 h^{-1}$ consists of the $4$-by-$4$ matrices $g$ such that
\[
g \in \begin{pmatrix}
\cO_K & \fram_K^{-\ell} & \fram_K^{-\ell} & \fram_K^{-m} \\
\fram_K^\ell & \cO_K & \cO_K & \fram_K^{-e} \\
\fram_K^\ell & \cO_K & \cO_K & \fram_K^{-e} \\
\fram_K^m & \fram_K^e & \fram_K^e & \cO_K
\end{pmatrix}
\]
and $\det(g) \in \cO_K^\times$.
Note that $H(L_1,L_2,I_1,I_2)$ is equal to the kernel of the homomorphism 
$\jmath(\bK_1 \cap \bK_2) \times \cO_K^\times \to \Aut_{\cO_K}(L_2/L_1) = (\cO_K/\fram_K^e)^\times$
that sends $(\jmath(g),\lambda) \in \jmath(\bK_1 \cap \bK_2) \times \cO_K^\times$ to the $(4,4)$-entry of $h g h^{-1}$ modulo $\fram_K^e$.
It follows from the descriptions of
$\bK_1$ and $\bK_2$ above that
$h(\bK_1 \cap \bK_2)h^{-1}$ consists 
of the $4$-by-$4$ matrices $g$ such that
\[
g \in \begin{pmatrix}
\cO_K & \fram_K^{-\ell} & \fram_K^{-\ell} & \fram_K^{-\ell} \\
\fram_K^\ell & \cO_K & \cO_K & \cO_K \\
\fram_K^\ell & \cO_K & \cO_K & \cO_K \\
\fram_K^m & \fram_K^e & \fram_K^e & \cO_K
\end{pmatrix}
\]
and $\det(g) \in \cO_K^\times$.
We claim that, for any element $(\jmath(g),\lambda)$ of
$G(\psi) \cap (\jmath(\bK_1 \cap \bK_2)\times \cO_K^\times)$, the matrix
$h g h^{-1}$ is an element of
\[
\begin{pmatrix}
\cO_K & \cO_K & \cO_K & \fram_K^{-\ell} \\
\fram_K^\ell & \cO_K & \cO_K & \cO_K \\
\fram_K^\ell & \cO_K & \cO_K & \cO_K \\
\fram_K^m & \fram_K^\ell & \fram_K^\ell & \cO_K
\end{pmatrix}.
\]
The claim is clear if $\ell=0$.
Suppose that $\ell \ge 1$.
Let us write $h g h^{-1} =(s_{i,j})_{1 \le i, j \le 4}$.
Let $g'=(s_{i,j})_{2 \le i,j \le 3}$ denote the 
$\{2,3\} \times \{2,3\}$ submatrix
of $h g h^{-1}$. 
For $i=1,\ldots,4$ let $v_i$ denote the $i$-th row of $h g h^{-1}$.
Then $\psi(v_2,v_3)$ belongs to 
$\cO_K^\times$. Since we have assumed $\ell \ge 1$, this implies that
$g'$ belongs to $\GL_2(\cO_K)$.
Thus the equalities 
$\psi(v_2,v_4) = \psi(v_3,v_4)=0$ show that the second and the third entries of $v_4$ belong to $\cO_K$.
Similarly, the equalities 
$\psi(v_1,v_2) = \psi(v_1,v_3)=0$ show that the second and the third entries of $v_1$ belong to $\fram_K^\ell$.
This completes the proof of the claim.

Thus, the subgroup 
$\wt{h} \wt{\Psi}^{-1}(\bK_X) \wt{h}^{-1}$ of 
$G(\psi) \cong \GSp_4(K)$ is equal to
the quasi-paramodular group $\mathbf{K}_1(m,e)
\subset \GSp_4(K)$ of level $m$ 
introduced by Okazaki \cite{Okazaki} in his study of 
local newforms for generic representations of $\GSp_4(K)$.

\section{A $Y$-site for a connected locally noetherian scheme over a scheme}
\label{sec:conn loc noeth}
In this section we introduce, for a scheme $S$
and a connected locally noetherian scheme $X$ over $S$,
a semi-cofiltered category $\SCH_{S,X}$.
The category $\SCH_{S,X}$ contains the category $\Et^0(X)$ 
of connected objects of the small \'etale site $\Et(X)$, and is contained in
the category of schemes over $S$.
Let $J$ denote the atomic topology on $\SCH_{S,X}$. 
We will prove that $(\SCH_{S,X},J)$ is a $Y$-site and
that for any morphism $f\colon Y \to Z$ in $\SCH_{S,X}$,
there are only finitely many automorphisms of $Y$ over $Z$.
Hence it follows from Proposition 6.2.1 of \cite{Grids} that there exists
a grid $(\cC_0,\iota_0)$ of the $Y$-site $(\SCH_{S,X},J)$.
We will give a more explicit construction of a grid
$(\cC_0,\iota_0)$ of $(\SCH_{S,X},J)$.

Let $M_{S,X}$ denote the absolute Galois monoid $M_\Cip$.
Later we will see in Corollary \ref{cor:hyp_M} that
$M_{S,X}$ is a locally profinite group
and is independent, up to isomorphisms, of the choice of
a grid $(\cC_0,\iota_0)$.
It follows from Theorem 5.8.1 of \cite{Grids} that 
the fiber functor $\omega_\Cip$ gives an equivalence
of categories from $\Shv(\SCH_{S,X},J)$ to the 
category of smooth left $M_{S,X}$-sets.

Let $(\cC_0,\iota_0)$ be a grid of $(\SCH_{S,X},J)$.
Then for an object $Y_0$ of $\cC_0$,
the open subgroup $\bK_{Y_0}$ of $M_{S,X}$ 
is a profinite group isomorphic to the 
\'etale fundamental group $\iota_0(Y_0)$.
In particular $M_{S,X}$ is a locally profinite group.

In the case when $S=X$, the locally profinite group $M_{S,X}$ 
is profinite and is isomorphic to the 
\'etale fundamental group $X$.
However, in general case,
the locally profinite group $M_{S,X}$ is not necessarily profinite.
In Section \ref{sec:hyp_counter}, we will give three
examples. One is given by a semi-abelian variety.
The others are given by a Shimura curve and
by an affine curve of positive characteristic.

\subsection{Preliminary} \label{sec:hyp_prelim}
Let $\cC$ be a category. Two objects $X$, $Y$ of $\cC$
are called $\Lambda$-related in $\cC$ if there exists a diagram
$X \leftarrow Z \to Y$ in $\cC$. 
So the category $\cC$ is $\Lambda$-connected if and only if
any two objects of $\cC$ are $\Lambda$-related in $\cC$.
If $\cC$ is semi-cofiltered, 
then $\Lambda$-relationship in $\cC$ gives an equivalence relation on 
the collection of objects of $\cC$.
For an object $X$ of $\cC$, the full subcategory of $\cC$ whose objects
are the objects of $\cC$ which are $\Lambda$-related to $X$ in $\cC$ 
is called the $\Lambda$-connected component of $X$ in $\cC$.
If $\cC_{/X}$ is $\Lambda$-connected, then
the $\Lambda$-connected component of $X$ in $\cC$ is $\Lambda$-connected.
If $\cC$ is semi-cofiltered, then
the $\Lambda$-connected component of $X$ in $\cC$ is $\Lambda$-connected
and semi-cofiltered.

\subsection{The $Y$-site $\SCH_{S,X}$}
Let $S$ be a scheme.
Let $\SCH_S$ be the following category. 
The objects of $\SCH_S$ are schemes over $S$
which are nonempty, connected, and locally noetherian.
For two objects $X$ and $Y$ of $\SCH_S$, the morphisms
from $X$ to $Y$ in $\SCH_S$ are the finite \'etale morphisms 
$X \to Y$ of schemes over $S$.
\begin{lem} \label{lem:hyp_E}
The category $\SCH_S$ is a semi-cofiltered $E$-category.
\end{lem}

\begin{proof}
First, we prove that any morphism in $\SCH_S$ is faithfully flat
and quasi-compact. Let $f\colon  X \to Y$ be a morphism in $\SCH_S$.
Then $f$ is flat since $f$ is \'etale, and $f$ is quasi-compact
since $f$ is finite and is in particular affine.
If suffices to show that $f$ is surjective.
Since $f$ is finite, it follows from Proposition 6.1.10 of \cite{EGAII}
that $f$ is universally closed.
Since $f$ is finite \'etale, $f$ is flat and is locally of
finite presentation. Hence it follows from Th\'eor\`eme 2.4.6 of \cite{EGAIV}
that $f$ is universally open.
Since $Y$ is connected, this implies that $f$ is surjective.

It follows from Corollaire 5.3 of \cite[Expos\'e VIII]{SGA1} that
any morphism in $\SCH_S$ is an epimorphism in the category of
schemes. This implies that $\SCH_S$ is an $E$-category.

Next, we prove that $\SCH_S$ is semi-cofiltered.
Let $Y_1 \to Z \leftarrow Y_2$ be a diagram in $\SCH_S$.
Let us consider the fiber product
$W = Y_1 \times_Z Y_2$ in the category of schemes.
Then $W$ is finite and surjective over a locally noetherian
scheme $Y_1$. Hence $W$ is nonempty and locally noetherian.
Let $W_0$ be a connected component of $W$.
Since $W$ is locally noetherian, it follows from
Corollaire 6.1.9 of \cite{EGAI} that $W_0$ is open and closed in $W$.
Hence the restriction to the projections
$W \to Y_i$ to $W_0$ is finite and \'etale.
Thus, we have a commutative diagram
$$
\begin{CD}
W_0 @>>> Y_2 \\
@VVV @VVV \\
Y_1 @>>> Z
\end{CD}
$$
in the category $\SCH_S$. This proves that $\SCH_S$ is
semi-cofiltered.
\end{proof}

Let $X$ be an object of $\SCH_S$ and consider the 
$\Lambda$-connected component $\SCH_{S,X}$ of $X$ in $\SCH_S$.
\begin{lem} \label{lem:hyp_conn}
The category 
$\SCH_{S,X}$ is $\Lambda$-connected and semi-cofiltered.
\end{lem}

\begin{proof}
It follows from Lemma \ref{lem:hyp_E} and the last sentence
in Section \ref{sec:hyp_prelim}.
\end{proof}

Let $J$ denote the atomic topology on $\SCH_{S,X}$.
\begin{prop} \label{prop:hyp_Y}
The pair $(\SCH_{S,X},J)$ is a $Y$-site.
Moreover, for any morphism $f\colon Y \to Z$ in $\SCH_{S,X}$,
there are only finitely many automorphisms of $Y$ over $Z$.
\end{prop}

\begin{proof}
First we prove that $(\SCH_{S,X},J)$ is a $Y$-site.
It follows from Lemma \ref{lem:hyp_E} that
$\SCH_{S,X}$ is an $E$-category.
It follows from Lemma \ref{lem:hyp_conn} that
$\SCH_{S,X}$ is $\Lambda$-connected and semi-cofiltered.
Hence it suffices to prove that
the collection of morphisms in $\SCH_{S,X}$ has enough 
Galois coverings.
Let $Y$ be an object of $\SCH_S$.
It suffices to prove that
the collection of morphisms in the overcategory 
$(\SCH_S)_{/Y}$ has enough Galois coverings.
The category $(\SCH_S)_{/Y}$ is
equal to the category $\Et^0(Y)$ of 
connected \'etale coverings of $Y$.
Since $\Et^0(Y)$ is the full subcategory of the Galois category
$\Et(Y)$ of \'etale coverings of $Y$, 
it is equivalent to the category $BG$
of nonempty finite sets with continuous transitive 
left $G$-actions for some profinite group $G$.
Observe that for any open subgroup $H \subset G$,
there exists an open normal subgroup of $G$ contained
in $H$.
Hence the collection of morphisms in $\Et^0(Y)$
has enough Galois coverings.
This proves that $(\SCH_{S,X},J)$ is a $Y$-site.

The last assertion follows since for any
two objects $S,T$ of $BG$, there exist only 
finitely many morphisms from $S$ to $T$ in $BG$.
\end{proof}

As we have seen in the proof of Proposition \ref{prop:hyp_Y},
the overcategory $(\SCH_{S,X})_{/X}$ is equivalent to
$BG$ for some profinite group $G$.
Let us fix an equivalence $\alpha\colon BG \to (\SCH_{S,X})_{/X}$.
Let $I$ denote the filtered poset of open subgroups of $G$.
For $H \in I$, let us write $\alpha(G/H) \colon  X_H \to X$.
Then $\wt{X} = (X_H)_{H \in I}$ is a pro-object in the category
$\SCH_{S,X}$.
Let $\cC'_0$ denote the following category.
An object of $\cC'_0$ is a pair
$(Y,f)$ of an object $Y$ of $\SCH_{S,X}$ and a morphism
$f\colon  \wt{X} \to Y$ in the category of pro-objects in $\SCH_{S,X}$,
i.e., $f$ is an element of the filtered colimit
$\varinjlim_{H \in I} \Hom_{\SCH_{S,X}}(X_H,Y)$.
For two objects $(Y,f)$ and $(Y',f')$ of $\cC'_0$,
a morphism from $(Y,f)$ to $(Y',f')$ in $\cC'_0$
is a morphism $g\colon  Y \to Y'$ in $\SCH_{S,X}$ 
satisfying $f' = g \circ f$ in the 
category of pro-objects in $\SCH_{S,X}$.

Since $\SCH_{S,X}$ is an $E$-category, it follows that
$\cC'_0$ is a quasi-poset.
Let $\iota'_0 \colon  \cC'_0 \to \SCH_{S,X}$ 
denote the functor that sends $(Y,f)$ to $Y$.

Let us choose a skeletal full subcategory $\cC_0$
of $\cC'_0$ and let $\iota_0$ denote the restriction
of $\iota'_0$ to $\cC_0$.

\begin{prop} \label{prop:hyp_grid}
The pair $(\cC_0,\iota_0)$ is a grid of
$(\SCH_{S,X},J)$.
\end{prop}

\begin{proof}
Since $J$ is the atomic topology, it suffices
to show that $(\cC_0,\iota_0)$ is a pregrid of $\SCH_{S,X}$.
We will check that $(\cC_0,\iota_0)$ satisfies the four conditions
in Definition 5.5.1 of \cite{Grids}.
It is easy to check that the poset $\cC_0$ is $\Lambda$-connected.
We note that, for any morphism $f\colon  Z \to X$ in $\SCH_{S,X}$, there
exists an element $H \in I$ such that $f$ and $\alpha(H)$ are
isomorphic in the overcategory $(\SCH_{S,X})_{/X}$.
Let $Y$ be an arbitrary object of $\SCH_{S,X}$.
By definition, we have a diagram $Y \xleftarrow{g} Z \xto{f} X$
in $\SCH_{S,X}$.
Let us choose an element $H \in I$ and an isomorphism
$\beta \colon  X_H \xto{\cong} Z$ from $\alpha(H)$ to $f$ in $(\SCH_{S,X})_{/X}$.
Then the pair $\wt{Y}$ of $Y$ and the class of $g \circ \beta$ is an object
of $\cC'_0$ satisfying $\iota'_0(\wt{Y}) = Y$.
This shows that the functor 
$\iota'_0 \colon  \cC'_0 \to \SCH_{S,X}$
is essentially surjective.
Hence the functor $\iota_0 \colon  \cC_0 \to \SCH_{S,X}$
is essentially surjective.
Let $(Y,f)$ be an object of $\cC_0$.
Then it is clear from the definition that
the functor $\cC_{0,(Y,f)/} \to (\SCH_{S,X})_{Y/}$
induced by $\iota_0$ is an equivalence of categories.
Hence $(\cC_0,\iota_0)$ satisfies (1), (2), and (4)
in Definition 5.5.1 of \cite{Grids}.

It remains to prove that $(\cC_0,\iota_0)$ satisfies (3)
in Definition 5.5.1 of \cite{Grids}.
Let $(Y,f)$ be an object of $\cC_0$ and let us choose
$H_0 \in I$ such that $f$ is the class of some morphism
$f'\colon X_{H_0} \to Y$ in $\SCH_{S,X}$.
Let us consider
the functor $\cC_{0,/(Y,f)} \to (\SCH_{S,X})_{/Y}$
induced by $\iota_0$.
Let $g\colon Y' \to Y$ be a morphism in $\SCH_{S,X}$.
Since $\SCH_{S,X}$ is semi-cofiltered, there exist
an object $Z$ of $\SCH_{S,X}$ and morphisms
$g' \colon  Z \to X_{H_0}$, $f'' \colon  Z \to Y'$ satisfying
$f' \circ g' = g \circ f''$.

Let $\phi\colon  S \to G/H_0$ be a morphism in $BG$.
Then $\phi$ is surjective. Choose $s \in \phi^{-1}(1 \cdot H_0)$.
Then the stabilizer $H'$ of $s$ in $G$ is an open subgroup of $G$
contained in $H_0$ and the map $G \to S$ that sends $g$ to $gs$
induces an isomorphism $\gamma\colon  G/H' \to S$ such that
$\phi \circ \gamma$ is equal to the canonical projection
$G/H' \to G/H_0$.
Since $\alpha$ is an equivalence of categories, this
in particular implies that there exist an 
element $H_1 \in I$ with $H_1 \subset H_0$ and
an isomorphism $\gamma_1 \colon  X_{H_1} \cong Z$ in $\SCH_{S,X}$
satisfying that the composite $g' \circ \gamma_1$
is equal to the transition morphism $X_{H_1} \to X_{H_0}$.
Let $h$ denote the class of the morphism
$f'' \circ \gamma_1$ in 
$\varinjlim_{H \in I} \Hom_{\SCH_{S,X}}(X_H,Y')$.
Then the pair $(Y',h)$ is an object of $\cC'_0$
and the morphism $g$ is a morphism from
$(Y',h)$ to $(Y,f)$ in $\cC'_0$.
This proves that
the functor $\cC_{0,/(Y,f)} \to (\SCH_{S,X})_{/Y}$
induced by $\iota_0$ is essentially surjective.
Hence $(\cC_0,\iota_0)$ satisfies (3)
in Definition 5.5.1 of \cite{Grids}.
This completes the proof.
\end{proof}

We define $M_{S,X}$ to be the absolute Galois monoid $M_\Cip$.

It follows from Proposition \ref{prop:hyp_top}
that,
as a topological monoid, $M_{S,X}$ is isomorphic to
the monoid $\End_{\Pro(\SCH_{S,X})}(\wt{X})$ of endomorphism of 
$\wt{X} =(X_{H})_{H \in I}$
in the category $\Pro(\SCH_{S,X})$ of pro-objects of $\SCH_{S,X}$.
Here we equip 
$$\Hom_{\Pro(\SCH_{S,X})}(\wt{X}, X_H) 
= \varinjlim_{H' \in I} \Hom_{\SCH_{S,X}}(X_{H'}, X_H)$$
with the discrete topology for each $H \in I$, 
and 
$$\End_{\Pro(\SCH_{S,X})}(\wt{X}) = \varprojlim_{H \in I}
\Hom_{\Pro(\SCH_{S,X})}(\wt{X},X_H)$$
with the limit topology.

\begin{lem} \label{lem:hyp_pi_1}
Let $(Y,f)$ be an object of $\cC_0$.
Then the open subgroup $\bK_{(Y,f)}$ of $M_{S,X}$ is
isomorphic to the \'etale fundamental group of $Y$.
\end{lem}

\begin{proof}
Let us consider the functor $\iota_0^{(Y,f)}\colon 
(\cC_0)_{/(Y,f)} \to (\SCH_{S,X})_{/Y}$ induced by $\iota_0$.
Then $(\SCH_{S,X})_{/Y}$ equipped with the atomic topology $J_Y$
is a $Y$-site, the pair $((\cC_0)_{/(Y,f)},\iota_0^{(Y,f)})$
is a grid of $((\SCH_{S,X})_{/Y},J_Y)$, and
$\bK_{(Y,f)}$ can be identified with the absolute Galois monoid
$M_{((\cC_0)_{/(Y,f)},\iota_0^{(Y,f)})}$.
Note that the Galois category $\Et(Y)$ 
is canonically equivalent to the category $\wt{\cFC}$ 
in Section \ref{sec:wtcFC} for $\cC =(\SCH_{S,X})_{/Y}$.
It is easy to see from the construction of
$\cC_0$ that the fiber functor
$\omega_{((\cC_0)_{/(Y,f)},\iota_0^{(Y,f)})}$
restricted to $\wt{\cFC}$ gives a fiber functor
of the Galois category $\Et(Y)$.
From this we can see that 
$\bK_{(Y,f)} \cong M_{((\cC_0)_{/(Y,f)},\iota_0^{(Y,f)})}$
is isomorphic to the \'etale fundamental group of $Y$.
\end{proof}

\begin{cor} \label{cor:hyp_M}
The absolute Galois monoid $M_{S,X}$ is a locally profinite group
and is independent, up to isomorphisms, of the choice of
a grid $(\cC_0,\iota_0)$.
Moreover, the fiber functor $\omega_\Cip$ gives an equivalence
of categories from $\Shv(\SCH_{S,X},J)$ to the 
category of smooth left $M_{S,X}$-sets.
\end{cor}

\begin{proof}
It follows from Proposition \ref{prop:hyp_Y}
that $(\SCH_{S,X},J)$ is a $Y$-site 
such that for any morphism $f\colon Y \to Z$ in $\SCH_{S,X}$,
there are only finitely many automorphisms of $Y$ over $Z$.
Hence it follows from Proposition 9.1.1 of \cite{Grids} that
$M_{S,X}$ is independent, up to isomorphisms, of the choice of
a grid $(\cC_0,\iota_0)$.
Since $J$ is the atomic topology, it follows from
Lemma 8.1.7 of \cite{Grids} that $M_{S,X}$ is a group.
Hence it follows from Lemma \ref{lem:hyp_pi_1} that
$M_{S,X}$ is a locally profinite group.

The last assertion follows from 
Theorem 5.8.1 of \cite{Grids}.
\end{proof}

Since all morphisms in the category $\SCH_S$ are affine,
there exists a limit 
\begin{equation} \label{eq:univ_cov}
\wh{X} = \varprojlim_{H\in I} X_H
\end{equation}
in the category of schemes over $S$. 
For schemes $Y,Z$ over $S$, let
$\Hom_S(Y,Z)$ denote the set of morphisms from $Y$ to $Z$
in the category of schemes over $S$.

\begin{prop} \label{prop:hyp_wh}
Let us consider the monoid $\End_S(\wh{X})$ of
endomorphisms of the scheme $\wh{X}$ over $S$.
For each $H \in I$, we endow $\Hom_S(\wh{X},X_H)$
with the discrete topology and $\End_S(\wh{X})
= \varprojlim_H \Hom_S(\wh{X},X_H)$ with the limit
topology.

Then both as a monoid and as a topological space, 
$M_{S,X}$ is equal to the
submonoid of endomorphisms $\sigma \in \End_S(\wh{X})$
satisfying the following property:
There exists $H_0 \in I$ and a morphism $f\colon  X_{H_0} \to X$
in $\SCH_{S,X}$ such that the diagram
\begin{equation} \label{eq:hyp_comm1}
\begin{CD}
\wh{X} @>{\sigma}>> \wh{X} \\
@VVV @VVV \\
X_{H_0} @>{f}>> X,
\end{CD}
\end{equation}
where the vertical arrows are the canonical projections,
is commutative.
\end{prop}

\begin{proof}
It follows from Proposition \ref{prop:hyp_top} that we have
an isomorphism $M_{S,X} = \varprojlim_{H} \varinjlim_{H'}
\Hom_{\SCH_{S,X}} (X_{H'}, X_H)$.
For each $H \in I$, the set 
$\varinjlim_{H'} \Hom_{\SCH_{S,X}} (X_{H'}, X_H)$ is a subset
of $\Hom_S (\wh{X}, X_H)$.
Hence we obtain an injective map from
$M_{S,X}$ to $\End_S(\wh{X}) = \varprojlim_{H} \Hom_S(\wh{X}, X_H)$.
Let $\sigma \in \End_S(\wh{X})$ and suppose that
there exist $H_0 \in I$ and a morphism $f\colon  X_{H_0} \to X$
in $\SCH_{S,X}$ such that the diagram
\eqref{eq:hyp_comm1} is commutative.
For each $H \in I$, let $\sigma_H$ denote the
composite of $\sigma$ with the projection $\wh{X} \to X_H$.

Let $I'$ denote the set of $H' \in I$ with $H' \subset H_0$.
For each $H' \in I'$, let $X'_{H'}$ denote $X_{H'}$ regarded
as a scheme over $X$ via the composite of $X_{H'} \to X_{H_0}$ with $f$.
Then for any $H \in I$, the morphism $\sigma_H$ can be regarded as
a morphism from $\wh{X}' = \varprojlim_{H' \in I'} X'_{H'}$ 
to $X_H$ in the category of schemes over $X$.
Since $X_H$ is finite \'etale over $X$ and $X$ is locally noetherian, 
the morphism $X_H \to X$ is locally of finite presentation.
Hence there exists an element $H' \in I'$ such that
$\sigma_H$ factors through the projection $\wh{X}' \to X'_{H'}$
and that the resulting morphism $f_H \colon  X'_{H'} \to X_H$ makes the diagram
$$
\begin{CD}
X_{H'} @>{f_H}>> X_H \\
@VVV @VVV \\
X_H @>{f}>> X
\end{CD}
$$
commutative, where the the vertical arrows are the canonical projections.
Since the vertical arrows and $f$ are finite \'etale, it follows from
Corollaire 4.8 of \cite[Expos\'e I]{SGA1} that $f_H$ is finite \'etale.
This shows that $\sigma_H \in \Hom_S(\wh{X},X_H)$ comes from
an element of $\varinjlim_{H'} \Hom_{\SCH_{S,X}}(X_{H'},X_H)$.
Hence $\sigma$ comes from an element of $M_{S,X}$.

The claim on the topology $M_{S,X}$ follows immediately from the
definition of the topology on $M_{S,X}$.
This completes the proof.
\end{proof}

\begin{lem} \label{lem:hyp_core}
Let $S$ be a scheme and let $X$ be a connected, locally noetherian scheme
over $S$. Then the locally profinite group $M_{S,X}$ is profinite
if and only if there exists a morphism $Y \to X$ and
a finite subgroup $K$ of
the group $\Aut_S(Y)$ of automorphisms of the scheme $Y$ over $S$ 
such that $M_{S,X}$ is isomorphic to the \'etale fundamental
group of the Deligne-Mumford stack $K \backslash Y$.
\end{lem}

\begin{proof}
By Lemma \ref{lem:hyp_pi_1}, we have an open subgroup
$\bK_X$ of $M_{S,X}$ isomorphic to the \'etale
fundamental group $G$ of $X$. Let us choose an open
subgroup $H \subset G$ such that $H$ is an open normal
subgroup of $M_{S,X}$.

By using the theory of fpqc descent, one can prove that
the pair of $X_H$ and the canonical projection 
$\wt{X} \to X_H$ is a quotient of
$\wt{X}$ by $H$ in the category $\Pro(\SCH_{S,X})$.
Hence it follows from the construction of
$\bK_X \inj M_{S,X}$ that for any $\sigma \in M_{S,X}$,
there exists a unique endomorphism $f_\sigma\colon  X_H \to X_H$ in $\SCH_{S,X}$ 
such that the diagram
$$
\begin{CD}
\wh{X} @>{\Phi(\sigma)}>> \wh{X} \\
@VVV @VVV \\
X_H @>{f_\sigma}>> X_H
\end{CD}
$$
is commutative. By the uniqueness we have
$f_{\sigma \tau} = f_\sigma \circ f_\tau$
for any $\sigma,\tau \in M_{S,X}$ and
we have $f_\sigma = \id_{X_H}$ if and only if $\sigma \in H$.
By sending $\sigma$ to $f_\sigma$ we obtain an
injective homomorphism $M_{S,X}/H \to \Aut_S(X_H)$
of groups.
Set $Y = X_H$ and let $K$ denote the image of
the homomorphism $G \cong M_{(\cC_0,\iota_0)}$.
Then the pair $(Y,K)$ satisfies the desired property.
\end{proof}

\subsection{Functoriality of $M_{S,X}$}

\subsubsection{Functoriality of $M_{S,X}$ with respect to $X$}
Let $S$ be a scheme and let $X$ and $Y$ be objects of $\SCH_S$.
Then we have locally profinite groups $M_{S,X}$ and $M_{S,Y}$
that contain the \'etale fundamental groups $\pi_1(X)$ and $\pi_1(Y)$,
respectively, as open compact subgroups.
Suppose that a morphism $f \colon  X \to Y$ of schemes over $S$ is given.
Then the covariant functoriality of the \'etale fundamental groups 
give a homomorphism $\pi_1(f) \colon  \pi_1(X) \to \pi_1(Y)$ that is canonical
up to conjugation by the elements of $\pi_1(Y)$.
If $f$ is a morphism in $\SCH_S$, then 
the homomorphism $\pi_1(f)$ extends
to an isomorphism $M_{S,X} \xto{\cong} M_{S,Y}$ 
of locally profinite groups.
However, for general $f$, 
the homomorphism $\pi_1(f)$ does not always extend
to a homomorphism $M_{S,X} \to M_{S,Y}$ 
of locally profinite groups.
So $M_{S,X}$ does not have a good functoriality with respect to $X$.

\subsubsection{Functoriality of $M_{S,X}$ with respect to $S$}
Let $g\colon  T \to S$ be a morphism of schemes
and let $X$ be an object of $\SCH_T$. Then it follows from
Proposition \ref{prop:hyp_wh} that we have an injective,
continuous homomorphism 
$g_* \colon  M_{T,X} \to M_{S,X}$ of topological monoids
whose restriction to $\pi_1(X)$ gives the identity morphism
on $\pi_1(X)$.

Below we study the image of the homomorphism 
$g_* \colon  M_{T,X} \to M_{S,X}$ when both $S$ and $T$ are
spectra of fields.

\begin{lem} \label{lem:hyp_AM}
Let $S$ be a locally noetherian scheme and
$X$ a connected scheme which is of finite type over $S$.
Then $X$ is an object of $\SCH_{S}$ and any object
of $\SCH_{S,X}$ is of finite type over $S$.
\end{lem}

\begin{proof}
It follows from Proposition 6.3.7 of \cite{EGAI} that
$X$ is locally noetherian.
Hence $X$ is an object of $\SCH_{S}$.
To prove the second assertion, we may assume that
$S$ is the spectrum of a noetherian ring $A$.
Let $Y$ be an object of $\SCH_{S,X}$.
Then there exists a scheme $Z$ over $S$
and finite \'etale morphisms $f\colon Z \to Y$ and
$Z \to X$ over $S$.
Let us choose an affine open subscheme $U \subset Y$
and set $V = f^{-1}(U)$. Since $f$ is finite,
$V$ is an affine open subscheme of $Z$.
Since $Z$ is of finite type over $S$,
it follows from Lemma 6.3.2.1 of \cite{EGAI} that
for any $x \in V$, there exists an affine open
neighborhood $V'$ of $x$ in $V$ such that $V'$
is of finite type over $S$.
Since $V$ is quasi-compact, it follows that
$V$ is of finite type over $S$.
Set $V=\Spec C$. Then it follows from Proposition 6.3.3
of \cite{EGAI} that $C$ is a finitely generated $A$-algebra.
Thus the claim follows from Proposition 7.8 of \cite{AtiMac}.
\end{proof}

\begin{lem} \label{lem:hyp_ring1}
Let $k$ be a field, 
$X$ a scheme of finite type over $\Spec k$,
and $a$ an element of $\Gamma(X,\cO_X)$.
Suppose that there exists a subfield of 
the ring $\Gamma(X,\cO_X)$ that contains $a$.
Then $a$ is algebraic over $k$.
\end{lem}

\begin{proof}
Suppose that $a$ is transcendental over $k$.
Let $k_0$ denote the prime subfield of $k$ and
let $B$ (\resp $B_0$) denote the $k$-subalgebra
(\resp $k_0$-subalgebra) of $\Gamma(X,\cO_X)$ 
generated by $a$.
Let us consider the morphism
$f\colon  X \to \Spec B$.
Since $a$ is transcendental over $k$,
it follows that $\Spec B$ and $\Spec B_0$ are 
non-empty and the image of $f$ is Zariski dense.
Since $X$ is of finitely type over $k$,
it follows from Chevalley's theorem that
the image of $f$ is a constructible subset
of $\Spec B$.
Hence there exists a finite set $S$ of
closed points of $\Spec B$ whose complement
is equal to the image of $f$.

Note that the morphism $g\colon  \Spec B \to \Spec B_0$
is surjective. This implies that the image of the composite
$g \circ f$ contains the complement of $g(S)$.
In particular the image of $g \circ f$ contains a
closed point of $\Spec B_0$.
It follows from our assumption that any non-zero element of
$B_0$ is invertible in $\Gamma(X,\cO_X)$. 
Hence the composite $g \circ f$
factors through the morphism $\Spec \Frac B_0 \to \Spec B_0$.
This implies that the image of the composite
$g \circ f$ is a singleton that consists of 
the generic point of $\Spec B_0$, which gives a contradiction.
\end{proof}

\begin{cor} \label{cor:hyp_field}
Let $k$ be a field and 
$X$ an integral scheme 
which is of finite type over $\Spec k$.
Let $L$ denote the set of elements of $\Gamma(X,\cO_X)$
that are algebraic over $k$.
Then $L$ is a finite extension of $k$ and
any subfield of $\Gamma(X,\cO_X)$ is contained in $L$.
\end{cor}

\begin{proof}
Since $X$ is an integral scheme, it follows that
$L$ is a subfield of the integral domain $\Gamma(X,\cO_X)$.
Let us choose an affine open $U = \Spec A$ of $X$
so that $A$ is a finitely generated $k$-algebra.
Then $L/k$ is a subextension of the finitely generated
field extension $\Frac\, A/k$. Hence it follows 
from Theorem 3.1.4 of \cite{Nagata} that 
$L/k$ is a finitely generated field extension.

Since $L/k$ is algebraic, 
$L$ is a finite extension of $k$.
It follows from Lemma \ref{lem:hyp_ring1} that
any subfield of $\Gamma(X,\cO_X)$ is contained in $L$.
Hence the claim follows.
\end{proof}

\begin{cor} \label{cor:hyp_LX}
Let $k$ be a field, and let $X$ and $Y$ be 
integral schemes
that are of finite type over $\Spec k$.
For $S \in \{X,Y\}$, let 
$L(S)$ denote the set of elements of $\Gamma(S,\cO_S)$
that are algebraic over $k$.
Then $L(X)$ and $L(Y)$ are finite extensions of $k$ and 
for any morphism $f\colon X \to Y$,
there exists a unique morphism $g\colon  \Spec L(X) \to \Spec L(Y)$
that makes the diagram
$$
\begin{CD}
X @>{f}>> Y \\
@VVV @VVV \\
\Spec L(X) @>{g}>> \Spec L(Y)
\end{CD}
$$
commutative.
\end{cor}

\begin{proof}
It follows from 
Corollary \ref{cor:hyp_field} that $L(X)$ and
$L(Y)$ are finite extensions of $k$.
Since $L(Y)$ is a field and
$\Gamma(X,\cO_X)$ is not equal to $\{0\}$,
the composite
$\beta \colon  L(Y) \inj \Gamma(Y,\cO_Y) \to \Gamma(X,\cO_X)$
is injective.
Hence it follows from Corollary \ref{cor:hyp_field} that
the image of $\beta$ is contained in $L(X)$.
This completes the proof.
\end{proof}

Let $K$ be a perfect field and $k$ a subfield of $K$.
Set $S = \Spec k$ and $T = \Spec K$.
Let $X$ be an integral scheme which is of finite type over $T$.
Let $I$ denote the filtered poset of open subgroups of $\pi_1(X)$ and
let us consider the scheme $\wh{X} = \varprojlim_{H \in I} X_H$ as
in \eqref{eq:univ_cov}.
For $H \in I$, let $L_H$ denote 
the set of elements of $\Gamma(X_H,\cO_{X_H})$
that are algebraic over $K$.
Then each $L_H$ is a finite separable extension
of $K$ and $\wh{L} = \varinjlim_{H \in I} L_H$ is
a separable closure of $K$.
\begin{prop}
Let $K$, $k$, $S$, $T$, $X$, and $\wh{L}$ 
be as above. Then there exists a canonical homomorphism
$\phi \colon  M_{S,X} \to \Aut(\wh{L}/k)$ of groups such that
the image of the homomorphism $M_{T,X} \to M_{S,X}$
is equal to $\phi^{-1}(\Gal(\wh{L}/K))$.
The homomorphism $\phi$ is surjective if at least one
of the following conditions is satisfied.
\begin{enumerate}
\item There exists a scheme $X_0$ over $S$ such that
$X$ is isomorphic to $X_0 \times_S T$ as a scheme over $T$.
\item $X$ is geometrically connected over $L_{\pi_1(X)}$
and for any $\sigma \in \Aut(\wh{L}/k)$, there exist
$H \in I$ and a morphism $\phi_\sigma \colon  X_H \to X$ of schemes over $S$ 
such that $\sigma(L_{\pi_1(X)}) \subset L_H$ and 
the diagram
$$
\begin{CD}
X_H @>{\phi_{\sigma}}>> X \\
@VVV @VVV \\
\Spec \sigma(L_{\pi_1(X)})
@>{\sigma^*}>> \Spec L_{\pi_1(X)}
\end{CD}
$$
is commutative.
\end{enumerate}
\end{prop}

\begin{proof}
Using Proposition \ref{prop:hyp_wh}, we regard
$M_{S,X}$ as a submonoid of $\End_S(\wh{X})$.
It follows from Corollary \ref{cor:hyp_LX} that
any $\sigma \in M_{S,X}$, there exists 
unique $\tau_\sigma \in \End_k(\wh{L})$
that makes the diagram
$$
\begin{CD}
\wh{X} @>{\sigma}>> \wh{X} \\
@VVV @VVV \\
\Spec \wh{L}
@>{\tau_\sigma^*}>> \Spec \wh{L}
\end{CD}
$$
is commutative.
Since $M_{S,X}$ is a group, we have
$\tau_\sigma \in \Aut_k(\wh{L})$.
By sending $\sigma$ to $\tau_\sigma^{-1}$,
we obtain a homomorphism
$\phi \colon  M_{S,X} \to \Aut(\wh{L}/k)$ that has
the desired property.
The claim on the surjectivity of $\phi$ is 
a consequence of Corollary \ref{cor:hyp_pro_surj}.
\end{proof}

\subsection{Examples} \label{sec:hyp_counter}

\subsubsection{ }
Let $S$ be a scheme and let $X$ be a connected,
locally noetherian scheme over $S$.
Suppose that the \'etale fundamental group of $X$ is trivial.
Then it follows easily from Proposition \ref{prop:hyp_wh}
that we have $M_{S,X} = \Aut_S(X)$
with discrete topology.

For example, let us take  $S= \Spec\, \Q$ and $X = \Spec\, \C$.
Then it follows easily from Proposition \ref{prop:hyp_wh}
that we have $M_{\Spec\, \Q, \Spec\, \C} = \Aut(\C/\Q)$
with discrete topology.

\subsubsection{ }
Let $S$ be a scheme and let $X$ be a connected,
locally noetherian scheme over $S$.

\begin{lem} \label{lem:hyp_non}
Suppose that there exists a finite \'etale endomorphism
$f\colon X \to X$ over $S$ of degree $n \ge 2$. Then
$M_{S,X}$ is not profinite.
\end{lem}

\begin{proof}
Lemma \ref{lem:hyp_pi_1} gives an open subgroup $H$ of $M_{S,X}$
isomorphic to the \'etale fundamental group $\pi_1(X)$ of $X$.
Then any integer $m \ge 1$, the $m$-th iteration of $f$
gives an open subgroup $H^{(m)}$ of $M_{S,X}$
such that $H^{(m)}$ contains $H$ as a subgroup of index $n^m$.
This in particular shows that $M_{S,X}$ is not profinite.
\end{proof}

\subsubsection{ }
Let $k$ be field and $A$ a semi-abelian variety over $k$.
Then for a positive integer $n \ge 2$ not divisible by 
the characteristic of $k$, the endomorphism $A \to A$
given by the multiplication by $n$ is finite \'etale.
Hence it follows from Lemma \ref{lem:hyp_non} that
if $A$ is not zero-dimensional, then 
$M_{\Spec\, k,A}$ is not profinite.

Suppose moreover that $k$ is an algebraically
closed field of characteristic zero.
Then the \'etale fundamental group $\pi_1(A)$ of $A$
is a free $\wh{\Z}$-module of finite rank.
In this case, 
it follows easily from Proposition \ref{prop:hyp_wh}
that we have 
$$
M_{\Spec\, k,A} = (\pi_1(A) \otimes_{\Z} \Q) \rtimes
(\End_k(A) \otimes_{\Z} \Q)^\times.
$$
Here $\End_k(A)$ denotes the ring of endomorphisms of
the commutative groups schemes $A$ over $\Spec\, k$,
and we equip $\pi_1(A) \otimes_{\Z} \Q$ 
(\resp $(\End_k(A) \otimes_{\Z} \Q)^\times$)
with the adelic (\resp the discrete) topology.

\subsubsection{ }
Let us consider the case when $S$ is the spectrum of subfield $k$
of the field $\C$ of complex numbers and
$X$ is a smooth, geometrically irreducible hyperbolic 
curve over $k$.
Then it follows from Lemma \ref{lem:hyp_core} 
that the group $M_{\Spec\, k, X}$ is profinite 
if and only if $X$ is not $k$-arithmetic in the sense 
of \cite[Remark 2.1.1]{Mochizuki}.

Let $(G,\frX)$ be a Shimura datum such that
a connected component of $\frX$ is isomorphic to 
the complex upper half plane.
Suppose that $k$ contains the reflex field of $(G,\frX)$
and that $X$ is a connected component a Shimura curve
for $(G,\frX)$ with respect to some sufficiently small 
compact open subgroup of $G(\A^\infty_\Q)$.
In this case $X$ is not $k$-arithmetic in the sense 
of \cite[Remark 2.1.1]{Mochizuki}.
Hence $M_{\Spec\, k, X}$ is not a profinite group.

\subsubsection{ }
Let $p$ be a prime number and let $k = \Fbar_p$ 
an algebraic closure of $\F_p$.
Let us consider the case when
$S = \Spec \F_p$ and $X = \A^1_k = \Spec k[t]$.
The Artin-Schreier morphism $f\colon  X \to X$ over $\Spec k$,
given by the homomorphism $k[t] \to k[t]$ of $k$-algebras
that sends $t$ to $t^p-t$, is finite \'etale of degree $p$.
Hence Lemma \ref{lem:hyp_non} shows that $M_{S, X}$
is not profinite.

It follows from Theorem 0.7 of \cite{Tama} that
any connected affine smooth curve $Y$ over $k$
is $\Lambda$-related to $X$ in the category $\SCH_S$.
In particular $Y$ is $\Lambda$-related to $X$ in $\SCH_S$.
Hence we have $M_{S,Y} \cong M_{S,X}$
for any connected affine smooth curve $Y$ over $k$.
This seems to suggest that it is important to study the
locally profinite group $M_{S, X}$.

In Conjecture 1.13 (ii) of \cite{Tama0}, 
Tamagawa made the following conjecture:
for any two algebraically closed fields $l$, $l'$ of characteristic
$p$ and for any
connected smooth affine curves $Y$ and $Y'$ over $l$ and $l'$, respectively,
the natural map $\Isom_{\F_p}(Y,Y') \to \Isom(\pi_1(Y),\pi_1(Y'))
/\Inn(\pi_1(Y'))$ is an isomorphism.

We would like to propose the following variant of
Tamagawa's conjecture when $l=l'=\Fbar_p$.
Our conjecture can be described 
purely in terms of the locally profinite group $M_{S,X}$.

\begin{conj} \label{conj:tama}
Let $S =\Spec\, \F_p$ and $X = \A^1_k$ be as above.
Let $\bK, \bK'$ be two open compact subgroups of $M_{S,X}$.
Suppose that an isomorphism $\gamma \colon  \bK \xto{\cong} \bK'$
of profinite groups are given.
Then there exists an element $g \in M_{S,X}$ such that
$\bK' = g \bK g^{-1}$ and that $\gamma$ is equal to
the isomorphism $\bK \to \bK'$ given by the conjugation by $g$.
\end{conj}

Let us fix a grid $(\cC_0,\iota_0)$ of the site 
$\SCH_{S,X}$ with the atomic topology.
Let $Y$ and $Y'$ be 
connected smooth affine curves over $k$.
Then it follows from Theorem 0.7 of \cite{Tama}
that $Y$ and $Y'$ are objects of $\SCH_{S,X}$.
Let us choose objects $Y_0$, $Y'_0$ of $\cC_0$
such that there exist isomorphisms
$\iota_0(Y_0) \cong Y$ and $\iota_0(Y'_0) \cong Y'$
in $\SCH_{S,X}$.
These isomorphisms induce isomorphisms
$\pi_1(Y) \cong \bK_{Y_0}$ and
$\pi_1(Y') \cong \bK_{Y'_0}$ that
are uniquely determined up to inner automorphisms.

Let $Y$, $Y'$, $Y_0$, $Y'_0$ as above.
We claim that Tamagawa's conjecture for $Y$ and $Y'$
is equivalent to Conjecture \ref{conj:tama} for
$\bK_{Y_0}$ and $\bK_{Y'_0}$.
Suppose that Tamagawa's conjecture is true for $Y$ and $Y'$.
Let $\gamma \colon  \bK_{Y_0} \xto{\cong} \bK_{Y'_0}$ be an
isomorphism of topological groups.
Let us consider the composite
$\gamma' \colon  \pi_1(Y) \cong \bK_{Y_0}
\xto{\gamma} \bK_{Y'_0} \cong \pi_1(Y')$.
By assumption, the isomorphism $\gamma'$
comes, up to inner automorphisms, from an
isomorphism $\beta' \colon  Y \xto{\cong} Y'$ of schemes
over $S$.
Let $\beta$ denote the composite
$\iota_0(Y_0) \cong Y \xto{\beta'} Y' \cong \iota_0(Y'_0)$.
It follows from Corollary \ref{cor:hyp_pro_surj} that
there exists an element $g = (\alpha,\gamma_\alpha) \in M_{S,X}$ 
such that $\alpha(Y_0) = Y'_0$ and 
$\gamma_\alpha(Y_0) = \beta$.
It is then straightforward to check that we have
$\bK_{Y'_0} = g \bK_{Y_0} g^{-1}$ and that $\gamma$ is equal to
the isomorphism $\bK_{Y_0} \to \bK_{Y'_0}$ given by the conjugation by $g$.
Hence Tamagawa's conjecture for $Y$ and $Y'$
implies Conjecture \ref{conj:tama} for
$\bK_{Y_0}$ and $\bK_{Y'_0}$.
Conversely, suppose that Conjecture \ref{conj:tama} is true for
$\bK_{Y_0}$ and $\bK_{Y'_0}$.
Suppose that an isomorphism
$\gamma' \colon  \pi_1(Y) \xto{\cong} \pi_1(Y')$ 
of topological groups is given.
Let us consider the composite
$\gamma \colon  \bK_{Y_0} \cong \pi_1(Y) 
\xto{\gamma'} \pi_1(Y') \cong \bK_{Y'_0}$.
By assumption, there exists an element 
$g \in M_{S,X}$ 
such that $\bK_{Y'_0} = g \bK_{Y_0} g^{-1}$ and that 
$\gamma$ is equal to
the isomorphism $\bK_{Y_0} \to \bK_{Y'_0}$ 
given by the conjugation by $g$.
Let $I$ denote the order dual of $\cC_0$.
Then $\iota_0(Y_0)$ and $\iota_0(Y'_0)$ are quotient objects
of the limit 
$\wh{X} = \varprojlim_{Z \in I} \iota_0(Z)$
in the category of schemes over $S$ with respect to
the actions of the groups $\bK_{Y_0}$ and $\bK_{Y'_0}$, respectively.
By Proposition \ref{prop:hyp_wh} we regard $g$ as an automorphism
of $\wh{X}$. Then $g$ induces automorphism
$\beta \colon  \iota_0(Y_0) \xto{\cong} \iota_0(Y'_0)$.
Let $\beta'$ denote the composite
$Y \cong \iota_0(Y_0) \xto{\beta} \iota_0(Y'_0) \cong Y'$.
It is then straightforward to check that we have
the isomorphism $\gamma'$ is equal, up to conjugation by
elements of $\pi_1(Y')$, to the isomorphism
$\pi_1(Y) \xto{\cong} \pi_1(Y')$ given by $\beta'$.
Hence Conjecture \ref{conj:tama} for
$\bK_{Y_0}$ and $\bK_{Y'_0}$ implies 
Tamagawa's conjecture for $Y$ and $Y'$.

\section{$Y$-sites for Riemannian symmetric spaces}
\label{sec:Riemannian}
A Riemannian manifold is called
a Riemannian locally symmetric space 
if its curvature tensor is invariant under 
the parallel transforms. 
It follows from Theorem 1.3 of \cite[Ch.\ IV]{Helgason}
that this definition is equivalent to that given in
\cite[p.\ 200]{Helgason}.

We say that a Riemannian manifold $X$ is geodesically
connected if for any two points $p,q \in X$,
there exists a finite geodesic segment from $p$ to $q$.
Let $X$ be a geodesically connected Riemannian manifold
and let $p \in X$. We say that $X$ is bounded 
if there exists a real number $r$ such that 
for any $q \in X$, the length of a geodesic from
$p$ to $q$ is at most $r$. It is easy to see that
this notion is independent of the choice of a point
$p \in X$.
Let us consider the following category $\LS$.
The objects of $\LS$ are the Riemannian 
locally symmetric spaces of non-positive curvature which are
geodesically connected, simply connected, and bounded. For two objects $U$, $V$
of $\LS$, the morphisms from $U$ to $V$ are the injective 
local isometries from $V$ to $U$. 
(Here local isometries are assumed to be open mappings.) 
Be careful of that the morphisms from $U$ to $V$
are not maps from $U$ to $V$, but maps from $V$ to $U$.

\begin{lem} \label{lem:RSS_0}
Let $U$ and $V$ be Riemannian locally symmetric space
and $\varphi_1, \varphi_2 \colon  U \to V$ local isometries.
Suppose that $U$ is path-connected and $\varphi_1$ and
$\varphi_2$ coincide on a non-empty open subset of $U$.
Then we have $\varphi_1 = \varphi_2$.
\end{lem}

\begin{proof}
It follows from Proposition 5.5 of \cite[Ch.\ IV]{Helgason}
that any local isometry $U \to V$ is real analytic.
Hence it follows from Lemma 4.3 \cite[Ch.\ VI]{Helgason}
that we have $\varphi_1 = \varphi_2$.
\end{proof}

\begin{lem} \label{lem:RSS_1}
Let $U$ be an object of $\LS$ and let $p \in U$.
Let $X$ be a complete, simply connected Riemannian globally symmetric space.
Suppose that there exists an open neighborhood $V$ of $p$ in $U$
and a local isometry $\varphi\colon  V \to X$.
Then $\varphi$ can uniquely be extended to a local isometry
$\wt{\varphi} \colon  U \to X$. Moreover $\wt{\varphi}$ is injective.
\end{lem}

\begin{proof}
Let $q \in U$. Then there exists a continuous curve
$\gamma \colon  [0,1] \to U$ with $\gamma(0)=p$ and $\gamma(1)=q$.
Then it follows from Proposition 11.3 of \cite[Ch.\ I]{Helgason}
that $\varphi$ is extendable along $\gamma$.
Precisely speaking, in the statement of Proposition 11.3 
of \cite[Ch.\ I]{Helgason}, it is assumed that
both the domain and the codomain of $\wt{\varphi}$ are complete.
However, the completeness of the domain is not used in the
proof.
Since $U$ is simply connected, it follows from
Proposition 11.4 of \cite[Ch.\ I]{Helgason}
that the value at $q$ of the extended map 
does not depend on the choice of $\gamma$.
Hence we obtain a local isometry $\wt{\varphi}\colon U \to X$
that extends $\varphi$.
The uniqueness of $\wt{\varphi}$ follows 
from Lemma \ref{lem:RSS_0}.

It remains to prove that $\wt{\varphi}$ is injective. 
Suppose that there exist $q,q' \in U$ satisfying $q \neq q'$
and $\wt{\varphi}(q) = \wt{\varphi}(q')$.
Let us choose a finite geodesic segment
$\delta \colon  [0,1] \to U$ satisfying $\delta(0)=q$ and
$\delta(\ell)=q'$. Then $\wt{\varphi} \circ \delta$ 
is a finite geodesic segment from $\wt{\varphi}(q)$
to itself. Since $X$ is of non-positive curvature,
it follows that $\wt{\varphi}\circ \delta$ 
is a constant map. Hence $\delta$ is a constant map,
which gives a contradiction.
\end{proof}

\begin{lem}
The category $\LS$ is semi-cofiltered.
\end{lem}

\begin{proof}
Let $U_1 \xto{f'_1} V \xleftarrow{f'_2} U_2$ be a diagram in $\LS$.
This diagram gives injective local isometries $f_1\colon  V \to U_1$ and
$f_2 \colon V \to U_2$.
Let us fix a point $p \in V$.
Then there exist an open neighborhood $W$ of $p$ in $V$
and a simply connected complete Riemannian globally symmetric space
$X$ such that there exists a local isometry $g\colon  W \to X$.
Since $U_1$ and $U_2$ are objects of $\LS$, it follows
from Lemma \ref{lem:RSS_1} that there exists
a unique local isometry $g_i \colon  U_i \to X$ such that
the restriction of $g_i \circ f_i$ to $W$ is equal to $g$.
By the uniqueness we have $g_1 \circ f_1 = g_2 \circ f_2$.
Choose a sufficiently large $r$ such that the open ball
$B(g(p),r)$ in $X$ contains the images of $g_1$ and $g_2$.
Then $g_i$ gives a morphism $g'_i \colon  B(g(p),r) \to U_i$
and we have $f'_1 \circ g'_1 = f'_2 \circ g'_2$.
This completes the proof.
\end{proof}

In particular, the $\Lambda$-relationship among the objects of $\LS$
gives an equivalence relation on the collection of objects of $\LS$.
Now let us fix an object $U_0$ of $\LS$ and let $\LS(U_0)$ denote the
full-subcategory of $\LS$ whose objects are the objects of $\LS$ which
are $\Lambda$-related to $U_0$.
Let $\cC$ be an arbitrary full subcategory of $\LS(U_0)$ satisfying 
the following cofinality condition: For any object $U$ of $\LS(U_0)$, 
there exists an object $V$ of $\cC$ and a morphism from $V$ to $U$ 
in $\LS(U_0)$.
\begin{prop}
$\cC$ is semi-cofiltered and
the pair $(\cC,J)$ of $\cC$ and the atomic topology $J$ on $\cC$ is 
a $Y$-site.
\end{prop}

\begin{proof}
It is clear that $\LS$ is an $E$-category.
Hence $\cC$ is an $E$-category.

Let us prove that $\cC$ is semi-cofiltered.
Let $U_1 \xto{f_1} V \xleftarrow{f_2} U_2$ be a diagram
in $\cC$.
As we mentioned in Section \ref{sec:hyp_prelim}, the category
$\LS(U_0)$ is semi-cofiltered, there exist
an object $T$ of $\LS(U_0)$ and morphisms
$g_1 \colon  T \to U_1$, $g_2\colon T \to U_2$ satisfying
$f_1 \circ g_1 = f_2 \circ g_2$.
Let us choose an object $T'$ of $\cC$ and a
morphism $h\colon  T' \to T$ in $\LS(U_0)$.
Then we have $f_1 \circ g_1 \circ h = f_2 \circ g_2 \circ h$.
This proves that $\cC$ is semi-cofiltered.
In a similar manner, one can prove that the
category $\cC$ is $\Lambda$-connected.

It follows from Lemma \ref{lem:RSS_0}
that any morphism in the category $\LS$ is a monomorphism.
This implies that
any morphism in $\cC$ is a Galois covering and its
Galois group is the trivial group that 
consists of the identity morphism.
This in particular shows the class of morphisms in $\cC$
has enough Galois coverings.
This complete the proof.
\end{proof}

\begin{rmk}
Let $X$ be a simply connected complete Riemannian globally symmetric space 
and let $p \in X$
be a point. Fix a subset $S \subset \R_{>0}$ with $\sup S = \infty$.
Then the full subcategory $\cC$ of $\LS$ whose set of objects is 
$\{ B(p,r)\ |\ r \in S \}$ gives an example of such a category.
Here $B(p,r)$ denotes the geodesic open ball of radius $r$ 
with center $p$, and we regard $B(p,r)$ just as a Riemannian manifold, 
by forgetting the inclusion $B(p,r) \inj X$.
In this case, we can take any object of $\cC$ as $U_0$.
\end{rmk}

Now let us construct a grid for $\cC$.
Let $\frg$ denote the Lie algebra over $\R$ of Killing vector fields on $U_0$. 
Let us choose a point $p \in U_0$. Then the geodesic symmetry with respect to $p$ on
a neighborhood of $p$ induces an involution $s$ of $\frg$, and the pair
$(\frg,s)$ is an effective orthogonal Lie algebra in the sense of
p.213 of Helgason's book \cite{Helgason}. 
Let $X$ be a simply connected complete Riemannian globally symmetric space 
associated with $(\frg,s)$.
Let $\cC_0$ denote the following category. The objects are the open subsets of $X$
which is isometric to an object of $\cC$. For two objects $U$, $V$ of $\cC_0$, there
is at most one morphism from $U$ to $V$ in $\cC_0$, and the morphism exists if and only
if $V \subset U$ as subsets of $X$.
The tautological assignment $U \mapsto U$ gives a functor from $\cC_0$ to $\cC$
which we denote by $\iota_0$. Then the pair $(\cC_0,\iota_0)$ is a grid for the
$Y$-site $(\cC,J)$.
Let $G = \mathrm{Isom}(X,X)$ denote the group of isometries from $X$ to itself.
We then have an isomorphism $G \cong M_{(\cC_0,\iota_0)}$ of groups.
We note that $G$ is a Lie group whose Lie algebra is isomorphic to $\frg$.

\begin{rmk}
In this settings, topology on $M_{(\cC_0,\iota_0)}$
is discrete. Hence the isomorphism
$G \cong M_{(\cC_0,\iota_0)}$ is not an isomorphism
of topological groups.
\end{rmk}

\chapter{Applications of the norm relation theorem}
\label{cha:applications}
In this chapter, we give applications of the main abstract 
norm relation theorem
 (Theorem~\ref{main theorem}).
The main (new) result is the theorem on Drinfeld modular schemes 
(Theorem~\ref{thm:Drinfeld Euler}).

We present the details of the computations in the cyclotomic case
in Section~\ref{sec:cyclotomic Euler}.

We defined Hecke operators in a general manner 
in Section~\ref{sec:def Hecke},
and used them in the statement of
Theorem~\ref{thm:Drinfeld Euler}.
In Section~\ref{sec:Hecke double}, 
we show that those Hecke operators coincide with 
the usual ones defined using double cosets.

It happens that if the presheaf of transfers $G$ in Theorem~\ref{main theorem}
is a sheaf, then the proof of the norm relations is simplified much.
We address this issue in Section~\ref{sec:universal Euler}.

\section{Cyclotomic units}
\label{sec:cyclotomic Euler}
In this section, we present the Euler system of cyclotomic units using our main theorem.
The norm relation of cyclotomic units boils down to the equation
$\mathrm{Norm}(1-\zeta_{np})=1-\mathrm{Frob}_p \zeta_n$
where Norm is the norm map for the elements in the field extension $\Q(\zeta_{np})$ over $\Q(\zeta_n)$.
The reader will see that the following proof using our theorem is rather long, 
but this is a proof that generalizes to higher ranks.    

Of course there is nothing new.    Our aim is to work out the details and present how we do the explicit computations in our category $\cC^1$.

\subsection{Some Galois group computations}
\subsubsection{}
Set $X=\Spec \Z$ so that 
$\A_X=\wh{\Z} \otimes_\Z \Q$
is the ring of finite adeles and $\wh{\cO}_X=\wh{\Z}=\prod_{p} \Z_p$
where $p$ runs over all primes.
Let $d=1$.
We study the category $\cCo{1}$.
We will identify an object $N$ 
of $\cCo{1}$ and the abelian group
$\Gamma(X, N)$.
By definition, an object of 
$\cCo{1}$ is isomorphic to $\Z/n\Z$ for some $n \ge 1$.
(If $n=1$, the object is 0.)
\begin{lem}
Let $n,n'$ be positive integers.
Then 
\[
\begin{array}{ll}
(1) &
\Hom_{\cCo{1}}(\Z/n\Z, \Z/ n'\Z)\cong
\left\{
\begin{array}{ll}
\displaystyle\coprod_{d|(n/n'), d \ge 1} 
(\Z/n' \Z)^\times
& \text{if } n' | n,\\
\emptyset  & \text{if } n'\not|\, n. 
\end{array}
\right.
\\
(2) & \Aut_{\cCo{1}}(\Z/n\Z) \cong (\Z/n\Z)^\times
\end{array}
\]
\end{lem}
\begin{proof}
An element of 
$\Hom_{\cCo{1}}(\Z/n \Z,
\Z/n' \Z)$ is represented by a diagram
$\Z/n' \Z
\stackrel{p}{\twoheadleftarrow} 
N
\stackrel{i}{\hookrightarrow} \Z/n \Z$
for some $N \in \cCo{1}$.
If $n' \not| \, n$,
such a diagram  cannot exist,
hence the claim in this case follows.
Suppose $n'| n$.
The diagram forces $N \cong d\Z/n\Z$
for some $d|n$
and $i$ to be the canonical inclusion.
Two such diagrams (with $p_1$ and $p_2$ as the surjections) 
are equivalent if and only if 
there exists an isomorphism $\beta$ such that the diagram
$$
\begin{array}{ccccc}
\Z/n' \Z  & \stackrel{p_1}{\twoheadleftarrow} 
& d\Z/n \Z & \inj & \Z/ n \Z \\
{\Large \parallel} & & 
\beta {\Large \downarrow} \cong & & 
{\Large \parallel} \\
\Z/n' \Z & \stackrel{p_2}{\twoheadleftarrow} & 
d \Z/n \Z & \inj & \Z/ n \Z
\end{array}
$$
is commutative.
This commutativity forces $\beta$ to be the identity map.
Hence the two diagrams are equivalent if and only if 
$p_1=p_2$.  
Now the set of such surjections is isomorphic to $(\Z/n'\Z)^\times$.
The claim (1) follows.
The claim (2) follows immediately from (1).
\end{proof}

\subsubsection{}
Let $n$ and $n'$ be positive integers.
Then any morphism $f \in \Hom_{\cCo{1}}
(\Z/n\Z, \Z/ n'\Z)$
is a Galois covering by Lemma~\ref{lem:Gal}.
\begin{lem}
\label{lem:Gal d=1}
The Galois group of $f$ is canonically isomorphic to
$\Ker[(\Z/n\Z)^\times \to (\Z/n'\Z)^\times]$
where the map is the modulo $n'$ map.
\end{lem}
\begin{proof}
The Galois group is the subgroup of 
elements $\alpha$ of $\Aut_{\cCo{1}}(\Z/n\Z)$
such that 
$f \circ \alpha =f$.
Suppose 
that $\alpha$ is represented by the diagram
$\Z/n \Z
\stackrel{=}{\twoheadleftarrow} 
\Z/n \Z
\stackrel{\alpha_0}{\hookrightarrow} \Z/n \Z$
for some $\alpha_0 \in 
\Aut_{{\cO}_X}(\Z/n \Z)$.
The composite $f \circ \alpha$ is computed using the diagram
\[
\begin{array}{ccccc}
 &&&& \Z/n \Z \\
 &&&& \uparrow \alpha_0 \\
 && \Z/n \Z  &\stackrel{=}{\twoheadleftarrow} & \Z/n \Z \\
 &&\uparrow \iota & \circ & \uparrow \iota\\
 \Z/n' \Z & \stackrel{\iota}{\twoheadleftarrow} & d\Z/n\Z 
& \stackrel{=}{\twoheadleftarrow} & 
d\Z/n \Z 
 \end{array}
 \]
and one can see that it is represented by the diagram
$\Z/n' \Z
\stackrel{p}{\twoheadleftarrow} 
d\Z/n \Z
\stackrel{\alpha_0 \circ \iota}{\hookrightarrow} \Z/n \Z$.
Hence $f \circ \alpha=f$ if there exists 
$\beta \in \Aut_{\cO_X}(d\Z/n\Z)$ such that the diagram
$$
\begin{array}{ccccc}
\Z/n' \Z  & \stackrel{p}{\twoheadleftarrow} & 
d \Z/n \Z & \stackrel{\alpha_0 \circ \iota}{\inj}
& \Z/ n \Z \\
{\Large \parallel} & & 
\beta {\Large \downarrow} \cong & & 
{\Large \parallel} \\
\Z/ n' \Z & \stackrel{p}{\twoheadleftarrow} 
& d \Z/n \Z & \stackrel{\iota}{\inj} & \Z/ n \Z
\end{array}
$$
is commutative.
Note that 
the commutativity of the left square 
forces $\beta \equiv \id \modx n'$.
So if $\alpha_0 \not\equiv \id \modx n'$, 
then such a $\beta$ does not exist.
If $\alpha_0 \equiv \id \modx n'$,
then one can construct such $\beta$.
This proves the claim.
\end{proof}

\subsubsection{ } \label{sec:4_0_d=1}
The previous lemma may be deduced from 
Example~\ref{ex:automorphisms} in the following 
way.

Let us assume $d=1$.
Let $f \colon  N \to N'$ be a morphism in 
$\cCo{1}$ and let $(N_1,N_2,\alpha) \in C(N, N')$ be the corresponding triple.
It follows from Lemma~\ref{lem:Gal}
that $f$ is a Galois covering.
Let us compute its Galois group $G$.
It follows from Lemma 3.1.3 of \cite{Grids} that
the group $G$ is equal to 
$\Aut_{N'}(N)$. 
Observe that any $\cO_X$-submodule of $N$ is stable
under any $\cO_X$-linear automorphism of $N$.
This gives a canonical homomorphism
$\varphi\colon  \Aut_{\cO_X}(N) \to \Aut_{\cO_X}(N_2/N_1)$ and
the group $G$ is isomorphic to its kernel.
Let $I, J \subset \cO_X$ denote
the annihilator ideal sheaf of $N$, $N'$, respectively.
The $\cO_X$-action on $N$ and $N_2/N_1$ induce isomorphisms
$\Aut_{\cO_X}(N) \cong \Gamma(X,\cO_X/I)^\times$
and $\Aut_{\cO_X}(N_2/N_1) \cong \Gamma(X,\cO_X/J)^\times$.
The homomorphism $\varphi$
is then identified with the homomorphism
$\Gamma(X,\cO_X/I)^\times
\to \Gamma(X,\cO_X/J)^\times$. Hence 
the Galois group $G$ is isomorphic to the kernel of
the last homomorphism.

\subsection{On Cyclotomic Units}
In this subsection, we give an application of our result to the cyclotomic case.
We set $X=\Spec \Z$ and $d=1$.   
We consider Situation II in Section~\ref{sec:Situations}.
We remark here that 
to actually meet the conditions in Situation II, 
we need to define $G(-)$ to be the ring of K-groups, that is,
$\oplus_n K_n(-)$, 
but we only work with $K_1(-)$ below, for simplicity.

\subsubsection{}
We define a cyclotomic field as the coordinate ring 
of the moduli of level structures on $\Gm$.
What follows below may seem familiar if the reader 
has seen how a Hecke algebra acts as algebraic correspondences
on the moduli of elliptic curves or abelian varieties with level structures.
The treatment may seem elaborate, but this suits well with the 
formalism we have developed in earlier sections. 

Let $N \in \cCo{1}$.
Let $F_N$ denote the functor from the category of 
$\Q$-schemes to the category of sets, that 
sends a $\Q$-scheme $S$
to the set of isomorphism 
classes of monomorphisms 
$\psi\colon  N_S \to {\Gm}_{,S}$
of $S$-group schemes.
Here $N_S$ is the constant $S$-group
scheme with the group structure $N$.
The functor is representable by a scheme,
which we also denote $F_N$.  Note that 
$N \cong \Z/n\Z$ for some integer $n \ge 1$.
Then $F_N$ is isomorphic to 
$\Spec \Q(\zeta_n)$ for a primitive $n$-th
root of unity $\zeta_n$.

We construct a covariant functor $G'$ from 
 $\cFCo{1}$ to the category of schemes as follows.
For an object $\coprod_i N_i \in \cFCo{1}$ with 
each $N_i \in \cCo{1}$, we set $G'(N)=\coprod_i F_{N_i}$.
Let $f$ be a morphism in $\cCo{1}$ 
represented by the diagram
$N
\stackrel{p}{\twoheadleftarrow} 
N'' 
\stackrel{\iota}{\hookrightarrow} 
N'.
$
Let $S$ be a $\Q$-scheme and
let $\psi \in G'(N')(S)$.
As $\psi(\iota(\Ker p))$ is identified 
with $\mu_h \subset {\Gm}_{,S}$
where $h$ is some integer and $\mu_h$ is the group scheme of 
$h$-th roots of unity, we have an isomorphism
${\Gm}_{,S}/\psi(\iota(\Ker p)) 
\xto{h\text{-th} \cong} {\Gm}_{,S}$
where the isomorphism is the map induced by the 
$h$-th power map ${\Gm}_{,S} \to {\Gm}_{,S}$.
We set  
$\varphi\colon N_S \xto{\psi \circ \iota \circ p^{-1}}
{\Gm}_{,S}/\psi(\iota(\Ker p)) \cong {\Gm}_{,S}$.
Then $\varphi$ belongs to $G'(N)(S)$
 and this defines a morphism 
$G'(N') \xto{G'(f)} G'(N)$ 
of schemes.

We need to know that if $f$ is a Galois covering,
then $G'(f)$ is also a Galois covering (or a Galois extension of fields).
We showed in Lemma~\ref{lem:Gal} 
that any morphism $f\colon  \Z/n\Z \to 
\Z/n'\Z$ in $\cFCo{d}$ is a Galois covering.
(Of course, the extension of fields $\Q(\zeta_n) \supset \Q(\zeta_{n'})$
is a Galois extension.)
The Galois group for the covering $f$ is canonically 
isomorphic to $\Ker ((\Z/n'\Z)^\times 
\to (\Z/n\Z)^\times)$ 
as we showed in Lemma~\ref{lem:Gal d=1}.
One then can check directly that 
$G(f')$ gives a canonical isomorphism between the 
Galois group of $f$ and the 
Galois group of the field extension $G'(f)$.

\subsubsection{}
Now we construct a punctured distribution.
Let $R$ be the presheaf on $\cFCo{1}$ 
obtained by taking the 
first K-group of each section of $G'$
so that $R(N)=K_1(G'(N))=K_1(F_N)$ for 
$N \in \cCo{1}$.  As the $K_1$ of a field is its
multiplicative group, if $N=\Z/n\Z$ 
with $n \ge 2$,
then $R(N)$ is isomorphic to 
$\Q(\zeta_n)^\times$ where 
$\zeta_n$ is a primitive $n$-th root of unity.
We remarked in the previous section that $G'$ preserves
Galois coverings.  It then follows that $R$ is a sheaf.

Let 
$m_{\Z/n\Z, \Z/n'\Z}\colon  \Z/n' \Z
\stackrel{q}{\twoheadleftarrow} 
\Z/n\Z 
\stackrel{=}{\hookrightarrow} 
\Z/n\Z
$
be a morphism in $\cCo{1}$.
Here $q$ sends $1 \modx n$ 
to $1 \modx n'$.  
Take 
$b\in \Gamma(\cO_X, \Z/n'\Z)\setminus \{0\}$.
We take $0 \le b \le n'-1$ and regard it as an element of 
$\Z$ or of $\Z/n\Z$ by abuse of notation.

Let $\psi_{\Z/n\Z}\colon  (\Z/n\Z)_{F_{\Z/n\Z}} 
\to {\Gm}_{,F_{\Z/n\Z}}$ 
denote the universal object for 
$F_{\Z/n\Z}$.
Let $\wt{b} \in q^{-1}(b)$.
Then the restriction of $\psi_{\Z/n\Z}$ 
to the section
$F_{\Z/n\Z} \subset (\Z/n\Z)_{F_{\Z/n\Z}}$
corresponding to $\wt{b}$ gives an element
in $K_1(F_{\Z/n\Z})= R(\Z/n\Z)$. We denote
this $n$-th root of unity by $\psi_{\Z/n\Z}(\wt{b})$.

\subsubsection{}
We compute $m_{\Z/n\Z, \Z/n'\Z}^*(\psi_{\Z/n'\Z}(b))$.
\begin{lem}
We have 
$m_{\Z/n\Z, \Z/n'\Z}^*(\psi_{\Z/n'\Z}(b))
=\psi_{\Z/n\Z}(\wt{b})^{\frac{n}{n'}}$.
\end{lem}
\begin{proof}
By definition, $m_{\Z/n\Z, \Z/n'\Z}^*(\psi_{\Z/n'\Z}(b))$
is given by the composite
$$
F_{\Z/n\Z} \inj (\Z/n'\Z)_{F_{\Z/n\Z}}
\xto{\psi} {\Gm}_{,F_{\Z/n\Z}}
$$
where the first morphism is the section corresponding to
$b$ and the second morphism $\psi$ is the base change of
$\psi_{\Z/n'\Z}$ with respect to 
$G'(m_{\Z/n\Z, \Z/n'\Z})\colon F_{\Z/n\Z} \to F_{\Z/n'\Z}$.
When regarded as an element
in $F_{\Z/n'\Z}(F_{\Z/n\Z})$, the morphism $\psi$
is the image under $G'(m_{\Z/n\Z,\Z/n'\Z})$ of 
the universal object 
$\psi_{\Z/n\Z} \in F_{\Z/n\Z}(F_{\Z/n\Z})$.

Hence it follows from the definition of 
$G'(m_{\Z/n\Z,\Z/n'\Z})$ that $\psi$ is equal to
the composite
$$
(\Z/n'\Z)_{F_{\Z/n\Z}}
\xto{\psi_{\Z/n\Z} \circ p^{-1}}
{\Gm}_{,F_{\Z/n\Z}}/\psi_{\Z/n\Z}(\Ker\, p)
\xto{\frac{n}{n'}} {\Gm}_{,F_{\Z/n\Z}}.
$$
Hence we have 
$m_{\Z/n\Z, \Z/n'\Z}^*(\psi_{\Z/n'\Z}(b))
=\psi_{\Z/n\Z}(\wt{b})^{\frac{n}{n'}}$.
\end{proof}

We construct a punctured distribution
$g\colon {\BS^*}' \to R$ as follows.
For $N\in \cCo{1}$, we set 
$g(N)\colon {\BS^*}'(N) \to 
\Z[\Gamma(\cO_X, N) \setminus \{0\}] \to K_1(F_N)=R(N)$
to be the map that sends 
$b\in \Gamma(\cO_X, N) \setminus \{0\}$ 
to 
$\psi_N(0) -\psi_N(b)=1-\psi_N(b)$.
This extends naturally to any object in $\cFCo{1}$.
We check that it is a morphism of presheaves.
Let $f\colon  N
\stackrel{p}{\twoheadleftarrow} 
N'' 
\stackrel{\iota}{\hookrightarrow} 
N'
$
be a morphism in $\cFCo{1}$.
We show that the diagram
\[
\begin{CD}
{\BS^*}'(N)    @>{g(N)}>>  R(N)
\\
@A{f^*}AA                    @A{f^*}AA
\\
{\BS^*}'(N')    @>{g(N')}>>  R(N')
\end{CD}
\]
is commutative.
We may assume that $f$ is a morphism in $\cCo{1}$.
We may further assume that either $\iota=\id$
or $p=\id$ since any morphism is a composite
of those two types of morphisms.
The commutativity in the case $p=\id$ is easy 
and is omitted.
Suppose $\iota=\id$.
We may assume 
$N=\Z/n\Z$ and $N'=\Z/n'\Z$ for some
$n, n' \in \Z$ 
and $p$ sends 1 to 1.

Then for $b \in \Gamma(\cO_X, \Z/n'\Z)\setminus \{0\}
\subset {\BS^*}'(\Z/n'\Z)$,
we have
\[
\begin{array}{lll}
g(\Z/n\Z)(f^*(b))
=g(\Z/n\Z)(\sum_{p(y)=b} y)\\
=\prod_{p(y)=b}(1-\psi_{\Z/n\Z}(y))
=\prod_{1 \le a \le n'}(1-\psi_{\Z/n\Z}
(\frac{n}{n'}a)\psi_{\Z/n\Z}(b))
\\
=1-\psi_{\Z/n\Z}(b)^{\frac{n}{n'}}
=f^*(1-\psi_{\Z/n'\Z}(b))
=f^*(g(\Z/n'\Z)(b)).
\end{array}
\]
At one point above, we used the identity
\[1-X^k=\prod_{1 \le a \le k}(1-\zeta_k^a X)\]
where $k \ge 1$ is an integer and $\zeta_k$ is 
a primitive $k$-th root of unity.
This proves that $g$ is a morphism of presheaves,
hence a punctured distribution.

\subsubsection{}
We now have Situation II.
Let us give the statement of the theorem in this case.
Suppose $n=n'p$ for some prime number $p$
such that $(n',p)=1$.
We defined $\kappa_{\Z/n\Z, b} \in R(\Z/n\Z)\cong \Q(\zeta_n)^\times$.
By definition, this is of the form $1-\psi_{\Z/n\Z}(b)$, 
that is, a cyclotomic unit.
Concerning these elements, we have the following corollary
of our main theorem.
\begin{cor}
We have
\[
m_{\Z/n\Z, \Z/n'\Z*}
\kappa_{\Z/n\Z,\wt{b}}
=(T_{[0]}-T_{[\Z/p\Z]}) 
\kappa_{\Z/n'\Z,b}.
\]
\end{cor}
\begin{proof}
This follows from Theorem~\ref{main theorem}.
\end{proof}
Below we compute explicitly the action of the Hecke operators 
that appeared above.
Of course, $T_{[0]}$ is the identity map.
\subsubsection{}
Let $N=\Z/n\Z\in \cCo{1}$ be of order prime to $p$.
Let 
$m_{N\oplus \Z/p\Z, N}\colon 
N
\stackrel{p}{\twoheadleftarrow} 
N \oplus \Z/ p \Z
\stackrel{=}{\hookrightarrow} 
N \oplus \Z/ p \Z$
and
$r_{N\oplus \Z/p\Z, N}\colon 
N
\stackrel{=}{\twoheadleftarrow} 
N 
\stackrel{\iota}{\hookrightarrow} 
N \oplus \Z/ p \Z$
where $p$ is the projection and 
$\iota(x)=(x,0)$.
\begin{lem}
\label{lem:4.2.3.1}
We have
\begin{enumerate}
\item
$m^*_{N\oplus \Z/p\Z,N}\kappa_{N,b}
=\kappa_{N\oplus \Z/p\Z, (pb,0)}$
\item
$r^*_{N \oplus \Z/p\Z, N} \kappa_{N,b}
=\kappa_{N \oplus \Z/p\Z, (b,0)}$
\end{enumerate}
\end{lem}
\begin{proof}
Let $\psi_N$
denote the universal 
object for $F_N$.
Its pullback by 
$G'(m_{N\oplus \Z/p\Z, N})$ then equals
\[
\psi\colon N_{F_{N\oplus \Z/p\Z}}
\xto{\psi_{N\oplus \Z/p\Z}\circ p^{-1}}
{\Gm}_{,F_{N\oplus \Z/p\Z}}/\psi_{N\oplus \Z/p\Z}(\Ker\, p)
\xto{\cong \text{p-th}}
{\Gm}_{,F_{N\oplus \Z/p\Z}}.
\]
Hence 
$m_{N\oplus \Z/p\Z, N}^*
\psi_N(b)
=\psi_{N\oplus \Z/p\Z}(p,0)$
and the claim (1) follows.

Similarly, the pullback by 
$G'(r_{N \oplus \Z/p\Z,N})$
of the universal object 
equals 
$\psi_{N\oplus \Z/p\Z}\circ \iota$.
Hence 
$r^*_{N\oplus \Z/p\Z}
\psi_N(b)
=\psi_{N\oplus \Z/p\Z}(b,0)$
and the claim (2) follows.
\end{proof}

\subsubsection{}
Now let us compute the action of the Hecke operators.
We have seen that $R$ is a sheaf.
Then by Lemma~\ref{lem:transfer_nonGal},
the map $m_{\Z/n\Z\oplus \Z/p\Z, \Z/n\Z*}$
is the composite
\[
R(N \oplus \Z/p\Z) \xto{\varphi}
R(N\oplus \Z/p\Z)^H
\xto{\overline{m}_{N\oplus \Z/p\Z,N}^{*-1}}
R(N)
\]
where $H=\Gal(m_{N\oplus \Z/p, N})$ is the 
Galois group and 
$\varphi(x)=\sum_{\sigma\in H}\sigma^*(x)$.
The definition of the Hecke operator tells us 
to identify
$R(N\oplus \Z/p)^H$
and $R(N\oplus [\Z/p])$.
Then we see that 
$m_{N\oplus[\Z/p\Z],N*}
=(\overline{m}_{N\oplus \Z/p\Z, N}^*)^{-1}$.

\begin{cor}
\label{cor:4.2.4.1}
Let $bp^{-1}\in \Z/n\Z$ denote the element 
such that $p (bp^{-1})=b \in \Z/n\Z$.
We have
\[
T_{[\Z/p\Z]} \kappa_{N,b}=\kappa_{N,bp^{-1}}.
\]
\end{cor}
\begin{proof}
By definition, 
$T_{[\Z/p\Z]}=m_{\Z/n\Z\oplus [\Z/p\Z]*}
r^*_{\Z/n\Z\oplus [\Z/p\Z]}$.
The two maps on the right-hand side can be computed using
Lemma~\ref{lem:4.2.3.1} and the discussion preceding this 
corollary.
\end{proof}
In this way we have avoided making the choice of roots of 
unity and yet have stated the norm relations.

\section{Hecke operators of double cosets}
\label{sec:Hecke double}
In this section, we show how our
Hecke operators defined in Section~\ref{sec:def Hecke}
coincide with the usual Hecke operators
defined using double cosets.

\subsection{}
\label{sec:Hecke comparison}
Let us consider the following setup.
Let $K$ be a nonarchimedean local field
and $\cO \subset K$ be the ring of integers.
Fix a positive integer $d$. 
Let $\bK=\GL_d(\cO) \subset \GL_d(K)$.
Let $\cH=\cH(\GL_d(K), \bK)$ denote the 
Hecke algebra of bi-$\bK$-invariant $\Z$-valued 
functions on $\GL_d(K)$.

Let $X=\Spec \cO$.
Let $F$ be an abelian sheaf on $\cFCot{d}$.
Set $V=\omega(F)$ so that it is a $\GL_d(K)$-module.
We have an action of $\cH$ on $V^\bK$.
Let $\varpi \in \cO$ be a uniformizer.
For $0\le r \le d$,
let 
$g_r=\mathrm{diag}(\varpi^{-1}, \dots, \varpi^{-1}, 1, \dots, 1)
\in \GL_d(K)$ 
denote the diagonal matrix 
with $\varpi^{-1}$ appearing
$r$ times and $1$ appearing $d-r$-times.
We let $T_{g_r} \in \cH$ denote the Hecke operator 
corresponding to (the characteristic function of)
the double coset $[\bK g_r \bK]$.
(This does not depend on the choice of the uniformizer.)

\subsubsection{}
Let $L_1 \subset L_2 \subset K^{\oplus d}$
be $\cO$-lattices.
Let $H \subset \Aut_{\cO}(L_2/L_1) \cong \Aut_{\cC^d}(L_2/L_1)$
be a subgroup.
Then by Proposition~\ref{prop:old cat equiv},
we have an isomorphism
$c_{L_1,L_2,H}\colon F((L_2/L_1)/H) \xto{\sim} V^{\bK_{L_1,L_2,H}}$.
We simply write $c_{L_1,L_2}=c_{L_1,L_2,\{e\}}$ when the group $H$ 
is the trivial group $\{e\}$.

\subsubsection{}
Set $L_0=\cO^{\oplus d}$.
We let $\varphi\colon  L_0/L_0 \to 0$ denote the morphism
in $\cFCo{d}$ represented by the diagram
$$
0 {\twoheadleftarrow} 0 
\inj L_0/L_0.
$$
(Remark: Of course $L_0/L_0 \cong 0$ but 
it is important to note explicitly the choice of the 
two lattices (in this case $L_0$ and $L_0$),
when considering the comparison isomorphisms
such as $c_{L_0, L_0}$.)  

Let $\wp \subset \cO$ be the maximal ideal and 
let $\kappa(\wp)=\cO/\wp$ be the residue field.
Recall that we defined Hecke operators $T_{[\kappa(\wp)^{\oplus r}]}$
for each 
$0 \le r \le d$.
It acts on $F(0)$.

We remark here that $\bK_{L_0,L_0}=\GL_d(\cO)$ by definition.
Hence $T_{g_r}$ acts on $V^{\bK_{L_0, L_0}}$.

We are ready to state the compatibility statement.
\begin{prop}
\label{prop:Hecke compatibility}
For each $0 \le r \le d$, the following diagram
\[
\begin{CD}
F(0)  @>{c_{L_0, L_0} \circ \varphi^*}>>  V^{\bK_{L_0, L_0}}
\\
@V{T_{[\kappa(\wp)^{\oplus r}]}}VV     @VV{T_{g_r}}V
\\
F(0)  @>{c_{L_0, L_0} \circ \varphi^*}>>    V^{\bK_{L_0, L_0}}
\end{CD}
\]
is commutative.
\end{prop}

\subsection{}
We first prove a series of lemmas concerning compatibilities
of the isomorphisms $c_{L_1, L_2, H}$
with respect to pullbacks and 
the $\GL_d(K)$-action.
Then we come back to the proof of Proposition later in 
Section~\ref{sec:Hecke compatibility}.
\begin{lem}
\label{lem:c_pullback}
Let $L_1 \subset L_2 \subset K^{\oplus d}$
and $H \subset \Aut_{\cC^d}(L_2/L_1)$
be a subgroup.
Then the diagram
\[
\begin{CD}
F((L_2/L_1)/H)   @>{(1)}>>   F(L_2/L_1)      \\
@V{\cong}V{c_{L_1,L_2,H}}V       @V{\cong}V{c_{L_1, L_2}}V  \\
V^{\bK_{L_1,L_2,H}}    @>{(2)}>>   V^{\bK_{L_1,L_2}}
\end{CD}
\]
is commutative.
Here  the map (1) is the pullback by the canonical quotient map
$L_2/L_1 \to (L_2/L_1)/H$.
The map (2) is the canonical inclusion (induced by the identity map on $V$).
\end{lem}
\begin{proof}
This follows from the construction of the isomorphisms 
$c_{L_1,L_2}$ and $c_{L_1, L_2, H}$ given in 
Proposition~\ref{prop:old cat equiv}.
\end{proof}

\begin{lem}
\label{lem:c_Rg}
Let $L_1' \subset L_1 \subset L_2 \subset L_2'$
be $\cO$-lattices.
Then the diagram
\[
\begin{CD}
F(L_2/L_1)   @>{(1)}>>   F(L_2'/L_1')   \\
@V{c_{L_1, L_2}}VV      @VV{c_{L_1', L_2'}}V\\
V^{\bK_{L_1, L_2}}    @>{(2)}>>   V^{\bK_{L_1', L_2'}}
\end{CD}
\]
is commutative.
Here the map (1) is the pullback by
$L_2/L_1   \twoheadleftarrow L_2/L_1'  \inj L_2'/L_1'$.
The map (2) is the inclusion induced by the identity map on $V$.
\end{lem}
\begin{proof}
This also follows from the construction of the isomorphisms 
$c_{L_1, L_2}$ and $c_{L_1', L_2'}$
given in Proposition~\ref{prop:old cat equiv}.
\end{proof}

\begin{lem}
Let $L_1' \subset L_1 \subset L_2 \subset L_2'$
 be $\cO$-lattices and $H \subset \Aut_{\cC^d}(L_2/L_1)$ 
be a subgroup.
Suppose that 
the composite map $f\colon  L_2'/L_1' \to L_2/L_1 \to (L_2/L_1)/H$,
where the first arrow is the map represented by 
$L_2/L_1 \twoheadleftarrow L_2/L_1' \inj L_2'/L_1'$
and the second arrow is the canonical quotient map,
is Galois.
Then $\bK_{L_1', L_2'} \subset \bK_{L_1,L_2, H}$
is a normal subgroup and the Galois group is canonically
isomorphic to $\bK_{L_1,L_2,H}/\bK_{L_1', L_2'}$.
\end{lem}
\begin{proof}
We have 
\[
\begin{array}{ll}
\Aut_{(L_1/L_2)/H}(L_1'/L_2')
\cong
\Aut_{\omega((L_1/L_2)/H)}(\omega(L_1'/L_2'))
\\
\cong 
\Aut_{\GL_d(K)/\bK_{L_1,L_2,H}}
(\GL_d(K)/\bK_{L_1',L_2'})
=
N_{\bK_{L_1,L_2,H}}(\bK_{L_1',L_2'})
/\bK_{L_1', L_2'}
\end{array}
\] 
where $N$ means the normalizer.
The first isomorphism uses the equivalence of categories
induced by $\omega$, and the second follows from the 
computation of $\omega$.

Since $f$ is Galois,
we have 
\[
[\bK_{L_1,L_2,H}:\bK_{L_1', L_2'}]
=\deg f
=\#\Aut_{(L_1/L_2)/H}(L_1'/L_2')
=\#N_{\bK_{L_1,L_2,H}}(\bK_{L_1', L_2'})
/\bK_{L_1', L_2'}.
\]
Here the first equality is by definition,
and the second equality follows easily
from 
Lemma~\ref{lem:transfer_nonGal}.
This implies that $\bK_{L_1', L_2'}$ 
is a normal subgroup of 
$\bK_{L_1, L_2, H}$
and the Galois group is isomorphic to 
$\bK_{L_1, L_2,H}/\bK_{L_1', L_2'}$ as claimed.
\end{proof}

\begin{lem}
\label{lem:c_pushforward}
Let $L_1' \subset L_1 \subset L_2 \subset L_2'$
be $\cO$-lattices.
Let $H \subset \Aut_{\cFCo{d}}(L_2/L_1)$ and 
$H' \subset \Aut_{\cFCo{d}}(L_2'/L_1')$
be subgroups.
Suppose there is a morphism  
$f\colon  (L_2'/L_1')/H' \to (L_2/L_1)/H$
in $\cFCot{d}$ induced by the morphism
\[
L_2/L_1 \twoheadleftarrow  L_2/L_1'  \inj L_2'/L_1',
\]
where the arrows are induced by the identity map on $L_2$
and the inclusion $L_2 \subset L_2'$ respectively.
Then the diagram
\[
\begin{CD}
F((L_2'/L_1')/H'))  @>{f_*}>>  F((L_2/L_1)/H))    \\
@V{c_{L_1', L_2', H'}}V{\cong}V       @V{\cong}V{c_{L_1, L_2, H}}V   \\
V^{\bK_{L_1', L_2', H'}}   @>{\tr}>>    V^{\bK_{L_1, L_2, H}}  
\end{CD}
\]
is commutative.  
Here $\tr(x)= \sum_{\sigma \in \bK_{L_1,L_2, H}/\bK_{L_1', L_2', H'}} \sigma (x)$
is the trace (norm, transfer) map.
\end{lem}

\begin{proof}
This follows from Lemma~\ref{lem:transfer_nonGal}.
\end{proof}

\subsection{}
\subsubsection{}
\label{sec:Hecke compatibility}
We give a proof of Proposition~\ref{prop:Hecke compatibility}.

Set $L_0=\cO^{\oplus d}$
and $\bK=\bK_{L_0, L_0}=\GL_d(\cO)$.
Fix $0 \le r \le d$ and write $g=g_r$ for short.
By definition, the action of $T_g$
on $V$ is given as the composite
\[
V^\bK  \xto{(1)}
V^{\bK \cap g^{-1} \bK g}
\xto{\tr}
V^{g^{-1} \bK g}
\xto{g\cdot -}
V^\bK
\]
where $(1)$ is the canonical inclusion and 
$\tr(v)=\sum_{\sigma \in g^{-1} \bK g \backslash \bK \cap g^{-1} \bK g} \sigma(v)$.

\subsubsection{}
Note that $L_0 g= (\wp^{-1})^{\oplus r} \oplus \cO^{\oplus d}$.
Let 
$$[L_0 g/L_0]
=(L_0 g/L_0)/\Aut_{\cC^d}(L_0g/L_0)
\in \cFCot{d}.$$

We have a map
$r\colon  L_0/L_0 =L_0/L_0 \inj L_0 g/L_0$
in $\cFCo{d}$.  (Remark again: Of course $L_0/L_0 \cong 0$
but it will be important to specify the lattice(s) when considering
the compatibility map such as $c_{L_1, L_2}$.)
We write $\overline{r}\colon  [L_0 g/L_0] \to L_0/L_0$
for the morphism
in $\cFCot{d}$
induced by $r$.
We have another map
\[
m\colon L_0g/L_0 g \twoheadleftarrow  L_0g/L_0 \xto{=}  L_0 g/L_0
\]
where the arrow is the map induced by the identity map on $L_0g$.
We write 
$\overline{m}\colon  [L_0g/L_0] \to L_0g/L_0g$
for the morphism in $\cFCot{d}$ induced by $m$.

We consider the following diagram.
{\small{
\[
\begin{CD}
\label{eqn:c_L0}
F(L_0/L_0) @>{\overline{r}^*}>>  F([L_0g/L_0])   @>{\overline{m}_*}>>  
F(L_0g/L_0g)   @>{R_g^*}>>F(L_0/L_0)  \\
@V{\cong}V{c_{L_0, L_0}}V 
@V{\cong}V{f}V 
@V{\cong}V{c_{L_0g,L_0g}}V 
@V{\cong}V{c_{L_0, L_0}}V \\
V^\bK   @>>>   V^{\bK \cap g^{-1}\bK g}  @>{\tr}>> 
V^{g^{-1}\bK g}   @>{g\cdot -}>>   V^\bK
\end{CD}
\]
}}
where the map $f$ is $c_{L_0g, L_0g, \Aut_{\cC^d}(L_0g/L_0)}$.
We see that the left square is commutative 
since 
$\bK \cap  g^{-1}\bK g=
\bK_{L_0, L_0g, \Aut_{\cC^d}(L_0g/L_0)}$
and by the injectivity of 
$V^{\bK \cap g^{-1} \bK g}
\to V^{\bK_{L_0, L_0g}}$
using Lemmas~\ref{lem:c_pullback},~\ref{lem:c_Rg}.
The square in the middle is commutative by 
Lemma~\ref{lem:c_pushforward}.
The commutativity of the right
square follows from
the definition of $R_g^*$.

\subsubsection{}
Now consider the following diagram in $\cFCo{d}$.
{\small{
\[
\begin{CD}
L_0/L_0  @<{r}<<
L_0g/L_0  @>{m}>>
L_0g/L_0g  @<{R_g}<<
L_0/L_0
\\
@A{\varphi}AA
@A{(1)}AA
@A{(2)}AA
@A{\varphi}AA
\\
0    @<{r_{\kappa(\wp)^{\oplus r}, 0}}<<
\kappa(\wp)^{\oplus r}   @>{m_{\kappa(\wp)^{\oplus r},0}}>>
0   @<{\id}<<
0
\end{CD}
\]
}}
The map (2) is defined in a similar manner as $\varphi$.
For the map (1), take any isomorphism 
$\alpha\colon  \kappa(\wp)^{\oplus r} 
\xto{\cong} L_0g /L_0$
of $\cO$-modules.
Then the map (1) is defined to be the map represented by
\[
L_0g/L_0 \xleftarrow{\alpha}
\kappa(\wp)^{\oplus r}
\xto{=} \kappa(\wp)^{\oplus r}.
\]
One can check directly that the two left squares are commutative.
The third square is commutative since
$\Hom_{\cFCo{d}}(0, L_0g/L_0g) \cong
\Hom_{\cFCo{d}}(0,0)$
consists of one element.

By Lemma~\ref{lem:quot_FCd},
$[L_0g/L_0]$ is a 
quotient object.
We thus obtain maps $\overline{m}$
and $\overline{r}$ as in the diagram below from $m$ 
and $r$ respectively
by the property of a quotient object.
We obtain the following commutative diagram in 
$\cFCot{d}$:
{\small{
\[
\begin{CD}
L_0/L_0    @<{\overline{r}}<<
[L_0g/L_0] @>{\overline{m}}>>
L_0g/L_0g @<{R_g}<<
L_0/L_0
\\
@A{\varphi}AA
@A{(1)'}AA
@A{(2)}AA
@A{\varphi}AA
\\
0                @<{r_{0, [\kappa(\wp)^{\oplus r}]}}<<
[\kappa(\wp)^{\oplus r}]   @>{m_{0, [\kappa(\wp)^{\oplus r}]}}>>
0
                  @<{\id}<<
0
\end{CD}
\]
}}
where $(1)'$ is the map induced by $(1)$.

We take the sections of $F$ of this diagram and obtain a 
commutative diagram:
 {\small{
\[
\begin{CD}
F(L_0/L_0)   @>{\overline{r}^*}>>
F([L_0g/L_0])   @>{\overline{m}_*}>>
F(L_0g/L_0g)   @>{R_g^*}>>
F(L_0/L_0)
\\
@V{\varphi^*}VV
@V{(1)'^*}VV
@V{(2)^*}VV
@V{\varphi^*}VV
\\
F(0)     
 @>{r^*_{0, [\kappa(\wp)^{\oplus r}]}}>>
F([\kappa(\wp)^{\oplus r}])
@>{m_{0,[\kappa(\wp)^{\oplus r}]*}}>>
F(0)
@>{=}>>
F(0).
\end{CD}
\]}}
By definition, the bottom row is 
$T_{[\kappa(\wp)^{\oplus r}]}$
on $F(0)$.
Using the commutative diagram above with
the commutative diagram in Section \ref{eqn:c_L0},
we obtain the proposition.
\qed

\section{The universal Euler system of Distribution sheaves}
\label{sec:universal Euler}
Recall that in the notation of our main theorem 
(Theorem~\ref{main theorem}), 
we took $G$ to be a presheaf with transfers.
In this section, we look at what happens if $G$ is a sheaf (recall that a sheaf has a canonical structure as a presheaf with transfers, and a presheaf with transfers with values in $\Q$-vector spaces is naturally a sheaf).
In this sheaf case, there is a universal setup for the theorem (to be made precise below), which can be constructed 
using the presheaf of distributions.
In the universal setup, we see the connection more directly with the (usual) distributions of Schwartz-Bruhat functions on the ring of finite adeles, justifying the use of the notation $\BS$ for our (pre)sheaf of distributions.

This is the formulation that appears in Colmez's article \cite{Colmez}
on Kato's Euler systems (i.e., $d=2$).
We do not include the proof here, but in this universal case, there is a much simpler proof of the norm relation theorem.   This proof appeared in Colmez for $d=2$, and in Grigorov \cite{Grigorov} 
and in Kondo-Yasuda \cite{Lepsilon}.

\subsection{}
We give an example of Situation I of Section~\ref{sec:Situations}.
Via the equivalence mentioned in Secion~\ref{sec:BS functions},
we see that our theorem translates (when $d=2$) to the statement
in Colmez's article (the proof of Prop.\ 1.10 in \cite{Colmez}).

\subsubsection{}
Let 
$G_\BS$ denote the sheaf of rings associated 
with the presheaf $\oplus_{r \ge 0} \BS'^{\otimes r}$
of tensor algebra on $\BS'$.
Set $G=G_\BS$.  For $1 \le i \le d$, 
let $F_i=F_i'=\BS$ and 
let
$\alpha_i\colon  \BS' \to F_i$ 
be the canonical map of 
a presheaf to its associated sheaf.
Each map $g_i\colon  F_i \to G_\BS$ 
is the inclusion into the degree 1 part.
Then the data $(G, F_i, F_i', \alpha_i, g_i)$
is an example of Situation I.
One can also obtain a Situation II example 
by replacing $\BS$ by $\BS^*$ above.
\subsubsection{}
This example is universal in the following sense.
Suppose we are given the data
$(G, F_i, F_i', \alpha_i, g_i)$
for Situation I.
Assume that  
(1) $G$ is a sheaf, and 
(2) $F_i=F$, $\alpha_i=\alpha$, and  $g_i=g$ for all $i$
 (for some $F$, $\alpha$, and $g$).
Let $F^a$ denote the sheaf associated with $F$ 
and 
$c\colon F \to F^a$ the canonical map.
Let $\alpha^a\colon F^a \to G^a=G$
denote the map associated with $\alpha$.
Then the data 
$(G, \BS, \BS, \id_\BS, \alpha^a \circ (c \circ g)^a)$
is a setup for Situation I.

One can check easily that the elements 
$\kappa_{\bN, (b_j)} \in G(\bN)$ constructed 
using the data 
$(G, F, F_i', \alpha, g)$ and 
those using the data
$(G, \BS, \BS, \id_\BS, \alpha^a \circ (c \circ g)^a)$
are the same.

From the map $(\alpha^a \circ c \circ g)^a\colon \BS \to G$
and the universality of the tensor algebra,
one obtains a map 
$\eta\colon G_\BS \to G$ 
of sheaves and 
a commutative diagram of presheaves:
\[
\begin{CD}
G_\BS  @>{\eta}>>  G  
\\
@AAA      @AA{\alpha^a}A 
\\
\BS    @>{(c \circ g )^a}>>    F^a
\\
@AAA         @AA{c \circ g}A 
\\
\BS'    @>{=}>>   \BS'
\end{CD}
\]
Here the left upper vertical arrow is the inclusion at degree 1
and the left lower vertical arrow is the canonical map.
Let $\bN$ and $b_j$ be as in Section~\ref{sec:Situations}.
Using this diagram, 
one can check that the element 
$\kappa_{\bN, (b_j)} \in G(\bN)$
is the image by $\eta(\bN)$ of 
$\kappa_{\bN, (b_j)} \in G_\BS(N)$.
As the map $\eta$ is compatible with 
transfers and pullbacks,
the norm relation for $G_\BS$
(i.e., the statement of Theorem~\ref{main theorem})
gives the relation for $G$.
Therefore, it suffices to prove the theorem in the 
case $(G_\BS, \BS, \BS, \id_\BS, \beta)$, where 
$\beta$ is the canonical inclusion to the degree 1 part,
to cover the cases where
the assumptions (1) and (2) hold.
\subsubsection{}
Let us write the statement for $G_\BS$
in terms of the Bruhat-Schwartz functions $S(\A^d)$
using the isomorphisms induced by the functor 
$\omega$ (see Section~\ref{sec:BS functions}).
We treat the simplest but essential case
(see the beginning of the proof of Theorem~\ref{main theorem}(2) in
Section~\ref{sec:3.4})
where 
$N_i=0, N_i'=A/ \wp, \bN=\oplus_{i=1}^d N_i=0$,
$\bN' =\oplus_{i=1}^d N_i' = (A/\wp)^d$.
Let $L=L_0=\wh{\cO}_X^d$ 
and $L_1=\wp \wh{\cO}_X^d$ be 
$\wh{\cO}^d_X$-lattices in $\cA^d$ so that
$\bN=L/L_0$ and $\bN'=L/L_1$.
By definition, 
$\bK_{L, L_0}=\GL_d(\wh{\cO}_X)$
and $\bK_{L, L_1}=
\Ker[\GL_d(\wh{\cO}_X) 
\to \GL_d(\wh{\cO}_X / \wp \wh{\cO}_X)]$.
We have 
$\BS(\bN)=\BS(L/L_0) \cong S(\A_X^{\oplus d})^{\bK_{L, L_0}}$
and  
$\BS(\bN')=\BS(L/L_1) \cong S(\A_X^{\oplus d})^{\bK_{L,L_1}}$.
The characteristic function $\chi_L$ on $L$ lies
in $\BS(\bN)$ and 
the characteristic
functions $\chi_{b+L_1}$ with $b \in L/L_1$
belong to $\BS(\bN')$.
The Galois group of 
$\bN\stackrel{p}{\twoheadleftarrow} \bN'
\stackrel{=}{\hookrightarrow} \bN'$
is isomorphic to $\GL_d(A/\wp)$ 
(see Lemma~\ref{lem:Gal d=1} for the 
case $d=1$; the general case is similar).
An element 
$g \in \GL_d(A/I)$ 
acts on $\BS(\bN')$ as 
$g \circ \chi_{b+L_1}=\chi_{bg+L_1}$.
Now consider $\BS^{\otimes d}$.
We have
$\BS^{\otimes d}(\bN')=
\BS^{\otimes d}(L/L_1)
\cong (S(\A_X^{\oplus d})^{\bK_{L,L_1}})^{\otimes d}
=S(\Mat_d(\A_X))^{\bK_{L,L_1}}
$
and 
$\BS^{\otimes d}(\bN) \cong S(\Mat_d(\A_X))^{\bK_{L,L_0}}$.
The element $\kappa_{N,(0)}$ is then 
the characteristic function 
$\chi_{\Mat_d(\wh{\cO}_X)}$.
Let 
$b_i=(\delta_{ij})_{1 \le j \le d} \in L/L_1$.
Then the element 
$\kappa_{\bN', (b_j)}$
is the characteristic function 
$\chi_{1+ \wp \Mat_d(\wh{\cO}_X)}$.
Here 1 means the identity matrix.
Hence we obtain 
$m_* \kappa_{\bN', (b_j)}= \chi_{\GL_d(\wh{\cO}_X)}
\in S(\Mat_d(\A_X))^{\bK_{L,L_0}}$.
Hence the theorem amounts to 
describing the relation 
between 
$\chi_{\GL_d(\wh{\cO}_X)}$
and $\chi_{\Mat_d(\wh{\cO}_X)}$
using Hecke operators.
(We refer to Section~\ref{sec:Hecke comparison}
for the comparison
of the usual Hecke operators and the 
ones defined in this article.)

The proof of the statements in Theorem~\ref{main theorem} 
in this universal sheaf case
is much simpler than the one presented in this article.
The proof for $e=d$ is found in Grigorov~\cite[Prop. 1.4.2]{Grigorov}, and 
for $e=d-1, d$ is found in our other article~\cite[Section 5]{Lepsilon}.  
The method of the proof in loc.\ cit.\
applies to general $e$.  
We remark here that 
this example for $d=2$ is alluded to in the article 
by Colmez~\cite[Proof of Prop. 1.10]{Colmez}.   
We hope this explains why he considers 
distributions (of Bruhat-Schwartz functions).

\section{In the motivic cohomology of Drinfeld modular schemes}
\label{sec:Drinfeld Euler}

We proved in previous sections that certain (abstract) elements in 
some presheaves with transfers constructed out of distributions
form an Euler system (a norm compatibility system).

In this section, we give an application.    The presheaves with transfers 
are motivic cohomology groups of Drinfeld modular schemes 
(defined over an open subset $U$ of $\Spec A$, see text below 
for notation) with integral coefficients.   The distributions are those of Siegel units.
The main theorem states that the elements in motivic cohomology constructed as
products of Siegel units form an Euler system.

Prior to this work, we have constructed elements 
in the rational K-theory of Drinfeld modular varieties
over $k$ and showed the norm compatibility.
That is, we improve our previous result in two 
respects: to work over $U$ rather than over $k$,
and to work with coefficients in $\Z$ rather than in $\Q$.
The first point should have been realized back when
we worked on our paper, but somehow it escaped 
us that the Siegel units we have been considering
are actually defined over $U$ (we always worked 
away from the level).
We have never written down the second point until now,
because there is not really an arithmetic application of 
this result.   (For the zeta elements over $k$, we did use the fact
that the elements form an Euler system to compute 
the regulator.)   We hope to see some applications of 
this result.

We also remark here that our result concerns motivic cohomology
and not algebraic K-theory.   There is a technical issue,
which disappears with coefficients in $\Q$, that we do not know how to avoid.   


\subsection{Drinfeld modular varieties}
\label{sec:nu Drinfeld modules}

\subsubsection{Notations}
\label{nu notation3}
Let $C$ be a smooth projective geometrically
irreducible curve over a finite
field $\F_q$ of $q$ elements.
Let $F$ denote the function field of $C$.
Fix a closed point $\infty$ of $C$.
Let $q_{\infty}$, $F_{\infty}$,
$|\ |_{\infty}\colon F_{\infty}\to q_{\infty}^{\Z}\cup\{0\}$
denote the cardinality of
the residue field of $C$ at $\infty$,
the completion of $F$ at $\infty$, and absolute value
at $\infty$, respectively.
Let $A=\Gamma(C\setminus \{\infty\},\cO_C)$ be the coordinate
ring of the affine $\F_q$-scheme
$C \setminus \{\infty\}$.

\subsubsection{}
We fix an integer $d \ge 1$.
\begin{defn}[\cite{Dr}]
Let $S$ be a scheme over $\Spec\, A$. A {\em Drinfeld module}
of rank $d$ over $S$ is a scheme $E$ in $A$-modules over
$S$ satisfying the following conditions:
\begin{enumerate}
\item Zariski locally on $S$, $E$ is isomorphic to $\Ga$
as a commutative group scheme.
\item If we denote the $A$-action on $E$ by
$\varphi \colon A\to \End_{S\text{-group}}(E)$,
then, for every $a \in A\setminus \{0\}$, the $a$-action
$\varphi(a)\colon E\to E$ on $E$ is finite, locally free
of constant degree $|a|_{\infty}^{d}$.
\item The $A$-action on $\Lie \,E$ induced by $\varphi$
coincides with the $A$-action on $\Lie\, E$ which comes
from the structure homomorphism $A\to \Gamma(S,\cO_S)$.
\end{enumerate}
\end{defn}

\subsubsection{}
Let $N$ be a torsion $A$-module.
Let $S$ be an $A$-scheme 
and $(E, \varphi)$ be a Drinfeld module of 
rank $d$ over $S$.

\begin{defn}
A {\it{structure of level $N$}}
on $(E, \varphi)$
is a morphism $\psi$ of $A$-modules
\[
\psi\colon N \to \Hom_{S\text{-schemes}}(S, E)
\]
such that there is an equality
\[
\sum_{m \in \Ker a\colon N \to N}
\psi(m)
=
\Ker(a\colon  E \to E)
\]
as effective Cartier divisors of $E$ for each
nonzero $a \in A$.
\end{defn}

Let $d \ge 1$ be an integer. Unless $N$ is cofinitely generated 
by at most $d$ elements as an $A$-module, there exists
no level $N$-structure on a Drinfeld module of rank $d$.
Below we only consider the structure of level $N$ for an $A$-module
$N$ of finite length.

\subsubsection{}
We recall the representability and regularity of moduli spaces from our other paper \cite{Regularity}.
Let $U \subset \Spec A$
be an open subscheme.
Let $\cM_{N, U}^d$ denote 
the functor
\[
(U\text{-scheme})
\to
(Set)
\]
that sends a $U$-scheme $S$ to 
the set of isomorphism classes 
of Drinfeld modules of rank $d$ over $S$
with structure of level $N$.

\begin{prop}
\label{prop:level N moduli}
Let $N$ be a nonzero finitely generated torsion $A$-module. 
\begin{enumerate}
\item
Suppose $N$ is a finitely generated $A$-module and 
$|\Supp N| \ge 2$.
Let $U \subset \Spec A$ 
be an open subscheme.
Then the functor $\cM_{N,U}^d$ is representable
by a regular affine $U$-scheme.
\item
Let $Z \subset \Supp N$ be a nonempty subset.
Let $U \subset \Spec A \setminus Z$ 
be an open subscheme.
Then the functor $\cM_{N,U}^d$
is representable by a regular affine 
$U$-scheme.
\end{enumerate}
\end{prop}
\begin{proof}
This is \cite[Prop. 4.2.1]{Regularity}.
\end{proof}

\subsubsection{}
Let $U \subset \Spec A$ be a nonempty open subscheme.
Let $\cC^d_U$ denote the full subcategory of $\cC^d$
consisting of objects $N$ such that the pair $(N, U)$ 
meets the condition of (1) or (2) of Proposition~\ref{prop:level N moduli}.
We consider the functor 
\[
\cC^d_U \to (U\text{-schemes})
\]
that sends an object $N \in \cC^d_U$ to $\cM^d_{N,U}$.
Let $h\colon  N_1 \to N_2$ be a morphism in $\cC^d_U$.
Suppose it is represented by the diagram:
\[
N_2 \stackrel{g}{\twoheadleftarrow} N_3 \stackrel{f}{\hookrightarrow} N_1.
\]
There is a morphism of schemes $\cM^d_{N_1,U} \to \cM^d_{N_3, U}$
corresponding to the injection $f$, and a 
morphism of schemes $\cM^d_{N_3, U} \to \cM^d_{N_2,U}$ corresponding 
to the surjection $g$ (recalled later in 
Section~\ref{sec:distribution Siegel}).   
We associate the 
composite morphism to the morphism $h$ and obtain 
the functor above.

\subsubsection{}
We consider the map 
\[
\mathrm{ob}\,\wt{\cC^m} \to (U\text{-schemes})
\]
where we set $\cC=\cC^d_U$,
that sends 
an object $H \backslash N$ of $\wt{\cC^m}$
to the quotient scheme $H \backslash \cM^d_{N,U}$.
We let  
$\wt{\cC}$
denote the full subcategory of 
$\wt{\cC^m}$ 
consisting of objects $H \backslash N$
such that $H \backslash \cM^d_{N,U}$
is regular.

\begin{rmk}
We proved regularity of certain types of quotients 
$H\backslash \cM_{N,U}^d$
in our paper \cite{Regularity}.
The most important examples 
for our application are
those quotients by parabolic subgroups
of $\GL_d(A/I)$
of the moduli $\cM^d_{(A/I)^d,U}$.
These will appear in the proof of our main theorem.
\end{rmk}

\subsubsection{}
\label{sec:functor M^d}
Let $(U\text{-schemes})_{ff}$
denote the subcategory where all the morphisms 
are assumed to be finite flat.
We construct a functor 
\[
\cM^d_{-, U}\colon \wt{\cC} \to (U\text{-schemes})_{ff}
\]
as follows.   
An object $H\backslash N$ is sent to
the object $H\backslash \cM^d_{N,U}$.
Let $H \backslash X \to K \backslash Y$
be a morphism in $\wt{\cC}$.   
The corresponding morphism is constructed
as follows.
Take a representative
\[X \xleftarrow{m} Z \xto{r} Y.
\]
such that 
$J \backslash Z \cong H \backslash X$
for some subgroup $J \subset \Aut_\cC(Z)$.
The quotient morphism 
$\cM_{Z,U}^d \to J \backslash \cM_{Z,U}^d$ 
is flat since $\cM_{Z,U}^d$ and 
$J \backslash \cM_{Z,U}^d$ are regular 
and any finite map between 
regular schemes is 
flat (see \cite[V, p.95, 3.6]{AK}).
It is faithfully flat since it is surjective.
Hence the morphism $r$ descends,
giving us a morphism $H \backslash 
\cM^d_{X,U} 
\cong J \backslash \cM^d_{Z,U}
\to K \backslash \cM^d_{Y,U}$.
Note that finiteness, flatness
and surjectivity
also descend.
This gives the desired functor.

Let $\wt{\cFC}$ denote the category 
obtained from $\wt{\cC}$ by 
adding finite coproducts (i.e., the 
procedure of Section~\ref{sec:add coproduct}).
This functor extends naturally to a functor
$\wt{\cFC} \to (U\text{-schemes})_{ff}$.
We also denote this functor by $\cM^d_{-,U}$.

\subsection{Theta functions}
\label{sec:theta functions}
Let $(E,\varphi)$ be a Drinfeld module of rank $d$
over an $A$-scheme $S$.
Let $\pi\colon E\to S$ denote the
structure morphism. We regard $S$ as a closed subscheme
of $E$ via the zero section $S\inj E$.
\begin{defn}[pre-theta function]
\label{def:pretheta}
Let $f\in \Gamma(E\setminus S, \cO_E^{\times})$.
We say $f$ is a pre-theta function if 
the following two conditions are satisfied:
\begin{enumerate}
\item For $a\in A\setminus \{0\}$, let
$N_a \colon  \cO_{E\setminus \Ker\varphi(a)}^{\times} \to
\cO_{E\setminus S}^{\times}$ denote the norm map
with respect to the finite locally free morphism
$\varphi(a)\colon E\setminus \Ker \varphi(a) \to E\setminus S$.
Then $N_a(f)=f$ for any $a\in A\setminus \{0\}$.
\item The order $\ord_{S}(f)$ of zero of $f$
at the closed subscheme $S$ is equal to $q_{\infty}^d -1$.
\end{enumerate}
\end{defn}

\begin{lem}[cf. {\cite[Prop 1.3, p.121]{Kato}} and {\cite[Thm 1.2.1, p.384]{Scholl}}]
\label{lem:theta}
Let the notations be as above. Suppose $S$
 is reduced. 
Then there exists a pre-theta function and 
is unique up to $\mu_{q^d_\infty-1}(S)$.
\end{lem}
\begin{proof}
Let us consider the exact sequence
$$
0 \to \cO_S^{\times} \to \pi_{*}\cO_{E \setminus S}^{\times}
\xto{\ord_{S}} \Z_S \to 0
$$
of Zariski-sheaves on $S$.
The multiplicative monoid $A\setminus\{0\}$
acts on $\cO_{E\setminus S}$ by the norm
map $N_a$ for $a\in A\setminus\{0\}$.
The above exact sequence induces the structure
of $A\setminus\{0\}$-module on $\Z_S$ and on $\cO_S^{\times}$,
becomes an exact sequence of
$A\setminus\{0\}$-modules, and defines an element of
the extension module
$\Ext^1_{\Z[A\setminus \{0\}]_S}(\Z_S,\cO_S^{\times})$
in the abelian category of
Zariski sheaves of $A\setminus \{0\}$-modules on $E$.
Since $A\setminus \{0\}$ acts trivially
on $\Z_S$ and via the character
$|\ |_{\infty}^{d}\colon A\setminus \{0\} \to q_{\infty}^{d\Z_{\ge 0}}$
on $\cO_S^{\times}$,
we have
$(|a|_{\infty}^{d} -1)
\Ext^1_{\Z[A\setminus \{0\}]_S}(\Z_S,\cO_S^{\times})
=0$ for any $a \in A\setminus\{0\}$.
Since the greatest common divisor of
$|a|^d_\infty-1$ as $a$ runs through $A\setminus\{0\}$
is $q_\infty^d-1$, 
the extension group
$\Ext^1_{\Z[A\setminus \{0\}]_S}(\Z_S,\cO_S^{\times})$
is annihilated by $q_{\infty}^d -1$.
In particular, the above exact sequence splits
after pulling back by $q_{\infty}^d -1 \colon \Z_S \to \Z_S$.
Now let $f$ be the image of $1\in \Z_S$ by the section
which gives the splitting.

The choice of such an $f$ 
is unique up to
$\Hom_{\Z[A\setminus \{0\}]_S}(\Z_S,\cO_S^{\times})
\cong \mu_{q_{\infty}^d -1}(S)$. 
\end{proof}

\begin{defn}[theta function]
\label{def:theta}
Let $f$ be a pre-theta function 
for a Drinfeld module $(E, \varphi)$ 
of rank $d$
over a reduced scheme $S$.
We set $\theta_{E/S}=f^{q_\infty^d-1}$.
By Lemma~\ref{lem:theta}, 
$\theta_{E/S}$ is uniquely determined,
independently of the choice of the pre-theta
function $f$.
We call $\theta_{E/S}$ the theta function.
\end{defn}

\subsection{Distribution property of Siegel units}
\label{sec:distribution Siegel}
\begin{lem}
\label{lem:pullback theta}
Consider the cartesian diagram
\[
\begin{CD}
E'=E \times_{S'}S     @>{h}>>   E
\\
@VVV     @VVV
\\
S'    @>>> S
\end{CD}
\]
where $E$ is a Drinfeld module 
over $S$ and $E'$ is a Drinfeld module over $S'$,
and $S$, $S'$ are reduced schemes.
Then
\[
h^* \theta_{E/S}
=
\theta_{E'/S'}.
\]
\end{lem}
\begin{proof}
Let $f$ be a pre-theta function for $E/S$.
It is easy to check that 
Conditions (1)(2) of Definition~\ref{def:pretheta} hold
for $h^* f$, i.e., 
$h^* f$ is a pre-theta function for $E'/S'$.
Taking the $(q_\infty^d-1)$-st power, we obtain
the claim.
\end{proof}

\begin{lem}
\label{lem:pushforward theta}
Let $E' \xto{h} E$
be an isogeny of Drinfeld module over a reduced base scheme $S$.
Then
\[
h_* \theta_{E'/S}=\theta_{E/S}
\]
\end{lem}
\begin{proof}
Because $h$ is an isogeny,
it commutes with the $A$-action,
i.e.,
\[
h \circ \varphi_a
=\varphi'_a \circ h
\]
where $\varphi\colon  A \to \End(E)$, 
$\varphi'\colon  A \to \End(E')$
are the structures of Drinfeld modules.
Let $f$ be a pre-theta function for $E'/S$.
It follows that $h_*f$ satisfies 
Definition~\ref{def:pretheta}(1).

One can check (2) directly so that $h_* f$
is a pre-theta function. 
Taking the $(q_\infty^d-1)$-st power, we obtain
the claim.
\end{proof}

\subsubsection{}
Let $d \ge 1$.
Let $(N, U)$ be a pair satisfying the condition of 
Proposition~\ref{prop:level N moduli}.
We let 
$E_{N,U}^d \to \cM_{N, U}^d$
denote the universal Drinfeld module.
Since
$\cM_{N,U}^d$
is regular (in particular reduced),
we have a pre-theta function 
in 
$\cO(E_{N, U}^d \setminus \{0\})^\times$.
We write
$\theta_{N,U}^{(d)'}$
for any one of the pre-theta functions.
We write 
$\theta_{N,U}^{(d)}=(\theta_{N,U}^{(d)'})^{q_\infty^d-1}$
for the theta function.

\subsubsection{}
Let $\psi\colon N \to
\Hom(
\cM_{N,U}^d, 
E_{N,U}^d)
$
denote
the universal level structure.

For $n \in N$, 
we set
\[
g_{N, U, n}^{(d)'}
=
\psi(n)^*
\theta_{N,U}^{(d)'}
\]
and
\[
g_{N, U, n}^{(d)}
=
\psi(n)^*
\theta_{N,U}^{(d)}.
\]
These are elements of 
$\cO(\cM_{N,U}^d)$.
Suppose 
$\psi(n)(\cM_{N,U}^d) \cap \psi(0)(\cM_{N,U}^d)=\emptyset$.
Then 
$g_{N, U, n}^{(d)'}$ and 
$g_{N, U, n}^{(d)}$
belong to 
$\cO(\cM_{N, U}^d)^\times$.
We call them Siegel units and pre-Siegel units.

Let $N_1, N_2 \in \cC_U^d$ and
$N_1 \xto{\alpha} N_2$
be an injection of $A$-modules.
We obtain a cartesian diagram:
\[
\begin{CD}
E_{N_2, U}^d
@>{r_E}>>
E_{N_1, U}^d
\\
@VVV    @VVV
\\
\cM_{N_2, U}^d
@>>{r}>
\cM_{N_1, U}^d.
\end{CD}
\]
Here the morphism $r$ is that given by the morphism of functors that 
sends
$(E \to S, \psi\colon  N_2 \to \Hom(S, E))$
to $(E \to S, \psi'\colon  N_1 \xto{\alpha} N_2 \xto{\psi} \Hom(S,E))$.

\begin{lem}
Let 
$n_1 \in N_1$
and set 
$n_2= \alpha(n_1) \in N_2$.
We have
\[
r^*
g_{N_1, U, n_1}^{(d)}
=
g_{N_2, U, n_2}^{(d)}.
\]
\end{lem}
\begin{proof}
From Lemma~\ref{lem:pullback theta},
we have
\[
r_E^* \theta_{N_1, U}^{(d)}
=
\theta_{N_2, U}^{(d)}.
\]
The claim follows by diagram chasing.
\end{proof}

\subsubsection{}
Let $N_1, N_2 \in \cC_U^d$
be nonzero objects and 
$N_1 \xto{\alpha} N_2$ be 
a surjection of $A$-modules.
We obtain a morphism
$\cM_{N_1, U}^d 
\to 
\cM_{N_2, U}^d$
by sending
$(E \to S, \psi\colon  N_1 \to \Hom(S, E))$
to 
$(E/\psi(\Ker \alpha) \to S,
\psi'\colon  N_2
\to \Hom(S, E/\psi(\Ker \alpha))$
where $\psi'$ is that induced by $\psi$.

From the definition of moduli functors, 
we obtain a commutative diagram
\[
\begin{CD}
E_1 @>{m_1}>> E_1/\psi(\Ker \alpha)
@>{\cong}>{m_2}>  
E_2 \times_{\cM^d_{N_2, U}} \cM^d_{N_1, U}
@>{m_3}>>
E_2
\\
@VVV @VVV @VVV @VVV
\\
\cM^d_{N_1, U}
@>{=}>>
\cM^d_{N_1, U}
@>{=}>>
\cM^d_{N_1, U}
@>{m}>>
\cM^d_{N_2, U}
\end{CD}
\]
where the last square is cartesian.
Here, we write $E_1=E_{N_1, U}^d$ and 
$E_2=E_{N_2, U}^d$ for short.

\begin{lem}
Let $n_2 \in N_2$.
Then
\[
m^*
g_{N_2, U, n_2}
=
\prod_{\alpha(n_1)=n_2}
g_{N_1, U, n_1}.
\]
\end{lem}
\begin{proof}
By Lemma~\ref{lem:pullback theta},
$m^*_2 m^*_3 \theta_{N_2, U}^{(d)}$
is the theta function for the Drinfeld module
\[
E_1/\psi(\Ker \alpha)
\to \cM_{N_1, U}^d.
\]
By Lemma~\ref{lem:pushforward theta},
this also equals 
$m_{1*}
\theta_{N_1, U}^{(d)}.$
Specializing this equation at $n_2$, 
we obtain the claim.
\end{proof}

\subsection{Elements in motivic cohomology}
We use the motivic cohomology of 
Mazza-Voevodsky-Weibel
from \cite{MVW}.

For $i=1, \dots, d$,
let $N_i \in \cC^d_U$
be a nonzero $A$-module of finite length
that is generated by one element.
Let $b_i \in N_i \setminus \{0\}$.
Set $\bN=\bigoplus_{i=1}^d N_i$.
For $i=1, \dots, d$,
let 
$\iota_i\colon N_i \to \bN$
be the inclusion into the $i$-th factor.
We take $b_i$'s so that 
$g_{N_i, b_i}^{(d)}$ is a unit.
Each $\iota_i$ induces
$f_i\colon  \cM_{\bN, U}^d \to \cM_{N_i, U}^d$.

We identify the group of units
$\cO(X)^\times$ and the 
motivic cohomology
$H^1_\cM(X, \Z(1))$ of a scheme $X$
smooth over a field.
Using the product structure 
$H^1_\cM(X, \Z(1))^{\otimes d} 
\to H^d_\cM(X, \Z(d))$
for positive integers $d$
of motivic cohomology, we set
\[
\kappa'_{\bN, U, (b_i)}=
f_1^* g_{N_1, (b_1), U}^{(d)'}
\cdots
f_d^* g_{N_d, (b_d), U}^{(d)'}
\in
H^d_\cM(\cM_{\bN, U}^d,
\Z(d))
\]
and
\[
\kappa_{\bN, U, (b_i)}
=(q_\infty^d-1)
\kappa'_{\bN, U, (b_i)}
\in
H^d_\cM(\cM_{\bN, U}^d,
\Z(d)).
\]
We note that 
the element $\kappa'_{\bN, U, (b_i)}$ depends on the choice of
the pre-Siegel units involved in the definition, but the
element
$\kappa_{\bN, U, (b_i)}$
is independent of the choice
because the pre-Siegel units are determined 
up to 
$(q_\infty^d-1)$-st roots of unity.

We remark here that the product of Siegel units 
$f_1^* g_{N_1, (b_1), U}^{(d)}
\cdots
f_d^* g_{N_d, (b_d), U}^{(d)}$
is uniquely 
determined, but it is $(q_\infty^d-1)^{d-1}$ times
$\kappa_{\bN, (b_i)}$, and is less sharp.

\subsubsection{}
We use the notation above.
Let $N_i'$ be a nonzero quotient $A$-module of $N_i$
for $i=1, \dots, d$.
(In particular, $N_i' \in \cC_U^d$.)
Let $b_i'$ denote the 
image of $b_i$ in $N_i'$.
We write
$\bN'=\bigoplus_{j=1}^d N_j'$
and $N_i''=\Ker(N_i \to N_i')$.
Let 
$m\colon  \cM^d_{\bN, U} \to \cM^d_{\bN', U}$
be the morphism induced by the surjection
$\bN \to \bN'$.
Since $m$ is finite surjective, we have the pushforward map
\[
m_*\colon 
H^d_\cM(\cM_{\bN, U}^d, \Z(d))
\to
H^d_\cM(\cM_{\bN', U}^d, \Z(d))
\]
between the motivic cohomology groups.

Our main theorem is as follows.   The proof will be given in
Section~\ref{sec:proof of Drinfeld Euler}.
\begin{thm}
\label{thm:Drinfeld Euler}
The following statements hold.
(1) If $\Supp N_i'' \subset \Supp N_j'$
for any 
$1 \le i,j \le d$, then
\[
m_* \kappa_{\bN, (b_j), U}
=
\kappa_{\bN', (b_j'), U}
\]

(2) Let $\wp$ be a closed point of $C$.
Suppose that
$\Supp N_i'' \subset \{\wp\}
\subset \Supp N_i$
for every $i$.
Let $e$ denote the number of $i$'s
with $\wp \notin \Supp N_i'$.
Then
\[
m_*
\kappa_{\bN, (b_i), U}=
\sum_{r=0}^e
(-1)^r
q_\wp^{r(r-1)/2}
T_{[\wp^{\oplus r}]}
\kappa_{\bN', (b_j'), U}
\]
where $T_{[\wp^{\oplus r}]}$
are the Hecke operators as 
defined in Section~\ref{sec:exam Hecke}.
\end{thm}

\subsection{A Lemma on motivic cohomology}
This section contains a lemma concerning
motivic cohomology groups.   
The lemma is stated generally, but
applied only to Drinfeld modular schemes.   
We use it in Lemma~\ref{lem:transfer Drinfeld} 
to show that the presheaf of motivic 
cohomology groups is equipped with transfers.
The notation and content of this section 
are independent of the rest of the sections.

\subsubsection{}
Let $k$ be a field.  
(It will be a finite field in our application.)
For smooth schemes $X$, $Y$ over $k$, let $Cor_k(X,Y)$ denote
the abelian group of finite correspondence from $X$ to $Y$.
By definition, $Cor_k(X,Y)$ is the free abelian group
generated by the integral closed subschemes 
$W \subset X \times_{\Spec k} Y$
such that the composite $W \inj X \times_{\Spec k} Y \to X$ 
is finite and surjective over a connected component of $X$.
The motivic cohomology groups are presheaves with 
transfers (see \cite[Example 2.7]{MVW}).
Let $f\colon X \to Y$ be a morphism of smooth $k$-schemes.
Then the graph $\Gamma_f$ of $f$ belongs to $Cor_k(X,Y)$.
We call the action (see the remark after 
\cite[Definition 2.1]{MVW})
by $\Gamma_f$ the pullback by $f$ and will denote it by $f^*$.
For a morphism $g\colon  X \to Y$ which is finite and surjective over 
a connected component of $Y$, 
the transpose ${}^t\Gamma_g$
of the 
graph of $g$
belongs to $Cor_k(Y, X)$.   The action by 
${}^t\Gamma_g$ will be called the pushforward by $g$
and will be denoted $g_*$.

\begin{lem}
Let
\begin{equation} \label{diag1}
\begin{CD}
X' @>{g'}>> X \\
@V{f'}VV @VV{f}V \\
Y' @>{g}>> Y
\end{CD}
\end{equation}
be a commutative diagram of equidimensional schemes 
which are smooth over $k$. Suppose that 
$f$ has the following property:
For any connected component of $X$, 
the restriction of $f$ to the component
is finite and surjective onto a connected component of $Y$.
Suppose $f'$ has a similar property.
Assume that $g$ satisfies (*).  
Suppose moreover that there exists a
dense open subscheme 
$U \subset Y$ such that the
base change
\begin{equation} \label{diag2}
\begin{CD}
X' \times_Y U @>{g'|_U}>> X \times_Y U \\
@V{f'|_U}VV @VV{f|_U}V \\
Y'\times_Y U @>{g|_U}>> U
\end{CD}
\end{equation}
of the diagram \eqref{diag1} is cartesian. Then the diagram
\begin{equation} \label{diag3}
\begin{CD}
H_\cM^i(X', \Z(j)) @<{g'^*}<< H_\cM^i(X, \Z(j)) \\
@V{f'_*}VV @VV{f_*}V \\
H_\cM^i(Y',\Z(j)) @<{g^*}<< H_\cM^i(Y,\Z(j))
\end{CD}
\end{equation}
of abelian groups is commutative.
\end{lem}
\label{lem:motivic cohomology}
\begin{proof}
Let 
$\Gamma_f \in Cor_k(X,Y)$,
$\Gamma_g \in Cor_k(Y',Y)$, 
$\Gamma_{f'} \in Cor_k(X', Y')$,
and $\Gamma_{g'} \in Cor_k(X',X)$ denote the graph of
$f$, $g$, $f'$, and $g'$, respectively. 
Let ${}^t \Gamma_f \in Cor_k(Y,X)$,
${}^t \Gamma_{f'} \in Cor_k(Y',X')$
denote the transposes.
Then we have
$g^* \circ f_* = \Gamma_g^* \circ {}^t \Gamma_f^*
= ({}^t \Gamma_f \circ \Gamma_g)^*$ and
$f'_* \circ g'^* = {}^t \Gamma_{f'}^* \circ \Gamma_{g'}^*
= (\Gamma_{g'} \circ {}^t \Gamma_{f'})^*$.
Hence it suffices to show the equality
${}^t \Gamma_f \circ \Gamma_g = \Gamma_{g'} \circ {}^t \Gamma_{f'}$
in $Cor_k(Y',X)$.
Our assumption on $g$ implies that 
$g$ maps the generic point of any connected component of
$Y'$ to the generic point of some connected component of $Y$.
It follows that, for any integral closed subscheme 
$W \subset Y' \times_{\Spec k} X$
which is finite and surjective over a connected component of $Y'$,
the image of the generic point $\xi_W$ of $W$ under the morphism
$W \to Y'$ belongs to $Y' \times_Y U$.
Hence to prove the claim, it suffices to prove the
${}^t \Gamma_{f|_U} \circ \Gamma_{g|_U} 
= \Gamma_{g'|_U} \circ {}^t \Gamma_{f'|_U}$
in $Cor_k(Y'\times_Y U,X \times_Y U)$.
The last claim can be checked easily by using the cartesian
diagram \eqref{diag2}.
\end{proof}

\subsubsection{}
\begin{lem}[projection formula]
\label{lem:projection formula}
Let $f\colon X \to Y$ be a finite morphism of
equidimensional schemes which are smooth over $k$.
Then for any $x \in H^i_\cM(X,\Z(j))$ 
and for any
$y \in H_\cM^{i'}(Y,\Z(j'))$, 
we have the projection formula
$$
f_*(x \cdot f^* y) = (f_*x) \cdot y.
$$
where the dot denotes the product in motivic cohomology.
\end{lem}
\begin{proof}
Omitted.
\end{proof}

\subsection{Motivic cohomology presheaf of rings with
transfers}
For $x \in \wt{\cFC}$ (recall that this category depends 
on a choice of $U$),
we set
\[
G(x)=
\bigoplus_{i \ge 0}
H_\cM^i
(\cM_{-,U}^d(x),
\Z(i))/((q_\infty^d-1)\text{-torsion}).
\]
Ideally, we do not want to divide out the 
$(q_\infty^d-1)$-torsion.  
The reason we mod out by 
$(q_\infty^d-1)\text{-torsion}$
arises from that the pre-Siegel units are defined only
up to $(q_\infty^d -1)$-st roots of unity.
We will obtain a result without modding out by 
$(q_\infty^d -1)$-torsion in the end.

Now, $G$ defined as above is a presheaf since the 
morphisms between the
$\cM_{-,U}^d(x)$'s
are flat 
and motivic cohomology is contravariant with respect to 
flat maps between regular schemes.
We have the pushforward maps, since the 
schemes $\cM_{-,U}^d(x)$ are regular and the
morphisms between them are finite (proper).
We show that these pushforward maps give the 
structure of transfers.
\begin{lem}
The presheaf $G$ is a presheaf of rings
with transfers.
\end{lem} 
\label{lem:transfer Drinfeld}
\begin{proof}
We need to check the three conditions 
in Definition~\ref{def:transfer} and 
the condition in Definition~\ref{def:ring transfer}.

Condition (1) holds trivially.
The condition in Definition~\ref{def:ring transfer}
holds because of Lemma~\ref{lem:projection formula}.

We show that Condition (2) holds.
Suppose we are given a cartesian diagram as in Condition (2).
Recall that the fiber product
in $\wt{\cFC}$ is 
computed explicitly in the proof of
Proposition~\ref{prop:FC_product}.
We are thus reduced to the case of the cartesian diagram
\eqref{diag:cartesian} in the proof of Proposition~\ref{prop:FC_product}.
In particular, 
the objects aside from the left upper corner are of the 
form $G_i \backslash Z$ for some $G_i$ with $i=1,2,3$.
Now consider the corresponding diagram 
of $U$-schemes via the functor
$\cM_{-,U}^d$.
We can choose a dense open $U' \subset U$
such that the morphisms in the restriction of the diagram of $U$-schemes to $U'$ are all \'etale.
Each morphism can be described purely in terms of the groups
$G_i$'s.   Using this description, 
one can check that the diagram of $U'$-schemes
is cartesian.
Then one 
can apply Lemma~\ref{lem:motivic cohomology} 
to deduce the claim. 

We show that Condition (3) holds.
We may assume that $f$ is a morphism in $\cC$.
Observe that the degree of a morphism 
between $A$-schemes equals that of a morphism between $k$-schemes.
The morphisms between Drinfeld modular varieties over $k$
are \'etale.   Hence,
the degree can be computed purely in terms
of the level structures, i.e., in terms of 
homomorphisms in the category of $A$-modules.
On the other hand, the degree of a morphisms in $\cC$
can be computed in terms of automorphisms groups in $\cC$
(see Section~\ref{sec:cyclotomic Euler} 
for the example computation when $d=1$)
and in turn in terms of $A$-modules.
One can check directly the cases of morphisms of type
$N_1=N_1 \hookrightarrow N_2$
and $N \twoheadleftarrow N_1 =N_1$
to deduce the claim.
\end{proof}

\subsection{}
\label{sec:proof of Drinfeld Euler}
Define a presheaf $H'$ on $\cFCo{d}$ by 
setting $H'(N)=\cO(\cM^d_N \times_{U_N} U_{\bN})^{\times}$ for
an object $N$ in $\cCo{d}$.
We claim that $H'$ is a sheaf.   
It suffices to check that for a Galois covering $N \to N'$ of Galois group $G$,
the sheaf property 
$H(N)^G=H(N')$ holds.
By the irreducibility of the moduli, 
it suffices to check the claim for unit groups 
of the corresponding moduli spaces over $k$.
There, the morphism is the quotient morphism
by the group $G$,
hence the claim follows.
We equip it with the canonical 
structure of transfers 
(see Section~\ref{sec:with transfers}).

Define a sub presheaf $H$ of $H'$ by setting
$H(\overline{N}) = 
\{x^{q_\infty^d-1}\, | \, x \in H'(\overline{N})\}
\subset H(\overline{N})$ 
for an object 
$\overline{N}$ in 
$\wt{\cFC}$.  The structure of transfers on $H$ induces a
structure of transfers on $H'$.

We define a morphism of presheaves
$\overline{g}\colon  \BS^{* \prime} \to H$.  
Suppose we are given an object $\overline{N} \in \wt{\cFC}$
and $\phi \in \BS^{* \prime}(\overline{N})$.  
Take a Galois covering $N \to \overline{N}$ of Galois group $G=\Gal(N/\overline{N})$ 
with an object $N=\coprod_{i=1}^m N_i \in \cFC^d$
such that the pullback of $\phi$ to $\BS^{* \prime}(N)$ is of the following form:
\[ 
\sum_{i=1}^r 
\sum_{g\in G/H_j} g a_j \phi_{N_{i(j)},b_j}.
\]
Here each $a_j$ is an integer, 
$b_j \in N_{i(j)} \setminus\{0\}$, 
and $H_j \subset G$ 
is the subgroup of elements which fix  $i(j) \in \{1,\dots, m\}=
\pi_0(N)$ and
further fix $b_j$ under the action given via $H_j \to \Aut(N_{i(j)})$.
Recall that we defined $\phi_{N_{i(j)}, b_i}$ to be the characteristic function
of $b_i$.

For each $j\in \{1,\dots, r\}$,
set 
$N_j'=\langle b_j \rangle \subset
N_{i(j)}$
to be the $A$-submodule generated by $b_j$.
Let 
$r_j\colon N_{i(j)} \to N_j'$
be the morphism represented by
$N_j'=N_j' \subset N_{i(j)}$.
We have
$r_j^* \phi_{N'_j,b_j}'=\phi_{N_{i(j)}, b_j}$.
Since $\Gal(r_j)=H_j$,
it follows that 
$\phi_{N_{i(j)}, b_j}$
is $H_j$-invariant.

Let
\[
\alpha=\prod_{i=1}^r \prod_{g\in G/H_j} g(r_j^*
\theta'_{N'_j, b_j})^{a_j}.
\]
It is an element of $H'(N)$.  Since
$r_j^*\theta'_{N_j, b_j}$ is $H_j$-invariant, the element $\alpha$ is
$G$-invariant.  Hence $\alpha^{q_\infty^d-1}$ defines an element of
$(H'(N)^G)^{q_\infty^d-1}=H(\overline{N})$.
We define $\overline{g}$ to be the morphism of presheaves that sends 
$\phi$ to $\alpha^{q_\infty^d-1}$.

We have maps
\[
\cO(\cM_{-,U}^d(\overline{N}))^\times
\xto{=}
H^1_\cM(\cM_{-,U}^d(\overline{N}), \Z(1))
\to
\bigoplus_{j \ge 1}
H^j_\cM(\cM_{-,U}^d(\overline{N}), \Z(j))
\]
where the first is the identification 
and the second is the canonical 
inclusion.
The multiplication-by-$(q_\infty^d-1)$ map 
induces an isomorphism
$\alpha\colon  \cO(\cM_{-,U}^d(\overline{N}))^\times/
((q_\infty^d-1)\text{-torsion})
\xto{\cong} 
(\cO(\cM_{-,U}^d(\overline{N}))^\times)^{q_\infty^d-1}.$
We define a morphism $H \to G$
by
\[
\begin{array}{ll}
H(\overline{N})=(\cO(\cM_{-,U}^d(\overline{N}))^\times)^{q_\infty^d-1}
\xto{\alpha^{-1}}
(\cO(\cM_{-,U}^d(\overline{N}))^\times)/((q_\infty^d-1)\text{-torsion}\\
=
\bigoplus_{j \ge 0}
H^j_\cM(\cM_{-,U}^d(\overline{N}), \Z(j))
/((q_\infty^d-1)\text{-torsion})
=G(\overline{N})
\end{array}
\]
for each $\overline{N} \in \wt{\cFC}$.

Now take $F_i=H, F_i'=H'$ in Situation II.
Then Theorem~\ref{main theorem} 
implies 
$$
m_*
\kappa'_{\bN,(b_j)}
=\sum_{r=0}^e (-1)^r q_\wp^{r(r-1)/2}
T_{[\wp^{\oplus r}]}
{\kappa'}_{\bN', (b'_j)}
$$
in $G(\bN')$.
(The theorem is stated only for 
presheaves with transfers on $\wt{\cFC^d}$,
but the argument there works for the category
$\wt{\cFC}$ that we are using.   
To use Theorem~\ref{main theorem},
we also need to check that the schemes occurring 
in the proof of the theorem are indeed regular.
This follows from Theorem 1.1.1 of \cite{Regularity}.)

Now, take any lift of the equation
to 
$\bigoplus_{j \ge 0}
H^j_\cM(\cM^d_{-,U}(\bN'), \Z(j))$.
Multiplying both sides by $(q_\infty^d-1)$  
gives the desired equality.
\qed

\section{Remark: Integrality of zeta elements}
This remark is independent of the other parts of this book.
We constructed what we call zeta elements in the rational 
$K$-theory of Drinfeld modular varieties and 
showed that the subspace generated by zeta elements 
surjects to the space of certain automorphic forms (automorphic
forms for Drinfeld modules) in \cite[p.531, Theorem 1.2]{KY:zeta}.
We addressed the question of the integrality of zeta elements
in \cite[p.532, Section 1]{KY:zeta}.   We answer this in 
the affirmative below.   

We recall that the zeta elements 
are generated by elements
$\kappa_{I, J, \gamma}$ in \cite[p.538, 2.4.7]{KY:zeta}.
These elements are in the rational K-theory of moduli spaces
$\cM_{I, J}^d$
where $\cM_{I,J}^d$ 
is $\cM^d_{(A/I)^{d-1} \oplus (A/J), \Spec A} \times_{\Spec A} \Spec F$ 
in this book
for ideals $I, J \subset A$.   
The elements $\kappa_{I, J, \gamma}$
can be expressed as a linear combination of elements
of the form
$\kappa_{N, \Spec F, (b_i)}$ 
using the notation in this book (but allowing the open subscheme
to be the limit), where $N$ may run through some modules, that descends to 
the rational $K$-theory of $\cM^d_{(A/I)^{d-1} \oplus (A/J),\Spec F}$.
As we have seen in this book, 
the elements $\kappa_{N, \Spec F, (b_i)}$ 
are indeed defined over $\Spec A$
in the sense that there is an element in 
the rational $K$-theory of $\cM^d_{(A/I)^{d-1} \oplus (A/J), \Spec A}$
that maps to  $\kappa_{N, \Spec F, (b_i)}$
via the restriction map.
This answers the integrality question in the affirmative.

\printindex

\end{document}